\documentclass[10pt, oneside]{book}

\usepackage{amssymb,amsmath,amsthm,bbm}
\usepackage{verbatim,float,url,dsfont}
\usepackage{graphicx,subfigure,psfrag}
\usepackage{algorithm, algpseudocode}
\usepackage{algorithmicx}
\usepackage{bibentry}
\algtext*{EndWhile}
\algtext*{EndIf}
\algtext*{EndFor}
\algtext*{EndProcedure}
\usepackage{mathtools,enumitem}
\usepackage{multirow}
\usepackage{ragged2e}
\usepackage{xr-hyper}
\usepackage{array}
\usepackage{stmaryrd,scalerel}
\usepackage[dvipsnames]{xcolor}
\usepackage{tikz}
\usepackage[utf8]{inputenc} 
\usepackage{etoolbox}
\usepackage{comment}
\newcommand{\commentAlt}[1]{\ignorespaces}
\newcommand{\commentLongAlt}[1]{\ignorespaces}

\newbool{usegrayscale}

\setbool{usegrayscale}{false}

\ifbool{usegrayscale}{%
  \newcommand{\diagramspath}{figures/diagrams_grayscale/}%
  \definecolor{theoremcolor}{HTML}{BBBBBB}
}{%
  \newcommand{\diagramspath}{figures/diagrams/}%
  \definecolor{theoremcolor}{HTML}{EBB900}
}

\usepackage[colorlinks=true,citecolor=blue,urlcolor=blue,linkcolor=blue]{hyperref}
\usepackage[top = 1in, bottom = 1in, left = 1.25in, right = 1.25in]{geometry}
\usepackage{natbib}

\usepackage[T1]{fontenc}    
\usepackage{booktabs}       
\usepackage{nicefrac}         
\usepackage{microtype}      
\usepackage{fancyhdr}
\fancyhfoffset[L,R]{0.25in}

\ifdefined\TimesFont 
\usepackage{times} 
\fi

\ifdefined\ParSkip 
\usepackage{parskip} 
\fi

\allowdisplaybreaks

\newtheorem{theorem}{Theorem}
\numberwithin{theorem}{chapter}
\newtheorem{lemma}[theorem]{Lemma}
\newtheorem{corollary}[theorem]{Corollary}
\newtheorem{proposition}[theorem]{Proposition}

\newtheorem{definition}[theorem]{Definition}

\newtheorem{example}[theorem]{Example}
\newtheorem{fact}[theorem]{Fact}
\newtheorem{algbox}[theorem]{Algorithm}
\newtheorem{blankbox}[theorem]{Blank}

\newcommand{\argmin}{\mathop{\mathrm{argmin}}}

\newcommand{\minimize}{\mathop{\mathrm{minimize}}}

\newcommand{\st}{\mathop{\mathrm{subject\,\,to}}}

\def\R{\mathbb{R}}

\def\E{\mathbb{E}}
\def\P{\mathbb{P}}
\def\Cov{\mathrm{Cov}}
\def\Var{\mathrm{Var}}
\def\indep{\perp\!\!\!\perp}
\def\hf{\hat{f}}
\def\cA{\mathcal{A}}

\def\cC{\mathcal{C}}
\def\cD{\mathcal{D}}
\def\cE{\mathcal{E}}

\def\cG{\mathcal{G}}

\def\cH{\mathcal{H}}
\def\cI{\mathcal{I}}

\def\cM{\mathcal{M}}
\def\cN{\mathcal{N}}
\def\cP{\mathcal{P}}
\def\cR{\mathcal{R}}
\def\cS{\mathcal{S}}

\def\cW{\mathcal{W}}
\def\cX{\mathcal{X}}
\def\cY{\mathcal{Y}}
\def\cZ{\mathcal{Z}}

\def\bfW{\mathbf{W}}
\def\bfX{\mathbf{X}}
\def\bfY{\mathbf{Y}}
\def\bfw{\mathbf{w}}
\def\bfx{\mathbf{x}}
\def\bfy{\mathbf{y}}

\def\bigo{\mathcal{O}}
\def\littleo{\mathrm{o}}

\newcommand{\err}{\mathrm{err}}
\newcommand{\quantile}{\mathrm{Quantile}}

\newcommand{\eqd}{\stackrel{\textnormal{d}}{=}}

\newcommand{\asto}{\stackrel{\textnormal{a.s.}}{\to}}

\newcommand{\asleq}{\stackrel{\textnormal{a.s.}}{\leq}}
\newcommand{\iidsim}{\stackrel{\textnormal{i.i.d.}}{\sim}}
\newcommand{\dtv}{\textnormal{d}_{\textnormal{TV}}}

\def\ind#1{\mathbbm{1}\left\{#1\right\}}
\def\indsub#1{\mathbbm{1}_{#1}}
\usepackage{mathtools}

\usepackage{mathtools}
\mathtoolsset{showonlyrefs}

\usepackage{afterpage}

\makeatletter
\def\blfootnote{\xdef\@thefnmark{}\@footnotetext}
\makeatother

\usepackage{mdframed}
\usepackage{tcolorbox}
\tcbuselibrary{skins}
\tcbuselibrary{breakable}

\newtcolorbox{proofframe}[2][]{
    enhanced,
    breakable=true,
    sharp corners=all,
    top=0mm, 
    left=4mm,
    enlarge top by=2em,
    enlarge bottom by=0.5em,
    left skip=1.5em,
    right skip=1.5em,
    borderline={0.2em}{0pt}{gray!20}, 
    boxrule=0pt,
    colback=gray!5,
    coltitle=black,
    colbacktitle=gray!5,
    colframe=gray!5,
    parbox=false,
    fonttitle=\bfseries,
    before upper={\vskip -0cm},
    title={{\begin{tikzpicture}[baseline=(title.base)]
            \node[rectangle, fill=gray!20, sharp corners, inner sep=1.5mm] (title) {#2};
          \end{tikzpicture}}
          },
    toptitle=-0.4cm,
}

\renewenvironment{proof}[1][]{%
    \begin{proofframe}{#1}\relax%
    }{\end{proofframe}}

\renewenvironment{fact}[1][]{%
   \refstepcounter{theorem}
    \ifstrempty{#1}%
    {\mdfsetup{%
    frametitle={%
        \tikz
        \node[anchor=east,rectangle,fill=theoremcolor!70]{Fact~\thetheorem};}} 
    }%
    {\mdfsetup{%
    frametitle={%
        \tikz
        \node[anchor=east,rectangle,fill=theoremcolor!70] {Fact~\thetheorem:~#1};}}%
    }%
    \mdfsetup{innertopmargin=0.2em,linecolor=theoremcolor!70, linewidth=0.2em, nobreak=true,
    topline=true, frametitleaboveskip=-0.9em,}
    \begin{mdframed}[]\relax%
    }{\end{mdframed}
}

\renewenvironment{theorem}[1][]{%
   \refstepcounter{theorem}
    \ifstrempty{#1}%
    {\mdfsetup{%
    frametitle={%
        \tikz
        \node[anchor=east,rectangle,fill=theoremcolor!70]{Theorem~\thetheorem};}} 
    }%
    {\mdfsetup{%
    frametitle={%
        \tikz
        \node[anchor=east,rectangle,fill=theoremcolor!70] {Theorem~\thetheorem:~#1};}}%
    }%
    \mdfsetup{innertopmargin=0.2em,linecolor=theoremcolor!70, linewidth=0.2em, nobreak=true,
    topline=true, frametitleaboveskip=-0.9em,}
    \begin{mdframed}[]\relax%
    }{\end{mdframed}
}

\renewenvironment{lemma}[1][]{%
    \refstepcounter{theorem}
    \ifstrempty{#1}%
    {\mdfsetup{%
    frametitle={%
        \tikz
        \node[anchor=east,rectangle,fill=theoremcolor!70]{Lemma~\thetheorem};}} 
    }%
    {\mdfsetup{%
    frametitle={%
        \tikz
        \node[anchor=east,rectangle,fill=theoremcolor!70] {Lemma~\thetheorem:~#1};}}%
    }%
    \mdfsetup{innertopmargin=0.2em,linecolor=theoremcolor!70, linewidth=0.2em, nobreak=true,
    topline=true, frametitleaboveskip=-0.9em,}
    \begin{mdframed}[]\relax%
    }{\end{mdframed}
}

\renewenvironment{blankbox}[1][]{%
    \ifstrempty{#1}%
    {\mdfsetup{%
    frametitle={%
        \tikz
        \node[anchor=east,rectangle,fill=theoremcolor!70]{};}} 
    }%
    {\mdfsetup{%
    frametitle={%
        \tikz
        \node[anchor=east,rectangle,fill=theoremcolor!70] {#1};}}%
    }%
    \mdfsetup{innertopmargin=0.2em,linecolor=theoremcolor!70, linewidth=0.2em, nobreak=true,
    topline=true, frametitleaboveskip=-0.9em,}
    \begin{mdframed}[]\relax%
    }{\end{mdframed}
}

\renewenvironment{example}[1][]{%
    \refstepcounter{theorem}
    \ifstrempty{#1}%
    {\mdfsetup{%
    frametitle={%
        \tikz
        \node[anchor=east,rectangle,fill=theoremcolor!70]{Example~\thetheorem};}} 
    }%
    {\mdfsetup{%
    frametitle={%
        \tikz
        \node[anchor=east,rectangle,fill=theoremcolor!70] {Example~\thetheorem:~#1};}}%
    }%
    \mdfsetup{innertopmargin=0.2em,linecolor=theoremcolor!70, linewidth=0.2em, nobreak=true,
    topline=true, frametitleaboveskip=-0.9em,}
    \begin{mdframed}[]\relax%
    }{\end{mdframed}
}

\renewenvironment{definition}[1][]{%
    \refstepcounter{theorem}
    \ifstrempty{#1}%
    {\mdfsetup{%
    frametitle={%
        \tikz
        \node[anchor=east,rectangle,fill=theoremcolor!70]{Definition~\thetheorem};}} 
    }%
    {\mdfsetup{%
    frametitle={%
        \tikz
        \node[anchor=east,rectangle,fill=theoremcolor!70] {Definition~\thetheorem:~#1};}}%
    }%
    \mdfsetup{innertopmargin=0.2em,linecolor=theoremcolor!70, linewidth=0.2em, nobreak=true,
    topline=true, frametitleaboveskip=-0.9em,}
    \begin{mdframed}[]\relax%
    }{\end{mdframed}
}

\renewenvironment{corollary}[1][]{%
    \refstepcounter{theorem}
\ifstrempty{#1}%
    {\mdfsetup{%
    frametitle={%
        \tikz
        \node[anchor=east,rectangle,fill=theoremcolor!70]{Corollary~\thetheorem};}} 
    }%
    {\mdfsetup{%
    frametitle={%
        \tikz
        \node[anchor=east,rectangle,fill=theoremcolor!70] {Corollary~\thetheorem:~#1};}}%
    }%
    \mdfsetup{innertopmargin=0.2em,linecolor=theoremcolor!70, linewidth=0.2em, nobreak=true,
    topline=true, frametitleaboveskip=-0.9em,}
    \begin{mdframed}[]\relax%
    }{\end{mdframed}
}

\renewenvironment{proposition}[1][]{%
    \refstepcounter{theorem}
    \ifstrempty{#1}%
    {\mdfsetup{%
    frametitle={%
        \tikz
        \node[anchor=east,rectangle,fill=theoremcolor!70]{Proposition~\thetheorem};}} 
    }%
    {\mdfsetup{%
    frametitle={%
        \tikz
        \node[anchor=east,rectangle,fill=theoremcolor!70] {Proposition~\thetheorem:~#1};}}%
    }%
    \mdfsetup{innertopmargin=0.2em,linecolor=theoremcolor!70, linewidth=0.2em, nobreak=true,
    topline=true, frametitleaboveskip=-0.9em,}
    \begin{mdframed}[]\relax%
    }{\end{mdframed}
}

\renewenvironment{algbox}[1][]{%
    \refstepcounter{theorem}
    \ifstrempty{#1}%
    {\mdfsetup{%
    frametitle={%
        \tikz
        \node[anchor=east,rectangle,fill=theoremcolor!70]{Algorithm~\thetheorem};}} 
    }%
    {\mdfsetup{%
    frametitle={%
        \tikz
        \node[anchor=east,rectangle,fill=theoremcolor!70] {Algorithm~\thetheorem:~#1};}}%
    }%
    \mdfsetup{innertopmargin=0.2em,linecolor=theoremcolor!70, linewidth=0.2em, nobreak=true,
    topline=true, frametitleaboveskip=-0.9em,}
    \begin{mdframed}[]\relax%
    }{\end{mdframed}
}

\usepackage{makeidx}
\makeindex

\fancypagestyle{plain}{%
    \fancyhf{} 
    \fancyhead[R]{\thepage}
    \fancyfoot[C]{\scriptsize\textcolor{gray}{This material will be published by Cambridge University Press as \emph{Theoretical Foundations of Conformal Prediction} by Anastasios N.\ Angelopoulos, Rina Foygel Barber, and Stephen Bates. This pre-publication version is free to view and download for personal use only. Not for re-distribution, re-sale or use in derivative works. \copyright Anastasios N. Angelopoulos, Rina Foygel Barber, and Stephen Bates, 2026.}} 
}
\title{Theoretical Foundations of Conformal Prediction
}
\author{Anastasios N.\ Angelopoulos\thanks{Department of Electrical Engineering and Computer Science, University of California at Berkeley}, Rina Foygel Barber\thanks{Department of Statistics, University of Chicago}, Stephen Bates\thanks{Department of Electrical Engineering and Computer Science, Massachusetts Institute of Technology}}
\date{}
\begin{document}

\frontmatter
\maketitle
\tableofcontents

\mainmatter

\chapter*{Preface}

\begin{blankbox}[A note to the reader]
This draft is a preliminary version of a forthcoming textbook. We welcome all feedback from readers---please contact the authors at \texttt{tfcpbook@gmail.com} with any typos, corrections, clarifications, or suggestions.
Since this manuscript is a work-in-progress, if referencing any of the results in the book please be aware that later updates may lead to changes in the statements and the numbering of the results.
\end{blankbox}

This book is about conformal prediction and related inferential techniques that build on permutation tests and exchangeability.
These techniques are useful in a diverse array of tasks, including hypothesis testing and providing uncertainty quantification guarantees for machine learning systems.
Much of the current interest in conformal prediction is due to its ability to integrate into complex machine learning workflows, solving the problem of forming prediction sets without any assumptions on the form of the data generating distribution. 
Since contemporary machine learning algorithms have generally proven difficult to analyze directly, conformal prediction's main appeal is its ability to provide formal, finite-sample guarantees when paired with such methods. 

The goal of this book is to teach the reader about the fundamental technical arguments that arise when researching conformal prediction and related questions in distribution-free inference.
Many of these proof strategies, especially the more recent ones, are scattered among research papers, making it difficult for researchers to understand where to look, which results are important, and how exactly the proofs work.
To bridge this gap, in this book we present what we believe to be some of the most important results in the literature.
We develop their proofs in a unified language, with illustrations, and with an eye towards pedagogy.
We note that this book does not focus on the question of how to apply conformal prediction in practice---for a more practical and application-oriented introduction to conformal prediction, the reader may instead prefer to read \emph{Conformal Prediction: A Gentle Introduction} \citep{angelopoulos2023conformal}.

This book is meant for those working on the development of statistical theory and methodology, broadly speaking.
We envision our audience including everyone from classically trained statisticians interested in finite-sample model-agnostic bounds, to machine learning researchers looking for a modular theory that applies to the ever-changing landscape of machine learning algorithms. The background required is generally at the level of first-year graduate coursework in theoretical statistics; some measure theory will be used occasionally, but the vast majority of results in this book do not require it. 
We hope that this book will provide the reader with a deep understanding of the theoretical underpinnings of the field, so that they may themselves contribute to the ongoing theoretical development of conformal prediction and other areas within distribution-free inference.

\section*{Acknowledgments}
We would like to give our deep thanks to the many friends, collaborators, mentors, students, and reviewers who gave us feedback on this work, including
Emmanuel Cand{\`e}s, Tiffany Ding, Alexander Gammerman, Anders Hoel, Michael I. Jordan, Ilmun Kim, Kallia Kleisarchaki, Hung Le, Gyumin Lee, Yonghoon Lee, Elena Yutong Li, Yushuo Li, David Liang, Ruiting Liang, Chenjia Lin, Lu\'is Marques, Peter McCullagh, Shawn Meng, Ethan Naegele, Theo Olausson, Drew Prinster, Andi Qu, Aaditya Ramdas, Ryan Tibshirani, Vladimir Vovk, Lekun (Bill) Wang, Ziyang Wei, Eric Weine, Andrew Yao, Christopher Yeh, Shangkai Zhu, and Tijana Zrni\'c.
We thank Natalie Tomlinson, Anna Scriven, and the Cambridge University Press for their support.

Lastly, we thank our families for their long-standing support in this effort---this book would not be possible without them.

\part{Background}
\label{part:background}

\chapter{Introduction}
\label{chapter:introduction}

Conformal prediction is a statistical technique that quantifies uncertainty in predictive models, without any assumptions at all on the model and with minimal assumptions on the distribution of the data.
Predictive models can be prone to unexpected inaccuracies and errors, complicating their practical usage.
Conformal prediction guards against these issues, giving rigorous error bounds on predictions. 
This book presents the foundational statistical theory of conformal prediction and related methods.

\section{Uncertainty quantification for prediction}

We now describe the problem of uncertainty quantification for prediction.
Consider a sequence of data points $(X_i, Y_i) \in \cX \times \cY$, for $i=1,\dots,n$. Here, $X_i$ is the feature vector and $Y_i$ is the response variable. 
We are then given a new feature vector  $X_{n+1}$, with the task of predicting its corresponding response value $Y_{n+1}$ (which is unobserved). 
Given a predictive model $\hf$, we can return a prediction $ \hf(X_{n+1})$. 
To communicate our uncertainty in this prediction, we can provide a margin of error around our prediction $\hf(X_{n+1})$, or more generally, a \emph{prediction set} $\cC(X_{n+1})\subseteq\cY$. A common aim for this set is the property of \emph{marginal coverage}, \index{coverage!marginal}
\begin{equation}
\label{eq:marg_coverage}
\P\Big(Y_{n+1} \in \cC(X_{n+1})\Big) \ge 1 - \alpha,
\end{equation}
where $\alpha\in(0,1)$ is a user-specified error level (e.g., $\alpha = 0.1$ for 90\% coverage). If the size of the set $\cC(X_{n+1})$ is large, this indicates high uncertainty in the prediction.

Conformal prediction is a technique for constructing sets $\cC(X_{n+1})$ that satisfy~\eqref{eq:marg_coverage} under no assumptions on $\hf$ or the form of the data distribution.
In particular, if the underlying predictive model $\hf$ is a poor fit to the data, then the accompanying set $\cC(X_{n+1})$ will be large---potentially even infinite.
On the other hand, accurate models will result in smaller sets, and under additional conditions can lead to stronger notions of coverage than that in~\eqref{eq:marg_coverage}.
In either case, the role of conformal prediction is to accurately quantify the level of uncertainty present when using a predictive model on the current data distribution.

\section{Preview of split conformal prediction}
\label{sec:conformal-preview}
\index{split conformal prediction|(}

We begin by presenting a version of the \emph{split conformal prediction} algorithm.
As before, suppose we have training data points $(X_i,Y_i)$ for $i=1,\dots,n$, and a test point $(X_{n+1}, Y_{n+1})$.
Taking $n$ to be an even number for simplicity, the training data will be split into $n/2$ points used for model fitting, and $n/2$ points used for calibration.

We consider the setting $\cY=\R$---that is, a regression problem with a real-valued response.We can construct prediction intervals for $Y_{n+1}$ via the following algorithm (which is one specific case of the more general split conformal algorithm, presented in Section~\ref{sec:intro-scores} below).

\begin{algbox}[Split conformal prediction, special case]
\label{alg:split-cp-preview-residual}
\begin{enumerate}
    \item Use data $(X_i,Y_i)$ for $i=1,\dots,n/2$ to fit a predictive model $\hf : \cX \to \R$.
    \item For $i=n/2 + 1,\dots,n$, compute the absolute residual $S_i = |Y_i - \hf(X_i)|$.
    \item Sort $S_{n/2 + 1},\dots, S_n$ in increasing order, and let $\hat{q}$ be the $\lceil (1-\alpha)(\tfrac{n}{2} + 1)\rceil$-th element in the sorted list.
    \item Return the prediction interval $\cC(X_{n+1}) = [\hf(X_{n+1}) - \hat{q}, \hf(X_{n+1}) + \hat{q}]$.
\end{enumerate}
\end{algbox}
This algorithm's output is the prediction interval $\cC(X_{n+1}) = [\hf(X_{n+1}) - \hat{q}, \hf(X_{n+1}) + \hat{q}]$, which represents our uncertainty about the prediction $\hf(X_{n+1})$ for the target value $Y_{n+1}$.

The predictive model $\hf$ in the first step of Algorithm~\ref{alg:split-cp-preview-residual} can be any function that is based only on $(X_1, Y_1),\dots,(X_{n/2}, Y_{n/2})$. For example, it might be a linear model fitted via least-squares regression.
The final set is an interval centered at the model's prediction, $\hf(X_{n+1})\pm \hat{q}$.
To interpret this choice of $\hat{q}$, we observe $\hat q$ is chosen such that the intervals $[\hf(X_i) - \hat{q}, \hf(X_i) + \hat{q}]$ contain the response variable $Y_i$ for approximately a $(1-\alpha)$ fraction of the points $i=n/2+1,\dots,n$ (i.e., the data points that were not used for training the model $\hf$). 

If the data points are independent and identically distributed (i.i.d.), then for any choice of the predictive model $\hf$, the prediction set above satisfies marginal coverage:
\begin{theorem}[Split conformal coverage guarantee, special case]
\label{thm:conformal_calibration_residual}
Suppose $(X_1,Y_1),\dots,(X_{n+1},Y_{n+1})$ are i.i.d., and let $\cC(X_{n+1})$ be the output of Algorithm~\ref{alg:split-cp-preview-residual}. Then the marginal coverage property~\eqref{eq:marg_coverage} holds.
\end{theorem}
In other words, if the data points are drawn i.i.d.\ from any distribution, then split conformal prediction offers marginal coverage---even if the fitted model $\hf$ is an extremely poor fit to the data, and even if the sample size $n$ is small. In fact, while the i.i.d.\ assumption is sufficient here, it is stronger than necessary---a weaker property known as exchangeability, which will be introduced in Chapter~\ref{chapter:exchangeability}, is the fundamental property required for conformal prediction.

\section{Conformal scores}
\label{sec:intro-scores}
\index{score function|(}

The example above builds intuition for valid coverage with any predictive model and dataset, but the exact form of the algorithm above is rather constrained: by construction, the prediction set will always be of the form $\hf(X_{n+1})\pm \hat{q}$, i.e., a band of constant width around the fitted predictive model $\hf$; see Figure~\ref{fig:score-shape}. Fortunately, the conformal framework is much more flexible than the above example, allowing for nearly unlimited choice in how $\cC(X_{n+1})$ is constructed. The key concept is the \emph{conformal score function}, which is a function $s(x,y)$ such that larger values indicate that the data point $(x,y)$ does \emph{not} agree with (does not `conform' to) the trends observed in the training data. For instance, given a fitted model $\hf$, a common choice for $s$ is the residual score, $s(x,y) = |y - \hf(x)|$, since a large value of the residual indicates that $(x,y)$ does \emph{not} appear to agree with the model trained on the available data.
We will give additional examples of score functions shortly. 

We now state a more general version of the split conformal prediction algorithm, using a generic conformal score function. As before, we again assume $n$ is even for simplicity.
\begin{algbox}[Split conformal prediction, general case]
\label{alg:split-cp-preview-general}
\begin{enumerate}
    \item Use data $(X_i,Y_i)$ for $i=1,\dots,n/2$ to construct a conformal score function $s : \cX\times \cY \to \R$, with the intuition that $s(x,y)$ measures how unusual $(x,y)$ is based on a model fit on data from this split.
    \item For $i=n/2 + 1,\dots,n$, compute the score $S_i = s(X_i, Y_i)$.
    \item Sort $S_{n/2 + 1},\dots, S_n$ in increasing order, and let $\hat{q}$ be the $\lceil (1-\alpha)(\tfrac{n}{2}+1)\rceil$-th element in the sorted list.
    \item Return the prediction set $\cC(X_{n+1}) = \{y\in\cY :  s(X_{n+1}, y ) \le \hat{q} \}$.
\end{enumerate}
\end{algbox}
It may not be immediately clear that the set $\cC(X_{n+1})$ can be computed efficiently, since it nominally requires iterating through all $y \in \cY$.
However, in many cases, it simplifies to an interval that can be computed explicitly, just as in Algorithm~\ref{alg:split-cp-preview-residual}---we will see some examples below.

While we are referring to split conformal prediction as `an algorithm', the flexibility in choosing the score function $s$ means that we should actually think of this as a \emph{family} of algorithms---any given choice of the conformal score function $s$ specifies a particular algorithm.
For example, by choosing the residual score function $s(x,y) = |y - \hf(x)|$, we can obtain
Algorithm~\ref{alg:split-cp-preview-residual} as a special case of Algorithm~\ref{alg:split-cp-preview-general}.

The following result states that the marginal coverage guarantee holds with any score function.
\begin{theorem}[Split conformal coverage guarantee, general case]
\label{thm:conformal_calibration}
Suppose $(X_1,Y_1),\dots,(X_{n+1},Y_{n+1})$ are i.i.d., and let $\cC(X_{n+1})$ be the output of Algorithm~\ref{alg:split-cp-preview-general}. Then the marginal coverage property~\eqref{eq:marg_coverage} holds.
\end{theorem}
Although any score function results in marginal coverage, in practice the choice of the conformal score function $s$ is the single most important decision when implementing conformal prediction: different choices can lead to very different procedures, and a poorly chosen conformal score function $s$ can lead to uninformative or overly large prediction sets. 
We next outline several of the most common conformal score functions and give intuition for the properties of the resulting conformal prediction method. 
Figure~\ref{fig:score-shape} gives an illustration of the sets resulting from some of these conformal score functions.

\index{split conformal prediction|)}

\paragraph{The residual score.} \index{score function!residual score}
We first return to the \emph{residual score}, $s(x,y) = |y - \hf(x)|$, as used in Algorithm~\ref{alg:split-cp-preview-residual} above, where $\hf$ is fitted on the data points $(X_1,Y_1),\dots,(X_{n/2},Y_{n/2})$. This construction will always return a prediction set $\cC(X_{n+1})$ of the same form: a symmetric interval, centered around the point prediction $\hf(X_{n+1})$ with the same width for all values of $X_{n+1}$. This is a simple and natural choice. 
However, the width of the interval does not adapt to $X$, making it far from ideal in many settings---for instance, if the response variable $Y$ has higher or lower noise variance depending on the value of $X$. 

\begin{figure}[t]
    \centering
    \includegraphics[width=\textwidth]{\diagramspath score-shape.pdf}
    \caption{\textbf{The conformal score function} determines the shape of the sets. The shaded band is a visualization of the prediction set $\cC(X_{n+1})\subseteq\cY$ as a function of $X_{n+1}\in\cX$. On the left, the residual score gives a fixed-width band around a fitted model $\hat{f}$. In the middle, the scaled residual score gives a symmetric band that adapts to the non-constant noise variance. On the right, the CQR score gives an asymmetric band that follows the quantiles of the distribution.}
    \commentAlt{Three scatterplots showing the same data points and fitted regression function, with different shaded regions. The first shaded band is constant-width, the second is varying-width but centered around the fitted function, and the third is asymmetric.}
    \label{fig:score-shape}
\end{figure}

\paragraph{The scaled residual score.} \index{score function!scaled residual score}
We can modify the residual score to result in intervals of different width (e.g., in settings where the variance of $Y$ is different at different values of $X$). This construction will again use a trained predictive model $\hf$, and will also require an estimate $\hat{\sigma}(x)$ of the scale of the noise in $Y$ given $X = x$ (e.g., an estimate of the standard deviation), where $\hf$ and $\hat{\sigma}$ are both fitted using data points $(X_1,Y_1),\dots,(X_{n/2},Y_{n/2})$. The \emph{scaled residual score} is then defined as
\[s(x,y) = \frac{|y - \hf(x)|}{\hat{\sigma}(x)},\]
and results in the prediction set \[\cC(X_{n+1}) = \hf(X_{n+1}) \pm \hat q \cdot \hat \sigma (X_{n+1}).\] This can lead to prediction intervals that are a better fit to the data as compared to the residual score, since the function $\hat\sigma(x)$ can capture the nonconstant variance of the noise in $Y$. 

\paragraph{The CQR score.} \index{score function!conformalized quantile regression (CQR)} \index{quantile regression}
Both the residual score and the scaled residual score will return symmetric intervals, which may be a poor fit for certain data distributions. This motivates a more nonparametric approach towards choosing the score. Suppose we use the data points $(X_1,Y_1),\dots,(X_{n/2},Y_{n/2})$ to obtain an estimate $\hat{\tau}(x; \alpha/2)$ of the $\alpha/2$ quantile of the distribution of $Y$ given $X = x$, and an estimate $\hat{\tau}(x; 1-\alpha/2)$ of the $1-\alpha/2$ quantile---for instance, we might fit these models by running a quantile regression method. A straightforward way to use these estimates when confronted with a test point $X_{n+1}$ would be to output the interval  $[\hat{\tau}(X_{n+1} ; \alpha/2), \hat{\tau}(X_{n+1}; 1-\alpha/2)]$, using the estimated quantiles---but this may not give a coverage level of $1-\alpha$ if the estimates $\hat \tau$ are imperfect. Instead, the \emph{conformalized quantile regression} (CQR) method uses these initial quantile estimates to construct a conformal score, so that the resulting prediction set is an adjusted version of the initial interval. Specifically, the CQR score is given by
\begin{equation}
s(x,y) = \max\left\{ \hat{\tau}(x; \alpha/2) - y , y - \hat{\tau}(x; 1-\alpha/2)\right\}
\end{equation}
which is the signed distance of $y$ to the interval $[\hat{\tau}(x ; \alpha/2), \hat{\tau}(x; 1-\alpha)]$. With this choice of $s$, the resulting conformal prediction set takes the form
\begin{equation}
    \cC(X_{n+1}) = [\hat{\tau}(X_{n+1} ; \alpha/2) - \hat q, \hat{\tau}(X_{n+1}; 1-\alpha/2) + \hat q].
\end{equation}
That is, conformal prediction with this score function takes the prediction interval from the initial quantile estimates, and then either inflates it if $\hat{q}$ is positive or shrinks it if $\hat{q}$ is negative.

\paragraph{The high-probability score.} \index{score function!high-probability score}
The split conformal prediction algorithm can also be applied to classification problems, i.e., when the response variable takes values in a discrete set. Suppose that $\cY = \{1,\dots,K\}$ is the set of possible labels, and suppose we have an estimate $\hat{\pi}(y\mid x)$ of the probability of label $Y=y$ given features $X=x$, which was trained on the data points $(X_1,Y_1),\dots,(X_{n/2},Y_{n/2})$. The \emph{high-probability score} is then given by
\[s(x,y) = -\hat{\pi}(y\mid x).\]
It is important to note that the score is the negative of the estimated probability; this is because a conformal score is intended to return larger values when the data point $(x,y)$ appears more unlikely, i.e., when $\hat{\pi}(y\mid x)$ is small.
If we use this score for split conformal prediction, the resulting prediction set is given by
\begin{equation}
    \cC(X_{n+1}) = \{y : \hat{\pi}(y \mid X_{n+1}) \ge -\hat{q}\} \subseteq \{1,\dots, K\},
\end{equation}
which is the set of all labels $y\in\{1,\dots,K\}$ with a sufficiently high estimated probability, given the test point features $X_{n+1}$. Of course, this type of score function may also be used in the case of a continuous response $Y$, if we use an estimated conditional density in place of the estimated conditional probability.

\index{score function|)}

\bigskip

At this point, we have seen that conformal prediction can provide a marginal coverage guarantee with only weak assumptions, and can leverage the power of arbitrary predictive models---the better the predictive model, the more precise the prediction set will be. Nonetheless, there are many natural questions at this stage. Is it necessary to split the data into two parts, one for model fitting and one for the calibration of confidence intervals, as in split conformal prediction---or can the data splitting step be avoided? Which conformal score functions are optimal? Can the method be extended to cases where the data are not i.i.d.? Can distribution-free guarantees be extended to address statistical problems beyond predictive coverage? We will address these questions throughout the book as we develop conformal prediction in full generality.

\section{The conformal prediction framework in context}

This book discusses the statistical theory underlying conformal prediction and related techniques for  providing uncertainty quantification for our predictive models---but of course, this question has long been studied in the statistics literature. What distinguishes the conformal prediction framework from other methods?

As we have seen above, conformal prediction guarantees a marginal coverage property for data drawn i.i.d.\ from \emph{any} distribution. We do not need to place conditions on the distribution (such as smoothness, or, a parametric model)---because of this, conformal prediction is often described as a \emph{distribution-free} approach to inference. Moreover, we do not need to place conditions on the underlying model fitted to the data (such as assuming that $\hf$ is a consistent estimator of the true association between $Y$ and $X$), and the result is finite-sample, in the sense that it holds at any value of $n$ rather than offering only an asymptotic guarantee.

Conformal prediction is closely related to the field of nonparametric statistics, which has also aimed to provide statistical methods that can flexibly handle data distributions that do not fall within some simple parametric model. However, there are some fundamental differences between these fields. In nonparametric and semiparametric statistics, most methods and results rely on regularity conditions that, while weaker than a parametric model assumption, are nonetheless much stronger than the minimal assumptions required by conformal prediction. For example, in nonparametric statistics it is common to assume smoothness conditions on the distribution of the data. With these types of assumptions, it is often possible to provide guarantees not only for predictive inference but also for estimation (for instance, estimating the mean of $Y$ given $X$). By contrast, exchangeability is a far weaker assumption than what is usually considered in nonparametric statistics. Without regularity conditions, it is still possible to provide useful and powerful methods for predictive inference, as we have seen with the marginal coverage guarantee for split conformal, above---but, as we will see in some of the hardness results presented in Chapter~\ref{chapter:conditional} and in Part~\ref{part:beyond-predictive-coverage} of this book, other types of inference questions become more challenging or even impossible.

Conformal prediction is intimately connected with permutation testing---we will soon see that it can be formulated as the inversion of a particular permutation test. 
It is also closely connected to quantile estimation and distribution estimation: 
a natural use of quantile estimation is to give prediction intervals for test points, although such a procedure would again require assumptions on the regularity of the distribution in order to have guaranteed coverage. 
Farther afield are resampling approaches, such as the bootstrap and cross-validation.
Unlike conformal prediction, these are most commonly applied for confidence intervals on functionals of the distribution, and require regularity conditions for validity.
However, we will explore distribution-free variants of cross-validation for the purpose of predictive inference later on in the book.

\section{Scope of this book}
After this introductory chapter, the remainder of Part~\ref{part:background} is an introduction to exchangeability, with a glossary of facts and properties that will be useful for the statistical results developed later in the book. We pay special attention to permutation tests, since conformal prediction can be reframed as inverting a permutation test. These tools will be critical to many of the proofs and intuitions in the remainder of the book.

Part~\ref{part:conformal} of the book then turns to the conformal prediction framework. In particular,  
we discuss full conformal prediction, a generalization of the split conformal prediction method we have already introduced, which reveals the basic statistical logic at play. We then describe stronger properties than marginal coverage, with a mix of positive results for various methods, and hardness results that show the limits of what is possible without more assumptions. We also examine conformal prediction from a model-based perspective, to see how prior knowledge about the distribution of the data can be incorporated into the workflow of the conformal prediction framework.

Part~\ref{part:extensions-conformal} of the book focuses on a broad range of different extensions to the conformal prediction methodology, including cross-validation based methods within the conformal framework, weighted versions of conformal prediction that allow us to move beyond the i.i.d.\ setting, online versions of conformal methods that are designed for streaming data, and computational shortcuts for conformal prediction. We also briefly cover additional topics such as variants of conformal prediction that can handle broader notions of risk, and connections with selective inference, multiple testing, and model aggregation---these topics are a sample of some recent work in the field, and are suggestive of the many directions for continued study. 

Finally, in Part~\ref{part:beyond-predictive-coverage}, we depart from our focus on predictive coverage, and study the problem of distribution-free inference for a range of other questions: estimating a regression function, calibrating probability estimates, and testing conditional independence.

\section*{Bibliographic notes}
\addcontentsline{toc}{section}{\protect\numberline{}\textnormal{\hspace{-0.8cm}Bibliographic notes}}

We refer the reader to the bibliographic notes in Chapter~\ref{chapter:conformal-exchangeability} for references about the most common variants of conformal prediction and a detailed history of the field. Here, we will briefly mention some other textbooks and tutorials on conformal prediction. The first such book was~\emph{Algorithmic Learning in a Random World} by~\citet{vovk2005algorithmic}, which introduced the mathematical framework behind conformal prediction. More recent textbooks and tutorials include the works of~\citet{shafer2008tutorial},~\citet{balasubramanian2014conformal}, and~\citet{angelopoulos2023conformal}. Turning to the specific algorithms in this section, split conformal prediction is first described in~\citet{papadopoulos2002inductive}, with the residual score function as a canonical example. The same paper also introduced the scaled residual score function as an alternative; see also \citet{lei2018distribution}. Lastly, the CQR score function is due to~\citet{romano2019conformalized}, while the high-probability score is studied in~\citet{Sadinle2016LeastAS} (see also earlier work by \citet{papadopoulos2008inductive}, which studies a related score).

\section*{Exercises}
\addcontentsline{toc}{section}{\protect\numberline{}\textnormal{\hspace{-0.8cm}Exercises}}
\begin{enumerate}[font=\bfseries, label={\thechapter.\arabic*}, labelsep=1em, itemsep=1em]
\item Consider split conformal prediction with $\cY = \R$. Suppose we are given an initial model $\hf :\cX \to \R$. Construct a score function $s(x,y)$ that would yield intervals that are twice as wide above $\hf$ as below $\hf$, i.e., intervals of the form $\cC(x) = [\hf (x)- \gamma, \hf(x) + 2 \gamma]$ for some $\gamma$.
\item Consider split conformal prediction in the setting of a multivariate response, $\cY=\R^d$. Suppose we are given an initial model $\hf:\cX\to\R^d$, which returns predictions $\hf(x) = (\hf(x)_1,\dots,\hf(x)_d)$. Construct a score function $s(x,y)$ that would return prediction sets of the form
\[\cC(x) = \big[\hf(x)_1 - c, \hf(x)_1 + c\big] \times\dots \times \big[\hf(x)_d - c, \hf(x)_d + c\big] \]
for some $c$.
\item Suppose we have a regression problem with $\cY = \R$ and two models $\hf_1 : \cX \to \R$ and $\hf_2 : \cX \to \R$. Consider split conformal prediction with the score function $s(x,y) = \min\{|y - \hf_1(x)|, |y - \hf_2 (x)|\}$. Derive an explicit expression for the set $\cC(X_{n+1})$ in terms of $\hf_1$, $\hf_2$, and $\hat q$. What is the shape of the prediction set when $\hat q$ is small, or when $\hat q$ is large? Briefly explain why this might be attractive if we expect the distribution of $Y$ given $X = x$ to be bimodal for some values of $x$.
\end{enumerate}

\chapter{Exchangeability and Permutations}
\label{chapter:exchangeability}

\index{exchangeability|(}
In this chapter, we provide an introduction to the idea of exchangeability, laying the core mathematical foundation of conformal prediction and many related methodologies. 
Exchangeability is a property
of a sequence of random variables---informally, it expresses the idea that the sequence is equally likely to appear
in any order.
A formal definition is as follows:
\begin{definition}[Exchangeability]
\label{def:exchangeability}
Let $Z_1,\dots,Z_n\in\cZ$ be random variables with a joint distribution. We say that the random vector $(Z_1, \ldots, Z_n)$ is \emph{exchangeable} if, for every permutation $\sigma\in\cS_n$,
\[(Z_1,\dots,Z_n)\eqd (Z_{\sigma(1)},\dots,Z_{\sigma(n)}),\]
where $\eqd$ denotes equality in distribution, and $\cS_n$ is the set of all permutations  on $[n]:=\{1,\dots,n\}$.
\end{definition}
Similarly, we say that an infinite sequence $Z_1,Z_2,\dots\in\cZ$ is exchangeable if $(Z_1,\dots,Z_n)$ is exchangeable
for every $n\geq 1$.

The elements of an exchangeable sequence are identically distributed, but not necessarily independent.
Exchangeability constrains the dependence structure so that all permutations are equally likely. Throughout this book, we might interchangeably say that a random vector $(Z_1,\dots,Z_n)$ is exchangeable, or that the random variables $Z_1,\dots,Z_n$ are exchangeable.

Exchangeability can arise in a broad range of scenarios. In particular,
exchangeability of a sequence $Z_1,\dots,Z_n$ arises in the following important special cases:
\begin{itemize}
\item $Z_1,\dots,Z_n$ are sampled uniformly without replacement from a finite set $\{z_1,\dots,z_N\}\subseteq\cZ$.
\item $Z_1,\dots,Z_n$ are drawn i.i.d.\ from a distribution $P$ on $\cZ$.
\end{itemize}
However, these common scenarios are far from exhaustive. As an intuitive example, consider the distribution on $(Z_1,Z_2)\in\{0,1\}^2$ given by 
\[\frac{1}{8}\delta_{(0,0)} + \frac{3}{8}\delta_{(0,1)} + \frac{3}{8}\delta_{(1,0)} + \frac{1}{8}\delta_{(1,1)}.\]
(Here and throughout the book, we will use the notation $\delta_z$ to denote the \emph{point mass} at a value $z$, i.e., the probability distribution that places probability $1$ on the value $z$. This is sometimes referred to as the \emph{Dirac delta function}.)
In other words, the joint distribution of $(Z_1,Z_2)$ is defined by
the probability mass function
\[p(0,0) = p(1,1) = \frac{1}{8}, \quad p(0,1) = p(1,0) = \frac{3}{8}.\]
This joint distribution is exchangeable, but cannot be expressed either via sampling without replacement or via i.i.d.\ sampling.
It instead illustrates a different way that exchangeability can arise: any mixture of exchangeable distributions is itself
an exchangeable distribution. Indeed, the above example can be derived as a mixture, where with probability $\frac{1}{2}$ we sample $Z_1,Z_2$ uniformly without replacement from the set $\{0,1\}$, and with probability $\frac{1}{2}$ we draw $Z_1,Z_2$ i.i.d.\ from the  $\textnormal{Bernoulli}(0.5)$ distribution.

As a technical note, here and throughout the remainder of the book, wherever needed we will assume mild regularity conditions on the underlying measure spaces, without comment: for instance, the assumption that $\sigma$-algebras are countably generated, and the existence of regular conditional probabilities (i.e., existence of measurable functions such as $x\mapsto \P(Y\in A\mid X=x)$), both of which hold in most common settings, such as $\R^d$ or any standard Borel space.

\section{Alternative characterizations of exchangeability}
\label{sec:exchangeability-alternative-characterizations}
Exchangeability can be formally described in a number of different ways. Here, we give several characterizations and properties of exchangeability of a random vector $(Z_1,\dots,Z_n)$, to help build intuition.

\paragraph{Symmetry of the joint density.}
Exchangeability has a simple characterization if the random vector $(Z_1,\dots,Z_n)$ is either discrete, or has a joint density.
First, supposing $\cZ$ is a countable space so that the $Z_i$'s are discrete, let $p:\cZ^n\rightarrow[0,1]$ be the probability
mass function for the joint distribution of $(Z_1,\dots,Z_n)$. Then this joint distribution is exchangeable if and only if 
\[p(z_1,\dots,z_n) = p(z_{\sigma(1)},\dots,z_{\sigma(n)})\textnormal{ for all $z_1,\dots,z_n\in\cZ$ and for all $\sigma\in\cS_n$}.\]
Analogously, if $\cZ = \R$ and the random vector $(Z_1,\dots,Z_n)$ has a joint density $f$ (with respect to Lebesgue measure on $\R^n$), then this joint distribution is exchangeable if and only if 
\begin{equation}
    f(z_1,\dots,z_n) = f(z_{\sigma(1)},\dots,z_{\sigma(n)})\textnormal{ for almost every $(z_1,\dots,z_n)\in\R^n$ and for all $\sigma\in\cS_n$}.
\end{equation}

\paragraph{Conditioning on the order statistics.}
For this next interpretation, we will consider the special case of real-valued random variables, $\cZ=\R$. In this setting, we define the order statistics $Z_{(1)}\leq \dots \leq Z_{(n)}$ as the sorted values of the vector $(Z_1,\dots,Z_n)$. In the simple setting where these $n$ values are distinct almost surely, exchangeability implies that, conditioning on the order statistics, the random vector $(Z_1,\dots,Z_n)$ is equally likely to be any one of the $n!$ unique permutations of the order statistics. More generally, without assuming that the $Z_i$'s are necessarily distinct, we can calculate the distribution of $(Z_1,\dots,Z_n)$, conditional on the order statistics, as
\begin{equation}
    \label{eq:exchangeability-uniform-ranks}
    (Z_1,\dots,Z_n)\mid (Z_{(1)},\dots,Z_{(n)}) \sim \frac{1}{n!}\sum_{\sigma\in\cS_n}\delta_{(Z_{(\sigma(1))},\dots,Z_{(\sigma(n))})},
\end{equation}
placing mass $\frac{1}{n!}$ on each of the $n!$ (potentially non-unique) possible orderings of the order statistics $(Z_{(1)},\dots,Z_{(n)})$.
To put it more simply, we can say that conditional on the unordered collection of values in the sequence, the order
in which they appear is simply a random shuffle.

We next highlight an important consequence of this equivalent characterization: under exchangeability of $Z_1,\dots,Z_n$, it must hold that
\begin{equation}\label{eqn:empirical-distr-order-statistic}Z_i \mid (Z_{(1)},\dots,Z_{(n)}) \sim \frac{1}{n}\sum_{j=1}^n \delta_{Z_{(j)}}\end{equation}
(i.e., each individual entry $Z_i$ is equally likely to be any one of the $n$ order statistics), and consequently,
\begin{equation}
    \label{eq:uniform-cdf-order-statistic}
    \P(Z_i \leq Z_{(k)}) \geq k/n,
\end{equation}
for each index $i\in[n]$ and each rank $k\in[n]$ (see Fact~\ref{fact:exchangeable-properties} below). Moreover, if the values $Z_1,\dots,Z_n$ are distinct almost surely, then this becomes an equality.

\paragraph{Conditioning on the empirical distribution.}
In the real-valued setting discussed above, we saw that each entry $Z_i$ can be viewed as a random draw from the values $Z_{(1)},\dots,Z_{(n)}$. In fact, this intuition can be extended to the general setting, beyond the case $\cZ=\R$. Define 
\[\widehat{P}_n = \frac{1}{n}\sum_{i=1}^n \delta_{Z_i},\]
which is the empirical distribution of the random vector $(Z_1,\dots,Z_n)$.
The following proposition tells us an important implication of exchangeability: essentially, each $Z_i$ is a draw from this empirical distribution $\widehat{P}_n$. This is simply an extension of the result~\eqref{eqn:empirical-distr-order-statistic} to the case of a general space $\cZ$.
\begin{proposition}[Exchangeability and the empirical distribution]\label{prop:empirical-distrib-exch}
    Let $(Z_1,\dots,Z_n)\in\cZ^{n}$ be an exchangeable random vector, and let $\widehat{P}_{n}$ be the empirical distribution of this vector. Then for all $i \in [n]$,
    \[Z_{i}\mid \widehat{P}_{n} \ \sim \widehat{P}_{n},\]
    i.e., if we condition on $\widehat{P}_{n}$, then $\widehat{P}_{n}$ is itself the conditional distribution of $Z_{i}$.
\end{proposition}
\begin{proof}[Proof of Proposition~\ref{prop:empirical-distrib-exch}]
Since $\widehat{P}_n$ is a symmetric function of the random variables $Z_1,\dots,Z_n$, by Lemma~\ref{lem:conditional_exchangeability} below it holds almost surely that $Z_1,\dots,Z_n$ are exchangeable conditional on $\widehat{P}_n$, and consequently, it holds almost surely that
\begin{equation}\label{eqn:exch_for_empir_distrib}\P(Z_i\in A \mid \widehat{P}_{n}) = \P(Z_j\in A \mid \widehat{P}_{n}),\end{equation}
for each $j\in[n]$ and every (measurable) $A\subseteq\cZ$. Assuming this holds, we then have
\begin{multline*}
\P(Z_i\in A \mid \widehat{P}_{n})
= \frac{1}{n}\sum_{j=1}^{n}\P(Z_j\in A \mid \widehat{P}_{n}) = \E\left[\frac{1}{n}\sum_{j=1}^{n}\ind{Z_j\in A}\, \middle| \, \widehat{P}_{n}\right]\\ = \E\left[\widehat{P}_{n}(A) \mid \widehat{P}_{n}\right] = \widehat{P}_{n}(A), \end{multline*}
which proves the desired claim.
\end{proof}
The proof of Proposition~\ref{prop:empirical-distrib-exch} relies
on the fact that $Z_1,\dots,Z_n$ are exchangeable even after conditioning on $\widehat{P}_n$. In fact, this is a special case of the following lemma, which verifies that exchangeability continues to hold after conditioning on any symmetric function of $Z_1,\dots,Z_n$.
\begin{lemma}[Conditional exchangeability given a symmetric function]\label{lem:conditional_exchangeability}
    Let $Z_1,\dots,Z_n\in\cZ$ be exchangeable, and let $f:\cZ^n\to\cW$ be a symmetric function, i.e., $f(z_1,\dots,z_n) = f(z_{\sigma(1)},\dots,z_{\sigma(n)})$ for all $z_1,\dots,z_n\in\cZ$ and all $\sigma\in\cS_n$. Then $(Z_1,\dots,Z_n)$ is conditionally exchangeable given $f(Z_1,\dots,Z_n)$, in the sense that the conditional distribution
    \[(Z_1,\dots,Z_n) \mid f(Z_1,\dots,Z_n)\]
    is, almost surely, an exchangeable distribution.
\end{lemma} \index{exchangeability!conditional}
\begin{proof}[Proof of Lemma~\ref{lem:conditional_exchangeability}]
    By definition of exchangeability, it suffices to verify that for any $\sigma\in\cS_n$ and any measurable set $A$, the following statement holds almost surely:
    \[\P\big((Z_1,\dots,Z_n)\in A \mid f(Z_1,\dots,Z_n)\big) = \P\big((Z_{\sigma(1)},\dots,Z_{\sigma(n)})\in A \mid f(Z_1,\dots,Z_n)\big).\]
    Equivalently, we need to show that
    \[\P\big((Z_1,\dots,Z_n)\in A , f(Z_1,\dots,Z_n) \in B \big) = \P\big((Z_{\sigma(1)},\dots,Z_{\sigma(n)})\in A, f(Z_1,\dots,Z_n) \in B\big)\]
    for all measurable $A\subseteq\cZ^n, B\subseteq\cW$. This holds because
    \begin{multline*}
        \P\big((Z_1,\dots,Z_n)\in A , f(Z_1,\dots,Z_n) \in B \big)
        \\=\P\big((Z_{\sigma(1)},\dots,Z_{\sigma(n)})\in A , f(Z_{\sigma(1)},\dots,Z_{\sigma(n)}) \in B\big)\\
        =\P\big((Z_{\sigma(1)},\dots,Z_{\sigma(n)})\in A , f(Z_1,\dots,Z_n) \in B\big),
    \end{multline*}
    where the first step holds since $(Z_1,\dots,Z_n)$ is exchangeable, while the second step holds by symmetry of $f$.    
\end{proof}

\index{exchangeability|)}

\section{Permutation tests}
\label{sec:perm_test}
\index{permutation test|(}
\index{exchangeability!testing}

Permutation tests are used in statistics for a wide range of different inference tasks. 
In fact, permutation
tests can be viewed as testing the null hypothesis of exchangeability.
We turn to this next, and then 
examine two concrete examples of commonly used permutation tests within the framework of exchangeability.

Let $\cP$ be the set of all distributions on $\cZ^n$,
and let $\cP_{\textnormal{exch}}\subseteq\cP$ be the subset of distributions for which 
exchangeability is satisfied. Consider a random vector $(Z_1,\dots,Z_n)$ drawn from some joint distribution $P$.
We would like to perform a hypothesis test of 
\[H_0: \ P\in\cP_{\textnormal{exch}} \textnormal{\quad versus \quad} H_1: \ P\in\cP\backslash\cP_{\textnormal{exch}}.\]
Before observing the data, we fix any function $T:\cZ^n\rightarrow\R$, with the intuition that a large value of our test statistic $T(Z_1,\dots,Z_n)$
will indicate evidence against exchangeability. Then we define the quantity
\begin{equation}
\label{eq:perm-pvalue}
p = \frac{\sum_{\sigma\in\cS_n}\ind{T(Z_{\sigma(1)},\dots,Z_{\sigma(n)}) \geq T(Z_1,\dots,Z_n)}}{n!}
\end{equation}
which compares the observed value of the test statistic, $T(Z_1,\dots,Z_n)$, against all possible values obtained via permutations of the data.
(Note that the identity permutation, $\sigma = \textnormal{Id}$, is one of the $n!$ many permutations included in the sum, and thus
it is not possible for $p$ to be smaller than $\frac{1}{n!}$.)
The following well-known result shows that the quantity defined in~\eqref{eq:perm-pvalue} is a valid p-value for testing the null hypothesis of exchangeability. 
\begin{theorem}[Validity of the permutation test]
\label{thm:perm-test}
For any function $T:\cZ^n\rightarrow\R$, the p-value $p$ defined in~\eqref{eq:perm-pvalue} satisfies $\P_P(p\leq \tau)\leq \tau$ for all $\tau\in[0,1]$ and all $P\in\cP_{\textnormal{exch}}$.
\end{theorem}

\begin{figure}[t]
    \centering
    \includegraphics[width=\textwidth]{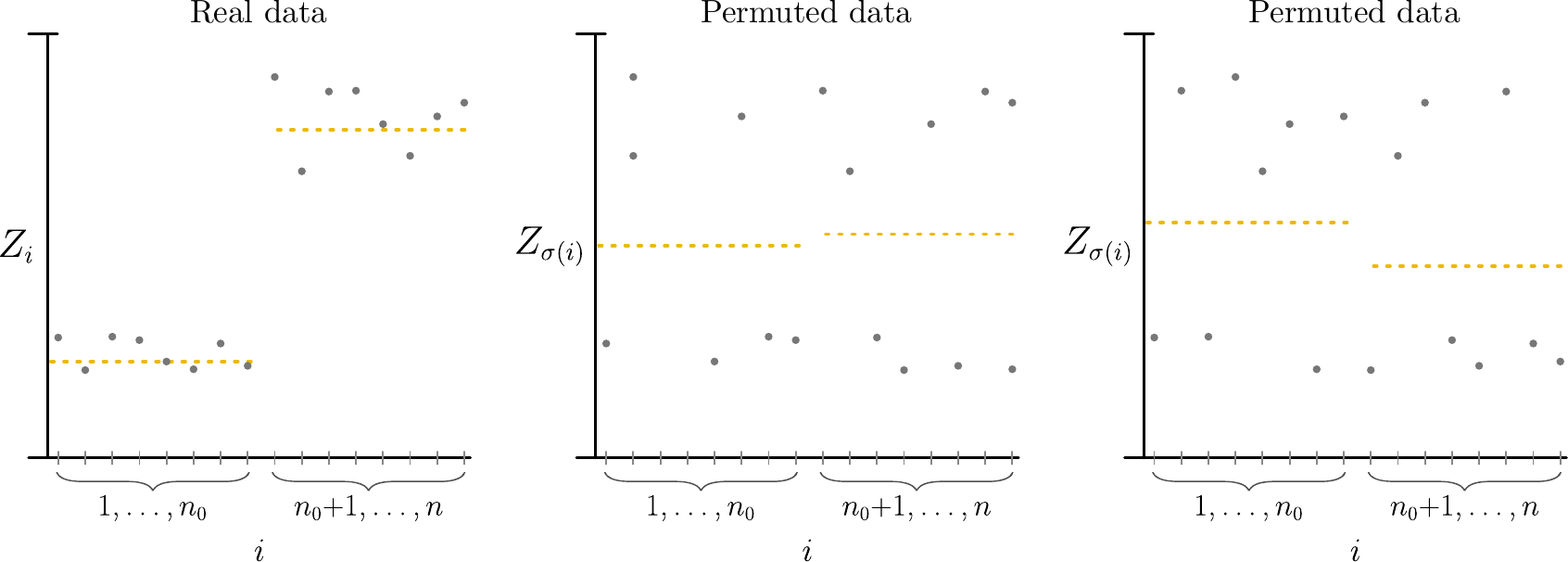}
    \caption{\textbf{Illustration of a permutation test for the equality of two real-valued distributions}, where the test statistic used is the difference in means between two groups of data points, as in~\eqref{eqn:perm_test_diff_means}. In each plot, these two group means are shown as two dashed lines. In the left plot, we show the values computed on the real ordering of the data $Z$. The middle and right plots show the values for two typical permutations $Z_\sigma$. The difference in means on the real data is far more extreme than on the permuted data, indicating evidence against the null hypothesis of exchangeability.}
    \commentAlt{The left plot shows a dataset where the first half of the values are low and the second half are high. The middle and right plots show different permutations of this data. In each plot, dashed lines indicate the sample means within the two groups.}
    \label{fig:perm-tests}
\end{figure}

In many settings, it is common to avoid the computational burden of computing all $n!$ permutations by instead sampling a smaller number $M$ of permutations $\sigma_1,\dots,\sigma_M\in\cS_n$
uniformly at random, to obtain the p-value \index{permutation test!p-value}
\begin{equation}
    \label{eq:perm-pval-random}
    p = \frac{1+ \sum_{m=1}^M\ind{T(Z_{\sigma_m(1)},\dots,Z_{\sigma_m(n)}) \geq T(Z_1,\dots,Z_n)}}{1+M}.
\end{equation}
This p-value is again valid against the null hypothesis of exchangeability:
\begin{theorem}[Validity of the permutation test with random permutations]\label{thm:perm-test-random}
    For any function $T:\cZ^n\rightarrow\R$, the p-value $p$ defined in~\eqref{eq:perm-pval-random} satisfies $\P_P(p\leq \tau)\leq \tau$ for all $\tau\in[0,1]$ and all $P\in\cP_{\textnormal{exch}}$, where the probability is now taken with respect to both the random draw of $(Z_1,\dots,Z_n)\sim P$, and the permutations $\sigma_1,\dots,\sigma_M$ sampled uniformly at random (with replacement) from $\cS_n$.
\end{theorem}
The `$+1$' term appearing in the numerator and denominator of the p-value $p$ constructed in~\eqref{eq:perm-pval-random} is necessary for obtaining this validity result---indeed, without this correction, the event $p=0$ could have nonzero probability under the null hypothesis.

\subsection{Examples}
\label{sec:permtest-examples}

To apply the permutation test, we need to specify a choice of the test statistic $T$.
If the statistic captures the deviations from exchangeability that we expect may occur, then it will lead to a powerful test.
We illustrate this with several common examples in the case of real-valued data, $\cZ = \R$---namely, testing for equality of distributions, and testing for outliers. We will study another common application of permutation tests, testing independence between two random variables, in Chapter~\ref{chapter:conditional-independence-testing}.

\paragraph{Testing equality of distributions.} 
Suppose that we have two independent samples from two potentially different distributions,
with $n_0$ draws from $P_0$ and $n_1=n-n_0$ draws from $P_1$.
Without loss of generality, we can take $Z_1,\dots,Z_{n_0}\iidsim P_0$ and $Z_{n_0+1},\dots,Z_n\iidsim P_1$.
If $P_0=P_1$, then the $Z_i$'s are i.i.d.\ from a single shared distribution $P_0=P_1$, and therefore exchangeability holds.
If we conjecture that any potential difference between $P_0$ and $P_1$ would likely lead to a difference of means, we might choose
the test statistic
\begin{equation}\label{eqn:perm_test_diff_means}T(z_1,\dots,z_n) = \left|\frac{1}{n_0}\sum_{i=1}^{n_0} z_i - \frac{1}{n_1}\sum_{i=n_0+1}^n z_i\right|,\end{equation}
the difference in the sample means. Alternatively, we might make a choice that is more agnostic to the type of difference
between the two distributions: the Kolmogorov--Smirnov statistic,
\[T(z_1,\dots,z_n) = \sup_{v\in\R} \left|\frac{1}{n_0}\sum_{i=1}^{n_0} \ind{z_i\leq v} - \frac{1}{n_1}\sum_{i=n_0+1}^n \ind{z_i\leq v}\right|,\]
which measures the maximum difference between the two empirical cumulative distribution functions (CDFs).

\paragraph{Testing if a new data point is an outlier.}
Next, suppose that we would like to test whether a particular data point---say, the last data point $Z_n$---is an outlier
relative to the rest of the sequence. 
In fact, as we will see later on, this use of the permutation test is central to the development of conformal prediction.

For example, we might conjecture that $Z_n$ is more likely to be unusually large
relative to the other $Z_i$'s. In this case, we could consider the test statistic
\[T(z_1,\dots,z_n) = \sum_{i=1}^n \ind{z_n > z_i},\]
which is large when the rank of the last value is large among the rest of the list.
For this particular test statistic, the permutation test p-value can be simplified.
Observe that
\[T(Z_{\sigma(1)},\dots,Z_{\sigma(n)}) = \sum_{i=1}^n \ind{Z_{\sigma(n)} > Z_{\sigma(i)}}= \sum_{i=1}^n \ind{Z_{\sigma(n)} > Z_i},\]
which simply captures the position of $Z_{\sigma(n)}$ relative to the original (unpermuted) sequence $Z_1,\dots,Z_n$.
Examining this quantity, we can then see that $T(Z_{\sigma(1)},\dots,Z_{\sigma(n)}) \geq T(Z_1,\dots,Z_n)$ if and only if $Z_{\sigma(n)}\geq Z_n$, 
and therefore,
the p-value can be simplified as 
\begin{multline*}p = \frac{\sum_{\sigma\in\cS_n}\ind{T(Z_{\sigma(1)},\dots,Z_{\sigma(n)}) \geq T(Z_1,\dots,Z_n)}}{n!}=
 \frac{1}{n!}\sum_{\sigma\in\cS_n}\ind{Z_{\sigma(n)} \geq Z_n} \\=
 \frac{1}{n!}\sum_{i=1}^n\sum_{\sigma\in\cS_n,\sigma(n)=i}\ind{Z_i \geq Z_n} =\frac{1}{n}\sum_{i=1}^n \ind{Z_i \geq Z_n}.
\end{multline*}
Here the last step holds since, for each $i\in[n]$, there are exactly $(n-1)! = \frac{n!}{n}$ permutations $\sigma\in\cS_n$ for which $\sigma(n)=i$. 

In fact, this example can be derived in a simpler way, without the terminology of permutation tests. By definition of $p$, we can verify that, for any $\tau\in[0,1)$,
\[p = \frac{1}{n}\sum_{i=1}^n \ind{Z_i \geq Z_n} \leq \tau \quad \Longleftrightarrow \quad  Z_n > Z_{(k)}\textnormal{ for $k=\lceil (1-\tau) n\rceil$.}\]
By~\eqref{eq:uniform-cdf-order-statistic} we know that exchangeability of $Z_1,\dots,Z_n$ implies that $\P(Z_n > Z_{(k)}) \leq 1 - k/n \leq \tau$, which directly verifies the validity of the p-value $p$. We state this result formally in the following corollary:
\begin{corollary}\label{cor:perm_test_pval}
Let $Z_1,\dots,Z_n\in\R$ be exchangeable. Then 
$p = \frac{\sum_{i=1}^n \ind{Z_i \geq Z_n}}{n}$
satisfies $\P(p\leq \tau)\leq \tau$ for all $\tau\in[0,1]$.
\end{corollary}

\section{Proving validity of permutation tests}
\label{sec:perm_test_proof}

We next turn to building a theoretical understanding of permutation tests, within the framework of exchangeability. While the intuition behind permutation tests is very natural,
here we will dive into the details that underlie their validity.
The proofs of Theorem~\ref{thm:perm-test} and Theorem~\ref{thm:perm-test-random} are similar, so we only give the proof of the first.

\begin{proof}[Proof of Theorem~\ref{thm:perm-test}] 

\textbf{Step 1: a CDF inequality.} 
First, for any $z\in\cZ^n$, write $z_\sigma = (z_{\sigma(1)},\dots,z_{\sigma(n)})$ to denote the permuted vector for any permutation $\sigma\in\cS_n$. We
define \[F(v;z) = \frac{\sum_{\sigma\in\cS_n}\ind{-T(z_\sigma) \leq v}}{n!}.\]
We can observe that $F(\cdot;z)$
is the cumulative distribution function (CDF) for the distribution of the quantity $-T(z_\sigma)$, when $\sigma\in\cS_n$ is drawn uniformly at random (while $z$ is treated as fixed). 
We therefore have
\begin{equation}
    \label{eq:permutation-pval-CDF-inequality}
    \P_\sigma\left(F(-T(z_\sigma);z)\leq \tau\right)\leq \tau,
\end{equation}
where the probability is taken with respect to $\sigma\in\cS_n$ drawn uniformly at random,
which is implied directly from the following basic property of CDFs: \begin{equation}\label{eq:CDF_basic_fact}
\textnormal{If the random variable $X$ has CDF $F$, then $\P(F(X)\leq \tau)\leq \tau$ for all $\tau\in[0,1]$}.
\end{equation}

\textbf{Step 2: using exchangeability.} 
Now, we incorporate the exchangeability assumption on $Z=(Z_1,\dots,Z_n)$. For any \emph{fixed} permutation $\sigma\in\cS_n$, $Z 
\eqd Z_\sigma$ by exchangeability; therefore it also holds that $Z 
\eqd Z_\sigma$ when $\sigma\in\cS_n$ is drawn uniformly at random (independently of $Z$).

Next, we observe that the p-value $p$ defined in~\eqref{eq:perm-pvalue} is equal to $p = F(-T(Z); Z)$.
Therefore, 
\begin{equation}
    \P(p\leq \tau) = \P(F(-T(Z); Z) \leq \tau) 
= \P(F(-T(Z_{\sigma}); Z_{\sigma})\leq \tau)= \P(F(-T(Z_\sigma); Z)\leq \tau),
\end{equation}
where the last two probabilities are calculated with respect to the distribution of both $Z$ and the randomly drawn $\sigma$. Here the second equality holds since $Z 
\eqd Z_\sigma$, while the last step holds since $F(v;z) = F(v; z_\sigma)$ for any $v$, $z$, and $\sigma$, by construction.
Finally, we know that $\P(F(-T(Z_\sigma); Z )\leq\tau\mid Z)\leq \tau$, almost surely, by Step 1, which implies that $\P(F(-T(Z_\sigma); Z)\leq \tau) \leq \tau$ by the tower law.
\end{proof}

The key tool in this proof is the CDF of the \emph{negative} values of the test statistic, i.e., $-T(Z)$ (and its permuted version, $-T(Z_\sigma)$). The reason for taking the negative is simply that the permutation test in~\eqref{eq:perm-pvalue} has a small p-value when $T(Z)$ is sufficiently large, while the CDF measures the probability of observing a value that is sufficiently small; by taking the negative, we can express the p-value as a CDF.
\index{permutation test|)}

\section{Appendix: order statistics, quantiles, and CDFs}
\label{sec:permutations_key_ingredients}

In this section, we will give some additional background on some technical details in order to build a strong foundation for theoretical results later in this work. This section will contain definitions and key facts about quantities such as quantiles and CDFs, which provide  some of the basic ingredients for studying exchangeability and developing conformal prediction methods.

The structure of the section is as follows. In Section~\ref{sec:deterministic-properties-quantiles-cdfs}, we state properties of order statistics, quantiles, and CDFs of lists that hold \emph{deterministically}---i.e., for any list of numbers.
In Section~\ref{sec:exchangeability-properties-quantiles-cdfs}, we state implications of these properties under the assumption of exchangeability.
Finally, in Section~\ref{sec:distributional-properties-quantiles-cdfs}, we look at the quantile and CDF function again from a distributional perspective, as opposed to operating over finite lists.

\subsection{Deterministic properties}
\label{sec:deterministic-properties-quantiles-cdfs}

\textbf{Definitions.} First, though we have referred to order statistics, CDFs, and quantiles informally throughout this chapter, we now define them formally for a list of real numbers, $z = (z_1, \ldots, z_n)$. We emphasize that, for the moment, we are treating the $z_i$'s as fixed values rather than random variables.
\begin{definition}[Order statistics of a finite list]
    \label{def:order-statistics}
    Let $k \in [n]$. Then the $k$th order statistic of $z\in\R^n$, written as $z_{(k)}$, is defined as
    \begin{equation}
        z_{(k)} = \inf\left\{ v: \sum\limits_{i=1}^n \ind{z_i \leq v } \geq k\right\}.
    \end{equation}
\end{definition} \index{order statistics}
The order statistics $z_{(1)}\leq \dots \leq z_{(n)}$ simply rearrange the values of $z$ into nondecreasing order. The $k$th order statistic of $z$ is always uniquely defined, although its value may occur more than once in the list in the case of ties. For example, in the vector $z=(3, 2, 1, 2)$, the order statistics are $z_{(1)}=1$, $z_{(2)} = z_{(3)} = 2$, and $z_{(4)} = 3$.

\begin{definition}[Cumulative distribution function (CDF) of a finite list]
    \label{def:CDF}
    The CDF of $z\in\R^n$ is the function $\widehat{F}_z:\R\rightarrow[0,1]$ defined as 
    \begin{equation}
        \widehat{F}_z(v) = \frac{1}{n}\sum_{i=1}^n \ind{z_i \leq v}.
    \end{equation}
\end{definition} \index{cumulative distribution function (CDF)}
The CDF evaluated at $v$ is the fraction of data points $z_1,\dots,z_n$ that lie at or below the value $v$. For convenience, we also define $\hat{F}_z(-\infty) = 0$ and $\hat{F}_z(+\infty) = 1$.

\begin{definition}[Quantile of a finite list]
    \label{def:quantile}
    For any $\tau\in[0,1]$, the $\tau$-quantile of $z\in\R^n$ is defined as 
    \begin{equation}
        \quantile\left(z; \tau \right) = \inf\left\{ v : \widehat{F}_z(v) \geq \tau \right\}.
    \end{equation}
\end{definition} \index{quantile}
In other words, the $\tau$-quantile is the smallest value $v$ such that a fraction $\tau$ of the data points lie at or below the value $v$. (Note that, at $\tau=0$, the quantile is given by $-\infty$.)

\textbf{Conversions between order statistics, quantiles, and CDFs.}
Order statistics, quantiles, and CDFs are intimately related. 
One can see them intuitively as different parameterizations of the exact same core concept: counting the number of data points falling below some value.
As we will see below, the order statistics and quantiles of a finite list are essentially equivalent.
Furthermore, the CDF is approximately the inverse of the quantile function when $z$ does not contain repeated values (and when it does, it provides a bound on the inverse).
We describe these conversions here.

\begin{fact}[Conversion between order statistics and quantiles.]
    \label{fact:conversion-order-stats-quantiles}
   For any $z\in\R^n$, for all $\tau \in (0,1]$,
    \begin{equation}
        z_{(\lceil \tau n \rceil)} = \quantile\left( z ; \tau \right).
    \end{equation}
\end{fact}
\begin{fact}[Conversion between order statistics and CDFs.]
    \label{fact:conversion-order-statistics-cdfs}
For any $z\in\R^n$, for all $k\in[n]$, 
\[\widehat{F}_z(z_{(k)}) \geq \frac{k}{n},\]
with equality in the case that all elements of $z$ are distinct.
\end{fact}
\begin{fact}[Conversion between quantiles and CDFs.]
    \label{fact:conversion-quantiles-cdfs}
    The following equivalences hold for any $z\in\R^n$:
    \begin{enumerate}[label=(\roman*)]
    \item\label{fact:conversion-quantiles-cdfs_part1} $\widehat{F}_z(v) = \sup\{ \tau : \quantile\left(z ; \tau \right) \leq v \}$ for all $v\in\R$;
    \item\label{fact:conversion-quantiles-cdfs_part2}  $\quantile\left(z ; \widehat{F}_z\left(  v \right) \right) \leq v$ for all $v\in\R$;
    \item\label{fact:conversion-quantiles-cdfs_part3}  $\widehat{F}_z\left( \quantile\left(z ; \tau \right) \right) \geq \tau$ for all $\tau\in[0,1]$;
    \end{enumerate}
    and furthermore, as a special case of~\ref{fact:conversion-quantiles-cdfs_part3}, 
    \begin{enumerate}[label=(\roman*)]
      \setcounter{enumi}{3}
      \item\label{fact:conversion-quantiles-cdfs_part4}   If all elements of $z$ are distinct, $\widehat{F}_z\left( \quantile\left(z ; \tau \right) \right) = \frac{\lceil \tau n \rceil}{n}$ for all $\tau\in[0,1]$.
    \end{enumerate}    
\end{fact}
Figure~\ref{fig:cdf-ties} illustrates parts~\ref{fact:conversion-quantiles-cdfs_part3} and~\ref{fact:conversion-quantiles-cdfs_part4} of this last fact. When there are no ties between values of $z$, Fact~\ref{fact:conversion-quantiles-cdfs}\ref{fact:conversion-quantiles-cdfs_part4} ensures that the CDF is approximately equal to $\tau$ at the $\tau$-quantile---namely, $\tau\leq \hat{F}_z(\quantile(z;\tau)) < \tau + 1/n$. In contrast, in the presence of ties, the value of the CDF at the $\tau$-quantile may be substantially larger than $\tau$. 

\begin{figure}[t]
    \centering
    \includegraphics[width=0.55\textwidth]{\diagramspath quantiles-cdfs.pdf}
    \caption{\textbf{An illustration of two quantiles chosen on the CDF.} The figure illustrates the empirical CDF of the vector $z=(1,1,2,3,4)$, and the calculation of $\quantile(z;\tau)$, at $\tau = 0.5$ and $\tau=0.1$. This vector has quantiles $\quantile(z;0.5) = 2$ and $\quantile(z;0.1) = 1$. We can see that $\widehat{F}_z(2) = 0.6$ (which is slightly larger than $\tau = 0.5$, due to discreteness), and $\widehat{F}_z(1) = 0.4$ (which is much larger than $\tau = 0.1$, due to the fact that the random vector has a tie at the value $1$).}
    \commentAlt{A plot of the empirical CDF of $z=(1,1,2,3,4)$. This is a step function, with a larger vertical jump at the value $1$ and a smaller vertical jump at values $2,3,4$. Arrows indicate the calculation of the quantile level $\tau=0.1$ and $\tau=0.5$.}
    \label{fig:cdf-ties}
\end{figure}

\textbf{Translations and transformations.}
What happens to the order statistics, quantiles, and CDFs when we take transformations of $z$? 
Here we show what happens when we take monotone transformations and when we permute.
We begin with monotone transformations.
\begin{fact}[Order statistics, quantiles, and CDFs under transformations]
    \label{fact:monotone-invariance-quantiles}
    Let $f:\R\to\R$ be a monotone nondecreasing function.
    Furthermore, for $z\in\R^n$, let $f(z)\in\R^n$ denote the elementwise application of $f$ to $z$.
    Then,
    \begin{enumerate}[label=(\roman*)]
        \item\label{fact:monotone-invariance-quantiles_part1} For all $k\in[n]$, $f(z)_{(k)} = f(z_{(k)})$.
        \item\label{fact:monotone-invariance-quantiles_part2} For all $\tau\in(0,1]$, $\quantile(f(z) ; \tau) = f(\quantile(z ; \tau))$.
        \item\label{fact:monotone-invariance-quantiles_part3} If additionally $f$ is a strictly increasing function, then for all $v\in\R$, $\widehat{F}_{f(z)}(f(v)) = \widehat{F}_z(v)$.
    \end{enumerate}
\end{fact}

Next, we consider permutations of $z$.
From the perspective of this book, this is the most important property of quantiles and CDFs.
\begin{fact}[Properties of order statistics, quantiles, and CDFs under permutations]
    \label{fact:permutation-invariance-quantiles}
    For any vector $z\in\R^n$ and any permutation $\sigma$ on $[n]$, let $z_{\sigma}=(z_{\sigma(1)}, \ldots, z_{\sigma(n)})$, i.e., the entries of $z$ are permuted according to $\sigma$.
    We have that
    \begin{enumerate}[label=(\roman*)]
        \item\label{fact:permutation-invariance-quantiles_part1} For all $k\in[n]$, $(z_{\sigma})_{(k)} = z_{(k)}$.
        \item\label{fact:permutation-invariance-quantiles_part2} For all $\tau\in[0,1]$, $\quantile(z_{\sigma} ; \tau) = \quantile\left(z ; \tau\right)$.
        \item\label{fact:permutation-invariance-quantiles_part3} For all $v\in\R$, $\widehat{F}_{z_\sigma}(v) = \widehat{F}_z(v)$.
    \end{enumerate}
\end{fact}
In other words, this fact tells us that the order statistics, quantiles, and CDF of a vector $z$ are all unchanged when we permute the vector.

\subsection{Properties under exchangeability}
\label{sec:exchangeability-properties-quantiles-cdfs}

The deterministic properties above can be directly translated into probabilistic equalities and inequalities, when applied to an exchangeable random vector.
We will state the most important of these here.
\begin{fact}[Order statistics, quantiles, and CDFs under exchangeability]
    \label{fact:exchangeable-properties}
    Assume $Z\in\R^n$ is exchangeable, and fix any $i\in[n]$.
    Then we have that
    \begin{enumerate}[label=(\roman*)]
        \item\label{fact:exchangeable-properties_part1} For any $k\in[n]$, $\P(Z_i \leq Z_{(k)}) \geq k/n$ and $\P(Z_i < Z_{(k)})\leq (k-1)/n$.
        \item\label{fact:exchangeable-properties_part2} For all $\tau\in[0,1]$, $\P\left(Z_i \leq \quantile(Z; \tau)\right) \geq \tau$ and, if $\tau>0$, $\P\left(Z_i < \quantile(Z; \tau)\right) < \tau$.
        \item\label{fact:exchangeable-properties_part3} For all $\tau\in[0,1]$, $\P\left(\widehat{F}_Z(Z_i) \leq \tau\right) \leq \tau$ and $\P\left(\widehat{F}_Z(Z_i)\geq \tau\right)\geq 1-\tau$.
    \end{enumerate}
    Furthermore, if all elements of $Z$ are distinct almost surely,
    \begin{enumerate}[label=(\roman*)]
      \setcounter{enumi}{3}
      \item\label{fact:exchangeable-properties_part4} For any $k\in [n]$, $\P(Z_i\leq Z_{(k)}) = k/n$.
      \item\label{fact:exchangeable-properties_part5} For all $\tau\in[0,1]$, $\P(Z_i\leq \quantile(Z;\tau)) = \frac{\lceil n \tau\rceil}{n}$.
    \item\label{fact:exchangeable-properties_part6} For all $\tau\in[0,1]$, $\P\left(\widehat{F}_Z(Z_i) \leq \tau\right) =  \frac{\lfloor n\tau\rfloor}{n}$.
      \end{enumerate}
      \end{fact} \index{exchangeability}
Note in particular that Fact~\ref{fact:exchangeable-properties} is closely connected to the result of Corollary~\ref{cor:perm_test_pval} above, which establishes that $p = \frac{\sum_{i=1}^n \ind{Z_i\geq Z_n}}{n}$ satisfies $\P(p\leq \tau)\leq \tau$ when $Z\in\R^n$ is exchangeable. This result can be derived as a consequence of the first part of Fact~\ref{fact:exchangeable-properties}\ref{fact:exchangeable-properties_part1} (applied with $i=n$), since $p> \tau$ holds if and only if $Z_n\leq Z_{(k)}$, for $k = \lceil (1-\tau)n\rceil$.

\subsection{A distributional view} 
\label{sec:distributional-properties-quantiles-cdfs}
It is more common to define quantiles of distributions, rather than of lists or sequences. For a distribution $P$ on $\R$, for any $\tau\in[0,1]$ the $\tau$-quantile of the distribution $P$ is defined as
\begin{equation}
   \quantile(P;\tau) = \inf\{x\in\R: \P_P\left(X\leq x\right) \geq \tau\}. 
\end{equation}
If $F$ is the CDF of the distribution $P$, the quantile is therefore equal to 
$\quantile(P;\tau) = \inf\{x\in\R : F(x)\geq \tau\}$, and we will sometimes equivalently write $\quantile(F;\tau)$.
Note that the quantile at $\tau=0$ will always be $-\infty$; on the other hand, the quantile at $\tau=1$ will be finite if the support of $P$ is bounded from above, or will be $+\infty$ otherwise.

In fact, in the setting of a finite list, this distributional definition of the quantile is equivalent to the earlier Definition~\ref{def:quantile}: if we consider the  empirical distribution of the vector $z=(z_1,\dots,z_n)$, 
\[\frac{1}{n}\sum_{i=1}^n \delta_{z_i},\]
then 
\begin{equation}
    \quantile\big(z;\tau\big) = \quantile\left(\frac{1}{n}\sum_{i=1}^n \delta_{z_i};\tau\right).
\end{equation}
Similarly,  $\widehat{F}_z$ is the CDF of the empirical distribution $\frac{1}{n}\sum_{i=1}^n \delta_{z_i}$.
This distribution-based representation will be important throughout later chapters.

\section*{Bibliographic notes}
\addcontentsline{toc}{section}{\protect\numberline{}\textnormal{\hspace{-0.8cm}Bibliographic notes}}
Exchangeability as a statistical tool was first studied by \cite{de1929funzione}, although some sources attribute its roots to early philosophical work in logic~\citep{johnson1924logical}.
It is a standard fact that any mixture of exchangeable distributions is exchangeable. The most well-known result of de Finetti is the converse of this fact---namely, that any \emph{infinite} exchangeable distribution can be represented as a mixture of i.i.d.\ distributions.
This was first proved by de Finetti for the binary setting $\cZ=\{0,1\}$. 
Many authors have extended this result to broader settings (e.g., any standard Borel space $\cZ$), most notably by \cite{hewitt1955symmetric}, but with many recent advances in the literature as well. 
While the theorem does not hold in the finite setting, approximate versions of the result can be established, namely, \cite{diaconis1980finite}'s well-known result showing that, if $Z_1,\dots,Z_n$ can be embedded into a longer exchangeable sequence $Z_1,\dots,Z_m$ for $m\gg n$, then its distribution can be approximated as a mixture of i.i.d.\ distributions. See also \citet{aldous1985exchangeability,kingman1978uses} for additional classical background on exchangeability. Conditional properties of exchangeable distributions (as in Proposition~\ref{prop:empirical-distrib-exch} and Lemma~\ref{lem:conditional_exchangeability}) are standard in the literature; see, e.g., \citet[Appendix A.5]{vovk2005algorithmic} for background. A classical example of an exchangeable (but not i.i.d.) distribution is P{\'o}lya's urn model \citep{polya1930quelques}---see the Chapter~\ref{chapter:exchangeability} exercises.

Permutation tests have been core objects of study in statistics since at least the time of Fisher's randomized experiments, including the famous `Lady Tasting Tea'~\citep{fisher1956mathematics}.
The permutation test as we know it today, as a valid significance test for arbitrary distributions (i.e., the result of Theorem~\ref{thm:perm-test}), was formalized by \cite{pitman1937significance}; see \cite{lehmann1986testing} for a comprehensive classical reference on permutation tests. A common application of the permutation test (in addition to the examples studied in this chapter) is the problem of marginal independence testing, which was formalized by \cite{pitman1937significanceII}: to test whether random variables $X,Y$ are independent given i.i.d.\ draws $(X_1,Y_1),\dots,(X_n,Y_n)$, we compare against permuted datasets, i.e., $(X_1,Y_{\sigma(1)}),\dots,(X_n,Y_{\sigma(n)})$ for $\sigma\in\cS_n$. We will return to the problem of independence testing in Chapter~\ref{chapter:conditional-independence-testing}.

Since the advent of powerful computers, permutation tests have become a practical solution for testing nonparametric null hypotheses, and thus their popularity has grown~\citep{ernst2004permutation}.
Critical to the practical use of permutation tests was the development of the randomized test by \cite{dwass1957modified}.
In the work of Dwass and throughout the late 20th century, the randomized permutation p-value was simply seen as an estimate of the quantity computed on all $n!$ permutations.
Thus, it was common to calculate estimates of the p-value that are valid asymptotically---essentially, this amounts to removing the `$+1$' in~\eqref{eq:perm-pval-random}.
However, it was pointed out by~\cite{phipson2010permutation} that this strategy can lead to extreme violations of the Type~I error rate in finite samples, and the exact test (i.e., adding the `$+1$') has much better practical performance (this was likely known to some beforehand as well, as the `$+1$' does also appear in~\cite{lehmann1986testing}).
The literature on permutation tests and other randomization tests has since developed significantly, and several generalizations have been proposed, e.g., to arbitrary groups~\citep{hemerik2021another, besag1989generalized}, to dependent data generated by Markov chains~\citep{besag1989generalized}, and to incorporate sampling weights~\citep{harrison2012conservative}. See also \cite{kuchibhotla2020exchangeability,zhang2023randomization} for further discussion of the connections between exchangeability and permutation testing.
Other work in the area includes efficient subroutines for the computation of permutation tests~\citep{koning2022faster,domingo2025cheap}, and the study of the optimality of the permutation test, e.g., in the minimax framework~\citep{kim2022minimax}.
A proof of the validity of the sampled permutation p-value (Theorem~\ref{thm:perm-test-random}) is available in~\cite{hemerik2018exact}, although the result itself was well-known beforehand.

The connections between order statistics, quantiles, and CDFs (Section~\ref{sec:permutations_key_ingredients}) are standard tools; for a detailed reference, see \citet[Chapter 7.3]{shorack2000probability}. For background on the regularity conditions we assume throughout the book (such as existence of regular conditional probability distributions, as mentioned at the beginning of the chapter), and other background in measure theory and probability, see \citet{durrett2019probability}.

\section*{Exercises}
\addcontentsline{toc}{section}{\protect\numberline{}\textnormal{\hspace{-0.8cm}Exercises}}
\begin{enumerate}[font=\bfseries, label={\thechapter.\arabic*}, labelsep=1em, itemsep=1em]
\item Let $n\geq 2$ be arbitrary. Construct a joint distribution of a random vector $(X_1,\dots,X_n)$, where for each $i\neq j$ it holds that $(X_i,X_j)$ is exchangeable, but $(X_1,\dots,X_n)$ is not exchangeable.
\item In this exercise we will examine the possible dependence that might exist between elements of an exchangeable vector.\begin{enumerate}
    \item Let $n\geq 2$ be arbitrary. Construct a finite exchangeable sequence $(X_1,\dots,X_n)$ such that $\Cov(X_1,X_2)<0$.
    \item Prove that if $(X_1,X_2,\dots)$ is an infinite exchangeable sequence then $\Cov(X_1,X_2)\geq 0$. (Assume that $\Var(X_1)<\infty$, so that this covariance is well-defined.)
\end{enumerate}    
\item Fix any $n\geq 1$ and $\tau\in(0,1)$. Under what conditions is it true that, for all $v\in\R^n$, $\quantile(v;1-\tau) = -\quantile(-v;\tau)$? (For instance, for $\tau = 0.1$, does the $90$th percentile of $v_1,\dots,v_n$ correspond to the $10$th percentile of $-v_1,\dots,-v_n$, up to sign?)
\item In this exercise, we will consider \emph{P{\'o}lya's urn}, a classical model that is constructed as follows. We begin with a finite collection of marbles of two colors, which here we will denote as $0$ and $1$, with $a$ many $0$'s and $b$ many $1$'s in the urn at the start of the process. At each time, we randomly sample one marble from the urn; it is removed, and replaced with $c$ many marbles from the same color. Here $a,b,c$ are all positive integers. (See Figure~\ref{fig:polyas_urn} for an illustration.)

    Let $X_1,\dots,X_n$ denote the sequence of draws from P{\'o}lya's urn, i.e., $X_i\in\{0,1\}$ denotes the color of the marble drawn at the $i$th iteration of the process described above.
    \begin{enumerate}
        \item Prove that, if $c>1$, then $X_1, \ldots, X_n$ are \emph{not} independent.
        \item Prove that $X_1, \ldots, X_n$ \emph{are} identically distributed (marginally).
        \item Prove that the vector $(X_1, \ldots, X_n)$ is exchangeable.
    \end{enumerate}

        \begin{figure}[t]
            \centering
            \includegraphics[width=0.65\textwidth]{\diagramspath urns.pdf}
            \caption{This figure depicts a single time step of P{\'o}lya's urn model, for the case $c=2$. We begin with an urn containing two colors of marbles, corresponding to the numbers $0$ and $1$.
            We then sample one at random, drawing $X_1 = 1$. This removed marble is then replaced by $c=2$ new marbles of the same color, and the urn is now ready for sampling the second draw $X_2$.}
            \commentAlt{The figure illustrates how data is sampled from P{\'o}lya’s urn model, when we choose the parameter value $c=2$. Four urns are shown in sequence to illustrate one stage of the sampling procedure. See long description.}
            \commentLongAlt{The figure illustrates how data is sampled from P{\'o}lya’s urn model, when we choose the parameter value $c=2$. Four urns are shown in sequence to illustrate one stage of the sampling procedure. The first urn contains $5$ balls: $3$ light gray and $2$ dark gray. The second urn shows $1$ light gray ball removed from the urn, and now there are $2$ remaining light gray balls and $2$ dark gray balls inside the urn. The third urn shows that the removed ball has now been replaced by $2$ light gray balls outside the urn, and there are still $2$ remaining light gray balls and $2$ dark gray balls inside the urn. The final, fourth urn shows that the balls have now all been placed back into the urn, so that there are now $4$ light gray balls and $2$ dark gray balls inside the urn.}
            \label{fig:polyas_urn}
        \end{figure}
    \item Prove all parts of Fact~\ref{fact:conversion-quantiles-cdfs}, using the definitions of the quantile and the CDF of a finite list.
    \item Prove all parts of Fact~\ref{fact:monotone-invariance-quantiles}, using the definitions of the quantile, the CDF, and the order statistics of a finite list.
    \item The key step in proving Theorem~\ref{thm:perm-test} (the validity of the permutation test) is establishing the inequality in~\eqref{eq:permutation-pval-CDF-inequality}, restated here for convenience: for any function $T:\cZ\to\R$,
    \[\P_\sigma\left(F(-T(z_\sigma);z)\leq \tau\right)\leq \tau,\]
    where $z\in\R^n$ is fixed, while $\sigma$ is a permutation sampled uniformly at random from $\cS_n$, and where we recall the definition
    \[F(v;z) = \frac{\sum_{\sigma\in\cS_n}\ind{-T(z_\sigma) \leq v}}{n!}.\] In the setting of randomly sampled permutations, the analogous inequality is the following:
    \begin{equation}\label{eq:permutation-pval-CDF-inequality-sample}\P(F(-T(z_{\sigma_0});z,\sigma_{1:M}) \leq \tau) \leq \tau,\end{equation}
    where we define
    \[F(v;z,\sigma_{1:M}) = \frac{1+\sum_{m=1}^M\ind{-T(z_{\sigma_m}) \leq v}}{1+M}.\]
    This bound is a key step in verifying Theorem~\ref{thm:perm-test-random}. Prove that~\eqref{eq:permutation-pval-CDF-inequality-sample} holds for any fixed $z\in\R^n$, where the probability is taken with respect to permutations $\sigma_0,\sigma_1,\dots,\sigma_M$ sampled uniformly at random (with replacement) from $\cS_n$.
\end{enumerate}

\part{Conformal Prediction}
\label{part:conformal}

\chapter{Conformal Prediction Under Exchangeability}
\label{chapter:conformal-exchangeability}

This chapter begins Part~\ref{part:conformal} of the book, where we introduce conformal prediction, examine the ideas underlying its construction, and study its theoretical properties---its statistical guarantees, both under exchangeability and under stronger assumptions, and the hardness results that capture its limitations.

The goal of this first chapter is to provide the reader with a theoretical understanding of the conformal prediction method.
We will build on the foundation of exchangeability given in the previous chapter.
This chapter aims to provide both an understanding of how conformal prediction works and why it is able to provide distribution-free guarantees.

\section{Setting: data points, datasets, and scores}
We begin with an exchangeable sequence of data points, $(X_1, Y_1), \ldots,  (X_{n+1}, Y_{n+1})$. As before, we will refer to $X_i\in\cX$ as the feature (or covariate) and $Y_i\in\cY$ as the response for the $i$th data point $(X_i,Y_i)\in\cX\times\cY$. In the setting of a prediction problem, the final response value $Y_{n+1}$ is unobserved---this is the value that we are trying to predict using the training data points $(X_1,Y_1),\dots,(X_n,Y_n)$ and the test feature $X_{n+1}$.

Conformal prediction constructs prediction sets $\cC(X_{n+1}) \subseteq \cY$ that satisfy marginal predictive coverage at level $1-\alpha$, that is, $\P(Y_{n+1}\in\cC(X_{n+1}))\geq 1-\alpha$.
The set will be constructed using a score function $s$---this is a function that maps a data point $(x,y)\in\cX\times\cY$ and a dataset $\cD\in(\cX\times\cY)^k$ (for any number $k$ of data points) to a real value, $s((x,y);\cD)\in\R$. 
The reader should think of the score  $s( (x,y) ; \cD)$ as a measure of a model's error on a single test point $(x,y)$ after training the model on the dataset $\cD$.
As in Chapter~\ref{chapter:introduction}, a motivating example of a score function is the residual score, \index{score function!residual score}
\begin{equation}
    \label{eq:absolute-residual}
    s( (x,y);\cD) = |y - \hf(x ; \cD)|,
\end{equation}
where $\hf(x ; \cD)$ is the prediction of a model trained on $\cD$ when given a feature $x$ as input.
Note that this score is large when the model is badly wrong in its prediction. 
This will be true for all scores used in this book: a high value of the score $s((x,y);\cD)$ indicates that the value $y$ is far from what the model would have predicted given $x$, after training on the data contained in $\cD$. While in this example we have $\cY = \R$, we highlight that the conformal prediction framework also applies to classification (where $\cY$ is a discrete set) and cases where $\cY$ is multidimensional. \index{score function}

Throughout Part~\ref{part:conformal} of the book, we require that the score function is \emph{symmetric}, meaning it is invariant to permutations of the dataset $\cD$:
\begin{definition}[Symmetric score function]\label{def:symmetric_score}
    A score function $s$ is \emph{symmetric} if for any data point $(x,y)\in\cX\times\cY$, any dataset $\cD \in (\cX\times\cY)^k$, and any permutation $\sigma$ on $[k]$, 
    \begin{equation}
        s( (x,y);\cD) = s( (x,y);\cD_{\sigma}).
    \end{equation}
\end{definition} \index{score function!symmetry}
Here, given a permutation $\sigma$, we use the notation $\cD_{\sigma}$ to refer to the dataset whose elements are permuted by $\sigma$, i.e., the $i$th element of $\cD_\sigma$ is given by the $\sigma(i)$th element of $\cD$.
In the case of the residual score in~\eqref{eq:absolute-residual}, symmetry implies that the fitted model $\hf(\cdot; \cD)$ is trained using some learning algorithm that is itself symmetric, i.e., it does not depend on the ordering of the data points in $\cD$.

\section{The full conformal prediction procedure}\label{sec:define_conformal}
\index{full conformal prediction|(}

We next describe the \emph{full conformal prediction} procedure, which generalizes the split conformal procedure from Section~\ref{sec:conformal-preview}.
At a high level, the prediction sets are constructed by inverting the score function $s$ to identify possible values $y\in\cY$ for the response $Y_{n+1}$ that agree (or \emph{conform}) with the trends observed in the available data. 

To make this precise, we need a few definitions. Let $\cD_n = ((X_1,Y_1),\dots,(X_n,Y_n))\in(\cX\times\cY)^n$ be the dataset containing the $n$ training points, and let $\cD_{n+1} = ((X_1,Y_1),\dots,(X_n,Y_n),(X_{n+1},Y_{n+1}))\in(\cX\times\cY)^{n+1}$ be the same with the test point $(X_{n+1},Y_{n+1})$ also included. For each $y\in\cY$, we also define the \emph{augmented dataset} $\cD^y_{n+1} = \left( (X_1,Y_1), \ldots, (X_n,Y_n), (X_{n+1}, y) \right)\in(\cX\times\cY)^{n+1}$, consisting of the $n$ training data points together with an additional point $(X_{n+1},y)$.
This last element of $\cD^y_{n+1}$ is the \emph{hypothesized test point}, where we substitute a hypothesized value $y$ in place of the unknown test response value $Y_{n+1}$. In the special case $y=Y_{n+1}$, this simply reduces to the combined training and test dataset $\cD_{n+1}$---that is, $\cD_{n+1}^{Y_{n+1}} = \cD_{n+1}$.

Next, as shorthand, we use the notation $S_{i}^y$ to refer to
the score for the $i$th data point within the augmented dataset $\cD^y_{n+1}$, using our score function $s$---that is, for the $n$ training points $(X_1,Y_1),\dots,(X_n,Y_n)$, we have scores
\[S_i^y = s( (X_i, Y_i) ; \cD^y_{n+1}), \ i=1,\dots,n,\]
while for the hypothesized test point $(X_{n+1},y)$ we have the score
\[S_{n+1}^y = s((X_{n+1},y);\cD^y_{n+1}).\]
Note that all these scores are computed with a model trained on the augmented dataset $\cD^y_{n+1}$---the hypothesized test point $(X_{n+1},y)$ has been included in the training process.
See Figure~\ref{fig:full-cp-panel} for a visual illustration of all these quantities in the context of the residual score~\eqref{eq:absolute-residual}.

\begin{figure}[t]
    \centering
    \includegraphics[width=0.8\textwidth]{\diagramspath full-conformal-panel-stretched.pdf}
    \caption{\textbf{Illustration of notation for a single hypothesized response $y$.} This figure illustrates the definitions of Section~\ref{sec:define_conformal}. The many small dots represent the training data points, $(X_i, Y_i)$. The larger, labeled dot is the hypothesized test point $(X_{n+1}, y)$. The regression model $\hf(x ; \cD^y_{n+1})$ is shown as a curve. Each score, $S_i^y$, is shown as a dotted line, representing the residual score as defined in~\eqref{eq:absolute-residual}---i.e., the absolute residual of the model $\hf(x ; \cD^y_{n+1})$ on the point $(X_i, Y_i)$ (or $(X_{n+1},y)$, for the case $i=n+1$). The quantile $\hat{q}^y$ is defined as in~\eqref{eq:full-cp-quantile}.}
    \commentAlt{A plot illustrating one step of the calculation of the full conformal prediction set. See long description.}
    \commentLongAlt{A plot illustrating one step of the calculation of the full conformal prediction set. The figure shows a scatterplot of $n$ data points $(X_i,Y_i)$, which comprise the training set, along with one additional data point representing the hypothesized test point, which is illustrated with a larger point and labeled $(X_{n+1},y)$. A fitted regression curve is shown, and is labeled as $\hat{f}(x;\cD^y_{n+1})$. Each of the $n+1$ data points is accompanied by a dashed line indicating its vertical distance to the curve, which is labeled as the score $S^y_i$ (for training points $i=1,\dots,n$) or $S^y_{n+1}$ (for the hypothesized test point $(X_{n+1},y)$). There is a shaded region of constant width, centered around the fitted regression curve, reaching to distance $\hat{q}^y$ from the fitted regression curve.}
    \label{fig:full-cp-panel}
\end{figure}

With this notation in place, we are now ready to define the full conformal prediction set. At a high level, when the conformal score $S^y_{n+1}$ is large compared to $S^y_{1}, \ldots, S^y_{n}$, then the hypothesized response $y$ is inconsistent with the data, and should therefore be excluded from the prediction set $\cC(X_{n+1})$.
To construct $\cC(X_{n+1})$, then, we simply take the set of all $y$ which \emph{are} consistent with the data---that is, those that yield sufficiently small scores:
\begin{equation}
    \label{eq:full-cp-set-construction}
    \cC(X_{n+1}) = \{ y : S^y_{n+1} \leq \hat{q}^y \},
\end{equation}
where
\begin{equation}
    \label{eq:full-cp-quantile}
    \hat{q}^y = \quantile\left(S_1^y, \ldots, S_n^y \; ; \; (1-\alpha)(1+1/n)\right).
\end{equation}
When $\cY$ is discrete and finite, we can iterate through each value of $y$ to construct the set in~\eqref{eq:full-cp-set-construction}.
On the other hand, when $\cY$ is infinite (such as $\cY=\R$), it may seem at first that~\eqref{eq:full-cp-set-construction} is impossible to construct in practice. We will return to the question of computation in Section~\ref{sec:computational-shortcuts}, where we will give efficient algorithms for computing (or approximating) full conformal prediction sets.

The quantity $\hat{q}^y$ is often referred to as the \emph{conformal quantile}---it is the threshold below which $S^y_{n+1}$ is considered small enough to `conform' with the given data.
Collecting all such values $y$ into the set $\cC(X_{n+1})$ results in a guarantee of coverage:
\begin{theorem}[Marginal coverage guarantee of conformal prediction]
    \label{thm:full-conformal}
    Suppose that $(X_1,Y_1),...,(X_{n+1},Y_{n+1})$ are exchangeable and that $s$ is a symmetric score function.
    Then the prediction set $\cC(X_{n+1})$ defined  in~\eqref{eq:full-cp-set-construction} satisfies the marginal coverage guarantee,
    \begin{equation}
        \P\left( Y_{n + 1} \in \cC(X_{n + 1}) \right) \geq 1-\alpha.
    \end{equation}
\end{theorem} \index{coverage!marginal}
We will prove Theorem~\ref{thm:full-conformal} in multiple ways in the subsequent sections of this chapter.
The coverage property above is $\emph{marginal}$, in the sense that it holds on average over the distribution of the entire dataset---both the training set $\cD_n$ (which is used to construct the prediction set) and the test point $(X_{n+1},Y_{n+1})$. See Chapter~\ref{chapter:conditional} for a detailed discussion of this point.

For concreteness, we also give an algorithmic description of full conformal prediction in Algorithm~\ref{alg:full-cp}, along with an illustration of the procedure in Figure~\ref{fig:full-cp}.
\begin{algbox}[Full conformal prediction]
    \label{alg:full-cp}
    \begin{enumerate}
        \item Input training data $(X_1, Y_1), ..., (X_n, Y_n)$, test point $X_{n+1}$, target coverage level $1-\alpha$, conformal score function $s$.
        \item For each possible response value $y \in \cY$,\begin{enumerate}
            \item Compute the score $S_i^y = s((X_i, Y_i); \cD^y_{n+1})$, for all $i \in [n]$.
            \item Compute the score for the hypothesized test point, $S_{n+1}^y = s((X_{n+1}, y);\cD^y_{n+1})$.
            \item Compute the quantile $\hat{q}^y = \quantile\left(S_1^y, \ldots, S_n^y ; (1-\alpha)(1+1/n)\right)$.
        \end{enumerate}
        \item Return the prediction set $\cC(X_{n+1}) = \{ y \in \cY : S_{n+1}^y \leq \hat{q}^y \}$.
    \end{enumerate}
\end{algbox}
Note that, for a small sample size $n$, the quantile level $(1-\alpha)(1+1/n)$ appearing in Algorithm~\ref{alg:full-cp} may in fact be $>1$; here and throughout the book, we will take the convention that, for any $\tau>1$, $\quantile(z;\tau)=+\infty$ for any vector $z$, and similarly $\quantile(P;\tau)=+\infty$ for any distribution $P$. In the case where $(1-\alpha)(1+1/n) > 1$, the conformal prediction algorithm would therefore return the full response space, $\cC(X_{n+1})=\cY$, for any test point $X_{n+1}$.

\begin{figure}[p]
    \centering
    \includegraphics[width=\textwidth]{\diagramspath full-conformal.pdf}
    \caption{\textbf{An illustration of full conformal prediction} with the residual score function. On the left-hand side are four hypothesized response values $y$, i.e., four distinct iterations of the `For' loop in Algorithm~\ref{alg:full-cp}.
    In the center, for each possible value $y$ of the response, we display a smaller version of the plot in Figure~\ref{fig:full-cp-panel}.
    Note that each center figure has a different fitted function $\hf(\cdot; \cD^y_{n+1})$, since  changing the value $y$ has an effect on the regression function.
    For each value of $y$, the width of the light gray shaded band indicates the conformal quantile $\hat{q}^y$, as in Figure~\ref{fig:full-cp-panel}.
    Finally, the right-hand side shows the final prediction set $\cC(X_{n+1})$.
    The set contains all hypothesized response values $y$ whose residuals are no larger than the conformal quantile $\hat{q}^y$---that is, all values of $y$ for which the larger data point, which represents $(X_{n+1},y)$, lies within the gray band.
    }
    \label{fig:full-cp}
\end{figure}

\subsection{Choosing the score function}\label{sec:score-preview}\index{score function}

Although marginal coverage will hold for any symmetric score function as stated in Theorem~\ref{thm:full-conformal}, the choice of score function affects the shape and size of the prediction sets; for example, we saw in Section~\ref{sec:intro-scores} that the residual score, the scaled residual score, and the CQR score all resulted in prediction sets with different shapes for the setting of regression with a real-valued response. Moreover, a score function constructed using a more accurate model will naturally result in smaller, and therefore more informative, prediction sets.
Thus, the choice of score function can strongly affect the quality of the prediction sets. In Chapter~\ref{chapter:model-based}, we will return to this point in more detail, examining the choice of the score function from a model-based theoretical perspective.
\index{full conformal prediction|)}

\section{Why does conformal prediction guarantee coverage?}\label{sec:full_conformal_first_proof}
In this section, we will give intuition for the marginal coverage guarantee of full conformal prediction, stated in Theorem~\ref{thm:full-conformal}, as well as a formal proof.

To build an understanding of the idea behind the proof, let us first consider a naive procedure that does \emph{not} guarantee coverage. 
Imagine that we have a model trained on data $\cD_n = ((X_1,Y_1), \ldots, (X_n,Y_n))$, and we want to predict a new response $Y_{n+1}$ based on the covariate $X_{n+1}$ for this test point.
The scores $s((X_1,Y_1);\cD_n), \ldots, s((X_n,Y_n);\cD_n)$ on the training data points might not be directly comparable to the test point's score, $s((X_{n+1},Y_{n+1});\cD_n)$, since our model will likely overfit to the training data---that is, the test point score $s((X_{n+1},Y_{n+1});\cD_n)$ would likely be higher than the others.
Thus, if we were to use the $(1-\alpha)$-quantile of the training scores, it would likely be too low to ensure the desired coverage probability for the test point.

However, if we could somehow include $(X_{n+1},Y_{n+1})$ in the training set, we could avoid this issue---the score for the test point would have the same distribution as the calibration scores.
This is where full conformal prediction comes into play: it trains on \emph{all the data}, including the test point.
This means the test point is treated exactly the same as the training instances, ensuring exchangeability: that is,
the scores
\begin{equation}\label{eqn:define_n+1_scores}S_i = S_i^{Y_{n+1}} = s((X_i,Y_i);\cD_{n+1}),\end{equation}
for training points $i=1,\dots, n$ and test point $i=n+1$, are exchangeable, as we will verify formally below.
Of course, computing these scores cannot be done in practice, since the test point's response value $Y_{n+1}$ is unknown, and so we cannot train our model on the larger dataset $\cD_{n+1}$. Instead, we consider some provisional value of $y \in \cY$ and add the hypothesized test point $(X_{n+1},y)$ into the dataset, repeating this process for \emph{each} possible value $y\in\cY$. This is the core idea of conformal prediction.

With this intuition in place, we are now ready to present our first proof of the marginal coverage property of full conformal prediction. We will give some alternative proofs and perspectives later on in Section~\ref{sec:reinterpret-conformal}.
\begin{proof}[Proof of Theorem~\ref{thm:full-conformal}]  
    \textbf{Step 1: A reformulation of the prediction set.}
    Our first step will derive an equivalent definition of the prediction set $\cC(X_{n+1})$.  To do this, we will need a lemma:
    \begin{lemma}[Replacement lemma] \index{replacement lemma}
        \label{lem:n+1-to-n-reduction}
        Let $v_1, \ldots, v_{n+1}\in\R$.
        Then for any $t\in[0,1]$,
        \[
            v_{n+1} \leq \quantile\left( v_1, \ldots, v_{n+1};t\right) \Longleftrightarrow v_{n+1} \leq \quantile\left( v_1, \ldots, v_n;t (1+1/n) \right).
        \]
    \end{lemma}
    Since by definition we have
    \begin{equation}
        \hat{q}^y = \quantile\left(S_1^y,\dots,S_n^y;(1-\alpha)(1+1/n)\right),
    \end{equation}
    for each $y\in\cY$, this lemma implies that
    \begin{equation}\label{eqn:standard-proof-coverage-equiv}
        y \in \cC(X_{n+1}) \Longleftrightarrow S_{n+1}^y \leq \hat{q}^y \Longleftrightarrow S_{n+1}^y\leq \quantile\left(S_1^y,\dots,S_{n+1}^y;1-\alpha\right).
    \end{equation}
In other words, the prediction set $\cC(X_{n+1}) $ can be constructed by comparing the test point score, $S_{n+1}^y$, against the $(1-\alpha)$-quantile of the \emph{augmented} list of scores, $S_1^y,\dots,S_{n+1}^y$ (i.e., this list includes the test point score $S_{n+1}^y$ itself).

This therefore implies that the coverage event can equivalently be written as
\begin{equation}\label{eqn:standard-proof-coverage-equiv_Yn+1}Y_{n+1}\in\cC(X_{n+1})\Longleftrightarrow S_{n+1}\leq \quantile\left(S_1,\dots,S_{n+1};1-\alpha\right),\end{equation}
where the scores $S_1,\dots,S_{n+1}$ are defined as in~\eqref{eqn:define_n+1_scores} above.
From this point on, then, to establish the coverage guarantee, we only need to show that the event $S_{n+1}\leq \quantile\left(S_1,\dots,S_{n+1};1-\alpha\right)$
holds with $\geq 1-\alpha$ probability.

This characterization leads right away to an important insight: the coverage event depends only on these scores $S_1,\dots,S_{n+1}$, which are values obtained when the model is fitted with $y=Y_{n+1}$, i.e., on the dataset $\cD_{n+1} = \big((X_1,Y_1),\dots,(X_n,Y_n),(X_{n+1},Y_{n+1})\big)$. In other words, while running full conformal prediction requires us to train the model using values $y\neq Y_{n+1}$ (since we must train using each possible value $y\in\cY$), for the theoretical coverage guarantee these resulting scores are irrelevant.\bigskip

    \textbf{Step 2: Exchangeability of scores when $y = Y_{n+1}$.}
    Next, we show that the scores $S_1,\dots,S_{n+1}$ are exchangeable, i.e., that they satisfy Definition~\ref{def:exchangeability}.
    Towards that end, let $\sigma$ be a permutation on $\{1,\dots,n+1\}$.
    Expanding our shorthand, recall that
    \begin{equation}
        \label{eq:expanded-score}
        S_{i} = s\left((X_i, Y_i); \cD_{n+1} \right).
    \end{equation}
    First, note that by the symmetry of the score function (Definition~\ref{def:symmetric_score}), it is deterministically true that
    $S_i = s\left((X_i,Y_i);\cD_{n+1}\right)=s\left((X_i,Y_i);(\cD_{n+1})_\sigma\right)$ for each $i$. In particular, this also holds for the index $\sigma(i)$, i.e.,
    \[S_{\sigma(i)} = s\left((X_{\sigma(i)},Y_{\sigma(i)});(\cD_{n+1})_\sigma\right).\]
    Therefore, to prove that
    \begin{equation}\label{eqn:scores_exchangeable_n+1}\left(S_1,\dots,S_{n+1}\right) \eqd \left(S_{\sigma(1)},\dots,S_{\sigma(n+1)}\right)\end{equation}
    it is equivalent to verify that
    \begin{equation}\label{eqn:standard-proof-eqd-scores}\Big(s\left((X_i, Y_i);\cD_{n+1}\right)\Big)_{i\in[n+1]}\eqd \Big(s\left((X_{\sigma(i)}, Y_{\sigma(i)});(\cD_{n+1})_{\sigma} \right)\Big)_{i\in[n+1]}.\end{equation}
But this last claim follows simply from exchangeability of the data: we have
\[\cD_{n+1}\eqd (\cD_{n+1})_\sigma,\]
and the two vectors of scores in~\eqref{eqn:standard-proof-eqd-scores} are derived by applying the \emph{same} function to $\cD_{n+1}$ (on the left-hand side) and to $(\cD_{n+1})_\sigma$ (on the right-hand side).
 \bigskip
 
    \textbf{Step 3: Completing the proof.}
By Step 2, we know that the scores $S_1,\dots,S_{n+1}$ are exchangeable. Applying Fact~\ref{fact:exchangeable-properties}\ref{fact:exchangeable-properties_part2}, this immediately implies
\[\P\left(S_{n+1}\leq \quantile(S_1,\dots,S_{n+1};\tau)\right)\geq \tau\]
for any $\tau\in[0,1]$. Choosing $\tau = 1-\alpha$, by the equivalent characterization of coverage given in~\eqref{eqn:standard-proof-coverage-equiv_Yn+1}, this completes the proof.
\end{proof}

\begin{proof}[Proof of Lemma~\ref{lem:n+1-to-n-reduction}]
First, if $t > \frac{n}{n+1}$, the result holds trivially since $\quantile\left( v_1, \ldots, v_{n+1};t  \right) = \max_{i\in[n+1]} v_i \geq v_{n+1}$, while $\quantile\left( v_1, \ldots, v_n;t (1+1/n) \right) = +\infty$. Furthermore if $t=0$ then again the result is trivial since both quantiles are equal to $-\infty$. From this point on, then, we will assume $0<t\leq \frac{n}{n+1}$ to avoid these trivial cases.

Let $v_{(n;1)}\leq \dots \leq v_{(n;n)}$ be the order statistics of $v_1,\dots,v_n$, and let $v_{(n+1;1)}\leq \dots \leq v_{(n+1;n+1)}$ be the order statistics of $v_1,\dots,v_{n+1}.$
Let $k = \lceil t(n+1)\rceil\in[n]$. Examining the definition of the quantile (see Section~\ref{sec:permutations_key_ingredients}), we see that 
\[\quantile\left( v_1, \ldots, v_{n+1};t  \right) = v_{(n+1;k)}\]
and 
\[\quantile\left( v_1, \ldots, v_n;t (1+1/n) \right) = v_{(n;k)},\]
by definition. So, we need to verify that 
$
v_{n+1}\leq v_{(n+1;k)}$ holds if and only if $v_{n+1}\leq v_{(n;k)}$.

First, by definition of the order statistics, we must have $v_{(n+1;k)} \leq v_{(n;k)}$ (i.e., the $k$th smallest entry in the list cannot increase if we add a new value to the list). Therefore,
\[v_{n+1}\leq v_{(n+1;k)} \Longrightarrow v_{n+1}\leq v_{(n;k)}.\]
To verify the converse, suppose that
$v_{n+1}> v_{(n+1;k)}$.
In this case, we must have $v_{(n+1;k)}=v_{(n;k)}$, again by definition of the order statistics (i.e., the $k$th smallest entry does not decrease if we add a large value $v_{n+1}$ to the list). Therefore we have
\[v_{n+1}> v_{(n+1;k)} \Longrightarrow v_{n+1}> v_{(n;k)}.\]
    \end{proof}
    
\section{Split conformal prediction as a special case}
\label{sec:split-conformal}
\index{split conformal prediction|(}

Split conformal prediction, which we previewed earlier in Chapter~\ref{chapter:introduction}, is a variant of conformal prediction that uses data splitting to avoid the need to retrain the conformal score function $s$ for each possible response value $y\in\cY$.
It  differs from full conformal prediction in two important ways, as summarized in Table~\ref{tab:conformal}.
\begin{table}[ht]
\centering
\begin{tabular}{c|c}
\textbf{Full Conformal} & \textbf{Split Conformal} \\ 
\hline
All data used for training and calibration & Disjoint datasets for training and calibration  \\ 
Retrains the model for each value $y\in\cY$ & No model retraining (requires only one model fit) \\ 
Requires a symmetric score function $s$ & Works for any (pretrained) score function $s$\\ 
\end{tabular}
\caption{\textbf{Comparison of full and split conformal methods.}}
\label{tab:conformal}
\end{table}

Despite the large operational differences in the algorithms, split conformal prediction is a special case of the full conformal method presented in Algorithm~\ref{alg:full-cp}.
We formally describe the specialization now.
Let us consider a \emph{pretraining set} $\cD_{\rm pre}$, which is disjoint from $\cD_n=((X_1,Y_1),\dots,(X_n,Y_n))$.
$\cD_{\rm pre}$ is used exclusively for model training, while $\cD_n$ is used exclusively for calibrating a threshold for scores (i.e., the conformal quantile $\hat{q}$) to define the prediction set $\cC(X_{n+1})$---consequently, in the context of split conformal, we will refer to $\cD_n$ as a \emph{calibration set}, rather than a training set. For example, in the context of a residual score~\eqref{eq:absolute-residual}, we can imagine a scenario where the predictive model $\hf$ was pretrained on earlier data, $\cD_{\rm pre}$, and we now have a calibration (or holdout) set $\cD_n$ consisting of $n$ \emph{new} data points that we may use to quantify our uncertainty around $\hf$. (Note that this differs from our earlier notation in Chapter~\ref{chapter:introduction} where we partitioned the $n$ total available data points into a pretraining set of size $n/2$ and a calibration set of size $n/2$.)

More formally, continuing with the residual score as a concrete example, the residual score in the context of split conformal prediction is given by
\begin{equation}\label{eq:pretrained_residual_score}
    s((x,y); \cD) = \left| y - \hf(x; \cD_{\rm pre}) \right|,
\end{equation}
Since this does not depend on $\cD$, we can simply write
\[s((x,y); \cD) = s(x,y)\]
for any data point $(x,y)$, suppressing the unused argument $\cD$ in the notation. (Here we are treating the pretraining set $\cD_{\rm pre}$, and consequently also the prefitted regression function $\hf(\cdot; \cD_{\rm pre})$, as fixed.) 
More generally, we can apply \emph{any} method to the pretraining set to produce a pretrained score function---that is, we can use $\cD_{\rm pre}$ arbitrarily to construct a score function $s:\cX\times\cY\to\R$. In particular, since $s((x,y);\cD) = s(x,y)$ does not depend on $\cD$, trivially this is a symmetric score function.

\begin{definition}[Split conformal prediction]
    Henceforth, \emph{split conformal prediction} refers to the case where the score function $s((x,y); \cD)$ does not depend on $\cD$. In this case, we write $s(x,y)$ as shorthand for $s( (x,y) ; \cD )$.
\end{definition}

With all this in mind, we can then construct the \emph{split conformal prediction} set as
\begin{equation}
    \label{eq:split-conformal-sets}
    \cC(X_{n+1}) = \{ y : s(X_{n+1},y) \leq \hat{q} \},
\end{equation}
where
\begin{equation}
    \label{eq:split-conformal-quantile}
    \hat{q} = \quantile\left( S_1, \ldots, S_n; \; (1-\alpha)(1+1/n)  \right),
\end{equation}
and where we use the notation $S_i$ as shorthand for $s(X_i, Y_i)$. Note that $\hat{q}$ is equivalent to the conformal quantile $\hat{q}^y$ defined for full conformal in~\eqref{eq:full-cp-quantile} (since the score function does not depend on the dataset, there is no dependence on $y$---i.e., $S_i = S_i^y$ for all $y$ and all $i\in[n]$, and consequently, $\hat{q} = \hat{q}^y$ for all $y$). 

An explicit algorithm for split conformal prediction is given in Algorithm~\ref{alg:split-cp}.
Note that it is simpler and more computationally efficient than Algorithm~\ref{alg:full-cp}, and therefore used more often in practice.
\begin{algbox}[Split conformal prediction]
    \label{alg:split-cp}
    \begin{enumerate}
        \item Input pretraining dataset $\cD_{\rm pre}$, 
    calibration data $(X_1, Y_1), ..., (X_n, Y_n)$, test point $X_{n+1}$, target coverage level $1-\alpha$.
    \item Using the pretraining dataset $\cD_{\rm pre}$, construct a score function $s:\cX\times\cY\to\R$.
    \item Compute the scores on the calibration set, $S_i = s(X_i, Y_i)$ for $i\in[n]$.
    \item Compute the conformal quantile $\hat{q} = \quantile\left(S_1, \ldots, S_n ; (1-\alpha)(1+1/n)\right)$.
    \item Return prediction set $\cC(X_{n+1}) = \{ y \in \cY : s(X_{n+1},y) \leq \hat{q} \}$.
    \end{enumerate}
\end{algbox}
Since split conformal prediction (Algorithm~\ref{alg:split-cp}) is a special case of full conformal prediction (Algorithm~\ref{alg:full-cp}), it inherits the marginal coverage guarantee of Theorem~\ref{thm:full-conformal}. 

\subsection{Exchangeability and split conformal: a closer look}
There is a subtle difference in the interpretation of the conformal score in split versus full conformal prediction.
In full conformal prediction, the scores $S_1^y,\dots,S_{n+1}^y$ are computed on the augmented dataset, $\cD^y_{n+1}$, using a score function trained on this same augmented dataset. 
Thus, in full conformal, verifying the exchangeability of $(S_1,\dots,S_{n+1}) = (S_1^{Y_{n+1}},\dots,S_{n+1}^{Y_{n+1}})$ only relies on the exchangeability of $\cD_{n+1} = D^{Y_{n+1}}_{n+1}$.
In contrast, for split conformal prediction,
there are two datasets at play: the calibration set $\cD_n$ and the pretraining set $\cD_{\rm pre}$.
The vector of scores $(S_1,\dots,S_{n+1})$ depends on \emph{both} datasets. 
So, to formally establish exchangeability of the scores, we need to consider the additional randomness coming from $\cD_{\rm pre}$.

The most straightforward formal argument to justify split conformal prediction is to assume that $\cD_{\rm pre}$ is independent of the calibration data and test point, i.e.,
\[((X_1,Y_1),\dots,(X_{n+1},Y_{n+1})) \indep \cD_{\rm pre}\textnormal{ (and $(X_1,Y_1),\dots,(X_{n+1},Y_{n+1})$ are exchangeable)}.\]
In that case, by conditioning on $\cD_{\rm pre}$, we can think of the score function $s:\cX\times\cY\to\R$ (which was constructed based on $\cD_{\rm pre}$) as a \emph{fixed} function---and therefore, exchangeability of $(X_1,Y_1),\dots,(X_{n+1},Y_{n+1})$ directly implies exchangeability of the scores $s(X_1,Y_1),\dots,s(X_{n+1},Y_{n+1})$. 

More generally, however, to achieve marginal coverage it is sufficient to assume a strictly weaker condition:
\begin{equation}
    \label{eq:split-training-conditional-exchangeability}
    ((X_1,Y_1),\dots,(X_{n+1},Y_{n+1}))\mid \cD_{\rm pre}\textnormal{ is exchangeable,}
\end{equation}
i.e., $((X_1,Y_1),\dots,(X_{n+1},Y_{n+1}))$ has an exchangeable conditional distribution (when conditioning on $\cD_{\rm pre}$).
For instance, this assumption holds whenever 
the pretraining, calibration, and test data (i.e., the entire dataset ($\cD_{\rm pre}, \cD_n , (X_{n+1},Y_{n+1}))$, containing $n_{\rm pre}+n+1$ data points) satisfies exchangeability. This is due to the following fact:
\begin{fact}[Conditional exchangeability of a subvector]\label{fact:conditional-exch-1}
    Fix any $1 \leq k < m$. If $(Z_1,\dots,Z_m)$ is exchangeable, then $(Z_{k+1},\dots,Z_m)$ is conditionally exchangeable given $(Z_1,\dots,Z_k)$---i.e., it holds almost surely that the conditional distribution
    \[(Z_{k+1},\dots,Z_m)\mid (Z_1,\dots,Z_k)\]
    is an exchangeable distribution.
\end{fact}

\subsection{Statistical versus computational efficiency for split and full conformal}

In general, split conformal is far more computationally efficient than full conformal---for example, for the residual score~\eqref{eq:absolute-residual}, split conformal only requires training one model $\hf(\cdot;\cD_{\rm pre})$  (i.e., train on $\cD_{\rm pre}$), but full conformal requires training $\hf(\cdot;\cD^y_{n+1})$ for every possible value $y\in\cY$. On the other hand, as we will now explain, full conformal is more  efficient in a statistical sense. 

In practical implementations of these methods, it is common to use $n$ to denote the \emph{total} sample size of available data---that is, for split conformal, the pretraining set and the calibration set are taken to be disjoint subsets of the available dataset, and thus have a \emph{combined} size of $n$.
This means that the greatly improved computational efficiency of Algorithm~\ref{alg:split-cp} (split conformal) comes with a statistical cost: the pretraining set and the calibration set each have sample size smaller than $n$ (say, $n/2$ data points each---as in our earlier notation, in Chapter~\ref{chapter:introduction}), potentially leading to wider prediction intervals. In contrast, in Algorithm~\ref{alg:full-cp}, full conformal prediction uses all $n$ data points both for training (i.e., fitting a score function) and for calibration.

In this book, for consistency of notation across both split and full conformal prediction, we will generally continue to use $n$ to denote the sample size of the calibration set in the case of split conformal, as in Algorithm~\ref{alg:split-cp}. This notation reveals the way in which split conformal prediction can be viewed as a special case of full conformal prediction, and enables a more unified presentation of the theoretical results across these different variants of the method.
\index{split conformal prediction|)}

\section{Reinterpretations of conformal prediction}\label{sec:reinterpret-conformal}

Next, we give some different interpretations of conformal prediction, and different proofs of its marginal coverage guarantee (Theorem~\ref{thm:full-conformal}), to give intuition about the procedure and connect it to standard ideas in statistics, such as permutation tests. 

\subsection{Conformal prediction as a permutation test}\label{sec:conformal_as_perm} \index{permutation test}

In this section, we formulate conformal prediction as a permutation test. In particular, it can be viewed as a test of whether $(X_{n+1},y)$ is an outlier relative to the other data points, $(X_1,Y_1),\dots,(X_n,Y_n)$, as in the last example from Section~\ref{sec:permtest-examples}.
We will do this by defining the \emph{conformal p-value}, which lets us build an explicit connection between constructing prediction sets and testing hypotheses.
\begin{definition}[The conformal p-value]\label{def:conformal-pvalue}
    Given training data $(X_1,Y_1),\dots,(X_n,Y_n)$, a test feature $X_{n+1}$, and a score function $s$, the conformal p-value is defined as
    \[p^y = \frac{1 + \sum_{i=1}^n \ind{S^y_i\geq S^y_{n+1}}}{n+1}\]
    for each $y\in\cY$,
    where as before,
    \[S^y_i = s((X_i,Y_i);\cD^y_{n+1}),\ i\in[n], \quad S^y_{n+1} = s((X_{n+1},y);\cD^y_{n+1}).\]
\end{definition} \index{conformal p-value}
This p-value is asking whether the hypothesized test point $(X_{n+1},y)$ appears to follow the same distribution as the training data $(X_1,Y_1),\dots,(X_n,Y_n)$. If not, its score $s((X_{n+1},y);\cD^y_{n+1})$ might be substantially larger than the other scores, and consequently its p-value $p^y$ will likely be small. Informally, we can interpret $p^y$ as a p-value testing the hypothesis of exchangeability of the data points $(X_1,Y_1),\dots,(X_n,Y_n),(X_{n+1},y)$. 
We can then repeat this reasoning for every possible value $y \in \cY$, and collect all the plausible values (i.e., values $y$ for which $p^y$ is not too small) into a prediction set. In fact, conformal prediction is exactly equivalent to this procedure:
\begin{proposition}[Full conformal prediction via the conformal p-value]\label{prop:conformal-via-pvalues}
The full conformal prediction set $\cC(X_{n+1})$ defined in~\eqref{eq:full-cp-set-construction} satisfies
\[\cC(X_{n+1}) = \left\{y\in\cY: p^y > \alpha\right\},\]
where $p^y$ is the conformal p-value (Definition~\ref{def:conformal-pvalue}).
\end{proposition}
\begin{proof}[Proof of Proposition~\ref{prop:conformal-via-pvalues}]
   First, by definition of the full conformal prediction set,
    \[y\not\in\cC(X_{n+1}) \Longleftrightarrow S^y_{n+1} > \quantile\left(S^y_1,\dots,S^y_n;(1-\alpha)(1+1/n)\right).\]
    Next, by definition of the quantile of a finite list, for any $\tau\in[0,1]$ and any $t\in\R$ it holds that
    \[\quantile\left(S^y_1,\dots,S^y_n;\tau\right)< t  \Longleftrightarrow \sum_{i=1}^n \ind{S^y_i < t} \geq n\tau.\]
    Choosing $\tau = (1-\alpha)(1+1/n)$, then, we have
    \[y\not\in\cC(X_{n+1}) \Longleftrightarrow 
    \sum_{i=1}^n \ind{S^y_i < S^y_{n+1}} \geq (1-\alpha)(n+1).\]
    Finally, we return to the p-value: we calculate
    \[p^y = \frac{1 + \sum_{i=1}^n \ind{S^y_i\geq S^y_{n+1}}}{n+1}  = 1 - \frac{\sum_{i=1}^n \ind{S^y_i<S^y_{n+1}}}{n+1},\]
    i.e., $p^y\leq \alpha$ if and only if $y\not\in\cC(X_{n+1})$, as desired.
\end{proof}
So far, we have only established a deterministic result: for any dataset, the full conformal set  can equivalently be constructed via the conformal p-value. Now we are ready to turn to the question of coverage. We will see that $p^y$ can be reinterpreted as the p-value for a permutation test, and then validity of permutation tests will imply the marginal coverage property of full conformal prediction.

\begin{proof}[Proof of Theorem~\ref{thm:full-conformal} via permutation tests]
For convenience, throughout this proof we will write $Z_i = (X_i,Y_i)$ for the $i$th data point, and will write $\cZ = \cX\times\cY$.

    First we recall the permutation test p-value, developed in Section~\ref{sec:perm_test}: given any test function $T:\cZ^{n+1}\to\R$, we define
    \[p_{\rm perm} = \frac{\sum_{\sigma\in\cS_{n+1}}\ind{T(Z_{\sigma(1)},\dots,Z_{\sigma(n+1)})\geq T(Z_1,\dots,Z_{n+1})}}{(n+1)!},\]
    where, as before, $\cS_{n+1}$ denotes the set of permutations on $[n+1]=\{1,\dots,n+1\}$.
    (Note that here, we are working with a vector of length $n+1$, rather than $n$ as in the original definition of this p-value in~\eqref{eq:perm-pvalue}.)
    If the combined training and test dataset $(Z_1,\dots,Z_{n+1})$ is indeed exchangeable, then we have $\P(p_{\rm perm}\leq \alpha)\leq \alpha$, as established in Theorem~\ref{thm:perm-test}.

    To complete the proof, then, we will verify that
    \[Y_{n+1}\in\cC(X_{n+1})\Longleftrightarrow p_{\rm perm}> \alpha,\]
    when the test function $T$ is chosen appropriately.
    In particular, in light of Proposition~\ref{prop:conformal-via-pvalues}, it suffices to show that
    \begin{equation}\label{eqn:full-conformal-perm-pvalue}p_{\rm perm} = p^{Y_{n+1}}.\end{equation}
    
    We will choose the test function $T$ as
    \[T(z_1,\dots,z_{n+1}) = s\left(z_{n+1};(z_1,\dots,z_{n+1})\right),\]
    i.e., the score for the last data point $z_{n+1}$ when trained on a dataset $(z_1,\dots,z_{n+1})$. 
    In particular, we can calculate the value of this test function on the permuted dataset $(Z_{\sigma(1)},\dots,Z_{\sigma(n+1)})$ as 
    \[T(Z_{\sigma(1)},\dots,Z_{\sigma(n+1)}) = s(Z_{\sigma(n+1)}; (\cD_{n+1})_\sigma) = s(Z_{\sigma(n+1)};\cD_{n+1}) = S_{\sigma(n+1)},\]
    where the first step holds since $(\cD_{n+1})_{\sigma}=(Z_{\sigma(1)},\dots,Z_{\sigma(n+1)})$ by definition, and the second step holds since $s$ is assumed to be symmetric. Therefore,
    \[p_{\rm perm} = \frac{\sum_{\sigma\in\cS_{n+1}} \ind{S_{\sigma(n+1)} \geq S_{n+1}}}{(n+1)!} = \sum_{i=1}^{n+1} n! \cdot \frac{ \ind{S_i \geq S_{n+1}}}{(n+1)!},\]
    where the last step holds since, for each $i\in[n+1]$, there are exactly $n!$ many permutations $\sigma\in\cS_{n+1}$ for which $\sigma(n+1)=i$. After simplifying, we then have
    \[p_{\rm perm} =   \frac{\sum_{i=1}^{n+1} \ind{S_i \geq S_{n+1}}}{n+1} = \frac{1 + \sum_{i=1}^n \ind{S_i \geq S_{n+1}}}{n+1} = p^{Y_{n+1}},\]
    with the last step holding since $S_i = S_i^{Y_{n+1}}$ for all $i$ by definition. This establishes~\eqref{eqn:full-conformal-perm-pvalue} as desired.
\end{proof}

\subsection{Conditioning on the empirical distribution of the data}
\label{sec:conformal-conditioning-bag}
Another interpretation of conformal prediction is given by treating the dataset as an unordered collection of data points, and then using exchangeability to argue that the test point $(X_{n+1},Y_{n+1})$ is equally likely to be any one of these points. In particular, if each data point $(X_1,Y_1),\dots,(X_{n+1},Y_{n+1})$ is distinct, and if we observe the unordered set of $n+1$ values (i.e., we observe the unordered set of data points $\{(X_1,Y_1),\dots,(X_{n+1},Y_{n+1})\}$, but do not know which data value corresponds to which index $i\in[n+1]$), then the test point $(X_{n+1},Y_{n+1})$ is equally likely to be any one of these $n+1$ values.

To formalize this intuition, given the dataset $\cD_{n+1} = ((X_1,Y_1),\dots,(X_{n+1},Y_{n+1}))$ containing the training data and test point, recalling our notation from Chapter~\ref{chapter:exchangeability} we define its empirical distribution as
\[\widehat{P}_{n+1} = \frac{1}{n+1}\sum_{i=1}^{n+1}\delta_{(X_i,Y_i)}.\]
This is a discrete distribution on $\cX\times\cY$, 
and in the case that the $n+1$ data points are distinct it can be simply defined by placing mass $\frac{1}{n+1}$ on each value that appears in the dataset. More generally, values that appear multiple times in the dataset will be given higher probability under this empirical distribution.
The following proof makes critical use of Proposition~\ref{prop:empirical-distrib-exch}, which tells us that due to the exchangeability of the dataset $\cD_{n+1}$, after conditioning on the empirical distribution $\widehat{P}_{n+1}$ the distribution of the test point $(X_{n+1},Y_{n+1})$ is equal to $\widehat{P}_{n+1}$ itself.

\begin{proof}[Proof of Theorem~\ref{thm:full-conformal} via the empirical distribution of the data]
First, since the score function $s$ is assumed to be symmetric, note that $s(\cdot;\cD_{n+1})$ depends on the dataset $\cD_{n+1}$ only via the empirical distribution $\widehat{P}_{n+1}$. To emphasize this point, we will use the notation $s(\cdot ;\widehat{P}_{n+1})$ in place of $s(\cdot;\cD_{n+1})$ throughout the remainder of this proof. We will also write $Z_i = (X_i,Y_i)$ for the $i$th data point, and $\cZ = \cX\times\cY$.

Let $h:\cZ\to\R$ be any function. Conditioning on $\widehat{P}_{n+1}$, by Proposition~\ref{prop:empirical-distrib-exch}, the test point $Z_{n+1}$ has distribution $\widehat{P}_{n+1}$.
This implies that, again conditioning on $\widehat{P}_{n+1}$, the random variable $h(Z_{n+1})$ has distribution $\frac{1}{n+1}\sum_{i=1}^{n+1}\delta_{h(Z_i)}$.
In particular, 
\[\P\left(h(Z_{n+1})\leq \quantile\big((h(Z_i))_{i\in[n+1]};1-\alpha\big)\,\middle|\,\widehat{P}_{n+1}\right) \geq 1-\alpha,\]
by definition of the quantile.
Next, observe that since this calculation is carried out conditionally on $\widehat{P}_{n+1}$, the same argument holds for a function $h$ that depends on the empirical distribution $\widehat{P}_{n+1}$ as well as on a data point $z\in\cZ$---in particular, we can take $h(z) = s(z;\widehat{P}_{n+1})$, to obtain
\[\P\left(s(Z_{n+1};\widehat{P}_{n+1}) \leq \quantile\Big((s(Z_i;\widehat{P}_{n+1}))_{i\in[n+1]};1-\alpha\Big)\,\middle|\,\widehat{P}_{n+1}\right) \geq 1-\alpha.\]
Since $S_i = s(Z_i;\cD_{n+1}) = s(Z_i;\widehat{P}_{n+1})$ for each $i\in[n+1]$, by definition, we then have
\[\P(Y_{n+1}\in\cC(X_{n+1})\mid \widehat{P}_{n+1}) = \P\left(S_{n+1}\leq \quantile(S_1,\dots,S_{n+1};1-\alpha)\,\middle|\,\widehat{P}_{n+1}\right) \geq 1-\alpha,\]
where the first step applies~\eqref{eqn:standard-proof-coverage-equiv_Yn+1}.
Marginalizing over $\widehat{P}_{n+1}$, we have proved the claim.
\end{proof}

\subsection{Tuning based on a plug-in estimate of the error rate}\label{sec:CP_plugin}
\index{risk control|(}

We next consider an especially simple and intuitive interpretation of split conformal prediction. Recall that in split conformal prediction, we work with a pretrained score function $s(x,y)$ that does not depend on the calibration set. With this in hand, consider prediction sets of the form
\begin{equation}
\label{eq:conformal_sets_with_lambda}
    \cC_\lambda(X_{n+1}) = \{y : s(X_{n+1}, y) \le \lambda\},
\end{equation}
where $\lambda \in \R $ is a parameter controlling the size of the set. Split conformal prediction can be understood as selecting $\lambda = \hat{q} = \quantile\left(S_1, \ldots, S_n ; (1-\alpha)(1+1/n)\right)$. We now explain this choice of $\lambda$.

A natural way to choose the parameter $\lambda$ is by choosing it to control the coverage on an available dataset. More formally, let 
\[
   \hat{R}(\lambda) = \frac{1}{n} \sum_{i=1}^n \ind{Y_i \notin \cC_\lambda(X_i)} \\
   =  \frac{1}{n} \sum_{i=1}^n \ind{s(X_i, Y_i) > \lambda} = \frac{1}{n}\sum_{i=1}^n\ind{S_i>\lambda}
\]
be the empirical miscoverage on the calibration data. If the data points $(X_i,Y_i)$ are i.i.d.\ draws from some distribution $P$ on $\cX\times\cY$, then $\hat{R}(\lambda)$ is a noisy estimate of the population miscoverage, 
\begin{equation}
   R(\lambda) = \E_P[\ind{Y \notin \cC_\lambda(X)}] = 1 - \P_P(Y \in \cC_\lambda(X)).
\end{equation}
If we had knowledge of $R(\lambda)$, we would simply set $\lambda$ to be the smallest value such that $R(\lambda) \le \alpha$. In words, we take the smallest sets that have the desired coverage level.

\begin{figure}[t]
    \centering
    \includegraphics[height = 2.25in]{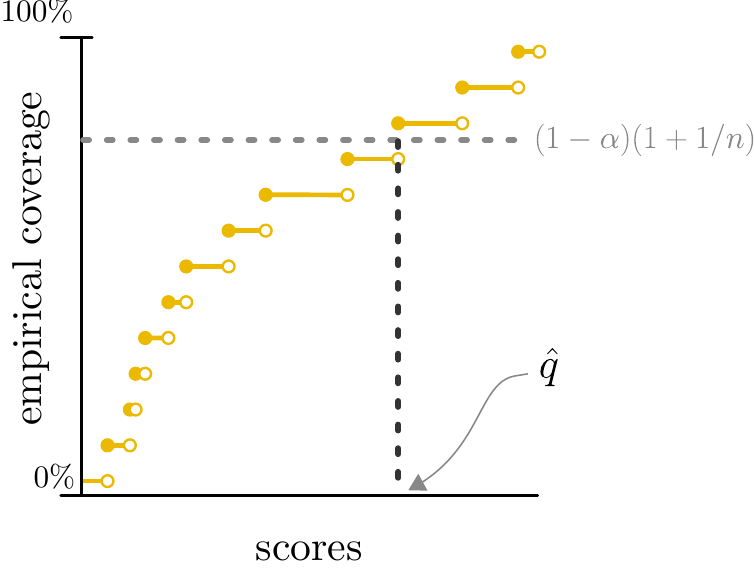}
    \caption{\textbf{A visualization of the split conformal quantile $\hat{q}$} (as defined in Algorithm~\ref{alg:split-cp}), with the interpretation discussed in Section~\ref{sec:CP_plugin}. Empirical coverage at level $\geq (1-\alpha)(1+1/n)$, as shown in this figure, is exactly equivalent to the criterion $\hat{R}(\lambda)\leq\alpha'$, as in~\eqref{eq:conformal-plugin-estimate-lambdahat}.
    }
    \commentAlt{A piecewise constant function shows the empirical CDF of the scores. The level $(1-\alpha)(1+1/n)$ is shown with a horizontal dashed line, and the score $\hat{q}$ (where the empirical CDF reaches this level) is shown with a vertical dashed line.}
    \label{fig:split-CP-qhat}
\end{figure}

Since we do not know $R(\lambda)$, one reasonable approach is a plug-in method: we could choose $\hat\lambda$ to be the smallest value such that $\hat{R}(\lambda)\leq \alpha$. This plug-in method is appealing, but it does not quite guarantee a bound on the true risk $R(\hat\lambda)$, since $\hat{R}$ is a noisy estimate of the true risk function. 

To correct for this issue, we will now see that we should instead find a value $\lambda$ such that the coverage is at least $(1-\alpha)(1+1/n)$ on the calibration data. That is, we should choose $\hat\lambda$ as
\begin{equation}
    \label{eq:conformal-plugin-estimate-lambdahat}
    \hat\lambda = \inf\left\{\lambda \in \R : \hat{R}(\lambda) \leq \alpha'\right\},
\end{equation}
where $\alpha' = \alpha - \frac{1-\alpha}{n}$.
The risk threshold $\alpha'$ is slightly lower than the naive value of $\alpha$, indicating that we must be slightly more conservative than the naive plug-in approach to guarantee coverage.
This choice of $\hat\lambda$ is exactly equivalent to the split conformal prediction algorithm; the next proposition formalizes this connection.  See Figure~\ref{fig:split-CP-qhat} for a visualization.

\begin{proposition}[Split conformal prediction via a plug-in estimate of risk]\label{prop:conformal-equivalent-plugin}
    The split conformal prediction set $\cC(X_{n+1})$ defined in~\eqref{eq:split-conformal-sets} satisfies
    \begin{equation}
        \cC(X_{n+1}) = \left\{y\in\cY: s(X_{n+1}, y) \leq \hat\lambda \right\},
    \end{equation}
    where $\hat\lambda$ is defined in~\eqref{eq:conformal-plugin-estimate-lambdahat}.
\end{proposition}
\begin{proof}[Proof of Proposition~\ref{prop:conformal-equivalent-plugin}]
    It suffices to prove that $\hat\lambda$ is equal to the conformal quantile $\hat{q}$, i.e., the $(1-\alpha)(1+1/n)$ quantile of the calibration scores $s(X_1,Y_1),\dots,s(X_n,Y_n)$.
    By definition of $\hat{R}$ and of $\alpha'$, we calculate
    \begin{align}
        \hat{R}(\lambda) \leq \alpha' & \Longleftrightarrow \frac{1}{n}\sum\limits_{i=1}^n \ind{S_i > \lambda} \leq \alpha-\frac{1-\alpha}{n} \\
        & \Longleftrightarrow \frac{1}{n}\sum\limits_{i=1}^n \ind{s(X_i, Y_i) \leq \lambda} \geq 1 - \left(\alpha-\frac{1-\alpha}{n}\right) =  (1-\alpha)(1+1/n).
    \end{align}
    Plugging this into the definition of $\hat\lambda$, we have
    \begin{equation}
        \hat\lambda = \inf\left\{ \lambda \in \R : \frac{1}{n}\sum\limits_{i=1}^n \ind{s(X_i, Y_i) \leq \lambda} \geq (1-\alpha)(1+1/n) \right\},
    \end{equation}
    which is exactly the definition of the conformal quantile  $\hat{q}$ for split conformal.
\end{proof}

Finally, we give a proof of Theorem~\ref{thm:full-conformal} (for the special case of split conformal) from the perspective of plug-in estimates of error rates.
(Later on, in Section~\ref{sec:conformal_risk_control}, we will extend this argument to other notions of error besides miscoverage.)
\begin{proof}[Proof of Theorem~\ref{thm:full-conformal} (split conformal case) via tuning an estimate of error]
    We will begin by defining an oracle threshold:
    \begin{equation}
        \lambda^* = \inf\left\{\lambda \in \R : \frac{1}{n+1}\sum\limits_{i=1}^{n+1} \ind{S_i > \lambda} \leq \alpha \right\}.
    \end{equation}
    The quantity $\lambda^*$ is an oracle version of $\hat\lambda$ that depends on the full dataset $\cD_{n+1}$, including both the calibration and test data (while, in contrast, $\hat\lambda$ depends only on the calibration data $\cD_n$). Since $\lambda^*$ is a symmetric function of the $n+1$ data points, exchangeability of the data implies that
    \begin{equation}\label{eqn:lambda_tilde_exchangeability}
    \P(S_{n+1} > \lambda^*) = \P(S_i > \lambda^*),
    \end{equation}
    for all $i\in[n+1]$ (recall Lemma~\ref{lem:conditional_exchangeability}).
    
    Below, we will show that $\lambda^* \leq \hat\lambda$ must always hold. This is enough to prove the result, since we then have
    \begin{align}
    \P(Y_{n+1}\not\in\cC(X_{n+1}))
    &=\P(S_{n+1})>\hat\lambda)\textnormal{ \ by Proposition~\ref{prop:conformal-equivalent-plugin}}\\
        &\leq \P(S_{n+1} > \lambda^*) \textnormal{ \ since $\lambda^* \leq \hat\lambda$}\\
        &= \frac{1}{n+1}\sum_{i=1}^{n+1}\P(S_i>\lambda^*)\\
        &= \E\left[\frac{1}{n+1}\sum\limits_{i=1}^{n+1}  \ind{S_i > \lambda^*} \right] \leq \alpha,
    \end{align}
    where the third step holds by~\eqref{eqn:lambda_tilde_exchangeability}, while the last step applies the definition of $\lambda^*$.

    To conclude, we show that $\lambda^* \leq \hat\lambda$.
    For all $\lambda \in \R$,
    \begin{align}
        \frac{1}{n}\sum\limits_{i=1}^{n} \ind{S_i > \lambda} \leq \alpha' &\Longleftrightarrow \frac{1}{n+1}\sum\limits_{i=1}^{n} \ind{S_i > \lambda} + \frac{1}{n+1} \leq \alpha \\
        & \implies \frac{1}{n+1}\sum\limits_{i=1}^{n+1} \ind{S_i > \lambda} \leq \alpha.
    \end{align}
    This means the infimum in the definition of $\lambda^*$ is taken over a larger set than the infimum in the definition of $\hat\lambda$, proving the desired claim.
\end{proof}
\index{risk control|)}

\section{Can conformal prediction be overly conservative?}\label{sec:conformal-overcover}
\index{coverage!upper bound|(}
So far, we have discussed lower bounds on coverage, which ensure that the prediction sets will cover the ground truth with \emph{at least} probability $1-\alpha$.
But this leaves open the possibility that the sets are unnecessarily large.
This section gives a distribution-free bound on exactly how conservative conformal prediction can be---i.e., how much the coverage probability could potentially exceed the target level $1-\alpha$.
\begin{theorem}[Bounding the overcoverage of conformal prediction]
    \label{thm:upper-bound}
    Under the conditions of Theorem~\ref{thm:full-conformal}, we have that
    \begin{equation}
        \P(Y_{n+1} \in \cC(X_{n+1})) \leq \frac{\lceil (1-\alpha)(n+1)\rceil}{n+1} + \epsilon_{\rm tie} \leq 1-\alpha+ \frac{1}{n+1}+\epsilon_{\rm tie},
    \end{equation}
    where $\epsilon_{\rm tie}$ captures the likelihood of the score of the $(n+1)$st data point being tied with any other data point,
    \begin{equation}
        \epsilon_{\rm tie} = \P\left(\exists j\in[n], \ S_{n+1} = S_j\right).
    \end{equation}
\end{theorem}
This theorem says that the coverage of conformal prediction is not too far from $1-\alpha$ as long as the distribution of the scores is unlikely to produce ties.
Without making further assumptions on the model, however, this does not directly translate to a bound on the \emph{size} of the prediction set $\cC(X_{n+1})$.
Instead, the bound says that this set is not conservative on the scale of coverage.
For example, for a residual score $s$ of the form $s((x,y);\cD) = |y - \hf(x;\cD)|$, 
a model $\hf(\cdot;\cD)$ that is a very poor fit to the data distribution will necessarily lead to wide prediction intervals. However, this result is telling us that the prediction intervals are no wider than is needed to compensate for the errors in $\hf(\cdot;\cD)$.

\begin{proof}[Proof of Theorem~\ref{thm:upper-bound}]
In our first proof of Theorem~\ref{thm:full-conformal} (given in Section~\ref{sec:full_conformal_first_proof}), we derived an equivalent characterization of the coverage event (see~\eqref{eqn:standard-proof-coverage-equiv_Yn+1}):
    \begin{equation}
        Y_{n+1} \in \cC(X_{n+1}) \Longleftrightarrow S_{n+1} \leq \quantile\left(S_1,\dots,S_{n+1};1-\alpha\right) .
    \end{equation}
    Writing $S_{(1)}\leq \dots\leq S_{(n+1)}$ as the order statistics of the scores $S_1,\dots,S_{n+1}$, we can equivalently write
    \begin{equation}
    \label{eq:cvg_score_proof_helper_eq}
    Y_{n+1} \in \cC(X_{n+1}) \Longleftrightarrow S_{n+1}\leq S_{(k)} \textnormal{ where $k=\lceil (1-\alpha)(n+1)\rceil$,}
    \end{equation}
    by the equivalence between quantiles and order statistics (Fact~\ref{fact:conversion-order-stats-quantiles}). We can assume $k \in[n]$ to avoid the trivial case.

    Next, by definition of the order statistics,
    \[S_{n+1}\leq S_{(k)} \Longleftrightarrow\textnormal{ either $S_{n+1} < S_{(k+1)}$ or $S_{n+1} = S_{(k)} = S_{(k+1)}$}.\]
    If $S_{n+1} = S_{(k)} = S_{(k+1)}$ then we must have $S_{n+1}=S_j$ for some $j\in[n]$, and therefore,
    \begin{multline*}\P(Y_{n+1}\in\cC(X_{n+1})) = \P(S_{n+1} < S_{(k+1)}) + \P(S_{n+1} = S_{(k)} = S_{(k+1)})\\ \leq \P(S_{n+1} < S_{(k+1)}) +\epsilon_{\rm tie}.\end{multline*}
    Finally, to complete the proof, we have 
    \[\P(S_{n+1} < S_{(k+1)}) \leq \frac{(k+1)-1}{n+1} = \frac{\lceil (1-\alpha)(n+1)\rceil}{n+1}\]
    by Fact~\ref{fact:exchangeable-properties}\ref{fact:exchangeable-properties_part1}.
\end{proof}
An important special case is when the scores are distinct almost surely (e.g., if the joint distribution of the scores is continuous). In this case, we can compute the coverage exactly.
\begin{proposition}
    \label{prop:upper-bound}
    Under the conditions of Theorem~\ref{thm:full-conformal}, if the scores $S_1,\dots,S_{n+1}$ are distinct almost surely, then
    \begin{equation}
 \P(Y_{n+1} \in \cC(X_{n+1})) = \frac{\lceil (1-\alpha)(n+1)\rceil}{n+1} \leq 1-\alpha + \frac{1}{n+1}.
\end{equation}
\end{proposition}
\begin{proof}[Proof of Proposition~\ref{prop:upper-bound}]
    We return to the calculation~\eqref{eq:cvg_score_proof_helper_eq}. When the scores are distinct with probability $1$, exchangeability implies that $P(S_{n+1} \le S_{(k)}) = \frac{k}{n+1} = \frac{\lceil (1-\alpha)(n+1)\rceil}{n+1}$ (by Fact~\ref{fact:exchangeable-properties}\ref{fact:exchangeable-properties_part4}), as desired. 
\end{proof}
It should be noted that the requirement of a continuous joint distribution, though not entirely distribution-free, is not very restrictive; for instance, this often holds for continuously distributed data (e.g., using the residual score, $s((x,y);\cD) = |y - \hf(x;\cD)|$).
Finally, we close by noting that a strengthening of Theorem~\ref{thm:upper-bound} is available via a randomized version of conformal, which avoids the issues of ties---we develop this randomized algorithm in Section~\ref{sec:cp-random}.
\index{coverage!upper bound|)}

\section*{Bibliographic notes}
\addcontentsline{toc}{section}{\protect\numberline{}\textnormal{\hspace{-0.8cm}Bibliographic notes}}

The earliest works describing a form of conformal prediction were by Gammerman, Vovk, and Vapnik~\citep{gammerman1998learning} and by Vovk, Gammerman, and Saunders~\citep{vovk1999machine}.
The work was extended into what we now know as conformal prediction by Vovk, Gammerman, and their students and collaborators: Saunders, Nouretdinov, Papadopoulos,
and Proedrou~\citep{saunders1999transduction,nouretdinov2001ridge, papadopoulos2002inductive,proedrou2002transductive,gammerman2007hedging}.
Importantly,~\cite{papadopoulos2002inductive} introduces split conformal prediction.
This framework is related to earlier ideas in the statistics literature, including tolerance regions \citep{wilks1941,wilks1942} and rank tests \citep{mann1947test,lehmann1953power}, and builds on early definitions of randomness developed by Kolmogorov \citep{kolmogorov1965three,kolmogorov1968logical,kolmogorov1983combinatorial}.
The book \emph{Algorithmic Learning in a Random World}, by Vovk, Gammerman, and Shafer, summarizes and extends upon these early contributions, and was the first book on the topic~\citep{vovk2005algorithmic}.
The book also chronicles the early history of conformal.
A tutorial on conformal prediction is available in~\cite{shafer2008tutorial}.
These references describe conformal prediction and many early applications, e.g., to support vector machines; they are the main references for the core ideas underlying full conformal prediction. 

More recently, there has been great interest in conformal prediction across the statistics and machine learning research communities. 
Our presentation in Sections~\ref{sec:define_conformal}--\ref{alg:split-cp} more closely follows the notation and style developed in more recent statistical literature---in particular, the seminal paper of Lei, G'Sell, Rinaldo, Tibshirani, and Wasserman~\citep{lei2018distribution}.
The problem of defining a score function to ensure good performance, which we preview in Section~\ref{sec:score-preview}, has emerged as a central practical question in the field; some popular choices are proposed and examined by \citet{vovk2005algorithmic,lei2013conformal,lei2018distribution,romano2019conformalized,chernozhukov2021distributional,Sadinle2016LeastAS}, among many others. The Replacement Lemma (Lemma~\ref{lem:n+1-to-n-reduction}) appears implicitly across many works in the literature; the closest explicit statement that we are aware of can be found in the proof of \citet[Lemma 1]{tibshirani2019conformal}. Our illustration of the implementation of full conformal prediction, shown in Figure~\ref{fig:full-cp}, is inspired by illustrations developed by Ryan Tibshirani in various lectures and tutorials. 

The various interpretations and proofs of conformal prediction presented in Section~\ref{sec:reinterpret-conformal} have been emerging since the aforementioned early works by Vovk, Gammerman, and colleagues.
The equivalence to permutation tests via the formulation of a conformal p-value (Section~\ref{sec:conformal_as_perm}) is developed in~\cite{vovk2005algorithmic} and~\cite{vovk2013transductive}; see also \cite{kuchibhotla2020exchangeability}.
The reinterpretation via conditioning on the empirical distribution of the data (Section~\ref{sec:conformal-conditioning-bag}) has also been discussed by early works such as \cite{vovk2005algorithmic,vovk2013transductive} (often described equivalently in the language of conditioning on a multiset or `bag' of data points). The reinterpretation via plug-in estimation (Section~\ref{sec:CP_plugin}) is due to \cite{angelopoulos2022conformal}. See also \cite{gupta2020distribution} for an additional reinterpretation of conformal prediction via families of nested sets. \index{bag of data points}

The version of the upper bound in Proposition~\ref{prop:upper-bound} is proven by \cite{lei2018distribution}.
The randomized version of conformal, which we mentioned in Section~\ref{sec:conformal-overcover} and will cover in more detail later on  in Section~\ref{sec:cp-random}, is discussed in \citet{vovk2005algorithmic,lei2018distribution}.

\section*{Exercises}
\addcontentsline{toc}{section}{\protect\numberline{}\textnormal{\hspace{-0.8cm}Exercises}}
\begin{enumerate}[font=\bfseries, label={\thechapter.\arabic*}, labelsep=1em, itemsep=1em]
\item \label{exercise:asymm_split_CP}
Consider the following prediction interval, for the setting of a real-valued response: given a pretrained model $\hf$, and calibration points $(X_1,Y_1),\dots,(X_n,Y_n)$, define
\[\cC(X_{n+1}) = [\hf(X_{n+1}) - \hat{q}_L, \hf(X_{n+1}) + \hat{q}_R]\]
where $\hat{q}_L = \quantile( -Y_i + \hf(X_i);(1-\alpha/2)(1+1/n))$, $\hat{q}_R = \quantile(Y_i - \hf(X_i);(1-\alpha/2)(1+1/n))$. 
Prove that, if data points $(X_1,Y_1),\dots,(X_{n+1},Y_{n+1})$ are exchangeable, then this interval offers an equal-tailed coverage guarantee: the probability of miscoverage on each of the two sides of the prediction interval is bounded as
\[\P(Y_{n+1} < \hf(X_{n+1}) - \hat{q}_L) \leq \alpha/2, \quad \P(Y_{n+1} > \hf(X_{n+1}) + \hat{q}_R) \leq \alpha/2.\]
(Note that this is a strictly stronger result than the usual marginal predictive coverage guarantee for split conformal prediction,
$\P(Y_{n+1}\in\cC(X_{n+1}))\geq 1-\alpha$.)
\item 
    This problem considers conformal prediction when the feature is constant. This can be understood as the case where there is no feature, but we include a constant feature so that this fits our notation from the chapter. Formally, consider an exchangeable sequence $(X_1, Y_1), \dots (X_{n+1}, Y_{n+1})$ where $X_i = 1$ for all $i$ with probability 1.
    Let $(Y_1,\dots,Y_{n+1})$ be multivariate Gaussian, where $Y_i \sim \mathcal N(0,\sigma^2)$ marginally and $\Cov(Y_i, Y_j) = \rho$ for $i \ne j$, for some $\rho\in[0,\sigma^2)$. These random variables are exchangeable but not independent when $\rho \ne 0$. The purpose of this question is to compare prediction intervals from conformal prediction with a different, model-based approach.
    \begin{enumerate}
    \item Suppose that $\rho$ and $\sigma^2$ are known. Give a prediction interval for $Y_{n+1}$ that has exactly $1-\alpha$ coverage via the conditional distribution of $Y_{n+1}\mid (Y_1,\dots,Y_n)$.
    \item Next, consider a split conformal approach. Here we will use half of the $n$ data points for training a score function, and the remaining half for calibration. Taking $n$ to be even for simplicity, let $\bar{Y}_{n/2} = \frac{1}{n/2}\sum_{i=1}^{n/2}Y_i$ be the sample mean of the first half of the data, and define the score function $s(x,y) = |y - \bar{Y}_{n/2}|$. Give the prediction set obtained by running split conformal prediction using this score function, with data points $i=n/2+1,\dots,n$ as the calibration set.
    \item Discuss how the two approaches above relate: if the equicorrelated Gaussian model for the data is true, how do we expect these two prediction sets to compare to each other? You may assume $n$ is large and $\rho>0$.
    \item If instead the equicorrelated Gaussian model for the data is incorrect, 
     how do we expect these two prediction sets to compare to each other? You may assume that $n$ is large and the data satisfies exchangeability.
\end{enumerate}
\item\label{exercise:split_CP_two_test_points} Consider split conformal prediction with the fixed score function $s(x,y)$. Suppose we have an exchangeable data set $((X_i, Y_i))_{i\in[n+2]}$, and assume the scores are distinct almost surely. Let $\cC$ be the split conformal prediction set~\eqref{eq:split-conformal-sets} when using $((X_i, Y_i))_{i\in[n]}$ as the calibration data. Compute
    \begin{equation}
        \P\left(Y_{n+1} \in \cC(X_{n+1}) \textnormal{ and } Y_{n+2} \in \cC(X_{n+2})\right).
    \end{equation}
\item In this exercise, we will examine the role of the symmetry assumption on the score function $s$ in full conformal prediction.
Construct an example of a score function $s((x,y);\cD)$ that is \emph{not} symmetric in $\cD$, 
such that for some sample size $n$ and for data $(X_1,Y_1),\dots,(X_{n+1},Y_{n+1})\iidsim P$ for some distribution $P$, the full conformal prediction interval exhibits undercoverage, $\P(Y_{n+1}\in\cC(X_{n+1})) < 1-\alpha$.
\item Consider split conformal prediction with a fixed score function $s(x,y)$. Suppose we have data $((X_i, Y_i))_{i\in[n]}$ that are i.i.d. draws from a distribution $P$. Let $P_s$ denote the distribution of $s(X_i, Y_i)$, and let $q^* = \quantile(P_s; 1-\alpha)$. Assume that $P_s$ has a positive density in a neighborhood of $q^*$. Writing $\hat{q}$ to denote the conformal quantile as in~\eqref{eq:split-conformal-quantile}, prove that as $n \to \infty$, $\hat{q}$ converges to $q^*$ in probability. 
\item In Theorem~\ref{thm:upper-bound}, we proved that the coverage of conformal prediction can exceed $1-\alpha$ by, at most, $\frac{1}{n+1} + \epsilon_{\textnormal{tie}}$, the probability that the test score is tied with any other score. 
\begin{enumerate}
    \item Prove an alternative upper bound: under the same conditions and notation as in Theorem~\ref{thm:upper-bound}, 
    \[\P(Y_{n+1}\in\cC(X_{n+1})) \leq 1-\alpha + \E\left[\frac{\sum_{i=1}^{n+1}\ind{S_i = \hat{q}}}{n+1}\right],\]
    where $\hat{q}=\quantile(S_1,\dots,S_{n+1};1-\alpha)$.
    \item Give an example where the upper bound above is much tighter than the result of Theorem~\ref{thm:upper-bound}.
    \end{enumerate}
\item In the setting of a real-valued response $\cY=\R$, suppose we have two pretrained functions, $\hf_1$ and $\hf_2$. Given calibration points $(X_1,Y_1),\dots,(X_n,Y_n)$ and a test point $X_{n+1}$, we can construct two different split conformal prediction sets:
\[\cC_j(X_{n+1}) = \{y\in\R : s_j(X_{n+1},y)\leq \hat{q}_j\} \textnormal{ where }\hat{q}_j = \quantile((s_j(X_i,Y_i))_{i\in[n]};(1-\alpha)(1+1/n)),\]
for each $j=1,2$, where $s_j(x,y) = |y-\hf_j(x)|$ is the residual score for fitted model $\hf_j$.
Suppose we choose to use the pretrained model that has lower error on the calibration set,
    \[\textnormal{Choose $\cC_J(X_{n+1})$ for }J = \argmin_{j\in\{1,2\}}\frac{1}{n}\sum_{i=1}^n (Y_i - \hf_j(X_i))^2.\]
    Explain why the marginal predictive coverage of this procedure,
    \[\P(Y_{n+1} \in \cC_J(X_{n+1})),\]
    is no longer guaranteed to be $\geq 1-\alpha$, even when we assume exchangeability of the data.
\end{enumerate}

\chapter{Conditional Coverage}\label{chapter:conditional}

Conformal prediction guarantees marginal coverage---on average over all the data---but in practice, we often want to ensure stronger guarantees.
A marginal coverage guarantee does not rule out the possibility of a scenario where we might observe an unusual batch of training data, leading us to miscover on nearly all future data points. Nor does it prevent issues of uneven coverage across different parts of the population: for example, we might have two subgroups in our data (say, higher values of $X$ and lower values of $X$), where the prediction intervals returned by conformal prediction are likely to systematically undercover for test points in one group while overcovering for test points in the other. (See Figure~\ref{fig:conditional-coverage} for an illustration of this phenomenon.)
If we want to ensure that such issues are unlikely to arise, we need to examine the \emph{conditional coverage} properties of conformal prediction.

\begin{figure}[t]
    \centering
    \includegraphics{\diagramspath conditional-coverage.pdf}
    \caption{\textbf{An illustration of marginal and conditional coverage.} The interval on the left satisfies marginal coverage, but does not satisfy coverage conditional on the value of the test point $X_{n+1}$---there is undercoverage for larger values of $X_{n+1}$. In contrast, the interval on the right satisfies both properties.}
    \commentAlt{Two scatterplots of the same dataset of $(X,Y)$ pairs, with an increasing trend and higher variability at larger $X$ values. The left panel shows a shaded region of constant width at all $X$, and the right shows a shaded region with increasing width.}
    \label{fig:conditional-coverage}
\end{figure}

In this chapter, we study conditional coverage in its various forms and whether it can be achieved distribution-free.
In the simplest cases, such as achieving coverage conditional on the event $X_{n+1} \in \cX_0$ for some large region $\cX_0\subseteq\cX$, we will provide procedures that achieve conditional guarantees, but in other contexts, we will establish hardness results---that is, results showing that achieving some types of conditional coverage is impossible without further assumptions.

\section{Training-conditional coverage for split conformal prediction}
\index{coverage!training-conditional|(}
We first examine the coverage of split conformal prediction conditional on the training set:
\begin{equation}
\P(Y_{n+1} \in \cC(X_{n+1}) \mid \cD_n).
\end{equation}
This is the empirical coverage that we would obtain if we used the points $\cD_n = ((X_1,Y_1),\dots,(X_n,Y_n))$ for calibration, and then evaluated our coverage over an infinite test set of i.i.d.\ data.
This quantity is a random variable, since it depends on the calibration set $\cD_n$. 

If the conformal prediction set satisfies marginal coverage at level $1-\alpha$, then the training-conditional coverage must be $\geq 1-\alpha$ in expectation: 
\[ \E\left[\P(Y_{n+1} \in \cC(X_{n+1}) \mid \cD_n)\right] = \P(Y_{n+1} \in \cC(X_{n+1})) \geq 1-\alpha.\]
However, a marginal coverage guarantee does not ensure that training-conditional coverage will be close to $1-\alpha$ \emph{with high probability}; for this, we need a stronger concentration-type statement.

In this section, we will consider split conformal prediction, where $s((x,y); \cD)$ has no dependence on $\cD$ (see Section~\ref{sec:split-conformal}). Throughout this section, we will assume the data points $(X_i,Y_i)$ are i.i.d., rather than the more general assumption of exchangeability. We will see that, for i.i.d.\ data, the training-conditional coverage will indeed concentrate near (or above) the nominal level $1-\alpha$, for the split conformal method.

In fact, there is a closed-form expression that characterizes the training-conditional coverage in this setting. For intuition, let us first consider a simplified setting where the score distribution is continuous (so that there are no ties, almost surely), and where $(1-\alpha)(n+1)$ is an integer. In that case, training-conditional coverage is exactly distributed as a Beta distribution:
\begin{equation}
    \label{eq:beta-training-conditional}
     \P\left(Y_{n+1} \in \cC\left(X_{n+1}\right) \big| \: \cD_n\right) \sim \textrm{Beta}\left((1-\alpha)(n+1),\alpha(n+1)\right).
\end{equation}
The Beta distribution arises from the order statistics of the uniform distribution---conformal prediction relies only on the ranks of the scores, so if there are no ties one can without loss of generality 
analyze conformal scores that are uniformly distributed.
More generally, without these assumptions, the Beta distribution provides a lower bound on the training-conditional coverage, as we see in the following result:

\begin{theorem}[Training-conditional coverage for split conformal prediction]\label{thm:training-conditional-split-CP}
Suppose the data points $(X_i,Y_i)$ are i.i.d., and let $\cC(X_{n+1})$ be constructed via split conformal prediction (Algorithm~\ref{alg:split-cp}) using any pretrained score function $s$. Then the training-conditional coverage
$\P\left(Y_{n+1} \in \cC\left(X_{n+1}\right) \big| \: \cD_n\right)$ stochastically dominates the $\textrm{Beta}\left((1-\alpha)(n+1),\alpha(n+1)\right)$ distribution, and in particular, for any $\Delta\geq 0$,
\[\P\Bigg( \P\left(Y_{n+1} \in \cC\left(X_{n+1}\right) \big| \: \cD_n\right) \, \leq \,   1-\alpha-\Delta \Bigg) \leq F_{\textrm{Beta}\left((1-\alpha)(n+1),\alpha(n+1)\right)}(1-\alpha-\Delta)  \, \leq \,  e^{-2n\Delta^2},\]
where $F_{\textrm{Beta}(a,b)}$ denotes the CDF of the $\textrm{Beta}(a,b)$ distribution.
\end{theorem}

One consequence of this result is an answer to the question, how large should the calibration set be for split conformal prediction? We know that the marginal coverage will be at least $1-\alpha$ for any $n$, but the above theorem tells us that a larger calibration set is still better because it leads to coverage that lies closer to $1-\alpha$ conditionally: 
since this Beta distribution has expected value $1-\alpha$ and variance $\frac{\alpha(1-\alpha)}{n+2} = \bigo(\frac{1}{n})$, the theorem implies that training-conditional coverage for split conformal can be lower-bounded as
\[\P(Y_{n+1}\in\cC(X_{n+1})\mid \cD_n) \geq 1- \alpha - \epsilon_n,\]
for an error term $\epsilon_n$ that has standard deviation scaling as $n^{-1/2}$.
See Figure~\ref{fig:beta} for a visualization of the distribution of conditional coverage for different values of $n$.

\begin{figure}[t]
    \centering
    \includegraphics[width=0.7\linewidth]{\diagramspath training-conditional-coverage-distribution.pdf}
    \caption{\textbf{The distribution of coverage of split conformal prediction conditional on the training data}, when the scores have no ties. The distribution concentrates around $1-\alpha$ with rate $n^{-1/2}$.}
    \commentAlt{The plot shows three curves, plotted against a horizontal axis ranging from $0.8$ to $1.0$. The three curves are all approximately bell-curve shaped, though slightly asymmetric. See long description.}
    \commentLongAlt{The plot shows three curves, plotted against a horizontal axis ranging from $0.8$ to $1.0$. The three curves are all approximately bell-curve shaped, though slightly skewed towards higher values. The curves correspond to the distribution of the training-conditional coverage level at $n=100$, $n=1000$, and $n=10000$. The target coverage level $0.9$ indicated by a vertical dashed line. The curve for $n=10000$ is very narrow, with nearly all mass placed extremely close to the target value of $0.9$. The curve for $n=1000$ is wider, and the curve for $n=100$ is extremely wide compared to the others.}
    \label{fig:beta}
\end{figure}

\begin{proof}[Proof of Theorem~\ref{thm:training-conditional-split-CP}]
Throughout this proof, we will treat the pretrained score function as fixed---formally, recalling Algorithm~\ref{alg:split-cp}, we are conditioning on the pretraining set $\cD_{\rm pre}$. 

Let $F$ be the CDF of the distribution of $s(X,Y)$, when the data point is sampled as $(X,Y)\sim P$. Define $S_i = s(X_i,Y_i)$ for $i\in[n+1]$, which are i.i.d.\ draws from the distribution with CDF $F$.
Next, let $S_{(1)}\leq \dots \leq S_{(n)}$ be the order statistics of $S_1,\dots,S_n$. By definition of split conformal, we have $Y_{n+1}\in\cC(X_{n+1})$ if and only if $S_{n+1}\leq S_{(k)}$ for $k = \lceil (1-\alpha)(n+1)\rceil$, and so we have
\[\P(Y_{n+1}\in\cC(X_{n+1})\mid \cD_n) = \P(S_{n+1}\leq S_{(k)}\mid \cD_n) = F(S_{(k)}),\]
where the last step holds since $S_{n+1}$ is independent of $\cD_n$ and has CDF $F$, while $S_{(k)}$ is a function of $\cD_n$. Therefore,
\[\P\big(\P(Y_{n+1}\in\cC(X_{n+1})\mid \cD_n) \leq 1-\alpha-\Delta\big) = \P\big(F(S_{(k)}) \leq1-\alpha-\Delta\big). \]

Next let $U_i = F(S_i)$ for all $i\in[n]$. Then the $U_i$'s are i.i.d., and by the basic property of CDFs~\eqref{eq:CDF_basic_fact}, their shared distribution is superuniform---that is, $\P(U_i\leq\tau)\leq\tau$ for all $\tau\in[0,1]$. Note also that, by monotonicity of $F$, the order statistics $U_{(1)}\leq\dots\leq U_{(n)}$ of $U_1,\dots,U_n$ satisfy $U_{(k)}=F(S_{(k)})$ (see Fact~\ref{fact:monotone-invariance-quantiles}\ref{fact:monotone-invariance-quantiles_part1}). Then
\[\P\big(F(S_{(k)})\leq 1-\alpha-\Delta\big) = \P\big(U_{(k)} \leq 1-\alpha-\Delta\big) \leq \P\big(U^*_{(k)}\leq 1-\alpha-\Delta\big),\]
where $U^*_1,\dots,U^*_n$ are i.i.d.\ $\textnormal{Unif}[0,1]$ random variables, with order statistics $U^*_{(1)}\leq \dots\leq U^*_{(n)}$. Finally, by definition of the Beta distribution, the $k$th order statistic from $n$ i.i.d.\ uniform random variables has distribution
\[U^*_{(k)}\sim\textnormal{Beta}(k,n+1-k).\]
This distribution stochastically dominates the $\textrm{Beta}\left((1-\alpha)(n+1),\alpha(n+1)\right)$ distribution, since $k\geq (1-\alpha)(n+1)$.
This proves the first inequality claimed in the theorem.

The second inequality (i.e., the bound  on the CDF of the Beta distribution) is due to the fact that the $\textrm{Beta}\left((1-\alpha)(n+1),\alpha(n+1)\right)$ distribution has mean $1-\alpha$, and is $\frac{1}{4n}$-subgaussian.
\end{proof}

\subsection{High-probability coverage}
\label{sec:training-conditional-pac}
As described above, the marginal coverage property is equivalent to a bound on the \emph{expected} training-conditional miscoverage,
\begin{equation}
    \P\Big(Y_{n+1} \in \cC(X_{n+1})\Big) \ge 1 - \alpha \iff \E\left[\P\left(Y_{n+1} \not\in \cC\left(X_{n+1}\right) \big| \: \cD_n\right)\right] \leq  \alpha,
\end{equation}
where the expectation is over the calibration data $\cD_n$. 
We can also consider a different approach: instead of bounding the expected error, we might wish to bound the error with some desired probability---that is, in this setting, we enforce a bound on the chance that we end up with poor training-conditional coverage. Concretely, for some $\delta \in (0,1)$, we could ask for a prediction set $\cC(X_{n+1})$ that satisfies
\begin{equation}
\label{eq:high-prob-control}
    \P\left( \ \P\left(Y_{n+1} \not\in \cC\left(X_{n+1}\right) \big| \: \cD_n\right)\le \alpha \ \right) \ge 1- \delta.
\end{equation}
In other words, there is only a $\delta$ chance that we draw a calibration set $\cD_n$ leading to a miscoverage rate on future data that is higher than $\alpha$.

In the special case of i.i.d.\ data, we can produce sets $\cC(X_{n+1})$ that satisfy the error-control property in~\eqref{eq:high-prob-control}, with a slight modification of split conformal prediction. Taking as a starting point the distribution of training-conditional coverage given in Theorem~\ref{thm:training-conditional-split-CP}, we see that by defining
\[\cC(X_{n+1}) = \left\{y\in\cY : s(X_{n+1},y)\leq \quantile\left(S_1,\dots,S_n;(1-\alpha')(1+1/n)\right)\right\}\]
(i.e., split conformal run at a nominal level $1-\alpha'$), then
we will achieve the desired high probability bound~\eqref{eq:high-prob-control} if we choose the appropriate nominal level $1-\alpha'$. In particular, given a fixed $\delta>0$, we can take $\alpha'$ to be the unique value that satisfies
\[F_{\textnormal{Beta}((1-\alpha')(n+1),\alpha'(n+1))}(1-\alpha)  = \delta.\]
By our calculations on the mean and variance of the Beta distribution, above, we can see that $\alpha' = \alpha - \bigo(n^{-1/2})$---that is, this is simply a slightly more conservative version of split conformal prediction.

\subsection{Exchangeable data or i.i.d.\ data?}
It is important to note that for this chapter, our positive results will primarily be in the setting of i.i.d.\ data. In contrast, for the marginal results of Chapter~\ref{chapter:conformal-exchangeability} (and for related results later on in the book as well), an assumption of exchangeability is sufficient. 
The reason for the distinction, in the case of training-conditional coverage for split conformal, is that we need the sample quantiles of the empirical distribution of scores in the calibration set, $(S_i)_{i\in[n]}$, to concentrate around their corresponding true quantiles. If we only assume exchangeability rather an i.i.d.\ assumption, then this concentration phenomenon is no longer guaranteed to hold, making training-conditional coverage impossible to guarantee.

\section{A hardness result for training-conditional coverage}
\label{sec:training-conditional-full-CP}

We have seen in Theorem~\ref{thm:training-conditional-split-CP} above
that split conformal prediction automatically satisfies a training-conditional coverage guarantee (as long as the data points are assumed to be i.i.d., rather than merely exchangeable). 
The following result shows that, without further assumptions, training-conditional coverage is \emph{not} guaranteed for full conformal prediction. 
Essentially, the reason is that, for split conformal in the i.i.d.\ data setting, the scores $S_i=s(X_i,Y_i)$ are i.i.d.\ (when we treat the pretrained conformal score function $s$ as fixed); for full conformal, on the other hand, the scores $S_i=s((X_i,Y_i);\cD_{n+1})$ are \emph{not} i.i.d.\ and thus may not exhibit the same favorable concentration type properties.

To make this more precise, we first introduce the terminology of a \emph{nonatomic} distribution.

\begin{definition}[Atoms and nonatomic distributions]\label{def:nonatomic}
For a distribution $P$ on $\cZ$, the \emph{atoms} of $P$ are all points $z\in\cZ$ with positive probability,
\[\textnormal{atom}(P) = \{z\in\cZ : \P_P(Z=z)>0\}.\]
The distribution $P$ is called \emph{nonatomic} if it has no atoms, i.e., $\textnormal{atom}(P)=\varnothing$ (the empty set).
\end{definition}
In the case of a real-valued random variable, $\cZ=\R$, a nonatomic distribution $P$ is also called \emph{continuous}, which includes any distribution that has a density with respect to Lebesgue measure.
In $\R^d$ when $d>1$, a distribution that has a density with respect to Lebesgue measure must again be nonatomic, but many common nonatomic distribution do not have a density: for example, the uniform distribution on the unit sphere is nonatomic.

We are now ready to state the theorem, which establishes the impossibility of guaranteeing training-conditional coverage for full conformal prediction in the setting of a nonatomic feature distribution.

\begin{theorem}[Failure of training-conditional coverage for full conformal prediction]\label{thm:training-conditional-full-CP}
    Let $P$ be any distribution on $\cX\times\cY$ such that the marginal $P_X$ is nonatomic. Then there exists a symmetric conformal score function $s$ such that, when running full conformal prediction with this choice of $s$, the training-conditional coverage probability  satisfies
    \[\P\bigg( \P(Y_{n+1} \in \cC(X_{n+1}) \mid \cD_n) = 0\bigg)\geq \alpha - \bigo\left(\sqrt{\frac{\log n}{n}}\right),\]
    where the probability is taken with respect to the training set $\cD_n=((X_1,Y_1),\dots,(X_n,Y_n))$ and test point $(X_{n+1},Y_{n+1})$ drawn i.i.d.\ from $P$.
\end{theorem} \index{hardness result!training-conditional coverage}

As we will see below, this theorem is proved by constructing an explicit counterexample, whose score function $s((x,y);\cD)$ is highly degenerate.
Of course, it will often be the case that, for the types of score functions we might encounter in practice, the prediction sets returned by full conformal may exhibit training-conditional coverage; this result is simply telling us that this type of property cannot be guaranteed to hold universally across all possible score functions. Likewise, if $X$ is instead discrete, then it may be possible to ensure training-conditional coverage.

\begin{proof}[Proof of Theorem~\ref{thm:training-conditional-full-CP}]
Let $\alpha_P(\cD_n) = \P(Y_{n+1} \notin \cC(X_{n+1}) \mid \cD_n)$.
Fix a large integer $N\approx \alpha n$---the exact value will be specified later. First, since $P_X$ is nonatomic, we can define a function $a:\cX\rightarrow\{0,\dots,n-1\}$ such that $a(X)\sim \textnormal{Unif}\{0,\dots,n-1\}$ when $X\sim P_X$. 
Consider the following score function: for dataset $\cD = ((x_1,y_1),\dots,(x_k,y_k))$ and an additional data point $(x,y)$, define
\[s\big((x,y);\cD\big) = \ind{\textnormal{mod}\left( -a(x) + \sum_{j=1}^k a(x_j)  , n \right) < N}.\]

Next consider applying this score to a test data point $(X_{n+1},Y_{n+1})$ within the dataset $\cD_{n+1} = ((X_1,Y_1),\dots,(X_n,Y_n),(X_{n+1},Y_{n+1}))$. 
Recall from the proof of Theorem~\ref{thm:full-conformal}  that the coverage event $Y_{n+1}\in\cC(X_{n+1})$ is equal to the event that $S_{n+1}\leq \quantile\left(S_1,\dots,S_{n+1};1-\alpha\right)$, where the scores are defined as $S_i = s((X_i,Y_i);\cD_{n+1})$ (see~\eqref{eqn:standard-proof-coverage-equiv_Yn+1}). Since the scores always take values in $\{0,1\}$ in this construction, coverage can fail only when we simultaneously have $S_{n+1}=1$ and $\quantile\left(S_1,\dots,S_{n+1};1-\alpha\right)=0$. In other words, we have
\[\alpha_P(\cD_n) = \P\left(S_{n+1}=1, \quantile\left(S_1,\dots,S_{n+1};1-\alpha\right)=0\mid \cD_n\right).\]
From this point on, the main idea of the proof will be as follows: we will define an event $\cE_{\textnormal{mod}}$, such that if $\cE_{\textnormal{mod}}$ occurs (which has probability $\approx \alpha$), then we will likely have high training-conditional miscoverage $\alpha_P(\cD_n)$.

Formally, we define
$\cE_{\textnormal{mod}}$ as the event that 
$\textnormal{mod}\left(\sum_{i=1}^n a(X_i)  , n \right) < N$. 
With our particular definition of the score function, we have
\[S_{n+1} = s\big((X_{n+1},Y_{n+1});\cD_{n+1}\big) = \ind{\textnormal{mod}\left(  \sum_{i=1}^n a(X_i)  , n \right) < N} = \indsub{\cE_{\textnormal{mod}}}.\]
Since this event only depends on the training data $\cD_n$ and not on the test point  $(X_{n+1},Y_{n+1})$, we can rewrite the conditional probability of noncoverage as
\[\alpha_P(\cD_n) = \indsub{\cE_{\textnormal{mod}}}\cdot \P\left( \quantile\left(S_1,\dots,S_{n+1};1-\alpha\right)=0\mid \cD_n\right).\]
In order to show that the event $\cE_{\textnormal{mod}}$ leads to high $\alpha_P(\cD_n)$, then, we now need to verify that the above conditional probability is likely to be high.

To do so, we need to define another event, $\cE_{\textnormal{unif}}$. First, we define a series of sliding windows: for any integer $k$, define
\[W_k = \big\{i\in \{0,\dots,n-1\} : \textnormal{mod}(-i + k - 1 , n )\geq N\big\}.\]
For example, we have $W_0=\{0,\dots,n-N-1\}$, $W_1 = \{1,\dots,n-N\}$, and so on.
Then, let $\cE_{\textnormal{unif}}$ be the event that
\begin{equation}\label{eqn:Ecal_unif}\sum_{i=1}^n\ind{a(X_i)\in W_k} \geq (1-\alpha)(n+1)\textnormal{ for all integers $k$},\end{equation}
i.e., each window $W_k$ contains a sufficient fraction of the sample.
This event is illustrated in Figure~\ref{fig:E_unif_illustration}.

For each training point $(X_i,Y_i)$, we can calculate
\[S_i = s\big((X_i,Y_i);\cD_{n+1}\big) = \ind{\textnormal{mod}\left( -a(X_i) + \sum_{j=1}^{n+1} a(X_j)  , n \right) < N}.\] 
By taking $k = 1 + \sum_{j=1}^{n+1} a(X_j)$, we have $S_i=\ind{a(X_i)\not\in W_k}$. We can see that if $\cE_{\textnormal{unif}}$ holds, then $\sum_{i=1}^n\ind{S_i=0} \geq (1-\alpha)(n+1)$ and therefore $\quantile\left(S_1,\dots,S_{n+1};1-\alpha\right) = 0$.
Therefore,
\[\alpha_P(\cD_n) \geq \indsub{\cE_{\textnormal{mod}}}\cdot \indsub{\cE_{\textnormal{unif}}}.\]
In other words, we have shown that
\[\P\left(\alpha_P(\cD_n)=1\right) \geq \P\left(\cE_{\textnormal{mod}}\cap \cE_{\textnormal{unif}}\right) \geq \P\left(\cE_{\textnormal{mod}}\right) -  \P\left(\cE_{\textnormal{unif}}^c\right)= \frac{N}{n} - \P\left(\cE_{\textnormal{unif}}^c\right).\]
Here the last step holds since $a(X_i)\iidsim \textnormal{Unif}\{0,\dots,n-1\}$ for each $i=1,\dots,n$, and therefore $\textnormal{mod}\left(\sum_{i=1}^n a(X_i),n\right)$ also follows this uniform distribution, i.e., $\P(\cE_{\textnormal{mod}}) = N/n$. 

Our last step is to bound $\P\left(\cE_{\textnormal{unif}}^c\right)$. First, we observe that, by definition of the mod function, $\cE_{\textnormal{unif}}$ holds if and only if $\sum_{i=1}^n\ind{\textnormal{mod}(-a(X_i)+k-1,n) \geq N} \geq (1-\alpha)(n+1)$ for all $k=0,\dots,n-1$, and so
\[\P\left(\cE_{\textnormal{unif}}^c\right) \leq \sum_{k=0}^{n-1} \P\left( \sum_{i=1}^n\ind{\textnormal{mod}(-a(X_i)+k-1,n) \geq N} < (1-\alpha)(n+1)\right).\]
Moreover, since $a(X_i)\iidsim \textnormal{Unif}\{0,\dots,n-1\}$, we see that
\[\sum_{i=1}^n\ind{\textnormal{mod}(-a(X_i)+k-1,n) \geq N}\sim \textnormal{Binomial}\left(n,1-\frac{N}{n}\right)\]
for each fixed $k$, and so
\[\P\left(\cE_{\textnormal{unif}}^c\right) \leq n \cdot \P\left(\textnormal{Binomial}\left(n,1-\frac{N}{n}\right) < (1-\alpha)(n+1)\right).\]
Taking $N = \alpha n - \sqrt{n\log n} - (1-\alpha)$, by Hoeffding's inequality,
\begin{multline*}\P\left(\textnormal{Binomial}\left(n,1-\frac{N}{n}\right) < (1-\alpha)(n+1)\right)
\leq \exp\left\{ -\frac{2}{n} \left[ n-N - (1-\alpha)(n+1)\right]^2\right\}\\
= \exp\left\{ -\frac{2}{n} \left[ \sqrt{n\log n}\right]^2\right\} = \frac{1}{n^2}.
\end{multline*}
Therefore, $\P\left(\cE_{\textnormal{unif}}^c\right)\leq 1/n$, and we then have
\[\P\left(\alpha_P(\cD_n)=1\right) \geq  \frac{N}{n} - \P\left(\cE_{\textnormal{unif}}^c\right) = \frac{\alpha n - \sqrt{n \log n} + (1-\alpha)}{n} - \frac{1}{n} = \alpha - \bigo\left(\sqrt{\frac{\log n}{n}}\right),\]
which completes the proof.
\end{proof}
\index{coverage!training-conditional|)}

\begin{figure}[t]
\centering
\includegraphics[width=0.9\textwidth]{\diagramspath E_unif.pdf}
\caption{\textbf{Illustration of the event $\cE_{\textnormal{unif}}$}, which is defined in the proof of Theorem~\ref{thm:training-conditional-full-CP}. Each plot shows a different random draw of the histogram of the values $a(X_i)\in\{0,\dots,n-1\}$, for training points $i\in[n]$, where $n=50$. Since we have $n$ data points and $n$ values, we have 1 observation per value on average, but the random draws are quite noisy. However, setting $N=20$, a sliding window of length $n-N=30$ typically contains approximately $n-N$ many data points. For each draw of the data, we highlight two examples of the sliding window of indices $W_k$, at $k=0$ and $k=15$. For the left-hand plot, the event $\cE_{\rm unif}$ holds---the window $W_k$ contains a sufficient fraction of the observed $a(X_i)$ values, for every $k\in\{0,\dots,n-1\}$. For the right-hand plot, the event $\cE_{\rm unif}$ fails---in particular, the window $W_{0}$ contains too few of the $a(X_i)$ values.}
\commentAlt{Each of the two panels shows a histogram with a large number of bins, with the first bin labeled as $0$ and the last bin labeled as $n-1$. The vertical axis is labeled as the `count'. See long description.}
\commentLongAlt{Each of the two panels shows a histogram with a large number of bins, with the first bin labeled as $0$ and the last bin labeled as $n-1$. The vertical axis is labeled as the `count', and the histogram bars show several different heights, including many bars that are absent (i.e., a count of zero). In both panels, a curly bracket highlighting a stretch of bins beginning with bin $0$ is labeled as `Window $W_0$', and another curly bracket highlighting a stretch of bins farther along the sequence is labeled as `Window $W_{15}$'. In the left panel, the histogram is fairly uniformly spread over the $n$ bins, and in particular the two highlighted windows each contain roughly equal counts. In the right panel, the histogram is less uniformly spread out, with a long stretch of consecutive bins that have a zero count, and consequently one of the two windows has a much lower total count than the other.}
\label{fig:E_unif_illustration}
\end{figure}

\section{Conditioning on the test point features}\label{sec:condition_test_point}
\index{coverage!test-conditional|(}

Next we turn to the problem of conditioning on the test features $X_{n+1}$, rather than the training data.
We call this `test-conditional coverage', meaning coverage conditional on the test \emph{feature} $X_{n+1}$ only, as is common in the literature.
In its strongest possible version, achieving test-conditional coverage would require that our construction of the prediction sets satisfies the following property:
for data drawn i.i.d.\ from any distribution, 
\begin{equation}
    \P\big(Y_{n+1}\in\cC(X_{n+1})\mid X_{n+1}\big)\geq 1-\alpha
\end{equation}
holds almost surely. (In particular, a guarantee of this type would ensure that we avoid the problem of uneven coverage illustrated in Figure~\ref{fig:conditional-coverage}.)
    
In this section, we will explore this version of the test-conditional coverage property, and will study its challenges.
Afterward, in later sections, we will examine relaxations of the property.

\subsection{The discrete setting}\label{sec:test_conditional_discrete}

First, we will consider a setting where test-conditional coverage can indeed be achieved: the discrete setting. To take a simple case, suppose that $\cX = \{x_1,\dots,x_K\}$ is a finite set---effectively, the feature $X_i$ for data point $i$ simply identifies which of $K$ many groups this data point belongs to. If the sample size $n$ is far larger than the number of feature groups $K$, we might expect that test-conditional coverage is easy to achieve. Indeed, we will now see that while conformal prediction may not itself achieve test-conditional coverage, a simple modification will enable this property. 

First, why doesn't conformal prediction already satisfy this property?  Conformal prediction is guaranteed to satisfy a marginal coverage guarantee, which we can rewrite for this setting as
\[1-\alpha \leq \P\big(Y_{n+1}\in\cC(X_{n+1})\big) = \sum_{k=1}^K \P(X_{n+1}=x_k)\cdot  \P\big(Y_{n+1}\in\cC(X_{n+1})\mid X_{n+1}=x_k\big).\]
This means that the test-conditional coverage values, $\P\big(Y_{n+1}\in\cC(X_{n+1})\mid X_{n+1}=x_k\big)$ for $k\in[K]$, are guaranteed to be $\geq 1-\alpha$ in the (weighted) average, but some may be lower than this threshold. 

Now we will see how to modify conformal prediction to achieve test-conditional coverage in this setting.
For simplicity, let us consider split conformal prediction, as defined earlier in Algorithm~\ref{alg:split-cp}. We recall that the split conformal prediction interval is given by the construction
\[\cC(x) = \{y : s(x,y) \leq \hat{q}\},\]
where $s:\cX\times\cY\to\R$ is the pretrained conformal score function, and where
\[\hat{q} = \quantile\left(S_1, \ldots, S_n ; (1-\alpha)(1+1/n)\right)\]
computes the quantile of the conformal scores on the calibration set. To achieve test-conditional coverage in the finite setting $\cX = \{x_1,\dots,x_K\}$, we will instead define
\begin{equation}\label{eqn:test_conditional_coverage_discrete_X}\cC(x_k) = \{y : s(x_k,y) \leq \hat{q}_k\},\end{equation}
for each $k\in[K]$, where
\[\hat{q}_k = \quantile\left((S_i)_{i\in [n], X_i = x_k} ; (1-\alpha)(1+1/n_k)\right)\]
instead computes the quantile of scores among only those data points $i$ for which $X_i = x_k$, and where $n_k = \sum_{i\in [n]} \indsub{X_i = x_k}$ counts the number of such points.
(To account for the case where a value $x_k$ might not appear in the calibration set, we will define the quantile of an empty list to be $+\infty$---that is, if $X_i\neq x_k$ for all $i\in[n]$, then $\hat{q}_k = +\infty$.)
This modified procedure will now satisfy the following coverage guarantee:
\begin{theorem}[Discrete test-conditional coverage guarantee]\label{thm:discrete_conditional}
    Suppose $\cX = \{x_1,\dots,x_K\}$, and $(X_1,Y_1),...,(X_{n+1},Y_{n+1})$ are exchangeable. Then the prediction interval $\cC(X_{n+1})$ defined in~\eqref{eqn:test_conditional_coverage_discrete_X} satisfies\\\[\P(Y_{n+1}\in\cC(X_{n+1})\mid X_{n+1} = x_k) \geq 1-\alpha,\]
     for all $k\in[K]$ with $\P(X_{n+1} = x_k)>0$.
\end{theorem}
We defer a proof of this claim to Section~\ref{sec:test-conditional-binning}, where we will see that this is a special case of a more general result.

We note that, if the distribution of $s(X,Y)$ conditional on $X=x_k$ is fairly similar for each $k\in[K]$, then the group-wise quantiles $\hat{q}_k$ will be fairly similar as well; in this setting, the original split conformal method may already be nearly achieving test-conditional coverage. On the other hand, if these $K$ distributions are substantially different, then the $\hat{q}_k$'s may be quite different; a single conformal quantile $\hat{q}$ that guarantees marginal coverage might lead to uneven coverage across the $K$ groups.

\subsection{Impossibility in the continuous setting}\label{sec:test-conditional-continuous-impossible}

In the previous section, we considered the discrete setting, where the feature $X$ takes only a small number of possible values. In this section, we will instead consider the other extreme: what if $X$ is continuously distributed? Can test-conditional coverage be guaranteed via any method, whether conformal prediction or some other procedure? To answer this question, we will now write $\cC$ to denote \emph{any} procedure that uses a training dataset of size $n$ to construct prediction sets. Formally, $\cC$ is a map from $(\cX\times\cY)^n\times \cX$ to the set of subsets of $\cY$, but we will continue to write $\cC(X_{n+1})$ for the resulting prediction set, i.e., we suppress the dependence on the training data $(X_1,Y_1),\dots,(X_n,Y_n)$ in the notation.

We will now consider the setting where, under the joint distribution $P$ of the pair $(X,Y)$, the marginal distribution $P_X$ of $X$ is nonatomic (recall Definition~\ref{def:nonatomic}). In this setting, we will see that test-conditional inference is impossible.

What does `impossible' mean in this context? For example, if we simply return $\cC(x) = \cY$ for any test feature $x$ (e.g., returning $(-\infty,\infty)$ as our prediction interval in the case of a real-valued response), then any desired coverage guarantee will hold. 
We can even be slightly less conservative by producing a randomized answer: for any $x$, return 
\begin{equation}\label{eqn:trivial_test_conditional}\cC(x) = \begin{cases}\cY,&\textnormal{ with probability $1-\alpha$},\\  \varnothing, & \textnormal{ with probability $\alpha$}.\end{cases}\end{equation} 
This solution is clearly still unsatisfactory, however, since the resulting prediction interval is completely uninformative---it does not even depend on the training data.
However, the following result will demonstrate that this trivial solution is the best that we can do.

\begin{theorem}[Hardness of test-conditional coverage for nonatomic distributions]\label{thm:conditional_infinite}
    Suppose $\cC$ is any procedure  that satisfies distribution-free conditional coverage, i.e., for any distribution $P$ on $\cX\times\cY$, 
    \[\P\big(Y_{n+1}\in\cC(X_{n+1})\mid X_{n+1}\big)\geq 1-\alpha\]
    holds almost surely, where the probability is taken with respect to $(X_1,Y_1),\cdots,(X_{n+1},Y_{n+1})\iidsim P$, and where $\cC$ implicitly depends on $(X_1,Y_1),\dots,(X_n,Y_n)$. Then, for any distribution $P$ on $\cX\times\cY$ for which the marginal $P_X$ is nonatomic, 
    \[\P\big(y\in \cC(x)\big)\geq 1-\alpha\textnormal{ for every $(x,y)\in\cX\times\cY$}.\]
\end{theorem} \index{hardness result!test-conditional coverage}
To see more concretely how this result establishes the impossibility of nontrivial test-conditional inference, in the case that $\cY = \R$, we can consider its implications for the  \emph{length} (or more formally, the Lebesgue measure) of the prediction interval:
\begin{corollary}\label{cor:infinite-lebesgue}
    Under the assumptions of Theorem~\ref{thm:conditional_infinite}, suppose also that $\cY=\R$. Then, for any distribution $P$ on $\cX\times\cY$ for which the marginal $P_X$ is nonatomic, and for every $x\in\cX$, it holds that
    \[\P\big(\textnormal{Leb}(\cC(x)) = \infty\big)\geq 1-\alpha,\]
    where $\textnormal{Leb}(\cdot)$ denotes the Lebesgue measure. 
\end{corollary}
In particular, this implies that the prediction set
constructed at the test point $X_{n+1}$ must have infinite expected length, $\E[\textnormal{Leb}(\cC(X_{n+1}))] = \infty$.

Now let's compare this to the trivial solution described above in~\eqref{eqn:trivial_test_conditional}, where, regardless of the training data, we return $\cC(x) = \cY$ with probability $1-\alpha$, and $\cC(x)=\varnothing$ otherwise. In that case, we have $\P(y\in\cC(x)) = 1-\alpha$ for all $(x,y)$. The result of Theorem~\ref{thm:conditional_infinite} can therefore be interpreted as saying that, in the nonatomic setting, \emph{any} procedure with distribution-free test-conditional coverage will return an output that is just as uninformative as this trivial construction.

Before proceeding, we pause to point out an important aspect of Theorem~\ref{thm:conditional_infinite} (and its corollary): the bounds apply to \emph{every} distribution $P$ (with nonatomic $P_X$), and we are not simply proving that nontrivial test-conditional coverage is impossible for \emph{some} particularly challenging distribution $P$. That is, we might have expected that we could construct some valid $\cC$ that returns reasonably narrow intervals for `nice' distributions $P$, and is excessively wide only for distributions $P$ that fail to satisfy some needed conditions (e.g., smoothness)---but instead we see from this theorem that, if $\cC$ achieves test-conditional coverage for every $P$, then it also returns a completely uninformative solution for every (nonatomic) $P$.

We are now ready to prove these results. For Theorem~\ref{thm:conditional_infinite}, the key idea of the proof is to construct a perturbation of the distribution $P$: we consider the mixture distribution
\[P' = (1-\epsilon)\cdot P + \epsilon \cdot \delta_{(x,y)},\]
where as before, $\delta_{(x,y)}$ denotes the point mass at $(x,y)$. That is, to sample a point from the distribution $P'$, we either draw from $P$ (with probability $1-\epsilon$), or we simply return the point $(x,y)$ (with probability $\epsilon$).
See Figure~\ref{fig:hardness-test-conditional} for an illustration.
The proof will use the fact that $\cC$ is assumed to satisfy test-conditional coverage with respect to any distribution of the data---and in particular, must therefore offer test-conditional coverage relative to the perturbed distribution $P'$.

\begin{figure}[t]
    \centering
    \includegraphics[width=0.75\linewidth]{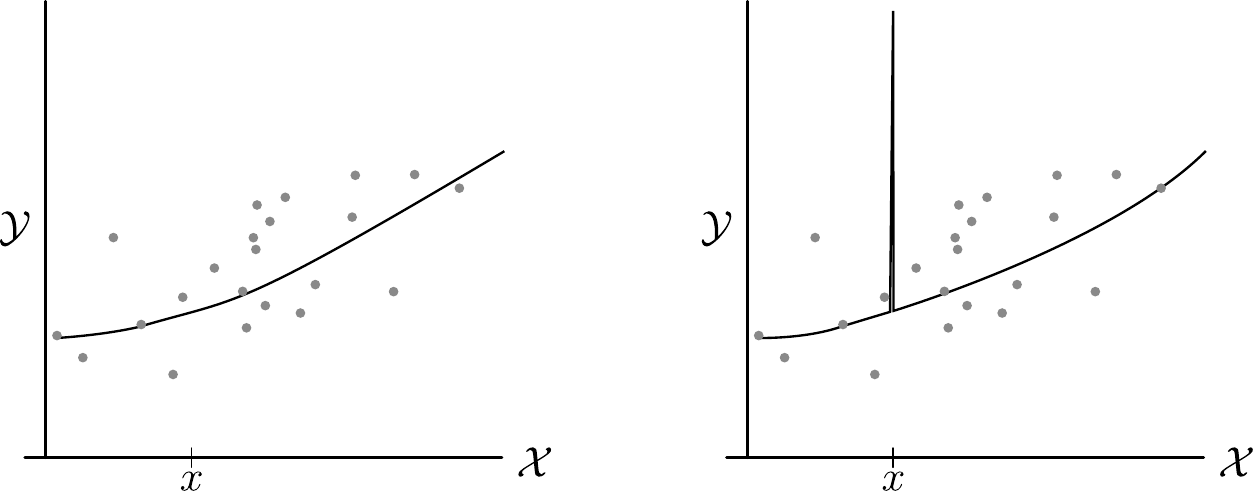}
    \caption{\textbf{Illustration of the hardness of conditional coverage,} and the idea of the proof of Theorem~\ref{thm:conditional_infinite}. The two panels show the same observed dataset of $(X,Y)$ pairs, with two different possibilities for the regression function that determines the distribution of $Y\mid X$. In one panel, the regression function has a large spike at the value $x$, while in the other, the regression function is smooth. Since $X_i = x$ is an unlikely event, this spike would not be observed in the training data, which suggests that it will be challenging to guarantee coverage conditional on $X_{n+1}=x$.}
    \commentAlt{Two panels show the same scatterplot of $(X,Y)$ pairs, overlaid with two different curves. In the left panel, the curve is smooth. In the right panel, the curve has been altered by adding a sharp spike at a value $x$ indicated on the horizontal axis.}
    \label{fig:hardness-test-conditional}
\end{figure}

\begin{proof}[Proof of Theorem~\ref{thm:conditional_infinite}]
Fix any $(x,y)\in\cX\times\cY$, and any $\epsilon>0$. Define a mixture distribution,
\[P' = (1-\epsilon)\cdot P + \epsilon \cdot \delta_{(x,y)}.\]
Since $\cC$ offers distribution-free conditional coverage, this means that $\cC$ must satisfy the coverage guarantee with respect to data drawn from the distribution $P'$---that is, 
\[\P_{P'}\big(Y_{n+1}\in\cC(X_{n+1})\mid X_{n+1}\big)\geq 1-\alpha\]
holds almost surely, where the  notation $\P_{P'}(\dots)$ indicates
that we are calculating probability with respect to training and test data
 $(X_1,Y_1),\dots,(X_{n+1},Y_{n+1})\iidsim P'$. Since the event $X_{n+1}=x$ has positive probability under $P'$, in particular this means that this bound must hold on the event $X_{n+1}=x$:
 \[\P_{P'}(Y_{n+1}\in\cC(X_{n+1})\mid X_{n+1}=x)\geq 1-\alpha.\]Moreover, by definition of $P'$ together with the assumption that $P_X$ is nonatomic, we can see that
$\P_{P'}\big(Y_{n+1}=y\mid X_{n+1}=x\big) = 1$, and so 
combining these facts, we must have
\[\P_{P'}\big(y\in\cC(X_{n+1})\mid X_{n+1}=x\big)\geq 1-\alpha,\]
or equivalently,
\[\P_{P'}\big(y\in\cC(x)\big)\geq 1-\alpha.\]
Note that the event no longer depends on the test point, and so we are calculating probability with respect to the draw of the training data only, i.e., with respect to $(X_1,Y_1),\dots,(X_n,Y_n)\iidsim P'$.

Next, observe that the result above holds for an arbitrarily small $\epsilon>0$. By taking $\epsilon\rightarrow0$, we have shown that the same bound must hold for data drawn from $P$, rather than from $P'$. To formalize this argument, we can compute
\[\P_{P}\big(y\in\cC(x)\big) \geq \P_{P'}\big(y\in\cC(x)\big) - \dtv(P^n,P'{}^n)\geq 1-\alpha  - \dtv(P^n,P'{}^n),\]
where $\dtv$ denotes the total variation distance, and where $P^n$ and $P'{}^n$ denote the corresponding product distributions, i.e., the distribution of an i.i.d.\ sample of size $n$ drawn from $P$ or from $P'$, respectively. 
Since $\dtv(P^n,P'{}^n) \leq n\dtv(P,P') = n\epsilon$, we therefore have
\[\P_{P}\big(y\in\cC(x)\big)  \geq 1-\alpha - n\epsilon.\]
Since this holds for any $\epsilon>0$, this completes the proof.
\end{proof}

Finally, to complete this section, we present a proof of the corollary, which verifies that in this setting, the interval $\cC(X_{n+1})$ must often have infinite length.
\begin{proof}[Proof of Corollary~\ref{cor:infinite-lebesgue}]
Fix any $x\in\cX$ and fix any constants $a,b>0$. Then, deterministically, it holds that
\[\textnormal{Leb}(\cC(x)) =\int_{y\in\R}\ind{y\in\cC(x)}\;\mathsf{d}y \geq \int_{y=0}^{a+b}\ind{y\in\cC(x)}\;\mathsf{d}y ,\]
and so
\[\textnormal{Leb}(\cC(x))\leq a \Longrightarrow \int_{y=0}^{a+b}\ind{y\in\cC(x)}\;\mathsf{d}y\leq a \Longleftrightarrow \int_{y=0}^{a+b}\ind{y\not\in\cC(x)}\;\mathsf{d}y\geq b.\]
By Markov's inequality,
\begin{multline*}
\P\left(\textnormal{Leb}(\cC(x))\leq a\right) \leq \P\left(\int_{y=0}^{a+b}\ind{y\not\in\cC(x)}\;\mathsf{d}y\geq b\right) \\ \leq \frac{\E\left[\int_{y=0}^{a+b}\ind{y\not\in\cC(x)}\;\mathsf{d}y\right]}{b} = \frac{\int_{y=0}^{a+b}\P(y\not\in\cC(x))\;\mathsf{d}y}{b} \leq \frac{(a+b)\alpha}{b},\end{multline*}
where the last step holds by Theorem~\ref{thm:conditional_infinite}, while the next-to-last step applies the Fubini--Tonelli theorem to swap order of integration. Taking $b\to\infty$, then, we have shown that
\[\P\left(\textnormal{Leb}(\cC(x))\leq a\right)\leq \alpha.\]
Since this holds for all finite $a>0$, we therefore have
\[\P\left(\textnormal{Leb}(\cC(x))<\infty\right)\leq \alpha.\]
\end{proof}

\section{Relaxations of test-conditional coverage: a binning-based approach}
\label{sec:test-conditional-binning}

So far, we have seen that it is impossible to achieve pointwise test-conditional coverage, i.e., conditional on $X_{n+1}=x$, in the nonatomic setting, but straightforward to achieve in the discrete setting. This discrepancy arises from the fact that the event $X=x$ might be observed many times in the training set in the discrete case, but (for any prespecified value $x$) will never occur in the training set for the nonatomic case. This suggests a way to relax the goal of conditional coverage in order to avoid the hardness result: we should condition on events that will be observed many times in a sample of size $n$. Specifically, we can consider partitioning the feature space into $K$ pre-specified bins, $\cX = \cX_1 \cup \dots \cup \cX_K$, and then requiring a bin-wise notion of conditional coverage:
\begin{equation}\label{eqn:test_cond_coverage_by_bin}\P(Y_{n+1}\in\cC(X_{n+1})\mid X_{n+1}\in\cX_k)\geq 1-\alpha \textnormal{ for all $k\in[K]$}.\end{equation}

Just like in the discrete case, the split conformal procedure can be modified to achieve this guarantee, by simply calculating a conformal quantile separately within each bin $k$: defining $k(x)\in[K]$ as the index of the bin containing $x$, we construct the prediction set
\begin{equation}\cC(X_{n+1}) = \{y : s(X_{n+1},y) \leq \hat{q}_{k(X_{n+1})}\},\end{equation}
where the score function $s$ is pretrained (i.e., this method is a variant of split conformal prediction), and where for each $k$,
\[\hat{q}_k = \quantile\left((S_i)_{i\in [n], X_i \in\cX_k} ; (1-\alpha)(1+1/n_k)\right)\]
for $n_k = \sum_{i\in [n]} \ind{X_i \in \cX_k}$ denoting the sample size for the $k$th bin.

Implicitly, in this split conformal version of the construction, we are assuming that the bins are fixed---that is, the bins are chosen independently of the data (or are chosen using the pretraining set). This type of construction can be extended to full conformal prediction as well, by defining
\begin{equation}
    \label{eqn:test_cond_coverage_by_bin_C}
    \cC(X_{n+1}) = \{y : S_{n+1}^y \leq \hat{q}_{k(X_{n+1})}^y\},
\end{equation}
where now the quantile is given by
\[\hat{q}^y_k = \quantile\left((S_i^y)_{i\in [n], X_i \in\cX_k} ; (1-\alpha)(1+1/n_k)\right),\]
for scores $S_i^y$ defined exactly as in Chapter~\ref{chapter:conformal-exchangeability}. Of course, as usual, the bin-wise split conformal method defined above is simply a special case of this procedure.

The following result verifies that this modified full conformal method achieves bin-conditional coverage:
\begin{theorem}[Bin-wise test-conditional coverage guarantee]\label{thm:bin_conditional}
    Suppose $(X_1,Y_1),...,(X_{n+1},Y_{n+1})$ are exchangeable. Then the bin-wise prediction interval $\cC(X_{n+1})$ defined in~\eqref{eqn:test_cond_coverage_by_bin_C} satisfies\[\P(Y_{n+1}\in\cC(X_{n+1})\mid X_{n+1}\in\cX_k)\geq 1-\alpha \]
     for all $k\in[K]$ with $\P(X_{n+1}\in\cX_k)>0$.
\end{theorem}
Pointwise coverage in the discrete case (i.e., the guarantee that $\P(Y_{n+1}\in\cC(X_{n+1})\mid X_{n+1}=x_k)\geq 1-\alpha$, for each $k$, which we covered in Section~\ref{sec:test_conditional_discrete} above) is actually a special case of this result: in the setting where $|\cX|=K$, we can simply take each bin $\cX_k$ to be a singleton set.

The key idea for the proof of this theorem lies in the following lemma:

\begin{lemma}[Conditional exchangeability within a bin]\label{lem:exch_conditional_subvector}
Suppose $(X_1,Y_1),\dots,(X_{n+1},Y_{n+1})$ are exchangeable. Fix any subset $\cZ_0\subseteq\cX\times\cY$, and for any fixed nonempty subset $I\subseteq[n+1]$, let $\cE_I$ be the event that $\{i \in[n+1]: (X_i,Y_i)\in\cZ_0\} = I$. If $\cE_I$ has positive probability, then $((X_i,Y_i))_{i\in I}$ is exchangeable conditional on $\cE_I$.
\end{lemma} \index{exchangeability!conditional}
In particular, this result has a useful interpretation in terms of conformal p-values:
\begin{corollary}\label{cor:pvalues_due_to_lem:exch_conditional_subvector}
    In the setting of Lemma~\ref{lem:exch_conditional_subvector}, define
    \[p = \frac{1 + \sum_{i\in[n]}\ind{(X_i,Y_i)\in\cZ_0,\ S_i\geq S_{n+1}}}{1 + \sum_{i\in[n]}\ind{(X_i,Y_i)\in\cZ_0}},\]
where $S_i = s((X_i,Y_i);\cD_{n+1})$ for some symmetric score function $s$. If $\P((X_{n+1},Y_{n+1})\in\cZ_0)>0$, then
\[\P(p\leq\alpha\mid (X_{n+1},Y_{n+1})\in\cZ_0)\leq \alpha\]
for any $\alpha\in[0,1]$.
\end{corollary}

We will use these results here to prove the bin-wise conditional coverage guarantee of Theorem~\ref{thm:bin_conditional}, and we will also use them again for several related results later on in the chapter.

\begin{proof}[Proof of Theorem~\ref{thm:bin_conditional}]
First, by construction of the set $\cC(X_{n+1})$, we can see that on the event $X_{n+1}\in\cX_k$, 
\begin{equation}\label{eqn:conformal-via-pvalues_for_test_conditional}Y_{n+1}\in \cC(X_{n+1}) \Longleftrightarrow \frac{1+\sum_{i\in[n],X_i\in\cX_k}\ind{S_i\geq S_{n+1}}}{1+n_k} >  \alpha\end{equation}
(the proof of this statement is analogous to the proof of Proposition~\ref{prop:conformal-via-pvalues}).

Next, fix any bin $k$ with $\P(X_{n+1}\in \cX_k)>0$, and let $\cZ_0 = \cX_k \times\cY \subseteq \cX\times \cY$. By Corollary~\ref{cor:pvalues_due_to_lem:exch_conditional_subvector}, the quantity
\[p = \frac{1+\sum_{i\in[n]}\ind{(X_i,Y_i)\in\cZ_0, \ S_i \geq S_{n+1}}}{1+\sum_{i\in[n]}\ind{(X_i,Y_i)\in\cZ_0}}\]
satisfies $\P(p\leq \alpha\mid X_{n+1}\in\cX_k) = \P(p\leq \alpha\mid (X_{n+1},Y_{n+1})\in\cZ_0)\leq \alpha$. But by definition of $\cZ_0$ and of $n_k$, we can rewrite this quantity as 
\[p = \frac{1+\sum_{i\in[n],X_i\in\cX_k}\ind{S_i\geq S_{n+1}}}{1+n_k}.\]
By~\eqref{eqn:conformal-via-pvalues_for_test_conditional}, this completes the proof.
\end{proof}

Finally, we prove the lemma and its corollary.
\begin{proof}[Proof of Lemma~\ref{lem:exch_conditional_subvector}]
For convenience, throughout this proof we will write $Z_i = (X_i,Y_i)$, and $Z=(Z_1,\dots,Z_{n+1})$.
Fix any nonempty $I\subseteq[n+1]$ such that $\P(\cE_I)>0$. Fix any permutation $\sigma$ on $I$, and let $\tilde{\sigma}$ be the permutation on $[n+1]$ defined as
\[\tilde{\sigma}(i) = \begin{cases}\sigma(i), & i\in I,\\ i, & i\not\in I.\end{cases}\]
As before, we will write $Z_{\tilde\sigma} = (Z_{\tilde\sigma(1)},\dots,Z_{\tilde\sigma(n+1)})$, and note that $Z\eqd Z_{\tilde\sigma}$ by exchangeability of the data.

For any $A\subseteq(\cX\times\cY)^{|I|}$, we have
\begin{align*}
    &\P((Z_i)_{i\in I}\in A, \ \cE_I)\\
    &=\P\big((Z_i)_{i\in I}\in A, \ \{i\in[n+1] :  Z_i\in\cZ_0\} = I\big)\\
    &=\P\big((Z_{\tilde{\sigma}(i)})_{i\in I}\in A, \ \{i\in[n+1] :  Z_{\tilde{\sigma}(i)}\in\cZ_0\} = I\big)\textnormal{\quad since $Z\eqd Z_{\tilde\sigma}$}\\
    &=\P\big((Z_{\tilde{\sigma}(i)})_{i\in I}\in A, \ \{i\in[n+1] :  Z_i\in\cZ_0\} = I\big)\textnormal{\quad  since $\tilde\sigma(i)\in I$ if and only if $i\in I$}\\ 
    &=\P((Z_{\tilde{\sigma}(i)})_{i\in I}\in A, \ \cE_I)\\
    &=\P((Z_{\sigma(i)})_{i\in I}\in A, \ \cE_I),
\end{align*}
where the last step holds by definition of $\tilde\sigma$. Therefore,
\begin{multline*}\P((Z_i)_{i\in I}\in A\mid \cE_I) = \frac{\P((Z_i)_{i\in I}\in A, \cE_I)}{\P(\cE_I)} \\=\frac{\P((Z_{\sigma(i)})_{i\in I}\in A, \cE_I)}{\P(\cE_I)} =  \P((Z_{\sigma(i)})_{i\in I}\in A\mid \cE_I).\end{multline*}
Since this holds for every $A$, this proves that $(Z_i)_{i\in I}$ and $(Z_{\sigma(i)})_{i\in I}$ are equal in distribution conditional on $\cE_I$, for any permutation $\sigma$ on $I$. This verifies exchangeability of $(Z_i)_{i\in I}$ conditional on $\cE_I$, as desired.
\end{proof}
\begin{proof}[Proof of Corollary~\ref{cor:pvalues_due_to_lem:exch_conditional_subvector}]
Fix any $I$ with $\P(\cE_I)>0$ and with $n+1\in I$.
Lemma~\ref{lem:exch_conditional_subvector} verifies that, conditional on $\cE_I$, $(Z_i)_{i\in I}$ is exchangeable. It consequently holds that the corresponding scores, $(S_i)_{i\in I}$, are also exchangeable conditional on $\cE_I$. To verify this formally, following the same steps as in the proof of Lemma~\ref{lem:exch_conditional_subvector}, for any $A\subseteq\R^{|I|}$ we calculate
\begin{align*}
    &\P((S_i)_{i\in I}\in A, \ \cE_I)\\
    &=\P\big((s(Z_i;\cD_{n+1}))_{i\in I}\in A, \{i\in[n+1] :  Z_i\in\cZ_0\} = I\big)\\
    &= \P\big((s(Z_{\tilde\sigma(i)};(\cD_{n+1})_{\tilde\sigma}))_{i\in I}\in A, \{i\in[n+1] :  Z_{\tilde\sigma(i)}\in\cZ_0\} = I\big)\textnormal{ since $\cD_{n+1}\eqd (\cD_{n+1})_{\tilde\sigma}$}\\
    &= \P\big((s(Z_{\tilde\sigma(i)};\cD_{n+1}))_{i\in I}\in A, \{i\in[n+1] :  Z_{\tilde\sigma(i)}\in\cZ_0\} = I\big)\textnormal{ since $s$ is symmetric}\\    &=\P\big((s(Z_{\tilde\sigma(i)};\cD_{n+1}))_{i\in I}\in A, \{i\in[n+1] :  Z_i\in\cZ_0\} = I\big)\textnormal{ since $\tilde\sigma(i)\in I$ $\Leftrightarrow$ $i\in I$}\\
    &=\P\big((s(Z_{\tilde\sigma(i)};\cD_{n+1}))_{i\in I}\in A, \ \cE_I\big)\\
    &=\P\big((S_{\tilde\sigma(i)})_{i\in I}\in A, \ \cE_I\big)\\
    &=\P\big((S_{\sigma(i)})_{i\in I}\in A, \ \cE_I\big),
\end{align*}
where $\sigma,\tilde\sigma$ are the same as in the proof of Lemma~\ref{lem:exch_conditional_subvector}.

Next, since $(S_i)_{i\in I}$ is exchangeable conditional on $\cE_I$, the quantity
\[p_I = \frac{\sum_{i\in I}\ind{S_i\geq S_{n+1}}}{|I|}\]
is a valid p-value conditional on $\cE_I$, i.e., $\P(p_I\leq \alpha\mid \cE_I)\leq \alpha$ (see Corollary~\ref{cor:perm_test_pval}). But on the event $\cE_I$, we have $p_I = p$ by construction. Since $Z_{n+1}\in\cZ_0$ holds if and only if $\cE_I$ holds for some $I\ni n+1$, we therefore have
\begin{align*}
    \P(p\leq \alpha, Z_{n+1}\in \cZ_0)
    &=\sum_{I\ni n+1}\P(p\leq \alpha, \cE_I)\\
    &=\sum_{I\ni n+1}\P(p_I\leq \alpha, \cE_I)\\
    &=\sum_{I\ni n+1}\P(\cE_I)\cdot \P(p_I\leq \alpha\mid \cE_I)\\
    &\leq\sum_{I\ni n+1}\P(\cE_I)\cdot \alpha\\
    &=\alpha\cdot\P(Z_{n+1}\in\cZ_0),
\end{align*}
which proves that $\P(p\leq \alpha\mid Z_{n+1}\in \cZ_0)\leq \alpha$, as desired.
\end{proof}
\index{coverage!test-conditional|)}

\section{Label-conditional coverage}\label{sec:lab-conditional}
\index{coverage!label-conditional}

In classification problems, another concept of validity that is stronger than marginal validity is label-conditional coverage. This term refers to coverage conditional on the response $Y_{n+1}$, in settings where the response is categorical (and is thus commonly referred to as the `label' of the data point). 

Formally, if $\cY = \{1,\dots,K\}$, we seek prediction sets $\cC$ such that
\begin{equation}
\label{eq:label-cond-cvg}
    \P\left(Y_{n+1} \in \cC(X_{n+1}) \mid Y_{n+1} = y\right) \ge 1-\alpha
\end{equation}
for all $y = 1,\dots,K$.  We will see that this is tractable with sufficient data, just like for test-conditional coverage in the setting where $|\cX|$ is finite. Moreover, the underlying proof of validity closely follows those from the previous two sections.

We now show how to achieve~\eqref{eq:label-cond-cvg} with a modification of the full conformal algorithm. Essentially, we split the data into classes, and then carry out conformal prediction within each class separately. However, since we do not know the class of the test data point, to implement this we will need to use a hypothesized test point value $y$ in place of the unknown $Y_{n+1}$. 

To be more precise, in order to achieve label-conditional coverage, we will compare the score $S_{n+1}^y = s((X_{n+1},y); \cD^y_{n+1})$ against  scores $S_i^y$ for training points $i$ for which $Y_i=y$.  To this end, let $\cI_y   = \{i\in[n] : Y_i = y\}$ and let $n_y = |\cI_y|$, and define the quantile
\[\hat{q}^y = \quantile\left((S_i^y)_{i \in \cI_y} ; (1-\alpha)(1+1/n_y)\right).\] If $S_{n+1}^y \le \hat q^y$, then $Y_{n+1}$ taking a value of $y$ is judged to be consistent with previous data. Repeating this for each value of $y$, we arrive at the following prediction set:
\begin{equation}\label{eqn:label_cond_coverage_by_bin_C}
    \cC(X_{n+1}) = \left\{y : S_{n+1}^y \le \hat{q}^y \right\}.
\end{equation}
This closely parallels the prediction set constructions for test-conditional coverage when the feature vector $X_{n+1}$ is discrete, as in~\eqref{eqn:test_conditional_coverage_discrete_X}, or when we relax our goal to bin-wise conditional coverage, as in~\eqref{eqn:test_cond_coverage_by_bin_C}. Just like for those two methods, we partition the training points over the different possible values $y\in\cY$, and calculate a different value of the conformal quantile for each one.

\begin{theorem}[Label-conditional coverage guarantee]
\label{thm:lab-cond-cvg}
Suppose $(X_1,Y_1),...,(X_{n+1},Y_{n+1})$ are exchangeable and that $s$ is a symmetric score function. Then the label-conditional prediction set $\cC(X_{n+1})$ defined in~\eqref{eqn:label_cond_coverage_by_bin_C} satisfies
\begin{equation}
    \P(Y_{n+1} \in \cC(X_{n+1}) \mid Y_{n+1}=y) \ge 1-\alpha
\end{equation}
for all $y \in \cY$ such that $\P(Y_{n+1} = y) > 0$.
\end{theorem}
This result is a special case of Theorem~\ref{thm:mondrian-cvg} below, so we omit the proof.

We point out that the prediction sets always include a label $y$ if there are insufficient examples with that value of $y$ in the calibration data. 
Thus, the prediction sets are more useful when there are at least $\approx 1/\alpha$ examples of each class in the data; if this is not the case for some $y$, then this candidate label value $y$ will \emph{always} be included in the prediction set.
This is analogous to the fact that in marginal conformal prediction, when $n$ is smaller than $\approx 1/\alpha$, the sets are uninformative---but
in the label-conditional setting, the relevant sample size is now the number of times we observed the value $y$, rather than the larger sample size $n$.

\section{Mondrian conformal prediction}\label{sec:mondrian} \index{Mondrian conformal prediction}
Binned conditional coverage and label-conditional coverage are both special cases of \emph{Mondrian conformal prediction}, a more general framework that enables us to enforce conditional coverage relative to events that capture information about the feature $X$, about the label $Y$, or both. To define this general method, the space $\cX \times \cY$ is first partitioned into a finite number of groups: formally, we let $g : \cX \times \cY \to \{1,\dots, K\}$ be the function that maps a data point to its group identifier. Mondrian conformal prediction is an algorithm that guarantees coverage at least level $1 - \alpha$ for each group, i.e.,
\begin{equation*}
    \P\left(Y_{n+1} \in \cC(X_{n+1}) \mid g(X_{n+1}, Y_{n+1}) = k\right) \ge 1-\alpha
\end{equation*}
for all $k \in \{1,\dots,K\}$.

The idea behind the algorithm is similar to that of the binning-based or label-conditional approaches described earlier in this chapter. For each hypothesized value of $y \in \cY$, we compute the group identifier $g(X_{n+1}, y)$. Then, we compare the conformal score $S_{n+1}^y = s((X_{n+1}, y); \cD^y_{n+1})$ to the conformal scores $S_{i}^y = s((X_{i}, Y_{i}); \cD^y_{n+1})$ for all $i\in[n]$ such that $g(X_i, Y_i) = g(X_{n+1}, y)$---that is, we compare to all those data points that lie in the same group as the hypothesized test point $(X_{n+1},y)$. 
To formalize this, defining $\cI_k = \{i \in [n] : g(X_i, Y_i) = k\}$ for each group $k\in[K]$, we compute the quantile
\[\hat{q}^y = \quantile((S_i^y)_{i \in \cI_{g(X_{n+1},y)}}; (1-\alpha)(1+1/|\cI_{g(X_{n+1},y)}|)).\] This leads us to the proposed prediction set
\begin{equation}\label{eqn:mondrian_cC}
    \cC(X_{n+1}) = \left\{y : S_{n+1}^y \le \hat{q}^y \right\}.
\end{equation}

\begin{theorem}[Mondrian coverage guarantee]
\label{thm:mondrian-cvg}
Suppose $(X_1,Y_1),...,(X_{n+1},Y_{n+1})$ are exchangeable and that $s$ is a symmetric score function. 
The Mondrian prediction set $\cC(X_{n+1})$ defined in~\eqref{eqn:mondrian_cC} satisfies the conditional coverage guarantee
\begin{equation}
    \P(Y_{n+1} \in \cC(X_{n+1}) \mid g(X_{n+1}, Y_{n+1}) =k) \ge 1-\alpha
\end{equation}
for all $k \in \{1,\dots,K\}$ such that $\P(g(X_{n+1},Y_{n+1}) = k)> 0$.
\end{theorem}

\begin{proof}[Proof of Theorem~\ref{thm:mondrian-cvg}]
\index{exchangeability!conditional}
This proof generalizes the proof of Theorem~\ref{thm:bin_conditional}, and follows an identical structure.
First, by construction of the set $\cC(X_{n+1})$, we can see that for any $k\in[K]$, on the event $g(X_{n+1},Y_{n+1})=k$, 
\begin{equation}\label{eqn:conformal-via-pvalues_for_mondrian}Y_{n+1}\in \cC(X_{n+1}) \Longleftrightarrow \frac{1+\sum_{i\in[n],g(X_i,Y_i) = k}\ind{S_i\geq S_{n+1}}}{1+|\cI_k|} >  \alpha\end{equation}
(the proof of this statement is analogous to the proof of Proposition~\ref{prop:conformal-via-pvalues}).

Next, fix any label $k\in[K]$ with $\P(g(X_{n+1},Y_{n+1})=k)>0$, and let $\cZ_0 = \{(x,y)\in\cX \times\cY: g(x,y)=k\} \subseteq \cX\times \cY$. By Corollary~\ref{cor:pvalues_due_to_lem:exch_conditional_subvector}, the quantity
\[p = \frac{1+\sum_{i\in[n]}\ind{(X_i,Y_i)\in\cZ_0,S_i\geq S_{n+1}}}{1+\sum_{i\in[n]}\ind{(X_i,Y_i)\in\cZ_0}}= \frac{1+\sum_{i\in[n],g(X_i,Y_i) = k}\ind{S_i\geq S_{n+1}}}{1+|\cI_k|}\]
satisfies $\P(p\leq \alpha\mid g(X_{n+1},Y_{n+1})=k)\leq \alpha$. 
By~\eqref{eqn:conformal-via-pvalues_for_mondrian}, this completes the proof.
\end{proof}

\section{Relaxations of test-conditional coverage: beyond binning}\label{sec:test-conditional-relax}

We now return to the question of relaxations for test-conditional coverage, which was the focus of Sections~\ref{sec:condition_test_point} and~\ref{sec:test-conditional-binning}. In particular, in Section~\ref{sec:test-conditional-continuous-impossible} we learned that the strongest form of the test-conditional coverage property, which requires that 
\begin{equation}\label{eqn:test-conditional-at-Xn+1}\P\big(Y_{n+1}\in\cC(X_{n+1})\mid X_{n+1}\big)\geq 1-\alpha\textnormal{ holds almost surely over $X_{n+1}$},\end{equation}
is in general too strong---for a nonatomic feature distribution $P_X$, it is essentially impossible to improve upon the trivial algorithm: outputting $\cY$ with probability $1-\alpha$ and $\varnothing$ otherwise. On the other hand, if we bin the space $\cX$ into a finite number of bins, this would could achieve a much weaker version of test-conditional coverage. In this section, we consider other possible relaxations, to see if it might be possible to obtain guarantees that are stronger than those obtained via binning.

\subsection{For all \texorpdfstring{$X$}{X} or for most \texorpdfstring{$X$}{X}?}
Instead of requiring test-conditional coverage to hold for all values of $X_{n+1}$, as in~\eqref{eqn:test-conditional-at-Xn+1}, we can consider relaxing the condition by requiring it to hold only for `most' values of $X_{n+1}$, to avoid a scenario where some nonnegligible fraction of $X_{n+1}$ values, say some set $\cX_0\subseteq\cX$, has a coverage level substantially lower than $1-\alpha$. To quantify this, we consider the following relaxation of the goal of test-conditional coverage:
\begin{equation}\label{eqn:test_conditional_relax_alpha_delta}\P\big(Y_{n+1}\in\cC(X_{n+1})\mid X_{n+1}\in\cX_0\big)\geq 1-\alpha\textnormal{ for all $P$, and all $\cX_0\subseteq\cX$ with $\P_P(X\in\cX_0)\geq \delta$}.\end{equation}
Here probability is taken with respect to data $(X_1,Y_1),\dots,(X_{n+1},Y_{n+1})$ sampled i.i.d.\ from $P$.
We can observe that if this relaxed condition holds for all values $\delta>0$, then we would simply recover the original condition~\eqref{eqn:test-conditional-at-Xn+1}, which we know to be impossible in the nonatomic setting. Instead, we will ask whether it is possible to ensure that this relaxed condition holds for some fixed and positive $\delta$.

Of course, even the original definition~\eqref{eqn:test-conditional-at-Xn+1} of test-conditional coverage can be satisfied if we are willing to accept the trivial algorithm of~\eqref{eqn:trivial_test_conditional}, which often returns $\cC(X_{n+1}) =\cY$. The impossibility result derived above in Theorem~\ref{thm:conditional_infinite} and Corollary~\ref{cor:infinite-lebesgue} can be interpreted as telling us that this trivial solution is a baseline that cannot be improved upon.

When we instead aim for the relaxed condition~\eqref{eqn:test_conditional_relax_alpha_delta}, a different trivial solution gives us a baseline for this setting: we can simply run any distribution-free method (e.g., split or full conformal) with target coverage level $1-\alpha\delta$ in place of $1-\alpha$. 
This strategy is trivial because in no way is it required to adapt to the difficulty of the prediction task as a function of $X_{n+1}$---it simply achieves the guarantee by being more conservative overall.
To verify that this would lead to the desired guarantee, we can observe that, for any $\cX_0$ with $\P_P(X\in\cX_0)\geq \delta$,
\begin{multline}\label{eqn:trivial_solution_alpha_delta_calculate}\P(Y_{n+1}\not\in \cC(X_{n+1})\mid X_{n+1}\in\cX_0) \leq \delta^{-1} \P(Y_{n+1}\not\in \cC(X_{n+1})\cdot\indsub{X_{n+1}\in \cX_0})\\
\leq \delta^{-1} \P(Y_{n+1}\not\in \cC(X_{n+1}))\leq \delta^{-1}\cdot \alpha\delta = \alpha,\end{multline}
where the last inequality holds if $\cC$ is constructed via some method guaranteeing marginal coverage at level $1-\alpha\delta$.
In fact, if we allow randomization in our construction, this trivial solution can be generalized to a family of solutions, parameterized by $c\in[0,1]$:
\begin{equation}\label{eqn:alpha_delta_trivial_solution}\begin{cases}
\textnormal{Construct $\cC'(X_{n+1})$, using any method that guarantees marginal coverage at level $1-c\alpha\delta$.}\\
\textnormal{With probability $\frac{1-\alpha}{1-c\alpha}$, return $\cC(X_{n+1}) = \cC'(X_{n+1})$; otherwise, return $\cC(X_{n+1}) = \varnothing$.}\end{cases}\end{equation}
\begin{proposition}\label{prop:relaxed-test-conditional-trivial-solution}
    For any $c\in[0,1]$, the procedure defined in~\eqref{eqn:alpha_delta_trivial_solution} satisfies the relaxed test-conditional coverage condition~\eqref{eqn:test_conditional_relax_alpha_delta}. 
\end{proposition}
\begin{proof}[Proof of Proposition~\ref{prop:relaxed-test-conditional-trivial-solution}]
By construction of the method, we have
\[\P(Y_{n+1}\in\cC(X_{n+1})\mid X_{n+1}\in \cX_0) = \frac{1-\alpha}{1-c\alpha}\cdot \P(Y_{n+1}\in\cC'(X_{n+1})\mid X_{n+1}\in \cX_0).\]
Next, following the same steps as in the calculations~\eqref{eqn:trivial_solution_alpha_delta_calculate} above,
\[\P(Y_{n+1}\not\in\cC'(X_{n+1})\mid X_{n+1}\in\cX_0) \leq \delta^{-1} \cdot c\alpha \delta = c\alpha.\]
Combining these two calculations proves the result.
\end{proof}

Consequently, in the setting of relaxed (rather than almost sure) test-conditional coverage, our question is somewhat different from before. Again working in the setting $\cY=\R$, while before we asked whether it is possible for a procedure satisfying the original condition~\eqref{eqn:test-conditional-at-Xn+1} to return an interval of \emph{finite} length, now we ask whether a procedure satisfying the relaxed condition~\eqref{eqn:test_conditional_relax_alpha_delta} can return an interval of \emph{shorter} length than the trivial solution obtained via~\eqref{eqn:alpha_delta_trivial_solution} above. For example, we may ask, for any $\cC$ that achieves the relaxed test-conditional coverage condition~\eqref{eqn:test_conditional_relax_alpha_delta}, 
\begin{equation}\label{eqn:hardness-test-conditional-coverage-relaxed-ideal-lower-bound}\textnormal{Is it true that }\E[\textnormal{Leb}(\cC(X_{n+1}))] \geq \inf_{c\in[0,1]} \left\{\frac{1-\alpha}{1-c\alpha} \cdot \E[\textnormal{Leb}(\cC_{\textnormal{split}; 1-c\alpha\delta}(X_{n+1}))]\right\} \ ? \end{equation}
Here $\cC_{\textnormal{split}; 1-c\alpha\delta}$ is the interval returned by running split conformal prediction at coverage level $1-c\alpha\delta$. Thus, the right-hand side of~\eqref{eqn:hardness-test-conditional-coverage-relaxed-ideal-lower-bound} can be interpreted as the expected length of the prediction set returned by the trivial solution defined in~\eqref{eqn:alpha_delta_trivial_solution}, when we use split conformal prediction to construct the initial set $\cC'(X_{n+1})$ (and minimize expected length over all possible choices of the parameter $c$).

It turns out that, while we cannot prove this exact bound, the lower bound established in the following theorem tells a very similar story. In order to present the theorem we first need to define a new quantity: for $t\in[0,1]$, 
let $L_P(1-t)$ be the minimum length of any \emph{oracle} prediction interval $\cC^P_{1-t}$,
which is constructed given knowledge of the distribution $P$ (i.e., not a distribution-free method), and has coverage level $1-t$:
\begin{equation}\label{eqn:L_P_def}L_P(1-t) = \inf\left\{ \E_P[\textnormal{Leb}(\cC^P_{1-t}(X))] \ : \textnormal{ $\cC^P_{1-t}$ satisfies }\P_P(Y\in \cC^P_{1-t}(X))\geq 1-t\right\}.\end{equation}
For example, if under the true distribution $P$ it holds that $Y=f(X) + \cN(0,1)$ for some function $f$, then the minimum-length interval is given by $\cC^P_{1-t}(X) = f(X) \pm z^*_{1-t/2}$, where $z^*_{1-t/2}$ is the $(1-t/2)$-quantile of the standard normal distribution, and so $L_P(1-t) = 2z^*_{1-t/2}$. For a general distribution $P$, we emphasize that this minimum length is \emph{not} over methods that have distribution-free validity; we are instead considering prediction intervals that are constructed using knowledge of $P$, which might be substantially narrower. These oracle lengths satisfy a natural property, which we state without proof:
\begin{fact}[Convexity of the oracle prediction interval length]\label{fact:L_P_decr_convex}
For any nonatomic distribution $P$, the oracle length $L_P(1-t)$ is a convex and nonincreasing function of $t$.
\end{fact}

Now we present the theorem:
\begin{theorem}[Hardness of relaxed test-conditional coverage (nonatomic setting)]\label{thm:hardness-test-conditional-coverage-relaxed}
Suppose $\cC$ satisfies the distribution-free relaxed test-conditional coverage condition~\eqref{eqn:test_conditional_relax_alpha_delta}, and let $\cY=\R$. Then, for any distribution $P$ on $\cX\times\R$ for which the marginal $P_X$ is nonatomic,
\[\E[\textnormal{Leb}(\cC(X_{n+1}))] \geq \inf_{c\in[0,1]} \left\{\frac{1-\alpha}{1-c\alpha} \cdot L_P(1-c\alpha\delta)\right\}.\] 
\end{theorem} \index{hardness result!test-conditional coverage}
We can see that this theorem establishes a lower bound that, at first glance, appears similar to the goal laid out in~\eqref{eqn:hardness-test-conditional-coverage-relaxed-ideal-lower-bound}, with a key difference: instead of the length of a split conformal interval, $\E[\textnormal{Leb}(\cC_{\textnormal{split}; 1-c\alpha\delta}(X_{n+1}))]$, we instead have the length of an oracle interval, $L_P(1-c\alpha\delta)$.
Of course, it may be the case that this oracle length $L_P(1-c\alpha\delta)$ is much shorter than the expected length of the split conformal interval---naturally we expect to pay a price for distribution-free validity. But as we will see in Chapter~\ref{chapter:model-based},  
in certain settings where we are able to find a score function that captures the true distribution of the data, the split conformal interval is asymptotically equivalent to the oracle prediction interval---that is, 
$\E[\textnormal{Leb}(\cC_{\textnormal{split},1-c\alpha\delta}(X_{n+1}))]\approx L_P(1-c\alpha\delta)$, and so
the lower bound in the theorem can be interpreted as implying that the bound~\eqref{eqn:hardness-test-conditional-coverage-relaxed-ideal-lower-bound} approximately holds.

The proof of the hardness result in Theorem~\ref{thm:hardness-test-conditional-coverage-relaxed} relies primarily on the following key lemma:
\begin{lemma}\label{lemma_for_thm:hardness-test-conditional-coverage-relaxed}
Suppose $\cC$ satisfies the distribution-free relaxed test-conditional coverage condition~\eqref{eqn:test_conditional_relax_alpha_delta}. Let $P$ be any distribution on $\cX\times\cY$ such that the marginal $P_X$ is nonatomic. Then 
\begin{equation}\label{eqn:conditional_cov_key_lemma_bound}\P_P\big(Y_{n+1}\in \cC(X_{n+1})\mid (X_{n+1},Y_{n+1})\in B\big) \geq 1-\alpha\textnormal{ for any $B\subseteq\cX\times\cY$ with $P(B)\geq \delta$}.\end{equation}
\end{lemma}
\begin{figure}[t]
    \centering
    \includegraphics[width=0.7\textwidth]{\diagramspath impossibility-lemma.pdf}
    \caption{\textbf{An illustration of the hardness result in Lemma~\ref{lemma_for_thm:hardness-test-conditional-coverage-relaxed}}. The lemma says that, if we satisfy conditional coverage over \emph{all} $\cX_0\subseteq\cX$ with measure at least $\delta$, then we must \emph{also} have coverage for all regions $B\subseteq \cX\times\cY$ with measure $\delta$. This leads to the lower bound in Theorem~\ref{thm:hardness-test-conditional-coverage-relaxed}, because a predictive inference procedure that has high coverage conditional on $(X,Y)\in B$ (for every possible $B$) must necessarily return very wide intervals.
    On the left is drawn an example set $\cX_0$, and on the right an example $B$.}
    \commentAlt{Two panels show the same heatmap for a conditional density of $Y$ given $X$. The left panel has a shaded region labeled $\cX_0$ consisting of several narrow vertical bands. The right panel has a shaded region $B$ consisting of two irregular shapes.}
    \label{fig:lemma-hardness-test-conditional-coverage-relaxed}
\end{figure}
The significance of this lemma may not be immediately evident. Indeed, its conclusion, the bound given in~\eqref{eqn:conditional_cov_key_lemma_bound}, looks extremely similar to the relaxed definition~\eqref{eqn:test_conditional_relax_alpha_delta} of conditional coverage that we have assumed to hold; the only difference lies in whether we condition on an event that depends on the test feature $X_{n+1}$ only, or on the entire test point $(X_{n+1},Y_{n+1})$. In fact, these two statements have very different implications. 

To illustrate this difference, let's consider a very simple example, without the distribution-free context. Suppose again that $Y = f(X) + \cN(0,1)$, i.e., we assume a Gaussian model on the noise. Then, with oracle knowledge of this conditional distribution, the prediction interval $\cC(x) = f(x) \pm 1.96$ offers 95\% coverage for the conditional distribution of $Y$ given $X=x$---and in particular, would satisfy a guarantee of the type~\eqref{eqn:test_conditional_relax_alpha_delta} for arbitrarily small $\delta$. On the other hand, a bound of the type~\eqref{eqn:conditional_cov_key_lemma_bound} would not be satisfied---for instance, choosing any $B \subseteq \{(x,y): |y - f(x)| > 1.96\}$, we can have $P(B)$ potentially as large as $\P(|Y - f(X)| > 1.96) = 0.05$, but our interval $\cC(x)$ would lead to \emph{zero} coverage conditional on the event $(X,Y)\in B$. See also Figure~\ref{fig:lemma-hardness-test-conditional-coverage-relaxed} for a visualization.

In other words, the difference is that the definition~\eqref{eqn:test_conditional_relax_alpha_delta} requires a high level of predictive coverage even at rare values of $X$, but the conclusion of Lemma~\ref{lemma_for_thm:hardness-test-conditional-coverage-relaxed} tells us that we therefore also have a high level of predictive coverage even at rare values of $Y$. While the first statement seems reasonable, the second statement inherently conflicts with the goal of accurate predictive inference, where we typically want to return an interval containing only plausible values of $Y$.

Next, we prove the key lemma. 
The main idea of the proof relies on the following strategy: when sampling $m$ i.i.d.\ draws from any distribution $P$, it is nearly equivalent to first sample $M\geq m$ i.i.d.\ draws from $P$, then sample $m$ times with replacement from this list, as long as $M$ is sufficiently large (specifically, $M\gg m^2$).
We will see this construction used again for other hardness results proved later in the book, in other contexts---we refer to it as the \emph{sample--resample} construction, and we include it as its own lemma:
\begin{lemma}[The sample--resample construction]\label{lem:sample-resample}
Let $P$ be a distribution on $\cZ$, and let $m,M\geq 1$. Let $P^m$ denote the corresponding product distribution on $\cZ^m$---that is, the distribution of $(Z_1,\dots,Z_m)$, where $Z_1,\dots,Z_m\iidsim P$. Moreover, let $Q$ denote the distribution on $\cZ^m$ obtained by the following process to generate $(Z_1,\dots,Z_m)$:
\begin{enumerate}
    \item Sample $Z^{(1)},\dots,Z^{(M)}\iidsim P$, and define the empirical distribution $\widehat{P}_M = \frac{1}{M}\sum_{i=1}^M \delta_{Z^{(i)}}$;
    \item Sample $Z_1,\dots,Z_m\iidsim \widehat{P}_M$.
\end{enumerate}
Then
\[\dtv\big(P^m,Q\big) \leq \frac{m(m-1)}{2M},\]
where $\dtv$ denotes the total variation distance between distributions.
\end{lemma} \index{sample-resample lemma}
With this technique in place, we are now ready to prove Lemma~\ref{lemma_for_thm:hardness-test-conditional-coverage-relaxed}.
\begin{proof}[Proof of Lemma~\ref{lemma_for_thm:hardness-test-conditional-coverage-relaxed}]
First we prove the result for any set $B$ with $\P_P((X,Y)\in B)>\delta$. 
Fix a large integer $M$, let $(X^{(1)},Y^{(1)}),\dots,(X^{(M)},Y^{(M)})$ be an arbitrary sequence of $M$ data points, and let $\widehat{P}_M = \frac{1}{M}\sum_{i=1}^M\delta_{(X^{(i)},Y^{(i)})}$ denote the empirical distribution of this sequence. Define
\[\cM = \big\{m\in\{1,\dots,M\} : (X^{(m)},Y^{(m)})\in B\big\}\] as the set of indices whose points lie in $B$, and let
 $\cX_0 =\{X^{(m)} : m\in\cM\}$ be the corresponding set of $X$ values.
Then $\P_{\widehat{P}_M}(X\in\cX_0)\geq |\cM|/M $ by construction.  Therefore, 
since $\cC$ satisfies $(1-\alpha,\delta)$-conditional coverage, if $|\cM| \geq \delta M$ (and so $\P_{\widehat{P}_M}(X\in \cX_0)\geq \delta$) then we must have
\[\P_{\widehat{P}_M}\left(Y_{n+1}\in\cC(X_{n+1})\mid X_{n+1}\in\cX_0\right)\geq 1-\alpha.\]
Therefore,
\[\P_{\widehat{P}_M}\left(Y_{n+1}\in\cC(X_{n+1}), X_{n+1}\in\cX_0\right)\geq (1-\alpha) \cdot \frac{|\cM|}{M} \cdot\indsub{|\cM|\geq \delta M}.\]
Moreover, if the values $X^{(1)},\dots,X^{(M)}$ are all distinct, then under the distribution $\widehat{P}_M$, $X_{n+1}\in\cX_0$ if and only if $(X_{n+1},Y_{n+1})\in B$. In other words, we have shown that for any sequence $(X^{(1)},Y^{(1)}),\dots,(X^{(M)},Y^{(M)})$ where the values $X^{(1)},\dots,X^{(M)}$ are all distinct, it must hold that
\begin{equation}\label{eqn:Phat_M_bound_for_using_sample_resample}\P_{\widehat{P}_M}\left(Y_{n+1}\in\cC(X_{n+1}), (X_{n+1},Y_{n+1})\in B\right)\geq (1-\alpha) \cdot \frac{|\cM|}{M} \cdot\indsub{|\cM|\geq \delta M}.\end{equation}

Next, consider the following construction:
\begin{enumerate}
    \item Sample $(X^{(1)},Y^{(1)}),\dots,(X^{(M)},Y^{(M)})\iidsim P$, and let $\widehat{P}_M = \frac{1}{M}\sum_{i=1}^M \delta_{(X^{(i)},Y^{(i)})}$ be the corresponding empirical distribution;
    \item Sample $(X_1,Y_1),\dots,(X_{n+1},Y_{n+1})\iidsim \widehat{P}_M$.
\end{enumerate}
Since $P_X$ is assumed to be nonatomic, this means that $X^{(1)},\dots,X^{(M)}$ are  distinct almost surely in this construction, and therefore, the inequality~\eqref{eqn:Phat_M_bound_for_using_sample_resample} holds almost surely with respect to the draw of $(X^{(1)},Y^{(1)}),\dots,(X^{(M)},Y^{(M)})$. Now define a distribution $Q$ on $(\cX\times\cY)^{n+1}$, given by the marginal distribution of $(X_1,Y_1),\dots,(X_{n+1},Y_{n+1})$ under the construction above.  Marginalizing~\eqref{eqn:Phat_M_bound_for_using_sample_resample} over the random draw of the data points $(X^{(1)},Y^{(1)}),\dots,(X^{(M)},Y^{(M)})$, then, we have shown that
\begin{multline*}\P_{((X_1,Y_1),\dots,(X_{n+1},Y_{n+1}))\sim Q}\Big(Y_{n+1}\in\cC(X_{n+1}), (X_{n+1},Y_{n+1})\in B\Big)\\\geq (1-\alpha)  \cdot \E\left[\frac{|\cM|}{M} \cdot\indsub{|\cM|\geq \delta M}\right]. \end{multline*}

Next, we also have $|\cM| \sim\textnormal{Binomial}(M,\delta')$ by construction, where we define $\delta' = \P_P((X,Y)\in B)>\delta$. We then have
\begin{align*}
    \E\left[\frac{|\cM|}{M} \cdot\indsub{|\cM|\geq \delta M}\right]
    &\geq \E\left[\frac{|\cM|}{M}\right] - \P(|\cM|<\delta M) \\
    &=\delta' - \P\big(\textnormal{Binomial}(M,\delta') < M\delta\big)\\
    &\geq \delta' -e^{ - 2M(\delta'-\delta)^2},
\end{align*}
where the last step holds by Hoeffding's inequality. Moreover, by Lemma~\ref{lem:sample-resample}, we have
\[\dtv(P^{n+1},Q)\leq \frac{n(n+1)}{2M}.\]
Therefore, combining all these calculations, we have shown that
\begin{multline*}
    \P_{(X_i,Y_i)\iidsim P}\Big(Y_{n+1}\in\cC(X_{n+1}), (X_{n+1},Y_{n+1})\in B\Big)\\
    \geq \P_{((X_1,Y_1),\dots,(X_{n+1},Y_{n+1}))\sim Q}\Big(Y_{n+1}\in\cC(X_{n+1}), (X_{n+1},Y_{n+1})\in B\Big) - \frac{n(n+1)}{2M}\\\geq (1-\alpha)\left[\delta' - e^{ - 2M(\delta'-\delta)^2}\right] - \frac{n(n+1)}{2M}.
\end{multline*}
Since $M$ can be taken to be arbitrarily large, therefore, this proves that
\[\P_P\left(Y_{n+1}\in\cC(X_{n+1}),(X_{n+1},Y_{n+1})\in B\right) \geq (1-\alpha)\delta' = (1-\alpha)\P_P\big((X_{n+1},Y_{n+1})\in B\big),\]
which establishes that $\P_P\left(Y_{n+1}\in\cC(X_{n+1})\mid (X_{n+1},Y_{n+1})\in B\right) \geq 1-\alpha$, as desired.

This completes the proof for the case $\P_{P}((X,Y)\in B)>\delta$. Now suppose instead that $\P_{P}((X,Y)\in B)=\delta$. Since $P_X$ (and therefore, also $P$) is nonatomic, we can construct a nested sequence of sets $B_1 \supseteq B_2 \supseteq \dots \supseteq B$, for which $P(B_i)>\delta$ for all $i$, and $P(B_i\backslash B)\rightarrow0$. By the work above, then, it holds that 
\[\P_{P}\left(Y_{n+1}\in \cC(X_{n+1})\mid (X_{n+1},Y_{n+1})\in B_i\right) \geq 1-\alpha\]
for all $i$. Taking a limit proves the desired result: we have
\begin{multline*}\P_{P}\left(Y_{n+1}\in \cC(X_{n+1})\mid (X_{n+1},Y_{n+1})\in B\right)  = \frac{\P_{P}\left(Y_{n+1}\in \cC(X_{n+1}), (X_{n+1},Y_{n+1})\in B\right)}{\delta} \\\geq \frac{\P_{P}\left(Y_{n+1}\in \cC(X_{n+1}), (X_{n+1},Y_{n+1})\in B_i\right) - \P_{P}\left((X_{n+1},Y_{n+1})\in B_i\backslash B\right)}{\delta}
\\\geq \frac{(1-\alpha) \cdot \left(\delta + P(B_i\backslash B)\right)- P(B_i\backslash B)}{\delta},\end{multline*}
for each $i$, and taking $i\rightarrow\infty$ completes the proof.
\end{proof}

Next, we need to prove the sample--resample lemma.
\begin{proof}[Proof of Lemma~\ref{lem:sample-resample}]
This proof follows by simply bounding the difference between sampling with and without replacement. First, if $M<m$ then the result holds trivially, so we can assume $M\geq m$ from this point on. Observe that drawing $m$ i.i.d.\ samples from $P$ is equivalent to sampling $m$ times without replacement from the list $Z^{(1)},\dots,Z^{(M)}$---or more formally,
\begin{enumerate}
    \item Sample $Z^{(1)},\dots,Z^{(M)}\iidsim P$;
    \item Sample indices $k_1,\dots,k_m\in\{1,\dots,M\}$, uniformly without replacement;
    \item Define $Z_i = Z^{(k_i)}$ for each $i\in[m]$.
\end{enumerate}
On the other hand, by definition of $Q$, sampling a dataset from $Q$ is equivalent to the following:
\begin{enumerate}
   \item Sample $Z^{(1)},\dots,Z^{(M)}\iidsim P$;
    \item Sample indices $k_1,\dots,k_m\in\{1,\dots,M\}$, uniformly with replacement;
    \item Define $Z_i = Z^{(k_i)}$ for each $i\in[m]$.
\end{enumerate}
The only difference between these two constructions lies in whether the indices $k_1,\dots,k_m$ are sampled without replacement (in the first procedure, while constructing a draw from $P^m$), or with replacement (in the second procedure, while constructing a draw from $Q$). Therefore, the total variation distance between these two procedures is bounded by the total variation distance for sampling with versus without replacement (when drawing $m$ samples from a finite population of size $M$). This is given by the probability of observing any duplicate values when sampling with replacement, which can be bounded as
\[\dtv(P^m,\tilde{Q})\leq 0 + \frac{1}{M} + \dots + \frac{m-1}{M} = \frac{m(m-1)}{2M},\]
where we are using the fact that, for the $i$th draw, the probability of drawing an index $k_i$ that we have seen before is at most $\frac{i-1}{M}$.
\end{proof}

We conclude this section with a proof of the hardness result of Theorem~\ref{thm:hardness-test-conditional-coverage-relaxed}. This proof is essentially a technical calculation that leverages the result of Lemma~\ref{lemma_for_thm:hardness-test-conditional-coverage-relaxed} above; we include it for completeness.
\begin{proof}[Proof of Theorem~\ref{thm:hardness-test-conditional-coverage-relaxed}]
We begin by defining a probability
\[p(x,y) = \P\left(y\in\cC(x)\right),\]
where $\cC$ is the prediction set in the theorem, i.e., $\cC$ is assumed to satisfy the relaxed test-conditional coverage condition~\eqref{eqn:test_conditional_relax_alpha_delta}. (Since this prediction set is fitted to the training data $((X_i,Y_i))_{i\in[n]}$, we are computing the probability with respect to a draw of this training dataset.)  For each $t\in[0,1]$, define also a new prediction set $\cC_t$ as
\[\cC_t(x) = \{y\in\R: p(x,y) \geq t\},\]
or equivalently,
\[\cC_t(x) = \{y\in\R: \P(y\in\cC(x)) \geq t\}.\]
Note that $\cC_t$ is no longer data-dependent---$\cC_t(x)$ depends on \emph{the distribution of} the random prediction set $\cC(x)$, not the actual $\cC(x)$ itself.

By construction, we can observe that the $\cC_t$'s are nested: for $t < t'$ we have $\cC_t(x)\supseteq \cC_{t'}(x)$. 
For each $t$, we can define the coverage level of this (deterministic) interval $\cC_t$:
\[1-\alpha_t = \P_P\left(Y\in \cC_t(X)\right).\]
Now we are ready to calculate the expected length of $\cC(X_{n+1})$, in terms of these new values $\alpha_t$. First, we need to rearrange some terms to represent the length of the prediction interval $\cC(x)$ in terms of the probabilities $p(x,y)$. For any fixed $x$, taking expectation with respect to the random prediction interval $\cC$ (i.e., with respect to the random draw of the training data), we have
\begin{align*}
    \E[\textnormal{Leb}(\cC(x))]
    &= \E\left[\int_{\R}\ind{y\in \cC(x)}\;\mathsf{d}y\right]
    = \int_{\R}\P\left(y\in \cC(x)\right)\;\mathsf{d}y
    = \int_{\R}p(x,y)\;\mathsf{d}y\\
    &= \int_{\R}\int_{t=0}^1 \ind{p(x,y)\geq t}\;\mathsf{d}t\;\mathsf{d}y
    = \int_{\R}\int_{t=0}^1 \ind{y\in \cC_t(x)}\;\mathsf{d}t\;\mathsf{d}y\\
    &= \int_{t=0}^1 \int_{\R}\ind{y\in \cC_t(x)}\;\mathsf{d}y\;\mathsf{d}t
    = \int_{t=0}^1 \textnormal{Leb}(\cC_t(x))\;\mathsf{d}t,
\end{align*}
where in the second step and again in the next-to-last step, we swap order of integration by the Fubini--Tonelli theorem.
Replacing the fixed $x$ with a random draw $X_{n+1}$, and applying the Fubini--Tonelli theorem one more time along with the calculations above,
we have
\begin{multline*}
    \E[\textnormal{Leb}(\cC(X_{n+1}))]
    =\E\big[\E[\textnormal{Leb}(\cC(X_{n+1}))\mid X_{n+1}]\big]\\
    =\E\left[\int_{t=0}^1 \textnormal{Leb}(\cC_t(X_{n+1}))\;\mathsf{d}t\right]
    =\int_{t=0}^1 \E\left[\textnormal{Leb}(\cC_t(X_{n+1}))\right]\;\mathsf{d}t.
\end{multline*}
Finally, since $\cC_t$ attains coverage at level $1-\alpha_t$ with respect to the distribution $P$ (by definition of $\alpha_t$), we must have $\E\left[\textnormal{Leb}(\cC_t(X_{n+1}))\right]\geq L_P(1-\alpha_t)$, and so we can give a lower bound on our calculation:
\[\E[\textnormal{Leb}(\cC(X_{n+1}))] \geq \int_{t=0}^1 L_P(1-\alpha_t)\;\mathsf{d}t.\]
Moreover, for any fixed $c\in[0,1]$, we can further lower bound this as follows:
\begin{multline}\E[\textnormal{Leb}(\cC(X_{n+1}))] \geq \int_{t=0}^{\frac{1-\alpha}{1-c\alpha}} L_P(1-\alpha_t)\;\mathsf{d}t
= \frac{1-\alpha}{1-c\alpha} \cdot \frac{1}{\frac{1-\alpha}{1-c\alpha}}\int_{t=0}^{\frac{1-\alpha}{1-c\alpha}}  L_P(1-\alpha_t)\;\mathsf{d}t\\
\geq \frac{1-\alpha}{1-c\alpha} L_P\left(1 - \frac{1}{\frac{1-\alpha}{1-c\alpha}} \int_{t=0}^{\frac{1-\alpha}{1-c\alpha}}\alpha_t\;\mathsf{d}t\right),\end{multline}
where the last step holds since $L_P$ is a convex function by Fact~\ref{fact:L_P_decr_convex}.

To complete the proof, we can see that only one question remains: we need to verify that, for some carefully chosen value of $c$, it holds that
\[L_P\left(1 - \frac{1}{\frac{1-\alpha}{1-c\alpha}} \int_{t=0}^{\frac{1-\alpha}{1-c\alpha}}\alpha_t\;\mathsf{d}t\right) \geq L_P(1-c\alpha\delta).\]
Since $s\mapsto L_P(1-s)$ is nonincreasing (by Fact~\ref{fact:L_P_decr_convex}), we can simply check that
\begin{equation}\label{eqn:last_step_proof_thm:hardness-test-conditional-coverage-relaxed}\frac{1}{\frac{1-\alpha}{1-c\alpha}} \int_{t=0}^{\frac{1-\alpha}{1-c\alpha}}\alpha_t\;\mathsf{d}t\leq c\alpha\delta.\end{equation}
For this, we will need to apply the key lemma, Lemma~\ref{lemma_for_thm:hardness-test-conditional-coverage-relaxed}. 
First, define $p_*$ as the $\delta$-quantile of the distribution of $p(X,Y)$, when $(X,Y)\sim P$. By definition of the quantile, we must therefore have $\P_P(p(X,Y)<p_*) \leq \delta \leq \P_P(p(X,Y)\leq p_*)$. 
Since $P_X$ is assumed to be nonatomic, we can therefore find a set $B\subseteq\cX\times\R$ such that 
\begin{equation}\label{eqn:sandwich_B} \{(x,y): p(x,y) < p_*\} \subseteq B \subseteq \{(x,y) : p(x,y)\leq p_*\},\end{equation}
and $P(B)= \delta$. By Lemma~\ref{lemma_for_thm:hardness-test-conditional-coverage-relaxed}, then, it must hold that
\[\P\left(Y_{n+1}\in\cC(X_{n+1})\mid (X_{n+1},Y_{n+1})\in B\right) \geq 1-\alpha,\]
or equivalently, 
\[\P\left(Y_{n+1}\in\cC(X_{n+1}),  (X_{n+1},Y_{n+1})\in B\right) \geq(1-\alpha)\delta.\]
Now we apply the tower law:
\begin{align*}
    (1-\alpha)\delta & \leq \P\left(Y_{n+1}\in\cC(X_{n+1}),(X_{n+1},Y_{n+1})\in B\right)\\
    &= \E[p(X_{n+1},Y_{n+1})\cdot\ind{ (X_{n+1},Y_{n+1})\in B}]\\
    &= p_*\P\left((X_{n+1},Y_{n+1})\in B\right) - \E[(p_*-p(X_{n+1},Y_{n+1}))\cdot\ind{ (X_{n+1},Y_{n+1})\in B}]\\
    &= p_*\delta - \E[(p_*-p(X_{n+1},Y_{n+1}))\cdot\ind{ p(X_{n+1},Y_{n+1}) < p_*}]\textnormal{ by~\eqref{eqn:sandwich_B}}\\
    &= p_*\delta - \int_{t = 0}^{p_*} \P\left(p(X_{n+1},Y_{n+1})< t\right)\;\mathsf{d}t\\
    &= p_*\delta - \int_{t = 0}^{p_*} \P\left(Y_{n+1}\not\in \cC_t(X_{n+1})\right)\;\mathsf{d}t
    = p_*\delta - \int_{t = 0}^{p_*} \alpha_t\;\mathsf{d}t.
\end{align*}
In particular, this implies that $p_*\geq 1-\alpha$. We then choose $c$ to satisfy $\frac{1-\alpha}{1-c\alpha} = p_*$, and so we have shown that
\[\frac{1}{\frac{1-\alpha}{1-c\alpha}}\int_{t = 0}^{\frac{1-\alpha}{1-c\alpha}} \alpha_t\;\mathsf{d}t = \frac{1}{\frac{1-\alpha}{1-c\alpha}}\int_{t = 0}^{p_*} \alpha_t\;\mathsf{d}t  \leq \frac{1}{\frac{1-\alpha}{1-c\alpha}} \cdot \delta(p_* - (1-\alpha)) = c\alpha\delta.\]
This verifies the bound~\eqref{eqn:last_step_proof_thm:hardness-test-conditional-coverage-relaxed}, and thus completes the proof of the theorem.
\end{proof}

\subsection{Preview: relaxation via localization} \index{localized conformal prediction}
The hardness result established in Theorem~\ref{thm:hardness-test-conditional-coverage-relaxed} tells us that the `relaxed' conditional coverage goal~\eqref{eqn:test_conditional_relax_alpha_delta} is in fact not sufficiently relaxed: to be able to achieve a distribution-free guarantee without overly wide interval, we need to relax further to an even more approximate notion of conditional coverage. 

One natural way to do this is to require that coverage holds conditionally on $X\in\cX_0$ for only certain special subsets $\cX_0\subseteq\cX$, instead of for all $\cX_0$ with sufficient probability $\delta$ as in~\eqref{eqn:test_conditional_relax_alpha_delta}. For instance, we might consider only sets $\cX_0$ that are given by a ball, i.e., requiring
\begin{multline}\label{eqn:test_conditional_relax_alpha_delta_local}\P\big(Y_{n+1}\in\cC(X_{n+1})\mid \|X_{n+1}-x_*\|\leq r_*\big)\geq 1-\alpha\\\textnormal{ for all $P$ and all $x_*\in\cX$, $r_*\geq 0$ with $P_X(\|X-x_*\|\leq r_*)\geq \delta$},\end{multline}
for some norm $\|\cdot\|$ on $\cX$.
We can think of this as roughly requiring that coverage holds `conditional on the event that $X_{n+1}\approx x_*$'.
A related approach is to require that coverage holds relative to a reweighted distribution over $\cX$, e.g., using a kernel centered at some $x_*$. We will study these types of goals in more detail in Chapter~\ref{chapter:weighted-conformal}, via weighted and localized versions of conformal prediction.

\section*{Bibliographic notes}
\addcontentsline{toc}{section}{\protect\numberline{}\textnormal{\hspace{-0.8cm}Bibliographic notes}}

\citet{vovk2012conditional} develops a framework for defining conditional notions of coverage, examining training-conditional, test-conditional (referred to as `object-conditional' in that work), and label-conditional coverage. The strong training-conditional guarantees for split conformal prediction (Theorem~\ref{thm:training-conditional-split-CP}) is proved in a slightly different form in \citet{vovk2012conditional}; the subgaussian tail bound on the Beta distribution is proved in \citet{elder2016bayesian,marchal2017sub}. This type of high-probability control is related to the classical notion of tolerance regions (see, e.g., \citep{wald1943}) and more recent \emph{probably approximately correct} (PAC) results in statistical learning theory~\citep{vapnik2015uniform, valiant1984theory}. More recent PAC-type analyses of conformal methods include the work of \cite{park2020PAC,park2021pac}.
\citet{bian2022training} prove the impossibility of establishing training-conditional coverage for full conformal. 
The asymptotic distribution of the training-conditional coverage rate in transductive conformal prediction (i.e., full conformal prediction with multiple test data points) is studied in~\citet{gazin2024asymptotics}.

Turning to test-conditional coverage, achieving coverage conditionally on a discrete covariate is discussed in many papers, dating back to early work such as \cite{vovk2003mondrian}. In contrast, the impossibility of test-conditional coverage in the nonatomic setting is established by \citet{lei2014distribution,vovk2012conditional}; Theorem~\ref{thm:conditional_infinite} and Corollary~\ref{cor:infinite-lebesgue} give a version of these results. \citet{duchi2024predictive} extends these results to a setting where the distribution of $X$ is nearly nonatomic (in the sense that $P_X$ can only place extremely small mass on any point). The relaxed version of test-conditional coverage, and the corresponding hardness result (Theorem~\ref{thm:hardness-test-conditional-coverage-relaxed}), appear in \citet{barber2021limits}.
Multiple recent works have proposed methodologies to consider more practical relaxations of test-conditional coverage, for detecting regions of $\cX$ where local coverage is too low (or too high) and adjusting the constructed prediction interval to avoid such issues, e.g., 
\cite{ali2022lifecycle}. 
The construction of the group-conditional conformal prediction set can also be motivated as a quantile regression on the set of scores, as developed in the work of \cite{jung2022batch}. \cite{jung2022batch,gibbs2023conformal} extends the group-conditional approach to handle overlapping groups, and other relaxations of conditional coverage.

\citet{lofstrom2015bias} study label-conditional coverage, which we discuss in Section~\ref{sec:lab-conditional}. \citet{ding2024class} extends 
label-conditional coverage to settings with a large number of classes, with an approach based on clustering.
Mondrian prediction (Section~\ref{sec:mondrian}) and its accompanying guarantee (Theorem~\ref{thm:mondrian-cvg}) are developed in \citet{vovk2003mondrian, vovk2005algorithmic}, and the results for bin-wise test-conditional coverage (Theorem~\ref{thm:bin_conditional}) and label-conditional coverage (Theorem~\ref{thm:lab-cond-cvg}) can be viewed as special cases. 

Finally, we add a few technical notes. The first step in the proof of Theorem~\ref{thm:training-conditional-full-CP}, where the map $a:\cX\rightarrow\{0,\dots,n-1\}$ is defined, is due to \citet[Proposition A.1]{dudley2011concrete}; the same result also enables the construction of the nested sets $B_1\supseteq B_2\supseteq \dots$ in the proof of Lemma~\ref{lemma_for_thm:hardness-test-conditional-coverage-relaxed}, and the construction of the set $B$ in the step~\eqref{eqn:sandwich_B} for proving Theorem~\ref{thm:hardness-test-conditional-coverage-relaxed}. Furthermore, the bound on total variation distance between sampling with replacement and sampling without replacement appearing in the proof of Lemma~\ref{lem:sample-resample} can be found in \cite{stam1978distance}.

\section*{Exercises}
\addcontentsline{toc}{section}{\protect\numberline{}\textnormal{\hspace{-0.8cm}Exercises}}
\begin{enumerate}[font=\bfseries, label={\thechapter.\arabic*}, labelsep=1em, itemsep=1em]
\item Consider split conformal prediction with a fixed score function $s(x,y)$. Suppose we have i.i.d.\ data $(X_i, Y_i)_{i\in[n]}$ and consider a hypothesized test point $(x_{\mathrm{test}}, y_{\mathrm{test}})$, which a fixed element of $\cX \times \cY$. Give an explicit expression for 
    \begin{equation}
        \P\left(y_{\mathrm{test}} \in \cC(x_{\mathrm{test}})\right),
    \end{equation}
    where $\cC$ is the split conformal prediction set defined in~\eqref{eq:split-conformal-sets}. (Since $(x_{\mathrm{test}}, y_{\mathrm{test}})$ is fixed, the randomness in the above expression is only the data $(X_i, Y_i)_{i\in[n]}$.)
\item Theorem~\ref{thm:training-conditional-split-CP} proved that, for i.i.d.\ data, the split conformal prediction method offers a training-conditional coverage guarantee---that is, $\P(Y_{n+1}\in\cC(X_{n+1})\mid \cD_n)$ concentrates at (or above) $1-\alpha$, for large $n$. Construct an example of a distribution on $Z_1,\dots,Z_{n+1}$ that is exchangeable (but not i.i.d.), along with a fixed score function $s$, for which this fails (even if $n$ is arbitrarily large). To make the task concrete, the goal is to find an example where the training-conditional coverage $\P(Y_{n+1}\in\cC(X_{n+1})\mid \cD_n)$ is zero (or approximately zero) with probability $\alpha$ (or $\approx \alpha$).
\item In Section~\ref{sec:test-conditional-relax}, we considered a relaxation of test-conditional coverage~\eqref{eqn:test_conditional_relax_alpha_delta}, restated here for convenience:
    \[\P(Y_{n+1}\in\cC(X_{n+1})\mid X_{n+1}\in\cX_0)\geq 1-\alpha \textnormal{ for all $P$, and all $\cX_0\subseteq\cX$ with $P_X(\cX_0)\geq \delta$.}\]
    Recall that Theorem~\ref{thm:hardness-test-conditional-coverage-relaxed},  established a hardness result for defining procedures that achieve this condition. 
    
    We might also consider alternative relaxations, for example,
    \begin{equation}\label{eqn:test_conditional_relax_alpha_delta__alternative}
        \P(\alpha(X_{n+1}) > \alpha) \leq \delta, \textnormal{ where }\alpha(X_{n+1}) = \P(Y_{n+1}\not \in\cC(X_{n+1})\mid X_{n+1}).
    \end{equation}
    Prove that these two relaxations are in fact essentially equivalent (which implies that this alternative relaxation is also hard to achieve distribution-free). Specifically, prove that~\eqref{eqn:test_conditional_relax_alpha_delta} implies~\eqref{eqn:test_conditional_relax_alpha_delta__alternative}. And, prove that if~\eqref{eqn:test_conditional_relax_alpha_delta__alternative} holds, then~\eqref{eqn:test_conditional_relax_alpha_delta} holds with $(\alpha',\delta')$ in place of $(\alpha,\delta)$---you will need to define $\alpha',\delta'$ in terms of $\alpha,\delta$.
\item Use the sample--resample construction (Lemma~\ref{lem:sample-resample}), or a variant of this result as appropriate, to prove the following hardness results (note that these results are for inference problems that are not related to prediction). 
    \begin{enumerate}
        \item Consider the problem of testing whether a distribution $P$ on $\R$ has any nonnegligible point masses, given data $X_1,\dots,X_n\iidsim P$. Specifically, fix any $\epsilon>0$, and let $\psi: \R^n\to \{0,1\}$ be a hypothesis test, returning an answer of $1$ to indicate that we are confident that every point $x\in\R$ has probability $<\epsilon$ under the distribution $P$. Suppose $\psi$ has a distribution-free bound on Type I error rate,
        \[\P_P(\psi(X_1,\dots,X_n) = 1)\leq \alpha\textnormal{ for any $P$ with $\sup_{x\in\R}P(x) \geq \epsilon$}.\] 
        Prove that if $n \ll \epsilon^{-1/2}$ then such a test cannot be powerful. Specifically, show that
        \[\P_P(\psi(X_1,\dots,X_n) = 1)\leq \alpha + \bigo(\epsilon n^2).\]
        \item Consider the problem of inference on the noise in a regression problem: given data points $(X_1,Y_1),\dots,(X_n,Y_n)\iidsim P$ for some distribution $P$ on $\cX\times\R$, our aim is to provide a confidence interval for the parameter $\E_P[\Var_P(Y\mid X)]$. Suppose that $\cC$ is a confidence interval for this parameter (i.e., $\cC\subseteq\R$ is trained a as a function of the data $(X_1,Y_1),\dots,(X_n,Y_n)$, which satisfies a distribution-free coverage guarantee:
        \[\P_P(\E_P[\Var_P(Y\mid X)]\in\cC)\geq 1-\alpha\textnormal{ for all $P$}.\]
        Prove that, in the case of nonatomic features, this confidence interval cannot provide a meaningful lower bound on the noise level:
        \[\P_P(0 \in \cC) \geq 1-\alpha\textnormal{ for any $P$ with $P_X$ nonatomic.}\]
        \end{enumerate}
        \item Suppose that $\cC$ is any procedure that satisfies distribution-free test-conditional coverage as in the assumption of Theorem~\ref{thm:conditional_infinite}.
    Suppose that $P$ is a distribution on $\cX\times\R$, and that $P_X$ has atoms, but that the total probability mass of the atoms is less than $1$. 
    Prove that $\P\big(\textnormal{Leb}(\cC(X_{n+1})) = \infty\big) > 0$.
    \item Suppose we have i.i.d. data $(X_1, Y_1), (X_2, Y_2), \dots, (X_{n+1}, Y_{n+1})$. Consider split conformal prediction with binning as in~\eqref{eqn:test_cond_coverage_by_bin_C}. Derive the distribution of the average bin-wise training-conditional coverage,
    \begin{equation}
    \frac{1}{K} \sum_{k=1}^K \P\left(Y_{n+1}\in\cC(X_{n+1})\mid \cD_n \, ;\,  X_{n+1}\in\cX_k\right),
    \end{equation}
    if we assume that scores are distinct almost surely (i.e., no ties).
    (You do not need to give an algebraic expression for the density or CDF, but describe the distribution in sufficient detail that you could simulate from it on a computer.)
\end{enumerate}

\chapter{A Model-Based Perspective on Conformal Prediction}
\label{chapter:model-based}

The preceding chapters of this book give distribution-free coverage guarantees for conformal prediction, assuming only exchangeability.
In this chapter, we give stronger guarantees for conformal prediction when we additionally assume some prior knowledge of the data distribution.
By incorporating distributional assumptions into the choice of the score function $s$, conformal prediction will inherit stronger, model-based guarantees, such as optimal set size or conditional coverage.
Conformal prediction with an appropriately chosen score function therefore offers the best of both worlds: if the model is correct, the resulting performance will have the same strong guarantees as if we had not used conformal; if instead our model assumptions are wrong, conformal still ensures marginal coverage. In summary:
\begin{align*}
    \textnormal{stronger modeling assumptions hold} &\Rightarrow \textnormal{stronger guarantees for conformal;} \\
    \textnormal{modeling assumptions fail but data is exchangeable} &\Rightarrow \textnormal{marginal coverage guarantee for conformal.}
\end{align*}
Conformal prediction should thus be understood as a tool that we use in conjunction with a model, rather than as an alternative to a model. 

This chapter will begin by studying a range of settings (`case studies') where, given a correctly specified model that can be consistently estimated from data, running conformal with an appropriately chosen score function $s$ will inherit the same optimality properties as a model-based oracle method. 
These case studies give intuition for how to choose the score function in practice.
Then, we will present a unified framework that leads to asymptotic optimality guarantees for conformal prediction under modeling assumptions.
Finally, we will close with a section on modeling assumptions that cause conformal prediction to be robust to violations of exchangeability.

\section{Oracle approximations for optimality guarantees}\label{sec:oracle_highlevel}
\index{optimality}
\index{oracle prediction interval}

We begin with a general design principle: to build a good conformal score, we should first define an `oracle' model that is designed to target the desired aims, and then design a conformal score to approximate this oracle. We will then make this approximation precise with asymptotic results. 

There are many possible aims, depending on the particular setting. For example, we might be interested in optimizing the length of the prediction interval, or, in achieving (an approximation of) test-conditional coverage. These aims, and several others that we will consider in this chapter, are quite different from each other, but can be studied in a unified way. 

We will examine a number of different case studies, each in a different setting with various aims, but each following the same high-level framework. First, we will define the \textbf{aim} for the particular setting (e.g., minimize prediction interval length). We then identify the \textbf{oracle} prediction set construction that would optimize this aim, if we had access to knowledge of the data distribution $P$. With this in place, our next step is \textbf{choosing the score} in a way that attempts to mimic this oracle. We then state a \textbf{model assumption} (for instance, consistent estimation of some true regression function), which enables an \textbf{asymptotic optimality} guarantee establishing that we are able to learn an approximation to the oracle prediction set via split conformal, and thus achieve the stated aim in an asymptotic sense.

The key tool for our convergence guarantees will be the following result, which we state informally here and develop more formally later on in the chapter:
\begin{theorem}[Informal version of Theorems~\ref{thm:splitCP_asymp_formal_qn} and \ref{thm:splitCP_asymp_formal_random_x}]
    \label{thm:splitCP_asymp_informal}
    Let the data points be drawn i.i.d.\ from $P$. 
    Define the oracle set
    \begin{equation}\label{eqn:oracle_set}
        \cC^*(x) = \left\{y : s^*(x,y) \leq q^* \right\},
    \end{equation} 
    where $s^*$ is an oracle score function and $q^*$ is the $(1-\alpha)$-quantile of the distribution of $s^*(X,Y)$ under $(X,Y)\sim P$. Let $\cC_n(x)$ be the usual split conformal set,
        \begin{equation}
        \cC_n(x) = \left\{y : s_n(x,y ) \leq \hat{q}_n \right\},
    \end{equation} 
    where $s_n$ is the pretrained conformal score function.
    Then, under appropriate regularity conditions, if $s_n\to s^*$ then
    \begin{equation}
        \label{eq:convergence-setdifference-as}
        \hat{q}_n \to q^*, \textnormal{ and } \cC_n \to \cC^*.
    \end{equation}
\end{theorem}
In words, this result says that if our base model is consistent for the true model (in that $s_n$ converges to an oracle score $s^*$), then conformal prediction will converge to the oracle prediction set given by $\cC^*$.
Of course, we will need to formalize all these notions of convergence.
With the appropriate definitions and formalities in place, we will see that these convergence guarantees allow us to show that the split conformal set $\cC_n$ asymptotically reaches the stated aim of the oracle score (e.g., minimizing the length of the prediction interval).

In the upcoming case studies, we will assume that the score function $s^*$ is capturing the true model for the data---that is, $s^*$ is truly an `oracle' and therefore leads to a prediction set $\cC^*$ that is optimally designed for the distribution $P$ of the data. However, the general theoretical framework of Theorem~\ref{thm:splitCP_asymp_informal} (and its formal versions, Theorems~\ref{thm:splitCP_asymp_formal_qn} and \ref{thm:splitCP_asymp_formal_random_x}) can be interpreted more broadly: 
the so-called oracle score function $s^*$ can technically be any fixed function, and does not necessarily need to reflect the true model for the data.
For instance, if the true model for $Y\mid X$ is not actually linear, but we are using linear quantile regression to construct the CQR score $s_n$, then we might expect the split conformal prediction set to converge to the best possible interval $\cC^*(x)$ whose endpoints are linear functions of $x$.

\paragraph{Defining the asymptotic regime.}
Since we will be developing asymptotic guarantees for the split conformal sets $\cC_n$, formally we are considering a \emph{sequence} of problems, indexed by $n\geq 1$, while the underlying distribution $P$ of the data remains fixed. At each value of $n$, we write $s_n$ to denote the pretrained conformal score function that will be used in the construction of the split conformal set, and then the conformal quantile $\hat{q}_n$ will be determined by a calibration set of size $n$. For $\cC_n(x)=\{y:s_n(y)\leq\hat{q}_n\}$ to converge to some oracle, then, we will need \emph{both} $s_n$ and $\hat{q}_n$ to converge. 

Recall from Section~\ref{sec:split-conformal} that when we run split conformal in practice, the available data is partitioned into two sets: a `pretraining set' $\cD_{{\rm pre},n}$ used for training the score function $s_n$, and a calibration set $\cD_n$ (of size $n$) used to then determine the conformal quantile $\hat{q}_n$. (Here we introduce a subscript $n$ into each of these objects in order to emphasize that we are now considering a sequence of problems indexed by $n$.) In order to achieve asymptotic guarantees, then, we would need the sample size of $\cD_{{\rm pre},n}$ to grow with $n$ in order to ensure convergence of $s_n$. 
For example, if we choose to use half the data for training the score function and the other half for calibration, then to obtain a calibration set of size $n$, we would need a total of $2n$ data points (so that the pretraining set $\cD_{{\rm pre},n}$ and the calibration set $\cD_n$ each contain $n$ data points).

In the case studies below, we will not make any explicit assumptions about the sample size of  the pretraining set $\cD_{{\rm pre},n}$. Instead, we will place assumptions on the convergence of $s_n$, which implicitly means that the sample size of $\cD_{{\rm pre},n}$ must grow as $n\to\infty$ in order to allow for this convergence to occur. (See Section~\ref{sec:unified_framework_splitCP_asymp} for a more formal definition of this asymptotic framework.)

\section{Case studies for classification}
\label{sec:case_studies_classification}

This section considers the classification setting, where the response $Y$ is categorical and takes values in a finite set $\cY$. For each of our two case studies, we will define
$\pi^*$ as the true label probability function,
\[\pi^*(y\mid x) = \P_{(X,Y)\sim P}(Y=y\mid X=x),\]
and $\hat{\pi}_n$ will denote a pretrained estimate of $\pi^*$ (that is, $\hat{\pi}_n$ is an estimate of $\pi^*$ that is fitted using the pretraining set $\cD_{{\rm pre},n}$, with $\hat{\pi}_n(\cdot\mid x)$ being a probability distribution over $\cY$).

\subsection{Case study: classification with minimal set size}
\label{sec:case_study_1_classification}
\index{optimality!set size}

\paragraph{Aim.}
In this first case study, our goal is to provide the sets with the smallest average size among those that have $1-\alpha$ marginal coverage---we would like to construct $\cC$ to solve the following optimization problem:
\begin{equation}
\begin{aligned}
\minimize \quad & \E_{X\sim P_X}[|\cC(X)|] \\
\st \quad & \P_{(X,Y)\sim P}(Y \in \cC(X)) \geq 1-\alpha.
\end{aligned}
\end{equation}
Here in the categorical setting, $|\cC(X)|$ refers to the cardinality of $\cC(X)$, which is a subset of the finite set of labels $\cY$.

\paragraph{Oracle.}

Suppose for simplicity that the conditional probability $\pi^*(Y\mid X)$, which is a random variable, has a continuous distribution (under $(X,Y)\sim P$). It is straightforward to prove that the solution to the above optimization problem then has the form $\cC^*(x) = \{ y : \pi^*(y \mid x) \geq t^* \}$, for some appropriate value $t^*$. Rewriting this into score function notation, we can define the oracle score
\begin{equation}
    \label{eqn:case_study_1_classification}
    s^*(x,y) = -\pi^*(y\mid x),
\end{equation}
where we take the negative of the conditional label probability $\pi^*$ because our prediction set should include labels $y$ with high probability (which therefore should correspond to low values of the score).
We can then construct the oracle prediction set,
\[\cC^*(x) = \{ y : s^*(x,y) \leq q^* \},\]
where 
\[q^* = \inf\left\{q \in\R: \P_{(X,Y)\sim P}(s^*(X,Y)\leq q) \geq 1-\alpha\right\},\]
i.e., the $(1-\alpha)$-quantile of the distribution of $s^*(X,Y)$, under $(X,Y)\sim P$.

\paragraph{Choosing the score.}\index{score function!high-probability score}
To approximate the oracle sets, we can instead use plug-in estimate $\hat \pi_n$ of the probability $\pi^*$. This choice leads to the score function
\begin{equation}
    s_n(x,y) = -\hat \pi_n(y \mid x),
\end{equation}
which is the  high-probability score defined previously in Chapter~\ref{chapter:introduction}. 
This results in the prediction set
\[\cC_n(x) = \{y\in\cY: s_n(x,y) \leq \hat{q}_n\} = \{y\in\cY : \hat{\pi}_n(y\mid x) \geq -\hat{q}_n\},\]
i.e., all possible labels $y\in\cY$ whose estimated conditional probability is sufficiently large.

\paragraph{Model assumption and asymptotic optimality.} We will assume that the estimated conditional distribution of $Y\mid X$, given by $\hat{\pi}_n$ as defined above, is consistent as an estimator of the true $\pi^*$. Specifically, at any $x\in\cX$, consider the quantity
\[\dtv\big(\pi^*(\cdot\mid x), \hat\pi_n(\cdot \mid x)\big) = \frac{1}{2}\sum_{y\in\cY} \big|\pi^*(y\mid x) - \hat\pi_n(y\mid x)\big|,\]
which is the total variation distance between 
the true and estimated conditional distributions of $Y\mid X=x$. We will assume that this estimation error is small on average over the test feature $X$---that is, we will measure consistency via the quantity
\[\E_{X\sim P_X}\left[\dtv\big(\pi^*(\cdot\mid X), \hat\pi_n(\cdot \mid X)\big)\right]= \E_{X\sim P_X}\left[\frac{1}{2}\sum_{y\in\cY} \big|\pi^*(y\mid X) - \hat\pi_n(y\mid X)\big|\right].\]
Informally, assuming that 
\[\E_{X\sim P_X}\left[\dtv\big(\pi^*(\cdot\mid X), \hat\pi_n(\cdot \mid X)\big)\right]\to 0\]
can be thought of as effectively claiming that, for large $n$, the conditional distribution of $Y\mid X=x$ is estimated accurately at most values $x$.

For this consistency condition, and hereafter throughout the chapter, the notation $\E_{X\sim P_X}$ and $\P_{X\sim P_X}$ (or, $\E_{(X,Y)\sim P}$ and $\P_{(X,Y)\sim P}$)  denotes that we are taking expected value or probability \emph{only} with respect to the randomness of a random test point $X$ sampled according to $P_X$ (or $(X,Y)\sim P$). In particular, we are conditioning on the pretraining and calibration datasets---that is, $\hat\pi_n$, $\hat{q}_n$, and other quantities that are constructed when running split conformal prediction, are all being treated as fixed.

We now arrive at our formal result. Under the above model assumption, the split conformal prediction set $\cC_n$ constructed with the score $s_n$ above is asymptotically a solution to the optimization problem~\eqref{eqn:case_study_1_classification}.

\begin{proposition}\label{prop:case_study_1_classification}
    Assume that the conditional probability $\pi^*(Y\mid X)$ is continuously distributed under $(X,Y)\sim P$.  
    Then the following claim holds almost surely: if
    \[\E_{X\sim P_X}\left[\dtv\big(\pi^*(\cdot\mid X), \hat\pi_n(\cdot \mid X)\big)\right]\to 0,\]
    then
    \[
 \limsup_{n\to\infty}\E_{X\sim P_X}[|\cC_n(X)|] 
  \leq \E_{X\sim P_X}[|\cC^*(X)|]
 \textnormal{ \  and \  }\liminf_{n\to\infty}\P_{(X,Y)\sim P}(Y\in\cC_n(X)) \geq 1-\alpha,\]
    i.e., $\cC_n$ is asymptotically optimal for the aim~\eqref{eqn:case_study_1_classification}.
\end{proposition}

In words, conformal prediction with the high-probability score yields asymptotically optimal set sizes as the fitted model probabilities $\hat \pi_n$ approach the true model $\pi^*$. This suggests that the high-probability score will generally be a good choice if one is seeking small average set size. Note, however, that the aim of having small average set length is often in tension with having good conditional coverage. We will turn to this in the next case study.

In the above result, we assume that the conditional probability $\pi^*(Y\mid X)$ is continuously distributed. This assumption (and analogous continuity conditions in the case studies that follow) allows us to avoid ties among the conformal scores, but is not necessary if we use randomized versions of conformal prediction that break ties at random (see Chapter~\ref{chapter:further-topics}).

\subsection{Case study: classification with minimal set size and conditional coverage}\label{sec:case_study_2_classification}
\index{optimality!conditional coverage}

\paragraph{Aim.}
Next, we provide sets that have  the smallest size among those that have \emph{conditional} coverage. Note that optimizing set size while only enforcing marginal coverage can result in poor behavior, as we saw in Chapter~\ref{chapter:conditional} (see Figure~\ref{fig:conditional-coverage}). Consequently, it is often better to target conditional coverage, as we do next.
The corresponding optimization problem is as follows:
\begin{equation}
\begin{aligned}
\label{eqn:case_study_2_classification}
\textnormal{minimize} \quad & \E_{X \sim P_X}[|\cC(X)|] \\
\textrm{subject to} \quad & \P_{Y\sim P_{Y\mid X}}(Y \in \cC(X) \mid X) \geq 1-\alpha' \textnormal{ almost surely.}
\end{aligned}
\end{equation}
(We will see shortly why we write $\alpha'$, rather than $\alpha$, to define the aim.)
The sets $\cC(X)$ that solve the above optimization problem also have the smallest size conditionally on $X$, almost surely.
In other words, for almost every $x\in\cX$, there cannot be a smaller subset of $\cY$ that would still maintain $1-\alpha'$ coverage, since this would contradict the optimality of our set.

\paragraph{Oracle.}
A short calculation will show that the following set is the solution to our aim: 
\begin{equation}\label{eqn:case_study_2_classification_preliminary}\cC^*(x) = \{y: \pi^*(y\mid x)\geq t^*(x)\},\end{equation}
where, at each $x\in\cX$, $t^*(x)$ is chosen so that $\P_{Y\sim P_{Y\mid X}}(\pi^*(Y\mid x)\geq t^*(x)\mid X=x) \geq 1-\alpha'$; see Figure~\ref{fig:case_study_2_classification} for an illustration of this set.
The reason for the optimality of this score is relatively straightforward to see: for any $x$, the oracle prediction set $\cC^*(x)$ includes only the most likely classes until the total mass reaches (at least) $1-\alpha'$.
Any other rule would have to include at least as many classes to reach the $1-\alpha'$ threshold.

However, the oracle prediction set given in~\eqref{eqn:case_study_2_classification_preliminary} cannot be immediately converted to a construction of the form $\{y:s^*(x,y)\leq q^*\}$, as in Theorem~\ref{thm:splitCP_asymp_informal}---this is because the threshold $t^*(x)$ appearing in~\eqref{eqn:case_study_2_classification_preliminary} is dependent on $x$, rather than a common value that is shared over all $x\in\cX$. Thus, we will need to reformulate this set. Define
\[s^*(x,y) = \sum_{y'\in\cY} \pi^*(y'\mid x) \cdot \ind{\pi^*(y'\mid x)> \pi^*(y\mid x)},\]
which is the (conditional) probability captured by the set $\{y' : \pi^*(y'\mid x)> \pi^*(y\mid x)\}$, i.e., all labels that are \emph{strictly more likely} than the label $y$ (given features $x$).
With this reformulation, we can verify that the set defined in~\eqref{eqn:case_study_2_classification_preliminary}  can equivalently be written as
\[\cC^*(x) = \{y : s^*(x,y) < 1-\alpha' \}.\]
If we again assume a continuity condition for simplicity---namely, that $\P_{(X,Y)\sim P}(s^*(X,Y)=1-\alpha')=0$---then the set
\begin{equation}
    \label{eqn:case_study_2_classification_reformulated}
    \cC^*(x) = \{y : s^*(x,y) \leq 1-\alpha' \}
\end{equation}
is also a solution to~\eqref{eqn:case_study_2_classification}, since it differs from the first definition only on an event of measure zero. In particular, this means that the oracle prediction set can be defined by a criterion of the form $s^*(x,y)\leq q^*$ (where $q^*=1-\alpha'$), which thus fits into the notation of Theorem~\ref{thm:splitCP_asymp_informal}. However, again recalling Theorem~\ref{thm:splitCP_asymp_informal}, in order to produce a conformal prediction set that approximates $\cC^*(X_{n+1})$ (when running conformal prediction at coverage level $1-\alpha$), we need $q^*$ to be the $(1-\alpha)$-quantile of $s^*(X,Y)$ under $(X,Y)\sim P$. In other words, if we define $\alpha'$ to satisfy
\begin{equation}\label{eqn:def_alpha_prime_case_study}1-\alpha' = q^*\textnormal{ where $q^*$ is the $(1-\alpha)$-quantile of $s^*(X,Y)$ under $(X,Y)\sim P$,}\end{equation}
then the oracle set $\cC^*(x)=\{y:s^*(x,y)\leq q^*\}$ will achieve \emph{conditional} coverage at level $1-\alpha'$ (not $1-\alpha$). This is not surprising: due to discreteness of the problem, in order to achieve $\geq 1-\alpha'$ conditional coverage at every $x$, we generally will achieve a higher level of marginal coverage, $1-\alpha > 1-\alpha'$.

\begin{figure}[t]
    \centering
    \includegraphics[width=0.55\textwidth]{\diagramspath oracle-classification-conditional-coverage.pdf}
    \caption{\textbf{An illustration of the oracle prediction set for classification with conditional coverage.} For each value of $x$, the highlighted bars illustrate the oracle prediction set $\cC^*(x)$ constructed in~\eqref{eqn:case_study_2_classification_preliminary}, consisting of all $y$ values for which $\pi^*(y\mid x)\geq t^*(x)$. The prediction sets all contain a $1-\alpha'$ proportion of the conditional probability mass, at each value $x$.}
    \commentAlt{Three bar plots, labeled $x=x_1$, $x=x_2$, $x=x_3$, display conditional probability $\pi^*(y\mid x)$. For each plot, a horizontal line marks a threshold $t^*(x_1)$, $t^*(x_2)$, or $t^*(x_3)$. Bars whose level lies above the threshold are highlighted.}
    \label{fig:case_study_2_classification}
\end{figure}

\paragraph{Choosing the score.} \index{score function!cumulative-probability score}
We again use a plug-in approach to approximate the oracle score distribution:
\begin{equation}
    s_n(x,y) = \sum_{y'\in\cY} \hat \pi_n(y'\mid x) \cdot \ind{\hat \pi_n(y'\mid x)> \hat \pi_n(y\mid x)}.
\end{equation}
To contrast this with the high-probability score of Section~\ref{sec:case_study_1_classification}, we will refer to this as the \emph{cumulative-probability score}. (This construction is also sometimes called an \emph{adaptive prediction set} score.)

\paragraph{Model assumption and asymptotic optimality.}

We will require the same convergence condition on $\hat\pi_n$ as in the previous case study (in Section~\ref{sec:case_study_1_classification}): we will again assume that
\begin{equation}
    \E_{X\sim P_X}\left[\dtv\big(\pi^*(\cdot\mid X), \hat\pi_n(\cdot \mid X)\big)\right] \to 0.
\end{equation}

We now want to show that, under the above model assumption, the split conformal prediction set $\cC_n$ constructed with $s_n$ above is asymptotically a solution to the optimization problem~\eqref{eqn:case_study_2_classification}.  To this end, we first need to formalize what we mean by an `asymptotic' solution---what does it mean to achieve conditional coverage in an asymptotic sense? There are multiple valid answers to this question; here we will choose a particular one. Recall that exact conditional coverage (at level $1-\alpha'$, rather than $1-\alpha$, as before) means that
\[\P_{Y\sim P_{Y\mid X}}(Y\in\cC_n(X)\mid X) \geq 1-\alpha'\]
holds almost surely, with respect to $X\sim P_X$. We will take the convention that achieving conditional coverage in an approximate sense means that this inequality holds, at least approximately, for most $X$---i.e., with probability $\approx$ 1 with respect to $X\sim P_X$.

\begin{proposition}\label{prop:case_study_2_classification}
    Assume that the oracle score $s^*(X,Y)$ satisfies the following condition:
\begin{equation}\label{eqn:continuity_for_case_study_2_classification}\P_{(X,Y)\sim P}(s^*(X,Y) = 0) < 1-\alpha, \textnormal{ and }
\P_{(X,Y)\sim P}(s^*(X,Y) = t) = 0\textnormal{ for all $t>0$}.
\end{equation}
Assume also that, for any $y\neq y'\in\cY$, we have $\pi^*(y\mid X)\neq\pi^*(y'\mid X)$ almost surely under $X\sim P_X$.  
    Then the following claim holds almost surely: if
     \begin{equation}
         \E_{X\sim P_X}\left[\dtv\big(\pi^*(\cdot\mid X), \hat\pi_n(\cdot \mid X)\big)\right]\to 0
     \end{equation}
    then
\[\limsup_{n\to\infty}\E_{X\sim P_X}[|\cC_n(X)| ]
  \leq \E_{X\sim P_X}[|\cC^*(X)|]\]
  and
  \[
         \lim_{n\to\infty}\P_{X\sim P_X}\Big( \,\P_{Y\sim P_{Y\mid X}}(Y\in\cC_n(X)\mid X) \geq 1-\alpha' - \epsilon\,\Big) = 1\textnormal{ for all $\epsilon>0$},\]
    i.e., $\cC_n$ is asymptotically optimal for the aim~\eqref{eqn:case_study_2_classification} with conditional coverage level $1-\alpha'$ (where $\alpha'$ is defined as in~\eqref{eqn:def_alpha_prime_case_study}).
\end{proposition}

In words, conformal prediction with the cumulative-probability score gives the smallest sets with conditional coverage as the fitted model probabilities $\hat \pi_n$ approach the true probabilities $\pi$. This suggests the cumulative-probability score is generally a good choice when one is seeking to approximately achieve conditional coverage.

\section{Case studies for regression}
\label{sec:case_studies_regression}

This section considers the regression setting, where the response $Y$ is real-valued, i.e., $\cY = \R$.
We will assume that the conditional distribution of $Y\mid X=x$, under the joint distribution $(X,Y)\sim P$, has a conditional density
\begin{equation}
    f^*(y \mid x)
\end{equation}
with respect to Lebesgue measure. For the first case study, we assume access to a pretrained estimate of $f^*$, given by a conditional density $\hf_n(y \mid x)$.
We also make use of the conditional quantile function, $\tau^*(x; \beta)$ (that is, the $\beta$-quantile of the conditional distribution of $Y$ given $X=x$), and its corresponding estimator, $\hat{\tau}_n(x ; \beta)$, in the second case study.

\subsection{Case study: regression with minimal set size}\label{sec:case_study_1_regression}
\index{optimality!set size}

\paragraph{Aim.}
Our next goal is to provide the sets with the smallest average size among those that have $1-\alpha$ marginal coverage in the regression setting.
This goal is entirely analogous to that in Section~\ref{sec:case_study_1_classification}, and the score solves the same optimization problem,
\begin{equation}
\begin{aligned}
\textnormal{minimize} \quad & \E_{X\sim P_X}[\textnormal{Leb}(\cC(X))] \\
\textrm{subject to} \quad & \P_{(X,Y)\sim P}(Y \in \cC(X)) \geq 1-\alpha,
\end{aligned}
\label{eq:case_study_1_regression}
\end{equation}
the only difference being that we now consider the Lebesgue measure $\textnormal{Leb}(\cC(X))$ of the prediction set, rather than its cardinality $|\cC(X)|$.

\paragraph{Oracle.}
By a standard argument reminiscent of the Neyman--Pearson lemma, if we assume for simplicity that $f^*(Y\mid X)$ is continuously distributed under $(X,Y)\sim P$, then the set construction with the smallest average size has the form $\cC^*(x) = \{ y : f^*(y \mid x) \geq t^* \}$, where $t^*$ is chosen as the largest possible value that still yields coverage at level $\geq 1-\alpha$:
\begin{equation}
    t^* = \sup \{t \in \R : \P_{(X,Y)\sim P}(f^*(Y \mid X) \geq t) \geq 1-\alpha \}.
\end{equation}
As before, this can equivalently be written as $\cC^*(x) = \{ y : s^*(x,y) \leq q^* \}$, where $s^*(x,y) = -f^*(y \mid x)$ and $q^* = -t^*$ (here $q^*$ is again the $(1-\alpha)$-quantile of $s^*(X,Y)$).

\paragraph{Choosing the score.} \index{score function!high-density score}
To approximate the oracle sets, we use an estimate of the density, $\hf_n$, in place of the unknown $f^*$:
\begin{equation}
    s_n(x,y) = -\hf_n(y \mid x).
\end{equation}
We will refer to this as the \emph{high-density score}---it is exactly analogous to the high-probability score of Section~\ref{sec:case_study_1_classification} but with conditional density in place of conditional probability. This leads to a prediction set of the form
\[\cC_n(x) = \{y\in\R: s_n(x,y) \leq \hat{q}_n\} = \{y\in\R : \hf_n(y\mid x) \geq -\hat{q}_n\},\]
i.e., all values $y$ with a sufficiently high estimated conditional density.

\paragraph{Model assumption and asymptotic optimality.}
We will assume that the estimated conditional density of $Y\mid X$, given by $\hf_n$ as defined above, is consistent as an estimator of the true $f^*$ with respect to TV distance. 
For any $x\in\cX$, we can calculate the TV distance between the true and estimated conditional distributions of $Y\mid X=x$ as
\begin{equation}
    \dtv\big(f^*(\cdot\mid x), \hf_n(\cdot \mid x)\big) = \frac{1}{2}\int_{y\in\R} \big|f^*(y\mid x) - \hf_n(y\mid x)\big|\;\mathsf{d}y.
\end{equation}
We will measure TV distance on average over $X$:
\begin{equation}
    \E_{X\sim P_X}\left[\dtv\big(f^*(\cdot\mid X), \hf_n(\cdot \mid X)\big)\right]= \E_{X\sim P_X}\left[\frac{1}{2}\int_{y\in\cY} \big|f^*(y\mid X) - \hf_n(y\mid X)\big|\;\mathsf{d}y\right].
\end{equation}
The core assumption will be that the estimation error vanishes in TV distance,
\begin{equation}
    \E_{X\sim P_X}\left[\dtv\big(f^*(\cdot\mid X), \hf_n(\cdot \mid X)\big)\right]\to 0,
\end{equation}
which means that for large $n$, the conditional density of $Y\mid X=x$ is estimated accurately at most values $x$.

We now show that, under the above model assumption, the split conformal prediction set $\cC_n$ constructed with $s_n$ above is asymptotically a solution to the optimization problem~\eqref{eq:case_study_1_regression}.

\begin{proposition}\label{prop:case_study_1_regression}
    Assume that the conditional density $f^*(Y\mid X)$ has a continuous distribution under $(X,Y)\sim P$.
    Furthermore, assume that $\sup_{(x,y)}f^*(y\mid x) < \infty$.
    Then the following claim holds almost surely: if
    \[\E_{X\sim P_X}\left[\dtv\big(f^*(\cdot\mid X), \hf_n(\cdot \mid X)\big)\right]\to 0\]
    then
    \[\limsup_{n\to\infty}\E_{X\sim P_X}[\textnormal{Leb}(\cC_n(X))] \leq \E_{X\sim P_X}[\textnormal{Leb}(\cC^*(X))] \textnormal{ \  and \  }\liminf_{n\to\infty}\P_{(X,Y)\sim P}(Y\in\cC_n(X)) \geq 1-\alpha,\]
    i.e., $\cC_n$ is asymptotically optimal for the aim~\eqref{eq:case_study_1_regression}.
\end{proposition}

This result suggests that the high-density score is a good choice when one wants prediction sets that have the smallest average length. However, as mentioned previously in Section~\ref{sec:case_study_1_classification}, aiming for small average length may conflict with achieving conditional coverage. We consider the question of conditional coverage next.

\subsection{Case study: regression with minimal-length equal-tailed intervals}\label{sec:case_study_2_regression}
\index{optimality!equal-tailed coverage}
\index{optimality!conditional coverage}
\index{quantile regression}

\paragraph{Aim.}
Our next goal is to provide the prediction intervals with the smallest average size among those with $1-\alpha$ \emph{equal-tailed conditional coverage}.
Equal-tailed conditional coverage means that, conditionally on $X=x$ for all $x$, the true value of $Y$ will lie below the lower endpoint of the interval with probability $\alpha/2$, and similarly, the true value of $Y$ will lie above the upper endpoint of the interval with probability $\alpha/2$.
Thus, errors happen equally frequently in the upper tail and lower tail.
Setting this up as an optimization problem, we want to solve
\begin{equation}
\begin{aligned}
\textnormal{minimize} \quad & \E_{X\sim P_X}[\textnormal{Leb}(\cC(X))] \\
\textrm{subject to} \quad & \P_{Y\sim P_{Y\mid X}}(Y > \sup\cC(X) \mid X ) \leq \alpha/2\textnormal{ almost surely}, \\
\quad & \P_{Y\sim P_{Y\mid X}}(Y < \inf\cC(X) \mid X ) \leq \alpha/2\textnormal{ almost surely},\\&\textnormal{$\cC(x)$ is an interval for all $x$.}
\end{aligned}
\label{eq:case_study_2_regression}
\end{equation}

\paragraph{Oracle.}
With access to the true conditional quantile function, we can simply take the oracle set to be
\begin{equation}
    \cC^*(x) = [\tau^*(x ; \alpha/2), \tau^*(x ; 1-\alpha/2)],
\end{equation}
where we recall that $\tau^*(x;\beta)$ is the $\beta$-quantile of the conditional distribution of $Y\mid X=x$. Since this conditional distribution is continuous, this prediction interval, whose endpoints are given by the $(\alpha/2)$- and $(1-\alpha/2)$-quantiles of $Y\mid X=x$, solves the optimization problem in~\eqref{eq:case_study_2_regression}.

\paragraph{Choosing the score.} \index{score function!conformalized quantile regression (CQR)}
We now want to define a split conformal prediction set $\cC_n$ that is asymptotically a solution to the optimization problem~\eqref{eq:case_study_1_regression}.
For this purpose, we will construct $\cC_n$ using the score function
\begin{equation}\label{eqn:score_CQR}
    s_n(x,y) = \max\{\hat{\tau}_n(x; \alpha/2) - y, y - \hat{\tau}_n(x; 1-\alpha/2)\}.
\end{equation}
As in Chapter~\ref{chapter:introduction}, since this is the score used in the conformalized quantile regression (CQR) method, we refer to it as the \emph{CQR score}.
This score leads to a prediction interval of the form
\begin{equation}
    \label{eqn:case_study_2_regression_set}
    \cC_n(x) = \left[ \hat{\tau}_n(x; \alpha/2) - \hat{q}_n ,  \hat{\tau}_n(x; 1-\alpha/2) + \hat{q}_n\right],
\end{equation}
where we take the convention that $[a,b]$ denotes the empty set if $a>b$, to handle degenerate cases.
Figure~\ref{fig:case_study_2_regression_set} illustrates the prediction interval in~\eqref{eqn:case_study_2_regression_set} above.

To see how this corresponds to the oracle interval, consider an oracle score,
\begin{equation}\label{eqn:oracle_score_CQR}s^*(x,y) = \max\{\tau^*(x; \alpha/2) - y, y - \tau^*(x; 1-\alpha/2)\}.\end{equation}
For a data point $(x,y)$, this score will be positive if $y$ lies outside the oracle interval $\cC^*(x)$, negative if $y$ lies strictly inside the interval, or zero if $y$ lies on one of the endpoints.
Then, by construction, we have $\cC^*(x) = \{y\in\R: s^*(x,y)\leq q^*\}$ when we choose $q^*=0$, and moreover, $q^*=0$ is the $(1-\alpha)$-quantile of $s^*(X,Y)$.

\begin{figure}[t]
    \centering
    \includegraphics[width=0.8\textwidth]{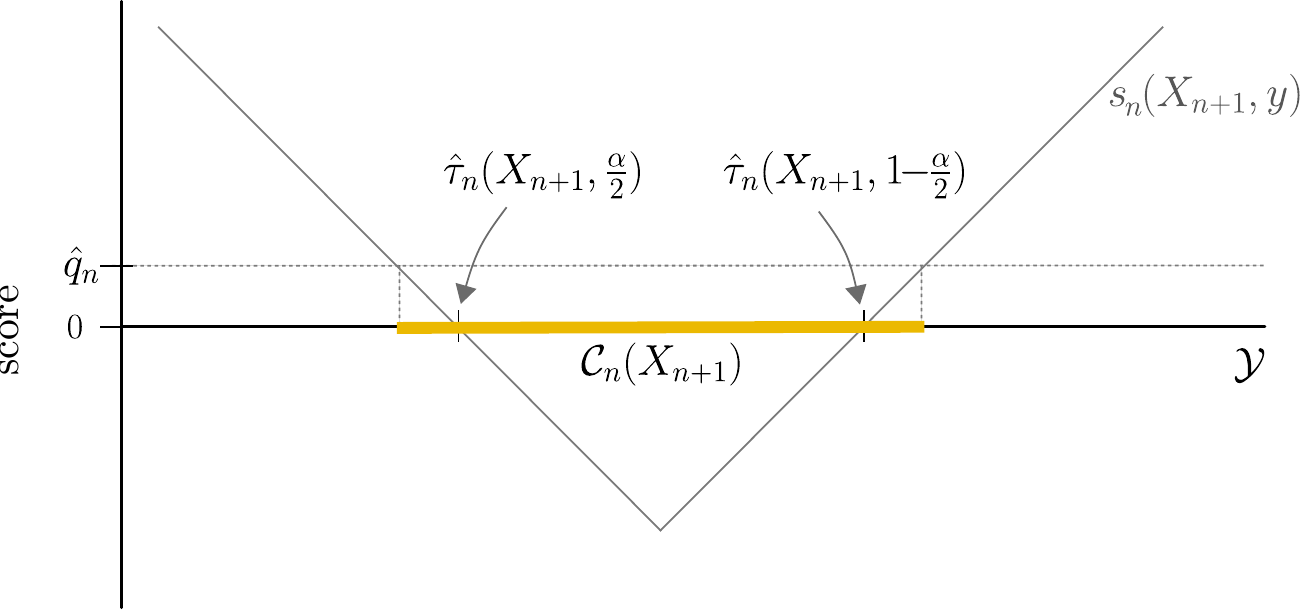}
    \caption{\textbf{An illustration of the prediction set for CQR.} The prediction set $\cC_n(x)$ constructed in~\eqref{eqn:case_study_2_regression_set} is shown as a highlighted region along the horizontal axis.}
    \commentAlt{The figure shows the score function for CQR, which is piecewise linear, first decreasing and then increasing. This function is labeled as $s_n(X_{n+1},y)$. See long description.}
    \commentLongAlt{The figure shows the score function for CQR, which is piecewise linear, first decreasing and then increasing. This function is labeled as $s_n(X_{n+1},y)$. Two horizontal lines indicate values $0$ (lower) and $\hat{q}_n$ (higher) along the vertical axis, which is labeld `score'. The points where the score function intersects the horizontal line at level $0$ are labeled as $\hat{\tau}_n(X_{n+1},\alpha/2)$ on the left, and $\hat{\tau}_n(X_{n+1},1-\alpha/2)$ on the right. A wider region, where the score function lies below the horizontal line at level $\hat{q}_n$, is shaded and labeled as $\cC_n(X_{n+1})$.}
    \label{fig:case_study_2_regression_set}
\end{figure}

\paragraph{Model assumption and asymptotic optimality.}
We will assume that the estimated conditional $(\alpha/2)$- and $(1-\alpha/2)$-quantiles of $Y \mid X$, given by $\hat{\tau}_n(X ; \alpha/2)$ and $\hat{\tau}_n(X ; 1-\alpha/2)$ respectively, converge to their oracle values:
\begin{equation}
    \E_{X \sim P_X}[|\hat{\tau}(X ; \alpha/2) - \tau^*(X ; \alpha/2)|] \to 0 \textnormal{ and } \E_{X \sim P_X}[|\hat{\tau}(X ; 1-\alpha/2) - \tau^*(X ; 1-\alpha/2)|] \to 0.
\end{equation}
This means that, for large $n$, the $(\alpha/2)$- and $(1-\alpha/2)$-quantiles converge on average over $X$.
In this case, the equal-tailed CQR intervals converge to the oracle interval, as stated next.

\begin{proposition}\label{prop:case_study_2_regression}
Assume that the conditional density $f^*$ satisfies $f^*(y\mid x)>0$ for all $(x,y)$, and that
    \begin{equation}\label{eqn:case_study_2_regression_finite_quantile_gap}
    \sup_{x\in\cX}\left(\tau^*(x;1-\alpha/2) - \tau^*(x;\alpha/2)\right) < \infty.\end{equation}
    Then the following claim holds almost surely: if
    \begin{equation}
        \E_{X \sim P_X}[|\hat{\tau}_n(X ; \alpha/2) - \tau^*(X ; \alpha/2)|] \to 0 \textnormal{ and } \E_{X \sim P_X}[|\hat{\tau}_n(X ; 1-\alpha/2) - \tau^*(X ; 1-\alpha/2)|] \to 0,
    \end{equation}
    then
    \begin{equation}
        \limsup_{n\to\infty}\E_{X\sim P_X}[\textnormal{Leb}(\cC_n(X)) ]
  \leq \E_{X\sim P_X}[\textnormal{Leb}(\cC^*(X))],    \end{equation}
    and, for all $\epsilon>0$,
    \begin{equation}
        \lim_{n\to\infty}\P_{X\sim P_X}\Big( \,\P_{Y\sim P_{Y\mid X}}(Y > \sup \cC_n(X)\mid X) \geq \alpha/2+\epsilon \,\Big) = 0\end{equation}
        and
        \begin{equation} \lim_{n\to\infty}\P_{X\sim P_X}\Big( \,\P_{Y\sim P_{Y\mid X}}(Y < \inf \cC_n(X)\mid X) \geq \alpha/2+\epsilon \,\Big) = 0.
    \end{equation}
    That is, $\cC_n$ is asymptotically optimal for the aim~\eqref{eq:case_study_2_regression}.
\end{proposition}

This result suggests that the CQR score is a good choice for approximate conditional coverage with regression problems. Note that other related scores have a similar feature; see the bibliographic notes at end of this chapter.

\section{A unified framework for asymptotic guarantees}
\label{sec:unified_framework_splitCP_asymp}

This section develops the framework underlying the asymptotic optimality guarantees in the previous sections, by formalizing the convergence guarantee presented informally in Theorem~\ref{thm:splitCP_asymp_informal}.
We will then give the proofs of the asymptotic optimality guarantees for the four case studies---Propositions~\ref{prop:case_study_1_classification},~\ref{prop:case_study_2_classification},~\ref{prop:case_study_1_regression}, and~\ref{prop:case_study_2_regression}---using a unified, step-by-step approach.
The theory in this chapter will be presented in a way that emphasizes its generality: this framework is a foundational approach that can be used to prove \emph{many} asymptotic guarantees about conformal prediction, including but not limited to the specific propositions in the case studies above.

\subsection{Defining the framework} 
To begin, we need to lay out a precise asymptotic framework for the question.
We will define a sequence of split conformal prediction problems, indexed by $n$, the size of the calibration set. 
In other words, for each $n$, we will define a pretraining set of size $m_n$ and a calibration set of size $n$:
\[\underbrace{(X'_{n,1},Y'_{n,1}),\dots,(X'_{n,m_n},Y'_{n,m_n})}_{\textnormal{pretraining set $\cD_{{\rm pre},n}$}}\, ,\ \underbrace{(X_{n,1},Y_{n,1}),\dots,(X_{n,n},Y_{n,n})}_{\textnormal{calibration set $\cD_n$}}.\]
The data points $((X'_{n,i},Y'_{n,i}))_{i\in[m_n]}$ form a pretraining set of size $m_n$, used to train a conformal score function:
\[\textnormal{$s_n:\cX\times\cY\to\R$ is constructed as a function of }\cD_{{\rm pre},n} = ((X'_{n,i},Y'_{n,i}))_{i\in[m_n]}.\]
The data points $((X_{n,i},Y_{n,i}))_{i\in[n]}$ form the calibration set, so that the quantile $\hat{q}_n$ is then computed as
\[\hat{q}_n = \quantile\left(s_n(X_{n,1},Y_{n,1}),\dots,s_n(X_{n,n},Y_{n,n}); 1-\alpha_n\right)\]
(where, in our usual definition of the split conformal method, we set $1-\alpha_n = (1-\alpha)(1+1/n)$).
Finally, at any test feature value $x$, the split conformal prediction set is given by
\[\cC_n(x) = \left\{y\in\cY : s_n(x,y) \leq \hat{q}_n\right\}.\]

In this chapter, we have focused on the i.i.d.\ data setting. Under our new notation, this means that we assume
\begin{equation}\label{eqn:iid_asymp}(X'_{n,1},Y'_{n,1}),\dots,(X'_{n,m_n},Y'_{n,m_n}),(X_{n,1},Y_{n,1}),\dots,(X_{n,n},Y_{n,n})\iidsim P,\end{equation}
for each $n\geq1$, where $P$ is some fixed distribution on $\cX\times\cY$. 
In particular, this implies that the calibration data points $(X_{n,i},Y_{n,i})$ are independent of the pretrained score function $s_n$, since $s_n$ is fitted on the pretraining set $\cD_{{\rm pre},n}$ (which is disjoint from the calibration set).
Note that this condition places no assumptions on the dependence structure of the \emph{sequence} of pretraining and calibration sets; for example, $\cD_n$ and $\cD_{n+1}$ may share all data points except the $(n+1)$st, or they may be two independent samples.

\subsection{Convergence of the conformal quantile}
As a preliminary result, we will study the convergence properties of the conformal quantiles $\hat{q}_n$ that define the thresholds for the split conformal prediction sets.
For this part of the analysis, we will need to assume that the sequence of score functions $s_n$ converges to the oracle score $s^*$, in a particular sense.
To be precise, we will write
$s_n\stackrel{\rm CDF}{\longrightarrow} s^*$ to mean that
\[F_{P,s_n}(t)\to F_{P,s^*}(t),\textnormal{ for any $t\in\R$ such that $F_{P,s^*}$ is continuous at $t$},\]
where, for any $s:\cX\times\cY\to\R$, we define $F_{P,s}$ to denote the CDF of $s(X,Y)$ under $(X,Y)\sim P$. 

Even though $s_n$ is a function of the randomly drawn pretraining set $\cD_{\rm pre ,n}$, the CDF $F_{P,s_n}$ is treating $s_n$ as fixed---that is, $F_{P,s_n}$ is the CDF of the score $s_n(X,Y)$ when we treat $s_n$ as fixed and only consider randomness coming from drawing an independent new data point $(X,Y)\sim P$. In other words, $s_n \stackrel{\rm CDF}{\longrightarrow} s^*$ is a (random) event---it is the event that $s_n(X,Y)\to s^*(X,Y)$ in distribution, when $(X,Y)\sim P$ is treated as random while the score functions $s_1,s_2,\dots$ are treated as fixed.

\begin{theorem}[Asymptotic convergence of the conformal quantile]\label{thm:splitCP_asymp_formal_qn}
For each $n\geq1$, assume the data follows assumption~\eqref{eqn:iid_asymp} for some distribution $P$ on $\cX\times\cY$, and define the split conformal prediction set $\cC_n(x) = \{y\in\cY: s_n(x,y) \leq \hat{q}_n\}$ where 
\[s_n(x,y) = s\left((x,y); \cD_{{\rm pre},n}\right)\]
for $\cD_{{\rm pre},n} = ((X'_{n,i},Y'_{n,i}))_{i\in[m_n]}$, and \[\hat{q}_n = \quantile\left(s_n(X_{n,1},Y_{n,1}),\dots,s_n(X_{n,n},Y_{n,n}); 1-\alpha_n\right),\]
for a sequence $\alpha_n\to\alpha\in(0,1)$.
Let $s^*:\cX\times\cY\to\R$ be any fixed score function, and define
\[q^* = \inf\{t: F_{P,s^*}(t)\geq 1-\alpha\}\textnormal{ \ and \ }
q^*_+ = \sup\{t: F_{P,s^*}(t)\leq 1-\alpha\}.\]
Then the following statement holds almost surely:
\[\textnormal{If $s_n\stackrel{\rm CDF}{\longrightarrow} s^*$, \ then \   $q^* \leq \liminf_{n\to\infty}\hat{q}_n \leq  \limsup_{n\to\infty}\hat{q}_n 
 \leq q^*_+$}.\]
\end{theorem}
To interpret the result of this theorem, we need to understand the quantities $q^*$ and $q^*_+$. Consider the condition
\[\P_{(X,Y)\sim P}(s^*(X,Y) < q) \leq 1-\alpha \leq \P_{(X,Y)\sim P}(s^*(X,Y) \leq q).\]

\begin{figure}[t]
    \centering
    \includegraphics[width=0.7\textwidth]{\diagramspath qstar-qstarplus.pdf}
    \caption{\textbf{An illustration of $q^*$ and $q^*_+$.} This figure illustrates the values $q^*$ and $q^*_+$ defined in the statement of Theorem~\ref{thm:splitCP_asymp_formal_qn}. The two left panels show examples where $q^*=q^*_+$, in settings where $s^*(X,Y)$ has a continuous distribution (top-left) or a discrete distribution (bottom-left). The two right panels show examples where $q^* < q^*_+$, in settings where $s^*(X,Y)$ has a continuous distribution (top-right) or a discrete distribution (bottom-right).}
    \commentAlt{Four plots show different CDFs, each with horizontal axis labeled $q$ and vertical axis labeled $F_{P,s^*}(q)$, with level $1-\alpha$ marked on the vertical axis, and values $q^*$ and $q^*_+$ marked on the horizontal axis. See long description.}
    \commentLongAlt{Four plots show different CDFs, each with horizontal axis labeled $q$ and vertical axis labeled $F_{P,s^*}(q)$, with level $1-\alpha$ marked on the vertical axis, and values $q^*$ and $q^*_+$ marked on the horizontal axis, with dashed lines connecting these values to the corresponding points on the function. The top two plots show continuous CDFs and the bottom two plots show discrete CDFs. In some cases $q^*$ and $q^*_+$ have the same value (the two plots on the left), and in other cases these values are not equal, with $q^*$ strictly smaller than $q^*_+$ (the two plots on the right).}
    \label{fig:qstar_qstarplus}
\end{figure}

Any $q$ that satisfies this condition could be called the $(1-\alpha)$-quantile (for the distribution of $s^*(X,Y)$, under $(X,Y)\sim P$). In this book, we have taken the convention that if the solution $q$ to this condition is not unique, then the $(1-\alpha)$-quantile of the distribution is taken to be the \emph{infimum} of all such solutions $q$ (see Section~\ref{sec:distributional-properties-quantiles-cdfs})---this is the value $q^*$ in the theorem above. In contrast, $q^*_+$ is, by definition, the \emph{supremum} of all such solutions $q$. (See Figure~\ref{fig:qstar_qstarplus} for an illustration.)
In particular, if the solution $q$ is unique, then we will have $q^*=q^*_+$, and in that case, we have a stronger result:
\begin{corollary}\label{cor:splitCP_asymp_formal_qn}
Under the notation and assumptions of Theorem~\ref{thm:splitCP_asymp_formal_qn}, assume also that $q^*=q^*_+$. 
Then the following statement holds almost surely:
\[\textnormal{If $s_n\stackrel{\rm CDF}{\longrightarrow} s^*$ then $\hat{q}_n\to q^*$.}\]
\end{corollary}

We now present the proofs of these results. Since Corollary~\ref{cor:splitCP_asymp_formal_qn} follows immediately from Theorem~\ref{thm:splitCP_asymp_formal_qn}, we only need to prove the theorem.
\begin{proof}[Proof of Theorem~\ref{thm:splitCP_asymp_formal_qn}]
For each $n\geq 1$, define the empirical CDF of the calibration scores,
\[\hat{F}_n(t) = \frac{1}{n}\sum_{i=1}^n \ind{s_n(X_{n,i},Y_{n,i})\leq t}, \ t\in\R. \]
We can observe that, by definition, at each $n\geq1$ the value $\hat{q}_n$ is defined as the $(1-\alpha_n)$-quantile of this empirical CDF $\hat{F}_n$.

\textbf{Step 1: a deterministic result for the empirical CDFs.}
First we prove a deterministic result relating properties of the empirical CDFs, $\hat{F}_n$, to limits of the quantiles, $\hat{q}_n$.
We claim that, for any $q\in\R$,
\begin{equation}\label{eqn:Fn_qn_1_for_thm:splitCP_asymp_formal_qn}\textnormal{ If $\limsup_{n\to \infty} \hat{F}_n(q) < 1-\alpha$ then $\liminf_{n\to\infty} \hat{q}_n \geq q$},\end{equation}
and
\begin{equation}\label{eqn:Fn_qn_2_for_thm:splitCP_asymp_formal_qn}\textnormal{ If $\liminf_{n\to \infty} \hat{F}_n(q) > 1-\alpha$ then $\limsup_{n\to\infty} \hat{q}_n \leq q$}.\end{equation}

To prove these claims, first fix any $q\in\R$ with $\limsup_{n\to \infty} \hat{F}_n(q) < 1-\alpha$.
Since $\alpha_n\to \alpha$, this means that for some finite $N\geq 1$,
it holds for all $n\geq N$ that $\hat{F}_n(q) < 1-\alpha_n$. By definition of the quantile, then, $\hat{q}_n = \quantile(\hat{F}_n;1-\alpha_n) > q$, for any $n\geq N$. This verifies~\eqref{eqn:Fn_qn_1_for_thm:splitCP_asymp_formal_qn}.
Next, fix any $q\in\R$ with $\liminf_{n\to \infty} \hat{F}_n(q)>1-\alpha$.
Since $\alpha_n\to \alpha$, this means that for some finite $N\geq 1$, it holds for all $n\geq N$ that $\hat{F}_n(q)\geq 1-\alpha_n$. By definition of the quantile, then, $\hat{q}_n = \quantile(\hat{F}_n;1-\alpha_n) \leq q$, for any $n\geq N$, which completes the proof of~\eqref{eqn:Fn_qn_2_for_thm:splitCP_asymp_formal_qn}.

\textbf{Step 2: refining the deterministic result.}
Now we will refine the deterministic results of Step 1, to make use of the assumption $s_n\stackrel{\rm CDF}{\longrightarrow} s^*$.
We claim that
\begin{equation}\label{eqn:sn_qn_1_for_thm:splitCP_asymp_formal_qn}\textnormal{ If $s_n\stackrel{\rm CDF}{\longrightarrow} s^*$ and $\|\hat{F}_n - F_{P,s_n}\|_\infty\to 0$ then $\liminf_{n\to\infty} \hat{q}_n \geq q^*$},\end{equation}
and, 
\begin{equation}\label{eqn:sn_qn_2_for_thm:splitCP_asymp_formal_qn}\textnormal{ If  $s_n\stackrel{\rm CDF}{\longrightarrow} s^*$ and $\|\hat{F}_n - F_{P,s_n}\|_\infty\to 0$ then $\limsup_{n\to\infty} \hat{q}_n \leq q^*_+$}.\end{equation}
First, fix any $q<q^*$. 
Since $F_{P,s^*}$ is a CDF and therefore has at most countably many discontinuities, we can find some $q'\in(q,q^*)$ such that $F_{P,s^*}$ is continuous at $q'$. And, by definition of $q^*$ we must have $F_{P,s^*}(q')<1-\alpha$. We then have $F_{P,s_n}(q')\to F_{P,s^*}(q')$ since $s_n\stackrel{\rm CDF}{\longrightarrow} s^*$. Since $\|\hat{F}_n-F_{P,s_n}\|_\infty\to 0$, this means $\hat{F}_n(q')\to F_{P,s^*}(q') <1-\alpha$. Therefore, 
by~\eqref{eqn:Fn_qn_1_for_thm:splitCP_asymp_formal_qn},  $\liminf_{n\to\infty}\hat{q}_n \geq q'>q$. Since this holds for any $q<q^*$, this completes the proof of~\eqref{eqn:sn_qn_1_for_thm:splitCP_asymp_formal_qn}. Finally, proving~\eqref{eqn:sn_qn_2_for_thm:splitCP_asymp_formal_qn} using~\eqref{eqn:Fn_qn_2_for_thm:splitCP_asymp_formal_qn} follows a similar argument.

\textbf{Step 3: almost sure convergence.}
Finally, we can see that the claim of the theorem follows immediately from~\eqref{eqn:sn_qn_1_for_thm:splitCP_asymp_formal_qn} (for the lower bound) and~\eqref{eqn:sn_qn_2_for_thm:splitCP_asymp_formal_qn} (for the upper bound), as long as we show that $\|\hat{F}_n-F_{P,s_n}\|_\infty\asto 0$.
To verify this, we apply the Dvoretzky--Kiefer--Wolfowitz inequality, which tells us that for each $n$ and for any $\epsilon>0$,
\[\P\left(\|\hat{F}_n-F_{P,s_n}\|_\infty\geq \epsilon\right)\leq 2e^{-2n\epsilon^2},\]
since, after conditioning on $\cD_{{\rm pre},n}$, $\hat{F}_n$ is the empirical CDF of $n$ i.i.d.\ draws from the distribution with CDF $F_{P,s_n}$.
Therefore, for any fixed $\epsilon>0$ and $N\geq 1$, we have
\begin{multline*}\P\left(\limsup_{n\to\infty}\|\hat{F}_n-F_{P,s_n}\|_\infty>\epsilon\right)\leq 
\P\left(\sup_{n\geq N}\|\hat{F}_n-F_{P,s_n}\|_\infty> \epsilon\right) \\\leq \sum_{n\geq N} 2e^{-2n\epsilon^2} = \frac{2e^{-2N\epsilon^2}}{1-e^{-2\epsilon^2}} .\end{multline*}
For any fixed $\epsilon>0$, since this holds for all $N\geq 1$, we therefore have $\limsup_{n\to\infty}\|\hat{F}_n-F_{P,s_n}\|_\infty\asleq \epsilon$. Finally,
since $\epsilon>0$ can be taken to be arbitrarily small, this proves the claim.
\end{proof}

\subsection{Convergence of the sets}

In the results above, we have studied convergence of the conformal quantile, $\hat{q}_n$. In this section, we will now translate these results into a statement about convergence of the split conformal prediction sets: we will aim to bound the difference between $\cC_n$, the split conformal prediction set, and the `oracle' prediction set $\cC^*$ given by
\[\cC^*(x) = \{y\in\cY : s^*(x,y)\leq q^*\},\]
where $s^*$ is some fixed score function (e.g., reflecting the true model of the data), and where $q^*$ is the $(1-\alpha)$-quantile of $s^*(X,Y)$ under the data distribution $(X,Y)\sim P$. At a high level, we are going to verify that
\begin{equation}\label{eqn:convergence_of_sets_intuition}\textnormal{If $s_n$ converges to $s^*$, and $s^*$ satisfies some regularity conditions, then $\cC_n$ converges to $\cC^*$,}\end{equation}
where $\cC_n(x) = \{y\in\cY:s_n(x,y)\leq \hat{q}_n\}$ is the split conformal prediction set.

The remainder of this section will develop a formal version of the statement~\eqref{eqn:convergence_of_sets_intuition}.
In order to be able to ask whether the prediction set construction $\cC_n$ converges to the oracle $\cC^*$, we will study 
the symmetric set difference between the sets, $\cC_n(X)\triangle\cC^*(X)$---specifically, we will establish a bound on $\P_{(X,Y)\sim P}(Y\in \cC_n(X)\triangle\cC^*(X))$, the mass of the symmetric set difference under a draw of a new data point $(X,Y)\sim P$.
To be able to achieve this type of bound,
we will need to define a stronger notion of convergence for the score functions $s_n$. While in the previous section, it was sufficient to assume $s_n\stackrel{\rm CDF}{\longrightarrow} s^*$ in order to establish results on the $\hat{q}_n$'s, this will no longer be sufficient: when we draw a test point $(X,Y)\sim P$, we need the evaluated scores $s_n(X,Y)$ and $s^*(X,Y)$ to return similar values, not just to be similar in distribution. 
To make this concrete, we will define a stronger notion of convergence:
we will write $s_n\stackrel{P}{\to}s^*$ to mean that
\[\lim_{n\to\infty}\P_{(X,Y)\sim P}\Big(|s_n(X,Y)-s^*(X,Y)|>\epsilon\Big) = 0 \textnormal{ for all $\epsilon>0$}.\]
Note that, treating the functions $s_n$ as fixed (as before), this simply means that the random variables $s_n(X,Y)$ converge in probability to $s^*(X,Y)$.
To compare the two notions of convergence, we have
\begin{equation}\label{eqn:compare_convergence_sn}s_n\stackrel{P}{\to}s^* \ \Longrightarrow \ s_n\stackrel{\rm CDF}{\longrightarrow} s^*.\end{equation}
(For fixed functions $s_n$, this is true simply because convergence in probability is strictly stronger than convergence in distribution; for random functions $s_n$, this holds since we define $s_n\stackrel{P}{\to}s^*$ and $s_n\stackrel{\rm CDF}{\longrightarrow} s^*$ by treating the $s_n$'s as fixed.)

We will now present our general theorem.
\begin{theorem}[Asymptotic convergence of the prediction set]\label{thm:splitCP_asymp_formal_random_x}
    Under the setting and notation of Theorem~\ref{thm:splitCP_asymp_formal_qn}, assume also that 
    \begin{equation}\label{eqn:asm_thm:splitCP_asymp_formal_random_x}\P_{(X,Y)\sim P}(s^*(X,Y)= q^*) = \P_{(X,Y)\sim P}(s^*(X,Y)= q^*_+)  = 0.\end{equation}
Then the following statement holds almost surely:
    \[\textnormal{If $s_n\stackrel{P}{\to} s^*$ \ then \    $\P_{(X,Y)\sim P}\left(Y\in\cC_n(X)\triangle\cC^*(X)\right)\to 0$.}\]
\end{theorem}
We remark also that the assumption~\eqref{eqn:asm_thm:splitCP_asymp_formal_random_x} will immediately hold if we assume that $s^*(X,Y)$ has a continuous distribution under $(X,Y)\sim P$, but (as we will see below in the proofs for one of our case studies) it can hold more generally as well.

Returning to our original goal~\eqref{eqn:convergence_of_sets_intuition} for this section, we can interpret this theorem as follows:
\[\textnormal{If }\underbrace{\textnormal{$s_n$ converges to $s^*$,}}_{\textnormal{i.e., $s_n\stackrel{P}{\to}s^*$}}\textnormal{ \ and \ }\underbrace{\textnormal{$s^*$ satisfies some regularity conditions,}}_{\textnormal{i.e., condition~\eqref{eqn:asm_thm:splitCP_asymp_formal_random_x} holds}}\textnormal{ \ then  \ }\underbrace{\textnormal{$\cC_n$ converges to $\cC^*$.}}_{\textnormal{i.e., bound $ \cC_n(X)\triangle\cC^*(X)$}}\]
Of course, while this result measures the difference between $\cC_n$ and $\cC^*$  in terms of probability (by bounding the mass placed by $P$ on the event $Y\in \cC_n(X)\triangle\cC^*(X)$), similar techniques can be used to bound the difference of the sets in other ways, as we will see in some of our case studies---for instance, in the setting of a real-valued response ($\cY=\R$), we might instead be interested in bounding the Lebesgue measure of $\cC_n(X)\triangle\cC^*(X)$.

\begin{figure}[t]
    \centering
    \includegraphics[width=0.6\linewidth]{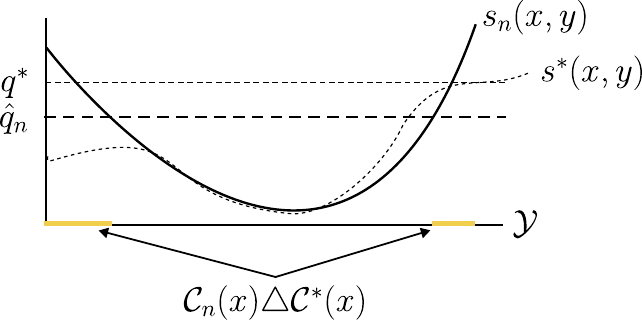}
    \caption{\textbf{An illustration of the set difference $\cC_n(x)\triangle\cC^*(x)$, at a fixed value of $x$.} These two sets are given by $\cC_n(x) = \{y: s_n(x,y)\leq \hat{q}_n\}$ and $\cC^*(x) = \{y:s^*(x,y)\leq q^*\}$; the curves corresponding to the functions $y \mapsto s_n(x,y)$ and $y \mapsto s^*(x,y)$, and the thresholds $\hat{q}_n$ and $q^*$, are labeled on the plot. The symmetric set difference $\cC_n(x)\triangle\cC^*(x)$ is highlighted on the horizontal axis, and illustrates the two cases derived in~\eqref{eqn:symmetric_set_difference_2_cases}: the portion of the highlighted region that is on the right corresponds to $y$ values where $s^*(x,y)$ is close to its threshold $q^*$, while the portion on the left corresponds to $y$ values where $|s_n(x,y)-s^*(x,y)|$ is large.}
    \commentAlt{A solid curve shows $s_n(x,y)$ and a dotted curve shows $s^*(x,y)$. The curves are compared against levels $\hat{q}_n$ and $q^*$. The symmetric set difference $\cC_n(x)\triangle\cC^*(x)$ of the two level sets is highlighted.}
    \label{fig:convergence-score-quantile-set}
\end{figure}

\begin{proof}[Proof of Theorem~\ref{thm:splitCP_asymp_formal_random_x}]
    By Theorem~\ref{thm:splitCP_asymp_formal_qn}, it holds almost surely that $s_n\stackrel{\rm CDF}{\longrightarrow} s^*$ implies 
$q^* \leq \liminf_{n\to\infty}\hat{q}_n \leq  \limsup_{n\to\infty}\hat{q}_n 
 \leq q^*_+$.
Therefore, combining this with~\eqref{eqn:compare_convergence_sn}, it suffices to prove the following \emph{deterministic} statement:
\begin{multline*}\textnormal{If $s_n\stackrel{P}{\to}s^*$ \ and \   $q^* \leq \liminf_{n\to\infty}\hat{q}_n \leq  \limsup_{n\to\infty}\hat{q}_n 
 \leq q^*_+$,}\\\textnormal{ then \ $\P_{(X,Y)\sim P}\left(Y\in\cC_n(X)\triangle\cC^*(X))\right)\to 0$.}\end{multline*}
From this point on, we assume that the events $s_n\stackrel{P}{\to}s^*$ and $q^* \leq \liminf_{n\to\infty}\hat{q}_n \leq  \limsup_{n\to\infty}\hat{q}_n 
 \leq q^*_+$ both occur.

First, fix any $(x,y)$, and any $\epsilon>0$. For sufficiently large $n$, it holds that $q^*-\epsilon < \hat{q}_n < q^*_+ + \epsilon$. Then if $y\in\cC_n(x)\triangle \cC^*(x)$, we either have
\begin{align*}
    y \in \cC_n(x) \backslash \cC^*(x) \ &\Longrightarrow \ s_n(x,y)\leq \hat{q}_n< q^*_+ + \epsilon\textnormal{ and } s^*(x,y)>q^* \\ &\Longrightarrow \ \textnormal{ $q^*<s^*(x,y)< q^*_++2\epsilon$ or $|s_n(x,y) - s^*(x,y)|>\epsilon$},
\end{align*}
or
\begin{align*}
    y \in \cC^*(x) \backslash \cC_n(x) \ &\Longrightarrow \ s_n(x,y)>\hat{q}_n> q^*-\epsilon\textnormal{ and } s^*(x,y)\leq q^* \\ &\Longrightarrow \ \textnormal{ $q^*-2\epsilon<s^*(x,y)\leq q^*$ or $|s_n(x,y) - s^*(x,y)|>\epsilon$}.
\end{align*}
Combining these two cases, we then have
\begin{equation}\label{eqn:symmetric_set_difference_2_cases} y\in \cC_n(x)\triangle \cC^*(x) \ \Longrightarrow  q^*-2\epsilon <  s^*(x,y)< q^*_+ +2\epsilon \textnormal{ \ or \  }|s_n(x,y) - s^*(x,y)|>\epsilon.\end{equation}
(This step is illustrated in Figure~\ref{fig:convergence-score-quantile-set}.)
Therefore,
\begin{multline*}\P_{(X,Y)\sim P}(Y\in \cC_n(X)\triangle \cC^*(X)) 
\leq \P_{(X,Y)\sim P}(q^*-2\epsilon <  s^*(X,Y)< q^*_+ +2\epsilon) +{}\\\P_{(X,Y)\sim P}(|s_n(X,Y) - s^*(X,Y)|>\epsilon).\end{multline*}
Since $\P_{(X,Y)\sim P}(|s_n(X,Y) - s^*(X,Y)|>\epsilon)\to 0$ as $n\to 0$ (because we have assumed $s_n\stackrel{P}{\to}s^*$), we therefore have
\[\limsup_{n\to\infty} \P_{(X,Y)\sim P}(Y\in \cC_n(X)\triangle \cC^*(X)) 
\leq \P_{(X,Y)\sim P}(q^*-2\epsilon <  s^*(X,Y)< q^*_+ +2\epsilon).\]
Since this holds for any fixed $\epsilon>0$, we therefore have
\begin{multline*}\limsup_{n\to\infty} \P_{(X,Y)\sim P}(Y\in \cC_n(X)\triangle \cC^*(X)) 
\leq \lim_{\epsilon\to0} \P_{(X,Y)\sim P}(q^*-2\epsilon <  s^*(X,Y)< q^*_+ +2\epsilon)\\=\P_{(X,Y)\sim P}(q^*\leq   s^*(X,Y)\leq q^*_+) ,\end{multline*}
by continuity of measure. Moreover, $ \P_{(X,Y)\sim P}(s^*(X,Y)\in\{q^*,q^*_+\}) = 0$ by our assumption~\eqref{eqn:asm_thm:splitCP_asymp_formal_random_x},  
and we also have
\[\P_{(X,Y)\sim P}(q^*<  s^*(X,Y)< q^*_+) =0\]
by definition of $q^*$ and $q^*_+$, which completes the proof.
\end{proof}

\subsection{Proofs for case studies}
With our general theory in place, we are now ready to present the proofs of the asymptotic optimality results for our four case studies.
All four results will be derived as applications of Theorem~\ref{thm:splitCP_asymp_formal_random_x}, and will follow the same general recipe. For \textbf{Step 1}, we will need to verify that the oracle score $s^*$ satisfies the condition~\eqref{eqn:asm_thm:splitCP_asymp_formal_random_x} required by the theorem, and for \textbf{Step 2} we need to check that the model assumptions (on how accurately we estimate the true model) are  sufficient to ensure that $s_n\stackrel{P}{\to}s^*$. 
Finally, in \textbf{Step 3}, we will show that asymptotic optimality of the split conformal set
follows from the guarantee of Theorem~\ref{thm:splitCP_asymp_formal_random_x}.

\paragraph{Proofs for the classification setting.} First we consider the two case studies for classification, developed in Section~\ref{sec:case_studies_classification}. 

\begin{proof}[Proof of Proposition~\ref{prop:case_study_1_classification}]

\textbf{Step 1: verifying condition~\eqref{eqn:asm_thm:splitCP_asymp_formal_random_x}.}
Since $\pi^*(Y\mid X)$ is assumed to have a continuous distribution under $(X,Y)\sim P$, the score $s^*(X,Y) = -\pi^*(Y\mid X)$ is therefore also continuously distributed, which immediately implies~\eqref{eqn:asm_thm:splitCP_asymp_formal_random_x}. 

\textbf{Step 2: verifying that $s_n\stackrel{P}{\to} s^*$.} 
We calculate
    \begin{align*}
        \E_{(X,Y)\sim P}[|s_n(X,Y)-s^*(X,Y)|]
        &=\E_{(X,Y)\sim P}[|\hat\pi_n(Y\mid X)-\pi^*(Y\mid X)|]\\
        &=\E_{X\sim P_X}\left[\sum_{y\in\cY}\pi^*(y\mid X) \cdot |\hat\pi_n(y\mid X)-\pi^*(y\mid X)| \right]\\
        &\leq \sup_{(x,y)}\pi^*(y\mid x)\cdot  \E_{X\sim P_X}\left[\sum_{y\in\cY}|\hat\pi_n(y\mid X)-\pi^*(y\mid X)| \right]\\
        &= \sup_{(x,y)}\pi^*(y\mid x)\cdot  \E_{X\sim P_X}\left[ 2\dtv(\hat\pi_n(\cdot\mid X), \pi^*(\cdot\mid X))\right],
    \end{align*}
   by definition of total variation distance. Note that $\sup_{(x,y)}\pi^*(y\mid x)\leq 1$, since $\pi^*$ is a conditional probability. Therefore, if we assume $\E_{X\sim P_X}\left[\dtv\big(\pi^*(\cdot\mid X), \hat\pi_n(\cdot \mid X)\big)\right]\to 0$ as in the proposition, this implies $\E_{(X,Y)\sim P}[|s_n(X,Y)-s^*(X,Y)|]\to 0$, which in turn implies $s_n\stackrel{P}{\to} s^*$.

\textbf{Step 3: establishing asymptotic optimality.} Applying Theorems~\ref{thm:splitCP_asymp_formal_qn} and~\ref{thm:splitCP_asymp_formal_random_x} (along with the results of Steps 1 and 2), we see that $\E_{X\sim P_X}[\dtv(\hat\pi_n(\cdot\mid X), \pi^*(\cdot\mid X))]\to 0$ implies $\P_{(X,Y)\sim P}(Y\in \cC_n(X)\triangle \cC^*(X))\to 0$ and $q^*\leq\liminf_{n\to\infty}\hat{q}_n\leq\limsup_{n\to\infty}\hat{q}_n\leq q^*_+$, almost surely. From this point on, we will assume that this event holds.
To complete the proof, we now need to show that this implies $\cC_n$ is asymptotically optimal. 

For asymptotic optimality of the set size, it suffices to show $
\E_{X\sim P_X}[|\cC_n(X)\backslash \cC^*(X)|] \to 0$. Define 
\begin{equation}\label{eqn:case_study_classification_1_define_c_n}c_n = \inf_{(x,y):y\in\cC_n(x)} \hat\pi_n(y\mid x).\end{equation}
We calculate
\begin{align}
\notag    &\E_{X\sim P_X}[|\cC_n(X)\backslash \cC^*|]
    =\E_{X\sim P_X}\left[\sum_{y\in\cY}\ind{y\in\cC_n(X)\backslash \cC^*(X)}\right]\\
\notag    &\leq c_n^{-1}\cdot \E_{X\sim P_X}\left[ \sum_{y\in\cY} \hat\pi_n(y\mid X)\cdot \ind{y\in\cC_n(X)\backslash \cC^*(X)}\right] \\
\notag    &\leq c_n^{-1}\cdot \E_{X\sim P_X}\left[ \dtv(\hat\pi_n(\cdot\mid X), \pi^*(\cdot\mid X)) + \sum_{y\in\cY} \pi^*(y\mid X)\cdot \ind{y\in\cC_n(X)\backslash \cC^*(X)}\right]\\
    \label{eqn:case_study_classification_1_E_set_size}&=c_n^{-1}\left(\E_{X\sim P_X}\left[ \dtv(\hat\pi_n(\cdot\mid X), \pi^*(\cdot\mid X))\right] +\P_{(X,Y)\sim P}(Y\in\cC_n(X)\backslash \cC^*(X))\right).
\end{align}
Since both the expected value and the probability on the right-hand side are vanishing as $n\to \infty$, from this point on we only need to verify that $\liminf_{n\to\infty} c_n >0$.
By definition of the score, we have  $\P_{(X,Y)\sim P}(s^*(X,Y)<0) = 1$. In particular, this implies that we must have $q^*\leq q^*_+< 0$, and since $\limsup_{n\to\infty} \hat{q}_n \leq q^*_+$, for sufficiently large $n$ we have $\hat{q}_n \leq q^*_+/2<0$. Therefore, for sufficiently large $n$,
\[y\in\cC_n(x) \Longleftrightarrow - \hat{\pi}_n(y\mid x) = s_n(x,y) \leq \hat{q}_n \Longrightarrow \hat\pi_n(y\mid x) \geq |q^*_+|/2,\]
and therefore,
\[\liminf_{n\to\infty}c_n \geq |q^*_+|/2>0.\]

Next,
    for the asymptotic coverage guarantee, since $\P_{(X,Y)\sim P}(Y\in\cC^*(X))\geq 1-\alpha$ by definition of the oracle $\cC^*$,
\begin{align*}
    \P_{(X,Y)\sim P}(Y\in\cC_n(X))
    &\geq 1-\alpha - \P_{(X,Y)\sim P}(Y\in\cC^*(X)\backslash \cC_n(X))\\
    &\geq 1-\alpha - \P_{(X,Y)\sim P}(Y\in\cC^*(X)\triangle \cC_n(X)),
\end{align*}
and therefore, $\liminf_{n\to\infty} \P_{(X,Y)\sim P}(Y\in\cC_n(X))\geq 1-\alpha$ almost surely.
This verifies asymptotic optimality of $\cC_n$, and thus completes the proof.
\end{proof}

\begin{proof}[Proof of Proposition~\ref{prop:case_study_2_classification}]

\textbf{Step 1: verifying condition~\eqref{eqn:asm_thm:splitCP_asymp_formal_random_x}.}
Since $\P_{(X,Y)\sim P}(s^*(X,Y) = 0) < 1-\alpha$ (by our assumption~\eqref{eqn:continuity_for_case_study_2_classification}), we must have $q^*_+ \geq q^* >0$ by definition of the quantile. Therefore (again applying~\eqref{eqn:continuity_for_case_study_2_classification}) we have $\P_{(X,Y)\sim P}(s^*(X,Y) = q^*)=\P_{(X,Y)\sim P}(s^*(X,Y) = q^*_+) = 0$.

\textbf{Step 2: verifying that $s_n\stackrel{P}{\to} s^*$.}
First, for each $x\in\cX$ define
\[\Delta(x) = \min_{y\neq y'\in\cY}|\pi^*(y\mid x) - \pi^*(y'\mid x)|.\]
By assumption, $\Delta(X)>0$ almost surely for $X\sim P_X$. Now fix some $x\in\cX$, and suppose that $\dtv(\pi^*(\cdot\mid x),\hat\pi_n(\cdot\mid x)) < \Delta(x)/2$. If this holds, then for any $y\neq y'\in\cY$,
\begin{multline*}\left|\left(\hat\pi_n(y\mid x) - \hat\pi_n(y'\mid x)\right) - \left(\pi^*(y\mid x) - \pi^*(y'\mid x)\right)\right| \\
\leq |\hat\pi_n(y\mid x) - \pi^*(y\mid x)| + |\hat\pi_n(y'\mid x) - \pi^*(y'\mid x)| \leq 2\dtv(\pi^*(\cdot\mid x),\hat\pi_n(\cdot\mid x))\\< 2\cdot \Delta(x)/2 \leq |\pi^*(y\mid x) - \pi^*(y'\mid x)|.
\end{multline*}
In particular, this implies
\[\ind{\hat\pi_n(y'\mid x)>\hat\pi_n(y\mid x)} = \ind{\pi^*(y'\mid x)>\pi^*(y\mid x)}\]
for all $y\neq y'\in\cY$, and therefore for any $y\in\cY$ we have
\begin{align*}
    &\left|s_n(x,y) - s^*(x,y)\right|\\
    &=\left|\sum_{y'\in\cY} \hat\pi_n(y'\mid x) \cdot\ind{\hat\pi_n(y'\mid x)  > \hat\pi_n(y\mid x)} - \sum_{y'\in\cY} \pi^*(y'\mid x) \cdot\ind{\pi^*(y'\mid x)  > \pi^*(y\mid x)}\right|\\
    &=\left|\sum_{y'\in\cY} \big(\hat\pi_n(y'\mid x) - \pi^*(y'\mid x)\big) \cdot\ind{\pi^*(y'\mid x)  > \pi^*(y\mid x)}\right|\\
    &\leq \dtv(\pi^*(\cdot\mid x),\hat\pi_n(\cdot\mid x)).
\end{align*}
In other words, we have proved that for all $(x,y)$, we have
\[
\left|s_n(x,y) - s^*(x,y)\right| \leq \begin{cases} \dtv(\pi^*(\cdot\mid x),\hat\pi_n(\cdot\mid x)), & \textnormal{ if }\dtv(\pi^*(\cdot\mid x),\hat\pi_n(\cdot\mid x))<\Delta(x)/2,\\ 1, & \textnormal{ otherwise}.\end{cases} \]
Therefore, for any $\epsilon>0$, 
\begin{multline*}\P_{(X,Y)\sim P}(|s_n(X,Y)-s^*(X,Y)|>\epsilon)\\\leq \P_{X\sim P_X}\Big(\dtv(\pi^*(\cdot\mid X),\hat\pi_n(\cdot\mid X)) \geq \min\{\epsilon,\Delta(X)/2\}\Big). \end{multline*}
Since $ \min\{\epsilon,\Delta(X)/2\}$ is positive almost surely, it must hold that if $\E_{X\sim P_X}[\dtv(\pi^*(\cdot\mid X),\hat\pi_n(\cdot\mid X))] \to 0$ then
\[\lim_{n\to\infty}\P_{X\sim P_X}\big(\dtv(\pi^*(\cdot\mid X),\hat\pi_n(\cdot\mid X)) \geq  \min\{\epsilon,\Delta(X)/2\} \big) = 0.\]
Therefore we have proved the desired claim.

\textbf{Step 3: establishing asymptotic optimality.}
Applying Theorems~\ref{thm:splitCP_asymp_formal_qn} and~\ref{thm:splitCP_asymp_formal_random_x} (along with the results of Steps 1 and 2), we see that $\E_{X\sim P_X}[\dtv(\hat\pi_n(\cdot\mid X), \pi^*(\cdot\mid X))]\to 0$ implies $\P_{(X,Y)\sim P}(Y\in \cC_n(X)\triangle \cC^*(X))\to 0$ and $q^*\leq\liminf_{n\to\infty}\hat{q}_n\leq\limsup_{n\to\infty}\hat{q}_n\leq q^*_+$, almost surely.
From this point on, we will assume that this event holds, and will show that this implies $\cC_n$ is asymptotically optimal. 

For asymptotic optimality of the set size, following the same steps as in~\eqref{eqn:case_study_classification_1_E_set_size} (in the proof of Proposition~\ref{prop:case_study_1_classification}),
we only need to show that $\liminf_{n\to\infty}c_n>0$, where $c_n$ is defined as in~\eqref{eqn:case_study_classification_1_define_c_n}.
Note that $\P_{(X,Y)\sim P}(s^*(X,Y)\leq 1)=1$ by construction. By our assumption~\eqref{eqn:continuity_for_case_study_2_classification}, along with the definition of $q^*_+$, we must have $q^*_+ < 1$. Since $\limsup_{n\to\infty}\hat{q}_n\leq q^*_+$, we can assume $\hat{q}_n \leq 1-\frac{1-q^*_+}{2}$ for all sufficiently large $n$.
Next, suppose $y\in\cC_n(x)$. Then, for sufficiently large $n$,
\begin{align*}1-\frac{1-q^*_+}{2} \geq \hat{q}_n \geq s_n(x,y) &= \sum_{y'\in\cY} \hat \pi_n(y'\mid x) \cdot \ind{\hat \pi_n(y'\mid x)> \hat \pi_n(y\mid x)} \\
&=1 - \sum_{y'\in\cY} \hat \pi_n(y'\mid x) \cdot \ind{\hat \pi_n(y'\mid x)\leq \hat \pi_n(y\mid x)} \\
&\geq 1 - \sum_{y'\in\cY} \hat \pi_n(y\mid x)  = 1 - |\cY| \cdot \hat\pi_n(y\mid x) .
\end{align*}
Therefore,
\[y\in\cC_n(x) \Longrightarrow \hat\pi_n(y\mid x) \geq \frac{1-q^*_+}{2|\cY|},\]
and so $\liminf_{n\to\infty} c_n \geq (1-q^*_+)/2|\cY|>0$, as desired.

    For the asymptotic conditional coverage guarantee,
    since $\P_{Y\sim P_{Y\mid X}}(Y\in\cC^*(X)\mid X)\geq 1-\alpha'$ almost surely by definition of the oracle $\cC^*$,  we have
\begin{align*}
    \P_{Y\sim P_{Y\mid X}}(Y\in\cC_n(X) \mid X) 
    &\geq 1-\alpha' - \P_{Y\sim P_{Y\mid X}}(Y\in\cC^*(X)\backslash \cC_n(X)\mid X)\\
    &\geq 1-\alpha' - \P_{Y\sim P_{Y\mid X}}(Y\in\cC^*(X)\triangle \cC_n(X)\mid X),
\end{align*}
and so for any $\epsilon>0$,
\begin{multline*}\P_{X\sim P_X}\left( \P_{Y\sim P_{Y\mid X}}(Y\in\cC_n(X) \mid X)  < 1 - \alpha' - \epsilon \right) 
\\\leq \P_{X\sim P_X}\left(\P_{Y\sim P_{Y\mid X}}(Y\in\cC^*(X)\triangle \cC_n(X) \mid X)> \epsilon \right)  \leq \epsilon^{-1}\P_{(X,Y)\sim P}(Y\in\cC^*(X)\triangle \cC_n(X)),
\end{multline*}
by Markov's inequality and the tower law. Thus \[\P_{X\sim P_X}\left( \P_{Y\sim P_{Y\mid X}}(Y\in\cC_n(X) \mid X)  \geq 1 - \alpha' - \epsilon \right)\to 1,\] almost surely. This verifies asymptotic optimality of $\cC_n$, and thus completes the proof.
\end{proof}

\paragraph{Proofs for the regression setting.} Next we turn to the two case studies for regression, developed in Section~\ref{sec:case_studies_regression}.

\begin{proof}[Proof of Proposition~\ref{prop:case_study_1_regression}]

\textbf{Step 1: verifying condition~\eqref{eqn:asm_thm:splitCP_asymp_formal_random_x}.}
Since $f^*(Y\mid X)$ is assumed to have a continuous distribution under $(X,Y)\sim P$, the score $s^*(X,Y) = -f^*(Y\mid X)$ is therefore also continuously distributed, which immediately implies~\eqref{eqn:asm_thm:splitCP_asymp_formal_random_x}.

\textbf{Step 2: verifying that $s_n\stackrel{P}{\to} s^*$.} This step's proof follows an identical argument as for the corresponding step in the proof of Proposition~\ref{prop:case_study_1_classification}, except with $f^*$ and $\hf_n$ in place of $\pi^*$ and $\hat{\pi}_n$, and with integration over $y\in\R$ in place of summation over $y\in\cY$.
(Note that this calculation relies on our assumption that $\sup_{(x,y)}f^*(y\mid x)<\infty$---while, for the classification setting with $\pi^*$ in place of $f^*$, $\sup_{(x,y)}\pi^*(y\mid x)$ is finite simply because conditional probability is always bounded by $1$).

\textbf{Step 3: establishing asymptotic optimality.} 
As for Step 2, 
the proof again follows an identical argument as for the corresponding step in the proof of Proposition~\ref{prop:case_study_1_classification}, where now
\begin{equation}\label{eqn:case_study_regression_1_define_c_n}c_n = \inf_{(x,y):y\in\cC_n(x)} \hf_n(y\mid x).\end{equation}
\end{proof} 

\begin{proof}[Proof of Proposition~\ref{prop:case_study_2_regression}]

\textbf{Step 1: verifying condition~\eqref{eqn:asm_thm:splitCP_asymp_formal_random_x}.}
It is sufficient to prove that $s^*(X,Y)$ has a continuous distribution under $(X,Y)\sim P$. For any $t\in\R$, we have
\begin{align*}
    &\P_{Y\sim P_{Y\mid X}}(s^*(X,Y)=t\mid X)\\
    &=\P_{Y\sim P_{Y\mid X}}\left(\max\{\tau^*(X; \alpha/2) - Y, Y - \tau^*(X; 1-\alpha/2)\} = t \, \middle|\, X\right)\\
    &\leq \P_{Y\sim P_{Y\mid X}}\left(Y \in\big\{ \tau^*(X; \alpha/2) - t\, , \tau^*(X; 1- \alpha/2) + t\big\}\, \middle|\, X\right)\\
    &=0,
\end{align*}
since the conditional distribution of $Y\mid X$ is continuous. Therefore, by the tower law, 
$\P_{(X,Y)\sim P}(s^*(X,Y)=t) = 0$ for all $t\in\R$, as desired.

\textbf{Step 2: verifying that $s_n\stackrel{P}{\to} s^*$.}
By definition of $s_n$ and $s^*$, we see that
\[|s_n(x,y) - s^*(x,y)|\leq \max_{\beta\in\{\alpha/2,1-\alpha/2\}}|\hat{\tau}_n(x ; \beta) - \tau^*(x ; \beta)|\]
for all $(x,y)$.
Therefore,
\begin{multline*}\P_{(X,Y)\sim P}\left(|s_n(X,Y) - s^*(X,Y)|>\epsilon\right) \\\leq \P_{X\sim P_X}\left(\max_{\beta\in\{\alpha/2,1-\alpha/2\}}|\hat{\tau}_n(X ; \beta) - \tau^*(X ; \beta)| > \epsilon\right)\\
\leq \epsilon^{-1}\sum_{\beta\in\{\alpha/2,1-\alpha/2\}} \E_{X \sim P_X}[|\hat{\tau}_n(X ; \beta) - \tau^*(X ; \beta)|],
\end{multline*}
by Markov's inequality.
Therefore, if
    $\E_{X \sim P_X}[|\hat{\tau}_n(X ; \beta) - \tau^*(X ; \beta)|] \to 0 $  holds  for each $\beta \in\{\alpha/2,1-\alpha/2\}$ (as is assumed in the proposition), then this
    implies $s_n\stackrel{P}{\to} s^*$.

\textbf{Step 3: establishing asymptotic optimality.}
Applying Theorems~\ref{thm:splitCP_asymp_formal_qn} and~\ref{thm:splitCP_asymp_formal_random_x} (along with the results of Steps 1 and 2), we see that $\E_{X \sim P_X}[|\hat{\tau}_n(X ; \beta) - \tau^*(X ; \beta)|] \to 0 $ for each $\beta\in\{\alpha/2,1-\alpha/2\}$ implies $\P_{(X,Y)\sim P}(Y\in \cC_n(X)\triangle \cC^*(X))\to 0$ and $q^*\leq\liminf_{n\to\infty}\hat{q}_n\leq\limsup_{n\to\infty}\hat{q}_n\leq q^*_+$, almost surely. From this point on, we will assume that this event holds, and will show that this implies $\cC_n$ is asymptotically optimal. 

First, for asymptotic optimality of the set size, we will use the fact that, for two intervals $[a,b]$ and $[c,d]$ in the real line, we can calculate
\[\Big|\, [a,b]\triangle[c,d]\,\Big| \leq |a-c| + |b-d|\]
(with equality if the two intervals overlap, but a strict inequality if they are disjoint). Then, by construction of the prediction intervals $\cC_n(X)$ and $\cC^*(X)$ in this particular setting, we have
\begin{multline*}\textnormal{Leb}(\cC_n(X)\triangle \cC^*(X)) \leq\\ \left| \big(\hat{\tau}_n(X;\alpha/2) - \hat{q}_n\big) - \tau^*(X;\alpha/2)\right| + \left| \big(\hat{\tau}_n(X;1-\alpha/2) + \hat{q}_n\big) - \tau^*(X;1-\alpha/2)\right| \\
\leq 2\max_{\beta\in\{\alpha/2,1-\alpha/2\}}\big|\hat{\tau}_n(X;\beta)-\tau^*(X;\beta)\big| + 2|\hat{q}_n|. 
\end{multline*}
Next, by construction of the oracle interval we have $q^*=0$, and moreover, the assumption that $f^*(y\mid x)>0$ for all $(x,y)$ ensures that $q^*_+=0$ as well (because, with a positive density, any positive inflation of the interval $\cC^*(X)$ would lead to a strictly higher probability of coverage---that is, $\P_{(X,Y)\sim P}(s^*(X,Y) \leq t)>1-\alpha$ for any $t>0$). Thus we have $\hat{q}_n\to 0$. Therefore,
\[\E_{X\sim P_X}[\textnormal{Leb}(\cC_n(X)\triangle\cC^*(X))] \leq 2\sum_{\beta\in\{\alpha/2,1-\alpha/2\}}\E_{X \sim P_X}[|\hat{\tau}_n(X ; \beta) - \tau^*(X ; \beta)|] + 2|\hat{q}_n|\to 0,\]
which therefore implies $\limsup_{n\to \infty} \E_{X\sim P_X}[\textnormal{Leb}(\cC_n(X))]  \leq \E_{X\sim P_X}[\textnormal{Leb}(\cC^*(X))]$.

Next we turn to the equal-tailed conditional coverage guarantee.
If $\cC_n(X)$ and $\cC^*(X)$ are disjoint (which includes the degenerate case where $\cC_n(X)$ is empty), we have
\begin{multline*}\P_{Y\sim P_{Y\mid X}}(Y>\sup \cC_n(X) \mid X) \leq 1 = \frac{\P_{Y\sim P_{Y\mid X}}(Y\in \cC^*(X))}{1-\alpha} \\\leq \frac{\P_{Y\sim P_{Y\mid X}}(Y\in \cC_n(X)\triangle \cC^*(X))}{1-\alpha}.\end{multline*}
If instead the intervals $\cC_n(X)$ and $\cC^*(X)$ overlap (and in particular $\cC_n(X)$ is nonempty), then
\[ \sup\cC^*(X) \geq  Y > \sup \cC_n(X) \Longrightarrow  Y\in \cC_n(X)\triangle \cC^*(X),\]
and so
\[\P_{Y\sim P_{Y\mid X}}(Y>\sup \cC_n(X) \mid X) \leq\alpha/2 +\P_{Y\sim P_{Y\mid X}}(Y\in \cC_n(X)\triangle \cC^*(X)) ,\]
using the oracle property of $\cC^*$, which ensures that $\P_{Y\sim P_{Y\mid X}}(Y>\sup \cC^*(X) \mid X) \leq\alpha/2$.
Combining these calculations, then, 
\[ \P_{Y\sim P_{Y\mid X}}(Y>\sup \cC_n(X) \mid X) \leq\alpha/2 +\frac{1}{1-\alpha}\cdot \P_{Y\sim P_{Y\mid X}}(Y\in \cC_n(X)\triangle \cC^*(X)),\]
across both possible cases. Therefore,
\begin{multline*}\P_{X\sim P_X}\left( \P_{Y\sim P_{Y\mid X}}(Y>\sup \cC_n(X) \mid X) > \alpha/2 + \epsilon\right) \\\leq \frac{\P_{(X,Y)\sim P}(Y\in \cC_n(X)\triangle \cC^*(X))}{(1-\alpha)\epsilon} \to 0,\end{multline*}
by Markov's inequality and the tower law. This
establishes asymptotic coverage in the right tail. An identical argument verifies coverage for the left tail as well, and thus completes the proof.
\end{proof}

\section{Model-based robustness to violations of exchangeability}\label{sec:models_without_exch}
\index{nonexchangeability}
\index{time series|(}

In this final section of Chapter~\ref{chapter:model-based}, we will turn to a question of a very different flavor: instead of asking about guarantees that can be obtained by placing model-based assumptions on the data \emph{in addition to} the assumption of exchangeability, we will now ask whether model-based assumptions allow for conformal type methods to perform well \emph{even if exchangeability does not hold}.

For simplicity, let us again consider only split conformal prediction throughout this section, so that we have a pretrained score function $s_n:\cX\times\cY\to \R$. The key insight is that the marginal coverage guarantee relies only on
finding a good approximation for the $(1-\alpha)$-quantile of the test point's score, $s_n(X_{n+1},Y_{n+1})$. 
Exchangeability of the data points $(X_1,Y_1),\dots,(X_{n+1},Y_{n+1})$ is certainly sufficient to ensure this (i.e., with the split conformal quantile $\hat{q}_n$), but it is not necessary: it is possible to estimate this quantile well under other assumptions even in the absence of exchangeability.
Thus, conformal prediction offers a double-robustness type guarantee: conformal prediction offers coverage guarantees as long as \emph{either} exchangeability holds (as in the theory we have seen in Chapter~\ref{chapter:conformal-exchangeability}), or if instead we can rely on model-based assumptions.
Although this is a general point, for the remainder of the chapter we focus on making this intuition precise in a specific setting, namely, stationary time-series. 

\paragraph{Stationary time series.}
In the setting of a time series, exchangeability may fail drastically due to dependence over time. Nonetheless, we will see that, under model-based assumptions, applying the conformalized quantile regression (CQR) method (recall Section~\ref{sec:case_study_2_regression}) can still ensure coverage. In particular, if our estimated conditional quantiles $\hat\tau_n(x;\beta)$ are a good approximation to the true conditional quantiles $\tau^*(x;\beta)$, we can expect to achieve coverage even without exchangeability.

We now define our setting. Consider a time series of data points,
\[(X_1,Y_1), \, (X_2,Y_2),\, \dots\]
where the response is real-valued, $\cY=\R$. Suppose that, to predict response $Y_{n+1}$ from features $X_{n+1}$ at time $n+1$, we train a quantile regression model on past data points $(X_1,Y_1),\dots,(X_{\lfloor n/2\rfloor },Y_{\lfloor n/2\rfloor })$---that is, we construct quantile estimates $\hat\tau_n(x;\alpha/2)$ and $\hat\tau_n(x;1-\alpha/2)$. We then use the more recent data points $(X_{\lfloor n/2\rfloor +1},Y_{\lfloor n/2\rfloor +1}),\dots,(X_n,Y_n)$ as our calibration set, to define $\hat{q}_n$ and return the corresponding CQR prediction set,
\[\cC_n(X_{n+1}) = \left[\hat\tau_n(X_{n+1};\alpha/2) - \hat{q}_n, \hat\tau_n(X_{n+1};1-\alpha/2) + \hat{q}_n\right].\] 
(Note that our notation for indexing the data points has changed for this time series setting---the calibration set size is $\lceil n/2\rceil $, rather than $n$ as has been the case throughout the chapter.)

\begin{proposition}[Asymptotic guarantee for the time series setting]
    \label{prop:time-series}
    Under the setting and notation defined above, assume that $(X_1,Y_1),(X_2,Y_2),\dots$
 is a time series of identically distributed (but not necessarily independent) data points, with $(X_i,Y_i)\sim P$ for each $i$. Assume that the conditional distribution $P_{Y\mid X}$ has a positive and bounded density $f^*(y\mid x)$, and let $\tau^*(x;\beta)$ denote the $\beta$-quantile of this conditional distribution. Assume also that the time series is \emph{absolutely regular} (also called \emph{$\beta$-mixing}), meaning that
    \[\lim_{m\to\infty} \left\{\sup_{I,J,k\geq 1} \sup_{\substack{\textnormal{Disjoint }A_1,\dots,A_I\in\cE_{\leq k}\\\textnormal{Disjoint }A'_1,\dots,A'_J\in\cE_{\geq k+m}}} \sum_{i=1}^I\sum_{j=1}^J\big|\P(A_i\cap A'_j) - \P(A_i)\P(A'_j)\big|\right\}  = 0,\]
    where $\cE_{\leq k}$ is the set of all events that depend only on $((X_1,Y_1),\dots,(X_k,Y_k))$, and $\cE_{\geq k+m}$ is the set of all events that depend only on $((X_{k+m},Y_{k+m}),(X_{k+m+1},Y_{k+m+1}),\dots)$.

    If the quantile estimates $\hat\tau_n$ satisfy
     \begin{equation}\label{eqn:assume_tau_n_consistent}
        \lim_{n \to \infty} \E\left[ \int_{\cX}| \hat\tau_n(x;\beta) - \tau^*(x;\beta)|\;\mathsf{d}P_X(x)\right] = 0
    \end{equation}
    for each $\beta\in\{\alpha/2,1-\alpha/2\}$,
    then
\begin{equation}
        \lim_{n \to \infty} \E\left[ \big| \P\left( Y_{n+1} \in \cC_n(X_{n+1}) \mid X_{n+1} \right)  - (1-\alpha)\big| \right]=0.
    \end{equation}
\end{proposition}
The assumption that the time series is absolutely regular can be interpreted as meaning that, for sufficiently large $m$, data points that are $\geq m$ time points apart are approximately independent---that is, $((X_1,Y_1),\dots,(X_k,Y_k))$ is approximately independent from $((X_{k+m},Y_{k+m}),(X_{k+m+1},Y_{k+m+1}),\dots)$.

\begin{proof}[Proof of Proposition~\ref{prop:time-series}]
Let $Z=(X,Y)\sim P$ denote an independent data point, and let $Z_i=(X_i,Y_i)$ as usual. We will use the following notation:
\[b_m = \sup_{k\geq 1} \dtv\big( (Z_1,\dots,Z_k,Z_{k+m}), (Z_1,\dots,Z_k,Z)\big).\]
By the assumption that the time series is absolutely regular, we must have $\lim_{m\to\infty}b_m=0$. Note that
since $\hat\tau_n$ is trained on $(Z_1,\dots,Z_{\lfloor n/2\rfloor})$, we therefore have
\begin{equation}\label{eqn:timeseries_gamma_m}\dtv\big( (\hat\tau_n, Z_{\lfloor n/2\rfloor +m}) , (\hat\tau_n,Z)\big) \leq b_m\end{equation}
for all $n$ and all $m$. In other words, for large $m$ (i.e., if $b_m\approx 0$), the trained model $\hat\tau_n$ is nearly independent of the future data point $Z_{\lfloor n/2\rfloor +m}=(X_{\lfloor n/2\rfloor +m },Y_{\lfloor n/2\rfloor +m})$.

The key challenge of the proof will be to verify that $\hat{q}_n$ converges to zero in probability, i.e.,
\begin{equation}\label{eqn:time_series_hatqn_limit}\lim_{n\to \infty} \P(|\hat{q}_n|>\epsilon) =0\end{equation}
for any $\epsilon>0$.

\paragraph{Step 1: show that convergence of $\hat{q}_n$ is sufficient.} First we assume~\eqref{eqn:time_series_hatqn_limit} holds.
By definition of $\tau^*$,
$\P\left( Y_{n+1} \in \cC^*(X_{n+1}) \mid X_{n+1}\right)=1-\alpha$ (almost surely), which implies
\[\big|\P\left( Y_{n+1} \in \cC_n(X_{n+1}) \mid X_{n+1}\right) - (1-\alpha)\big|\\  \leq \P\left( Y_{n+1} \in \cC_n(X_{n+1})\triangle \cC^*(X_{n+1}) \mid X_{n+1}\right).\] 
Therefore, recalling the CQR score $s_n$~\eqref{eqn:score_CQR} and the oracle CQR score $s^*$~\eqref{eqn:oracle_score_CQR} from Section~\ref{eq:case_study_2_regression}, we have
\begin{multline*}
    \E\left[\left|\P\left( Y_{n+1} \in \cC_n(X_{n+1}) \mid X_{n+1}\right) - (1-\alpha)\right|\right]\\
    \leq \P\left(Y_{n+1}\in \cC_n(X_{n+1})\triangle \cC^*(X_{n+1})\right)
    =\P\left(\ind{s_n(Z_{n+1})\leq \hat{q}_n} \neq \ind{s^*(Z_{n+1})\leq 0}\right)\\
    \leq \P\left(|s_n(Z_{n+1}) - s^*(Z_{n+1})|>\epsilon\right) + \P\left(|s^*(Z_{n+1})|\leq 2\epsilon\right)  + \P\left(|\hat{q}_n|>\epsilon \right),
\end{multline*}
where for the last step we
fix any $\epsilon>0$. 

Next, by definition of $s^*$, we have
\[|s^*(Z_{n+1})|\leq 2\epsilon \ \Longrightarrow \ \min_{\beta\in\{\alpha/2,1-\alpha/2\}}|Y_{n+1} - \tau^*(X_{n+1};\beta)| \leq 2\epsilon, \]
and therefore
\begin{multline*}\P(|s^*(Z_{n+1})|\leq 2\epsilon)\leq \sum_{\beta\in\{\alpha/2,1-\alpha/2\}}\P\left(-2\epsilon\leq Y_{n+1} - \tau^*(X_{n+1};\beta) \leq 2\epsilon\right)\\ \leq 2\cdot  4\epsilon \cdot \sup_{(x,y)}f^*(y\mid x),\end{multline*}
since $f^*(\cdot\mid X_{n+1})$ is the conditional density of $Y_{n+1}\mid X_{n+1}$.

Furthermore, by definition of $s_n$ and $s^*$, we have
\begin{multline*}
    \P\left(|s_n(Z_{n+1}) - s^*(Z_{n+1})|>\epsilon\right)
    \leq \sum_{\beta\in\{\alpha/2,1-\alpha/2\}}\P\left(|\hat\tau_n(X_{n+1};\beta)-\tau^*(X_{n+1};\beta)| > \epsilon\right) \\
    \leq  \sum_{\beta\in\{\alpha/2,1-\alpha/2\}}\left(\P\left( |\hat\tau_n(X;\beta)-\tau^*(X;\beta)| > \epsilon\right) + b_{\lceil n/2\rceil +1}\right),
\end{multline*}
by~\eqref{eqn:timeseries_gamma_m}.
We also have $\lim_{n\to\infty}b_{\lceil n/2\rceil +1}=0$, and moreover for each $\beta\in\{\alpha/2,1-\alpha/2\}$, $\lim_{n\to\infty} \P\left( |\hat\tau_n(X;\beta)-\tau^*(X;\beta)| > \epsilon\right)=0$ due to the assumption~\eqref{eqn:assume_tau_n_consistent}. Therefore, we have shown that $\lim_{n\to \infty}\P\left(|s_n(Z_{n+1}) - s^*(Z_{n+1})|>\epsilon\right)=0$.

Combining all these calculations, and using the fact that $ \lim_{n\to\infty} \P(|\hat{q}_n|>\epsilon)=0$ by~\eqref{eqn:time_series_hatqn_limit}, we have shown that
\[\limsup_{n\to\infty}\E\left[\left|\P\left( Y_{n+1} \in \cC_n(X_{n+1}) \mid X_{n+1}\right) - (1-\alpha)\right|\right] \leq 8\epsilon \cdot \sup_{x,y}f^*(y\mid x).\]
Since these calculations hold for arbitrarily small $\epsilon>0$, and since $\sup_{(x,y)}f^*(y\mid x)<\infty$ by assumption, this completes the proof.

\paragraph{Step 2: prove convergence of $\hat{q}_n$.} We now need to verify~\eqref{eqn:time_series_hatqn_limit}.
We will prove that, for any $\epsilon>0$, $\P(\hat{q}_n\leq \epsilon)\to 1$; the proof for the lower bound on $\hat{q}_n$ is similar so we omit the details. 

First we need to show a concentration property of $\hat\tau_n$ over the calibration set. We calculate
\begin{align*}
    &\E\left[\frac{1}{\lceil n/2\rceil}\sum_{i=\lfloor n/2\rfloor +1}^n \ind{| \hat\tau_n(X_i;\beta) - \tau^*(X_i;\beta)|>\epsilon/2}\right]\\
    &=\frac{1}{\lceil n/2\rceil}\sum_{i=\lfloor n/2\rfloor +1}^n  \P\left(| \hat\tau_n(X_i;\beta) - \tau^*(X_i;\beta)|>\epsilon/2\right)\\
    &\leq\frac{1}{\lceil n/2\rceil}\sum_{i=\lfloor n/2\rfloor +1}^n \left( \P\left(| \hat\tau_n(X;\beta) - \tau^*(X;\beta)|>\epsilon/2\right) + b_{i-\lfloor n/2\rfloor}\right)\textnormal{\quad by~\eqref{eqn:timeseries_gamma_m}}\\
    &= \P\left(| \hat\tau_n(X;\beta) - \tau^*(X;\beta)|>\epsilon/2\right) + \frac{1}{\lceil n/2\rceil}\sum_{i=1}^{\lceil n/2\rceil}b_i\\
    &\to 0,
\end{align*}
as $n\to\infty$, where the last step holds since the first term is vanishing by the assumption~\eqref{eqn:assume_tau_n_consistent}, while the second term is vanishing since $\sup_mb_m\leq 1$ and $\lim_{m\to\infty}b_m=0$.
By Markov's inequality, then, fixing any $\delta>0$, we have
\begin{equation}\label{eqn:strongly_mixing_step1}\P\left(\frac{1}{\lceil n/2\rceil}\sum_{i=\lfloor n/2\rfloor + 1}^n \ind{| \hat\tau_n(X_i;\beta) - \tau^*(X_i;\beta)|>\epsilon/2} > \delta \right)\to 0\end{equation}
for each $\beta\in\{\alpha/2,1-\alpha/2\}$.

Next, let
\[1-\alpha' = \P_{(X,Y)\sim P}(Y\in\cC^{*,\epsilon/2}(X)),\]
where
\[\cC^{*,\epsilon/2}(x) = \big[\tau^*(x;\alpha/2)-\epsilon/2, \tau^*(x;1-\alpha/2)+\epsilon/2\big].\]
Since the conditional density $f^*(\cdot \mid X)$ of $Y\mid X$ is assumed to be positive, and $\epsilon>0$, we must have
$\alpha' < \alpha$. By the Law of Large Numbers for strongly mixing time series, we have
\begin{equation}\label{eqn:strongly_mixing_LLN}\P\left(\frac{1}{\lceil n/2\rceil} \sum_{i=\lfloor n/2\rfloor + 1}^{n} \ind{Y_i\in\cC^{*,\epsilon/2}(X_i)} < 1-\alpha' - \delta\right) \to 0\end{equation}
for any $\delta>0$.
Next, by definition of the CQR score, $s_n(x,y) =\max\{\hat\tau_n(x;\alpha/2) - y,  y - \hat\tau_n(x;1-\alpha/2)\}$, it holds that
\[\textnormal{If $y\in\cC^{*,\epsilon/2}(x)$ then }s_n(x,y) \leq \epsilon/2 + \max_{\beta\in\{\alpha/2,1-\alpha/2\}} |\hat\tau_n(x;\beta) - \tau^*(x;\beta)|.\]
Consequently,
\begin{multline*}\frac{1}{\lceil n/2\rceil}\sum_{i=\lfloor n/2\rfloor +1}^n \ind{s_n(X_i,Y_i)\leq \epsilon}
\geq \frac{1}{\lceil n/2\rceil}\sum_{i=\lfloor n/2\rfloor +1}^n \ind{Y_i\in\cC^{*,\epsilon/2}(X_i)}\\{} - \sum_{\beta\in\{\alpha/2,1-\alpha/2\}} \frac{1}{\lceil n/2\rceil}\sum_{i=\lfloor n/2\rfloor +1}^n \ind{ |\hat\tau_n(x;\beta) - \tau^*(x;\beta)|> \epsilon/2}
\end{multline*}
By~\eqref{eqn:strongly_mixing_step1} and~\eqref{eqn:strongly_mixing_LLN}, for any $\delta>0$, we therefore have
\[\P\left(\frac{1}{\lceil n/2\rceil}\sum_{i=\lfloor n/2\rfloor +1}^n \ind{s_n(X_i,Y_i)\leq \epsilon} \geq 1 - \alpha' - 3\delta\right) \to 1.\]
Choosing $\delta$ sufficiently small so that $1-\alpha'-3\delta > 1-\alpha$, then,
\[\P\left(\frac{1}{\lceil n/2\rceil}\sum_{i=\lfloor n/2\rfloor +1}^n \ind{s_n(X_i,Y_i)\leq \epsilon} \geq (1 - \alpha)\left(1+\frac{1}{\lceil n/2\rceil}\right)\right) \to 1.\]
But by definition of $\hat{q}_n$ as the quantile of the scores on the calibration data, this means that $\P(\hat{q}_n\leq \epsilon) \to 1$, as desired.
\end{proof}
\index{time series|)}

\section*{Bibliographic notes}
\addcontentsline{toc}{section}{\protect\numberline{}\textnormal{\hspace{-0.8cm}Bibliographic notes}}

Many works in the literature have established asymptotic optimality properties for conformal prediction, in a range of different settings. These works have generally focused on specific choices of the score and/or the underlying distributional assumptions. For the two case studies presented in Section~\ref{sec:case_studies_classification} for the classification setting, the first one uses the high-probability score, which is studied by~\cite{Sadinle2016LeastAS}; this work also establishes asymptotic optimality statements (see also earlier work by~\cite{lei2014classification}). The second case study uses the cumulative-probability score, which is a score proposed by~\cite{romano2020classification} and studied further by~\cite{angelopoulos2020uncertainty}. 
Turning to the regression setting, for the two case studies presented in Section~\ref{sec:case_studies_regression}, the first uses the high-density score, studied by~\cite{lei2013distribution} and \cite{izbicki2019flexible}, which both establish asymptotic optimality results for the score. The second uses the conformalized quantile regression (CQR) score, as developed by  \cite{romano2019conformalized}; subsequent work by \cite{sesia2019comparison} develops asymptotical optimality results for the CQR score.
CQR builds on the large literature in quantile regression, originating from~\cite{koenker1978regression}, with further developments in~\cite{chaudhuri1991global, koenker2005quantile, hao2007quantile, koenker2017quantile, koenker2018handbook}. Related quantile-based scores, which offer similar oracle properties as CQR, include the works of~\cite{chernozhukov2021distributional} and~\cite{kivaranovic2020adaptive}.

Early results establishing asymptotic optimality properties for conformal prediction include results by \citet[Chapter 3]{vovk2005algorithmic} for the classification setting, using a score based on nearest-neighbor classifiers, and the work of \cite{lei2018distribution} in the regression setting, which considers the residual score (this score was introduced in Chapter~\ref{chapter:introduction}, but is not studied in this present chapter).
Additional asymptotic optimality results for various settings include the work of \citet{burnaev2014efficiency} for ridge regression, \cite{lei2014distribution} for kernel density estimates, \cite{lei2011efficient} for unsupervised learning, \cite{papadopoulos2011regression,gyorfi2019nearest} for nearest-neighbors, and \cite{chernozhukov2018exact} for counterfactual inference in synthetic controls. Asymptotic results for conditional coverage are also developed by \cite{azizi2025clear}, by considering a decomposition into aleatoric uncertainty, arising from errors in the model, and epistemic uncertainty, arising from inherent noise in the data distribution (see also earlier work by \citet{rossellini2024integrating}).
Finally, the work of~\citet{hoff2023bayes} establishes an optimality guarantee for the conformalized Bayesian posterior: essentially, the conformalized posterior distribution has the smallest expected volume under the Bayesian model among all sets with distribution-free $1-\alpha$ frequentist coverage---see also~\citet{bersson2024optimal}.

Of course, all of the results in this chapter (and all the works cited above) assume some knowledge of the underlying data distribution in order to establish optimality properties. This is necessary, since although conformal prediction offers distribution-free marginal coverage,  stronger properties can only hold under additional assumptions. For instance, hardness results in Chapter~\ref{chapter:conditional} establish impossibility of assumption-free conditional coverage guarantees (see see Theorems~\ref{thm:conditional_infinite} and~\ref{thm:hardness-test-conditional-coverage-relaxed}). Similarly, \citet{gao2025volume} establish that, for any method $\cC_n$ providing distribution-free marginal predictive coverage, there must exist some distributions for which $\cC_n$ is far from optimal in terms of the size of the prediction set.

The general theory developed in Section~\ref{sec:unified_framework_splitCP_asymp} is adapted from \cite{duchi2024predictive}. We note that this paper studies a more general problem, where split conformal prediction (or, in the paper, alternatively a cross-validation based method) is applied across data gathered from multiple different environments; the i.i.d.\ setting can be viewed as a special case.

Finally, the results of Section~\ref{sec:models_without_exch}, for the time series setting, are based on the work of~\cite{xu2023conformal, xu2023sequential}. The proof of Proposition~\ref{prop:time-series} relies on mixing conditions for time series, including a Law of Large Numbers for strongly mixing time series; see \citet{andrews1988laws,bradley2005basic} for background. Additional results for conformal prediction with time series data exist in the literature for a range of different settings, for instance, the results of~\cite{chernozhukov2021distributional,oliveira2024split,barber2025predictive}.

\section*{Exercises}
\addcontentsline{toc}{section}{\protect\numberline{}\textnormal{\hspace{-0.8cm}Exercises}}
\begin{enumerate}[font=\bfseries, label={\thechapter.\arabic*}, labelsep=1em, itemsep=1em]
\item In the case study considered in Proposition~\ref{prop:case_study_2_classification}, why is it important to allow $\P(s^*(X,Y)=0)$ to be positive? In other words, if we instead replaced the condition with just assuming $\P(s^*(X,Y)=t)=0$ for all $t$ (rather than for all $t>0$), why would this assumption lead to a vacuous result?
\item Consider the case study discussed in Proposition~\ref{prop:case_study_2_regression}, showing that CQR is optimal for minimizing length subject to a symmetric-tail condition.
The result requires a positive density assumption, $f^*(y\mid x)>0$. In this exercise, we aim to show why this assumption is needed---we cannot remove it (even
when all other assumptions are satisfied).
Construct a distribution with a conditional density $f^*(y\mid x)$, such that the CQR interval $\cC_n(X_{n+1})$ has expected length that is substantially larger (even for large $n$) than the 
oracle interval $\cC^*(X_{n+1})$, even if the estimated quantiles $\hat\tau_n(x;\beta)$ (for $\beta\in\{\alpha/2,1-\alpha/2\}$) are highly accurate.
\item Consider the aim of minimal length for an interval with conditional coverage, \begin{equation}
\begin{aligned}
\textnormal{minimize} \quad & \E_{X\sim P_X}[\textnormal{Leb}(\cC(X))] \\
\textrm{subject to} \quad & \P_{Y\sim P_{Y\mid X}}(Y \in\cC(X) \mid X ) \geq 1-\alpha\textnormal{ almost surely},\\&\textnormal{$\cC(x)$ is an interval for all $x$.}
\end{aligned}
\end{equation}
Compared with the aim~\eqref{eq:case_study_2_regression} in the case study considered in Section~\ref{eq:case_study_2_regression}, here we are requiring conditional coverage $\geq 1-\alpha$, but we do not require the miscoverage probability to be equal in the left and right tails.

Recall that the oracle interval from the case study~\eqref{eq:case_study_2_regression} is given by
\begin{equation}
    \cC^*(x) = [\tau^*(x ; \alpha/2), \tau^*(x ; 1-\alpha/2)],
\end{equation}
\begin{enumerate}
    \item Assume that, at any value $x$, the conditional distribution of $Y\mid X$ has a positive, symmetric, and unimodal density $f^*(y\mid x)$ (i.e., symmetric around some value that depends on $x$). Prove that the same interval $\cC^*(x)$ is optimal for this new aim.
    \item Next, construct an example to show that, with a conditional density that is positive but not necessarily symmetric and unimodal, the above definition of $\cC^*(x)$ may not be optimal for this new aim (even though it is optimal for the aim of equal-tailed conditional coverage considered in~\eqref{eq:case_study_2_regression}).
    \end{enumerate}
\end{enumerate}

\part{Extensions of Conformal Prediction}
\label{part:extensions-conformal}

\chapter{Cross-Validation Based Conformal Methods}
\label{chapter:cv}
\index{cross-validation|(}

This chapter begins Part~\ref{part:extensions-conformal} of the book.
The previous part has focused on conformal prediction for batches of exchangeable data points.
In this next part of the book, we present extensions to the conformal prediction procedure that modify it to have a more flexible computational/statistical tradeoff, to handle non-exchangeable data, to operate in an online setting with streaming data, and other extensions of the methodology.

In this chapter, we will focus on cross-validation (CV) style approaches to conformal prediction, covering the main methodologies in this area along with their corresponding theoretical guarantees. Recall that in Part~\ref{part:conformal}, we introduced full conformal prediction and split conformal prediction, which can be viewed as lying at two extremes of a computational/statistical tradeoff.
Split conformal only requires fitting the model once, but may lead to wider intervals due to the statistical cost of data splitting. In contrast, full conformal may provide narrower intervals as it uses the entire available training set for model fitting, but in principle may require a large number of (or even infinitely many) calls to the model fitting algorithm and therefore is very expensive to compute even approximately.

A natural question is therefore whether we can use a strategy that is more statistically efficient than data splitting, without paying the steep computational price of full conformal prediction: namely, can we use cross-validation?
Interestingly, the theoretical analysis of CV-type methods proves to be surprisingly different from the analysis of full and split conformal, and this chapter will go through these arguments in detail.

\section{Preliminaries: split conformal as a tournament}\label{sec:split_as_tournament}
\index{tournament}
\index{split conformal prediction|(}

\begin{figure}[t]
    \centering
    \includegraphics[width=0.8\textwidth]{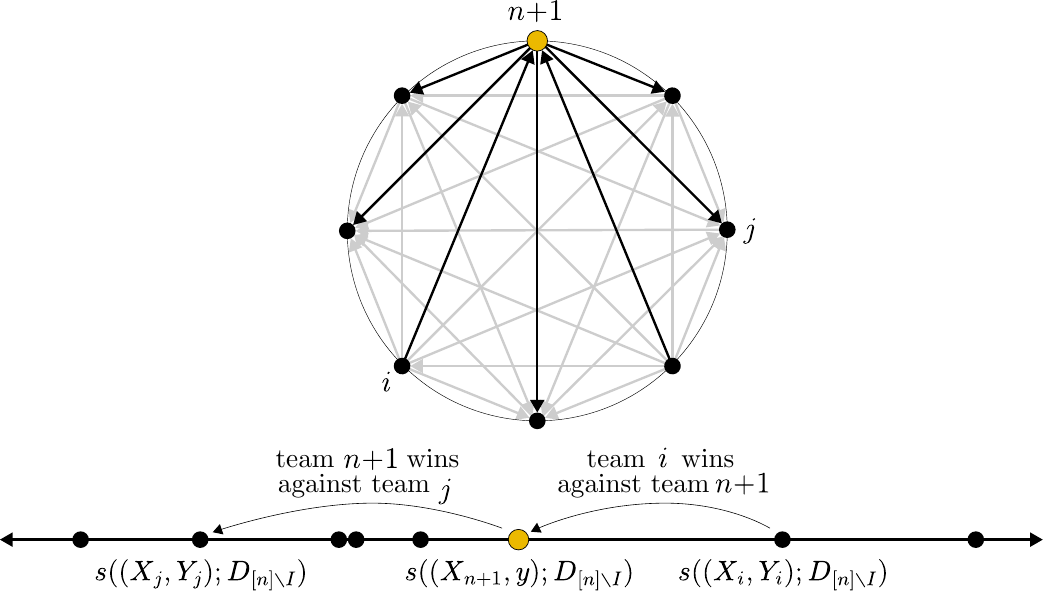}
    \caption{\textbf{Visualization of split conformal prediction as a tournament,} as described in Section~\ref{sec:split_as_tournament}. Given a score function $s$ trained on dataset $D_{[n]\setminus I}$, we construct a `tournament' among the calibration points $(X_i,Y_i)$ (for $i\in I$) and the hypothesized test point $(X_{n+1},y)$, where a team `wins' against another team if its score is strictly larger. The scores are shown in the bottom panel of the figure. The corresponding tournament graph at the top of the figure shows the outcomes of all the games, with arrows pointing from the winning team to the losing team in each pairwise comparison.}
    \commentAlt{The top panel shows a fully connected directed graph among $8$ nodes, with one node highlighted and labeled as $n+1$, and additional nodes $i$ (which points to $n+1$) and $j$ (with node $n+1$ pointing to node $j$). See long description.}
    \commentLongAlt{The top panel shows a fully connected directed graph among $8$ nodes, with one node highlighted and labeled as $n+1$, and additional nodes $i$ (which points to $n+1$) and $j$ (with node $n+1$ pointing to node $j$). Below, a horizontal axis shows $8$ points, with one point highlighted and labeled as $s((X_{n+1},y);\cD_{[n]\backslash I})$, and two other points labeled as $s((X_i,Y_i);\cD_{[n]\backslash I})$ (whose value is higher) and $s((X_j,Y_j);\cD_{[n]\backslash I})$ (whose value is lower). A directed arrow points from $s((X_i,Y_i);\cD_{[n]\backslash I})$ to $s((X_{n+1},y);\cD_{[n]\backslash I})$, and another arrow from $s((X_{n+1},y);\cD_{[n]\backslash I})$ to $s((X_j,Y_j);\cD_{[n]\backslash I})$.}
    \label{fig:split-conformal-tournament}
\end{figure}

Before presenting cross-validation-style variants of the conformal prediction method, it will help to first return to split conformal, and reformulate the construction of that method. Let $I\subseteq[n]$ be a subset of the indices of the available labeled data, and suppose we will use this subset $\cD_I=((X_i,Y_i))_{i\in I}$ as our calibration set, after training our model on the remaining data $\cD_{[n]\setminus I} = ((X_i,Y_i))_{i\in[n]\setminus I}$. 
Note that here we depart from the notation of the earlier chapters: in this section, we assume a total of $n$ labeled data points, and thus $\cD_n$ must be partitioned to provide both data for model training and for calibration. For example, for a data point $i\in I$ used for calibration in the split conformal procedure, its score will be given by $s((X_i,Y_i);\cD_{[n]\setminus I})$. 

The split conformal method operates by comparing the score of a hypothesized test point, $s((X_{n+1},y);\cD_{[n]\setminus I})$, against the scores for calibration data points, $s((X_i,Y_i);\cD_{[n]\setminus I})$, for all $i\in I$.
Concretely, in this setting where we have $|I|$ (rather than $n$) as the number of calibration points, the split conformal interval is defined as
\[\cC(X_{n+1}) = \left\{y : s((X_{n+1},y);\cD_{[n]\setminus I}) \leq \quantile\left(\left(s((X_i,Y_i);\cD_{[n]\setminus I})\right)_{i\in I} ; (1-\alpha)\left( 1 + 1/|I| \right)\right)\right\}.\]
We observe that, by definition of the quantile, this prediction interval can equivalently be defined as
\begin{equation}\label{eq:compare_split_to_CC}\cC(X_{n+1}) = \left\{y : \sum_{i\in I} \ind{ s((X_{n+1},y);\cD_{[n]\setminus I}) > s((X_i,Y_i);\cD_{[n]\setminus I})} < (1-\alpha)(|I|+1)\right\}.\end{equation}
With this reformulation of the prediction interval, we can now see that split conformal can be interpreted as a tournament between teams---that is, between data points $i\in I\cup\{n+1\}$. Each team's strength is determined by its score, namely, $s((X_i,Y_i);\cD_{[n]\setminus I})$ for each team $i\in I$, and $s((X_{n+1},y);\cD_{[n]\setminus I})$ for team $i=n+1$. When team $i$ plays against team $j$, it wins if its score is strictly higher; if the scores are equal, neither team wins. We can therefore see that, to determine whether a value $y$ is included in the prediction interval $\cC(X_{n+1})$, we are equivalently asking:
\begin{quote}
    Did team $n+1$ win fewer than $(1-\alpha)(|I|+1)$ many games, when playing against all teams $i\in I$?
\end{quote}
This tournament interpretation of split conformal, visualized in Figure~\ref{fig:split-conformal-tournament}, will help us to understand the construction of the cross-validation methods that we will study next, as well as the proofs of their corresponding theoretical guarantees.

\index{split conformal prediction|(}

\section{Cross-conformal prediction}
\index{cross-conformal prediction|(}
Next, we will extend this reformulation of split conformal to a cross-validation style approach. To begin, we partition the available data into $K\geq 2$ disjoint \emph{folds}, given by
\[[n] = I_1 \cup \dots \cup I_K.\]
(Typically we would take the folds to be of equal size $n/K$, or as close to equal as possible if $n/K$ is not an integer.) Then for each $k=1,\dots,K$, the fold $I_k$ plays the role of the calibration set: we compare the score of the hypothesized test point against the scores of data points indexed by $i\in I_k$, using the model trained on $\cD_{[n]\setminus I_k}$. That is, for each $k$,
\[\textnormal{Compare $s((X_{n+1},y);\cD_{[n]\setminus I_k})$ against $\left(s((X_i,Y_i);\cD_{[n]\setminus I_k})\right)_{i\in I_k}$}.\]
To pool these $K$ sets of comparisons together, and construct a \emph{single} prediction interval, we define
\begin{equation}\label{eq:cross-conf}\cC(X_{n+1}) = \left\{y : \sum_{k=1}^K\sum_{i\in I_k} \ind{s((X_{n+1},y);\cD_{[n]\setminus I_k}) > s((X_i,Y_i);\cD_{[n]\setminus I_k})} < (1-\alpha)(n+1)\right\},\end{equation}
which is known as the \emph{cross-conformal prediction} set.
In constructing this set, the innermost sum calculates how many games were won by player $n+1$ within fold $k$, and the outermost sum is over the $K$ folds.
Thus, the left-hand side of the inequality simply counts the total number of  games won by player $n+1$.

Now compare to~\eqref{eq:compare_split_to_CC}: that reformulation of split conformal makes it immediately apparent that the cross-conformal method is simply an extension of split conformal to a setting where we cycle through multiple folds.
To ask whether a value $y$ will be included in the prediction set, we can return to the tournament interpretation: when team $n+1$ plays against team $i\in I_k$, the two teams have `strengths' determined by $s((X_{n+1},y);\cD_{[n]\setminus I_k})$ and $s((X_i,Y_i);\cD_{[n]\setminus I_k})$, respectively. In particular, unlike for split conformal, here the strength of team $n+1$, when it plays a game against team $i$, will vary depending on the fold $I_k$ to which team $i$ belongs (see Figure~\ref{fig:cross-conformal-tournament} for an illustration). In both settings, however, the intuition is essentially the same: a value $y\in\cY$ is included in the prediction set if, with this value of $y$, the test point does not win too many of its games. \index{tournament}

The case $K=n$, where each fold contains a single data point, is a special case. For this variant of the method, which is often called leave-one-out cross-conformal prediction, the procedure requires a model to be trained on each of the $n$ leave-one-out datasets, $\cD_{[n]\setminus \{i\}} = ((X_j,Y_j))_{j\in[n]\setminus \{i\}}$.

\subsection{Coverage guarantees for cross-conformal prediction}
\label{sec:theory-CV}

We next establish marginal coverage guarantees for the $K$-fold cross-conformal method.
We will present two results, which follow completely different proof strategies, and offer complementary guarantees. The first theorem uses an argument relating to averaging p-values, and offers a more favorable guarantee in the regime where $K$ is small; the second theorem relies on a tournament-matrix type argument, and gives a more favorable guarantee in the regime where $K$ is large (including the case $K=n$).

\begin{figure}[t]
    \centering
    \includegraphics[width=0.8\textwidth]{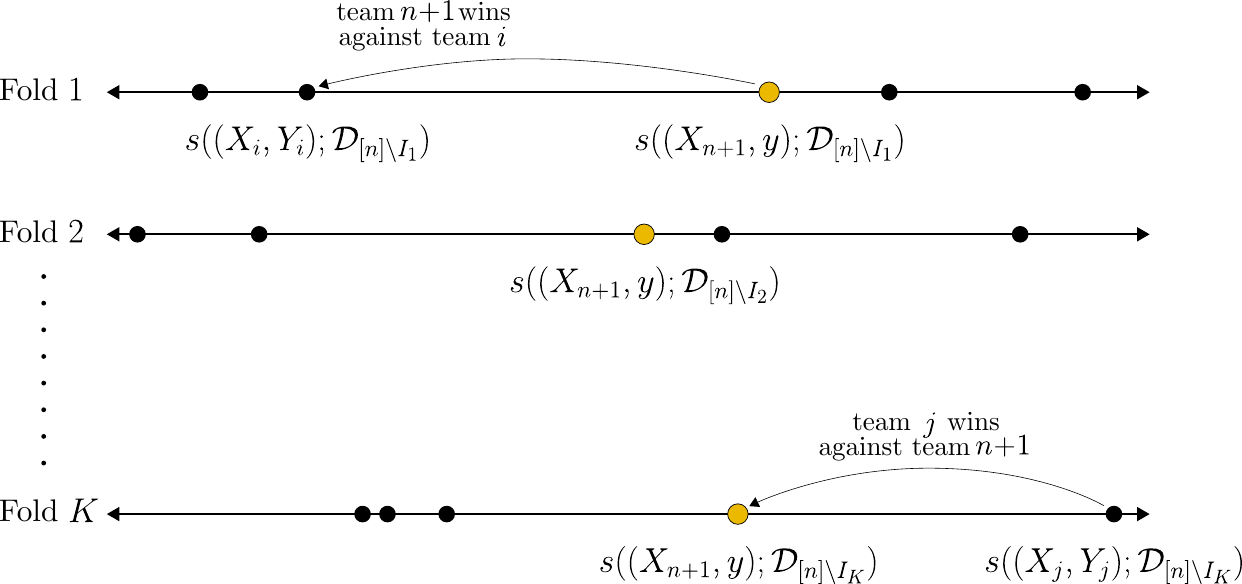}
    \caption{\textbf{Visualization of cross-conformal prediction as a tournament,} for a  partition $[n] = I_1\cup \dots \cup I_K$. For each fold $k$, and for each data point $i\in I_k$, the hypothesized test point $(X_{n+1},y)$ `wins' its game against a data point $(X_i,Y_i)$ if its score is strictly higher, when trained on the data with the $k$th fold held out---that is, we compare $s((X_{n+1},y);\cD_{[n]\backslash I_k})$ against $s((X_i,Y_i);\cD_{[n]\backslash I_k})$.}
    \commentAlt{The top panel shows a horizontal axis labeled as `Fold 1', plotting $5$ scores along the axis, highlighting one as the test point score. A directed arrow points from the higher test point score to a lower training point score. See long description.}
    \commentLongAlt{The top panel shows a horizontal axis labeled as `Fold 1', plotting $5$ scores along the axis, highlighting one as the test point score. A directed arrow points from the higher test point score $s((X_{n+1},y);\cD_{[n]\backslash I_1})$ to a lower training point score $s((X_i,Y_i);\cD_{[n]\backslash I_1})$. The next panel does the same for `Fold 2', and then this repeats until `Fold $K$'. The value of the test point score varies slightly across the different panels, and the values of the other points in each fold are varied.}
    \label{fig:cross-conformal-tournament}
\end{figure}

\begin{theorem}[Cross-conformal prediction coverage guarantee (Part I)]\label{thm:CC_smallK}
If the data points $(X_1,Y_1),\dots,(X_{n+1},Y_{n+1})$ are exchangeable, then the cross-conformal prediction method~\eqref{eq:cross-conf} (run with $K$ folds of equal size $n/K$) satisfies
    \[\P\left(Y_{n+1}\in\cC(X_{n+1})\right) \geq 1-2\alpha -  2(1-\alpha)\cdot\frac{1-1/K}{n/K+1}.\]
\end{theorem}

For the second bound, we will need to assume symmetry of the score $s$ (recall Definition~\ref{def:symmetric_score}). We will also need a slightly stronger form of exchangeability: we will need to assume that the $n+1$ training and test data points can be embedded into a longer, and still exchangeable, sequence of data points:
\begin{equation}\label{eqn:longer_exch_sequence_cross_conf}
\underbrace{(X_1,Y_1),\dots,(X_n,Y_n),(X_{n+1},Y_{n+1})}_{\textnormal{the observed training data + test point}},(X_{n+2},Y_{n+2}),\dots,(X_{n+n/K},Y_{n+n/K})\textnormal{ are exchangeable.}
\end{equation}
This is only mildly stronger than our usual assumption of exchangeability of the training and test data---for instance, it is satisfied if the data points are drawn i.i.d.\ from any distribution, or if they are sampled without replacement from a finite population of length $N\geq n+n/K$. However, when $K<n$, it will not hold in all settings---for example, if the $n+1$ data points are sampled uniformly without replacement from a finite population of size $n+1$.

\begin{theorem}[Cross-conformal prediction coverage guarantee (Part II)]\label{thm:CC_bigK}
        Under the same setting as Theorem~\ref{thm:CC_smallK},
        and assuming also that the exchangeability condition~\eqref{eqn:longer_exch_sequence_cross_conf} holds and that the score function is symmetric (as in Definition~\ref{def:symmetric_score}), the cross-conformal prediction method~\eqref{eq:cross-conf} satisfies
        \[\P\left(Y_{n+1}\in\cC(X_{n+1})\right) \geq 1-2\alpha - 2(1-\alpha)\cdot\frac{1-K/n}{K+1}.\]
\end{theorem}

In particular, Theorem~\ref{thm:CC_bigK} establishes that, for leave-one-out cross-conformal (i.e., the case $K=n$),
 \[\P\left(Y_{n+1}\in\cC(X_{n+1})\right) \geq 1-2\alpha.\]
(Note that for the leave-one-out setting, the condition~\eqref{eqn:longer_exch_sequence_cross_conf} is simply equivalent to exchangeability of the $n+1$ training and test points, since $n+n/K=n+1$ in this case.)

For the general case where $K$ may take any value, by combining the two theorems, we see that $K$-fold cross-conformal satisfies the following bound:
\begin{corollary}[Combined coverage guarantee for cross-conformal prediction]\label{cor:CC_allK}
Under the same setting as Theorem~\ref{thm:CC_bigK}, the cross-conformal prediction method~\eqref{eq:cross-conf} satisfies
\[\P\left(Y_{n+1}\in\cC(X_{n+1})\right) \geq 1-2\alpha - 2(1-\alpha)\cdot \min\left\{\frac{1-1/K}{n/K+1},  \frac{1-K/n}{K+1}\right\} \geq 1-2\alpha - \frac{2}{\sqrt{n}},\]
for any choice of $K$.
\end{corollary}
Ignoring the $\bigo(1/\sqrt{n})$ term, this result essentially tells us that, for any $K$, cross-conformal will always yield marginal coverage at level $\geq 1-2\alpha$---but by construction, we would expect to achieve $\geq 1-\alpha$ coverage (and indeed, this is generally achieved in practice). We will discuss this gap in more detail in Section~\ref{sec:why-factor-of-2} below.

We next prove the individual bounds.
\begin{proof}[Proof of Theorem~\ref{thm:CC_smallK}]
For each fold $k$, and for any $(x,y)$, define
\[p_k(x,y) = \frac{1 + \sum_{i\in I_k} \ind{s((x,y);\cD_{[n]\setminus I_k}) \leq s((X_i,Y_i);\cD_{[n]\setminus I_k})}}{1+n/K}.\]
In fact, $p_k(X_{n+1},y)$ is simply the conformal p-value for a hypothesized test point $(X_{n+1},y)$ (recall Definition~\ref{def:conformal-pvalue}), if we were to run split conformal with pretraining set $\cD_{[n]\setminus I_k}$ and calibration set $I_k$. As in Section~\ref{sec:conformal_as_perm}, then, $p_k(X_{n+1},Y_{n+1})$ satisfies
\[\P\left(p_k(X_{n+1},Y_{n+1})\leq t\right) \leq t\textnormal{ for all $t\in[0,1]$}.\]

Next,
for any $y$, we can calculate
\begin{multline*}\frac{p_1(X_{n+1},y) + \dots + p_K(X_{n+1},y)}{K} \\= \frac{K + \sum_{k=1}^K \sum_{i\in I_k} \ind{s((X_{n+1},y);\cD_{[n]\setminus I_k}) \leq s((X_i,Y_i);\cD_{[n]\setminus I_k})}}{K+n}.\end{multline*}
Comparing to the definition of the cross-conformal method~\eqref{eq:cross-conf}, we see that 
\[y \in \cC(X_{n+1}) \Longleftrightarrow  \frac{p_1(X_{n+1},y) + \dots + p_K(X_{n+1},y)}{K} > \alpha + (1-\alpha)\cdot\frac{K-1}{K+n}.\]
Therefore,
\begin{multline*}\P\left(Y_{n+1}\not\in\cC(X_{n+1})\right) \\= \P\left( \frac{p_1(X_{n+1},Y_{n+1}) + \dots + p_K(X_{n+1},Y_{n+1})}{K} \leq \alpha + (1-\alpha)\cdot\frac{K-1}{K+n}\right).\end{multline*}
Finally, to complete the proof, we apply the fact that averaging p-values provides a quantity that is a p-value up to a factor of 2, i.e., 
\[\P\left(\frac{p_1(X_{n+1},Y_{n+1}) + \dots + p_K(X_{n+1},Y_{n+1})}{K}\leq t \right) \leq 2t\textnormal{ for all $t$.}\]
\end{proof}

\begin{proof}[Proof of Theorem~\ref{thm:CC_bigK}]
The proof of this result will refer back to the tournament interpretation of the cross-conformal prediction interval, described in Section~\ref{sec:split_as_tournament}. Recall that, to determine whether $Y_{n+1}$ is covered by $\cC(X_{n+1})$, we examine the outcomes of `games' played between team $n+1$ and each team $i=1,\dots,n$, where for $i\in I_k$, the outcome of the game is determined by comparing scores $s((X_{n+1},Y_{n+1});\cD_{[n]\setminus I_k})$ and $s((X_i,Y_i);\cD_{[n]\setminus I_k})$.\bigskip

\textbf{Step 1: constructing the tournament.}
We begin by imagining that rather than a single test point $(X_{n+1},Y_{n+1})$,
we instead have $n/K$ test points, $(X_{n+i},Y_{n+i})$, for $i=1,\dots,n/K$.
This will allow us to have $K+1$ folds each containing $n/K$ points, where the
first $K$ folds contain the training data, while the last fold contains a test set. (We can construct this longer sequence of $n+n/K$ many exchangeable data points---that is, $(X_1,Y_1),\dots,(X_{n+n/K},Y_{n+n/K})$---due to assumption~\eqref{eqn:longer_exch_sequence_cross_conf}.)

Next consider a tournament where team $i$ plays against team $j$,
for each pair of teams that are \emph{not} in the same fold. Specifically,
suppose that $i\in I_k$ and $j\in I_\ell$, for $k\neq \ell$.
(Here we have $k,\ell \in[K+1]$, that is, we may choose to have one of the folds
be the $(K+1)$st fold that comprises the test set.)
Next, define $\cD_{n+n/K}$ to be the complete dataset,
with all $n+n/K$ (training and test) data points, 
and for any folds $k\neq \ell\in[K+1]$, let $\cD_{[n+n/K]\setminus (I_k\cup I_\ell)}$ be the corresponding dataset of size $n-n/K$
where \emph{two} folds are removed from $\cD_{n+n/K}$.

We now construct a $(n+n/K)$-by-$(n+n/K)$ matrix of scores,
\[S_{ij} = \begin{cases} s\big((X_i,Y_i); \cD_{[n+n/K]\setminus (I_k\cup I_\ell)}\big),& \textnormal{ if $i,j$ are in different folds: $i\in I_k$, $j\in I_\ell$, $k\neq \ell$},\\
 0, & \textnormal{ if $i,j$ are in the same fold}.\end{cases}\]
 In other words, $S_{ij}$ is the score assigned to the $i$th data point, when the model
 is trained after removing the two folds corresponding to data point $i$ and data point $j$ (and if instead
 $i$ and $j$ are in the same fold, then we assign the arbitrary value $S_{ij}=0$ since no game
 was played).

Next, consider the tournament matrix $A\in\{0,1\}^{(n+n/K)\times(n+n/K)}$, defined as follows: when team $i$ plays agains team $j$, we declare team $i$ to be the winner if its score is strictly higher:
 \[A_{ij} = \ind{S_{ij} > S_{ji}},\]
 for each $i,j\in[n+n/K]$. 
The following lemma deterministically bounds the number of teams which achieve a minimum  number of wins in this tournament. (This result is closely related to Landau's theorem on tournaments.)
\begin{lemma}[Tournament lemma]\label{lem:A-rowsum}
Let $N\geq 1$ and let $A\in\{0,1\}^{N\times N}$ satisfy $A_{ij} + A_{ji} \leq1$ for all $i,j\in[N]$. Then for any $t\in[0,1]$,
\[\sum_{i=1}^N \ind{\sum_{j=1}^N A_{ij} \geq N(1-t)} \leq 2tN.\]
\end{lemma} \index{tournament}
Applying this lemma with $N=n+n/K$ and $t = \alpha + (1-\alpha)\cdot \frac{1-K/n}{K+1}$, we then obtain
\begin{equation}\label{eq:A-rowsum-CC}\sum_{i=1}^{n+n/K}\ind{\sum_{j=1}^{n+n/K} A_{ij} \geq (1-\alpha)(n+1)} \leq 2\left(\alpha + (1-\alpha)\cdot \frac{1-K/n}{K+1}\right)\cdot (n+n/K).\end{equation}
\smallskip

 \textbf{Step 2: relating the tournament to the prediction set.}
 Next we consider the special case where one of the two folds is the test fold. For any $k\in[K]$, $\cD_{[n+n/K]\setminus(I_k\cup I_{K+1})} = \cD_{[n]\setminus I_k}$---that is,
removing the $k$th training fold and also the test fold (the $(K+1)$st fold) from the complete
dataset $\cD_{n+n/K}$, is equivalent to removing \emph{only} the $k$th fold from the training dataset $\cD_n$.
Therefore, considering the definition of $S_{ij}$, if $i=n+1$ and $j\in I_k$ for some $k\leq K$, then we can see that
\[S_{n+1,j} = s((X_{n+1},Y_{n+1});\cD_{[n]\setminus I_k}),\]
while if $j=n+1$ and $i\in I_k$ for some $k\leq K$, then we instead have
\[S_{i,n+1} = s((X_i,Y_i);\cD_{[n]\setminus I_k}).\]
Therefore, for any $k\in[K]$ and $i\in I_k$,
\[s((X_{n+1},Y_{n+1});\cD_{[n]\setminus I_k}) > s((X_i,Y_i);\cD_{[n]\setminus I_k})  \Longleftrightarrow A_{n+1,i} = 1.\]
Recalling the definition of the cross-conformal method~\eqref{eq:cross-conf}, 
 we then have
 \[Y_{n+1}\not\in\cC(X_{n+1}) \Longleftrightarrow \sum_{i=1}^n A_{n+1,i} \geq (1-\alpha)(n+1) \Longleftrightarrow \sum_{i=1}^{n+n/K} A_{n+1,i} \geq (1-\alpha)(n+1),\]
 where the last step holds since $A_{n+1,i} = 0$ for all $i\in I_{K+1} = \{n+1,\dots,n+n/K\}$, by definition.\bigskip
 
\textbf{Step 3: applying exchangeability.}
Next, we establish an exchangeability property for the tournament matrix $A\in\{0,1\}^{(n+n/K)\times(n+n/K)}$. In particular, 
consider any $\sigma$ on $[n+n/K]$ that preserves the equivalence relation induced by the folds:
\begin{equation}\label{eqn:sigma_preserves_folds}\textnormal{If $i,j\in I_k$ for some $k\in[K+1]$, then $\sigma(i),\sigma(j)\in I_\ell$ for some $\ell\in[K+1]$},\end{equation}
so that, for all $i,j\in[n+n/K]$, it holds that $i,j$ are in the same fold if and only if $\sigma(i),\sigma(j)$ are in the same fold. We will show that for any such
$\sigma$,
\begin{equation}\label{eq:tournament-matrix-perm}A \eqd A_\sigma,\end{equation}
where we define $(A_\sigma)_{ij} = A_{\sigma(i),\sigma(j)}$, that is, the permuted tournament matrix $A_\sigma$ is derived from the original tournament matrix $A$ by applying the same permutation $\sigma$ to both row and column indices. To verify~\eqref{eq:tournament-matrix-perm}, we observe that we can write $A$ explicitly as a function of the data: we write $A = A(\cD_{n+n/K})$, where, given any dataset $\cD=(z_r)_{r\in[n+n/K]}$ we define the matrix $A(\cD)$ to have entries 
\[\big(A(\cD)\big)_{ij} = \ind{s\big(z_i;(z_r)_{r\in[n+n/K]\setminus (I_k\cup I_\ell)}\big)> s\big(z_j;(z_r)_{r\in[n+n/K]\setminus (I_k\cup I_\ell)}\big)},\]
if $i\in I_k$, $j\in I_\ell$ for some $k\neq \ell\in [K+1]$, or otherwise if $i,j$ lie in the same fold then we set
\[\big(A(\cD)\big)_{ij}=0.\]
With this new notation in place, it holds immediately that
\[A = A(\cD_{n+n/K}) \eqd A((\cD_{n+n/K})_\sigma),\]
where the first step holds by definition of our new notation of $A$ as a function of the dataset, while the second step holds since the dataset $\cD_{n+n/K}$ is assumed to be exchangeable. Note that thus far, this claim holds for any permutation $\sigma$.
Finally, it is also straightforward to verify that, since the score
function $s$ is symmetric, as long as we choose a permutation $\sigma$ that preserves the folds as in~\eqref{eqn:sigma_preserves_folds}, then $A_\sigma = A((\cD_{n+n/K})_\sigma)$---that is, permuting the tournament matrix $A$ that was computed on the original dataset $\cD_{n+n/K}$, is equivalent to computing the tournament matrix using the permuted dataset $(\cD_{n+n/K})_\sigma$. We have therefore shown that~\eqref{eq:tournament-matrix-perm} holds.

Consequently, for any $\sigma$ that satisfies~\eqref{eqn:sigma_preserves_folds},
\begin{align*}\P\left(Y_{n+1}\not\in\cC(X_{n+1})\right)
&=\P\left(\sum_{j=1}^{n+n/K}A_{n+1,j} \geq (1-\alpha)(n+1)\right) \\
&= \P\left(\sum_{j=1}^{n+n/K}(A_\sigma)_{n+1,j} \geq (1-\alpha)(n+1)\right)\\
&= \P\left(\sum_{j=1}^{n+n/K}A_{\sigma(n+1),\sigma(j)} \geq (1-\alpha)(n+1)\right)\\
&= \P\left(\sum_{j=1}^{n+n/K}A_{\sigma(n+1),j} \geq (1-\alpha)(n+1)\right),\end{align*}
where the last step holds since summing over $\sigma(j)$, for $j\in[n+n/K]$, is equivalent to simply summing over $j\in[n+n/K]$. For any $i\in[n+n/K]$, we can choose some $\sigma$ satisfying~\eqref{eqn:sigma_preserves_folds} such that $\sigma(n+1)=i$, so we have
\[\P\left(Y_{n+1}\not\in\cC(X_{n+1})\right) = \P\left(\sum_{j=1}^{n+n/K}A_{ij} \geq (1-\alpha)(n+1)\right).\]
After averaging over all $i\in[n+n/K]$, then,
\begin{multline*}
    \P\left(Y_{n+1}\not\in\cC(X_{n+1})\right)
    = \frac{1}{n+n/K}\sum_{i=1}^{n+n/K} \P\left(\sum_{j=1}^{n+n/K}A_{ij} \geq (1-\alpha)(n+1)\right)\\
    = \E\left[\frac{1}{n+n/K}\sum_{i=1}^{n+n/K} \ind{\sum_{j=1}^{n+n/K}A_{ij} \geq (1-\alpha)(n+1)}\right]
    \leq 2\alpha + 2(1-\alpha)\cdot \frac{1-K/n}{K+1},
\end{multline*}
where the last step holds by applying~\eqref{eq:A-rowsum-CC} to the sum inside the expected value.

\end{proof}

\subsection{Proof of the tournament lemma}
In this section, we provide the formal proof of the tournament lemma, and discuss the intuition behind this result.
\begin{proof}[Proof of Lemma~\ref{lem:A-rowsum}]
Let 
\[J = \left\{i \in [N] : \sum_{j=1}^N A_{ij} \geq (1-t)N\right\},\]
so that our goal is to bound $|J|$.
We then calculate
\begin{align}
    |J| \cdot (1-t)N
    \label{eq:tournament_step1}&\leq \sum_{i \in J} \sum_{j=1}^N A_{ij}\textnormal{ by definition of $J$}\\
    \notag&= \sum_{i\in J} \sum_{j\in J} A_{ij} + \sum_{i\in J}\sum_{j\in[N]\setminus J}A_{ij}\\
    \notag&= \frac{1}{2} \sum_{i\in J} \sum_{j\in J} (A_{ij} + A_{ji}) + \sum_{i\in J}\sum_{j\in[N]\setminus J}A_{ij}\\
    \label{eq:tournament_step2}&\leq  \frac{1}{2} |J|^2 + |J|(N - |J|),
\end{align}
where the last step holds since, for all $i,j$, $A_{ij}\in\{0,1\}$ and $A_{ij}+A_{ji}\leq 1$, by assumption. Rearranging terms, we have proved the desired bound.
\end{proof}

\begin{figure}[t]
    \centering
    \includegraphics[width=0.3\textwidth]{\diagramspath tournament-matrix.pdf}
    \caption{\textbf{An illustration of the tournament matrix} and the set $J$ constructed in the proof of Lemma~\ref{lem:A-rowsum}.}
    \commentAlt{A square matrix $A$ partitioned into four blocks, with Blocks 1, 2, 3, 4 in the top-left, top-right, bottom-left, and bottom-right. The rows are partitioned into sets $J$ (top) and $[N]\setminus J$ (bottom), and same for the columns (left and right).}
    \label{fig:tournament-matrix}
\end{figure}

The key idea of the proof is a counting argument, which is illustrated in Figure~\ref{fig:tournament-matrix}. By definition of $J$, the total number of $1$'s in Blocks 1 and 2 of the figure, must be at least $|J|\cdot (1-t)N$, as in~\eqref{eq:tournament_step1}. On the other hand, the upper bound in step~\eqref{eq:tournament_step2} is due to the fact that, since $A$ is a tournament matrix, at most half of the entries of Block 1 can be equal to $1$.

\subsection{Why the factor of 2?}\label{sec:why-factor-of-2}

We now explain why the miscoverage rate bound of cross conformal is $2\alpha$ rather than $\alpha$, and give an example showing that the factor of $2$ is indeed necessary.
We first reexamine the proof of full conformal prediction through the lens of the tournament matrix construction.
For full conformal, the key construction in the proof is a \emph{vector} of scores,
\[(S_1,\dots,S_{n+1}),\]
with $i$th entry given by
\[S_i = s((X_i,Y_i); \cD_{n+1}),\]
where $\cD_{n+1}$ comprises the training and test data, $(X_1,Y_1),\dots,(X_n,Y_n),(X_{n+1},Y_{n+1})$. The next step is to simply observe that at most $\alpha$ fraction of this list can be above the $(1-\alpha)$-quantile of this list,
\[\sum_{i=1}^{n+1}\ind{S_i > \quantile\left(S_1,\dots,S_{n+1}; 1-\alpha\right)} \leq \alpha(n+1)\]
(recall Fact~\ref{fact:conversion-quantiles-cdfs}\ref{fact:conversion-quantiles-cdfs_part3}).
In other words, if we were to define a tournament matrix as 
\[A_{ij} = \ind{S_i > S_j},\]
which we can interpret as saying that team $i$ `wins' against team $j$ if its score is strictly higher,
then the above quantile inequality could be written as
\begin{equation}\label{eq:A-rowsum-fullCP}\sum_{i=1}^{n+1}\ind{\sum_j A_{ij} \geq (1-\alpha)(n+1)} \leq \alpha(n+1),\end{equation}
that is, if we count teams that win at least $(1-\alpha)$ of their games,
these teams can comprise at most an $\alpha$ fraction of the total number of teams in the league. By contrast, Lemma~\ref{lem:A-rowsum} shows that, for a general tournament matrix $A$, the bound inflates by a factor of 2.

The key difference is that in full conformal, the the tournament matrix $A$ is induced by a vector of scores and corresponds to a \emph{total ordering} on the $n+1$ data points. In particular, this tournament matrix satisfies transitivity: if $A_{ij}=1$ and $A_{jr}=1$ (i.e., team $i$ wins against $j$, and team $j$ wins against $r$), then we cannot have $A_{ri}=1$ (i.e., team $r$ cannot then win against $i$).
In contrast, for cross-conformal, it is possible to have a failure of transitivity where, say, $A_{ij}=1$, $A_{jr}=1$, and $A_{ri} = 1$---that is, team $i$ wins against $j$, team $j$ wins against $r$, and team $r$ wins against $i$---and thus $A$ is not consistent with any total ordering. 
 
The following construction illustrates a case where the factor of 2 is needed. For simplicity, take $K=n$, so that we have $A\in\{0,1\}^{(n+1)\times(n+1)}$, and assume that $m = \alpha(n+1) - 1$ is an integer.  Then define $A$ as follows:
\[A_{ij} = \begin{cases}
    1, & i,j\leq 2m+1\textnormal{ and }1\leq \textnormal{mod}(j-i,2m+1) \leq m, \\ 
    1, & i\leq 2m+1 < j, \\
    0, & \textnormal{ otherwise.}
\end{cases}
\]
For example, if $n+1 = 10$ and $\alpha = 0.4$ so that $m=3$, we have
\begin{equation}\label{eqn:jackknife_worstcase}A = \left(\begin{array}{ccccccc|ccc}
0 & 1 & 1 & 1 & 0 & 0 & 0 & 1 & 1 & 1\\
0 & 0 & 1 & 1 & 1 & 0 & 0 & 1 & 1 & 1\\
0 & 0 & 0 & 1 & 1 & 1 & 0 & 1 & 1 & 1\\
0 & 0 & 0 & 0 & 1 & 1 & 1 & 1 & 1 & 1\\
1 & 0 & 0 & 0 & 0 & 1 & 1 & 1 & 1 & 1\\
1 & 1 & 0 & 0 & 0 & 0 & 1 & 1 & 1 & 1\\
1 & 1 & 1 & 0 & 0 & 0 & 0 & 1 & 1 & 1\\\hline
0 & 0 & 0 & 0 & 0 & 0 & 0 & 0 & 0 & 0\\
0 & 0 & 0 & 0 & 0 & 0 & 0 & 0 & 0 & 0\\
0 & 0 & 0 & 0 & 0 & 0 & 0 & 0 & 0 & 0\end{array}\right),
\end{equation}
where the top-left block is of size $(2m+1)$-by-$(2m+1)$. We can see that, for each $i=1,\dots,2m+1$, we have \[\sum_{j=1}^{n+1} A_{ij} = \sum_{j=1}^{2m+1} A_{ij} + \sum_{j=(2m+1)+1}^{n+1} A_{ij} = m + (n+1 - (2m+1)) = n - m = (1-\alpha)(n+1).\]
We therefore have
\[\sum_{i=1}^{n+1}\ind{\sum_j A_{ij} \geq (1-\alpha)(n+1)} = 2m+1 = 2\alpha (n+1) - 1,\]
which shows that it is possible to have the row sum bound~\eqref{eq:A-rowsum-fullCP} fail,
by (nearly) a factor of 2.
\index{cross-conformal prediction|)}

\section{CV+ and jackknife+}\label{sec:CV+_jack+}
\index{cross-validation!CV+|(}

In this section, we will first step away from the framework of conformal and cross-conformal and consider a more classical approach: the prediction interval constructed via cross-validation (CV), for the problem of predicting a real-valued response (i.e., $\cY=\R$). 
Specifically, given a fitted model $\hf:\cX\rightarrow\R$ trained on the full training set $\cD_n=((X_i,Y_i))_{i\in[n]}$, a popular approach for constructing a prediction interval around $\hf(X_{n+1})$ is to use cross-validation to estimate the margin of error (i.e., the width of the interval).
CV has similar benefits to cross-conformal prediction when compared to full and split conformal prediction: it allows one to traverse the tradeoff between computational and statistical efficiency.
However, we will see that prediction intervals formed via CV can fail in the distribution-free setting, but a small modification (the CV+ method) resolves the issue. 
We will then see how CV+ is actually closely connected to cross-conformal.

Before proceeding, we introduce a new piece of notation. For the remainder of this chapter, we will write $\cA$ to denote a regression algorithm that is run in order to produce the fitted model $\hf$---for instance, the least-squares algorithm. That is, $\cA$ is a function that inputs a dataset $\cD$ (a collection of points in $\cX\times\R$), and returns a fitted model---the function $\hf:\cX\rightarrow\R$. (To compare to our earlier notation, when we discussed the residual score in Chapter~\ref{chapter:conformal-exchangeability}, the notation $\hf(\cdot;\cD)$ denoted this same fitted model.)

\subsection{Cross-validation for prediction}

Cross-validation refers to a genre of classical methods for estimating predictive error. Given a training set of $n$ data points, $\cD_n = ((X_i,Y_i))_{i\in[n]}$,
the cross-validation procedure for constructing a prediction interval can be implemented as follows.
\begin{enumerate}
    \item Split the training dataset $\cD_n$  into $K$ disjoint folds, via a partition $[n] = I_1\cup\ldots\cup I_K$.
    \item For each fold $k \in [K]$, fit the model $\hf_{-I_k} = \cA(\cD_{[n]\setminus I_k})$, where $\cD_{[n]\setminus I_k} = ((X_i,Y_i))_{i\in[n]\setminus I_k}$ denotes the training set with $k$th fold $I_k$ held out.
    \item For each data point $i\in[n]$, compute $S_i = |Y_i - \hf_{-I_{k(i)}}(X_i)|$, where $k(i)$ is the fold of the $i$th data point (i.e., the value $k$ such that $i\in I_k$).
    \item Compute the margin of error, $\hat{q}_{\rm CV} = \quantile\left(S_1,\dots,S_n;1-\alpha\right)$.
    \item Refit the model on the full dataset, $\hf=\cA(\cD_n)$, and output the prediction interval
        \begin{equation}\label{eq:classical_CV}\cC(X_{n+1}) = \hf(X_{n+1}) \pm \hat{q}_{\rm CV}.\end{equation}
\end{enumerate}
The reason cross-validation works well in practice is that the residuals $S_i$ have approximately the same distribution as the residuals of the fully trained model $\hf$ on new data, because the data point $(X_i,Y_i)$, was not used in training the model $\hf_{-I_{k(i)}}$. 
When $K=n$, this algorithm is sometimes called the jackknife, or, leave-one-out cross-validation.

Although the CV algorithm above is well-motivated, it does not have an assumption-free guarantee of coverage.
To see why, consider the following failure case:
\begin{example}[Failure of predictive coverage for CV]\label{example:jackknife_fails}
Suppose that our regression algorithm $\cA$ is the following: given a dataset $\cD$, to fit the model $\hf=\cA(\cD)$, if $\cD$ has an even number of data points then we return the constant model $\hf(x)= 0$, while if $\cD$ has an odd number of data points then we instead return $\hf(x)= 1$.

Now consider a distribution $P$ with $\P_P(Y=0) = 1$. Suppose that $n$ and $n/K$ are both odd numbers. Then, for each fold $k$, $\hf_{-I_k}(x)=0$ for all $x$, because the dataset $\cD_{[n]\setminus I_k}$ contains an even number of data points. Consequently, the $i$th training data point has residual $S_i=0$, for all $i\in[n]$. On the other hand, $\cD_n$ has an odd number of data points, and so we will have $\hf(X_{n+1})=1$. The $K$-fold CV interval will then be equal to $\cC(X_{n+1}) =1 \pm 0 = \{1\}$ almost surely, which will have zero probability of coverage under the distribution $P$.
\end{example}

This example, while highly degenerate, shows that CV can exhibit undercoverage if the performance of the model $\hf$ (fitted on $n$ data points) is substantially different from the performance of models $\hf_{-I_k}$ trained on a smaller sample size $n-n/K$. To avoid such failure cases, below we will derive a slight modification to the CV algorithm.
This modification (the CV+ method, and its special case, jackknife+) enjoys an assumption-free coverage guarantee. 

\subsection{An alternative formulation of CV}
Before proceeding, we pause to give an equivalent definition of the CV prediction interval, which will be useful for extending to the CV+ method below.

The $K$-fold CV prediction interval~\eqref{eq:classical_CV} can equivalently be defined as
\begin{equation*}\cC(X_{n+1}) = \left[\hf(X_{n+1}) - \quantile\left( (S_i)_{i\in[n]} ; 1-\alpha \right),\right.
\left. \hf(X_{n+1}) + \quantile\left( (S_i)_{i\in[n]} ; 1-\alpha \right)\right].\end{equation*}
Rewriting this one more time, we have
\begin{multline}\label{eq:classical_CV_reinterpret}\cC(X_{n+1}) = \left[  -\quantile\left( \left( -(\hf(X_{n+1}) -  S_i)\right)_{i\in[n]} ; 1-\alpha \right), \right.  \\\left.\quantile\left( \left( \hf(X_{n+1}) +  S_i\right)_{i\in[n]} ; 1-\alpha\right)\right].\end{multline}

The unusual `double negative' form of the left endpoint is due to the fact that, in general, $-\quantile\big((-z_1,\dots,-z_n);\tau\big)$ and $ \quantile\big((z_1,\dots,z_n);1-\tau\big)$ are not necessarily equal; for defining this interval, using the `double negative' for the left endpoint makes its construction analogous to that of the right endpoint.

\subsection{Defining CV+ and jackknife+}
We are now ready to define the modified versions of the CV and jackknife prediction intervals. 
First, for $K$-fold CV+, we define
\begin{multline}\label{eq:CV+}\cC(X_{n+1}) = \left[  -\quantile\left( \left( -( \hf_{-I_{k(i)}}(X_{n+1}) -  S_i)\right)_{i\in[n]} ; (1-\alpha)(1+1/n) \right), \right. \\ \left.\quantile\left( \left( \hf_{-I_{k(i)}}(X_{n+1}) +  S_i\right)_{i\in[n]} ; (1-\alpha)(1+1/n) \right)\right],\end{multline}
where as before we define $S_i = |Y_i - \hf_{-I_{k(i)}}(X_i)|$, and $k(i)\in[K]$ denotes the fold to which data point $i$ belongs.
Comparing to~\eqref{eq:classical_CV}, we can see that the difference
lies in using the leave-one-fold-out predictions, $\hf_{-I_k}(X_{n+1})$, in place of the prediction $\hf(X_{n+1})$ obtained by retraining on the full dataset. 
The quantile level $1-\alpha$ has also been replaced by $(1-\alpha)(1+1/n)$, which is useful for obtaining finite-sample guarantees, just as for conformal prediction in earlier chapters.
For the special case $K=n$,
the jackknife+ prediction interval is defined as
\begin{multline}\label{eq:jackknife+}\cC(X_{n+1}) = \left[  -\quantile\left( \left(-( \hf_{-i}(X_{n+1}) -  S_i)\right)_{i\in[n]} ;  (1-\alpha)(1+1/n) \right), \right. \\ \left.\quantile\left( \left(\hf_{-i}(X_{n+1}) +  S_i\right)_{i\in[n]} ; (1-\alpha)(1+1/n) \right)\right],\end{multline}
where $\hf_{-i}=\cA(\cD_{[n]\backslash\{i\}})$ is the leave-one-out fitted model.
Note that for $\alpha \leq 0.5$, the left endpoint is guaranteed to take a value less than or equal to the right endpoint. For larger $\alpha$, if this is violated, we define $\cC(X_{n+1})$ to be the empty set.

Section~\ref{sec:CC_CV} below will relate the CV+ and jackknife+ methods to the cross-conformal framework defined earlier in this chapter, and will present distribution-free theoretical guarantees for these methods. For the moment, we pause to give some brief intuition about why the CV+ and jackknife+ methods avoid the 
worst-case scenario that arose for CV and jackknife in Example~\ref{example:jackknife_fails}. In that degenerate example, 
the issue arose from the fact that CV (i.e., traditional CV, rather than CV+) uses residuals $S_i = |Y_i - \hf_{-I_{k(i)}}(X_i)|$ to try to estimate the value of the error on the test point, $|Y_{n+1} - \hf(X_{n+1})|$.
These two quantities are not drawn from the same distribution: for $i\in[n]$, $S_i$ is an out-of-sample error for a model fitted on $n -n/K$ data points, while the test error $|Y_{n+1} - \hf(X_{n+1})|$ is an out-of-sample error for a model fitted on $n$ data points. In contrast, for CV+, by using predictions $\hf_{-I_k}(X_{n+1})$ (in place of the original prediction $\hf(X_{n+1})$), we are implicitly comparing $S_i$ with a test point error $|Y_{n+1} - \hf_{-I_k}(X_{n+1})|$. These two errors are now drawn from the same distribution (and are in fact exchangeable), since both are out-of-sample errors for the same fitted model $\hf_{-I_k}$.

\subsection{Interpreting the interval}
Unlike the CV (or jackknife), this new CV+ (or jackknife+) interval is no longer constructed as a symmetric interval around the prediction $\hf(X_{n+1})$ fitted on the entire training dataset---and indeed in extreme cases it can even be the case that the prediction $\hf(X_{n+1})$ is not even contained in the CV+ interval $\cC(X_{n+1})$. Therefore, we cannot interpret this new construction as giving us an interval around $\hf(X_{n+1})$. However, an alternative interpretation does hold: 
\begin{proposition}\label{prop:CV+_contains_median}
Assume $\alpha\leq 0.5$.
The CV+ interval~\eqref{eq:CV+} (run with $K$ folds of equal size $n/K$) must contain the median prediction
\[\hf^{\textnormal{med}}(X_{n+1}) = \textnormal{Median}\left(\left(\hf_{-I_k}(X_{n+1})\right)_{k\in [K]}\right).\]
In particular for the case $K=n$, the jackknife+ prediction interval contains the median leave-one-out prediction,
\[\hf^{\textnormal{med}}(X_{n+1}) = \textnormal{Median}\left(\left(\hf_{-i}(X_{n+1})\right)_{i\in [n]}\right).\]
\end{proposition}
We can view this median value as an ensembled prediction, obtained by retraining $\hf$ on different subsets of the training data, and can therefore interpret CV+ as providing an interval around $\hf^{\textnormal{med}}(X_{n+1})$ to quantify our uncertainty about this ensembled predictor.

\begin{proof}[Proof of Proposition~\ref{prop:CV+_contains_median}]
The upper bound of the CV+ interval~\eqref{eq:CV+} satisfies
\begin{align*}
    &\quantile\left( \left(\hf_{-I_{k(i)}}(X_{n+1}) +  S_i\right)_{i\in[n]} ; (1-\alpha)(1+1/n) \right)\\
    &\geq \quantile\left( \left(\hf_{-I_{k(i)}}(X_{n+1})\right)_{i\in[n]} ; (1-\alpha)(1+1/n) \right)\\
    &\geq \textnormal{Median}\left(\left(\hf_{-I_{k(i)}}(X_{n+1})\right)_{i\in[n]}  \right),
\end{align*}
where the first inequality holds since $S_i\geq 0$ for all $i$ by construction, and the second inequality holds since $\alpha\leq 0.5$. Moreover, we also have
\[\textnormal{Median}\left(\left(\hf_{-I_{k(i)}}(X_{n+1})\right)_{i\in[n]}  \right) = \textnormal{Median}\left(\left(\hf_{-I_k}(X_{n+1})\right)_{k\in[K]}\right) = \hf^{\textnormal{med}}(X_{n+1}),\]
where the first equality holds since the folds are assumed to be of equal size $n/K$: we are using the fact that the median of a vector $v=(v_1,v_2,\dots,v_{K-1},v_K)$ will be unchanged if we repeat each value an equal number of times, e.g., $(v_1,v_1,v_2,v_2,\dots,v_{K-1},v_{K-1},v_K,v_K)$. A similar argument shows that the lower bound of the CV+ interval satisfies
\[-\quantile\left( \left( -( \hf_{-I_{k(i)}}(X_{n+1}) -  S_i)\right)_{i\in[n]} ; (1-\alpha)(1+1/n) \right) \leq \hf^{\textnormal{med}}(X_{n+1}),\]
and therefore, it must hold that $\hf^{\textnormal{med}}(X_{n+1})\in \cC(X_{n+1})$.
\end{proof}

\section{Coverage guarantees for CV+ and jackknife+}\label{sec:CC_CV}
On the surface, the CV+ method (with jackknife+ as a special case) appears to be quite different from cross-conformal. The CV+ method is constructed to form intervals around a predicted value of $Y_{n+1}$ based on computing residuals $S_i$ for the leave-one-out or leave-one-fold-out models, while cross-conformal prediction can return a prediction set that is not an interval, and does not even need to use a score that is based on a regression model.
However, we will now see that CV+ is closely related to the cross-conformal prediction method: 
the cross-conformal set is \emph{always} contained in the CV+ interval, if cross-conformal is implemented with the  residual score, $s((x,y);\cD) = |y - [\cA(\cD)](x)|$. (Here $[\cA(\cD)](x)$ denotes the prediction at $x$ when the algorithm $\cA$ is trained on dataset $\cD$---that is, if $\hf = \cA(\cD)$ is the fitted model, then $s((x,y);\cD) = |y-\hf(x)|$.) 

In particular, this means that cross-conformal has the advantage of being potentially less conservative, but on the other hand CV+ is potentially more interpretable as it is guaranteed to return an interval (rather than, perhaps, a disconnected set) and to contain the median ensembled prediction, by construction. In practice, however, the output of the two methods is often identical.

\index{cross-conformal prediction}
\begin{proposition}[Relating CV+ and jackknife+ to cross-conformal prediction]\label{prop:CC_CV+}
    Let $[n] = I_1\cup \dots \cup I_K$ be any partition.
    Let $\cC^{\textnormal{CV+}}(X_{n+1})$ be the $K$-fold CV+ interval
    constructed with some regression algorithm $\cA$, as in~\eqref{eq:CV+}. 
    Define a score function
    \[s((x,y);\cD) =|y - [\cA(\cD)](x)|,\]
    and let $\cC^{\textnormal{CC}}(X_{n+1})$ be the 
    $K$-fold cross-conformal interval constructed as in~\eqref{eq:cross-conf} with this score function. Then it holds surely that
    \[\cC^{\textnormal{CC}}(X_{n+1})\subseteq \cC^{\textnormal{CV+}}(X_{n+1}).\]
\end{proposition}
We note that in the special case $K=n$, this result implies that the jackknife+ interval contains the $n$-fold (i.e., leave-one-out) cross-conformal prediction set. 

Since, by Proposition~\ref{prop:CC_CV+}, the prediction interval constructed by jackknife+ or CV+ is guaranteed to contain the cross-conformal prediction set, the coverage guarantees in Section~\ref{sec:theory-CV} apply immediately to jackknife+ and to CV+. 
For completeness, these are stated in the following corollary.
\begin{corollary}[Predictive coverage guarantee for CV+ and jackknife+]\label{cor:CV_jack}
Assume $\cA$ is a symmetric algorithm (i.e., $\cA(\cD) = \cA(\cD_\sigma)$ for any dataset $\cD$ and any permutation $\sigma$). If the data points $(X_1,Y_1),\dots,(X_{n+1},Y_{n+1})$ are exchangeable, then the jackknife+ interval~\eqref{eq:jackknife+} satisfies
    \[\P\left(Y_{n+1}\in\cC(X_{n+1})\right) \geq 1-2\alpha.\]
Moreover, for any $K$, if the exchangeability condition~\eqref{eqn:longer_exch_sequence_cross_conf} holds, then the CV+ interval~\eqref{eq:CV+} (run with $K$ folds of equal size $n/K$) satisfies
            \begin{equation}
        \P\left(Y_{n+1}\in\cC(X_{n+1})\right) \geq 1-2\alpha - 2(1-\alpha)\cdot\min\left\{\frac{1-1/K}{n/K+1},\frac{1-K/n}{K+1}\right\} \geq 1-2\alpha -2/\sqrt{n}.
    \end{equation}
\end{corollary}

\begin{proof}[Proof of Proposition~\ref{prop:CC_CV+}]
    If $y\not\in\cC^{\textnormal{CV+}}(X_{n+1})$, then either $y$ lies above the right endpoint of the interval $\cC^{\textnormal{CV+}}(X_{n+1})$, i.e.,
    \[y > \quantile\left( \left(\hf_{-I_{k(i)}}(X_{n+1}) +  S_i\right)_{i\in[n]} ; (1-\alpha)(1+1/n) \right),\]
    or $y$ lies below the left endpoint of the interval. By symmetry, these two cases are proved analogously, so without loss of generality we will assume the first case holds. Then, equivalently, we have
    \[\sum_{i=1}^n \ind{y > \hf_{-I_{k(i)}}(X_{n+1}) +  S_i} = \sum_{k=1}^K \sum_{i\in I_k} \ind{y > \hf_{-I_k}(X_{n+1}) +  S_i} \geq (1-\alpha)(n+1).\]
    For any $i\in I_k$, 
    \begin{multline*}y > \hf_{-I_k}(X_{n+1}) +  S_i \ \Longrightarrow \ |y - \hf_{-I_k}(X_{n+1})| > S_i \\ \Longleftrightarrow \ s((X_{n+1},y);\cD_{[n]\setminus I_k}) > s((X_i,Y_i);\cD_{[n]\setminus I_k}), \end{multline*}
    where the last step holds by definition of the score function,
    and so
    \[\sum_{k=1}^K \sum_{i\in I_k} \ind{s((X_{n+1},y);\cD_{[n]\setminus I_k}) > s((X_i,Y_i);\cD_{[n]\setminus I_k})} \geq (1-\alpha)(n+1).\]
    Comparing to the definition of $\cC^{\textnormal{CC}}(X_{n+1})$ given in~\eqref{eq:cross-conf}, we see that this implies $y\not\in\cC^{\textnormal{CC}}(X_{n+1})$, as desired.    
\end{proof}
\index{cross-validation!CV+|)}

\section{Training-conditional coverage for cross-validation type methods}
\index{coverage!training-conditional|(}
In Chapter~\ref{chapter:conditional}, for the full and split conformal methods, we considered the question of training-conditional coverage---do the methods offer reliable predictive coverage, conditional on the dataset $\cD_n$? In this section, we will now ask this question for the cross-validation based methods that are the focus of this chapter.

First, we will see that for $K$-fold cross-conformal, we can establish a training-conditional coverage guarantee as long as the size $n/K$ of each fold is sufficiently large. At a high level, this result holds because the scores $(S_i)_{i\in I_k}$, where  $S_i = s\big((X_i,Y_i);\cD_{[n]\setminus I_k}\big)$, are i.i.d.\ conditional on the data $\cD_{[n]\setminus I_k}$ used for training for the $k$th fold, and therefore concentration results can be applied. In this sense, the result below can be viewed as an extension of the training-conditional coverage guarantee for split conformal, which was given in Theorem~\ref{thm:training-conditional-split-CP}. Before stating the result, we first recall a definition from Chapter~\ref{chapter:conditional}: we define $\alpha_P(\cD_n) =\P(Y_{n+1} \not\in \cC(X_{n+1}) \mid \cD_n)$, the  probability of coverage failing when we condition on $\cD_n$. As before, the goal of training-conditional coverage is to establish a high-probability upper bound on $\alpha_P(\cD_n)$.

\index{cross-conformal prediction}
\begin{theorem}[Training-conditional coverage for cross-conformal prediction]\label{thm:training-conditional-cross-conformal}
Let $P$ be any distribution on $\cX\times\cY$, and let $s$ be any score function. Let $\delta\in(0,1)$. Then cross-conformal (with $K$ folds of equal size $n/K$) satisfies
the training-conditional coverage guarantee    \[\P\left(\alpha_P(\cD_n) \leq 2\alpha + \sqrt{\frac{2\log(K/\delta)}{n/K}}\right)\geq 1-\delta,\]
    where the probability is taken with respect to the training set $\cD_n=((X_i,Y_i))_{i\in[n]}$ drawn i.i.d.\ from $P$.
\end{theorem}
Since $K$-fold CV+ is strictly more conservative than $K$-fold cross-conformal (by Proposition~\ref{prop:CC_CV+}), the same bound holds as well for $K$-fold CV+ implemented with any regression algorithm $\cA$.

Before turning to the proof, we first briefly comment on the nature of this result. Note that we have not proved a bound showing that $\alpha_P(\cD_n)$ is approximately bounded by $\alpha$, with high probability---rather, we instead have $\alpha_P(\cD_n) \lessapprox 2\alpha$. This is because the marginal coverage guarantee for cross-conformal only establishes coverage at level (approximately) $1-2\alpha$, rather than $1-\alpha$ (see Theorems~\ref{thm:CC_smallK} and~\ref{thm:CC_bigK}), or in other words, $\E[\alpha_P(\cD_n)]$ is bounded (approximately) by $2\alpha$ rather than by $\alpha$. Therefore, we cannot hope to remove this factor of 2 from the conditional result without further assumptions.
\begin{proof}[Proof of Theorem~\ref{thm:training-conditional-cross-conformal}]
First, let $P_k$ be the distribution of $s((X,Y);\cD_{[n]\setminus I_k})$, where this distribution is calculated with respect to a new data point $(X,Y)\sim P$, after conditioning on the data $\cD_{[n]\setminus I_k}$ used for training in the $k$th fold. Then, conditional on $\cD_{[n]\setminus I_k}$, we have
\[(S_i)_{i\in I_k}\iidsim P_k,\]
where as before, for each $k\in[K]$ and each $i\in I_k$, we define $S_i = s((X_i,Y_i);\cD_{[n]\setminus I_k})$.
A standard concentration result (namely, the Dvoretzky--Kiefer--Wolfowitz inequality) implies that, for each $k\in[K]$,
\[\frac{1}{n/K}\sum_{i\in I_k}\ind{S_i \geq t} \geq \P_{S\sim P_k}(S \geq t\mid \cD_{[n]\backslash I_k}) - \sqrt{\frac{\log(K/\delta)}{2n/K}}\]
holds simultaneously for all $t\in\R$, with probability at least $1-\delta/K$. Note also that $\P_{S\sim P_k}(S \geq t\mid \cD_{[n]\backslash I_k}) = \P_{S\sim P_k}(S \geq t\mid \cD_n)$ (since $P_k$ depends on $\cD_n$ only through the subset of data, $\cD_{[n]\backslash I_k}$). Taking a union bound over all folds $k\in[K]$, we then see that
\[\frac{1}{n/K}\sum_{i\in I_k}\ind{S_i \geq t} \geq \P_{S\sim P_k}(S \geq t\mid \cD_n) - \sqrt{\frac{\log(K/\delta)}{2n/K}}\]
holds
simultaneously for all $k\in[K]$ and all $t\in\R$, with probability at least $1-\delta$. 

Next, define a p-value 
\[p^*_k(x,y) = \P_{S\sim P_k}\big(S \geq s((x,y);\cD_{[n]\setminus I_k}) \mid \cD_n\big).\]
Comparing to the p-value $p_k(x,y)$ defined in the proof of Theorem~\ref{thm:CC_smallK}, here we are comparing the score $s((x,y);\cD_{[n]\setminus I_k})$ for the data point $(x,y)$ against the true score distribution, $P_k$, while in 
the proof of Theorem~\ref{thm:CC_smallK}, we were comparing against the empirical distribution of scores obtained from data points in the $k$th fold $I_k$.

Note that, for $(X_{n+1},Y_{n+1})\sim P$, the value $p^*_k(X_{n+1},Y_{n+1})$ is indeed a valid p-value conditional on the training data, i.e., $\P(p^*_k(X_{n+1},Y_{n+1})\leq \tau\mid \cD_n)\leq \tau$ for all $\tau\in[0,1]$. This is true since, conditional on $\cD_n$, $s((X_{n+1},Y_{n+1});\cD_{[n]\setminus I_k})$ is a draw from the (random) distribution $P_k$. Since averaging p-values provides a quantity that is a p-value up to a factor of 2, then,
\[\P\left(\frac{\sum_{k=1}^K p^*_k(X_{n+1},Y_{n+1})}{K}\leq \tau\ \Bigg\vert\ \cD_n\right)\leq 2\tau\]
holds almost surely for any $\tau\in[0,1]$. Combining this with the concentration result above, then, with probability at least $1-\delta$,
\[\P\left(\frac{\sum_{k=1}^K \frac{1}{n/K}\sum_{i\in I_k}\ind{S_i \geq s((X_{n+1},Y_{n+1});\cD_{[n]\setminus I_k})}}{K}\leq \tau - \sqrt{\frac{\log(K/\delta)}{2n/K}} \ \Bigg\vert\ \cD_n\right)\leq 2\tau.\]
Choosing $\tau = \alpha +\sqrt{\frac{\log(K/\delta)}{2n/K}}$, we have shown that
\[ \P\left(\frac{1}{n}\sum_{k=1}^K\sum_{i\in I_k}\ind{S_i \geq s((X_{n+1},Y_{n+1});\cD_{[n]\setminus I_k})}\leq \alpha \ \Bigg\vert\ \cD_n\right)\leq 2\alpha + \sqrt{\frac{2\log(K/\delta)}{n/K}}\]
holds  with probability at least $1-\delta$. Finally, it follows directly from the definition of the $K$-fold cross-conformal method that
\[Y_{n+1}\not\in\cC(X_{n+1}) \ \Rightarrow \ \frac{1}{n}\sum_{k=1}^K\sum_{i\in I_k}\ind{S_i \geq s((X_{n+1},Y_{n+1});\cD_{[n]\setminus I_k})}\leq \alpha,\]
and therefore,
\begin{multline*}\alpha_P(\cD_n) = \P\big(Y_{n+1}\not\in\cC(X_{n+1})\mid \cD_n\big) \\\leq \P\left(\frac{1}{n}\sum_{k=1}^K\sum_{i\in I_k}\ind{S_i \geq s((X_{n+1},Y_{n+1});\cD_{[n]\setminus I_k})}\leq \alpha \ \Bigg\vert\ \cD_n\right), \end{multline*}
which completes the proof.
\end{proof}

In the proof of the result above, the key idea is that, when the fold $I_k$ has a large size $n/K$, the empirical distribution of scores $(S_i)_{i\in I_k}$ within this fold exhibits concentration. This is the reason that the term $n/K$ appears in the denominator, in the upper bound. If instead $n/K$ is small---in particular, with jackknife+ or leave-one-out cross-conformal we have $K=n$---then this bound no longer gives a meaningful result. 

The next theorem shows that, for jackknife+, training-conditional coverage can in fact fail to hold if we do not place any assumptions on the regression algorithm $\cA$. By Proposition~\ref{prop:CC_CV+}, this result therefore applies to leave-one-out cross-conformal as well: that is, without placing assumptions on the score function $s$, we cannot guarantee training-conditional coverage for leave-one-out cross-conformal.

\index{cross-validation!CV+}
\begin{theorem}[Failure of training-conditional coverage for jackknife+]\label{thm:training-conditional-jackknife+}
Let $P$ be any distribution on $\cX\times\cY$ such that the marginal $P_X$ is nonatomic while $P_Y$ has bounded support. Then there exists a symmetric regression algorithm $\cA$ such that, for jackknife+, the training-conditional noncoverage probability $\alpha_P(\cD_n)$ satisfies
    \[\P\left(\alpha_P(\cD_n) = 1\right)\geq \alpha - \bigo\left(\sqrt{\frac{\log n}{n}}\right),\]
    where the probability is taken with respect to the training set $\cD_n=((X_i,Y_i))_{i\in[n]}$ drawn i.i.d.\ from $P$.
\end{theorem}

The proof of this result is very similar to that of Theorem~\ref{thm:training-conditional-full-CP}, which is the analogous result proving that training-conditional coverage can fail for full conformal prediction.

\begin{proof}[Proof of Theorem~\ref{thm:training-conditional-jackknife+}] 
Since $P_Y$ has bounded support, we can fix a finite $B$ with $\P_P(|Y|< B) = 1$. Fix a large integer $N\approx \alpha n$, to be specified later. As in the proof of Theorem~\ref{thm:training-conditional-full-CP}, define a function $a$ with $a(X)\sim \textnormal{Unif}\{0,\dots,n-1\}$ when $X$ is drawn from $P_X$. 
Next, define the regression algorithm $\cA$: when the input is given by a training dataset $((x_1,y_1),\dots,(x_{n-1},y_{n-1}))$, $\cA$ returns the fitted model $\hf$ given by
\[\hf(x) = \begin{cases} 0, &\textnormal{ if }\textnormal{mod}\left(a(x) + \sum_{j=1}^{n-1}a(x_j), n\right)  < N,\\
2B, & \textnormal{ otherwise}.\end{cases}\]
In particular, when running jackknife+, we have
\[\hf_{-i}(X_i) = \begin{cases} 0, &\textnormal{ if }\textnormal{mod}\left(\sum_{j=1}^na(X_j), n\right)  < N,\\
2B, & \textnormal{ otherwise}.\end{cases}\]
Note that $\hf_{-i}(X_i)$ takes the same value for every $i\in[n]$, by construction.

Define events $\cE_{\textnormal{mod}}$ and $\cE_{\textnormal{unif}}$ exactly as in the proof of Theorem~\ref{thm:training-conditional-full-CP}: $\cE_{\textnormal{mod}}$ is the event that 
$\textnormal{mod}\left(\sum_{i=1}^n a(X_i)  , n \right) < N$, and 
$\cE_{\textnormal{unif}}$ is the event that
$\sum_{i=1}^n\ind{a(X_i)\in W_k} \geq (1-\alpha)(n+1)$ for all integers $k$ where the sets $W_k$ are as defined in the proof of Theorem~\ref{thm:training-conditional-full-CP}
(this event is illustrated in Figure~\ref{fig:E_unif_illustration}). On the event $\cE_{\textnormal{mod}}$, we then have
$\hf_{-i}(X_i) = 0 $ for all $i=1,\dots,n$,
and consequently,
\[S_i = |Y_i - \hf_{-i}(X_i)| = |Y_i - 0|< B\textnormal{ for all $i=1,\dots,n$}.\]
The jackknife+ prediction interval then satisfies
\begin{multline*}\cC(X_{n+1}) 
\subseteq \left(  -\quantile\left( ( -(\hf_{-i}(X_{n+1}) -B) )_{i\in[n]}  ; (1-\alpha)(1+1/n) \right) , \right. \\ \left.\quantile\left( ( \hf_{-i}(X_{n+1}) +B )_{i\in[n]}  ; (1-\alpha)(1+1/n) \right) \right). \end{multline*}
Next, for each $i$, we calculate
\[\hf_{-i}(X_{n+1}) = \begin{cases} 0, &\textnormal{ if }\textnormal{mod}\left(-a(X_i) + \sum_{j=1}^{n+1}a(X_j), n\right)  < N,\\
2B, & \textnormal{ otherwise}.\end{cases}\]
On the event $\cE_{\textnormal{unif}}$, then, we have $\hf_{-i}(X_{n+1}) = 2B$ for $\geq (1-\alpha)(n+1)$ many training points $i\in[n]$, and $\hf_{-i}(X_{n+1})=0$ for the others. In particular, this implies that the left endpoint of $\cC(X_{n+1})$ is given by
\[-\quantile\left( ( -(\hf_{-i}(X_{n+1}) -B) : i\in[n] )  ; (1-\alpha)(1+1/n) \right) = B,\]
and therefore, if $\cE_{\textnormal{mod}}$ and $\cE_{\textnormal{unif}}$ both occur,
\[\cC(X_{n+1}) \subseteq [B,\infty).\]
Since $P_Y$ is supported on $(-B,B)$, i.e., $\P_P(Y_{n+1}\in [B,\infty)) = 0$, we see that if $\cE_{\textnormal{mod}}$ and $\cE_{\textnormal{unif}}$ both occur, we therefore have $\alpha_P(\cD_n) = 1$. This proves that
\[\P_P(\alpha_P(\cD_n)=1) \geq \P_P(\cE_{\textnormal{mod}}\cap \cE_{\textnormal{unif}}).\]
The last steps of the proof consist of proving a lower bound on this probability for an appropriate choice of $N$, exactly as in 
the proof of Theorem~\ref{thm:training-conditional-full-CP}; we omit the details.
\end{proof}

\index{coverage!training-conditional|)}

\section{Algorithmic stability and prediction with the jackknife}\label{sec:alg-stability-jackknife}
\index{algorithmic stability|(}
\index{cross-validation!jackknife|(}

The results of Corollary~\ref{cor:CV_jack} show that CV+ and jackknife+ can offer coverage at level $\geq 1-2\alpha$, but this does not match the target level $1-\alpha$; in contrast, in practice, we generally see coverage around level $1-\alpha$ empirically. However, Section~\ref{sec:why-factor-of-2} explains why the factor of 2 in the theory cannot be removed. Strikingly, we saw in Example~\ref{example:jackknife_fails} that CV and jackknife (rather than their modified versions, CV+ and jackknife+) can even reach zero coverage.
The counterexample works by proposing a model $\hf=\cA(\cD_n)$ that when fitted to a training set $\cD_n$ of size $n$ behaves very differently from the model $\hf_{-i} = \cA(\cD_{[n]\setminus \{i\}})$, which was fitted to the same training set with a single data point removed. 
This can be described as an issue of \emph{instability} of the training algorithm: a small change to the training set creates a large change in the resulting fitted model.
See Figure~\ref{fig:stability} for an illustration.

\begin{figure}[t]
    \centering
    \includegraphics[width=0.75\textwidth]{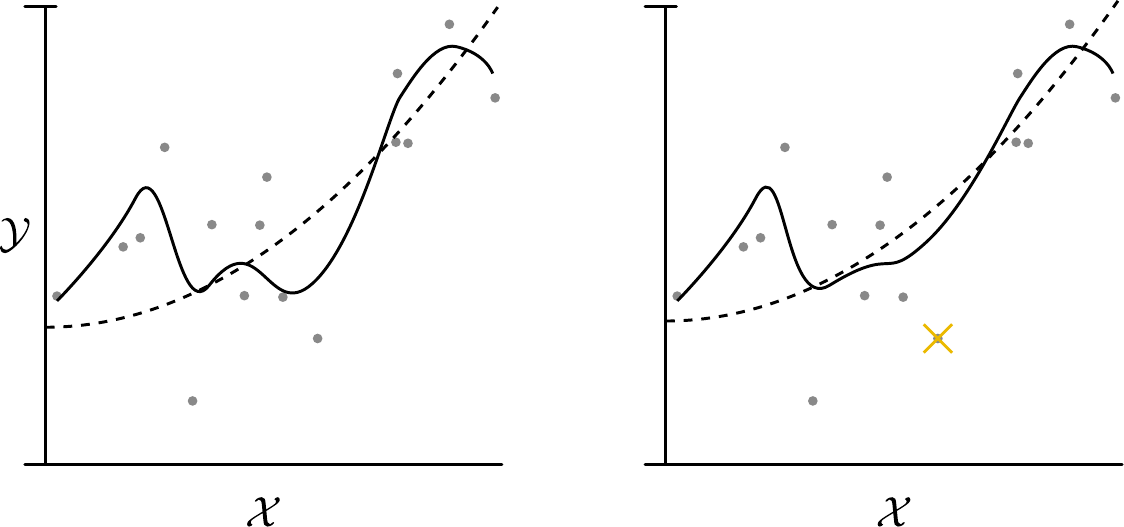}
    \caption{\textbf{An illustration of algorithmic stability in a regression.} In each plot, the gray dots represent data points $(X_i,Y_i)$, the dashed line represents the output of a quadratic regression algorithm $\cA_{\textnormal{quad}}$, and the solid line represents the output of a spline regression algorithm $\cA_{\textnormal{spline}}$. The left plot shows the fitted models on the full dataset $((X_j,Y_j))_{j\in[n]}$, while the right plot shows the fitted models trained on the dataset $((X_j,Y_j))_{j\in[n]\backslash\{i\}}$, where a single data point $(X_i,Y_i)$ (marked by a cross) has been removed.
    A stable algorithm is one for which the removal of a typical data point would not cause a large change in the predicted values returned by the fitted model (as in Definition~\ref{def:alg_stab_eps_delta}). In this figure, $\cA_{\textnormal{quad}}$ appears to satisfy this property, but $\cA_{\textnormal{spline}}$ does not.}
    \commentAlt{Two panels with scatterplots of the same dataset of $(X,Y)$ pairs. In the right panel one point is crossed out. Each panel has a dashed curve for quadratic regression and a solid curve for spline regression. The two solid curves are quite different.}
    \label{fig:stability}
\end{figure}

We will now study algorithmic stability in more detail, developing a formal definition of this property and analyzing how it can improve our theoretical guarantees for the cross-validation-based predictive inference methods.

\subsection{The algorithmic stability framework}\label{sec:alg_stab_framework}

At a high level, algorithmic stability simply means that small perturbations to the training dataset (e.g., adding or removing a single data point) do not lead to large changes in the fitted model. However, there are a multitude of different ways to formalize this intuition into a precise definition, and various choices can have extremely different implications for downstream results. Here we will focus on a version of the definition that is most relevant for analyzing the jackknife.

\begin{definition}[Algorithmic stability]\label{def:alg_stab_eps_delta}
An algorithm $\cA$ and distribution $P$ satisfy \emph{algorithmic stability}, with parameters $\epsilon,\delta\geq 0$ at sample size $n$, if for all $i\in[n]$ it holds that
\[\P\left( \big|\hf(X_{n+1}) - \hf_{-i}(X_{n+1})\big|\leq \epsilon \right) \geq 1- \delta,\]
with respect to data points $(X_1,Y_1),\dots,(X_{n+1},Y_{n+1})\iidsim P$,
where $\hf = \cA(((X_j,Y_j))_{j\in[n]})$ and $\hf_{-i} = \cA(((X_j,Y_j))_{j\in[n]\setminus\{i\}})$.
\end{definition}
Importantly, this condition requires the change in the prediction to be small when $\hf$ and $\hf_{-i}$ are evaluated on an independent test point $X_{n+1}$---a data point that does not appear in the training dataset for either $\hf$ or $\hf_{-i}$.
In contrast, requiring $|\hf(X_i)-\hf_{-i}(X_i)|$ to be small would be a much stronger assumption---it would imply that the prediction at $X_i$ is essentially the same regardless of whether $(X_i,Y_i)$ was used for training, which cannot hold for algorithms $\cA$ that exhibit substantial overfitting.

A number of common regression algorithms are known to satisfy this type of stability property. One example is $K$-nearest-neighbor regression, which satisfies Definition~\ref{def:alg_stab_eps_delta} with $\epsilon=0$ and $\delta = K/n$ (since the predictions $\hf(X_{n+1})$ and $\hf_{-i}(X_{n+1})$ are actually equal, unless $X_i$ is one of the $K$ nearest neighbors of the test point). Other examples include algorithms based on strongly convex optimization problems such as ridge regression, and algorithms based on ensembled or bootstrapped methods. See the bibliographic notes at the end of this chapter for more details.

\subsection{Guarantees for jackknife under stability}
We will now see that the algorithmic stability assumption ensures predictive coverage for the  jackknife method. Similar results can also be established for jackknife+, or for the $K$-fold versions of these methods, but we do not consider these extensions here.

Recall the jackknife (i.e., leave-one-out cross-validation) prediction interval, defined as
\[\cC(X_{n+1}) = \hf(X_{n+1}) \pm \hat{q}_{\textnormal{CV}} ,\]
where the margin of error is given by the quantile of the leave-one-out residuals,
\[\hat{q}_{\textnormal{CV}} = \quantile((S_1,\dots,S_n);1-\alpha)\textnormal{ for }S_i = |Y_i - \hf_{-i}(X_i)|.\]
Here $\hf$ is the model trained on the full training set, $\hf = \cA(((X_j,Y_j))_{j\in[n]})$, while for each $i\in[n]$, $\hf_{-i}$ is the leave-one-out model, $\hf_{-i} = \cA(((X_j,Y_j))_{j\in[n]\setminus\{i\}})$.

We will now see that algorithmic stability enables us to establish coverage at a level that is nearly the target level $1-\alpha$, for a slightly inflated version of the jackknife interval.

\begin{theorem}[Coverage guarantee for jackknife under algorithmic stability]\label{thm:jackknife-stability}
Suppose $(X_1,Y_1),\dots,(X_{n+1},Y_{n+1})$ are drawn i.i.d.\ from some distribution $P$, and assume that $\cA$ is a symmetric algorithm that satisfies algorithmic stability as in Definition~\ref{def:alg_stab_eps_delta} with parameters $\epsilon, \delta\geq 0$. Then the inflated jackknife prediction interval
\[\cC(X_{n+1}) = \hf(X_{n+1}) \pm (\hat{q}_{\textnormal{CV}} + \epsilon)\]
satisfies
\[\P(Y_{n+1}\in \cC(X_{n+1})) \geq 1- \alpha - 2\sqrt{\delta} - \frac{1}{n+1}.\]
\end{theorem}

The reason that jackknife can fail in the absence of algorithmic stability is that since $\hf$ and $\hf_{-i}$ are fitted using different training same sizes, the test error $|Y_{n+1}-\hf(X_{n+1})|$ and the leave-one-out residuals $S_i = |Y_i - \hf_{-i}(X_i)|$ may have very different distributions---and in particular, the test error $|Y_{n+1}-\hf(X_{n+1})|$ fails to be exchangeable with the $S_i$'s.
The idea of the proof is to instead construct a different vector of scores $\tilde{S}_1,\dots,\tilde{S}_n$, which \emph{are} exchangeable with the test error $|Y_{n+1}-\hf(X_{n+1})|$, and then show that the algorithmic stability assumption is sufficient to ensure $S_i\approx \tilde{S}_i$ for most training points $i\in[n]$.

\begin{proof}[Proof of Theorem~\ref{thm:jackknife-stability}]
For each $i\in[n+1]$, define a fitted model
\[\tilde{f}_{-i} = \cA(\cD_{[n+1]\setminus \{i\}}).\]
Note that for $i=n+1$, we simply have $\tilde{f}_{-(n+1)}=\hf$, while for $i\in[n]$, this model is trained on the dataset consisting of the $n-1$ training points when data point $i$ is held out, together with the test point $(X_{n+1},Y_{n+1})$ as an $n$th training point. Define also a modified leave-one-out residual
\[\tilde{S}_i = |Y_i - \tilde{f}_{-i}(X_i)|,\]
for each $i\in[n+1]$. Since $\cA$ is a symmetric algorithm, by symmetry of the construction we see that $\tilde{S}_1,\dots,\tilde{S}_{n+1}$ are exchangeable. We therefore have
\[\P\left(\tilde{S}_{n+1} \leq \quantile(\tilde{S}_1,\dots,\tilde{S}_{n+1};1-\alpha')\right)\geq 1-\alpha',\]
by Fact~\ref{fact:exchangeable-properties}\ref{fact:exchangeable-properties_part2}, for any $\alpha'\in[0,1]$.
By the Replacement Lemma (Lemma~\ref{lem:n+1-to-n-reduction}), together with the fact that $\tilde{S}_{n+1} = |Y_{n+1} - \hf(X_{n+1})|$ by definition, we then have
\[\P\left(|Y_{n+1} - \hf(X_{n+1})| \leq \quantile\big((\tilde{S}_i)_{i\in[n]};(1-\alpha')(1+1/n)\big)\right)\geq 1-\alpha'.\]
On the other hand, by definition of the inflated jackknife prediction interval $\cC(X_{n+1})$, we have
\[Y_{n+1}\in\cC(X_{n+1}) \ \Longleftrightarrow \ |Y_{n+1} - \hf(X_{n+1})| \leq \hat{q}_{\textnormal{CV}} + \epsilon.\]
Therefore,
\[\P(Y_{n+1}\in\cC(X_{n+1})) \geq 1-\alpha' - \P\left(\hat{q}_{\textnormal{CV}}+\epsilon < \quantile\left((\tilde{S}_i)_{i\in[n]}; (1-\alpha')(1+1/n)\right) \right).\]

Plugging in our definition of $\hat{q}_{\textnormal{CV}}$, our last step is to bound the probability
\[ \P\left(\quantile\left((S_i)_{i\in[n]}; 1-\alpha\right)  + \epsilon < \quantile\left((\tilde{S}_i)_{i\in[n]}; (1-\alpha')(1+1/n)\right) \right) ,\]
or equivalently,
\[\P(S_{(k)} + \epsilon < \tilde{S}_{(k')}),\]
where $k = \lceil (1-\alpha)n\rceil$ while $k' = \lceil (1-\alpha')(n+1)\rceil$, and where $S_{(1)}\leq \dots \leq S_{(n)}$ and $\tilde{S}_{(1)}\leq \dots \leq \tilde{S}_{(n)}$ are the order statistics of $S_1,\dots,S_n$ and of $\tilde{S}_1,\dots,\tilde{S}_n$, respectively.
But by definition of the order statistics, 
if $S_{(k)} + \epsilon < \tilde{S}_{(k')}$ holds then this implies that we must have $\tilde{S}_i > S_i + \epsilon$ for at least $k - k'+1$ indices $i\in[n]$. We now need to use the assumption of algorithmic stability: since $\cA$ is stable, we expect to obtain similar predictions $\hf_{-i}(X_i)\approx \tilde{f}_{-i}(X_i)$ (i.e., the fitted model is only slightly altered by including an additional data point $(X_{n+1},Y_{n+1})$ in the training process), and so we will have $S_i\approx \tilde{S}_i$ for most indices $i$. To formalize this, we calculate
\begin{align*}
    \P(S_{(k)} + \epsilon < \tilde{S}_{(k')})&\leq\P\left(\sum_{i\in[n]} \ind{\tilde{S}_i > S_i + \epsilon} \geq k - k'+1\right)\\
    &\leq \P\left(\sum_{i\in[n]} \ind{|\tilde{f}_{-i}(X_i) - \hf_{-i}(X_i)|>\epsilon} \geq k - k'+1\right)\\
    &\leq \frac{\E\left[\sum_{i\in[n]} \ind{|\tilde{f}_{-i}(X_i) - \hf_{-i}(X_i)|>\epsilon}\right]}{k-k'+1}\textnormal{\quad by Markov's inequality}\\
    &= \frac{\sum_{i\in[n]} \P(|\tilde{f}_{-i}(X_i) - \hf_{-i}(X_i)|>\epsilon)}{k-k'+1}\\
    &\leq \frac{n\delta}{k-k'+1},
\end{align*}
where the last step uses the algorithmic stability assumption (Definition~\ref{def:alg_stab_eps_delta}).
Finally, by construction, $k-k'+1 \geq (1-\alpha)n - (1-\alpha')(n+1)$, and so combining everything, we obtain
\[\P(Y_{n+1}\in\cC(X_{n+1})) \geq 1-\alpha' - \frac{n\delta}{(1-\alpha)n - (1-\alpha')(n+1)}.\]
Choosing $\alpha' = \alpha + \sqrt{\delta} + \frac{1}{n+1}$ completes the proof.
\end{proof}

\paragraph{Can algorithmic stability be certified?}
We have seen above that algorithmic stability leads to important downstream guarantees:  it implies that the jackknife offers guaranteed predictive coverage and provides confidence intervals for the risk of a trained model, without distributional assumptions. To make use of this in practice, however, we would need to know that jackknife is being applied to a stable algorithm.  Earlier,
we mentioned several examples of regression algorithms $\cA$ that are stable by construction---algorithms such as $K$-nearest neighbors, for which we can theoretically establish that stability holds---but such examples are rare and leave out many modern machine learning methods. 

We are therefore motivated to turn to the second option: is it possible to \emph{certify} that an algorithm $\cA$ is stable by testing it empirically, without needing to analyze $\cA$ theoretically? Interestingly, it turns out that in general this is not possible. 
While formalizing this type of hardness result is beyond the scope of this book (we refer the reader to the bibliographic notes for references), the implications for distribution-free inference are important to mention. These hardness results imply that a method such as jackknife, whose validity relies on a stability condition, cannot be viewed as a completely assumption-free method in general, because aside from special cases (i.e., simple algorithms such as $K$-nearest neighbors), the method's guarantees require untestable assumptions in order to hold.

\index{algorithmic stability|)}
\index{cross-validation!jackknife|)}
\index{cross-validation|)}

\section*{Bibliographic notes}
\addcontentsline{toc}{section}{\protect\numberline{}\textnormal{\hspace{-0.8cm}Bibliographic notes}}

Cross-validation style methods date back at least 50 years (see, e.g., \citet{stone1974cross}). Cross-conformal predictors were first introduced by \citet{vovk2015cross}, and studied further by \citet{vovk2018cross}; in particular, the latter work establishes a theoretical result proving marginal coverage at a level $\approx 1-2\alpha$ when the number of folds $K$ is not too large, which appears (in a slightly modified form) in Theorem~\ref{thm:CC_smallK} above. This result is based on the fact that an average of p-values is itself a valid p-value up to a factor of 2, as established by \citet{ruschendorf1982random} and studied more generally in \citet{vovk2020combining}.

The jackknife+ and CV+ algorithms are proposed by \citet{barber2021predictive}, where the relationship between jackknife+/CV+ and cross-conformal is also established, as in Proposition~\ref{prop:CC_CV+} above.  This paper also proves the coverage guarantee given in Theorem~\ref{thm:CC_bigK}, which covers the case where $K$ is large (e.g., $K=n$, for jackknife+) and the connection with intervals around the median prediction in Proposition~\ref{prop:CV+_contains_median}. Lemma~\ref{lem:A-rowsum}, which is a key ingredient in that proof, is related to Landau's theorem on tournaments \citep{landau1953dominance}.
The factor of 2 appearing in the coverage guarantee is explained in \citet{barber2021predictive}, which also formally establishes that a tighter bound is not possible; the intuition behind the argument for this result comes from the example illustrated  in~\eqref{eqn:jackknife_worstcase}, which is inspired by an earlier construction by \citet{vovk2015cross}

\citet{gupta2020nested} generalizes and reinterprets cross-conformal, as well as jackknife+ and CV+, through a formulation relating to nested families of sets. \citet{kim2020predictive} gives an version of jackknife+ adapted for efficient implementation in the setting of a bootstrapped or ensembled base algorithm; see also~\cite{linusson2020efficient} for a related method via out-of-bag calibration.
Training-conditional coverage results for jackknife+ (i.e., the hardness result of Theorem~\ref{thm:training-conditional-jackknife+}), and $K$-fold CV+ and $K$-fold cross-validation (i.e., the coverage guarantee of Theorem~\ref{thm:training-conditional-cross-conformal}), appear in \citet{bian2022training}.
The proof of Theorem~\ref{thm:training-conditional-cross-conformal} relies on the Dvoretzky--Kiefer--Wolfowitz inequality \citep{dvoretzky1956asymptotic}, a classical concentration inequality on the empirical CDF of i.i.d.\ data.

Algorithmic stability, as introduced in Section~\ref{sec:alg_stab_framework}, has been broadly studied in the fields of statistics and learning theory. \cite{bousquet2002stability} establishes connections between algorithmic stability and the notion of generalization.
The specific formulation of stability we use in this work, Definition~\ref{def:alg_stab_eps_delta}, is closely related to the \emph{hypothesis stability} condition in that work. Extensions of those results to the setting of randomized algorithms are developed by \cite{elisseeff2005stability}. \cite{shalev2010learnability} studies the connections between stability and \emph{learnability}, i.e., whether an algorithm is able to achieve some minimal risk on a supervised learning problem. 
The finite-sample stability-based guarantee for jackknife (Theorem~\ref{thm:jackknife-stability}) is developed by \cite{barber2021predictive}, along with improved predictive coverage guarantees for jackknife+. The proof of Theorem~\ref{thm:jackknife-stability} is related to the conformal jackknife method, a version of full conformal prediction that uses a leave-one-out regression based score function \cite[Section 2.2]{vovk2005algorithmic}. Earlier work by \cite{steinberger2018conditional} proves an asymptotic coverage guarantee for jackknife under stability. \cite{liang2023algorithmic,amann2023assumption} establish that algorithmic stability leads to training-conditional coverage for jackknife+ (and also for other related methods, namely, full conformal for the former paper, and CV and CV+ for the latter). \cite{ndiaye2022stable} also establishes computational shortcuts for full conformal prediction, under a stronger form of the stability assumption.

A related line of work studies a different inference question for cross-validation-based methods: for a loss function $\ell$, is the average leave-one-out loss $\frac{1}{n}\sum_{i=1}^n\ell(Y_i,\hf_{-i}(X_i))$ an accurate estimator of the true risk of the fitted models? \cite{austern2020asymptotics,bayle2020cross} establish asymptotic normality of this estimator. These results rely on a weaker formulation of stability, relying on stability of the loss $\ell(\cdot,\hf(\cdot))$ rather than the predictive model $\hf$ itself; other works such as \cite{kearns1997algorithmic,kale2011cross,kumar2013near} have also studied this type of loss-based stability.

Several types of algorithms have been proved to satisfy stability by their construction. As mentioned in Section~\ref{sec:alg_stab_framework}, nearest-neighbor type methods satisfy stability (see, e.g., \cite{rogers1978finite,devroye1979distribution}); results for ridge regression (and strongly convex regularizers more generally) are discussed by \cite{bousquet2002stability}. \cite{hardt2016train} establishes stability of stochastic gradient descent.
Ensembling or bootstrapping based methods have also been shown to satisfy stability properties, see, e.g., \cite{poggio2002bagging,chen2022debiased,soloff2023bagging}.
Results on the hardness of testing algorithmic stability, as mentioned at the end of Section~\ref{sec:alg-stability-jackknife}, appear in \cite{kim2021black,luo2024algorithmic}. 

\section*{Exercises}
\addcontentsline{toc}{section}{\protect\numberline{}\textnormal{\hspace{-0.8cm}Exercises}}
\begin{enumerate}[font=\bfseries, label={\thechapter.\arabic*}, labelsep=1em, itemsep=1em]
\item In the proof of Theorem~\ref{thm:jackknife-stability}, we defined models
    \[\tilde{f}_{-i} = \cA(\cD_{[n+1]\setminus\{i\}}),\]
    trained on leave-one-out versions of the combined training and test dataset $\cD_{n+1}$. 
    These models can be used for a variant of full conformal prediction, known as the conformal jackknife: in the context of the residual score, the conformal jackknife prediction set is defined as
    \[\cC(X_{n+1}) = \left\{y\in\cY : |y - \hf(X_{n+1})| \leq \hat{q}^y\right\},\]
    where $\hf = \cA(\cD_n)$, and where
    \[\hat{q}^y = \quantile\left( (|Y_i - \tilde{f}^y_{-i}(X_i|)_{i\in[n]};(1-\alpha)(1+1/n)\right),\]
    for
    \[\tilde{f}^y_{-i} = \cA(\cD^y_{[n+1]\setminus\{i\}}),\]
    where $\cD^y_{[n+1]\setminus\{i\}}$ denotes the dataset comprised of all training data except the $i$th point, along with the hypothesized test point $(X_{n+1},y)$.
    
    For this exercise, define a symmetric score function $s((x,y);\cD)$, such that the conformal jackknife is equivalent to running full conformal prediction (Algorithm~\ref{alg:full-cp}) with this score function. 
\item\label{exercise:jackknife_needs_inflation} In Theorem~\ref{thm:jackknife-stability}, we proved that assuming algorithmic stability ensures predictive coverage for the \emph{inflated} jackknife interval---but this does not necessarily guarantee that coverage will hold for the original (un-inflated) jackknife interval.
    In this exercise, we will show that algorithmic stability cannot ensure any guarantee of coverage for the original jackknife interval. Specifically, fix any arbitrarily small $\epsilon,\delta>0$, and  construct an explicit example of a data distribution $P$ and a regression algorithm $\cA$, such that algorithmic stability (Definition~\ref{def:alg_stab_eps_delta}) is satisfied, but the jackknife interval $\cC(X_{n+1})=\hf(X_{n+1})\pm \hat{q}_{\textnormal{CV}}$ has zero coverage, \[\P(Y_{n+1}\in\cC(X_{n+1}))=0,\]
    for data $(X_i,Y_i)\iidsim P$. 
\item Continuing from Exercise~\ref{exercise:jackknife_needs_inflation}, we will now show that the (un-inflated) jackknife interval does offer a coverage guarantee if we make an additional assumption. Specifically, suppose the distribution $P_{Y\mid X}$ (i.e., the conditional distribution of $Y\mid X$, when $(X,Y)\sim P$) has a bounded density $g(y\mid x)\in [0,B]$. Prove that the jackknife interval $\cC(X_{n+1})=\hf(X_{n+1})\pm \hat{q}_{\textnormal{CV}}$ satisfies
        \[\P(Y_{n+1}\in\cC(X_{n+1})) \geq 1-\alpha - 2\sqrt{\delta} - \frac{1}{n+1} - 2\epsilon B.\]
\item In an earlier exercise (Exercise~\ref{exercise:asymm_split_CP}), we constructed an asymmetric-tail version of split conformal prediction, with $\leq \alpha/2$ probability of miscoverage in each tail.
Define an analogous asymmetric-tail version of the jackknife+ method, and prove that it has $\geq 1-2\alpha$ marginal coverage.
\item Construct an example where cross-conformal prediction, with real-valued $Y$ and using the residual score, will often result in a prediction set $\cC(X_{n+1})$ that is not an interval.
(Here `using the residual score' means that $s((x,y);\cD) = |y-\hf(x)|$, where $\hf=\cA(\cD)$ is a regression model trained on the dataset $\cD$, for some choice of a symmetric regression algorithm $\cA$, as in the notation of Section~\ref{sec:CV+_jack+}.)
\end{enumerate}

\chapter{Weighted Variants of Conformal Prediction}
\label{chapter:weighted-conformal}
\index{weighted conformal prediction|(}

In the previous chapters, our conformal algorithms have treated all data points equally.
Often, however, some data points are more relevant than others. For example, data collected more recently is generally expected to be more closely related to the test point, and data points closer in feature space to the test point are expected to be more informative about the test point.

This chapter concerns weighted variants of conformal prediction, which allow us to give different data points a different level of influence.
We begin by defining a generic algorithm for weighted conformal prediction in Section~\ref{sec:weighted-algorithm}. After that, in the remaining sections, we show how to instantiate weighted conformal prediction for various settings:
\begin{itemize}
    \item In Section~\ref{sec:covariate_and_label_shift}, we consider known distribution shifts relating the training data distribution and the test point, such as covariate shifts and label shifts. For this setting, we use data-dependent weights to correct for this distribution shift.
    \item In Section~\ref{sec:localized}, we present the localized conformal prediction method, which aims to improve conditional coverage by placing higher weights on data points that are close to the test point $X_{n+1}$.
    \item In Section~\ref{sec:fixed-weights}, we consider arbitrary (and unknown) distribution shifts. We use fixed weights to prioritize more trusted data points, which leads to coverage closer to $1-\alpha$ when exchangeability is violated. 
    \item In Section~\ref{sec:general-outlook-permutations}, we consider a more general and unified view of conformal prediction, providing a framework for deriving conformal algorithms that generalize beyond the assumptions of exchangeability or of a symmetric score function. 
\end{itemize}
These setups all use variants of the same weighted conformal prediction algorithm, but with different choices of the weights, to achieve different effects.

\section{Weighted quantiles and the weighted conformal algorithm}
\label{sec:weighted-algorithm}
In this section, we will define a weighted version of the split and full conformal prediction algorithms. 
While conformal prediction computes $\hat{q}$ (or $\hat{q}^y$) as an unweighted quantile of the data points' scores, thus implicitly placing equal weight on each data point, we will now compute it as a weighted quantile, placing different weights on the various data points. 
This modification will allow us to provide coverage guarantees for conformal prediction under a variety of distribution shifts, as well as enabling approximate conditional guarantees, in later sections. In this section we restrict our attention to the mechanics of the procedure.

In unweighted full conformal prediction, we include the value $y$ in the prediction set whenever $S^y_{n+1} \leq \hat{q}^y$, where $\hat{q}^y$ is calculated as a quantile of the data points' scores.
By contrast, in weighted full conformal prediction, we take a weighted $1-\alpha$ quantile, \index{weighted quantile}
\begin{equation}
    \hat{q}^y_w = \quantile\left(\sum_{i=1}^{n+1} w_i \delta_{S_i^y} ; 1-\alpha\right),
\end{equation}
for a vector $w=(w_1,\dots,w_{n+1})$ of nonnegative, unit-sum weights. This weighted quantile is the $(1-\alpha)$-quantile of the distribution that places mass $w_i$ on the value $S_i^y$, for each $i$.
Points with higher weights will have a greater influence on the chosen quantile.
See Figure~\ref{fig:weighted-quantile} for an illustration, and Algorithm~\ref{alg:wtd-full-cp} for the weighted full conformal algorithm for a generic (and possibly data-dependent) vector of weights.
\begin{figure}[t]
    \centering
    \includegraphics[width=0.7\linewidth]{\diagramspath weighted-quantile.pdf}
    \caption{\textbf{An illustration of the weighted conformal quantile.} The horizontal axis shows the scores. The vertical axis shows the (unweighted) CDF of the scores on the left, and the weighted CDF on the right. The quantile level $1-\alpha$ is shown as a light gray line on both plots. The unweighted of quantile $\hat{q}^y$ is shown on the left, and the weighted quantile $\hat{q}_w^y$ is shown on the right.}
    \commentAlt{The left panel plots the CDF of $9$ scores. The right panel has the weighted CDF for the same scores, with vertical jumps determined by weights $w_i$. The quantile $\hat{q}^y$ is marked on the left, and weighted quantile $\hat{q}^y_w$ on the right.}
    \label{fig:weighted-quantile}
\end{figure}
\begin{algbox}[Weighted full conformal prediction]
    \label{alg:wtd-full-cp}
    \begin{enumerate}
        \item Input training data $(X_1, Y_1), ..., (X_n, Y_n)$, test point $X_{n+1}$, weights $w_1,\dots,w_{n+1}\geq 0$ with $\sum_i w_i =1$, target coverage level $1-\alpha$, conformal score function $s$.
                \item For each possible response value $y \in \cY$,\begin{enumerate}
            \item Compute $S_i^y = s((X_i, Y_i); \cD^y_{n+1})$ for all $i\in[n]$, and $S_{n+1}^y = s((X_{n+1}, y);\cD^y_{n+1})$.
            \item Compute the weighted conformal quantile $\hat{q}^y_w = \quantile\left(\sum_{i=1}^{n+1} w_i \delta_{S_i^y} ; 1-\alpha\right)$.
        \end{enumerate}
        \item Return the prediction set $\cC(X_{n+1}) = \{ y \in \cY : S_{n+1}^y \leq \hat{q}^y_w \}$.
    \end{enumerate}
\end{algbox}
Comparing to unweighted full conformal prediction (Algorithm~\ref{alg:full-cp}), we can see that this weighted algorithm  differs only in the calculation of the thresholds $\hat{q}^y_w$. 
Indeed, we can recover unweighted conformal prediction by choosing constant weights $w_i = \frac{1}{n+1}$ in the weighted full conformal algorithm.
The following calculation shows the equivalence: if we use constant weights in Algorithm~\ref{alg:wtd-full-cp}, then
\[\hat{q}^y_w =\quantile\left(\frac{1}{n+1}\sum_{i=1}^{n+1}  \delta_{S_i^y} ; 1-\alpha\right)= \quantile\left(S_1^y, \ldots, S_n^y, S_{n+1}^y ; 1-\alpha\right) \textnormal{ for $w_i= \frac{1}{n+1}$}.\]
But in fact, this results in an identical prediction set as Algorithm~\ref{alg:full-cp}, because by the Replacement Lemma (Lemma~\ref{lem:n+1-to-n-reduction}), for each $y\in\cY$,
\[S^y_{n+1}\leq \hat{q}^y = \quantile\left(S_1^y, \ldots, S_n^y ; (1-\alpha)(1+1/n)\right) \ \Longleftrightarrow \ S^y_{n+1}\leq \quantile\left(S_1^y, \ldots, S_n^y, S_{n+1}^y ; 1-\alpha\right).\]

We can also define a weighted version of split conformal prediction. As for the unweighted case, this can be viewed as a special case of weighted full conformal prediction.
\begin{algbox}[Weighted split conformal prediction]
    \label{alg:wtd-split-cp}
    \begin{enumerate}
                \item Input pretraining dataset $\cD_{\rm pre}$, calibration data  $(X_1, Y_1), ..., (X_n, Y_n)$, test point $X_{n+1}$, weights $w_1,\dots,w_{n+1}\geq 0$ with $\sum_i w_i =1$, target coverage level $1-\alpha$.
                \item Using the pretraining dataset $\cD_{\rm pre}$, construct a conformal score function $s:\cX\times\cY\to\R$.
                \item Compute the conformal scores on the calibration set, $S_i = s(X_i, Y_i)$ for $i\in[n]$.
                \item Compute the weighted conformal quantile $\hat{q}_w = \quantile\left(\sum_{i=1}^n w_i \delta_{S_i} + w_{n+1}\delta_{+\infty};1-\alpha\right)$.
                \item  Return the prediction set $\cC(X_{n+1}) = \{ y \in \cY : s(X_{n+1},y) \leq \hat{q}_w \}$.
    \end{enumerate}
\end{algbox}
Here $\delta_{+\infty}$ is a point mass at $+\infty$, and so
\[\quantile\left(\sum_{i=1}^n w_i \delta_{S_i} + w_{n+1}\delta_{+\infty};1-\alpha\right) = \inf\left\{t : \sum_{i=1}^n w_i \ind{S_i\leq t} \geq 1-\alpha\right\}.\]
As before, this is exactly equivalent to split conformal prediction (Algorithm~\ref{alg:split-cp}), if we choose constant weights $w_i=\frac{1}{n+1}$.

In the remainder of this chapter, we will give coverage guarantees for the weighted conformal algorithm and its extensions in various settings.

\section{Conformal prediction under covariate and label shifts}\label{sec:covariate_and_label_shift}
\index{covariate shift|(}

In this section, we see how the weighted conformal algorithm can be applied to handle covariate shift and label shift---that is, differences between the distribution of the training data and the test data, so that $(X_{n+1},Y_{n+1})$ is drawn from a covariate- or label-shifted distribution as compared to $(X_1,Y_1), \ldots, (X_n,Y_n)$. 
The strategy for constructing weights in these two settings will be a common one: setting the weights $w_i$ proportionally to the likelihood ratio relating these two distributions.

\subsection{Covariate shift}\label{sec:covariate-shift}
We will begin by using weighted conformal prediction to address covariate shift: the setting where the distribution of $X$ differs between the training and test data, but the distribution of $Y \mid X$ remains the same.
Covariate shift is a common model in practice; for example, it captures demographic shifts, where certain demographic covariates might be more common in the training population than in the test population, but the relationship between the covariates and the response remains fixed.

We assume that $(X_1,Y_1), \ldots, (X_n,Y_n) \iidsim P_X \times P_{Y \mid X}$ and that $(X_{n+1},Y_{n+1})\sim Q_X \times P_{Y \mid X}$ independently, for different covariate distributions $P_X$ and $Q_X$. Here, the notation $P_X \times P_{Y \mid X}$ means that $X$ follows distribution $P_X$ and $Y \mid X$ follows distribution $P_{Y \mid X}$ (and, $Q_X \times P_{Y \mid X}$ is defined similarly).
The key insight is that if we know the \emph{likelihood ratio} $\frac{\mathsf{d}Q_X}{\mathsf{d}P_X}: \cX \to [0,\infty) $, we can construct weights that compensate for the distribution shift to give us coverage. Formally, $\frac{\mathsf{d}Q_X}{\mathsf{d}P_X}$ denotes the Radon--Nikodym derivative relating the distribution $Q_X$ of the test data point's feature vector to the distribution $P_X$ of the training features. For example, if $P_X$ and $Q_X$ both have densities, then $\frac{\mathsf{d}Q_X}{\mathsf{d}P_X}$ can simply be taken to be the ratio of the densities.

The strategy for setting the weights is straightforward: we set $w_i \propto \frac{\mathsf{d}Q_X}{\mathsf{d}P_X}(X_i)$. 
To ensure that the weights sum to $1$, we set
\begin{equation}
    \label{eq:weights-covariate-shift}
    w_i = \frac{\frac{\mathsf{d}Q_X}{\mathsf{d}P_X}(X_i)}{\sum\limits_{j=1}^{n+1}\frac{\mathsf{d}Q_X}{\mathsf{d}P_X}(X_j)}
\end{equation}
for all $i \in [n+1]$.
We will now see that 
this choice of weights gives us a coverage guarantee.
\begin{theorem}[Coverage guarantee under covariate shift]
    \label{thm:covariate-shift}
    Let $(X_1,Y_1), \ldots, (X_n,Y_n) \iidsim P_X \times P_{Y \mid X}$ and $(X_{n+1},Y_{n+1})\sim Q_X \times P_{Y \mid X}$ independently.
    Furthermore, assume that $Q_X$ is absolutely continuous with respect to $P_X$, and so $\frac{\mathsf{d}Q_X}{\mathsf{d}P_X}(x) < \infty$ for all $x \in \cX$. Fix any symmetric score function $s$, and define the prediction set
    \[\cC(X_{n+1}) = \left\{y : S^y_{n+1} \leq \hat{q}^y_w\right\}\textnormal{ where }\hat{q}^y_w = \quantile\left(\sum_{i=1}^{n+1} w_i \delta_{S_i^y};1-\alpha\right),\]
    where the weights $w_i$ are defined as in~\eqref{eq:weights-covariate-shift}.
    Then,
    \begin{equation}
        \P(Y_{n+1} \in \cC(X_{n+1})) \geq 1-\alpha.
    \end{equation}
\end{theorem}
In other words, coverage for the covariate shift setting is obtained by running Algorithm~\ref{alg:wtd-full-cp} with weights $w_i$ defined as in~\eqref{eq:weights-covariate-shift}. This result is a corollary to the more general result of Theorem~\ref{thm:train-test-shift-weighted} (see Section~\ref{sec:train-test-shifts-weighted-lr} below), so we do not present the proof.

\begin{figure}[t]
    \centering
    \includegraphics[width=0.8\linewidth]{\diagramspath covariate-shift-example.pdf}
    \caption{\textbf{Visualization of a covariate shift} in two groups on the left, and the resulting unweighted and weighted quantiles on the right. On the left, there are two groups, $X=1$ and $X=2$, with different frequencies under $P$ and $Q$. 
    On the right, we show the density of the scores conditional on $X=1$ and $X=2$, in the same colors.
    We additionally show the density of the score under the mixture distributions $P$ (dotted) and $Q$ (dashed). If we compute a quantile using training data drawn from $P$, then the unweighted quantile $\hat{q}^y$ reflects the fact that $P$ places high weight on $X=1$ (which is associated with lower scores), and is consequently much lower than the weighted quantile $\hat{q}^y_w$, which reflects the fact that $Q$ places higher weight on $X=2$ (which is associated with higher score values).}
    \commentAlt{A histogram shows distributions $P$ and $Q$ with different proportions of $X=1$ or $X=2$. A plot shows two densities of scores when $X=1$ (mostly smaller scores) or when $X=2$ (larger scores), and the different mixture distributions under $P$ or $Q$.}
    \label{fig:covariate-shift}
\end{figure}

\textbf{Why does reweighting give coverage?} 
First, it is worth considering why unweighted conformal prediction does not provide coverage under covariate shift. 
Imagine that the covariate takes only two values, i.e., $\cX = \{1, 2\}$, and that the scores in the subpopulation where $X=1$ tend to be much smaller than in the subpopulation where $X=2$. 
If $\P_{P_X}(X=1) > \P_{Q_X}(X=1)$, then
$X=1$ is over-represented in the training data (relative to the test point's distribution), and the conformal quantile will be biased downwards; see Figure~\ref{fig:covariate-shift} for a visualization. 
The fundamental problem is that the training and test scores $S_1,\dots,S_n,S_{n+1}$ are no longer exchangeable---the test point's score $S_{n+1}$ no longer behaves like a draw from the empirical distribution of $S_1, \ldots, S_{n+1}$, and so calculating an unweighted quantile cannot be expected to achieve the right coverage level.
Weighted conformal prediction works because the distribution of $S_{n+1}$ can be explicitly represented with a weighted empirical distribution instead---we will describe this in detail soon in Section~\ref{sec:train-test-shifts-weighted-lr}.
\index{covariate shift|)}

\subsection{Label shift}
\label{sec:label-shift}
\index{label shift|(}

In label shift, the distribution of $Y$ changes and the distribution of $X \mid Y$ is held fixed.
Label shift is common in many classification problems.
Our strategies for handling covariate shift and label shift are almost identical, with one key difference: the weights are now a function of $Y$ as opposed to $X$.

The setup is analogous to the covariate shift procedure.
We assume that $(X_1,Y_1), \ldots, (X_n,Y_n) \iidsim P_{X \mid Y} \times P_Y$ and that $(X_{n+1},Y_{n+1})\sim P_{X \mid Y} \times Q_Y$ independently, for a different label distribution $Q_Y$. Here, the notation $P_{X \mid Y} \times P_Y$ means that $Y$ follows distribution $P_Y$ and then $X \mid Y$ follows distribution $P_{X \mid Y}$, and similarly, for $P_{X \mid Y} \times Q_Y$.
Then, we define the likelihood ratio $\frac{\mathsf{d}Q_Y}{\mathsf{d}P_Y}: \cY \to \R$.

To run Algorithm~\ref{alg:wtd-full-cp}, we need to make a slight modification in the calculation of the weighted quantile, since the weights are now a function of $Y$ (which, for the test point, is unknown), rather than of $X$ (where the test point value $X_{n+1}$ is known). 
In particular, define
\begin{equation}
    \label{eq:weights-label-shift}
    w_i^y = \frac{\frac{\mathsf{d}Q_Y}{\mathsf{d}P_Y}(Y_i)}{\sum\limits_{j=1}^{n}\frac{\mathsf{d}Q_Y}{\mathsf{d}P_Y}(Y_j) + \frac{\mathsf{d}Q_Y}{\mathsf{d}P_Y}(y)} \textnormal{ for } i \in [n] \quad \textnormal{ and } \quad w_{n+1}^y = \frac{\frac{\mathsf{d}Q_Y}{\mathsf{d}P_Y}(y)}{\sum\limits_{j=1}^{n}\frac{\mathsf{d}Q_Y}{\mathsf{d}P_Y}(Y_j) + \frac{\mathsf{d}Q_Y}{\mathsf{d}P_Y}(y)}.
\end{equation}
For the label shift setting, we will use these weights $w_i^y$ for computing the conformal quantile---that is, this procedure is a slight variant of Algorithm~\ref{alg:wtd-full-cp}, where the weights $w_i$ did not depend on the hypothesized test point $y$. 

Our next result verifies a coverage guarantee for this setting. 
\begin{theorem}[Coverage guarantee under label shift]
    \label{thm:label-shift}   
    Let $(X_1,Y_1), \ldots, (X_n,Y_n) \iidsim P_{X \mid Y} \times Q_Y$ and $(X_{n+1},Y_{n+1})\sim P_{X \mid Y} \times Q_Y$ independently.
    Furthermore, assume that $Q_Y$ is absolutely continuous with respect to $P_Y$, and so $\frac{\mathsf{d}Q_Y}{\mathsf{d}P_Y}(y) < \infty$ for all $y \in \cY$.
    Fix any symmetric score function $s$, and define the prediction set
    \[\cC(X_{n+1}) = \left\{y : S^y_{n+1} \leq \hat{q}^y_w\right\}\textnormal{ where }\hat{q}^y_w = \quantile\left(\sum_{i=1}^{n+1} w_i^y \delta_{S_i^y};1-\alpha\right),\]
    where the weights $w_i^y$ are defined as in~\eqref{eq:weights-label-shift}.
    Then,
    \begin{equation}
        \P(Y_{n+1} \in \cC(X_{n+1})) \geq 1-\alpha.
    \end{equation}
\end{theorem}
This result is again a corollary of Theorem~\ref{eq:weights-train-test-shift}, so we do not prove it here.

\subsection{General shifts between the training distribution and test distribution}
\label{sec:train-test-shifts-weighted-lr}

We now turn to a more general setting that unifies covariate shift and label shift, and clarifies the core statistical logic at play.

Let $(X_1,Y_1), \ldots, (X_n,Y_n) \iidsim P$ and $(X_{n+1},Y_{n+1})\sim Q$ independently, for a different distribution $Q$, where we assume that $Q$ is absolutely continuous with respect to $P$ so that the likelihood ratio (i.e., the Radon--Nikodym derivative) $\frac{\mathsf{d}Q}{\mathsf{d}P}: \cX\times\cY \to [0,\infty) $ is well-defined.
Paralleling the label shift section, we will set the weights as
\begin{equation}
    \label{eq:weights-train-test-shift}
    w_i^y = \frac{\frac{\mathsf{d}Q}{\mathsf{d}P}(X_i, Y_i)}{\sum\limits_{j=1}^{n}\frac{\mathsf{d}Q}{\mathsf{d}P}(X_j, Y_j) + \frac{\mathsf{d}Q}{\mathsf{d}P}(X_{n+1}, y)} \textnormal{ for } i \in [n], \quad w_{n+1}^y = \frac{\frac{\mathsf{d}Q}{\mathsf{d}P}(X_{n+1}, y)}{\sum\limits_{j=1}^{n}\frac{\mathsf{d}Q}{\mathsf{d}P}(X_j, Y_j) + \frac{\mathsf{d}Q}{\mathsf{d}P}(X_{n+1}, y)}.
\end{equation}
As a side note, because these weights are self-normalized, we only need to know $\frac{\mathsf{d}Q}{\mathsf{d}P}$ up to a constant of proportionality---that is, if $\frac{\mathsf{d}Q}{\mathsf{d}P}(x,y) \propto w(x,y)$ for some function $w:\cX\times\cY\to\R$, it suffices to know the function $w$, without calculating its normalizing constant.
(This was equally true in Sections~\ref{sec:covariate-shift} and~\ref{sec:label-shift} for the covariate-shift and label-shift settings.)

With the weights defined, we are now ready to state the general result.
\begin{theorem}[Coverage guarantee under distribution shift]
    \label{thm:train-test-shift-weighted}
    Let $(X_1,Y_1), \ldots, (X_n,Y_n) \iidsim P$ and $(X_{n+1},Y_{n+1})\sim Q$ independently.
    Furthermore, assume that $Q$ is absolutely continuous with respect to $P$, and so $\frac{\mathsf{d}Q}{\mathsf{d}P}(x,y) < \infty$ for all $(x,y) \in \cX\times\cY$.
Fix any symmetric score function $s$, and define the prediction set
    \[\cC(X_{n+1}) = \left\{y : S^y_{n+1} \leq \hat{q}^y_w\right\}\textnormal{ where }\hat{q}^y_w = \quantile\left(\sum_{i=1}^{n+1} w_i^y \delta_{S_i^y};1-\alpha\right),\]
    where the weights $w_i^y$ are defined as in~\eqref{eq:weights-train-test-shift}.
    Then,
    \begin{equation}
        \P(Y_{n+1} \in \cC(X_{n+1})) \geq 1-\alpha.
    \end{equation}
\end{theorem}
The validity of the covariate shift algorithm in Section~\ref{sec:covariate-shift} and the label shift algorithm in Section~\ref{sec:label-shift} are special cases of this result. In this general algorithm, we require knowledge of the distribution shift $w(x,y) = \frac{\mathsf{d}Q}{\mathsf{d}P}(x,y)$ in order to run this algorithm, and so from a methodological perspective this algorithm is useful only in settings where we are able to estimate $w$.
Covariate shift and label shift are the primary instances where this is feasible:
in these two special cases, this ratio depends only on $x$ (in the case of covariate shift), or only on $y$ (in the case of label shift), meaning that we do not need to assume any knowledge of the dependence between $X$ and $Y$ in order to estimate the weight function $w$. 

Our proof of this result will generalize the ideas developed in Section~\ref{sec:conformal-conditioning-bag}, where we proved the coverage guarantee for unweighted full conformal prediction (in a setting without distribution shift) by considering the empirical distribution $\widehat{P}_{n+1} = \frac{1}{n+1}\sum_{i=1}^{n+1}\delta_{(X_i,Y_i)}$ of the training and test data.
The argument in Section~\ref{sec:conformal-conditioning-bag} showed that, in the setting of exchangeability, the conditional distribution of the test point $(X_{n+1},Y_{n+1})$, if we condition on $\widehat{P}_{n+1}$, is equal to the empirical distribution $\widehat{P}_{n+1}$ itself. This is no longer true in the distribution shift setting---if $P\neq Q$, then the test point is not equally likely to be equal to each one of the points in the dataset. The following proposition shows that the conditional distribution is instead given by a reweighted version of the empirical distribution.

\begin{proposition}[Weighted exchangeability and the empirical distribution]
    \label{prop:weighted-exchangeability}
    Let $Z_1, \ldots, Z_n \iidsim P$ and $Z_{n+1}\sim Q$ independently, for some distributions $P,Q$ on $\cZ$.
     Furthermore, assume that $Q$ is absolutely continuous with respect to $P$, and so $\frac{\mathsf{d}Q}{\mathsf{d}P}(z) < \infty$ for all $z \in \cZ$.
    Let
    \[\widehat{P}_{n+1} = \frac{1}{n+1}\sum_{i=1}^{n+1}\delta_{Z_i}\]
    denote the empirical distribution of the $n+1$ data points. Then the conditional distribution of $Z_{n+1}$, given the empirical distribution $\widehat{P}_{n+1}$, is given by
    \begin{equation}    
        Z_{n+1} \mid \widehat{P}_{n+1} \sim \sum\limits_{i=1}^{n+1}  w_i\cdot \delta_{Z_i},
    \end{equation}
    where 
    \[w_i = \frac{\frac{\mathsf{d}Q}{\mathsf{d}P}(Z_i) }{\sum_{j=1}^{n+1} \frac{\mathsf{d}Q}{\mathsf{d}P}(Z_j) }, \ i\in[n+1].\]
\end{proposition}
With this result in place, we now turn to the proof of the theorem.
\begin{proof}[Proof of Theorem~\ref{thm:train-test-shift-weighted}]
By construction of the prediction set, we have
\[Y_{n+1}\in\cC(X_{n+1})\Longleftrightarrow S_{n+1}\leq \hat{q}_w^{Y_{n+1}} = \quantile\left(\sum_{i=1}^{n+1} w_i^{Y_{n+1}} \delta_{S_i};1-\alpha\right),\]
where $w_i^{Y_{n+1}}$ is defined as in~\eqref{eq:weights-train-test-shift} (with $y=Y_{n+1}$).
Now write $Z_i=(X_i,Y_i)$ for each $i\in[n+1]$, and 
define
\[w_i = \frac{\frac{\mathsf{d}Q}{\mathsf{d}P}(Z_i) }{\sum_{j=1}^{n+1} \frac{\mathsf{d}Q}{\mathsf{d}P}(Z_j) }\]
(as in the statement of Proposition~\ref{prop:weighted-exchangeability}). Note that
\[w_i = w_i^{Y_{n+1}}\]
by definition. Therefore, 
\[Y_{n+1}\in\cC(X_{n+1})\Longleftrightarrow S_{n+1}\leq \quantile\left(\sum_{i=1}^{n+1} w_i \delta_{S_i};1-\alpha\right).\]
We can therefore write 
\begin{multline*}\P(Y_{n+1}\in\cC(X_{n+1})) 
= \P\left(S_{n+1}\leq \quantile\left(\sum_{i=1}^{n+1} w_i \delta_{S_i};1-\alpha\right)\right)\\
 = \E\left[ \P\left(S_{n+1}\leq \quantile\left(\sum_{i=1}^{n+1} w_i \delta_{S_i};1-\alpha\right) \,\middle|\, \widehat{P}_{n+1}\right)\right],\end{multline*}
by the tower law. Next, since the score function $s$ is assumed to be symmetric, note that $s(\cdot;\cD_{n+1})$ depends on the dataset $\cD_{n+1}$ only via the empirical distribution $\widehat{P}_{n+1}$. To emphasize this point, we will use the notation $s(\cdot ;\widehat{P}_{n+1})$ in place of $s(\cdot;\cD_{n+1})$ throughout the remainder of this proof. Since $S_i = s(Z_i;\widehat{P}_{n+1})$ for each $i$ by construction, we then have
\begin{multline*} \P(Y_{n+1}\in\cC(X_{n+1})) \\=  \E\left[ \P\left(s(Z_{n+1};\widehat{P}_{n+1})\leq \quantile\left(\sum_{i=1}^{n+1} w_i \delta_{s(Z_i;\widehat{P}_{n+1})};1-\alpha\right) \,\middle|\, \widehat{P}_{n+1}\right)\right].\end{multline*}

Now we will consider the conditional probability on the right-hand side. First, observe that the quantile,
\[\quantile\left(\sum_{i=1}^{n+1} w_i \delta_{s(Z_i;\widehat{P}_{n+1})};1-\alpha\right),\]
can be expressed as a function of the empirical distribution $\widehat{P}_{n+1}$, i.e., it does not depend on the order of the data points---this is because the reweighted empirical distribution $\sum_{i=1}^{n+1} w_i \delta_{Z_i}$ is itself $\widehat{P}_{n+1}$-measurable (in particular, this is implied by Proposition~\ref{prop:weighted-exchangeability}). Using the conditional distribution of $(X_{n+1},Y_{n+1})$ as derived in Proposition~\ref{prop:weighted-exchangeability}, then, we have
\begin{multline}\label{eqn:weighted_quantile_eqn_for_citing_ref}\P\left(s(Z_{n+1};\widehat{P}_{n+1})\leq \quantile\left(\sum_{i=1}^{n+1} w_i \delta_{s(Z_i;\widehat{P}_{n+1})};1-\alpha\right) \,\middle|\, \widehat{P}_{n+1}\right)\\ = \sum_{j=1}^{n+1} w_j \cdot \ind{s(Z_j;\widehat{P}_{n+1})\leq \quantile\left(\sum_{i=1}^{n+1} w_i \delta_{s(Z_i;\widehat{P}_{n+1})};1-\alpha\right)}.\end{multline}
But this must deterministically be $\geq 1-\alpha$, by definition of the weighted quantile---this is due to the following fact (which we can view as a weighted version of Fact~\ref{fact:conversion-quantiles-cdfs}\ref{fact:conversion-quantiles-cdfs_part3}):
\begin{fact}\label{fact:conversion-quantiles-cdfs-weighted}
    Fix weights $w_1,\dots,w_k\geq 0$ with $\sum_{i=1}^kw_i=1$. Then for any $z\in\R^k$ and any $\tau\in[0,1]$,
    \[\sum_{i=1}^kw_i\ind{z_i\leq \quantile\left(\sum_{j=1}^kw_j\delta_{z_j};\tau\right)} \geq \tau.\]
\end{fact}
Therefore,
we have
\[\P(Y_{n+1}\in\cC(X_{n+1})) \geq \E[1-\alpha] = 1-\alpha,\]
which completes the proof.
\end{proof}

Finally, to complete this section, we prove Proposition~\ref{prop:weighted-exchangeability}.

\begin{proof}[Proof of Proposition~\ref{prop:weighted-exchangeability}]
In this proof, we will use the following two facts, which are measure-theoretic in nature, and we state without proof.
The first is that proving the desired statement is equivalent to verifying that, for all (measurable) sets $A$ and $B$, it holds that
\begin{equation}
    \label{eq:equivalence-weighted-distribution}
    \P(Z_{n+1}\in A, \widehat{P}_{n+1}\in B) = \E\left[\sum_{i=1}^{n+1} w_i \ind{Z_i \in A}\ind{\widehat{P}_{n+1}\in B} \right],
\end{equation}
since $\sum_{i=1}^{n+1} w_i \ind{Z_i \in A}$ is the claimed conditional probability of the event $Z_{n+1}\in A$ given $\widehat{P}_{n+1}$.
The second fact is that for any (measurable) function $h : \cZ^{n+1} \to \R$, we have
\begin{multline}
    \label{eq:radon-nikodym-expectation}
    \E_{(Z_1, \ldots, Z_{n+1}) \sim P^n\times Q}\left[ h(Z_1, \ldots, Z_{n+1}) \right] \\= \E_{(Z_1, \ldots, Z_{n+1}) \sim P^{n+1}}\left[ \frac{\mathsf{d}Q}{\mathsf{d}P}(Z_{n+1}) \cdot h(Z_1, \ldots, Z_{n+1}) \right],
\end{multline}
as long as these expected values are defined,
which is a property of the Radon--Nikodym derivative.

Now, we move to verifying~\eqref{eq:equivalence-weighted-distribution}.
We have  
\begin{align}
    &\P_{(Z_1, \ldots, Z_{n+1}) \sim P^n\times Q}(Z_{n+1} \in A, \widehat{P}_{n+1} \in B)\\
    &=  \E_{(Z_1, \ldots, Z_{n+1}) \sim P^n\times Q}\left[ \ind{Z_{n+1} \in A} \ind{\widehat{P}_{n+1} \in B} \right] \\
    &=  \E_{(Z_1, \ldots, Z_n, Z_{n+1}) \sim P^{n+1}}\left[ \frac{\mathsf{d}Q}{\mathsf{d}P}(Z_{n+1})\cdot  \ind{Z_{n+1} \in A} \ind{\widehat{P}_{n+1} \in B} \right],
\end{align}
where the last equality above follows by~\eqref{eq:radon-nikodym-expectation}.
Next, we redefine the weights as  
\[w_i = \begin{cases} \frac{\frac{\mathsf{d}Q}{\mathsf{d}P}(Z_i)}{\sum\limits_{j=1}^{n+1} \frac{\mathsf{d}Q}{\mathsf{d}P}(Z_j)}, & \sum\limits_{j=1}^{n+1} \frac{\mathsf{d}Q}{\mathsf{d}P}(Z_j)>0 ,\\ 0, & \sum\limits_{j=1}^{n+1} \frac{\mathsf{d}Q}{\mathsf{d}P}(Z_j)=0.\end{cases} \]
(Note that, if we draw $(Z_1,\dots,Z_{n+1})\sim P^{n+1}$, it is no longer the case that $\frac{\mathsf{d}Q}{\mathsf{d}P}(Z_{n+1})$ must be positive almost surely---and therefore the event
\[\sum\limits_{j=1}^{n+1} \frac{\mathsf{d}Q}{\mathsf{d}P}(Z_j)>0\]
could potentially have probability $<1$. However, if this event does hold, then these new weights are equivalent to those defined in the statement of the proposition---and also, on this event, we have $\sum_{j=1}^{n+1}w_j = 1$.) Then
\[\frac{\mathsf{d}Q}{\mathsf{d}P}(Z_{n+1}) = w_{n+1} \cdot \sum_{j=1}^{n+1}\frac{\mathsf{d}Q}{\mathsf{d}P}(Z_j)
\]
holds
almost surely under $(Z_1, \ldots, Z_n, Z_{n+1}) \sim P^{n+1}$, and so
\begin{align}
    &\P_{(Z_1, \ldots, Z_{n+1}) \sim P^n\times Q}(Z_{n+1} \in A, \widehat{P}_{n+1} \in B)\\
    &=  \E_{(Z_1, \ldots, Z_n, Z_{n+1}) \sim P^{n+1}}\left[ w_{n+1} \cdot \sum_{j=1}^{n+1}\frac{\mathsf{d}Q}{\mathsf{d}P}(Z_j)\cdot  \ind{Z_{n+1} \in A} \ind{\widehat{P}_{n+1} \in B} \right]\\
    &=  \sum_{j=1}^{n+1} \E_{(Z_1, \ldots, Z_n, Z_{n+1}) \sim P^{n+1}}\left[ w_{n+1} \cdot \frac{\mathsf{d}Q}{\mathsf{d}P}(Z_j)\cdot  \ind{Z_{n+1} \in A} \ind{\widehat{P}_{n+1} \in B} \right].
\end{align}
Because the expectation above is taken with respect to i.i.d.\ random variables, we can reindex these random variables however we like.
Swapping indexes $j$ and $n+1$ in the $j$th term, for each $j$, yields that the above expression is equal to
\begin{align}
    = & \sum_{j=1}^{n+1} \E_{(Z_1, \ldots, Z_n, Z_{n+1}) \sim P^{n+1}}\left[ w_j \cdot \frac{\mathsf{d}Q}{\mathsf{d}P}(Z_{n+1})\cdot  \ind{Z_j \in A} \ind{\widehat{P}_{n+1} \in B} \right]\\
    = & \sum_{j=1}^{n+1} \E_{(Z_1, \ldots, Z_n, Z_{n+1}) \sim P^n\times Q}\left[ w_j \cdot   \ind{Z_j \in A} \ind{\widehat{P}_{n+1} \in B} \right],
\end{align}
where the last step holds by applying~\eqref{eq:radon-nikodym-expectation}. 
This is equivalent to the right-hand side of~\eqref{eq:equivalence-weighted-distribution}, as desired.
\end{proof}
\index{label shift|)}

\subsection{Comparing distribution shift and conditional coverage}
\index{covariate shift}
\index{label shift}
\index{coverage!test-conditional}
\index{coverage!label-conditional}

Robustness to distribution shift is closely related to conditional coverage.
In fact, requiring test-conditional coverage (i.e., coverage conditional on the test feature $X_{n+1}$) is strictly stronger
than requiring coverage relative to a covariate shift:
test-conditional coverage implies coverage with respect to \emph{any} 
covariate shift (i.e., for any $Q_X$).
Similarly, label-conditional coverage implies 
coverage with respect to \emph{any} 
label shift (i.e., for any $Q_Y$).

However, conditional coverage guarantees are harder to obtain.
For example, if label-conditional coverage holds (that is, $\P(Y_{n+1}\in\cC(X_{n+1})\mid Y_{n+1}=y)\geq 1-\alpha$ for each possible label $y$) then it implies coverage for any label-shifted distribution---but these prediction sets will be very large unless we have a large number of observations in each category (i.e., $\sum_{i\in[n]} \ind{Y_i=y}$ is large, for each $y\in\cY$). In contrast, the method studied in Theorem~\ref{thm:label-shift} gives a marginal guarantee relative to a particular label-shifted distribution $P_{X \mid Y} \times Q_Y$ and could therefore produce more informative prediction sets. Similarly, test-conditional coverage implies that coverage is maintained under all possible covariate shifts, but in Section~\ref{sec:covariate-shift} we study the more limited goal of maintaining coverage under a specific, known covariate shift.

\section{Localized conformal prediction}
\label{sec:localized}
\index{localized conformal prediction|(}
\index{coverage!test-conditional}

We next use reweighting for the purpose of improving test-conditional coverage. Suppose we wish to give confidence sets that approximately give conditional coverage for a test point $X_{n+1}$, as we have discussed in earlier chapters. Recall that, in Chapter~\ref{chapter:conditional} we established that it is impossible in general to ensure test-conditional coverage, but as in Chapter~\ref{chapter:model-based}, certain score functions (such as the CQR score) lead to approximate conditional coverage under additional assumptions. In this section, we pursue a complementary strategy based on weighting.

To see why reweighting may be useful, notice that conformal prediction chooses the cutoff $\hat{q}$ (or $\hat{q}^y)$ by looking at the conformal score for all training points, even those that are far away from the test point $X_{n+1}$. This can lead to a failure of test-conditional coverage if the distribution of scores is very different across different regions of $\cX$. If we instead wish to improve test-conditional coverage, we might choose to give more weight to data points close to the test point $X_{n+1}$. This is the key idea behind localized conformal prediction.

To make this concrete, suppose we have a function $H : \cX \times \cX \to \R_{\ge 0}$ that is capturing the similarity between two feature vectors---we will refer to $H$ as the \emph{localization kernel}. For example, when $\cX = \R^d$, we could take $H(x, x')  = \exp\{-\lVert x - x' \rVert^2_2 / 2h^2\}$, where $h>0$ is a bandwidth parameter. An initial idea is to run Algorithm~\ref{alg:wtd-full-cp} with the weight $w_i$ on data point $(X_i,Y_i)$ chosen to be proportional to $H(X_i,X_{n+1})$: 
\begin{equation}\label{eqn:baseLCP}\cC(X_{n+1}) = \left\{y\in\cY : S_{n+1}^y\leq \quantile\left(\sum_{i=1}^{n+1}w_i \delta_{S_i^y};1-\alpha\right)\right\} \textnormal{ where }w_i = \frac{H(X_i,X_{n+1})}{\sum_{j=1}^{n+1}H(X_j,X_{n+1})}\end{equation}
That is, when running conformal prediction, we weight data points according to their distance from the test point (see Figure~\ref{fig:lcp} for an illustration). 
This initial approach is intuitively appealing and in practice would likely improve test-conditional coverage, but the distribution-free marginal coverage guarantee no longer holds for this weighted algorithm.
We will remedy this next.
 
\subsection{The localized conformal prediction algorithm}

\begin{figure}[t]
    \centering
    \includegraphics[width=0.75\linewidth]{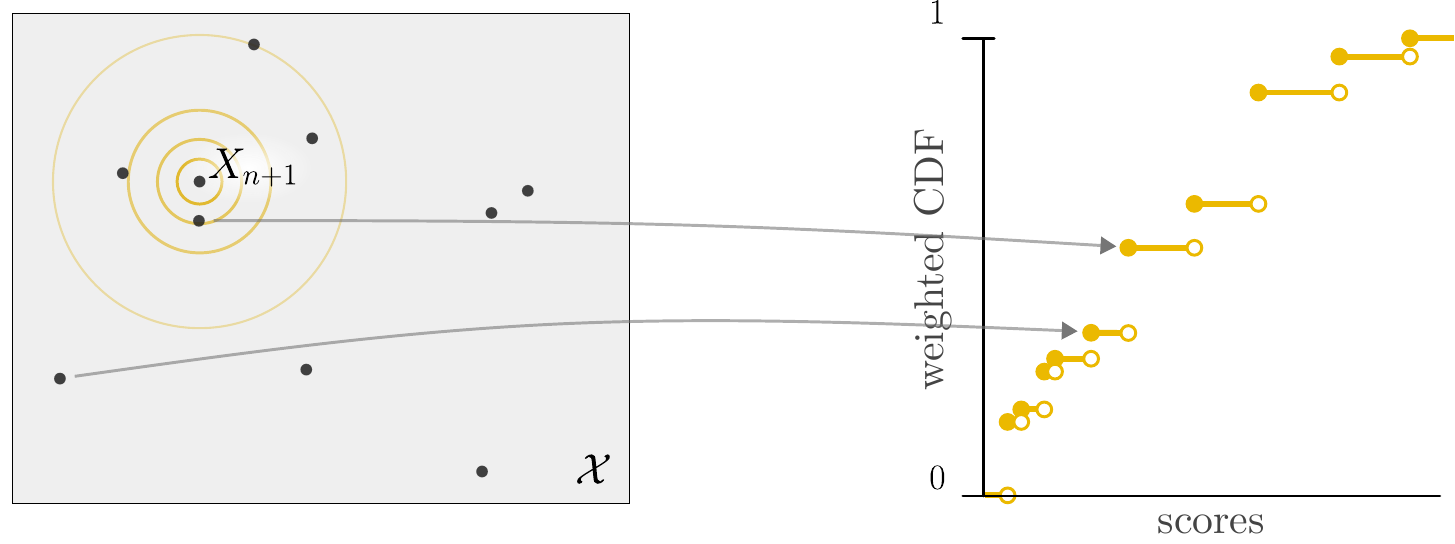}
    \caption{\textbf{Constructing the weighted CDF used for localized conformal prediction.} The left panel shows the covariate space $\cX$, where the test point $X_{n+1}$ is labeled, and each additional dot represents a calibration point $X_1, \ldots, X_n$. Concentric circles centered at $X_{n+1}$ illustrate the level sets of the localization kernel $H$, which assigns higher weights to calibration points closer to $X_{n+1}$. The right panel shows the resulting weighted CDF of the scores: each jump corresponds to a calibration point, with the size of the jump reflecting its proximity to $X_{n+1}$. The two arrows indicate the relationship between a calibration point $X_i$ and its corresponding weight in the CDF: the closer point produces a larger jump, and the farther one, a smaller jump.}
    \commentAlt{A scatterplot has a point labeled $X_{n+1}$, with concentric circles showing distance to its neighbors. Arrows connect the plot to a weighted CDF of scores, showing that a point near $X_{n+1}$ has a large weight, and a farther point has small weight.}
    \label{fig:lcp}
\end{figure}

To endow the localized conformal approach with a distribution-free marginal coverage guarantee, it can be modified with a recalibration step, as stated next. 
\begin{algbox}[Localized conformal prediction]
    \label{alg:lcp}
    \begin{enumerate}
        \item Input training data $(X_1, Y_1), ..., (X_n, Y_n)$, test point $X_{n+1}$, localization kernel $H$, target coverage level $1-\alpha$, conformal score function $s$.
                \item For each possible response value $y \in \cY$,\begin{enumerate}
            \item Compute $S_i^y = s((X_i, Y_i); \cD^y_{n+1})$ for all $i\in[n]$, and $S_{n+1}^y = s((X_{n+1}, y);\cD^y_{n+1})$.
            \item Compute $w_{i,j} = H(X_j, X_i) / \sum_{j'=1}^{n+1} H(X_{j'}, X_i)$ for each $i,j\in[n+1]$.
            \item For $i=1,\dots,n+1$ define
            \[\tilde{S}^y_i = \sum_{j=1}^{n+1} w_{i,j} \ind{S^y_j<S^y_i}.\]
            \item Compute the conformal quantile $\tilde{q}^y = \quantile\left(\tilde{S}^y_1,\dots,\tilde{S}^y_{n+1};1-\alpha\right)$.
        \end{enumerate}
        \item Return the prediction set $\cC(X_{n+1}) = \{ y \in \cY : \tilde{S}_{n+1}^y \leq \tilde{q}^y \}$.
    \end{enumerate}
\end{algbox}
How does this algorithm relate to our original formulation of the localized approach in~\eqref{eqn:baseLCP}? To see the connection, first observe that $w_{n+1,i}$ (the weight defined here in Algorithm~\ref{alg:lcp}) is equal to $w_i$ (the weight defined in~\eqref{eqn:baseLCP}), by construction. When running Algorithm~\ref{alg:lcp}, therefore, the prediction set $\cC(X_{n+1})$ contains all values $y$ for which $\tilde{S}_{n+1}^y\leq \tilde{q}^y$, or equivalently,
\[\sum_{i=1}^{n+1} w_i \ind{S^y_i < S_{n+1}^y}  \leq \tilde{q}^y,\]
while the set defined in~\eqref{eqn:baseLCP} instead contains $y$ whenever
\[S_{n+1}^y \leq \quantile\left(\sum_{i=1}^{n+1}w_i \delta_{S_i^y};1-\alpha\right) \Longleftrightarrow \sum_{i=1}^{n+1} w_i\ind{S_i^y < S_{n+1}^y} < 1-\alpha.\]
Thus, the localized conformal prediction method defined in Algorithm~\ref{alg:lcp} is \emph{exactly the same} as the intuitive approach outlined in~\eqref{eqn:baseLCP}, except with a data-dependent value $\tilde{q}^y$ replacing the usual threshold $1-\alpha$. In other words, the localized conformal prediction method can essentially be interpreted as running the intuitive weighted conformal method suggested in~\eqref{eqn:baseLCP}, but with a modified value of $\alpha$ (which varies with $y$).
 
This modification is needed in order to ensure valid marginal coverage for localized conformal prediction---and in particular, the following result verifies that this method is a special case of full conformal prediction:
\begin{proposition}[Relating localized conformal prediction to full conformal]\label{prop:LCP_is_fullCP}
The prediction set defined in Algorithm~\ref{alg:lcp} is equivalent to the full conformal prediction set, when full conformal is run with the score function
\[\tilde{s}((x,y);\cD) = \sum_{j=1}^m \frac{H(x_j,x)}{\sum_{j'=1}^m H(x_{j'},x)}\ind{s((x_j,y_j);\cD)<s((x,y);\cD)},\]
for any $(x,y)$ and any dataset $\cD = ((x_1,y_1),\dots,(x_m,y_m))$. 
\end{proposition}
This result follows directly from the construction of Algorithm~\ref{alg:lcp}. Note that the function $\tilde{s}((x,y);\cD)$ is symmetric in $\cD$ as long as the original score function $s((x,y);\cD)$ is itself symmetric in $\cD$. Consequently, by Theorem~\ref{thm:full-conformal}, marginal coverage must hold for localized conformal prediction as long as we assume that the data is exchangeable:
\begin{corollary}[Marginal coverage of localized conformal prediction]\label{cor:LCP_coverage}
    Suppose that $(X_1,Y_1),\dots,(X_{n+1},Y_{n+1})$ are exchangeable and that $s$ is a symmetric score function. Then the prediction set $\cC(X_{n+1})$ defined in Algorithm~\ref{alg:lcp} satisfies the marginal coverage guarantee,
    \[\P(Y_{n+1}\in\cC(X_{n+1}))\geq 1-\alpha.\]
\end{corollary}
\index{localized conformal prediction|)}

\subsection{The randomly-localized conformal prediction algorithm}
\index{localized conformal prediction!randomly-localized|(}
Formally, the localized conformal prediction algorithm offers a marginal coverage guarantee. In practice, since the weights focus attention near the test point $X_{n+1}$, the method will likely yield approximate test-conditional coverage as well---but establishing theoretical guarantees for approximate conditional coverage is challenging and requires strong technical conditions. Instead, we next turn to the \emph{randomly-localized conformal prediction} algorithm, which comes with a simple and explicit guarantee of approximate conditional coverage.

We continue as before with a localization kernel function $H : \cX \times \cX \to \R_{\ge 0}$, and now we additionally assume $H(x, \cdot)$ defines a density with respect to some measure $\nu$ for each $x \in \cX$:
\begin{equation}
    \int_{\cX} H(x, x') \;\mathsf{d}\nu(x') = 1.
\end{equation}

\begin{algbox}[Randomly-localized conformal prediction]
    \label{alg:rlcp}
    \begin{enumerate}
        \item Input training data $(X_1, Y_1), ..., (X_n, Y_n)$, test point $X_{n+1}$, localization kernel $H$, target coverage level $1-\alpha$, conformal score function $s$.
        \item Sample $\tilde{X}_{n+1}$ from the distribution with density $H(X_{n+1},\cdot)$.
                \item For each possible response value $y \in \cY$,\begin{enumerate}
            \item Compute $S_i^y = s((X_i, Y_i); \cD^y_{n+1})$ for all $i\in[n]$, and $S_{n+1}^y = s((X_{n+1}, y);\cD^y_{n+1})$.
            \item Compute $w_i = H(X_i,\tilde X_{n+1}) / \sum_{j=1}^{n+1} H(X_j,\tilde X_{n+1})$ for each $i\in[n+1]$.
            \item Compute the conformal quantile $\hat{q}^y = \quantile\left(\sum_{i=1}^{n+1} w_i \delta_{S^y_i};1-\alpha\right)$.
        \end{enumerate}
        \item Return the prediction set $\cC(X_{n+1}) = \{ y \in \cY : S_{n+1}^y \leq \hat{q}^y \}$.
    \end{enumerate}
\end{algbox}
In other words, we are simply running the weighted conformal prediction (Algorithm~\ref{alg:wtd-full-cp}), with the weights $w_i$ defined using the kernel $H$ centered at the \emph{random} sample $\tilde{X}_{n+1}$.

This algorithm comes with a guarantee that can be viewed as an approximation of conditional coverage.
\begin{theorem}[Coverage guarantee for randomly-localized conformal prediction]
\label{thm:rlcp-coverage}
Suppose $(X_1,Y_1),\dots,(X_{n+1}, Y_{n+1})$ are i.i.d.\ and $s$ is a symmetric score function, and
let $\cC(X_{n+1})$ be the output of Algorithm~\ref{alg:rlcp}. Then,
    \begin{equation}
        \P\left(Y_{n+1} \in \cC(X_{n+1}) \mid \tilde X_{n+1}\right) \ge 1-\alpha, \textnormal{ \ almost surely.}
    \end{equation}
\end{theorem} \index{coverage!test-conditional}
Note that this also implies the usual marginal coverage guarantee, by the tower property. We can interpret the coverage property guaranteed in Theorem~\ref{thm:rlcp-coverage} as follows. Since we choose $H(X_{n+1}, \cdot)$ to be localized near $X_{n+1}$, the distribution of $X_{n+1} \mid \tilde X_{n+1}$ is concentrated near $\tilde X_{n+1}$. Thus, the result of Theorem~\ref{thm:rlcp-coverage} is essentially a guarantee of coverage for a test point $X_{n+1}$ sampled from a distribution that is concentrated near $\tilde{X}_{n+1}$---a relaxation of conditioning on the exact value of the test point $X_{n+1}$.

\begin{proof}[Proof of Theorem~\ref{thm:rlcp-coverage}]
First, note that $(X_1, Y_1),\dots,(X_n, Y_n)$ are independent of $(X_{n+1}, Y_{n+1}, \tilde{X}_{n+1})$, by construction. Next, by the definition of $\tilde X_{n+1}$, the joint distribution of $(X_{n+1}, Y_{n+1}, \tilde{X}_{n+1})$ is defined by
\begin{align}
    & X_{n+1} \sim P_X \\
    & Y_{n+1} \mid X_{n+1} \sim P_{Y \mid X} \\
    & \tilde X_{n+1} \mid (X_{n+1}, Y_{n+1}) \sim H(X_{n+1}, \cdot).
\end{align}
We then have that
\begin{equation}
    (X_{n+1}, Y_{n+1}) \mid \tilde X_{n+1} \sim (P_X \circ H(\cdot, \tilde X_{n+1})) \times P_{Y \mid X},
\end{equation}
where $P_X \circ H(\cdot, \tilde X_{n+1})$ denotes a reweighted distribution on $\cX$, defined by
\[\frac{\mathsf{d}(P_X \circ H(\cdot, \tilde X_{n+1}))}{\mathsf{d}P_X}(x) \propto H(x,\tilde X_{n+1}).\]
Thus, conditional on $\tilde X_{n+1}$, this is an instance of covariate shift---the distribution of $X_{n+1}$ is no longer $P_X$, but the conditional distribution of $Y_{n+1}\mid X_{n+1}$ is unchanged.

We now note that Algorithm~\ref{alg:rlcp} is the same as the weighted conformal algorithm with covariate shift. To see this, observe that the weights satisfy
\begin{equation}
    w_i \propto H(X_i, \tilde X_{n+1}) \propto \frac{\mathsf{d}(P_X \circ H(\cdot, \tilde X_{n+1}))}{\mathsf{d} P_X}(X_i),
\end{equation}
exactly as in~\eqref{eq:weights-covariate-shift} (where the test distribution $Q_X$ of the covariates is now replaced by $P_X \circ H(\cdot, \tilde X_{n+1})$---note that $\tilde X_{n+1}$ is being treated as fixed, since the proposition seeks to establish coverage conditional on $\tilde X_{n+1}$). Thus, the validity of weighted conformal for covariate shift (Theorem~\ref{thm:covariate-shift}) implies the desired claim.
\end{proof}
\index{localized conformal prediction!randomly-localized|)}

\section{Fixed-weight conformal prediction}
\label{sec:fixed-weights}
\index{nonexchangeability|(}

Next, we focus on conformal prediction with \emph{fixed} weights---ones that don't depend on $X$ or $Y$---in order to be more robust to deviations from exchangeability.
For example, in a time series, we may want to give recent data higher weights relative to data from long ago, since we expect that the distribution of recent points is closer to that of the test point.
In this section, we will, a-priori, assign each of our data points a weight, $w_i \geq 0$ for all $i \in [n+1]$, with larger weights indicating greater relevance.
When we run Algorithm~\ref{alg:wtd-full-cp} with this non-data-dependent vector of weights, we call it \emph{fixed-weight} conformal prediction. 

Incorporating weights allows us to achieve a coverage guarantee for exchangeable datasets, and bound the loss of coverage for datasets that violate the exchangeability assumption. The key idea is that if we choose the weights well, the weighted procedure will be more robust to certain violations of exchangeability. We state this formally next. Throughout this section, we will use the notation $Z_i = (X_i,Y_i)$ for the data points, and will define scores $S^y_i$ as usual.
\begin{theorem}[Coverage guarantee for fixed-weight conformal prediction]\label{thm:fixed-weight-conformal}
    Let weights $w_1,\dots,w_{n+1}\geq 0$ be fixed, with $\sum_i w_i = 1$ and $w_{n+1}\geq w_i$ for all $i$. Let $Z_1,\dots,Z_{n+1}$ be random variables with any joint distribution. 
    Fix any symmetric score function $s$, and define the prediction set
    \[\cC(X_{n+1}) = \left\{y : S^y_{n+1} \leq \hat{q}^y_w\right\}\textnormal{ where }\hat{q}^y_w = \quantile\left(\sum_{i=1}^{n+1} w_i \delta_{S_i^y};1-\alpha\right).\]
    Then,
    \[\P(Y_{n+1}\in\cC(X_{n+1})) \geq 1 - \alpha - \sum_{i=1}^n w_i \cdot {\rm d}_{\rm TV}\left((Z_1,\dots,Z_{n+1}), (Z_1,\dots,Z_{i-1},Z_{n+1},Z_{i+1},\dots,Z_n,Z_i)\right),\]
    where ${\rm d}_{\rm TV}$ denotes the total variation distance between distributions. \index{total variation distance}
\end{theorem}
This guarantee gives a bound on the loss of coverage for conformal prediction run on a general dataset, regardless of whether it satisfies the exchangeability assumption.
The bound depends on two quantities: the TV distance of the dataset when the $i$th data point is swapped with the test point (which we can view as a measure of the degree of non-exchangeability), and the weights.
When data points $Z_i$ that are differently distributed from the test point $Z_{n+1}$ are given high weight, this can hurt the coverage.
On the other hand, if they have a low weight, then coverage is (approximately) preserved.

Two important implications of the theorem are as follows:
\begin{enumerate}
    \item If the dataset is exchangeable, the fixed-weight procedure has coverage at least $1-\alpha$ (since the total variation distance appearing in the bound is equal to zero in that case).
    \item If the dataset is not exchangeable, the loss of coverage can be mitigated by choosing the weights strategically---e.g., by placing higher weights on more recently gathered data points.
\end{enumerate}
\begin{proof}[Proof of Theorem~\ref{thm:fixed-weight-conformal}]
    First we define some notation. Recall from our notation in Chapter~\ref{chapter:conformal-exchangeability} that we write $S_i = s((X_i,Y_i);\cD_{n+1})$ to denote the score of data point $Z_i=(X_i,Y_i)$ when the algorithm is trained on the entire dataset $\cD_{n+1} = (Z_1,\dots,Z_{n+1})$. We will write $S=(S_1,\dots,S_{n+1})$ to denote this vector of scores, and will write
    \[S^i = (S_1,\dots,S_{i-1},S_{n+1},S_{i+1},\dots,S_n,S_i)\]
    to denote the vector $S$ with $i$th and $(n+1)$st entries swapped, for each $i\in[n]$.

    Our first step is to observe that the total variation distance between the vectors $S$ and $S^i$, is bounded by the total variation distance on the corresponding vectors of data points. Specifically, define a function $h: (\cX\times\cY)^{n+1}\to\R^{n+1}$ as
    \[h(z_1,\dots,z_{n+1}) = \left( s(z_1;(z_1,\dots,z_{n+1})), \dots, s(z_{n+1} ; (z_1,\dots,z_{n+1}))\right),\]
    i.e., the function that computes a score for each $z_i$ (relative to the dataset comprised of $z_1,\dots,z_{n+1}$). 
    Then by definition we have
    $S = h(Z_1,\dots,Z_{n+1})$. Moreover, the assumption of symmetry on $s$ ensures that
    $S^i = h(Z_1,\dots,Z_{i-1},Z_{n+1},Z_{i+1},\dots,Z_n,Z_i)$---and therefore, by the Data Processing Inequality,
    \begin{multline*}{\rm d}_{\rm TV}(S,S^i)= {\rm d}_{\rm TV}\left(h(Z_1,\dots,Z_{n+1}), h(Z_1,\dots,Z_{i-1},Z_{n+1},Z_{i+1},\dots,Z_n,Z_i)\right)\\ \leq {\rm d}_{\rm TV}\left((Z_1,\dots,Z_{n+1}), (Z_1,\dots,Z_{i-1},Z_{n+1},Z_{i+1},\dots,Z_n,Z_i)\right).\end{multline*}
    From this point on, then, we will now aim to show that
    \begin{equation}\label{eqn:dtv_scores_or_data}
    \P(Y_{n+1}\in\cC(X_{n+1})) \geq 1 - \alpha - \sum_{i=1}^n w_i \cdot {\rm d}_{\rm TV}(S,S^i).\end{equation}

    Defining $S^{n+1}=S$ for notational convenience, we then calculate
    \begin{align*}
        &\P(Y_{n+1}\in\cC(X_{n+1}))
        = \P\left(S_{n+1} \leq \quantile\left(\sum_{j=1}^{n+1} w_j \cdot \delta_{S_j};1-\alpha\right)\right)\\
        &=\sum_{i=1}^{n+1} w_i \cdot \P\left(S_{n+1} \leq \quantile\left(\sum_{j=1}^{n+1} w_j \cdot \delta_{S_j};1-\alpha\right)\right)\\
        &\geq \sum_{i=1}^{n+1} w_i  \left[\P\left((S^i)_{n+1} \leq \quantile\left(\sum_{j=1}^{n+1} w_j \cdot \delta_{(S^i)_j};1-\alpha\right)\right) - {\rm d}_{\rm TV}(S,S^i)\right]\\
        &=\E\left[ \sum_{i=1}^{n+1} w_i\cdot   \ind{(S^i)_{n+1} \leq \quantile\left(\sum_{j=1}^{n+1} w_j \cdot \delta_{(S^i)_j};1-\alpha\right)}\right] - \sum_{i=1}^n w_i \cdot {\rm d}_{\rm TV}(S,S^i),
    \end{align*}
    where the first step simply applies the by definition of $\cC(X_{n+1})$ for the weighted full conformal method, the second step uses the fact that $\sum_i w_i=1$, and the inequality step uses the definition of total variation distance.

Our last task is to simplify the remaining expected value.
Below we will prove that the following claim holds deterministically:
\begin{equation}\label{eqn:claim_for_S_vs_Si}S_i \leq \quantile\left(\sum_{j=1}^{n+1} w_j \cdot \delta_{S_j};1-\alpha\right)\Longrightarrow
(S^i)_{n+1} \leq \quantile\left(\sum_{j=1}^{n+1} w_j \cdot \delta_{(S^i)_j};1-\alpha\right) 
 .
\end{equation}
Assuming that this holds, we then have
\begin{align*}
        &\P(Y_{n+1}\in\cC(X_{n+1}))\\
        &\geq\E\left[ \sum_{i=1}^{n+1} w_i\cdot   \ind{S_i \leq \quantile\left(\sum_{j=1}^{n+1} w_j \cdot \delta_{S_j};1-\alpha\right)}\right] - \sum_{i=1}^n w_i \cdot {\rm d}_{\rm TV}(S,S^i)\\
        &\geq 1 - \alpha- \sum_{i=1}^n w_i \cdot {\rm d}_{\rm TV}(S,S^i),
\end{align*}
where the last step follows since
\begin{equation}\label{eqn:weighted_quantile_eqn_for_citing_ref_2}\sum_{i=1}^{n+1} w_i\cdot   \ind{S_i \leq \quantile\left(\sum_{j=1}^{n+1} w_j \cdot \delta_{S_j};1-\alpha\right)}\geq 1-\alpha\end{equation}
holds deterministically by Fact~\ref{fact:conversion-quantiles-cdfs-weighted}.

To complete the proof, we now need to verify~\eqref{eqn:claim_for_S_vs_Si}. For $i=n+1$ the claim holds trivially (since $S=S^{n+1}$), so we restrict our attention to the case $i\in[n]$.
We will need the following result, which state without proof---we can view this result as a
    generalized version of the Replacement Lemma (recall Lemma~\ref{lem:n+1-to-n-reduction}):
    \begin{lemma}
        \label{lem:generalized-replacement-lemma}
        Let $P_0,P_1$ be any distributions on $\R$, and let
        \[P = (1 - \epsilon)\cdot P_0 + \epsilon \cdot P_1\]
        be their mixture, for some $\epsilon\in [0,1]$. 
        Then for any $x\in\R$, and any $\tau\in[0,1]$,
        \[
            x \leq \quantile(P;\tau) \Longrightarrow x \leq \quantile\left( (1-\epsilon)\cdot P_0 + \epsilon\cdot \delta_x ; \tau\right) .
        \]
    \end{lemma}
Now define $\epsilon = w_{n+1} - w_i$ (which is nonnegative by assumption), $P_1 = \delta_{S_{n+1}}$, and 
\[P_0 = \frac{\sum_{j\in[n],j\neq i} w_j\cdot \delta_{S_j} + w_i \cdot (\delta_{S_i} + \delta_{S_{n+1}})}{1 - \epsilon}.\]
Then by Lemma~\ref{lem:generalized-replacement-lemma} applied with $\tau=1-\alpha$,
\begin{equation}\label{eqn:claim_for_S_vs_Si_v2}S_i \leq \quantile((1-\epsilon)P_0 + \epsilon P_1;1-\alpha) \Longrightarrow S_i \leq \quantile((1-\epsilon)P_0 + \epsilon \delta_{S_i};1-\alpha).\end{equation}
By construction, we can verify that 
\[(1-\epsilon)P_0 + \epsilon P_1  = \sum_{j=1}^{n+1} w_j \cdot \delta_{S_j}\]
and similarly
\[(1-\epsilon)P_0 + \epsilon \delta_{S_i} = \sum_{j\in[n],j\neq i} w_j \cdot \delta_{S_j} + w_{n+1}\cdot\delta_{S_i} + w_i\cdot\delta_{S_{n+1}}.\]
Therefore, applying the definition of the swapped score vector $S^i$, we see that~\eqref{eqn:claim_for_S_vs_Si} is exactly equivalent to~\eqref{eqn:claim_for_S_vs_Si_v2}, which completes the proof.
\end{proof}
It is worth noting that this proof does not directly use the TV distance between the swapped datasets, but rather, only uses the TV distance between swapped score vectors, as in~\eqref{eqn:dtv_scores_or_data}. Indeed, in  many settings the bound~\eqref{eqn:dtv_scores_or_data} may be much stronger than the result stated in the theorem---for instance, if the data is high-dimensional, we might expect that the total variation distance between high-dimensional data points $Z_i$ and $Z_{n+1}$ could be large, but the distributions of the corresponding scores $S_i,S_{n+1}\in\R$ might nonetheless be fairly similar.

\section{A general outlook through weighted permutations}
\label{sec:general-outlook-permutations}

To conclude this chapter, let us zoom out and reinterpret conformal prediction through general distributions over permutations of the data. With this lens, we will see that it is possible to generalize conformal prediction to work with any joint distribution of the data. In particular, we now give an extension that applies even when the joint distribution and test point are not exchangeable, and the algorithm is not necessarily symmetric.
 
The fundamental idea is that if we know the probability of observing every ordering of the data points conditionally on their values, this is enough to run conformal prediction. To make this more concrete, we will again write $Z_i =(X_i,Y_i)$, and define  $f:(\cX\times\cY)^{n+1}\to\R$ to be the joint density of $(Z_1,\dots,Z_{n+1})$, with respect to some base measure on $\cX\times\cY$ (in particular, if the data has a continuous joint distribution, then $f$ can be taken to be the joint density in the usual sense, while if the data is discrete, then $f$ can be the joint probability mass function). If we assume exchangeability, then $f$ must be a symmetric function---but in this section, we are working in a general setting where $f$ may be arbitrary. 
Now recall that $\widehat{P}_{n+1}$ denotes the empirical distribution of the $n+1$ data points---informally, observing $\widehat{P}_{n+1}$ means that we have observed the \emph{collection} of data points $Z_1,\dots,Z_{n+1}$, but not their order, and in particular we do not know which one of these $n+1$ data points is equal to the test point, $Z_{n+1}$. Then we can verify that
the conditional distribution of the test point $Z_{n+1}$, given the empirical distribution $\widehat{P}_{n+1}$, is equal to
\begin{equation}\label{eqn:empirical_distribution_general_perm}
    Z_{n+1} \mid \widehat{P}_{n+1} \sim \frac{1}{\sum\limits_{\sigma \in \cS_{n+1}} f(Z_{\sigma(1)}, \ldots, Z_{\sigma(n+1)}) }\sum\limits_{\sigma \in \cS_{n+1}} f(Z_{\sigma(1)}, \ldots, Z_{\sigma(n+1)})  \cdot \delta_{Z_{\sigma(n+1)}}.
\end{equation}
That is, once we fix the \emph{values} of the data points, then we can explicitly describe the distribution of the test point in terms of their \emph{ordering}, if we have knowledge of the joint density $f$.
This can be viewed as a generalized version of Proposition~\ref{prop:weighted-exchangeability}.

With this conditional distribution in hand, we can now construct prediction sets using a generalized form of weighted conformal prediction. First, we need some notation: let $Z^y_1,\dots,Z^y_{n+1}$ index the dataset comprised of the training points and the hypothesized test point $(X_{n+1},y)$, i.e.,
\[Z^y_i = Z_i, \ i\in[n], \quad Z^y_{n+1} = (X_{n+1},y),\]
so that the dataset $\cD^y_{n+1}$ is equal to $(Z^y_1,\dots,Z^y_{n+1})$.
Then for every $\sigma \in \cS_{n+1}$, we define
\begin{equation}
    \label{eq:weights-general-permutation}
    w_\sigma^y = \frac{f\left(Z^y_{\sigma(1)}, \ldots, Z^y_{\sigma(n+1)}\right)}{\sum\limits_{\sigma' \in \cS_{n+1}}f\left(Z^y_{\sigma'(1)}, \ldots, Z^y_{\sigma'(n+1)}\right)}.
\end{equation}
An important point is that we often do not need complete knowledge of $f$ in order to compute these weights: for instance, if the data points are i.i.d.\ from an unknown distribution, we immediately have $w_\sigma^y= \frac{1}{(n+1)!}$.

Note that here the weight depends on the entire permutation $\sigma$, which includes the index $\sigma(n+1)$ assigned to the test point position but also the ordering of the training data points as well. In contrast, for covariate shift and label shift (studied in Section~\ref{sec:covariate_and_label_shift}), the weights $w^y_i$ depend only on a single index $i$, corresponding to the question of which data value is placed into the test point position. 

With these weights defined, we now present a coverage guarantee for this general formulation.
\begin{theorem}[Conformal prediction via distributions over permutations]
    \label{thm:general-weighted-permutation}   
    Let $(Z_1, \ldots, Z_n, Z_{n+1}) \sim P$ for an arbitrary joint distribution $P$ on $(\cX\times\cY)^{n+1}$, and let $f$ be the joint density of this distribution with respect to any base measure on $\cX\times\cY$. Fix any score function $s$ (not necessarily symmetric), and define the prediction set 
    \[\cC(X_{n+1}) = \left\{y : s((X_{n+1},y);\cD_{n+1}^y) \leq \hat{q}^y_w\right\}\]
    where
    \[\hat{q}^y_w = \quantile\left(\sum_{\sigma \in \cS_{n+1}} w_\sigma^y \delta_{s\left(Z_{\sigma(n+1)}^y ; (\cD_{n+1}^y)_\sigma\right)} ; 1-\alpha\right),\]
    where the weights $w_\sigma^y$ are defined as in~\eqref{eq:weights-general-permutation}, and where $(\cD_{n+1}^y)_\sigma$ denotes the permuted dataset, $(Z^y_{\sigma(1)},\dots,Z^y_{\sigma(n+1)})$. 
Then,
    \begin{equation}
        \P(Y_{n+1} \in \cC(X_{n+1})) \geq 1-\alpha.
    \end{equation}
\end{theorem} \index{weighted quantile}
This result can be proved as a consequence of the calculation~\eqref{eqn:empirical_distribution_general_perm} above (analogous to the way in which Proposition~\ref{prop:weighted-exchangeability} leads to the proof of Theorem~\ref{thm:train-test-shift-weighted}).

The consequence of this theorem is that if we are able to compute the weights in~\eqref{eq:weights-general-permutation}, then we can carry out conformal prediction, even though we may have nonexchangeable data and/or a nonsymmetric algorithm. However, this general version of the algorithm is not practically implementable without further structure, for both statistical reasons (we are assuming prior knowledge of the joint density $f$, so that the weights in~\eqref{eq:weights-general-permutation} can be computed) and computational reasons (since computing the weights, and the quantile $\hat{q}^y_w$, involves performing calculations over $(n+1)!$ permutations $\sigma$). 
Nonetheless, it points toward an underlying structure that facilitates conformal prediction. 

In particular, many our preceding results can be viewed as special cases of this general formulation, where the statistical and computational challenges are simplified due to the structure of the specific setting. For example, in the case of covariate shift (Section~\ref{sec:covariate-shift}), we have assumed that $Z_1,\dots,Z_n$ are i.i.d.\ from some distribution $P$, and that the test distribution is a known covariate shift of $P$. In this case, the weights $w_\sigma^y$ depend only on the known covariate shift but not the whole distribution $f$, so they can in principle be computed without full knowledge of $f$. Moreover, they can be computed efficiently---the weights $w_\sigma^y$ in~\eqref{eq:weights-general-permutation} depend on $\sigma$ only through $\sigma(n+1)$. Thus, covariate shift is a tractable special case of the general view.

In summary, this abstract viewpoint unifies existing cases of the conformal algorithm. In particular, exchangeable conformal prediction, and conformal prediction with covariate shift, label shift, or general train-test shift, are all special cases. This unified framework can also be viewed as a roadmap for developing extensions to additional settings as well: it suggests that we can aim to identify more complex settings (e.g., if the data exhibits dependence over time) where the structure of the problem enables computation of the weights $w^y_\sigma$ and the weighted quantiles $\hat{q}^y_w$, in order to be able to apply conformal methodology.

\index{weighted conformal prediction|)}
\index{nonexchangeability|)}

\section*{Bibliographic notes}
\addcontentsline{toc}{section}{\protect\numberline{}\textnormal{\hspace{-0.8cm}Bibliographic notes}}

The weighted conformal algorithm introduced in Section~\ref{sec:weighted-algorithm} is proposed by \citet{tibshirani2019conformal}, initially for the setting of covariate shift. In this chapter, we focus on weighted versions of the full and split conformal algorithms; relatedly, \citet{prinster2022jaws,prinster2023jawsx} develop a weighted version of jackknife+ and CV+.

In the setting of a distribution shift between the training and test data (discussed in Section~\ref{sec:covariate_and_label_shift}), 
the weighted conformal method for covariate shift and for label shift is studied by \citet{tibshirani2019conformal} and  by \citet{podkopaev2021distribution}, respectively. These results build on the framework of weighted exchangeability (and its properties such as Proposition~\ref{prop:weighted-exchangeability}), which is introduced in the former paper. These methods have been applied to the problem of predictive inference in a range of statistical problems and settings, including causal inference (specifically, estimation of individualized treatment effects) \citep{lei2020conformal,jin2023sensitivity}, censored data and survival analysis \citep{candes2021conformalized,gui2024conformalized}, and adaptive learning \citep{fannjiang2022conformal}, and also to the problem of testing hypotheses about distribution shift \citep{hu2020distributionfree,bordino2025density}. Results guaranteeing approximate coverage, in the setting where the distribution shift is only known approximately when we define the weights in the WCP algorithm, appear in many works in the literature, e.g., \citet{lei2020conformal,yang2022doubly,jin2023sensitivity,yin2021conformal,gui2024conformalized}. A different approach is taken by \citet{cauchois2024robust}, which establishes robustness of (unweighted) conformal prediction to bounded distribution shifts. As a technical note, Fact~\ref{fact:conversion-quantiles-cdfs-weighted}, used in the proofs of Theorems~\ref{thm:train-test-shift-weighted} and~\ref{thm:fixed-weight-conformal}, is due to \cite{harrison2012conservative}.

Section~\ref{sec:localized} presents the localized conformal prediction method, developed by~\citet{guan2023localized}, and the randomly-localized conformal prediction method, due to~\citet{hore2023conformal}. In the case where $s$ is a pretrained score function, \citet{guan2023localized} provides an efficient implementation of localized conformal prediction (i.e., for the conformalization procedure to determine $\tilde{q}^y$ in Algorithm~\ref{alg:lcp}). 
Both of these works also establish approximate test-conditional coverage guarantees for (randomly-)localized conformal prediction, namely, coverage when the test point is drawn from a small neighborhood. 
The result of Proposition~\ref{prop:LCP_is_fullCP}, establishing that localized conformal prediction can be viewed as an instance of full conformal prediction, appears in \citet{hore2023conformal}.

The fixed-weight conformal prediction approach described in Section~\ref{sec:fixed-weights} is due to \citet{barber2022conformal} (where this method is referred to as `nonexhangeable conformal prediction', or NexCP). In this book, we have presented a special case of that result, by focusing on the setting of a symmetric score function $s$; the work of \citet{barber2022conformal} also extends to the case of non-symmetric $s$, as well.

Finally, in Section~\ref{sec:general-outlook-permutations}, we considered a more general viewpoint that unifies some of the results presented earlier in the chapter. Our exposition in this section follows ideas developed by \citet{prinster2024conformal}, which explains the viewpoint of conformal prediction through weighted permutations to handle arbitrary data distributions, and the work of \citet{barber2025unifying}, which extends this framework to allow nonsymmetric algorithms.
A related viewpoint is developed by \cite{dobriban2023symmpi}, which extends the idea of distributional invariance under permutations to the more general setting of invariance under an arbitrary group action. A specific example of a different form of exchangeability is the setting where data follows a hierarchical sampling structure---for instance we may have data from multiple (random) subpopulations, each of which contributes multiple data points to the sample. \citet{dunn2023distribution,lee2023distribution,duchi2024predictive} develop conformal methodologies for this setting.

\section*{Exercises}
\addcontentsline{toc}{section}{\protect\numberline{}\textnormal{\hspace{-0.8cm}Exercises}}
\begin{enumerate}[font=\bfseries, label={\thechapter.\arabic*}, labelsep=1em, itemsep=1em]
\item Prove Lemma~\ref{lem:generalized-replacement-lemma}.
\item Consider the setting of covariate shift, as studied in Section~\ref{sec:covariate-shift}. Suppose that the likelihood ratio function $x\mapsto \frac{\mathsf{d}Q_X}{\mathsf{d}P_X}(x)$ is not known exactly, but is instead known approximately: for some class $\cW$ of functions $\cX\to (0,\infty)$, we know that $\frac{\mathsf{d}Q_X}{\mathsf{d}P_X} \in \cW$. (For simplicity we will assume that $\frac{\mathsf{d}Q_X}{\mathsf{d}P_X}(x)$ is positive, as well as finite, for all $x$.) 

    With this partial information about the likelihood ratio function, we now define a prediction set (using some symmetric score function $s$): 
    \[\cC(X_{n+1}) = \left\{y\in\cY: S^y_{n+1} \leq \sup_{w\in\cW}\hat{q}^y_w\right\},\]
    where for any function $w:\cX\to(0,\infty)$,
    \[\hat{q}^y_w = \quantile\left(\frac{\sum_{i=1}^{n+1}w(X_i) \delta_{S^y_i}}{\sum_{i=1}^{n+1}w(X_i)};1-\alpha\right).\]
    This type of procedure allows us to account for uncertainty about the likelihood ratio function. 
    \begin{enumerate}
    \item Assuming that the likelihood ratio function $\frac{\mathsf{d}Q_X}{\mathsf{d}P_X}$ is in the class $\cW$, prove that
        \[\P(Y_{n+1}\in\cC(X_{n+1}))\geq 1-\alpha.\]

    \item Now let $\cW$ be the set of all (measurable) functions $w:\cX\to(0,\infty)$. Prove that $\cC(X_{n+1})$ satisfies test-conditional coverage,
        \[\P(Y_{n+1}\in\cC(X_{n+1})\mid X_{n+1})\geq 1-\alpha\textnormal{ almost surely}.\]
        This implies that there is a tradeoff in our choice of the set $\cW$: if $\cW$ is too small then the coverage guarantee of the previous part may not apply, but if $\cW$ is too large then the resulting prediction set may not be informative (since, as we know from Chapter~\ref{chapter:conditional}, requiring test-conditional coverage may lead to meaningless prediction sets).
\end{enumerate}
\item In this exercise we will prove that the RLCP method (Algorithm~\ref{alg:rlcp}) satisfies an approximate notion of test-conditional coverage. Assume that $\cX = \R^d$, and choose the kernel $H$ as
    \[H(x,x') = \frac{\ind{\|x-x'\|_2 \leq h}}{h^dV_d},\]
    where $V_d$ is the volume of the unit ball in $\R^d$. That is, $H(x,\cdot)$ is the density corresponding to the uniform distribution over $\mathbb{B}(x,h)$, which is the ball of radius $h$ centered at $x$. 

    Next, let $B\subseteq\R^d$ be any subset, and for any $r>0$ define
    \[B_{-r} = \{x\in B : \mathbb{B}(x,r)\subseteq B\},\]
    the set of all points $x$ that lie strictly in the interior of $B$ (with a positive margin $r$). We can interpret this as a `deflated' version of $B$.

    Prove the following coverage property: if the data points are i.i.d.\ from some distribution $P$, then for any $B\subseteq\R^d$ with $\P_P(X\in B)>0$,
    \[\P\left(Y_{n+1}\in\cC(X_{n+1})\mid X_{n+1}\in B\right) \geq (1-\alpha)\cdot \frac{\P_P(X\in B_{-2h})}{\P_P(X\in B)}.\]
    In particular, if $B$ is large relative to the bandwidth $h$ of then under some regularity conditions we would typically have $\frac{\P_P(X\in B_{-2h})}{\P_P(X\in B)}\approx 1$, meaning that this result can be viewed as an approximate version of test-conditional coverage.
    \emph{Hint: apply Theorem~\ref{thm:rlcp-coverage}, along with the fact that $\|X_{n+1}-\tilde{X}_{n+1}\|_2\leq h$ almost surely by our choice of the kernel $H$.}
\item Write $Z_i = (X_i,Y_i)$, for data points $i\in[n+1]$. Suppose that $Z_1,\dots,Z_n\iidsim P$ for some unknown distribution $P$. Then, conditional on $\cD_n=(Z_i)_{i\in[n]}$, the test point $Z_{n+1}$ is generated via the following procedure:
    \begin{itemize}
        \item Sample a candidate feature $X_{\rm new}\sim P_X$. If $X_{\rm new}\not\in\{X_1,\dots,X_n\}$ then we set $X_{n+1}=X_{\rm new}$. Otherwise, discard this data point and try again, iterating as many times as needed. (We assume that $P_X$ is not supported on any set of size $\leq n$, to avoid the trivial case where we can never accept any new sample.)
        \item Then sample $Y_{n+1}\mid X_{n+1}$ from the conditional distribution, $P_{Y\mid X}$.
    \end{itemize}
    Derive a simple expression for the weights $w^y_\sigma$ defined in~\eqref{eq:weights-general-permutation}, for this setting.
\end{enumerate}

\chapter{Online Conformal Prediction}
\label{chapter:online-conformal}
\index{time series|(}

So far, we have been discussing conformal prediction in the batch setting: that is, after obtaining a training dataset $\cD_n$ comprising $n$ data points, we then train conformal prediction on this dataset to provide predictive inference on future test points.
On the other hand, there are many applications where we observe data points \emph{online} (i.e., sequentially), and conformal prediction can be extended to this setting.
More precisely, imagine at time $t$, we have observed data points $(X_1, Y_1), \ldots, (X_{t-1}, Y_{t-1})$, and the covariate $X_t$ for the next data point in the series, and need to construct a prediction set $\cC_t(X_t)$ for $Y_t$.
Then at time $t+1$, we add the newly observed data point $(X_t, Y_t)$ into our dataset, and the process repeats.
See Figure~\ref{fig:online-conformal-setting} for an illustration of this setup.
In this new setting, unique problems arise relating to constructing multiple prediction sets, while reusing data, all under potentially shifting distributions. 
This chapter describes some solutions to these problems.

\section{Online conformal prediction with exchangeable data}
\label{sec:online_conformal_data_reuse}

We first consider the case where we have a sequence of data points $(X_1, Y_1), (X_2, Y_2), \dots, (X_T, Y_T)$ that are exchangeable. At each time step $t$, we have seen $(X_1, Y_1), \dots, (X_{t-1}, Y_{t-1})$ and the next feature vector $X_t$, and we would like to construct a prediction set $\cC_t(X_t)$ for the as-yet-unseen response $Y_t$ with correct coverage:
\begin{equation}
\label{eq:online_timet_cvg}
    \P(Y_t \in \cC_t(X_t)) \ge 1-\alpha.
\end{equation}
To this end, suppose we generate prediction sets at each step with full conformal prediction in
the usual way. That is, $\cC_t(X_t)$ is the output of Algorithm~\ref{alg:full-cp} on training data $(X_1, Y_1), \dots, (X_{t-1}, Y_{t-1})$ with test point $X_t$, implemented with some choice of the score function $s$. For completeness, we also define $\cC_1(X_1) = \cY$.

\begin{figure}[t]
    \centering
    \includegraphics[width=0.75\textwidth]{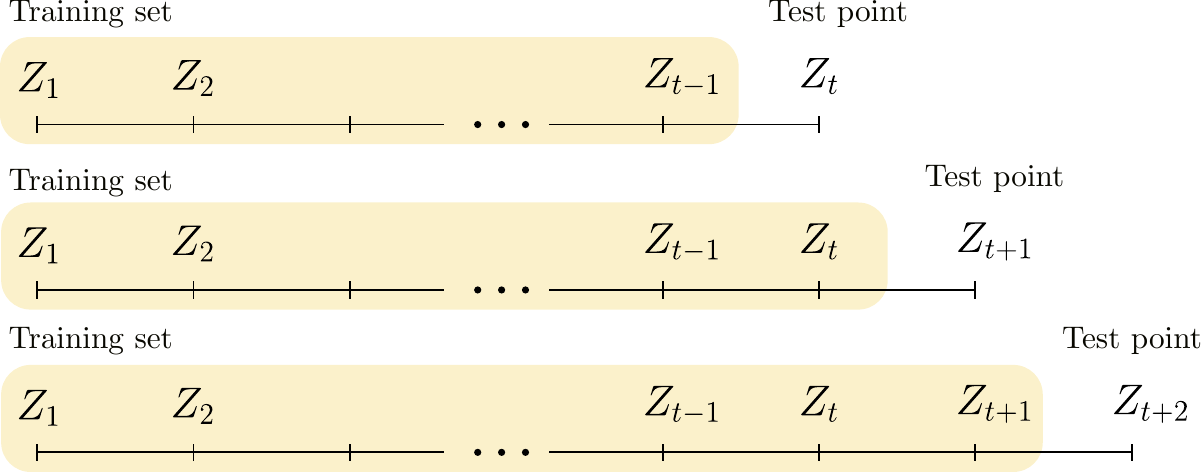}
    \caption{\textbf{Online prediction setting}, with data points $Z_i=(X_i,Y_i)$ for $i=1,2,\dots$. The rows represent times $t$, $t+1$, and $t+2$, respectively, in the streaming process of running online conformal prediction.
    At time $t$, we use the historical training set $(Z_1, \ldots, Z_{t-1})$ and the test covariate $X_t$ to produce a prediction set for the test label $Y_t$.
    We then repeat this process at times $t+1$, $t+2$, etc, folding the new data points into the training set as we go. 
    }
    \commentAlt{A sequence of data points $Z_1,\dots,Z_{t-1}$ is highlighted with shading and labeled `Training set', while the next point $Z_t$ is not highlighted and is labeled `Test point'. This picture is repeated two more times, for the two next time points.}
    \label{fig:online-conformal-setting}
\end{figure}

An immediate consequence of the validity of conformal prediction (Theorem~\ref{thm:full-conformal}) is that for each $t$, the coverage property in~\eqref{eq:online_timet_cvg} holds. 
Interestingly, we will now see that this algorithm has desirable properties beyond this basic marginal coverage result.
We define
\begin{equation}
\label{eq:err-def}
    \err_t = 
    \begin{cases}
    1 & \textnormal{ if } Y_t \notin \cC_t(X_t) \\
    0 & \textnormal{ if } Y_t \in \cC_t(X_t)
    \end{cases}
\end{equation}
as the indicator variable of the miscoverage event at time step $t$. Our marginal coverage property 
in~\eqref{eq:online_timet_cvg} is then equivalent to the statement that $\E[\err_t] \le \alpha$ for all $t$. This by itself is not enough to imply that the average coverage of this procedure for a large sequence of data points (i.e., averaging $\err_t$ over a long range of times $t$) would concentrate around $1-\alpha$; since each data point $t$ is reused in constructing $\cC_{t'}$ for every $t'>t$, the coverage events would appear to be dependent in a complex way. 

A surprising and deep result about conformal prediction is that the events $\err_t$ are in fact independent across different time steps.
\begin{proposition}[Independence of errors for online conformal]
\label{prop:online-indep-errors}
Suppose $(X_1, Y_1), (X_2, Y_2), \dots, (X_T, Y_T)$ are exchangeable, the score function $s$ is symmetric, and the scores are distinct almost surely at each time $t$ (i.e., at each $t\in[T]$, it holds almost surely that $s((X_i,Y_i);\cD_t)$ are distinct over all $i\in[t]$). Then the indicators of miscoverage, $\err_t$, are mutually independent.
\end{proposition}
Here $\cD_t$ denotes the dataset comprised of the first $t$ data points, $\cD_t=((X_1,Y_1),\dots,(X_t,Y_t))$, for each $t\in[T]$.

This is a consequence of the slightly more general result that
the conformal p-values at each step are independent. To be explicit,
the conformal p-value at time step $t$ is defined as
\begin{equation}
\label{eq:online-conformal-pvalue}
    p_t = \frac{\sum_{i=1}^{t} \ind{s\left((X_i, Y_i); \cD_t \right) \geq s\left((X_{t}, Y_t); \cD_t \right)}}{t}.
\end{equation}

\begin{theorem}[Independence of online conformal p-values]
\label{thm:online-indep-pvals}
Suppose $(X_1, Y_1), (X_2, Y_2), \dots, (X_T, Y_T)$ are exchangeable, the score function $s$ is symmetric, and the scores are distinct almost surely at each time $t$. Then, $p_t$ is distributed as a
discrete random variable on $\{1/t,2/t,\dots,1\}$, and $p_1,\dots,p_T$ are
mutually independent.
\end{theorem} \index{conformal p-value}
Note that Proposition~\ref{prop:online-indep-errors} follows directly from Theorem~\ref{thm:online-indep-pvals} because $\err_t$
is a function of $p_t$: namely, recalling Proposition~\ref{prop:conformal-via-pvalues}, it holds that $\err_t = \ind{p_t \le \alpha}$.

To see the main idea of the proof, let us return to Figure~\ref{fig:online-conformal-setting}. Observe that in the middle row, by symmetry of the score function, the conformal p-value $p_{t+1}$ (and the outcome of $\err_{t+1}$) does not depend on the ordering of the $t$ data points in the shaded block---that is, $(X_1, Y_1),\dots,(X_t,Y_t)$ are still exchangeable even if we condition on $p_{t+1}$. By contrast, $p_t$ depends \emph{only} on the ordering of the first $t$ points. This leads to independence between $p_t$ and $p_{t+1}$.

\begin{proof}[Proof of Theorem~\ref{thm:online-indep-pvals}]
For each $t\in[T]$, let $\cD_t = (Z_1,\dots,Z_t)$ denote the data points observed up to time $t$ (where $Z_t = (X_t,Y_t)$), and let $\widehat{P}_t = \frac{1}{t}\sum_{t'=1}^t \delta_{Z_{t'}}$ denote its empirical distribution. Below, we will prove that for each $t\in [T]$, the random variable $p_t$ has the following conditional distribution:
\begin{equation}\label{eqn:online-indep-pvals_step1}
    p_t \mid \big(\widehat{P}_t,Z_{t+1},\dots,Z_T\big) \ \sim \ \textnormal{Unif}\big(\{1/t,2/t,\dots,1\}\big).
\end{equation}
In particular, this is sufficient to verify the first claim (i.e., the marginal distribution of $p_t$). We will also prove that
\begin{equation}\label{eqn:online-indep-pvals_step2}
    \textnormal{For each $t'>t$, $p_{t'}$ can be written as a function of $\big(\widehat{P}_t,Z_{t+1},\dots,Z_T\big)$.}
\end{equation}
In particular, combining~\eqref{eqn:online-indep-pvals_step1} and~\eqref{eqn:online-indep-pvals_step2} yields
\[p_t \mid \big(\widehat{P}_t,Z_{t+1},\dots,Z_T,p_{t+1},\dots,p_T\big)\ \sim \ \textnormal{Unif}\big(\{1/t,2/t,\dots,1\}\big) ,\]
and then marginalizing we have
\[p_t \mid \big(p_{t+1},\dots,p_T\big)\ \sim \ \textnormal{Unif}\big(\{1/t,2/t,\dots,1\}\big) ,\]
for each $t\in[T]$. This implies that $p_1,\dots,p_T$ are mutually independent, as desired.

\paragraph{Proving~\eqref{eqn:online-indep-pvals_step1}.}
By exchangeability of the data, we can see that
\[(Z_1,\dots,Z_t) \mid \big(\widehat{P}_t,Z_{t+1},\dots,Z_T\big)\]
is conditionally exchangeable (formally, this statement can be verified via Lemma~\ref{lem:conditional_exchangeability} and Fact~\ref{fact:conditional-exch-1}). Consequently, the scores
\[s(Z_1;\cD_t),\dots, s(Z_t;\cD_t)\]
are also exchangeable conditionally on $\big(\widehat{P}_t,Z_{t+1},\dots,Z_T\big)$, since $s$ is symmetric. Now let $S_{(1;t)}\leq \dots \leq S_{(t;t)}$ denote the order statistics of these $t$ scores, at time $t$. Since we have assumed these $t$ scores are distinct, almost surely, Fact~\ref{fact:exchangeable-properties}\ref{fact:exchangeable-properties_part4} implies that $\P(s(Z_t;\cD_t) > S_{(t-k;t)}) = k/t$, for each $k\in[t-1]$. But by definition of the p-value $p_t$, we have $p_t \leq k/t$ if and only if $s(Z_t;\cD_t) > S_{(t-k;t)}$; since $p_t$ can only take values in the grid $\{1/t,\dots,1\}$, by construction, the conclusion follows.

\paragraph{Proving~\eqref{eqn:online-indep-pvals_step2}.}
For this last step, we will write the p-value at time $t$ as $p_t(\cD_T)$, to make explicit its dependence on the dataset $\cD_T$ (the definition of this p-value is still given by~\eqref{eq:online-conformal-pvalue}, as before). 

Let $\sigma\in\cS_T$ be any permutation that only permutes the first $t$ indices, i.e., $\sigma(t')=t'$ for any $t'>t$. Let $(\cD_T)_{\sigma} = (Z_{\sigma(1)},\dots,Z_{\sigma(T)})$ be the dataset $\cD_T$ permuted according to $\sigma$. Note that, since $\sigma$ maps $[t']$ to $[t']$ for any $t'>t$, this means that the first $t'$ many entries of $(\cD_T)_{\sigma}$ are, up to permutation, equal to $\cD_{t'}$.

We now verify that permuting the data according to $\sigma$ does not change p-values after time $t$: for any $t'>t$,
\begin{align}
    p_{t'}((\cD_T)_{\sigma}) &= \frac{\sum_{i=1}^{t'} \ind{s\left(Z_{\sigma(i)}; \cD_{t'} \right) \geq s\left(Z_{\sigma(t')}; \cD_{t'} \right)}}{t'}\\ 
    &= \frac{\sum_{i=1}^{t'} \ind{s\left(Z_{\sigma(i)}; \cD_{t'} \right) \geq s\left(Z_{t'}; \cD_{t'} \right)}}{t'}\textnormal{\quad since $\sigma(t')=t'$} \\
    &= \frac{\sum_{i=1}^{t'} \ind{s\left(Z_i; \cD_{t'} \right) \geq s\left(Z_{t'}; \cD_{t'}\right)}}{t'}\textnormal{\quad since $\sigma$ maps $[t']$ to $[t']$}\\   
    &=p_{t'}(\cD_T),
\end{align}
where the first step holds since $s$ is symmetric and, as mentioned above, up to permutation it holds that $\cD_{t'}$ is equal to the first $t'$ many entries of $(\cD_T)_{\sigma}$.

In particular, this implies that $p_{t'}=p_{t'}(\cD_T)$ depends on $\cD_T = (\cD_t,Z_{t+1},\dots,Z_T)$ only through $(\widehat{P}_t,Z_{t+1},\dots,Z_T)$, since permuting the first $t$ data points (i.e., changing the order of the data points in $\cD_t$) does not change the value of $p_{t'}$.

\end{proof}

This result is important since it implies that conformal prediction in the online setting will have total coverage approaching $1-\alpha$:
\begin{corollary}[Average coverage of online conformal]\label{cor:avg_online_conformal}
Suppose $(X_1, Y_1), (X_2, Y_2), \dots$ are exchangeable, the score function is symmetric, and the scores are distinct almost surely at each time $t$. Then,
\begin{equation}
    \frac{1}{T} \sum_{t=1}^T \ind{Y_{t} \in \cC_t(X_t)} \to 1-\alpha,
\end{equation}
almost surely as $T \to \infty$.
\end{corollary} \index{coverage!online}
This result follows from Proposition~\ref{prop:online-indep-errors} and the Law of Large Numbers. There is a minor subtlety here in that $\E[\err_t]$ is not exactly the same for each $t$---rather, it is some value lying in the range $(\alpha - 1/t,\alpha]$. Nonetheless, this is easily handled. More refined versions of this result can be given with Hoeffding's inequality and other refinements of the Law of Large Numbers. 

If we do not assume that the scores are distinct almost surely at each time $t$, then it may no longer be the case that the p-values $p_t$ are independent. Nonetheless, due to the fact that conformal prediction can only become more conservative in the presence of ties, and thus a more conservative version of Corollary~\ref{cor:avg_online_conformal} holds: if $(X_1,Y_1),(X_2,Y_2),\dots$ are exchangeable, then
\[\liminf_{T\rightarrow\infty} \frac{1}{T} \sum_{t=1}^T \ind{Y_{t} \in \cC_t(X_t)} \geq 1-\alpha,\]
almost surely.

These powerful results show that we can reliably use conformal prediction to construct a sequence of prediction intervals in the online exchangeable setting. Moreover, having independent p-values at each time step facilitates other goals when analyzing a sequence of data, which is the focus of the next section.

\paragraph{Online conformal without online training?}
Throughout this section, we have assumed that each test point is added to the training set for the next time step---that is, at time $t$ the point $(X_t,Y_t)$ is the test point, but at time $t+1$ (and beyond), this same point $(X_t,Y_t)$ is now part of the training set. 
In some settings, it may not be possible or desirable to add the test data into the training set in this online way.
The independence results of this section may no longer apply; if the test points $t=n+1,n+2,\dots$ are all compared to the \emph{same} training set $((X_i,Y_i))_{i\in[n]}$ the conformal p-values $p_{n+1},p_{n+2},\dots$ will be dependent. 
We refer the reader to Chapter~\ref{chapter:extensions} (see Section~\ref{sec:multiplicity__outlier_FDR}) for results that characterize the dependence among the p-values for this setting.

\section{Testing exchangeability online}\label{sec:testing_exch_online}
\index{nonexchangeability|(}
\index{exchangeability!testing|(}

We next consider the ability to test the exchangeability hypothesis online---an important task when monitoring the online deployment of learning algorithms to identify the presence of distribution shift. 
In this section, we show how to test this hypothesis using the conformal p-values from~\eqref{eq:online-conformal-pvalue}. More concretely, given a stream of data $(X_1,Y_1),(X_2,Y_2),\dots$ (with finite or infinite length), we would like to detect any failure of exchangeability in this data stream---for instance, due to a sudden changepoint.

The main message is that the conformal p-values can be combined in such a way that, under exchangeability, the resulting statistic is a supermartingale.
Any supermartingale will generally take only small values, so large values of this statistic constitute evidence against the null hypothesis of exchangeability.
In this section, we show how to define such supermartingales and how to use them to test the hypothesis of exchangeability online.
\begin{definition}[Supermartingale]
    A (finite or infinite) sequence of random variables $M_1, M_2,\ldots$ is a supermartingale if, for all $t$, $\E[|M_{t}|] < \infty$ and, for all $t\geq 2$,
    \begin{equation}
        \E[M_{t} \mid M_1, \ldots, M_{t-1}] \leq M_{t-1}.
    \end{equation}
\end{definition} \index{supermartingale}
If the inequality above is replaced by an equality, then this sequence is called a martingale.
An important fact about supermartingales that we will soon use is \emph{Ville's inequality}: if $M_1,M_2,\dots$ is a nonnegative supermartingale, then for any $a > 0$,
\begin{equation}
    \label{eq:ville}
    \P\left( \sup_{t \geq 1} M_t \geq a \right) \leq \frac{\E[M_1]}{a}.
\end{equation}

For the problem of testing the hypothesis of exchangeability, we need to define a sequence $M_t$ that is large if the hypothesis is violated, but is a supermartingale if the hypothesis holds. 
Fortunately, by Theorem~\ref{thm:online-indep-pvals} above, the conformal p-values are independent for an exchangeable stream of data points, which enables the construction of a supermartingale. This next result gives a simple example:

\begin{proposition}
    \label{prop:simple-conformal-supermartingale}
    Consider the (finite or infinite) sequence of conformal p-values $p_1, p_2, \ldots$ from~\eqref{eq:online-conformal-pvalue}, constructed on a (finite or infinite) streaming dataset $(X_1,Y_1),(X_2,Y_2),\dots$.
    Assume the data is exchangeable and the score function is symmetric. Let $\lambda\in[0,1]$ be a fixed parameter.
    Then, if the scores are distinct almost surely at each time $t$, the sequence
    \begin{equation}
        M_t = \prod_{t'=1}^t\frac{1-\lambda p_{t'}}{1-\lambda/2}
    \end{equation}
    is a supermartingale.
\end{proposition}
\begin{proof}[Proof of Proposition~\ref{prop:simple-conformal-supermartingale}]
    Since $p_t \in [0,1]$ for all $t$, $M_{t}$ is bounded almost surely for each $t$ and thus satisfies $\E[|M_t|]<\infty$. 
    We also have $M_t = M_{t-1} \cdot \frac{1-\lambda p_t}{1-\lambda/2}$ by definition, and thus, under exchangeability,
    \begin{multline}
        \E[M_{t} \mid M_{1}, \ldots, M_{t-1}] = \E\left[ 
        M_{t-1}\cdot \frac{1-\lambda p_t}{1-\lambda/2} \,\middle|\, M_{1}, \ldots, M_{t-1} \right]\\ = M_{t-1} \cdot \frac{\E[1-\lambda p_{t}\mid M_{1}, \ldots, M_{t-1} ]}{1-\lambda/2}.
    \end{multline}
    Finally, we have 
    \[\E[1-\lambda p_{t}\mid M_{1}, \ldots, M_{t-1} ] \leq 1-\lambda/2,\]
    since, by Theorem~\ref{thm:online-indep-pvals}, $p_t$ is superuniform (i.e., $\P(p_t\leq\tau)\leq\tau$ for all $\tau\in[0,1]$), and is independent from $p_1,\dots,p_{t-1}$ (and is thus independent from $M_1,\dots,M_{t-1}$).
\end{proof}
Thus, the values $M_t$ will be small under the hypothesis of exchangeability.
Alternatively, if the data contains strong evidence against exchangeability, the p-values $p_t$ will likely be small, and thus the sequence $M_t$ could potentially become quite large as $t$ increases. This intuition can be formalized into a statistical test for exchangeability based on Ville's inequality, as we will see next. Of course, this particular definition of $M_t$ is only one specific (and very simple) example; the design of powerful supermartingales is an active topic of research, and this next theorem applies to a more general recipe for constructing the sequence $M_t$.

\begin{theorem}[Online test for exchangeability]
    \label{thm:supermartingale-test}
Let $f_t:[0,1]\rightarrow[0,\infty)$ be any sequence of nonincreasing functions such that $\int_{r=0}^1 f_t(r)\;\mathsf{d}r\leq 1$. Let $(X_1,Y_1),(X_2,Y_2),\dots$ be a (finite or infinite) sequence of data points, 
and let
\begin{equation}\label{eqn:conformal_martingale_recipe}M_t = \prod_{t'=1}^t f_{t'}(p_{t'}),\end{equation}
where we define the conformal p-values $p_t$ as in~\eqref{eq:online-conformal-pvalue}.
Then, if the data points are exchangeable, the score function is symmetric, and the scores are distinct almost surely at each time $t$, the sequence $M_t$ is a supermartingale. Moreover, under the same assumptions,
for any $\alpha\in[0,1]$ we have
\[\P\left(\sup_t M_t \geq 1/\alpha\right) \leq \alpha.\]
\end{theorem}

Theorem~\ref{thm:supermartingale-test} tells us that, under exchangeability, with probability $\geq 1-\alpha$ the supermartingale $M_t$ constructed in~\eqref{eqn:conformal_martingale_recipe} never reaches the threshold $1/\alpha$ over its entire time of existence.
Thus, it implies a simple algorithm for detecting distribution shift online:
\begin{enumerate}
    \item Choose functions $f_1,f_2,\dots$ as specified in Theorem~\ref{thm:supermartingale-test} (for example, we might take $f_t(r) = \frac{1-\lambda r}{1-\lambda/2}$ as in Proposition~\ref{prop:simple-conformal-supermartingale}).
    \item At each time $t=1,2,\dots$:
    \begin{itemize}
    \item Observe the new data point $(X_t,Y_t)$, compute conformal p-value $p_t$, and compute $M_t$.
    \item If $M_t \geq 1/\alpha$, reject the hypothesis of exchangeability. Otherwise, continue to the next time.
    \end{itemize}
\end{enumerate}
See Figure~\ref{fig:exchangeability-martingale-test} for an illustration of this procedure.

\begin{figure}[t]
    \centering
    \includegraphics[width=0.85\textwidth]{\diagramspath ville-test.pdf}
    \caption{\textbf{The online test for exchangeability.} We apply the online test for exchangeability to two example draws from the same sequence $(M_t)_{t\geq1}$.
    On the left-hand side is a draw for which the test does not reject the null hypothesis of exchangeability, because the realized value of the sequence never reaches $1/\alpha$. 
    On the right, the hypothesis of exchangeability is rejected, since the realized value of the sequence exceeds $1/\alpha$ at some point in its history.
    }
    \commentAlt{Two plots with axes labeled $t$ and $M_t$. The left plot (`Failure to reject exchangeability') shows a trajectory that stays below threshold $1/\alpha$. On the right ('Rejection of exchangeability'), the trajectory crosses the threshold.}
    \label{fig:exchangeability-martingale-test}
\end{figure}

\begin{proof}[Proof of Theorem~\ref{thm:supermartingale-test}]
First, we verify that $M_t$ is a supermartingale under the hypothesis of exchangeability. For each $t\geq1$ we have
\[\E[f_t(p_t)]\leq \int_{r=0}^1 f_t(r)\;\mathsf{d}r\leq 1,\]
where the first step holds since $p_t$ is superuniform according to Theorem~\ref{thm:online-indep-pvals}, and $f_t$ is nonincreasing. Next,
we observe that $M_t = M_{t-1} \cdot f_t(p_t)$ for each $t$, and so
\[\E[M_t\mid M_1,\dots,M_{t-1}] = M_{t-1}\cdot \E[f_t(p_t)\mid M_1,\dots,M_{t-1}] = M_{t-1}\cdot \E[f_t(p_t)]\leq M_{t-1},\]
where the next-to-last step holds since $p_t$ is independent from $p_1,\dots,p_{t-1}$ (and therefore, from $M_1,\dots,M_{t-1}$) by Theorem~\ref{thm:online-indep-pvals}. Thus $M_t$ is a supermartingale. (Since $M_t$ is nonnegative, and $\E[M_t]\leq 1$, we have also verified $\E[|M_t|<\infty]$.)
Finally, the last claim holds by applying Ville's inequality~\eqref{eq:ville}.
\end{proof}

As we can see in the proof above, the bound $\alpha$, bounding the probability of a false rejection of the hypothesis of exchangeability, is due simply to Ville's inequality. This would hold for any supermartingale---so why was conformal prediction critical?
The role of conformal prediction is that it allows us to reuse the same data over time for every test of exchangeability without having to pay an explicit multiplicity penalty for data reuse.
This is because of the independence of the conformal p-values (Theorem~\ref{thm:online-indep-pvals}), as used in the proof of Proposition~\ref{prop:simple-conformal-supermartingale}. 
This is one of the most unique and powerful features of conformal prediction.

As an example, suppose the conformal score at each step is simply the value of $y$:
$s((x,y); \cD) = y$. In this case, the conformal p-value $p_t$ is a re-scaling of the rank of the new data point $Y_t$ relative to those seen previously. We can take these p-values and construct the exchangeability supermartingale $M_t$ as in Proposition~\ref{prop:simple-conformal-supermartingale} in order to test for exchangeability. It is clear that $M_t$ will grow if we frequently observe  small p-values, which corresponds to seeing values of $Y_t$ that are larger than most of the previous observed values $Y_1,\dots,Y_{t-1}$. Thus, we expect this to detect violations of exchangeability where the values of $Y_t$ are trending up. More broadly, Theorem~\ref{thm:supermartingale-test} has practical ramifications for the problem of changepoint detection: if the conformal score function is constructed to measure some notion of error of the data point $(X_t,Y_t)$ relative to a trained model, then
online tests for exchangeability allow us to identify deviations of the error from its typical behavior.

Finally, we note that while all results in Sections~\ref{sec:online_conformal_data_reuse} and~\ref{sec:testing_exch_online} are stated under the assumption that scores are distinct almost surely at each time $t$, the results can be generalized to remove this assumption (in particular, in Chapter~\ref{chapter:further-topics} we will present a randomized mechanism for breaking ties, under which the conformal p-values $p_t$ become i.i.d.\ uniform).
\index{exchangeability!testing|)}

\section{Prediction sets for adversarial sequences}\label{sec:ACI_and_quantile_tracking}
\index{adversarial sequence|(}

The previous sections have dealt with online conformal prediction for exchangeable sequences.
Recent efforts have aimed to generalize beyond that setting, to one where the incoming sequence of data points $(X_1, Y_1), (X_2, Y_2), \ldots$ is arbitrary---the data points might be drawn from distributions that change over time (i.e., non-i.i.d.\ data), or might even be deterministic rather than random.
Because we do not want to make any distributional assumptions
in this setting, the best one can generally hope for is a form of \emph{long-run coverage}, i.e., that a large fraction of the prediction sets $\cC_t(X_t)$ contain $Y_t$.
A recently developed line of work has leveraged conformal prediction to achieve guarantees of the form
\begin{equation}
   \left| \frac{1}{T} \sum_{t=1}^T \err_t - \alpha \right|\rightarrow 0,
\end{equation}
meaning that, on average, the miscoverage rate is essentially equal to $\alpha$---without any assumption of exchangeability or, indeed, even assuming that the data is randomly generated at all.

Building on conformal prediction, we consider
prediction sets $\cC_t$ of the form
\begin{equation}\label{eqn:Ct_st_qt}
    \cC_t(X_t) = \{ y : s_t(X_t, y) \leq q_t \},
\end{equation}
where the subscript $t$ indicates a dependence on all previous time steps $1, 2, \ldots, t-1$. That is, the score function $s_t$, and the threshold $q_t$, may be designed in a way that depends on the previous data.

In the adversarial sequence setting, where the data may no longer be exchangeable, we will allow the score functions $s_t$ to be chosen arbitrarily, and the thresholds $q_t$ will then be defined in an online way.
Specifically, after initializing at any value $q_1\in[0,B]$, for each $t\geq 1$ the threshold $q_{t+1}$ will be defined as follows:\index{quantile tracking}
\begin{equation}
    \label{eq:quantile-tracking}
    q_{t+1} = q_t + \eta_t(\err_t - \alpha),
\end{equation} where $\eta_t>0$ is a step size parameter, and where $\err_t$ is defined as in~\eqref{eq:err-def}.
This \emph{quantile tracking} algorithm does something very simple: when we fail to cover (i.e., $\err_t = 1$), it increases the threshold $q_{t+1}$ in the next time step, to be slightly more conservative.
Alternatively, when we cover (i.e., $\err_t=0$), we decrease the next threshold, to be slightly less conservative.
This construction leads to the following guarantee on average coverage.
\begin{theorem}[Average coverage guarantee for quantile tracking]
    \label{thm:quantile-tracking}
    Consider an arbitrary sequence of data points $(X_1, Y_1), (X_2, Y_2), \ldots$, and sequence of score functions $s_1, s_2, \ldots$ satisfying $s_t(x,y) \in [0, B]$ for all $t$ and $(x,y)$.
    Furthermore, let the sequence of step sizes $\eta_t$ be any nonincreasing positive sequence.
    Then the quantile tracking algorithm~\eqref{eq:quantile-tracking}, initialized at any $q_1\in[0,B]$, satisfies
    \begin{equation}
        \frac{1}{T}\sum\limits_{t=1}^T \err_t \in \left[ \alpha \pm \frac{B + \eta_1}{\eta_T T}\right]\textnormal{ for all $T\geq 1$}.
    \end{equation}
\end{theorem} \index{coverage!online}
As a special case, for a constant step size $\eta_t=\eta$, the empirical miscoverage rate concentrates to $\alpha$ at a fast rate of $\bigo(1/T)$.
We emphasize that this result holds deterministically, for any sequence of data points and any choice of score functions $s_t$. In particular, we have not assumed exchangeability: by adjusting the threshold $q_t$ based on the method's coverage over past data points, the procedure is able to maintain approximately the desired coverage level over time, even under distribution shift or other non-exchangeable settings.
\begin{proof}[Proof of Theorem~\ref{thm:quantile-tracking}]
    \textbf{Step 1: The iterates are bounded.}
    We will first show that $q_t \in [-\eta_{1}\alpha, B+\eta_{1}(1-\alpha)]$ deterministically for all $t$. Since $q_1\in[0,B]$ by definition, this must hold at $t=1$. Now suppose that the statement fails for the first time at some $t\geq 2$.
    Assume for the sake of contradiction that $q_t > B+\eta_{1}(1-\alpha)$; since $t$ is the first time the statement fails, we also have $q_{t-1} \leq B+\eta_{1}(1-\alpha)$. 
    Since $q_t > q_{t-1}$ we must have that $\err_{t-1} = 1$ by~\eqref{eq:quantile-tracking}. 
    But, since $\eta_{t-1}\leq \eta_1$, we have $q_{t-1} = q_t - \eta_{t-1}(1-\alpha) \geq q_t - \eta_1(1-\alpha)> B$, and so
    \[\err_{t-1}= \ind{s_{t-1}(X_{t-1},Y_{t-1}) > q_{t-1}} \leq \ind{s_{t-1}(X_{t-1},Y_{t-1}) > B} = 0,\]
    since scores are bounded by $B$. This is a contradiction.
    The other side follows a similar argument.

    \textbf{Step 2: Simultaneous bounds on the weighted coverage gap.}
    Examining the form of~\eqref{eq:quantile-tracking}, we can see that for all $T_1 \geq T_0 \geq 1$, 
    \begin{equation}\label{eq:qT_equals_sum}
        q_{T_1+1} - q_{T_0}= \sum\limits_{t=T_0}^{T_1} \eta_t(\err_t - \alpha).
    \end{equation}
    Since both $q_{T_1+1}$ and $q_{T_0}$ are bounded by Step 1, we then have
    \begin{align}
        \left|\sum\limits_{t=T_0}^{T_1} \eta_t(\err_t - \alpha)\right| & = |q_{T_1+1}-q_{T_0}| \leq B + \eta_1.
    \end{align}
    
    \textbf{Step 3: Bounding the long-run coverage gap.} Set $\eta_0 = \infty$ for convenience, so that
    \[\frac{1}{\eta_t} = \sum_{r=1}^t \left(\frac{1}{\eta_r}-\frac{1}{\eta_{r-1}}\right)\]
    holds for all $t=1,\dots,T$. We calculate
    \begin{align}
        \left|\frac{1}{T}\sum\limits_{t=1}^T(\err_t - \alpha)\right| &= \left|\frac{1}{T}\sum\limits_{t=1}^T\frac{1}{\eta_t}\eta_t(\err_t - \alpha)\right| \\
        &= \left|\frac{1}{T}\sum\limits_{t=1}^T\sum_{r=1}^t\left(\frac{1}{\eta_r} - \frac{1}{\eta_{r-1}}\right)\eta_t(\err_t - \alpha) \right| \\
        & = \left|\frac{1}{T}\sum\limits_{r=1}^T\left(\frac{1}{\eta_r} - \frac{1}{\eta_{r-1}}\right)\sum_{t=r}^T\eta_t(\err_t - \alpha) \right|\\
        & \leq \frac{1}{T}\sum\limits_{r=1}^T\left(\frac{1}{\eta_r} - \frac{1}{\eta_{r-1}}\right)\left|\sum_{t=r}^T\eta_t(\err_t - \alpha) \right|\\
        & \leq \frac{B + \eta_1}{T}\sum\limits_{r=1}^T\left(\frac{1}{\eta_r} - \frac{1}{\eta_{r-1}}\right) \\
        & = \frac{B + \eta_1}{\eta_T T},
    \end{align}
    where the first inequality holds since $\eta_t$ is a positive and nonincreasing sequence, and the second inequality holds from the calculation in Step 2 above.
\end{proof}
The key idea of the proof is that the threshold $q_t$ is deterministically bounded over all times $t$.
Since $q_T$ is equal to the (reweighted) long-run coverage gap as in~\eqref{eq:qT_equals_sum}, this implies that the number of historical errors cannot be too large. See Figure~\ref{fig:quantile-tracker-coverage-plot} for a visual intuition.
\begin{figure}[t]
    \centering
    \includegraphics[width=0.7\textwidth]{\diagramspath qt-coverage-plot.pdf}
    \caption{\textbf{The long-run coverage gap.} The plot displays the long-run coverage gap $\frac{1}{T}\sum_{t=1}^T \err_t - \alpha$ over time $T$. The procedure is initialized with $q_1 = 0$ and run with a fixed step size $\eta_t = \eta$. By~\eqref{eq:qT_equals_sum}, the coverage gap at time $T$ is equal to $\frac{q_{T+1}}{\eta T}$. The thick gray lines at $\frac{B + \eta}{\eta T}$ define a deterministic envelope for the long-term coverage gap; the proof says that the long-run coverage gap must always lie within these bounds. The thin gray envelope defines the region where $q_{T+1} > B$ (top) or $q_{T+1} < 0$ (bottom).
    One way to understand the proof of coverage is that, once the long-run coverage gap leaves the thin gray envelope, it is forced to return---because the next prediction set $\cC_{T+1}$ will be $\cY$ (if $q_{T+1}>B$) or $\varnothing$ (if $q_{T+1}<0$), forcing the miscoverage to move back towards the target value $\alpha$ in the next step.
    }
    \commentAlt{The plot displays the accumulated average error deviation $\frac{1}{T}\sum_{t=1}^T (\err_t-\alpha)$, against time $T$. The trajectory moves up and down, but always lies between two curves, $\frac{B+\eta}{\eta T}$ and $-\frac{B+\eta}{\eta T}$.}
    \label{fig:quantile-tracker-coverage-plot}
\end{figure}

When implementing this method in practice, how should the step sizes $\eta_t$ be defined? 
This choice depends on the context, and different choices lead to different flavors of guarantees.
When tracking a highly varying score sequence, such as one generated by a dynamical system, it can make sense to think of the quantile tracker as a control algorithm and pick a non-decaying $\eta_t = \eta$ for some constant $\eta$, or even one that adapts to the score sequence. 
However, this can lead to thresholds $q_t$ that fluctuate wildly, yielding many infinite-size sets and empty sets.

Alternatively, we might choose $\eta_t \propto t^{-(1/2+\epsilon)}$, for some small $\epsilon>0$.
This choice of $\eta_t$ allows the threshold to adapt to distribution drift (since Theorem~\ref{thm:quantile-tracking} guarantees average error $\lesssim T^{-(1/2-\epsilon)}$ at each time $T$), while stabilizing the threshold and even converging when possible.

\paragraph{The i.i.d.\ setting.} In the setting of i.i.d.\ data, we can formalize our empirical observation that a fixed step size $\eta_t= \eta$ leads to fluctuating behavior while a decaying sequence $\eta_t$ offers more stability.
In particular, consider a setting where $(X_t,Y_t)\iidsim P$ (for some distribution $P$), and where the score functions $s_t:\cX\times\cY\rightarrow[0,B]$ are trained online---that is, $s_t$ may depend on the earlier data, $(X_1,Y_1),\dots,(X_{t-1},Y_{t-1})$ (which includes the setting of applying split or full conformal prediction at each time $t$, as discussed earlier). For any function $s$, let $F_s$ denote the CDF of $s(X,Y)$, under $(X,Y)\sim P$.
At each time $t$, this means that the probability of coverage for the next data point $(X_t,Y_t)$ is given by $F_{s_t}(q_t)$.
We state the following result without proof:
\begin{theorem}[Asymptotics of quantile tracking with fixed or decaying step sizes]\label{thm:online_conformal_converge}
    Under the assumptions and notation above, if $\eta_t= \eta$ for some $\eta>0$, we have
    \[\liminf_{t\rightarrow \infty} F_{s_t}(q_t) = 0, \quad \limsup_{t\rightarrow \infty} F_{s_t}(q_t) = 1\]
    almost surely.
    On the other hand, if $\eta_t$ is a nonnegative sequence with $\sum_t \eta_t = \infty$ and $\sum_t \eta_t^2 <\infty$, then for any continuous and strictly increasing CDF $F$, it holds almost surely that 
    \[\textnormal{If $F_{s_t}\stackrel{\rm d}{\rightarrow} F$ then $F_{s_t}(q_t)\rightarrow 1-\alpha$.}\]
\end{theorem} \index{quantile tracking}
In other words, a constant step size leads to infinitely many oscillations between undercoverage and overcoverage (e.g., $\cC_t(X_t)=\varnothing$ and $\cC_t(X_t)=\cY$). On the other hand, an appropriately chosen decaying sequence of step sizes, such as $\eta_t \propto t^{-(1/2+\epsilon)}$, for some small $\epsilon>0$, leads to convergence to the desired coverage rate, as long as the score functions $s_t$ converge (for instance, if $s_t(x,y) = |y-\hf_t(x)|$ is a residual score for some fitted regression model $\hf_t$, then we are effectively assuming that the $\hf_t$'s converge to some fixed function).
\index{coverage!upper bound}

\index{nonexchangeability|)}
\index{adversarial sequence|)}
\index{time series|)}
\section*{Bibliographic notes}
\addcontentsline{toc}{section}{\protect\numberline{}\textnormal{\hspace{-0.8cm}Bibliographic notes}}

Our exposition of the exchangeable online setting follows that of~\citet{vovk2005algorithmic}. The independence of online p-values for exchangeable sequences (Theorem~\ref{thm:online-indep-pvals}) is a foundational result due to~\citet{vovk2002line}. 
Testing exchangeability online using conformal p-values as in Theorem~\ref{thm:supermartingale-test} (or relatedly, testing for outliers) is introduced in~\citet{vovk2003testing}; see also \citet{fedorova2012plug,laxhammar2015inductive, volkhonskiy2017inductive, vovk2021retrain, shaer2026testing, hore2026conformal}. 
The martingale we present in Proposition~\ref{prop:simple-conformal-supermartingale} is only for illustrative purposes; more powerful and efficient martingales are available in~\citet{vovk2005algorithmic} and~\citet{ramdas2022testing}. 
This work builds on martingale theory, particularly Ville's inequality~\citep{ville1939etude}. 
We refer the reader to~\citet{williams1991probability} for an introduction to martingales.
This strategy of using supermartingales in a sequential data setting can be leveraged for tasks beyond testing exchangeability.
In particular, we note that the exchangeability supermartingale in Proposition~\ref{prop:simple-conformal-supermartingale}, and the more general construction in Theorem~\ref{thm:supermartingale-test}, are each a special case of an e-value, and these are useful more broadly; see \cite{shafer2001probability,ramdas2023game} for background on the field of game-theoretic statistics, and see \citet{ramdas2024hypothesis} for a recent survey on e-values.

The adversarial sequence model is commonly studied in online learning~\citep[e.g.,][]{hazan2016introduction}. Conformal prediction with coverage guarantees in this setting (see Section~\ref{sec:ACI_and_quantile_tracking}) is introduced in~\citet{gibbs2021adaptive} under the name \emph{adaptive conformal inference} (ACI). Subsequent work by~\citet{zaffran2022adaptive, gibbs2022conformal, bhatnagar2023improved} studies notions of adaptivity and considers adaptive step sizes, and~\citet{bastani2022practical} discusses conformal prediction with subgroup guarantees in the online setting.
The quantile tracking algorithm here is studied  in~\citet{feldman2023achieving, angelopoulos2023conformalpid}.  
The coverage result in Theorem~\ref{thm:quantile-tracking} is established in \citet{angelopoulos2024online}, extending the coverage guarantees in the above works to allow for decaying step sizes. This same work also proves the results for the i.i.d.\ setting, given in Theorem~\ref{thm:online_conformal_converge}; these results are extended to a setting where the model is asymptotically well-specified (but data is not required to be i.i.d.) by \cite{areces2025online}.
Online conformal prediction also connects with online optimization, regret minimization, and optimal betting, as established in~\cite{gibbs2022conformal, bhatnagar2023improved, zhang2024discounted, podkopaev2024adaptive, angelopoulos2025gradient, srinivas2026online, liu2026online}.

\section*{Exercises}
\addcontentsline{toc}{section}{\protect\numberline{}\textnormal{\hspace{-0.8cm}Exercises}}
\begin{enumerate}[font=\bfseries, label={\thechapter.\arabic*}, labelsep=1em, itemsep=1em]
\item Theorem~\ref{thm:online-indep-pvals} establishes that, when running online full conformal prediction, the conformal p-values $p_t$ are independent as long as we assume that there are no ties among scores, almost surely. We will now examine the role of this last assumption. Construct an example of a distribution $P$ and a symmetric score function $s$, such that for data points $(X_t,Y_t)\iidsim P$, the $p_t$'s are not independent.
\item Consider the setting of Theorem~\ref{thm:supermartingale-test}. Suppose in addition that $X_1,\dots,X_{T}\iidsim \textnormal{Unif}[0,1]$ and $X_{T + 1},\dots, X_{2T}\iidsim\textnormal{Unif}[1,2]$, and that $Y_t = 0$ for all $t$.  Consider the score function $s(x,y) = x$. 

    What is the joint distribution of $(p_{T+1},\dots,p_{2T})$? Explain why we should expect that a test for exchangeability would have high power to detect that this is not exchangeable when $2T$ is large.
\item Consider the online conformal prediction setting of Section~\ref{sec:ACI_and_quantile_tracking}, with i.i.d.\ data as in the setting of Theorem~\ref{thm:online_conformal_converge}. In this exercise we will prove a partial version of that theorem, for the case of a constant step size.
    
    Consider a fixed score function: $s_t = s$ for all $t$, taking values in a bounded range, $s(x,y)\in[0,B]$. Assume $\P_P(s(X,Y) \leq B-\epsilon)<1$ for any $\epsilon>0$. Suppose we run the quantile tracking algorithm as defined in~\eqref{eq:quantile-tracking} with a constant step size $\eta_t = \eta = 1$ for all $t$, and with target coverage level $1-\alpha = 0.9$. We initialize at any $q_1\in[0,B]$.
    Prove that, for some sufficiently large $C$ and some $\delta>0$, the following holds: for any $t\geq 1$,
    \[\P\left(\max\{q_t,\dots,q_{t+C}\} \geq B\,\middle|\, (X_1,Y_1),\dots(X_{t-1},Y_{t-1})\right) \geq \delta.\]
    (Verifying this claim is a key step towards showing that, almost surely, we will have $\cC_t(X_t) = \cY$ infinitely often.)
\end{enumerate}

\chapter{Additional Results in Conformal Prediction}
\label{chapter:further-topics}

In this chapter, we return to the core conformal prediction framework of Chapter~\ref{chapter:conformal-exchangeability}, taking a closer look at some interesting details of the algorithm. We discuss the role of randomization, addressing how to use conformal prediction with randomized algorithms such as stochastic gradient descent. We next turn to computational results for full conformal prediction, explaining cases where the set can be computed exactly, and exploring valid ways to employ discretization to maintain coverage. Lastly, we turn to the universality of conformal prediction: any method with distribution-free validity that is permutation invariant is equivalent to a conformal method. Each section in this chapter stands alone, but together they supply many additional insights about conformal prediction.

\section{Conformal prediction with randomization}
\label{sec:cp-random}
\index{randomization|(}

Thus far, we have presented full conformal prediction as a deterministic algorithm---the output $\cC(X_{n+1})$ is a deterministic function of the training data $\big((X_i,Y_i)\big)_{i\in[n]}$ and test point $X_{n+1}$. In practice, however,
we may want to allow for the score $s$ to be determined with a randomized algorithm. For example, for a residual score of the form $s((x,y);\cD)=|y-\hf(x;\cD)|$, in some settings we may want to fit the function $\hf(\cdot;\cD)$ with a randomized regression algorithm, such as stochastic gradient descent. Section~\ref{sec:cp-random-1} will extend our framework to allow for the possibility of pairing conformal prediction with a randomized algorithm.

Randomization can also enter into the conformal prediction framework in a different form. Specifically, we recall from Chapter~\ref{chapter:conformal-exchangeability} that the (full or split) conformal prediction intervals are guaranteed to have at least $1-\alpha$ level coverage; as explored in Theorem~\ref{thm:upper-bound}, overcoverage can arise from the possibility of ties among scores and/or from the fact that $(1-\alpha)(n+1)$ is not an integer. Introducing randomization into the construction of the conformal prediction set removes this overcoverage, resulting in exactly $1-\alpha$ coverage, as we will see in Section~\ref{sec:cp-random-2}.

\subsection{Allowing for randomization in the score function}\label{sec:cp-random-1}
We first consider the case where our score functions may 
contain auxiliary randomness, such as a model trained using
stochastic gradient descent.
In this case, the requirement of a symmetric score function, as in Definition~\ref{def:symmetric_score}, needs to be modified.

To be more concrete, let $\xi\in[0,1]$ denote a stochastic noise term that we use for randomization---we will refer to this as the random seed. The score function is now randomized: given a data point $(x,y)$, a dataset $\cD$, and a random seed value $\xi$, we can now compute a score 
\[s((x,y);\cD,\xi).\]
Returning to our example of stochastic gradient descent, we might write $\hf(\cdot;\cD,\xi)$ to denote the fitted model when stochastic gradient descent is run on dataset $\cD$ with random seed $\xi$ (in practice, this means that $\xi$ determines the order in which we cycle through points in the dataset $\cD$, when taking gradient descent steps for our optimization problem). Then we can define a score such as
$s((x,y);\cD,\xi) = |y - \hf(x;\cD,\xi)|$.

What does it mean for the score to be symmetric in this randomized definition? Here, Definition~\ref{def:symmetric_score} is replaced by a requirement that the score function is symmetric in a distributional sense: 
\begin{definition}[Symmetric randomized score function]\label{def:symmetric_randomized_score}
A randomized score function $s$ is symmetric if for any fixed dataset $\cD=((x_1,y_1),\dots,(x_m,y_m))$, any fixed test points $(x'_1,y'_1),\dots,(x'_{m'},y'_{m'})$, and any permutation $\sigma\in\cS_m$,
\begin{equation}
\label{eq:symmetric_randomized_score}
\Big(s\big((x'_i,y'_i);\cD,\xi\big)\Big)_{i\in[m']}\eqd \Big(s\big((x'_i,y'_i);\cD_\sigma,\xi\big)\Big)_{i\in[m']}
\end{equation}
where $\cD_\sigma = ((x_{\sigma(1)},y_{\sigma(1)}),\dots,(x_{\sigma(m)},y_{\sigma(m)}))$ denotes the permuted dataset, and where the equality in distribution holds with respect to drawing the random seed as $\xi\sim\textnormal{Unif}[0,1]$.  
\end{definition}
Note that the property in~\eqref{eq:symmetric_randomized_score} is weaker 
than requiring that $s\big((x'_i,y'_i);\cD,\xi\big) = s\big((x'_i,y'_i);\cD_\sigma,\xi\big)$ for all test points $(x_i',y_i')$ (i.e., requiring equality, rather than equality in distribution). This 
distinction is important in practice. With stochastic gradient descent, for instance, for a fixed random seed $\xi$ the fitted function will in general be changed if we permute the data---if 
the random seed $\xi$ determines that the first gradient descent step is taken with respect to the $i$th data point for some particular $i$, this data point is equal to $(X_i,Y_i)$ in the dataset $\cD$, but is instead equal to $(X_{\sigma(i)},Y_{\sigma(i)})$ for the permuted dataset $\cD_\sigma$. Thus, we would generally have different fitted models, $\hf(\cdot;\cD,\xi)\neq \hf(\cdot;\cD_\sigma,\xi)$, leading to different score functions, $s(\cdot;\cD,\xi)\neq s(\cdot;\cD_\sigma,\xi)$---but equality in distribution (as in~\eqref{eq:symmetric_randomized_score}) does hold.

This symmetry condition will allow us to verify that marginal coverage holds for conformal prediction when using a randomized score function.
\begin{theorem}[Marginal coverage of conformal prediction with a randomized score]
    \label{thm:full-conformal-randomized-score}
    Suppose that $(X_1,Y_1),...,(X_{n+1},Y_{n+1})$ are exchangeable and that $s$ is a symmetric randomized score function. Let $\xi\sim\textnormal{Unif}[0,1]$ be drawn independently of the data.
    Define the prediction set 
    \[\cC(X_{n+1}) = \left\{y\in\cY : s((X_{n+1},y);\cD^y_{n+1},\xi) \leq \hat{q}^y(\xi)\right\}\]
    where $\cD^y_{n+1} = ((X_1,Y_1),\dots,(X_n,Y_n),(X_{n+1},y))$ and
    \[\hat{q}^y(\xi) = \quantile\left(s((X_1,Y_1);\cD^y_{n+1},\xi),\dots,s((X_n,Y_n);\cD^y_{n+1},\xi) ;(1-\alpha)(1+1/n)\right).\]
    Then $\cC(X_{n+1})$ satisfies
    \begin{equation}
        \P\left( Y_{n + 1} \in \cC(X_{n + 1}) \right) \geq 1-\alpha.
    \end{equation}
\end{theorem}
\begin{proof}[Proof of Theorem~\ref{thm:full-conformal-randomized-score}]
    Define $S_i = s((X_i,Y_i);\cD_{n+1},\xi)$, for each $i\in[n+1]$, where $\cD_{n+1} = ((X_i,Y_i))_{i\in[n+1]}$ as before.
    As for the proof of Theorem~\ref{thm:full-conformal}, which establishes the marginal coverage property of full conformal prediction (without randomization), it suffices to verify that $S_1,\dots,S_{n+1}$ are exchangeable (recall the key step~\eqref{eqn:scores_exchangeable_n+1} in the proof of Theorem~\ref{thm:full-conformal}).

    For any permutation $\sigma\in\cS_{n+1}$, let $(\cD_{n+1})_\sigma = ((X_{\sigma(1)},Y_{\sigma(1)}),\dots,(X_{\sigma(n+1)},Y_{\sigma(n+1)}))$ denote the permuted dataset. We then have
    \begin{multline*}(S_{\sigma(1)},\dots,S_{\sigma(n+1)})
= (s((X_{\sigma(i)},Y_{\sigma(i)});\cD_{n+1},\xi))_{i\in[n+1]}\\ \eqd (s((X_{\sigma(i)},Y_{\sigma(i)});(\cD_{n+1})_\sigma,\xi))_{i\in[n+1]},\end{multline*}
where the equality in distribution in the last step holds  by applying Definition~\ref{def:symmetric_randomized_score} with the permuted data points $(X_{\sigma(1)},Y_{\sigma(1)}),\dots,(X_{\sigma(n+1)},Y_{\sigma(n+1)})$ playing the role of the test points $(x_1',y_1'),\dots,(x'_{m'},y'_{m'})$. But by exchangeability of the data, we also have
\[\Big(s((X_{\sigma(i)},Y_{\sigma(i)});(\cD_{n+1})_\sigma,\xi)\Big)_{i\in[n+1]}\eqd \Big(s((X_i,Y_i);\cD_{n+1},\xi)\Big)_{i\in[n+1]}=(S_1,\dots,S_{n+1}).\]
This verifies that the scores $S_1,\dots,S_{n+1}$ are exchangeable in this randomized setting, as desired.
\end{proof}

From this point on, throughout the book, for clarity of the presentation we will generally only consider deterministic score functions, but in general, all results can easily be generalized to encompass randomized score functions.

\subsection{Using a randomized calibration step}\label{sec:cp-random-2}
Next, we will consider a different type of randomization in the construction of the conformal prediction set for the purpose of achieving prediction sets with exact coverage, rather than conservative coverage.
While in Section~\ref{sec:cp-random-1} we considered randomization in the construction of the scores, here we instead consider randomization in comparing the values of the scores, which will help handle ties.
Recall from Proposition~\ref{prop:conformal-via-pvalues} that the conformal prediction set can equivalently be defined as
\[\cC(X_{n+1}) = \{y\in\cY : p^y>\alpha\},\]
where the conformal p-value is given by
\[p^y = \frac{\sum_{i=1}^{n+1} \ind{S^y_i \geq S^y_{n+1}}}{n+1}\]
(as in Definition~\ref{def:conformal-pvalue}).
If the scores are distinct, then exchangeability ensures that the p-value $p^{Y_{n+1}}$, corresponding to the true test response value $y=Y_{n+1}$, is uniformly distributed on the grid $\{\frac{1}{n+1},\dots,\frac{n}{n+1},1\}$, reflecting the fact that the scores $S_1,\dots,S_{n+1}$ are equally likely to be ranked in any order (see Theorem~\ref{thm:online-indep-pvals}). This means that the coverage event
$p^{Y_{n+1}}>\alpha$ holds with probability exactly $\frac{k}{n+1}$, where $k = \lceil(1-\alpha)(n+1)\rceil$, which leads to overcoverage if $(1-\alpha)(n+1)$ is not an integer. Moreover, this calculation has not accounted for ties between the scores---another potential source of overcoverage.

To avoid this issue, we can define a randomized version of the conformal prediction set for which coverage will hold at exactly level $1-\alpha$. This modification is sometimes referred to as \emph{smoothing}: the distribution of $p^{Y_{n+1}}$, which is supported on the grid $\{\frac{1}{n+1},\dots,\frac{n}{n+1},1\}$, is smoothed to achieve a $\textnormal{Unif}[0,1]$ distribution.

The smoothed version of the prediction set is easiest to define using the p-value-based construction. Given a hypothesized value $y\in\cY$ for the test response, we redefine the conformal p-value as
\begin{equation}\label{eqn:cp-p-value-randomized}p^y(\xi) = \frac{\sum_{i=1}^{n+1} \ind{S^y_i > S^y_{n+1}} + \xi \cdot \sum_{i=1}^{n+1} \ind{S^y_i = S^y_{n+1}}}{n+1},\end{equation}
which now depends additionally on a randomization term $\xi\sim \textnormal{Unif}[0,1]$. The conformal prediction set is then given by
\begin{equation}\label{eqn:define-cp-smoothed}\cC(X_{n+1}) = \left\{y\in\cY : p^y(\xi) > \alpha\right\},\end{equation}
where the dependence on the random $\xi$ is implicit in the notation. Operationally, it is common to use a single draw of $\xi$ that is shared across all values $y\in\cY$; that is, p-values $p^y(\xi)$ and $p^{y'}(\xi)$ are computed using the same value of $\xi$.

\begin{theorem}[Exact coverage for full conformal prediction with smoothing]\label{thm:cp-smoothed}
    Suppose that $(X_1,Y_1),...,(X_{n+1},Y_{n+1})$ are exchangeable and that $s$ is a symmetric score function. Let $\xi\sim\textnormal{Unif}[0,1]$ be drawn independently of the data.
    Then the smoothed conformal prediction set $\cC$ defined in~\eqref{eqn:define-cp-smoothed} has marginal coverage level exactly $1-\alpha$,
\[        \P\left( Y_{n + 1} \in \cC(X_{n + 1}) \right) = 1-\alpha.\]
\end{theorem}\index{coverage!upper bound}
The idea of the proof is to verify the equivalent statement that the smoothed conformal p-value is uniformly distributed, $p^{Y_{n+1}}(\xi)\sim\textnormal{Unif}[0,1]$.

\begin{proof}[Proof of Theorem~\ref{thm:cp-smoothed}]
The argument follows the same structure as the permutation-test-based proof of Theorem~\ref{thm:full-conformal}, which was given in Section~\ref{sec:conformal_as_perm}. We have
\[Y_{n+1}\in\cC(X_{n+1})  \ \Longleftrightarrow \ p^{Y_{n+1}}(\xi) >\alpha,\]
where by definition, we have
\[ p^{Y_{n+1}}(\xi) = \frac{\sum_{i=1}^{n+1} \ind{S_i > S_{n+1}} + \xi \cdot \sum_{i=1}^{n+1} \ind{S_i = S_{n+1}}}{n+1}.\]
As in the proof of Theorem~\ref{thm:full-conformal}, we know that the scores
\[S_1,\dots,S_{n+1}\]
are exchangeable. Therefore, our last remaining step is to prove the validity of a randomized version of permutation testing, which is established with the following lemma:
\begin{lemma}\label{lem:randomize-perm}
    Let $S_1,\dots,S_{n+1}$ be exchangeable, and let $\xi\sim\textnormal{Unif}[0,1]$ be drawn independently from $S_1,\dots,S_{n+1}$. Define
    \[p(\xi) = \frac{\sum_{i=1}^{n+1} \ind{S_i > S_{n+1}} + \xi \cdot \sum_{i=1}^{n+1} \ind{S_i = S_{n+1}}}{n+1}.\]
    Then $\P(p(\xi)\leq \tau)= \tau$ for all $\tau\in[0,1]$.
\end{lemma} \index{conformal p-value}
The lemma immediately implies that $\P(p^{Y_{n+1}}(\xi) \leq \alpha)= \alpha$, which completes the proof of the theorem.
\end{proof}
To complete this section, we prove the lemma.
\begin{proof}[Proof of Lemma~\ref{lem:randomize-perm}]
We can assume $\tau<1$ to avoid the trivial case.
Let
\[q = \quantile(S_1,\dots,S_{n+1};1-\tau),\]
and define
\[K = \sum_{i=1}^{n+1}\ind{S_i \leq q}, \ L = \sum_{i=1}^{n+1} \ind{S_i <q}.\]
By definition of the quantile for a finite list (recall Fact~\ref{fact:conversion-order-stats-quantiles}), we have
\[ K \geq (1-\tau)(n+1) > L.\]
Next, we split into cases: by construction of the randomized p-value, we can verify that 
\begin{align*}
     S_{n+1} < q & \Longrightarrow \ p(\xi)  > \tau,\\
     S_{n+1} > q & \Longrightarrow \ p(\xi)  \leq \tau,\\
     S_{n+1} =q & \Longrightarrow \ p(\xi)  =\frac{n+1 - K + \xi(K-L)}{n+1}.
     \end{align*}
Therefore,
\begin{align*}&\P(p(\xi)\leq \tau \mid S_1,\dots,S_{n+1})\\& = \ind{S_{n+1}>q} + \ind{S_{n+1}=q}\cdot \P\left(\frac{n+1 - K + \xi(K-L)}{n+1}\leq \tau\right) \\&= \ind{S_{n+1}>q} + \ind{S_{n+1}=q}\cdot \frac{K-(n+1)(1-\tau)}{K-L}.\end{align*}
Note also that $q,K,L$ can all be expressed as functions of $\widehat{P}_{n+1}=\frac{1}{n+1}\sum_{i=1}^{n+1}\delta_{S_i}$, the empirical distribution of $S_1,\dots,S_{n+1}$, and therefore, by Proposition~\ref{prop:empirical-distrib-exch},
\[\P(S_{n+1}\leq q\mid \widehat{P}_{n+1}) =\frac{K}{n+1}, \quad \P(S_{n+1}<q \mid \widehat{P}_{n+1}) = \frac{L}{n+1}.\]
Therefore,
\begin{align*}\P(p(\xi)\leq \tau) &= \E\left[\ind{S_{n+1}>q} + \ind{S_{n+1}=q}\cdot \frac{K-(n+1)(1-\tau)}{K-L}\right] \\&= \E\left[\P(S_{n+1}>q\mid \widehat{P}_{n+1}) + \P(S_{n+1}=q\mid \widehat{P}_{n+1}) \cdot \frac{K-(n+1)(1-\tau)}{K-L}\right] \\&= \E\left[\frac{n+1-K}{n+1} + \frac{K-L}{n+1}\cdot \frac{K-(n+1)(1-\tau)}{K-L}\right] = \tau.\end{align*}
\end{proof}
\index{randomization|)}

\section{Computational shortcuts for full conformal prediction}
\label{sec:computational-shortcuts}
As we have seen, full conformal prediction is computationally expensive---for each possible value $y\in\cY$ of the test response $Y_{n+1}$, full conformal requires evaluating the score function $s(\cdot;\cD^y_{n+1})$, which generally requires refitting a model. For example, when the score function is of the form $s((x,y);\cD) = |y - \hf(x;\cD)|$, we must fit the model $\hf$ on many different datasets. While this might be easy in the case that $\cY$ is a small finite set of possible labels, this becomes prohibitively expensive if $Y$ is categorical but has a large number of possible values, or even impossible if $Y$ is continuous.

In this section, we will focus on the case $\cY=\R$.  First we will examine some special cases where, due to the particular construction of the score, the full conformal prediction interval can be computed exactly. We will then turn to a general strategy that uses discretization to allow for efficient computation regardless of the choice of score function.

\subsection{Special case: linear regression}\label{sec:special_case_linear_regression}
\index{score function!residual score|(}
\index{linear regression|(}

First, we examine linear regression, in the low-dimensional case where $\cX = \R^d$ for some $d\leq n$. This is an important case and is the first step toward more flexible
models such as ridge regression and Lasso.
We restrict attention to the score function given by the residual for the least-squares regression model: we have
\[s((x,y);\cD) = |y-\hf(x;\cD)|,\]
where the fitted model is given by
\[\hf(x;\cD) = x^\top\hat\beta \textnormal{ for }\hat\beta= \argmin_{\beta\in\R^d}\left\{\sum_{i=1}^m(y_i-x_i^\top\beta)^2\right\}\textnormal{ where }\cD = ((x_1,y_1),\dots,(x_m,y_m)).\]

\begin{proposition}[Computing the full prediction set for linear regression methods]\label{prop:full_CP_leastsq}
Let $\cX = \R^d$ and $\cY=\R$.
Fix a training dataset $\cD_n=((X_i,Y_i))_{i\in[n]}$ and test point $X_{n+1}$, and 
let $X_{[n+1]}\in\R^{(n+1)\times d}$ denote the matrix with $i$th row given by $X_i$, for each $i\in[n+1]$. Assume $X_{[n+1]}$ has full column rank, i.e., rank $d$. Suppose the score function is given by the residual score for least-squares regression, as defined above. 
Then the full conformal prediction set is equal to
\begin{equation}\label{eqn:CP_set_lin_reg}\cC(X_{n+1}) = \left\{y \in\R :\frac{ 1+ \sum_{i=1}^n \ind{|y \cdot (-b_i) - (a_i - Y_i)| \geq |y\cdot (1-b_{n+1}) - a_{n+1}|}}{n+1} > \alpha\right\}, \end{equation}
where we define $a,b\in\R^{n+1}$ with entries
\[a_i = \sum_{j\in[n]} H_{ij}Y_j, \quad b_i = H_{i,n+1}\]
for matrix $H\in\R^{(n+1)\times(n+1)}$  defined 
as
\[ H = X_{[n+1]}\left(X_{[n+1]}^\top X_{[n+1]}\right)^{-1}X_{[n+1]}^\top .\]
\end{proposition}
To see why this reformulation of the prediction $\cC(X_{n+1})$ leads to an efficient calculation, we note that
\[y\mapsto \sum_{i=1}^n \ind{|y \cdot (-b_i) - (a_i - Y_i)| \geq |y\cdot (1-b_{n+1}) - a_{n+1}|}\]
is a piecewise constant function, and its changepoints can only occur at values $y$ that satisfy
\[y \cdot (-b_i) - (a_i - Y_i) = \pm\big(y\cdot (1-b_{n+1}) - a_{n+1}\big)\]
for some $i=1,\dots,n$.
Therefore, after solving $2n$ univariate linear equations to identify the candidate values $y$ that might be changepoints of the above function, we can derive the exact set $\cC(X_{n+1})$.
See Figure~\ref{fig:ols-full-conformal} for an illustration.
\begin{figure}[t]
    \centering
    \includegraphics[width=0.55\linewidth]{\diagramspath ols-full-conformal.pdf}
    \caption{\textbf{Visualization of the conformal p-value resulting from the least-squares predictor.} The top plot shows the residuals on the augmented dataset as a function of the hypothesized label $y$. The residuals are all linear in $y$. The emphasized line corresponds to the test residual. In the middle plot, we show the absolute residuals on the augmented dataset.
    The bottom plot shows the conformal p-values resulting from the residual score. Since the conformal p-value only changes when the test point's absolute residual crosses any one of the gray lines, we do not have to check all $y \in \R$ to compute the prediction set $\{y \in \cY : p^y > \alpha\}$; it suffices to consider only the $y$ values at the crossing points.}
    \commentAlt{The top panel shows many linear functions, and the middle panel shows absolute values of these functions. The bottom panel, labeled `p-value', shows a step function whose changepoints align with intersections of the absolute value functions.}
    \label{fig:ols-full-conformal}
\end{figure}

\begin{proof}[Proof of Proposition~\ref{prop:full_CP_leastsq}]
Given training data $\cD_n=((X_i,Y_i))_{i\in[n]}$ and test point $X_{n+1}$, the score function is given by
\[s((x',y');\cD^y_{n+1}) = \left|y' - x'{}^\top \hat\beta^y\right|, \]
where
\[\hat\beta^y = \argmin_{\beta\in\R^d}\left\{\sum_{i=1}^n (Y_i - X_i^\top\beta)^2 + (y - X_{n+1}^\top\beta)^2\right\}\]
is the fitted coefficient vector for least-squares regression. 
We can rewrite this as
\[\hat\beta^y = (X_{[n+1]}^\top X_{[n+1]})^{-1} X_{[n+1]}^\top \left(\begin{array}{c}Y_1 \\ \vdots \\ Y_n \\ y\end{array}\right).\] For each $i\in[n+1]$, then,
\[X_i^\top \hat\beta^y = X_i^\top 
(X_{[n+1]}^\top X_{[n+1]})^{-1} X_{[n+1]}^\top \left(\begin{array}{c}Y_1 \\ \vdots \\ Y_n \\ y\end{array}\right) = \sum_{j\in[n]} H_{ij}\cdot  Y_j + H_{i,n+1} \cdot y = a_i + b_iy,\]
by plugging in the definitions of $H$, $a$, and $b$ from the statement of the proposition.
We can then calculate the scores of the training and test data points as
\[S^y_i = s((X_i,Y_i);\cD^y_{n+1}) = |Y_i - X_i^\top\hat\beta^y| = |y \cdot (-b_i) - (a_i - Y_i)|, \]
for $i=1,\dots,n$, and\[S^y_{n+1} = s((X_{n+1},y);\cD^y_{n+1}) = |y - X_{n+1}^\top\hat\beta^y| = |y\cdot (1-b_{n+1}) - a_{n+1}|.\]

Next, following the reinterpretation of full conformal given in Section~\ref{sec:conformal_as_perm}, we have
\[\cC(X_{n+1}) = \{y\in\R: p^y>\alpha\},\]
where we compute the conformal p-value as
\begin{multline*}p^y = \frac{1 + \sum_{i\in[n]}\ind{S^y_i \geq S^y_{n+1}}}{n+1} \\= \frac{1 + \sum_{i\in[n]}\ind{|y \cdot (-b_i) - (a_i - Y_i)| \geq |y\cdot (1-b_{n+1}) - a_{n+1}|}}{n+1}.\end{multline*}
This completes the proof.
\end{proof}

\paragraph{Related special cases.}
\index{ridge regression}
\index{nearest-neighbor regression}
\index{kernel regression}

The example of linear regression can be generalized to a broader class of regression algorithms (while still using the residual score); here we mention some of these special cases briefly, without going into details.  
Suppose that, for all $x'\in\cX$,  the prediction $\hf(x';\cD^y_{n+1})$ depends linearly on $y$---that is,
\begin{equation}\label{eqn:hat_f_is_linear_in_y}\hf(x';\cD^y_{n+1}) = a(x';\cD_n,X_{n+1}) + b(x';\cD_n,X_{n+1})\cdot y,\end{equation}
where $a(x';\cD_n,X_{n+1})$ and $b(x';\cD_n,X_{n+1})$ can depend arbitrarily on their arguments. In addition to linear regression, methods in this class include ridge regression, nearest-neighbor regression, and various kernel regression methods.
In this type of setting, after calculating $a_i=a(X_i;\cD_n,X_{n+1})$ and $b_i=b(X_i;\cD_n,X_{n+1})$ for each $i=1,\dots,n+1$, which we only need to compute once as these values do not depend on $y$, we then have $\cC(X_{n+1})$ given by the calculation~\eqref{eqn:CP_set_lin_reg}, exactly as for linear regression, 
and so the same calculation strategy can be applied to compute this set. 
\index{linear regression|)}

\subsection{Special case: Lasso regression}\label{sec:fullCP-Lasso}
\index{lasso|(}

Another special case arises if the regression algorithm is given by the Lasso, i.e., $\ell_1$-penalized linear regression,
where the fitted model is now given by
\[\hf(x;\cD) = x^\top\hat\beta \textnormal{ for }\hat\beta= \argmin_{\beta\in\R^d}\left\{\frac{1}{2}\sum_{i=1}^m(y_i-x_i^\top\beta)^2 + \lambda\|\beta\|_1\right\}\textnormal{ where }\cD = ((x_1,y_1),\dots,(x_m,y_m)).\]
The addition of the $\ell_1$ penalty allows for accurate estimation even in the high-dimensional setting, under the assumption of sparsity. We again consider the residual score,
\[s((x,y);\cD) = |y-\hf(x;\cD)|.\]

This choice of the base algorithm is more computationally complex than the special case of linear regression---the Lasso does not satisfy a linearity property along the lines of~\eqref{eqn:hat_f_is_linear_in_y}. Nonetheless,
efficient computation of $\cC(X_{n+1})$ can be achieved 
by exploiting the observation that the map $y\mapsto \hf(x';\cD^y_{n+1})$ is still \emph{piecewise} linear in $y$. We will give some intuition for how we may calculate the prediction interval, without a formal derivation.

Write
\[\hat\beta^y = \argmin_{\beta\in\R^d}\left\{\frac{1}{2}\sum_{i=1}^n (Y_i - X_i^\top\beta)^2 + \frac{1}{2}(y - X_{n+1}^\top\beta)^2 + \lambda\|\beta\|_1\right\},\]
the coefficient vector returned by running Lasso on the dataset $\cD^y_{n+1}$. Define also the support,
\[\hat{I}^y = \{j\in[d] : \hat\beta^y_j\neq 0\}\subseteq [d],\]
and the signs,
\[\hat\gamma^y = (\textnormal{sign}(\hat\beta^y_j))_{j\in \hat{I}^y}\in \{\pm 1\}^{|\hat{I}^y|}.\]

For any $y\in\R$, the first-order optimality conditions for the solution $\hat\beta^y$ to the Lasso optimization problem yield the following identity:
\[\textnormal{If $(\hat{I}^y,\hat\gamma^y) = (I,\gamma)$ then }\hat{f}(x';\cD^y_{n+1}) = x'_I{}^\top \left(X_{[n+1],I}^\top X_{[n+1],I}\right)^{-1}\left(X_{[n+1],I} ^\top \left(\begin{array}{c}Y_1 \\ \vdots \\ Y_n \\ y\end{array}\right) - \lambda \gamma\right),\]
were $X_{[n+1]}\in\R^{(n+1)\times d}$ is the matrix with $i$th row given by $X_i$, as in Proposition~\ref{prop:full_CP_leastsq}, while the notation $X_{[n+1],I}$ indicates that we take the submatrix given by columns $I\subseteq[d]$. 
In particular this implies that 
if we know the support and signs of the solution $\hat\beta^y$ then the prediction $\hat{f}(x';\cD^y_{n+1})$ depends linearly on $y$. Moreover, the nature of the Lasso optimization problem ensures that the support and signs are piecewise constant (that is, the map $y\mapsto (\hat{I}^y,\hat\gamma^y)$ is piecewise constant over $y\in\R$), and consequently, the map $y\mapsto \hf(x';\cD^y_{n+1})$ is piecewise linear in $y$: we can write
\begin{equation}\label{eqn:hat_f_is_piecewise_linear_in_y}\hf(x';\cD^y_{n+1}) = a_r(x';\cD_n,X_{n+1}) + b_r(x';\cD_n,X_{n+1})\cdot y \textnormal{ for all $y\in B_r$},\end{equation}
where $\R=B_1\cup \dots \cup B_R$ is some partition of $\R$ into intervals. This can be viewed as a generalization of the previous linear setting~\eqref{eqn:hat_f_is_linear_in_y}, and once we have computed the piecewise linear representation~\eqref{eqn:hat_f_is_piecewise_linear_in_y}, calculating $\cC(X_{n+1})$ follows similar steps as before.

\index{lasso|)}
\index{score function!residual score|)}

\subsection{A general approach: discretization}\label{sec:discretize_fullCP}
\index{discretized conformal prediction|(}

Moving beyond simple cases like linear regression and the Lasso, in general settings where the regression method and/or score function may be arbitrarily complex, we may no longer be able to derive computational shortcuts for computing $\cC(X_{n+1})$ efficiently. Therefore, we now consider a more general approach, where the full conformal prediction set is approximated via discretization.

\paragraph{Naive discretization.}
Since $\cC(X_{n+1})$ is impossible to compute exactly for $\cY=\R$ aside from special cases like the above, in practice it is common to use an informal discretization strategy, approximating $\R$ with a fine grid. The strategy is typically implemented as follows. Let $y^{(1)}<\dots<y^{(M)}$ be a prespecified list of values---typically these values are chosen to have equal spacing, $y^{(m+1)}-y^{(m)}=\Delta$ for each $m=1,\dots,M-1$. For each grid point $m\in[M]$, we then determine whether the full conformal prediction set contains the hypothesized value $y^{(m)}$: 
\[\cI = \left\{ m\in[M] \, : \, 
 s((X_{n+1},y^{(m)});\cD^{y^{(m)}}_{n+1}) \leq \quantile\left((s((X_i,Y_i);\cD^{y^{(m)}}_{n+1}))_{i\in[n]};(1-\alpha)(1+1/n)\right)\right\}.\]
Then we compute an approximation to the full conformal prediction set by constructing the set
\begin{equation}\label{eqn:full_CP_grid_informal}\cC(X_{n+1}) = \bigcup_{m\in\cI}\big(y^{(m-1)}, y^{(m+1)}\big).\end{equation}
In other words, if we observe that $y^{(m)}$ needs to be included in the prediction interval $\cC(X_{n+1})$, then we also include any value $y$ that lies in between grid points $y^{(m-1)}$ and $y^{(m)}$, and any value $y$ that lies in between grid points $y^{(m)}$ and $y^{(m+1)}$, since we are not certain of where the true boundary of the set lies. (For completeness, we can define $y^{(0)}=-\infty$ and $y^{(M+1)}=+\infty$.)

While this strategy generally works well in practice, we cannot verify a distribution-free guarantee for this approximate method. The reason is that the training data points $i\in[n]$ and the test point $n+1$ are no longer being treated symmetrically: when computing our score functions, we use data points $(X_i,Y_i)$ for each training point $i\in[n]$, but the test point $(X_{n+1},Y_{n+1})$ is never used, unless by coincidence we happen to have $Y_{n+1}=y^{(m)}$ for one of the preselected grid points.

\paragraph{Symmetry-preserving discretization.}
The following result addresses this issue, by modifying the method to restore symmetry.
At a high level, it works by discretizing all the points, not just the test point, thus yielding a symmetric algorithm.
\begin{proposition}[Coverage for discretized conformal prediction]\label{prop:discretized_CP}
Let $y_1,\dots,y_M\in\cY$ be a prespecified collection of values and let $k:\cY\rightarrow[M]$ be any function. (Intuitively, we think of this as a rounding function, where for an observed response $Y\in\cY$, the nearest grid point is $y_{k(Y)}$.) Let $s$ be a symmetric score function, and for each $m\in[M]$, define a function
\[s_m(x,y) = s\Big((x;y);\big( (X_1,y_{k(Y_1)}),\dots,(X_n,y_{k(Y_n)}), (X_{n+1},y_m)\big)\Big).\]
Define
\begin{multline*}\cC(X_{n+1}) = \\\bigcup_{m\in[M]} \left\{y\in\cY : k(y) = m , \ s_m(X_{n+1},y)\leq \quantile\big((s_m(X_i,Y_i))_{i\in[n]};(1-\alpha)(1+1/n)\big)\right\}.\end{multline*}
Then this prediction set satisfies distribution-free marginal coverage, i.e., if the data points $(X_1,Y_1),\dots,(X_{n+1},Y_{n+1})$ are exchangeable then
\[\P(Y_{n+1}\in\cC(X_{n+1}))\geq 1-\alpha.\]
\end{proposition}

As for the informal grid-based approximation defined in~\eqref{eqn:full_CP_grid_informal},
the cost of constructing this new discretized full conformal prediction interval is simply the
cost of $M$ many runs of the modeling algorithm. For example, if we use a
residual score function, then writing $\hf_m$ to denote the fitted model when we run our regression algorithm on data points $(X_1,y_{k(Y_1)}),\dots,(X_n,y_{k(Y_n)}), (X_{n+1},y_m)$, the $m$th trained score function $s_m$ is given by $s_m(x,y) = |y-\hf_m(x)|$. This means that, to compute $\cC(X_{n+1})$, we need to run our regression algorithm $M$ many times (as compared to infinitely many times---at least in theory---for the original full conformal method).

The choice of $M$ will affect both the statistical performance and the computational cost of this procedure. If $M$ is chosen to be very small, then the procedure is computationally much more efficient, but the trained models $\hf_m$ will be much less accurate since all the response values are rounded---that is, each $Y_i$ is replaced with a value on the grid, $y_{k(Y_i)}$. We can therefore expect to see a tradeoff between the computation time and the width of the resulting prediction intervals, with higher values of $M$ typically leading to more precise prediction intervals at the cost of additional computation.

\begin{proof}[Proof of Proposition~\ref{prop:discretized_CP}]
The key step for the proof is to observe that the proposed prediction set $\cC(X_{n+1})$, which is an 
approximation of the full prediction set that we would obtain if we use $s$ as our score function, is actually exactly equal to the full conformal prediction set for a different choice of score function.

To make this precise, we define a new score function $\tilde{s}$ as follows.
For any dataset $\cD = ((x'_1,y'_1),\dots,(x'_\ell,y'_\ell))$ and data point $(x',y')$, let
\[\tilde{s}\big((x',y');\cD\big) = 
s\big((x',y'); (x'_1,y_{k(y'_1)}),\dots,(x'_\ell,y_{k(y'_\ell)})\big).\]
In other words, the score function $\tilde{s}$ is obtained by implementing the original score function $s$,
but on a modified version of the dataset $\cD$, obtained by rounding all $Y$ values to the grid $\{y_1,\dots,y_M\}$
 via the map $k:\cY\rightarrow[M]$. Note that $\tilde{s}$ is a symmetric score, due to the symmetry of $s$.

Now suppose that we run full conformal prediction with $\tilde{s}$ as our score function. Let $\tilde\cC(X_{n+1})$
be the resulting prediction set, so that we have
\[\P(Y_{n+1}\in \tilde\cC(X_{n+1}))\geq 1-\alpha\]
by Theorem~\ref{thm:full-conformal}, since the data is exchangeable and the score function $\tilde{s}$ is symmetric.
For any $y\in\cY$, write
\[\tilde{S}^y_i = \tilde{s}((X_i,Y_i); \cD^y_{n+1})\]
for $i\in[n]$, and 
\[\tilde{S}^y_{n+1} = \tilde{s}((X_{n+1},y); \cD^y_{n+1}).\]
Then we have
\[\tilde\cC(X_{n+1}) = \left\{y\in\cY : \tilde{S}^y_{n+1} \leq \quantile\big((\tilde{S}^y_i)_{i\in[n]}; (1-\alpha)(1+1/n)\big)\right\},\]
by definition of full conformal prediction. Since the sets $\{y\in\cY : k(y)=m\}$, indexed by $m\in[M]$, form a partition
of the response space $\cY$, we can equivalently write
\begin{equation}\label{eqn:cC_for_stilde_discretizedCP}\tilde\cC(X_{n+1}) = \bigcup_{m\in[M]} \left\{y\in\cY : k(y) = m, \ \tilde{S}^y_{n+1} \leq \quantile\big((\tilde{S}^y_i)_{i\in[n]}; (1-\alpha)(1+1/n)\big)\right\}.\end{equation}

Now fix any $m\in[M]$. By construction of $\tilde{s}$, we can verify that, for any $y\in\cY$ with $k(y) = m$, we have
\begin{multline*}\tilde{s}(\cdot ; \cD^y_{n+1}) = s(\cdot; (X_1,y_{k(Y_1)}),\dots,(X_n,y_{k(Y_n)}), (X_{n+1}, y_{k(y)}))
\\= s(\cdot; (X_1,y_{k(Y_1)}),\dots,(X_n,y_{k(Y_n)}), (X_{n+1}, y_m)) = s_m(\cdot).\end{multline*}
Therefore, for any $y\in\cY$ with $k(y)=m$,
\[\tilde{S}^y_i = \begin{cases} s_m(X_i,Y_i), & i\in[n], \\ s_m(X_{n+1},y), & i = n+1.\end{cases}\]
Thus, we can see that the prediction set $\cC(X_{n+1})$ defined in the statement of the proposition is exactly equal to the set $\tilde\cC(X_{n+1})$ given in~\eqref{eqn:cC_for_stilde_discretizedCP}, and consequently,
\[\P(Y_{n+1}\in\cC(X_{n+1})) = \P(Y_{n+1}\in\tilde\cC(X_{n+1}))\geq 1-\alpha.\]
\end{proof}

\index{discretized conformal prediction|)}

\section{The universality of conformal prediction}\label{sec:universality}
\index{universality of conformal prediction|(}

Throughout this book, we have seen that conformal prediction provides a strategy for ensuring distribution-free marginal predictive coverage under only an assumption of exchangeability. But are there alternative methods that might achieve the same goal---and, perhaps, offer more informative prediction intervals?

The following result proves a universality property of full conformal prediction. It demonstrates that any method achieving distribution-free marginal coverage must actually be equivalent to running full conformal prediction (with some choice of score function), as long as we assume that the method is symmetric in the training data.

\begin{theorem}[Universality of conformal prediction]\label{thm:universality}
    Let $\cC$ be any predictive inference procedure, which maps any training dataset and test point to a prediction interval (or prediction set), 
    \[(\cD, x) \mapsto \cC(x;\cD)\subseteq \cY.\]
    Assume $\cC$ is symmetric in the training data, i.e., $\cC(x;\cD)=\cC(x;\cD_\sigma)$ for any $x\in\cX$, any dataset $\cD$, and any permutation $\sigma$. Assume also that $\cC$ satisfies distribution-free predictive coverage at level $1-\alpha$:
    \[\textnormal{If $(X_1,Y_1),\dots,(X_{n+1},Y_{n+1})$ are exchangeable then $\P(Y_{n+1}\in\cC(X_{n+1};\cD_n))\geq 1-\alpha$}.\]
    Then there exists a symmetric conformal score function $s$ such that $\cC$ is equal to the full conformal prediction interval constructed with this score function.
\end{theorem} \index{exchangeability}

This result says that the conformal prediction methodology is so broad---due to the flexibility in choosing the score function $s$---that any method with distribution-free validity can be expressed as a special case. Therefore, if we are restricting ourselves to symmetric methods, it is not fruitful to ask whether we can find a distribution-free method that will perform better than conformal prediction; instead, we should simply search for a better score function if we are not satisfied with the performance of conformal prediction on a particular application.

However, the symmetry assumption is important. As an example of a distribution-free method that does not satisfy symmetry, we can recall the weighted version of conformal prediction (with fixed weights), covered in Chapter~\ref{chapter:weighted-conformal} (see Section~\ref{sec:fixed-weights}); this method achieves marginal coverage under exchangeability, but cannot be expressed as an instance of (unweighted) conformal prediction. 

\begin{proof}[Proof of Theorem~\ref{thm:universality}]
First we define some notation.
Given a dataset $\cD$ and a data point $(x,y)$, we define $\cD_{\setminus(x,y)}$ as the dataset obtained by removing one copy of $(x,y)$, if $(x,y)$ appears (one or more times) in $\cD$, or otherwise simply returning $\cD$.

Now define the score function $s$ as follows: for any dataset $\cD'$ and any point $(x,y)$, let
\[s((x,y);\cD') = \ind{y\not\in\cC(x; \cD'_{\setminus (x,y)})}.\]
The assumption of symmetry on $\cC$ implies that $s$ is a symmetric score function.
Given training data $\cD_n = ((X_i,Y_i))_{i\in[n]}$ and test point $X_{n+1}$, we now claim that $\cC(X_{n+1};\cD_n) = \cC_{\textnormal{CP}}(X_{n+1})$, where $\cC_{\textnormal{CP}}(X_{n+1})$ denotes the output of full conformal prediction (Algorithm~\ref{alg:full-cp}) run with the score function $s$ defined above, and with $\cD_n$ as the training data.

As usual for full conformal, for any $y\in\cY$, let $\cD^y_{n+1}$ denote the augmented dataset,
\[\cD^y_{n+1} = ((X_1,Y_1),\dots,(X_n,Y_n),(X_{n+1},y)),\]
and note that $(\cD^y_{n+1})_{\setminus (X_{n+1},y)} = \cD_n$ (or, we may have that $(\cD^y_{n+1})_{\setminus (X_{n+1},y)}$ is equal to some permutation of $\cD_n$, in the case that the data point $(X_{n+1},y)$ happens to be equal to one of the training points $(X_i,Y_i)$).

By definition of full conformal prediction, we have
\[\cC_{\textnormal{CP}}(X_{n+1}) = \{y\in\cY : S^y_{n+1} \leq \quantile(S^y_1,\dots,S^y_n; (1-\alpha)(1+1/n))\},\]
where 
\[S^y_i = s((X_i,Y_i);\cD^y_{n+1}), \ i\in[n], \quad S^y_{n+1} = s((X_{n+1},y);\cD^y_{n+1}).\]
In particular, by definition of $s$, we see that 
\[S^y_{n+1} = \ind{y\not\in\cC(X_{n+1};(\cD^y_{n+1})_{\setminus (X_{n+1},y)})} = \ind{y\not\in\cC(X_{n+1};\cD_n)},\]
where for the last step we use the fact that $(\cD^y_{n+1})_{\setminus (X_{n+1},y)}$ is equal to $\cD_n$ up to a permutation, and $\cC$ is a symmetric procedure.

Below, we will verify the following claim:
\begin{equation}\label{eqn:universality_claim}
\textnormal{For any $y\in\cY$, \ $\quantile(S^y_1,\dots,S^y_{n+1}; 1-\alpha) = 0$.}\end{equation}
Assuming that this is true, then for any $y\in\cY$, we have
\begin{multline*}S^y_{n+1} = 0 \ \Longleftrightarrow \ 
 S^y_{n+1} \leq \quantile(S^y_1,\dots,S^y_{n+1}; 1-\alpha) \\ \ \Longleftrightarrow \ 
 S^y_{n+1} \leq \quantile(S^y_1,\dots,S^y_n; (1-\alpha)(1+1/n)) \ \Longleftrightarrow \ y\in\cC_{\textnormal{CP}}(X_{n+1}),\end{multline*}
 where the first step holds by~\eqref{eqn:universality_claim} together with the fact that $S^y_{n+1}\in\{0,1\}$ by definition, the second step holds by the Replacement Lemma (Lemma~\ref{lem:n+1-to-n-reduction}), and the last step holds by definition of the full conformal prediction set. But by our calculations above, $S^y_{n+1}=0$ if and only if $y\in\cC(X_{n+1};\cD_n)$. Therefore, we have verified that $\cC(X_{n+1};\cD_n)=\cC_{\textnormal{CP}}(X_{n+1})$.

 To complete the proof, we will now need to verify~\eqref{eqn:universality_claim}. This step relies on the distribution-free coverage property of $\cC$.
In fact, this is an immediate consequence of the following deterministic lemma:
\begin{lemma}\label{lem:lemma_for_universality}
    Let $\cD=((x_1,y_1),\dots,(x_m,y_m))$ be any dataset. Then, if $\cC$ satisfies the assumptions of Theorem~\ref{thm:universality}, it holds that
    \[\sum_{i=1}^m\ind{y_i\in \cC(x_i;\cD_{\setminus (x_i,y_i)})} \geq (1-\alpha)m.\]
In particular, defining $s_i=\ind{y_i \not\in \cC(x_i;\cD_{\setminus (x_i,y_i)})}\in\{0,1\}$ for each $i\in[m]$, it holds that $\quantile(s_1,\dots,s_m;1-\alpha)=0$.
\end{lemma}
The claim~\eqref{eqn:universality_claim} then follows by applying this lemma, with $m=n+1$ and with $\cD$ given by the augmented dataset $\cD^y_{n+1}=((X_1,Y_1),\dots,(X_n,Y_n),(X_{n+1},y))$.
\end{proof}

\begin{proof}[Proof of Lemma~\ref{lem:lemma_for_universality}]
Let $\sigma\in\cS_m$ be a permutation drawn uniformly at random, and define 
\[(\tilde{X}_i,\tilde{Y}_i) = (x_{\sigma(i)},y_{\sigma(i)}), \ i\in[m].\]
That is, the new dataset $(\tilde{X}_1,\tilde{Y}_1),\dots,(\tilde{X}_m,\tilde{Y}_m)$ is obtained by sampling uniformly without replacement from the list $(x_1,y_1),\dots,(x_m,y_m)$, and therefore satisfies exchangeability. Distribution-free validity of $\cC$ then implies that
\[\P\big(\tilde{Y}_m\in\cC(\tilde{X}_m;((\tilde{X}_i,\tilde{Y}_i))_{i\in[m-1]})\big) \geq 1-\alpha.\]
We can rewrite this as
\[1-\alpha\leq \P\big(y_{\sigma(m)}\in\cC(x_{\sigma(m)};((x_{\sigma(i)},y_{\sigma(i)}))_{i\in[m-1]})\big) =\P\big(y_{\sigma(m)}\in\cC(x_{\sigma(m)};\cD_{\setminus (x_{\sigma(m)},y_{\sigma(m)})}\big) ,\]
where $\sigma\in\cS_m$ is drawn uniformly at random. Here the last step holds since we  have assumed symmetry of $\cC$, and, up to permutation, the dataset $((x_{\sigma(i)},y_{\sigma(i)}))_{i\in[m-1]}$ is equal to $((x_i,y_i))_{i\in[m]\setminus \{\sigma(m)\}}$, which (again up to permutation) is equal to $ \cD_{\setminus (x_{\sigma(m)},y_{\sigma(m)})}$.
Finally, since the $(x_i,y_i)$'s are treated as fixed while $\sigma$ is a uniformly random permutation (and thus $\sigma(m)$ is distributed uniformly over $\{1,\dots,m\}$), we have
\[\P\big(y_{\sigma(m)}\in\cC(x_{\sigma(m)};\cD_{\setminus (x_{\sigma(m)},y_{\sigma(m)})})\big) = \frac{1}{m}\sum_{i=1}^m \ind{y_i\in\cC(x_i;\cD_{\setminus (x_i,y_i)})},\]
as desired.
\end{proof}

\index{universality of conformal prediction|)}

\section*{Bibliographic notes}
\addcontentsline{toc}{section}{\protect\numberline{}\textnormal{\hspace{-0.8cm}Bibliographic notes}}

Conformal prediction with a randomized score function, presented in Section~\ref{sec:cp-random-1}, appears in \cite{vovk2005algorithmic}. That same work also defines the smoothed version of the conformal prediction interval (described in Section~\ref{sec:cp-random-2}), and establishes this algorithm's guarantee of exact coverage, given in Theorem~\ref{thm:cp-smoothed}.

The closed-form calculation of the full conformal prediction interval $\cC(X_{n+1})$ in the special case of linear regression (as in Proposition~\ref{prop:full_CP_leastsq}) and related methods, such as ridge regression, is due to \cite{nouretdinov2001ridge}. \cite{lei2019fast} develops the piecewise linear homotopy approach for computing $\cC(X_{n+1})$ in the case of Lasso (i.e., $\ell_1$-penalized linear regression \citep{tibshirani1996regression}), which is presented in Section~\ref{sec:fullCP-Lasso}. The properties of the Lasso that are leveraged for this construction are related to those developed earlier in the selective inference literature, e.g., \cite{tibshirani2016exact,lee2014exact}.

The discretized version of full conformal prediction, presented in Proposition~\ref{prop:discretized_CP}, appears in the work of \cite{chen2018discretized}. Related works in the literature include
\cite{chen2016trimmed}, which examines the question of reducing the range of values $y\in\cY$ that need to be considered for inclusion in the prediction set $\cC(X_{n+1})$; \cite{ndiaye2019}, which proposes a homotopy-based algorithm for more efficient computation of $\cC(X_{n+1})$; and \cite{Ndiaye2022root-finding}, which proposes a root-finding algorithm in general settings where the model $\hf(\cdot;\cD)$ is fitted via solving an optimization problem. Furthermore, as mentioned in the bibliographic notes of Chapter~\ref{chapter:cv}, \cite{ndiaye2022stable} 
shows that assuming a certain type of algorithmic stability can also enable a computationally efficient approximation for the full conformal prediction set.

The universality of conformal prediction, stated in Theorem~\ref{thm:universality}, appears in \citet{vovk2005algorithmic} (see also \citet{vovk2025universality,vovk2026randomness} for extensions to the setting of methods that guarantee coverage only for i.i.d.\ data, rather than for all exchangeable distributions). This result can be viewed as a modern perspective on classical results demonstrating that a permutation test is universal as a test of exchangeability (see, e.g., \cite{scheffe1943measure}).

\section*{Exercises}
\addcontentsline{toc}{section}{\protect\numberline{}\textnormal{\hspace{-0.8cm}Exercises}}
\begin{enumerate}[font=\bfseries, label={\thechapter.\arabic*}, labelsep=1em, itemsep=1em]
\item This exercise will examine a randomized version of full conformal prediction based on subsampling. 
Suppose that instead of comparing the test score $S^y_{n+1}$ against all $n$ training scores $S^y_1,\dots,S^y_n$, we compare against a random subsample, $\{S^y_i\}_{i\in \cI}$. (Here $\cI\subseteq[n]$ is sampled at random, e.g., by choosing a random subset of $[n]$ of a fixed size.) The prediction interval is then given by

$$\cC(X_{n+1}) = \left\{y \in\cY  : S^y_{n+1}\leq\quantile\left( (S^y_i)_{i\in\cI} ; (1-\alpha)(1+1/|\cI|)\right)\right\}.$$

Prove marginal coverage for this modified interval, under the usual assumptions of exchangeability of the data and symmetry of the score function, and assuming also that $\cI$ is sampled independently of the data.
\item In this exercise we consider an additional example of a score function for which the full conformal prediction set can be computed efficiently. Consider $K$-nearest-neighbor regression in the setting $\cY=\R$: the fitted model is given by
\[\hf(x;\cD) =\frac{1}{K}\sum_{i\in N(x,\cD)}y_i,\]
where for a dataset $\cD = ((x_1,y_1),\dots,(x_{n+1},y_{n+1}))$, the set $N(x,\cD)\subseteq[n+1]$ indexes the $K$ nearest neighbors of $x$ among $x_1,\dots,x_{n+1}$ (with ties handled via any tie-breaking rule, if needed). 
Show that this regression algorithm can be expressed in the form of~\eqref{eqn:hat_f_is_linear_in_y} (which, as explained in Section~\ref{sec:special_case_linear_regression}, means that the full conformal prediction set can be computed efficiently).
\item Consider the setting of a real-valued response, $\cY=\R$. Fix any sample size $n\geq 1$. For any integer $M\geq 1$ and any $B>0$, and let $y^{(1)} < \dots < y^{(M)}$ denote a uniform grid over the range $[-B,B]$, i.e., $y^{(1)}=-B$, $y^{(2)} = -B + \frac{2B}{M-1}$, etc. Construct an example of a distribution $P$ and a symmetric score function $s$, such that the naive discretization strategy for constructing $\cC(X_{n+1})$ (as defined in~\eqref{eqn:full_CP_grid_informal}) has zero marginal coverage, for any choice of $M$ and $B$.
\item Let $\cA$ be any regression algorithm, which inputs a dataset $\cD=((x_1,y_1),\dots,(x_m,y_m))$ and returns a fitted model $\hf(\cdot;\cD)$. We do not assume that $\cA$ is symmetric. Define a randomized score function $s((x,y);\cD,\xi)$ as follows. Let $\xi\mapsto\sigma_\xi\in\cS_m$ be a map such that, if $\xi\sim\textnormal{Unif}[0,1]$, then $\sigma_\xi$ is uniformly distributed over $\cS_m$. Then define the score function
    \[s((x,y);\cD,\xi) = \left|y - \hf(x; \cD_{\sigma(\xi)})\right|.\]
    Show that this randomized score function $s$ is symmetric (i.e., satisfies Definition~\ref{def:symmetric_randomized_score}).
\end{enumerate}

\chapter{Extensions of Conformal Prediction}
\label{chapter:extensions}

We now consider other prediction settings beyond creating a single uncertainty set for a single prediction. The algorithms we study in this section extend conformal prediction outside the traditional setting and highlight opportunities for further work. We begin with conformal risk control, which is a generalization of conformal prediction to handle notions of error other than miscoverage. This yields conformal-type algorithms applicable for prediction tasks where the output is a high-dimensional object, such as in image segmentation, natural language tasks, and so on. Next, we consider multiplicity and outlier detection, explaining how split conformal p-values have a positive dependence that still enables control of the false discovery rate. We then turn to the issue of selective coverage, where we seek to maintain validity even when we focus on a selected subset of test data, rather than marginal coverage. Lastly, we consider the topic of conformal prediction when aggregating multiple models.

\section{Beyond miscoverage: conformal risk control}
\label{sec:conformal_risk_control}
\index{risk control|(}

Conformal prediction guarantees predictive coverage---that $Y_{n+1} \in \cC(X_{n+1})$ with some desired probability---but this notion of accuracy may not be well suited for certain settings.
Particularly in machine learning, the space $\cY$ may be structured such that a failure of coverage, $Y_{n+1}\not\in\cC(X_{n+1})$, is not a particularly informative notion of error. 
For example, $Y$ can be very high-dimensional (e.g., a high-resolution image), in which case it is too strict to cover every coordinate of $Y$.
For such tasks, other notions of error can be more helpful, such as:
\begin{itemize}
    \item \textbf{Coordinatewise coverage rate.} If $Y_{n+1} \in \R^d$, we may want to return a prediction set such that most coordinates of the response $Y_{n+1}$ are contained in their respective prediction intervals, i.e., we want the coordinatewise miscoverage rate
    \begin{equation}\label{eqn:risk_control_loss_example_1}\frac{1}{d} \sum_{j=1}^d \ind{(Y_{n+1})_j \not\in \cC(X_{n+1})_j}\end{equation}
    to be bounded in expectation by some target error level $\alpha$, where our prediction set is of the form $\cC(X_{n+1}) = \cC(X_{n+1})_1\times \dots \times \cC(X_{n+1})_d$.

    \item \textbf{Accuracy in hierarchical classification.} In hierarchical classification, the classes $\cY$ are the leaf nodes of a hierarchy---for example, a `poodle' is an instance of a `dog', which is an instance of an `animal', and so on. We consider making predictions that correspond to nodes in the hierarchy---predicting a leaf node (`poodle') is a high-resolution prediction, whereas an interior node (`animal') is a low-resolution prediction.
    Here, we wish to bound the probability of event that we have made an error,
    \begin{equation}\label{eqn:risk_control_loss_example_2}
        \ind{\hat Y_{n+1} \textnormal{ is not an ancestor of } Y_{n+1} \textnormal{ in the hierarchy}} ,
    \end{equation}
    by some target level $\alpha$.

    \item \textbf{False negative rate.} Consider a setting where $Y_{n+1} \in \{0,1\}^d$,
    with a $1$ indicating the presence of some signal (e.g., the $j$th pixel of the image contains a vehicle). Then
    we might output a vector $\hat{Y}_{n+1} \in \{0, 1\}^d$ indicating where we believe the true signals lie, and would like to require that our set of (estimated) $1$'s contains most of the positives. Here, we would aim to control the false negative rate:
    \begin{equation}\label{eqn:risk_control_loss_example_3}
        \frac{\sum_{j=1}^d \ind{(Y_{n+1})_j = 1, (\hat{Y}_{n+1})_j=0}}{\sum_{j=1}^d \ind{(Y_{n+1})_j=1}}
    \end{equation}
    to be bounded (in expectation) by some target level $\alpha$.
\end{itemize}

Note that marginal predictive coverage can also be expressed as control of the expectation of miscoverage indicator,
\begin{equation}\label{eqn:risk_control_loss_example_4}
\ind{Y_{n+1}\not\in \cC(X_{n+1})}.
\end{equation}
Thus, marginal miscoverage, and the three other examples above, can all be thought of as a \emph{risk}---the expected value of a loss.
The basic idea of conformal risk control is to replace the indicator of miscoverage with other loss functions, to generalize beyond the goal of marginal coverage.
Below we present the extension to arbitrary monotone loss functions.

We now introduce our notation. Let $\cC_\lambda$ denote the set-valued output of a procedure, with some tuning parameter $\lambda$---for instance, $\cC_\lambda(X)$ might denote a prediction interval for the response $Y$. The tuning parameter $\lambda\in\Lambda\subseteq\R$ indexes a nested family of sets, with
\begin{equation}
\label{eq:nested_sets}
    \lambda_1 \leq \lambda_2 \implies \cC_{\lambda_1}(x) \subseteq \cC_{\lambda_2}(x).
\end{equation}
In particular, split conformal prediction with a pretrained score function $s$ corresponds to choosing $\cC_\lambda(x)=\{y\in\cY:s(x,y)\leq\lambda\}$.

As our notion of statistical error, let $L$ be a {\em loss function} on prediction sets, which is assumed to be monotone in $\cC$,
\begin{equation}
\cC \subseteq \cC' \implies L(y, \cC) \ge L(y, \cC').
\label{eq:monotone_loss}
\end{equation}
i.e., larger sets lead to smaller loss. 
For instance, for conformal prediction we have the miscoverage loss $L(y,\cC) = \ind{y\notin\cC}$ (as in~\eqref{eqn:risk_control_loss_example_4}), but this notation can accommodate more general settings such as the examples above. 
We hope to choose a value of $\lambda$ so that the confidence set $\cC_\lambda$ satisfies
\begin{equation}
\label{eq:e_risk_control}
    \E\left[L(Y_{n+1}, \cC_\lambda(X_{n+1}))\right] \le \alpha,
\end{equation}
a property we refer to as \emph{risk control}.

As hinted above, we can extend the split conformal prediction algorithm in this more general setting. The algorithm is as follows. Consider a plug-in estimator of the expected loss at each value of the tuning parameter $\lambda\in\Lambda$, given by
\begin{equation}
    \hat{R}(\lambda) = \frac{1}{n} \sum_{i=1}^n L\left(Y_i, \cC_\lambda(X_i)\right).
\end{equation}
Intuitively, we would expect that we can then choose a value of $\lambda$ satisfying $\hat{R}(\lambda)\leq\alpha$ to achieve the desired bound on expected loss, but this does not lead to risk control due to finite-sample fluctuations in $\hat{R}(\lambda)$. To adjust for this and obtain the guarantee in~\eqref{eq:e_risk_control}, our procedure needs to be slightly more conservative. For a loss $L$ that takes values in $[0,1]$, we select a value of $\lambda$ where the plug-in estimate of risk is slightly smaller than $\alpha$: \index{conformal quantile}
\begin{equation}
    \label{eq:lhat}
    \hat{\lambda} = \inf \left\{ \lambda : \hat{R}(\lambda) \le \alpha - (1-\alpha) / n \right\}.
\end{equation}
See Figure~\ref{fig:risk-control} for a visualization of this parameter selection, and how it relates to conformal prediction: it is exactly equivalent to the selection of the conformal quantile in the case of the miscoverage loss.

\begin{figure}[t]
    \centering
    \includegraphics[width=0.8\linewidth]{\diagramspath risk-control.pdf}
    \caption{\textbf{Visualization of split conformal risk control as it compares to split conformal prediction.} The left-hand plot shows the quantile calculation for split conformal prediction, reframed in terms of miscoverage risk. The right-hand plot shows the parameter selection for conformal risk control for a smooth risk.}
    \commentAlt{The left panel shows a decreasing piecewise constant function, and the right panel shows a continuous strictly decreasing function. See long description.}
    \commentLongAlt{The left panel shows a decreasing piecewise constant function, and the right panel shows a continuous strictly decreasing function. On the left, the vertical axis ranges from $0$ to $1$ and is labeled `empirical miscoverage', with level $\alpha-\frac{1-\alpha}{n}$ marked by a dashed line. The horizontal axis is labeled `scores', and the value $\hat{q}$ is shown where the level is achieved. On the right, the vertical axis ranges from $0$ to $1$ and is labeled `empirical risk', again with level $\alpha-\frac{1-\alpha}{n}$ marked. The horizontal axis is labeled $\lambda$, and the value $\hat\lambda$ is shown where the level is achieved.}
    \label{fig:risk-control}
\end{figure}

The next result shows that this procedure provides risk control. For simplicity, we will assume $\Lambda=\R$. In order to ensure that $\hat\lambda$ is well-defined, we will assume that $\alpha - (1-\alpha) / n>0$, and that for any $(x,y)$, it holds that $\lim_{\lambda\to\infty}L(y,\cC_\lambda(x))=0$ and $\lim_{\lambda\to-\infty}L(y,\cC_\lambda(x))=1$.
\begin{theorem}[Conformal risk control]
\label{thm:risk_control}
 Suppose that $(X_1, Y_1),\dots, (X_{n+1}, Y_{n+1})$ are exchangeable, and that the loss $L$ takes values in $[0,1]$. Assume that the map $\lambda\mapsto L(y,\cC_\lambda(x))$ is right-continuous and is monotone nonincreasing, for any $(x,y)$.
 Then, with $\hat{\lambda}$ selected as in~\eqref{eq:lhat}, we have
 \begin{equation}
     \E\left[L(Y_{n+1}, \cC_{\hat\lambda}(X_{n+1}))\right] \le \alpha.
 \end{equation}
\end{theorem}
Here, the expectation is over both the calibration data (i.e., the randomness in $\hat \lambda$) and the test point, as with conformal prediction.
Note that the monotonicity condition (i.e., requiring $\lambda\mapsto L(y,\cC_\lambda(x))$ to be nonincreasing) is an immediate consequence of assuming nested sets $\cC_\lambda$, as in~\eqref{eq:nested_sets}, and monotonicity of the loss $L$, as in~\eqref{eq:monotone_loss}. The assumption of right-continuity holds in many practical instances, such as split conformal and the other examples above.

In the case of the miscoverage loss, the conformal risk control procedure exactly recovers the split conformal prediction method (as seen in Proposition~\ref{prop:conformal-equivalent-plugin}), and the result of Theorem~\ref{thm:risk_control} exactly recovers the marginal predictive coverage guarantee for split conformal prediction.

\begin{proof}[Proof of Theorem~\ref{thm:risk_control}]

\textbf{Step 1: Deterministic inequalities on monotone risks.} Consider an arbitrary dataset $\cD = ((x_1,y_1),\dots,(x_m,y_m))$.
We will denote the empirical risk on the dataset as 
\begin{equation}
     \hat{R}(\lambda ; \cD) = \frac{1}{m} \sum\limits_{i=1}^{m}L(y_i, \cC_{\lambda}(x_i)),
\end{equation}
and define the following family of thresholds:
\begin{equation}
    \hat{\lambda}(\cD, \beta) = \inf\left\{ \lambda : \hat{R}(\lambda ; \cD)  \leq \beta \right\},
\end{equation}
for $\beta\in(0,1)$.
The right-continuity of $L$ implies that $\hat{R}(\lambda ; \cD)$ is right-continuous for any $\cD$, giving us the following deterministic inequality for any $\beta$:
\begin{equation}
    \label{eq:deterministic-risk-threshold}
    \hat{R}\left( \hat{\lambda}(\cD; \beta); \cD \right) \leq \beta.
\end{equation}

\textbf{Step 2: Applying exchangeability.}
The family $\hat\lambda(\cD,\beta)$, defined above, generalizes our previous choice of threshold in~\eqref{eq:lhat}: we have that $\hat{\lambda} = \hat{\lambda}(\cD_n, \alpha') $, where $\cD_n=((X_1,Y_1),\dots,(X_n,Y_n))$ and $\alpha' = \alpha - (1-\alpha)/n$. However, here we will consider applying this definition to a different dataset---not to the calibration set $\cD_n$, but instead, to the dataset $\cD_{n+1}=((X_i,Y_i))_{i\in[n+1]}$ containing both the calibration set and test point: define
\[\lambda^* = \hat\lambda(\cD_{n+1}, \alpha)\]
(note that here we use the original target risk level $\alpha$, not the modified level $\alpha'$).
Since the data points are exchangeable, and $\lambda^*$ is a symmetric function of the dataset $\cD_{n+1}$, we have
\[\E\left[L(Y_{n+1},C_{\lambda^*}(X_{n+1}))\right]
=\E\left[L(Y_i,C_{\lambda^*}(X_i))\right], \]
for all $i\in[n+1]$ (since, by Lemma~\ref{lem:conditional_exchangeability}, the data points $(X_1,Y_1),\dots,(X_{n+1},Y_{n+1})$ are exchangeable when we condition on $\lambda^*$, and in particular this implies $L(Y_{n+1},C_{\lambda^*}(X_{n+1}))\eqd L(Y_i,C_{\lambda^*}(X_i))$). Therefore,
\[\E\left[L(Y_{n+1},C_{\lambda^*}(X_{n+1}))\right]
= \frac{1}{n+1}\sum_{i=1}^{n+1}\E\left[L(Y_i,C_{\lambda^*}(X_i))\right] = \E\left[\hat{R}(\lambda^*;\cD_{n+1})\right]\leq \alpha,\]
where the first step holds by exchangeability, the second step by definition of $\hat{R}$, and the third step by~\eqref{eq:deterministic-risk-threshold}.

\textbf{Step 3: proving the bound.}
Since we have assumed $L$ takes values in $[0,1]$, we have 
\[\hat{R}(\lambda;\cD_{n+1})\leq \frac{1}{n+1} + \frac{n}{n+1}\hat{R}(\lambda;\cD_n).\] 
Therefore, for any $\lambda$, by definition of $\alpha'$ we have
\[\hat{R}(\lambda;\cD_n)\leq \alpha' \ \Longrightarrow \ \hat{R}(\lambda;\cD_{n+1})\leq \alpha,\]
and so it must hold that $\lambda^*\leq \hat\lambda$. Monotonicity of $L$ implies that $L(Y_{n+1},C_{\hat\lambda}(X_{n+1}))\leq L(Y_{n+1},C_{\lambda^*}(X_{n+1}))$ holds surely, and therefore
\[\E\left[L(Y_{n+1},C_{\hat\lambda}(X_{n+1}))\right]\leq \E\left[L(Y_{n+1},C_{\lambda^*}(X_{n+1}))\right]\leq \alpha.\]
\end{proof}
\index{risk control}

\section{Family-wise error rate for multiple test points}\label{sec:multiplicity_prediction_FWER}
\index{multiple testing|(}
\index{family-wise error rate|(}

Conformal prediction is typically presented as a method for constructing a prediction interval for a single test point $X_{n+1}$, given a training set of size $n$.
However, it is often of interest to provide prediction intervals for multiple test points in a batch.
This is a question of multiplicity: do we need to adjust our construction of the prediction sets to account for the fact that we are performing inference for multiple test points?

Throughout this section, we will use our previous notation for the split conformal prediction setting: we write $\cD_n=((X_i,Y_i))_{i\in[n]}$ to denote the calibration set, while $s:\cX\times\cY\to\R$ is the pretrained score function. 
Given a set of test data points $X'_1, \ldots, X'_m$, we might be interested in obtaining a guarantee of the form
\begin{equation}\label{eqn:FWER_goal}
    \P\big( Y'_1 \in \cC(X'_1), \ldots, Y'_m \in \cC(X'_m) \big) \ge 1-\alpha_{\textnormal{FWER}},
\end{equation}
meaning \emph{all} of the $m$ prediction sets cover (with probability at least $1-\alpha_{\textnormal{FWER}}$). 
Here $\alpha_{\textnormal{FWER}}$ is a bound on the family-wise error rate: the probability of \emph{at least one} miscoverage among our $m$ prediction sets.

Let us now consider how the split conformal prediction method would perform relative to this goal. Since the prediction set is given by $\cC(x) = \{y\in\cY: s(x,y)\leq \hat{q}\}$, 
where $\hat{q} = \quantile(S_1,\dots,S_n;(1-\alpha)(1+1/n))$, 
we can rephrase the goal~\eqref{eqn:FWER_goal} above as
\begin{equation}
    \label{eq:multiplicity-scores}
    \P\left( S'_1 \leq \hat{q}, \ldots, S'_m \leq \hat{q} \right) \ge 1-\alpha_{\textnormal{FWER}},
\end{equation}
where the test point scores are given by $S'_i = s(X'_i,Y'_i)$.
The key challenge is that the events $ \{S'_i \leq \hat{q} \}_{i \in [m]}$ are not independent, since the conformal quantile $\hat{q}$ (which is random) provides a threshold that is shared across all $m$ events.

In this next result, we consider both the finite-sample setting, where the number of calibration points $n$ is fixed, and an asymptotic regime,
where the number of calibration points is $n\rightarrow\infty$ (while the number of test points $m$, and the pretrained score function $s$, remain fixed). To accommodate both of these settings,
we will write $\hat{q}_n$ (rather than $\hat{q}$) to emphasize that $n$ may be increasing.
\begin{proposition}[The FWER of split conformal prediction (asymptotic)]
    \label{prop:asymptotic-fwer-conformal}
    Let $(X'_1,Y'_1),\ldots,(X'_m,Y'_m),(X_1,Y_1),(X_2,Y_2),\dots$ be sampled i.i.d.\ from some distribution $P$. Let $s:\cX\times\cY\to\R$ be a pretrained score function, and for each calibration set size $n\geq 1$, define the split conformal prediction set as
    \[\cC_{\hat{q}_n}(x) = \{y:s(x,y)\leq\hat{q}_n\}\]
    where
    \[\hat{q}_n = \quantile\big(S_1,\dots,S_n;(1-\alpha)(1+1/n)\big)\]
    for scores $S_i = s(X_i,Y_i)$, $i\in[n]$. Define $q^* = \quantile(F,1-\alpha)$, where $F$ is the CDF of $s(X,Y)$ when $(X,Y)\sim P$.  
    
    Then for any fixed $n,m\geq 1$,
    \begin{equation}
        \P\left( Y'_1 \in \cC_{\hat{q}_n}(X'_1), \ldots,  Y'_m \in \cC_{\hat{q}_n}(X'_m) \right) \geq (1-\alpha)^m.
    \end{equation}
    Moreover, if $F(q)$ is continuous and strictly increasing at $q=q^*$, then for any fixed $m\geq 1$,
    \begin{equation}
        \lim_{n\rightarrow\infty}\P\big( Y'_1 \in \cC_{\hat{q}_n}(X'_1),  \ldots, Y'_m \in \cC_{\hat{q}_n}(X'_m) \big) = (1-\alpha)^m.
    \end{equation}
\end{proposition}
\begin{proof}[Proof of Proposition~\ref{prop:asymptotic-fwer-conformal}]
    By construction of the split conformal prediction set, $Y'_i\in\cC_{\hat{q}_n}(X'_i)$ holds if and only if $s(X'_i,Y'_i)\leq \hat{q}_n$. Since the test point scores $S'_i=s(X'_i,Y'_i)$ are i.i.d.\ draws from the CDF $F$, and $\hat{q}_n$ depends only on the calibration data (and is therefore independent of the test scores), we have
    \[
        \P\left(S'_1\leq \hat{q}_n, \ldots \, S'_m \leq \hat{q}_n\right)
        = \E\left[\P\left(S'_1 \leq \hat{q}_n , \ldots , S'_m\leq \hat{q}_n\mid\hat{q}_n\right)\right] = \E\left[F(\hat{q}_n)^m\right].
    \]
    The function $t\mapsto t^m$ is convex on $t\geq 0$, so $\E[F(\hat{q}_n)^m] \geq \E[F(\hat{q}_n)]^m$ by Jensen's inequality. Finally, by definition of $F$,  we can write $F(\hat{q}_n) = \P(s(X'_1,Y'_1) \leq \hat{q}_n\mid \hat{q}_n) = \P(Y'_1\in\cC_{\hat{q}_n}(X'_1)\mid \hat{q}_n)$, where $(X'_1,Y'_1)\sim P$ is a test point. Thus, by the tower law,
    \[\E[F(\hat{q}_n)] = \E[\P(Y'_1\in\cC_{\hat{q}_n}(X'_1)\mid \hat{q}_n)] = \P(Y'_1\in\cC_{\hat{q}_n}(X'_1)) \geq 1-\alpha,\]
    where the last step holds by the distribution-free marginal coverage guarantee for the split conformal prediction method. This completes the proof for the finite-sample case.
    
    Next, assume $F(q)$ is continuous at $q=q^*$, and that $q^*$ is the unique $(1-\alpha)$-quantile, then for any fixed $m\geq 1$. By Theorem~\ref{thm:splitCP_asymp_formal_qn}, it holds almost surely that $\hat{q}_n\to q^*$. Since $F$ is assumed to be continuous at $q^*$, we then have $F(\hat{q}_n)\to F(q^*) = 1-\alpha$, almost surely. By the dominated convergence theorem, then,  it holds that $\E\left[F(\hat{q}_n)^m\right] \to (1-\alpha)^m$, which proves the desired result for the asymptotic setting.
\end{proof}

This proposition reveals a fundamental fact about FWER control with conformal prediction: it is not possible to do much better than a union bound. That is, the aim of FWER control, as stated in~\eqref{eqn:FWER_goal}, will only hold (asymptotically) if we choose
\[1-\alpha_{\textnormal{FWER}} = (1-\alpha)^m \ \Longleftrightarrow \ \alpha = 1 - (1-\alpha_{\textnormal{FWER}})^{1/m} \approx  \alpha_{\textnormal{FWER}}/m,\]
i.e., this is nearly the Bonferroni correction.
This is because, as the calibration set size $n$ approaches infinity, the events $\{ S'_i \leq \hat{q}_n \}_{i \in [m]}$ become approximately independent. 

\index{family-wise error rate|)}

\section{Outlier detection with false discovery rate control}\label{sec:multiplicity__outlier_FDR}
\index{exchangeability!testing|(}
\index{false discovery rate|(}
\index{outlier detection|(}

In this section, we turn to a different setting where multiplicity arises: the problem of outlier detection.
As before, assume we have i.i.d.~calibration points $(X_1,Y_1),\dots,(X_n,Y_n)$ from a distribution $P$, and an independent set of test points $(X_1',Y'_1),\dots,(X_m',Y_m')$. In contrast to the prediction problem, our goal now is to detect outliers (i.e., test points $(X'_i,Y'_i)$ that were \emph{not} drawn from $P$), rather than to construct prediction intervals. In particular, $Y'_1,\dots,Y_m'$ are observed, which was not previously the case.

To formalize the question, for each test point $i\in[m]$ define the null hypothesis
\[H_{0,i}: (X'_i,Y'_i) \sim P.\]
Rejecting $H_{0,i}$ means that we claim to have detected that test point $(X'_i,Y'_i)$ is an outlier. 
For the outlier detection problem, when we are simultaneously testing $m$ many data points, we generally wish to detect outliers without incurring too many false positives. 
In particular, consider a procedure that  sees both the calibration points and the test points and then returns a subset $\cR \subseteq \{1,\dots,m\}$ of the test points that are believed to be outliers---this is the set of \emph{rejections} among the null hypotheses $H_{0,1},\dots,H_{0,m}$ being tested. 

When the number of test points $m$ is large, a natural notion of error in this setting is the false discovery rate (FDR), defined as 
\begin{equation}
\label{eq:fdr_def_outlier_detection}
\textnormal{FDR} = \E\left[\frac{\sum_{i=1}^m\ind{i\in\cR, \,H_{0,i}\textnormal{ is true}}}{\max\{1,|\cR|\}}\right],
\end{equation}
the expected proportion of false rejections.

We now consider a conformal p-value approach towards testing the $m$ null hypotheses with FDR control.
Assume that we have a
pretrained score function $s : \cX\times\cY \to \R$. We will test each null hypothesis $H_{0,i}$ by comparing the score $S'_i=s(X'_i,Y'_i)$ against the calibration set scores $S_1,\dots,S_n$. Concretely, we consider the conformal p-value for test point $i$ (recall Section~\ref{sec:conformal_as_perm}):
\begin{equation}
\label{eq:outlier-pvalue}
p_i = \frac{1 + \sum_{j=1}^n \ind{S'_i \le S_j}}{n+1}.
\end{equation}
Notice that the p-values $(p_1,\dots,p_m)$ are dependent, because they are all calculated using the same calibration set. This is the main challenge we address in this section---we will characterize the dependence structure and explain how this enables the use of existing FDR control algorithms together with these conformal p-values.

\begin{figure}[t]
    \centering
    \includegraphics[width=\linewidth]{\diagramspath p-value-bias-calibration-set.pdf}
    \caption{\textbf{Visualization of dependence among conformal p-values for outlier detection.} The two panels plot the conformal p-value for two realizations of the calibration set. Specifically, the calibration set determines the map from a test point score, $S'_i$, to the corresponding conformal p-value, $p_i$, as specified in~\eqref{eq:outlier-pvalue}; in each panel, this is illustrated by the solid piecewise constant function. This function is essentially an empirical estimate of $s\mapsto 1-F(s)$, where $F$ is the CDF of the distribution of the score $s(X,Y)$ on a new data point $(X,Y)\sim P$; this population-level function is illustrated by the dotted line. We can see that, given a draw of the calibration set, the conformal p-values $p_i$ for test points $i=1,\dots,m$ are dependent. The calibration scores on the left are smaller than average (due to chance), so the p-values $p_i$ assigned to the test points will tend to be smaller, while in the right-hand panel the opposite is true. 
    }
    \commentAlt{Two plots each have vertical axis `p-value' and horizontal axis `score'. On each plot, $8$ scores are marked, labeled as the calibration score $s(X_i,Y_i)$. A decreasing step function is shown, labeled as the conformal p-value. See long description.}
    \commentLongAlt{Two plots each have vertical axis `p-value' and horizontal axis `score'. On each plot, $8$ scores are marked, labeled as the calibration score $s(X_i,Y_i)$. A decreasing step function is shown, labeled as the conformal p-value. For the left plot, the scores are slightly lower, and the step function mostly lies slightly below a dashed line labeled $1-F(s)$. For the right plot, the scores are slightly higher, and the step function mostly lies slightly above the same dashed line.}
    \label{fig:outlier-pvalues}
\end{figure}

\subsection{FDR control with the Benjamini--Hochberg procedure}
\index{Benjamini--Hochberg procedure}
Controlling the FDR is well understood if we have independent p-values. For example, the widely-used Benjamini--Hochberg algorithm is defined as
\begin{equation}\label{eqn:define_BH_procedure}\cR = \left\{i : p_i \leq \frac{\alpha_{\textnormal{FDR}}\hat{k}}{m}\right\}\textnormal{ where }\hat{k} = \max\left\{k : \sum_{i=1}^m \ind{p_i\leq\frac{ \alpha_{\textnormal{FDR}} k}{m}} \geq k\right\},\end{equation}
where $\alpha_{\textnormal{FDR}} \in [0,1]$ is the target FDR level.
When applied to independent p-values, this algorithm has the property that $\textnormal{FDR}\leq \alpha_{\textnormal{FDR}}$.

This leads to the main idea of this section: for the outlier detection problem, even though the conformal p-values $p_i$ defined above are not independent, we can still control the FDR using the Benjamini--Hochberg algorithm. 
This is because the conformal p-values have a particular form of dependence:
\begin{definition}[Positive regression dependence on a subset (PRDS)]
A random vector $W \in \R^m$ satisfies the PRDS property with respect to a set $I_0 \subseteq \{1,\dots,m\}$ if for any $i \in I_0$ and any nondecreasing set $A$, $\P(W \in A \mid W_i = w_i)$ is nondecreasing in $w_i$.
\end{definition} \index{positive regression dependence on a subset (PRDS)}
Here, a nondecreasing set $A \subseteq \R^m$ is a set where if $(a_1,\dots,a_m) \in A$, then any vector $(b_1,\dots,b_m)$, with $b_i\geq a_i$ for all $i\in[m]$, must also lie in $A$.
Importantly, the Benjamini--Hochberg algorithm is known to control the FDR under this type of dependence---i.e., for PRDS p-values, the Benjamini--Hochberg procedure~\eqref{eqn:define_BH_procedure} guarantees $\textnormal{FDR}\leq\alpha_{\textnormal{FDR}}$.

We now state the main result: the PRDS property of conformal p-values constructed with a shared calibration set. 
\begin{theorem}[Conformal p-values are PRDS]
\label{thm:prds}
Suppose $(X_1,Y_1),\dots,(X_n,Y_n),(X'_1,Y'_1),\dots,(X'_m,Y'_m)$ are independent. Assume also that $(X_i,Y_i)\sim P$ for all $i\in[n]$, and that $(X'_i,Y'_i)\sim P$ for all $i\in I_0$. Let $s$ be a pretrained score function, and assume that all calibration and test scores are distinct almost surely. Then, the conformal p-values~\eqref{eq:outlier-pvalue} are PRDS on the subset $I_0$.
\end{theorem} \index{conformal p-value}
We may compare Theorem~\ref{thm:prds} with the independence property of full conformal p-values in the online setting in Section~\ref{sec:online_conformal_data_reuse}. There, the conformal p-values were computed by re-running conformal prediction at each step (with all previous data used as the calibration set). In that setting, we saw that the conformal p-values are independent. By contrast, the setting studied here has a fixed calibration set that is re-used for computing the p-value for every test point. 

In this setting, the p-values are no longer independent, but the result of Theorem~\ref{thm:prds} establishes that their dependence must be positive. Indeed, Figure~\ref{fig:outlier-pvalues} gives an intuitive illustration of why this should be the case: we see positive dependence among the conformal p-values $p_1,\dots,p_m$ because they are each comparing against the same reference set of scores (i.e., the scores of the calibration set).
For example, if the calibration set happens to contain an unusually high fraction of data points $(X_i,Y_i)$ with a low score $S_i=s(X_i,Y_i)$, then all p-values $p_1,\dots,p_m$ will be more likely to be small.

The practical implication of this theorem is that the Benjamini--Hochberg algorithm with the conformal p-values controls the FDR. We record this fact in the following corollary.
\begin{corollary}[FDR control with conformal p-values]\label{cor:prds_fdr}
   In the setting above, let $\cR$ be the set of rejections of the Benjamini--Hochberg algorithm, when run at target FDR level $\alpha_{\textnormal{FDR}}$ based on the p-values $p_1,\dots,p_m$. Then, the FDR is controlled:
    \begin{equation}
        \E\left[\frac{|\{i\in\cR , H_{0,i}\textnormal{ is true}\}|}{\max\{1,|\cR|\}}\right] \le \alpha_{\textnormal{FDR}}.
    \end{equation}
\end{corollary}

Lastly, we remark that the PRDS result of Theorem~\ref{thm:prds} is not unique to the outlier detection setting and holds for split conformal p-values in general, but it is particularly important in the setting of outlier detection since it enables FDR control.

\begin{proof}[Proof of Theorem~\ref{thm:prds}]
Fix any test point in $I_0$---without loss of generality we can assume it is the first one, $(X'_1,Y'_1)$. First, for each $i\in\{2,\dots,m\}$, we define a count
\[C_i = \sum_{j=1}^n \ind{s(X'_i,Y'_i)\leq s(X_j,Y_j)} + \ind{s(X'_i,Y'_i)\leq s(X'_1,Y'_1)}.\]
For intuition, we can observe that, for each $i\in\{2,\dots,m\}$, the quantity
$\frac{1 + C_i}{n+2}$ is the conformal p-value we would obtain for test point $(X'_i,Y'_i)$ if we were comparing to an augmented calibration set $((X_1,Y_1),\dots,(X_n,Y_n),(X'_1,Y'_1))$ rather than only comparing to $\cD_n=((X_1,Y_1),\dots,(X_n,Y_n))$.

\textbf{Step 1: verifying independence of $p_1$ and the $C_i$'s.} First, we will prove that $p_1\indep (C_2,\dots,C_m)$. 
This claim holds by exchangeability of the augmented calibration set $(X_1,Y_1),\dots,(X_n,Y_n),(X'_1,Y'_1)$---the counts $C_i$ depend symmetrically on this augmented dataset, while $p_1$ expresses the rank of $(X'_1,Y'_1)$ within this dataset and is therefore uniformly distributed over $\{\frac{1}{n+1},\dots,\frac{n}{n+1},1\}$, even after conditioning on $C_2,\dots,C_m$, since we have assumed that the scores are distinct almost surely. (We can observe that this argument is nearly identical to that of Theorem~\ref{thm:online-indep-pvals}, which proved independence of conformal p-values in the online testing regime.)

\textbf{Step 2: recovering $p_i$'s from $C_i$'s.} Next, we claim that for each $i\in\{2,\dots,m\}$, 
\[s(X'_i,Y'_i)\leq s(X'_1,Y'_1) \ \Longleftrightarrow \ C_i \geq (n+1)p_1.\]
To see why, define $f(t) = \sum_{j=1}^n \ind{t\leq s(X_j,Y_j)} + \ind{t\leq s(X'_1,Y'_1)}$.
Then by construction, $C_i = f(s(X'_i,Y'_i))$ for each $i\in\{2,\dots,m\}$, and $p_1 = \frac{f(s(X'_1,Y'_1))}{n+1}$. Since $f$ is monotone, and is strictly decreasing over score values, this verifies the claim.
With this calculation in place, we will now see that the original p-values $p_i$ can be recovered as functions of $p_1$ and $C_i$: for each $i\in\{2,\dots,m\}$, by definition of $p_i$ and of $C_i$, we have
\begin{multline*}p_i = \frac{1 + \sum_{j=1}^n \ind{s(X'_i,Y'_i)\leq s(X_j,Y_j)}}{n+1} = \frac{1 + C_i - \ind{s(X'_i,Y'_i)\leq s(X'_1,Y'_1)}}{n+1}\\
= \frac{1 + C_i - \ind{C_i \geq (n+1)p_1}}{n+1}.\end{multline*}

\textbf{Step 3: the PRDS condition.} Now let $A\subseteq\R^m$ be an nondecreasing set. We need to verify that
\[t\mapsto \P((p_1,\dots,p_m)\in A\mid p_1=t)\]
is nondecreasing in $t$. From Step 2, this function can equivalently be written as
\[t\mapsto \P(V_t \in A \mid p_1 =t)\]
where we define the random vector $V_t$  as
\[V_t = \left(t, \frac{1 + C_2 - \ind{C_2 \geq (n+1)t}}{n+1},\dots,\frac{1 + C_m - \ind{C_m \geq (n+1)t}}{n+1}\right).\]
Since we have proved that $p_1 \indep (C_2,\dots,C_m)$ in Step 1, while $V_t$ is a function of $(C_2,\dots,C_m)$, we have $\P(V_t \in A \mid p_1 =t)=\P(V_t\in A)$. Therefore, we now only need to verify that
\[t\mapsto \P(V_t\in A)\]
is nondecreasing in $t$.
Next, we can observe that $t\mapsto V_t$ is  always a coordinatewise monotone nondecreasing function (i.e., this holds for any value of $C_2,\dots,C_m$). Since $A$ is a nondecreasing set, then, it holds almost surely that \[V_t \in A \ \Longrightarrow \ V_{t'}\in A,\]for any $t\leq t'$. 
Therefore, $t\mapsto \P(V_t\in A)$ is an nondecreasing function, as desired.
\end{proof}

\index{false discovery rate|)}
\index{multiple testing|)}
\index{exchangeability!testing|)}
\index{outlier detection|)}

\section{Selective coverage}
\label{sec:selective_cov}
\index{selective inference}

One important use of predictive models and prediction intervals is to identify promising examples for follow-up study. For example, we might have a list of medical drug candidates and we wish to identify candidates that appear to be highly effective, in order to pursue follow-up studies on these candidates. For problems of this flavor where we wish to focus on a set of `interesting' units, the marginal coverage property is insufficient---for the `interesting' units, the coverage from conformal prediction may be far less than the nominal rate. 

For example, suppose that, for a test feature vector $X_{n+1}$, we output a prediction interval of the form $\cC(X_{n+1}) = [\hf(X_{n+1}) - \hat{q},\hf(X_{n+1}) + \hat{q}]$, as would be the case for split conformal prediction run with the residual score. 
The marginal coverage guarantee for conformal prediction ensures that $\P(Y_{n+1}\in\cC(X_{n+1}))\geq 1-\alpha$, on average over all possible test points. But if this unit is selected for followup study only if some criterion $\hf(X_{n+1}) \geq c$ is reached (i.e., only if the predicted output is sufficiently  high), the probability of coverage conditional on the unit being selected for followup testing may be lower---that is, we may have
\begin{equation}
    \label{eq:selective-coverage-violation}
    \P(Y_{n+1}\in\cC(X_{n+1})\mid \hf(X_{n+1}) \geq c) < 1-\alpha.
\end{equation}
This is the problem of selective inference: a guarantee that holds marginally may no longer hold if we condition on the outcome of a (data dependent) selection rule. See Figure~\ref{fig:selective-coverage} for an illustration of a situation where coverage is violated on a selection event.
\begin{figure}[t]
    \centering
    \includegraphics[width=0.8\linewidth]{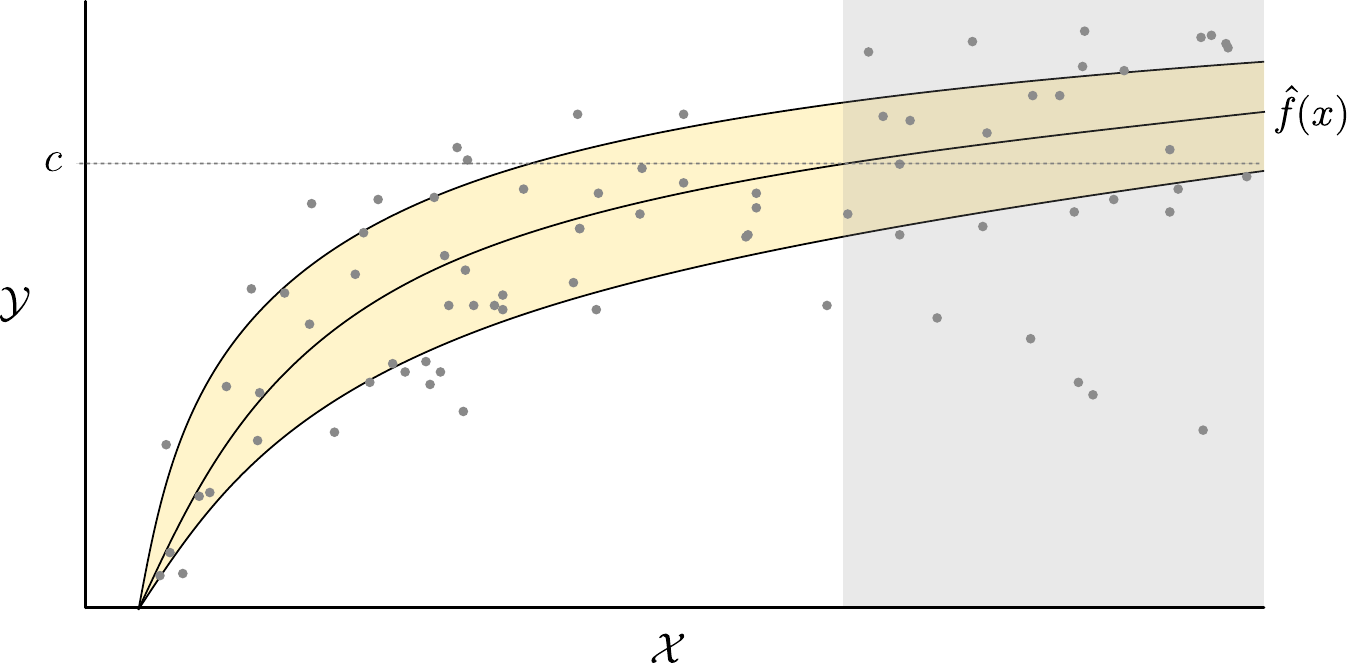}
    \caption{\textbf{Illustration of a selection event causing a coverage violation.} The central gray line represents a predictor $\hat{f}(x)$, and the surrounding band represents a prediction interval constructed at each $x\in\cX$. The shaded region on the right shows the region in which the predicted value $\hat{f}(x)$ exceeds a threshold $c$ shown by the gray dotted line. Within the shaded region, the coverage of the prediction interval is substantially lower than outside the shaded region, indicating a violation of coverage. This is an example of the scenario defined in~\eqref{eq:selective-coverage-violation}.}
    \commentAlt{A scatterplot of $(X,Y)$ data with an increasing trend and higher variability for large $X$. An increasing curve shows a fitted regression function. A shaded region on the right part of the plot begins where the function crosses threshold $c$.}
    \label{fig:selective-coverage}
\end{figure}

We next explain how conformal prediction can be adapted to avoid undercoverage on selected units.
Here, we will ask for coverage conditional on a point being identified for follow-up---that is, conditional on selection.

Consider a selection rule $\cI$ mapping a dataset $\cD\in(\cX\times\cY)^{n+1}$
to a selected subset of data points $\cI(\cD)\subseteq[n+1]$. 
To give a concrete example, the map $\cI$ could select all points $i$ with $\hf(X_i)\geq c$ for some prefitted model $\hf$ and some predetermined threshold $c$, as above. Or, as another example, the selection rule could depend in a more complex way on the entire dataset, by selecting the $k$ points with the largest predicted response $\hf(X_i)$ for some fixed $k$. 
We assume $\cI$ is symmetric, in the following sense:
\begin{equation}\label{eqn:symmetry_of_selection_rule}\sigma(i)\in \cI(\cD) \ \Longleftrightarrow \ i\in\cI(\cD_\sigma),\end{equation}
for any dataset $\cD$, any permutation $\sigma$, and any index $i\in[n+1]$. This condition means that the ordering of the data points does not affect which data points are selected. 

Given such a selection rule, we now turn to the algorithm. The prediction set will take the following form:
\begin{equation}
\label{eq:selective_conformal_set}
  \cC(X_{n+1}) = \{y : S_{n+1}^y \le \hat{q}^y, \ n+1\in\cI(\cD^y_{n+1})\},
\end{equation}
where $S_{n+1}^y$ and $\cD^y_{n+1}$ are defined as in Chapter~\ref{chapter:conformal-exchangeability}. The last part of the condition, $n+1\in\cI(\cD^y_{n+1})$, ensures that we only consider values $y$ that would lead to the test data point being selected (since we are aiming for coverage to hold conditional on this event). The other part of the condition, $S_{n+1}^y \le \hat{q}^y$ has the same form as in full conformal prediction, but we will need to define $\hat{q}^y$ in a new way to guarantee that coverage holds when restricting attention to selected cases.

The main idea for achieving selective coverage is that $\hat{q}^y$ is computed using only the scores from the set of selected points (rather than using all scores as for marginal coverage, in Chapter~\ref{chapter:conformal-exchangeability}). These are the data points that, \emph{even after conditioning on selection}, are exchangeable with a test point. Formally, we define
\begin{equation}\label{eq:selective_conformal_set_qhat}
    \hat{q}^y = \quantile\left((S_{i}^y)_{i \in \cI^y}; (1-\alpha)(1+1/|\cI^y|) \right),
\end{equation}
where $\cI^y = \{i\in[n] : i\in\cI(\cD^y_{n+1})\}$ is the subset of training points that are selected when the selection rule is run with test point $(X_{n+1},y)$ (or, if this set is empty, then we set $\hat{q}^y = +\infty$).
With this in hand, we have the following result.
\begin{theorem}[Coverage guarantee of selective conformal prediction]
    \label{thm:selective_conformal_coverage}
    Suppose $(X_1,Y_1),...,(X_{n+1},Y_{n+1})$ are exchangeable. Let $s$ be a symmetric score function, and let $\cI$ be a symmetric selection rule as in~\eqref{eqn:symmetry_of_selection_rule}. Assume the event $n+1\in\cI(\cD_{n+1})$ (i.e., the event that the test point is selected) has positive probability.
    Then the prediction set $\cC(X_{n + 1})$ defined above in~\eqref{eq:selective_conformal_set} and~\eqref{eq:selective_conformal_set_qhat} satisfies
    \begin{equation}
        \P\left( Y_{n + 1} \in \cC(X_{n + 1}) \mid n+1\in\cI(\cD_{n+1}) \right) \geq 1-\alpha.
    \end{equation}
\end{theorem} \index{coverage!selective}

Computationally, forming the set in~\eqref{eq:selective_conformal_set} potentially requires iterating through all $y \in \cY$ and then computing the selection sets $\cI(\cD^y_{n+1})$ (as well as the scores $S_i^y$ for $i=1,\dots,n$). 
However, the computation simplifies considerably when the selection rule depends only on the features $X$ (for instance, if we select data points $i$ for which $\hf(X_i)\geq c$), since $\cI(\cD^y_{n+1})$ then does not depend on $y$.

The reader should note that this algorithm is similar to Mondrian conformal prediction from Section~\ref{sec:mondrian}, with the selection rule $\cI$ serving as the analogue to the function $g$ from that section. The difference is that now $\cI$ depends on all the data points, whereas $g$ depended only on one data point; we can view this method as a generalization of the procedure from Section~\ref{sec:mondrian}.

The proof of the theorem follows from the fact that the conformal scores are exchangeable conditional on the output of a symmetric selection rule, which we formally state next. 
\begin{lemma}[Conditional exchangeability after selection]
\label{lem:exch_conditional_equivariant}
Suppose $(X_1,Y_1),...,(X_{n+1},Y_{n+1})$ are exchangeable, and let $\cI$ be a symmetric selection rule as in~\eqref{eqn:symmetry_of_selection_rule}. Let $\cE_I$ be the event that $\cI(\cD_{n+1}) = I$, for some fixed nonempty subset $I\subseteq[n+1]$. Assume $\cE_I$ has positive probability. Then $((X_i,Y_i))_{i\in I}$ is exchangeable conditional on $\cE_I$.
\end{lemma} \index{exchangeability!conditional}
This lemma then implies a useful result for conformal p-values:
\index{conformal p-value!selective}
\begin{corollary}\label{cor:pvalues_due_to_lem:exch_conditional_equivariant}
    In the setting of Lemma~\ref{lem:exch_conditional_equivariant}, define
    \[p = \frac{1 + \sum_{i\in[n]}\ind{i\in\cI(\cD_{n+1}),S_i\geq S_{n+1}}}{1 + \sum_{i\in[n]}\ind{i \in \cI(\cD_{n+1})}}\]
where $S_i = s((X_i,Y_i);\cD_{n+1})$ for some symmetric score function $s$. If $\P(n+1\in\cI(\cD_{n+1}))>0$, then
\[\P(p\leq\alpha\mid n+1\in\cI(\cD_{n+1}))\leq \alpha\]
for any $\alpha\in[0,1]$.
\end{corollary}
We note that these results are very similar to Lemma~\ref{lem:exch_conditional_subvector} and Corollary~\ref{cor:pvalues_due_to_lem:exch_conditional_subvector}, which were used for bin-wise conditional coverage results in Chapter~\ref{chapter:conditional}.

\begin{proof}[Proof of Lemma~\ref{lem:exch_conditional_equivariant}]
To verify the first part of the lemma, we fix any nonempty $I\subseteq[n+1]$ such that $\P(\cE_I)>0$. Fix any permutation $\sigma$ on $I$, and let $\tilde{\sigma}$ be the permutation on $[n+1]$ defined as
\[\tilde{\sigma}(i) = \begin{cases}\sigma(i), & i\in I,\\ i, & i\not\in I.\end{cases}\]
Write $Z_i = (X_i,Y_i)$ for each data point $i\in[n+1]$, and let $Z_{\tilde\sigma}=(Z_{\tilde\sigma(1)},\dots,Z_{\tilde\sigma(n+1)})$.
For any $A\subseteq(\cX\times\cY)^{|I|}$, we have
\begin{align*}
    \P((Z_i)_{i\in I}\in A, \cE_I)
    &=\P((Z_i)_{i\in I}\in A, \cI(Z_1,\dots,Z_{n+1}) = I)\\
    &=\P((Z_{\tilde\sigma(i)})_{i\in I}\in A, \cI(Z_{\tilde\sigma(1)},\dots,Z_{\tilde\sigma(n+1)}) = I)\textnormal{\ since $Z\eqd Z_{\tilde\sigma}$}\\
    &=\P((Z_{\tilde\sigma(i)})_{i\in I}\in A, \cI(Z_1,\dots,Z_{n+1}) = I) \textnormal{\  using~\eqref{eqn:symmetry_of_selection_rule} and $\tilde\sigma(I)=I$}\\ 
    &=\P((Z_{\tilde\sigma(i)})_{i\in I}\in A, \cE_I)\\
    &=\P((Z_{\sigma(i)})_{i\in I}\in A, \cE_I),
\end{align*}
where the last step holds by definition of $\tilde\sigma$. As in the proof of Lemma~\ref{lem:exch_conditional_subvector}, since this calculation holds for any $A$ and any $\sigma$, this is sufficient to verify that $(Z_i)_{i\in I}$ is exchangeable conditional on $\cE_I$.
\end{proof}
Corollary~\ref{cor:pvalues_due_to_lem:exch_conditional_equivariant} then follows from Lemma~\ref{lem:exch_conditional_equivariant}, by identical arguments as for the proof of Corollary~\ref{cor:pvalues_due_to_lem:exch_conditional_subvector} (following from Lemma~\ref{lem:exch_conditional_subvector}), and we omit the proof.

\begin{proof}[Proof of Theorem~\ref{thm:selective_conformal_coverage}]
Define a conformal p-value for this selective setting,
\[p^y = \frac{1 + \sum_{i\in[n]}\ind{i\in\cI(\cD^y_{n+1}),S^y_i\geq S^y_{n+1}}}{1 + \sum_{i\in[n]}\ind{i \in \cI(\cD^y_{n+1})}}.\]
By definition, $p^{Y_{n+1}}$ is equal to the p-value $p$ defined in Corollary~\ref{cor:pvalues_due_to_lem:exch_conditional_equivariant} , and consequently,
\[\P(p^{Y_{n+1}}\leq \alpha\mid n+1\in\cI(\cD_{n+1})) \leq \alpha.\]
Moreover, the selective conformal prediction set can be defined in terms of these conformal p-values,
\[\cC(X_{n+1}) = \{y\in\cY: p^y > \alpha, \ n+1\in\cI(\cD^y_{n+1})\}\]
(this claim is analogous to the result of Proposition~\ref{prop:conformal-via-pvalues}, which proves the same equality in the setting of conformal prediction without selection). Therefore,
\[\P(Y_{n+1}\in\cC(X_{n+1})\mid n+1\in\cI(\cD_{n+1}))
= \P(p^{Y_{n+1}}> \alpha\mid n+1\in\cI(\cD_{n+1})) \geq 1-\alpha.\]
\end{proof}

\section{Aggregating conformal sets}
\label{sec:aggregating_sets}
\index{model aggregation|(}

A common strategy in predictive modeling is model ensembling---that is, combining multiple predictive models to achieve better prediction accuracy. This leads naturally to the question of how we might combine multiple prediction \emph{sets} (rather than combining point predictions). 

Such a question is subtle, however, because the precision of our estimate will vary depending on how dependent the original prediction sets are: if we aggregate $K$ prediction sets that provide relatively independent information, we might hope that the miscoverage rate could decrease substantially, while if we aggregate $K$ sets that essentially each repeat the same information, then we would not expect their aggregated version to be more accurate. Indeed, we might even be concerned that aggregation in the presence of arbitrary dependence could potentially even inflate the miscoverage error. To address this last concern, we will begin by showing that there is a simple worst-case bound ensuring that aggregation via majority vote will preserve the coverage guarantee up to a factor of $2$. Next, we will outline a simple strategy that leads to more precise guarantees on coverage, but requires some additional data for calibration.

\subsection{Aggregating prediction sets}\label{sec:aggregating_via_majority_vote}
Suppose we have $K$ conformal prediction algorithms. For the first result, we require only access to the confidence sets $\cC_1(x),\dots,\cC_K(x)$, which are each subsets of $\cY$ that depend on the feature vector $x \in \cX$. For a test point with features $X_{n+1}$, we consider merging the $K$ prediction sets by majority vote:
\begin{equation}
    \cC^{\textnormal{mv}}(X_{n+1}) = \left\{y : \frac{1}{K} \sum_{k=1}^K \ind{y \in \cC_k(X_{n+1})} > 1/2\right\}. 
\end{equation}
The majority vote set has coverage at least $1-2\alpha$, as stated next.
\begin{theorem}[Coverage guarantee for majority vote aggregation]
Suppose $\cC_1(x),\dots,\cC_K(x)$ are any sets satisfying $\P(Y_{n+1} \in \cC_k(X_{n+1})) \ge 1-\alpha$ for each $k=1,\dots,K$. Then
\label{thm:mv_aggregation}
\begin{equation}
\P\Big(Y_{n+1} \in \cC^{\textnormal{mv}}(X_{n+1})\Big) \ge 1 - 2\alpha.   
\end{equation}
\end{theorem}

\begin{proof}[Proof of Theorem~\ref{thm:mv_aggregation}]
We calculate
\begin{align}
    \P\left(Y_{n+1} \notin \cC^{\textnormal{mv}}(X_{n+1})\right)
    &= \P\left(\frac{1}{K} \sum_{k=1}^K \ind{Y_{n+1} \notin \cC_k(X_{n+1})} \ge \frac{1}{2}\right) \\
    &\le 2 \E\left[\frac{1}{K}\sum_{k=1}^K \ind{Y_{n+1} \notin \cC_k(X_{n+1})}\right] \\
    &= \frac{2}{K} \sum_{k=1}^K \E\left[\ind{Y_{n+1} \notin \cC_k(X_{n+1})}\right] \\
    &\le 2\alpha,
\end{align}
where the first inequality is due to Markov's inequality, and the second inequality uses the marginal coverage property of each $\cC_k$.
\end{proof}

This result provides basic reassurance that the majority vote procedure is reasonable, in that the coverage cannot degrade dramatically. 
However, we might expect that aggregating sets would actually lead to \emph{increased} coverage when the sets provide distinct information.
We next outline an aggregation strategy that gives a more precise coverage guarantee.

\subsection{Re-calibration after aggregation}
We next consider an approach that involves first aggregating multiple prediction sets, and then calibrating to achieve a precise coverage level (using a small amount of extra data).

Suppose we have $K$ prediction sets $\cC_1(x ; 1-\alpha),\dots,\cC_K(x; 1-\alpha)$, where each set $\cC_k$ can be evaluated at any $x\in\cX$ and at any confidence level $1-\alpha$. Intuitively, each $\cC_k(\cdot;1-\alpha)$ is trained with the aim of achieving (marginal) predictive coverage at level $1-\alpha$, but unlike the majority-vote aggregation result in Theorem~\ref{thm:mv_aggregation} above, here we do not assume a marginal coverage property for the base prediction sets $\cC_k$. Instead, recalling the conformal risk control method of Section~\ref{sec:conformal_risk_control}, we will use an additional independent dataset $\cD_n=((X_1, Y_1),\dots,(X_n, Y_n))$ to recalibrate an aggregated prediction set. (Note that we are assuming the $\cC_k$'s are pretrained---that is, the new dataset $\cD_n$ is independent of the $K$ constructed prediction sets.) 

Turning to the details, for technical reasons we consider a version of the prediction set $\cC_k$ that is modified to be monotone and left-continuous in $\alpha$:
\begin{equation}
    \cC_k'(x ; 1 - \alpha) = \left\{y : \textnormal{ for all $\epsilon>0$, $y \in \cC_k(x ; 1 - \alpha')$ for some $\alpha' \geq \alpha -\epsilon$}\right\},
\end{equation}
for $\alpha\in(0,1]$, and $\cC_k'(x;1)=\cY$.
In other words, this guarantees that for $\alpha'\leq\alpha$, we have $\cC_k'(x;1-\alpha)\subseteq \cC_k'(x;1-\alpha')$ (i.e., the set can only increase if we require a higher confidence level), and that $\alpha\mapsto \ind{y\in \cC_k'(x;1-\alpha)}$ is left-continuous. Note that if the $\cC_k$'s already satisfy these properties, then we simply have $\cC_k'=\cC_k$, i.e., the sets are unchanged.

We are now ready to define an aggregation procedure.
We will use the same majority-vote aggregation as in Section~\ref{sec:aggregating_via_majority_vote}, except that here we introduce a tuning parameter $\lambda$:
\begin{equation}
    \cC^{\textnormal{mv}}(x ; \lambda) = \left\{y : \frac{1}{K} \sum_{k=1}^K \ind{y \in \cC'_k(x_{n+1}; \lambda)} > 1/2 \right\}.
\end{equation}
To compare to the procedure studied in the previous section, here if we choose $\lambda = 1-\alpha$ then we would obtain the aggregated set from Section~\ref{sec:aggregating_via_majority_vote} (if we ignore the minor distinction between the $\cC_k$'s and the $\cC_k'$'s).
The intuition for this new procedure is that instead of setting $\lambda = 1-\alpha$ a priori (which only leads to the weak coverage guarantee given in Theorem~\ref{thm:mv_aggregation}---and in fact, might either undercover \emph{or} overcover in practice), we instead use the available data to calibrate the tuning parameter $\lambda$ to achieve tighter control of the coverage level.

Our next task is then to determine how to use the dataset $\cD_n$ to calibrate the tuning parameter $\lambda$. We will define
\begin{equation}
    \hat \lambda = \inf \left\{\lambda\in[0,1] : \frac{1}{n}\sum_{i=1}^n \ind{Y_i \notin \cC^{\textnormal{mv}}(X_i ; \lambda)} \le \alpha - (1-\alpha)/n \right\},
\end{equation}
and return the aggregated set $\cC^{\textnormal{mv}}(X_{n+1} ; \hat \lambda)$.
This choice of $\hat\lambda$ agrees with the tuning step of the conformal risk control procedure~\eqref{eq:lhat} (and, as in Section~\ref{sec:conformal_risk_control}, we can also equivalently formulate this method as an instance of split conformal prediction). 
\begin{theorem}[Coverage guarantee for post-aggregation calibration]
\label{thm:calibrate_after_aggregation}
Suppose $(X_1, Y_1),\dots,(X_{n+1}, Y_{n+1})$ are exchangeable, and are independent of the pretrained prediction sets $\cC_1,\dots,\cC_K$. Then under the notation and definitions above,
\begin{equation}
\P\Big(Y_{n+1} \in \cC^{\textnormal{mv}}(X_{n+1} ; \hat \lambda)\Big) \ge 1 - \alpha.   
\end{equation}
\end{theorem}

\begin{proof}[Proof of Theorem~\ref{thm:calibrate_after_aggregation}]
This result follows immediately from the guarantee of conformal risk control (Theorem~\ref{thm:risk_control}), by choosing
\[\cC_\lambda(x) = \begin{cases}\cC^{\textnormal{mv}}(x;\lambda),&\lambda\in[0,1), \\ \varnothing, & \lambda<0, \\ \cY, & \lambda\geq 1,\end{cases}\]
and
\[L(y,\cC_\lambda(x)) = \ind{y\not\in\cC_\lambda(x)}.\]
Note that using the sets $\cC'_k$ in place of $\cC_k$ ensures that this loss is monotone and is right-continuous, as required by Theorem~\ref{thm:risk_control}.
\end{proof}

In addition to the lower bound on the coverage, under mild continuity assumptions on the data we can establish an upper bound of $1-\alpha + \frac{1}{n+1}$ (as in Theorem~\ref{thm:upper-bound}). This means that calibrating after aggregation leads to sets with almost exactly the desired coverage level, unlike the original majority-vote approach of Section~\ref{sec:aggregating_via_majority_vote} above. In this way, calibrating after aggregation adapts to the amount of distinct information in each initial set $\cC_k$.

\subsection{Connection with cross-conformal}
\index{cross-validation}
We conclude with a remark about how this section relates to cross-validation type conformal methods.
Note that cross-conformal, CV+, and jackknife+ from Chapter~\ref{chapter:cv} can all be viewed as ways of aggregating distinct conformal sets into one final set, although in a setting with additional structure on how the original sets are related. In particular, we can observe that the result of Theorem~\ref{thm:CC_smallK} is similar to that of Theorem~\ref{thm:mv_aggregation}, although arrived at by different means---both results are based on the idea of using an average of conformal p-values.
\index{conformal p-value}

\index{model aggregation|)}

\section*{Bibliographic notes}
\addcontentsline{toc}{section}{\protect\numberline{}\textnormal{\hspace{-0.8cm}Bibliographic notes}}

The conformal risk control algorithm, and its accompanying theoretical guarantee given in Theorem~\ref{thm:risk_control}, are developed in~\cite{angelopoulos2022conformal}. This builds on work 
that controls monotone and non-monotone risk functions with high-probability guarantees~\citep{bates2021distribution, angelopoulos2021learn}, which in turn
stems from the high-probability guarantees of split conformal prediction~\citep{vovk2012conditional}. Algorithms for conformal risk control with cross-validation or leave-one-out data reuse are introduced in \cite{cohen2024cross} and \cite{angelopoulos2024notefull}, respectively.
See~\cite{fisch2022conformal,schuster2021consistent,cauchois2021knowing,angelopoulos2022image,feldman2023achieving,teneggi2023diffusion} for recent adaptations of the conformal risk control framework to various machine learning tasks, and~\cite{angelopoulos2026conformal} for an extension of conformal risk control to non-monotonic losses.

In Section~\ref{sec:multiplicity_prediction_FWER}, we discussed the problems of providing prediction intervals and of detecting outliers, in the context of multiple test points. Conformal-style approaches to outlier detection, particularly in online settings, have been studied in many works such as
\cite{laxhammar2015inductive,smith2015conformal,ishimtsev2017conformal,guan2022prediction}. See also the references corresponding to Section~\ref{sec:testing_exch_online}, which discussed the related problem of testing exchangeability online. The issue of controlling for multiple comparisons that arises in this type of setting is examined by \cite{vovk2013transductive}, which proposes using a Bonferroni correction for family-wise error rate (FWER) control; more broadly, this work examines the problem of \emph{transductive conformal prediction}, where the aim is to provide a confidence region simultaneously for multiple test points. See also~\citet{gazin2024transductive, huang2023uncertainty, marques2025universal} for recent results in this theme. 

For the problem of outlier detection on multiple test points, addressed in Section~\ref{sec:multiplicity__outlier_FDR},
\cite{bates2023testing} propose using the Benjamini--Hochberg procedure for false discovery rate (FDR) control for this problem. The Benjamini--Hochberg procedure is introduced in~\cite{benjamini1995controlling};
the PRDS property, and the result that the Benjamini--Hochberg procedure controls the
FDR with PRDS p-values, is due to~\cite{benjamini2001control}. The PRDS property is one
of several related notions of positive dependence. In particular, it is a relaxation of the 
PRD property analyzed in~\cite{sarkar1969some}.
The result that split conformal p-values are PRDS is due to~\cite{bates2023testing}, which leads to the FDR control property of Corollary~\ref{cor:prds_fdr}. This result is extended in~\cite{marandon2022adaptive} to handle exchangeable, but not independent, data. Further work develops techniques for outlier detection with FDR control via conformal e-values rather than conformal p-values~\citep{bashari2023derandomized, lee2024boosting}.

The selective coverage algorithm in Section~\ref{sec:selective_cov} can be viewed as a special case of Mondrian conformal prediction~\citep{vovk2003mondrian} (recall Section~\ref{sec:mondrian}). 
The emphasis on requiring coverage to hold conditionally on the test point is motivated by the topic of selective inference in statistics~\citep[e.g.,][]{berk2013valid, fithian2015topics, lee2016exact}.
Selective versions of conformal prediction have been developed by~\citet{jin2023model, jin2023selection}, which extend the ideas described in Section~\ref{sec:selective_cov} to more complex settings, including FDR control and covariate shift.  
The procedure we give in Section~\ref{sec:selective_cov} (along with its accompanying theoretical guarantees) is explicitly discussed in~\cite{bao2024selective, jin2024confidence}, who also develop algorithms for selective coverage with other selection rules and covariate shift. 

Turning to Section~\ref{sec:aggregating_sets}, the general problem of aggregating prediction sets is studied by \citet{linusson2017calibration}, \citet{cherubin2019majority}, \citet{solari2022multi}, and~\citet{gasparin2024merging}. In particular, the statement of Theorem~\ref{thm:mv_aggregation} is given in~\citet{cherubin2019majority} and the proof we give is taken from~\citet{gasparin2024merging}. Our approach toward tuning the confidence level of the initial sets before aggregation uses on the nested-sets interpretation of split conformal prediction developed by~\citet{gupta2020nested}. The related topic of selecting a conformal set from many candidates is also studied in~\citet{yang2024selection}. Lastly, aggregating prediction sets that are formed from resampling on a single dataset (such as cross-validation approaches) is an important topic that we discuss in detail in Chapter~\ref{chapter:cv}---see the references therein. 

\section*{Exercises}
\addcontentsline{toc}{section}{\protect\numberline{}\textnormal{\hspace{-0.8cm}Exercises}}
\begin{enumerate}[font=\bfseries, label={\thechapter.\arabic*}, labelsep=1em, itemsep=1em]
\item In Theorem~\ref{thm:risk_control}, which gives a guarantee for conformal risk control, we assume that the loss function $\lambda\mapsto L(y,\cC_\lambda(x))$ is right-continuous. Here we will explore the importance of this assumption. Recall that we can view split conformal prediction as a special case of conformal risk control, by taking $\cC_\lambda(x) = \{y\in\cY : s(x,y)\leq \lambda\}$, for a pretrained score function $s$, and loss function  $L(y,\cC_\lambda(x)) = \ind{y\not\in\cC_\lambda(x)}$. 
In this problem, we will consider running conformal risk control with a different choice of the sets $\cC_\lambda(x)$: we instead define
    \[\cC_\lambda(x) = \{y\in\cY: s(x,y) < \lambda\}\]
    (and use the same loss function). \begin{enumerate}
    \item Verify that the assumption of right-continuity no longer holds.
    \item Compute the conformal risk control threshold $\hat\lambda$ (as defined in~\eqref{eq:lhat}), and the resulting prediction set $\cC_{\hat\lambda}(X_{n+1})$.
    \item Suppose that there are no ties among scores, almost surely. Show that $\cC_{\hat\lambda}(X_{n+1})$ has marginal coverage $\geq 1-\alpha$.
    \item Now we remove the assumption that there are no ties among scores. Construct an example to show that $\cC_{\hat\lambda}(X_{n+1})$ may now have marginal coverage $< 1-\alpha$.
\end{enumerate}
\item Recall the setting of Exercise~\ref{exercise:split_CP_two_test_points}. There, we computed $\P\left(Y_{n+1} \in \cC(X_{n+1}) \textnormal{ and } Y_{n+2} \in \cC(X_{n+2})\right)$ exactly, under the assumption that there are no ties among scores, almost surely. In this exercise, show that the PRDS property of conformal p-values can be applied directly to prove that the coverage events have positive dependence,
    $$\P\left(Y_{n+1} \in \cC(X_{n+1}) \textnormal{ and } Y_{n+2} \in \cC(X_{n+2})\right) \ge (1-\alpha)^2,$$
    without computing this probability explicitly.
\item In this exercise we consider a different type of selective conformal prediction. We observe features for two test points, $X_{n+1}$ and $X_{n+2}$. Given these features, we select one of the two (i.e., $X_{n+I}$, where $I\in\{1,2\}$ may depend on the pair $(X_{n+1},X_{n+2})$), and now want to construct a prediction set for the corresponding response value, $Y_{n+I}$. Given a pretrained score function $s$, define a conformal prediction set $\cC(X_{n+I})$ that guarantees coverage,
    \[\P\left(Y_{n+I}\in\cC(X_{n+I})\right)\geq 1-\alpha,\]
    assuming that the calibration and test data points, $(X_1,Y_1),\dots,(X_{n+2},Y_{n+2})$, are exchangeable.
\item In this exercise we consider the prediction set $\cC^{\textnormal{mv}}(X_{n+1})$ constructed by majority vote, as in Section~\ref{sec:aggregating_via_majority_vote}. Let $K=2$, and let $s:\cX\times\cY\to\R$ be a pretrained score function. Let $\cC_1,\cC_2$ be constructed using split conformal prediction, each using the same dataset $(X_1,Y_1),\dots,(X_n,Y_n)$ as the calibration set, where $\cC_1$ uses the score $s$ and $\cC_2$ uses the score $-s$. Give an explicit expression for the aggregated set $\cC^{\textnormal{mv}}(X_{n+1})$ in this case, and compute an exact expression for its marginal coverage level, assuming that there are no ties among scores, almost surely. (We will see that the marginal coverage level is close to $1-2\alpha$, so this is a case where the coverage bound on $\cC^{\textnormal{mv}}$ is nearly tight.)
\end{enumerate}

\part{Beyond Predictive Coverage}
\label{part:beyond-predictive-coverage}

 \chapter{Inference on the Regression Function}
\label{chapter:regression}
\index{regression function!inference|(}

This chapter begins Part~\ref{part:beyond-predictive-coverage}, the last part of the book, where we move beyond conformal prediction and examine a range of different statistical problems through the lens of the distribution-free framework. 
This chapter will focus on the problem of \emph{regression}: given training data $\{(X_i,Y_i)\}$ drawn from an unknown distribution $P$, we would like to estimate the \emph{regression function} $\mu_P(x) = \E_P[Y\mid X=x]$. 
Given an estimate $\hat{\mu}$ of this unknown function, can we provide a meaningful confidence interval around $\hat{\mu}(x)$ that has distribution-free validity? \index{regression function}

In particular, we might hope to construct a confidence interval for $\mu_P(x)$ whose width is vanishing as the sample size $n$ increases. 
This type of result is standard in more classical settings, such as for parametric models, or even for nonparametric regression with smoothness assumptions. 
In the distribution-free setting, however, our ability to achieve vanishing interval widths will vary in different contexts: we will show that constructing such intervals is possible when the covariate $X$ is discrete, but impossible when it is nonatomic, motivating us to consider various relaxations that circumvent this hardness result. This conclusion is qualitatively similar to the results obtained in Chapter~\ref{chapter:conditional}, where we showed that test-conditional coverage for predictive inference is straightforward in the case where $X$ is discrete but impossible when $X$ is nonatomic.
This similarity is not coincidental;
towards the end of this chapter, we will see that these two problems---inference for regression, and test-conditional predictive inference---are related on a fundamental level.

\section{Problem formulation and background}\label{sec:regression_intro}

Throughout this chapter, we will assume the available data is given by $(X_1,Y_1),\dots,(X_n,Y_n)\iidsim P$ for some unknown distribution $P$ on $\cX\times\cY$, where $\cY\subseteq\R$. 
Our goal is to provide distribution-free confidence intervals on the regression function, $\mu_P(x) = \E_P[Y \mid X=x]$, for all $x$.
More formally, let $\cC$ be trained on data $(X_1,Y_1),\dots,(X_n,Y_n)$, with $\cC(x)\subseteq\R$. 
We say that $\cC$ is a distribution-free confidence interval for regression at level $1-\alpha$ if
\begin{equation}\label{eqn:DF_confidence_regression}
\P(\mu_P(X_{n+1})\in\cC(X_{n+1}))\geq 1-\alpha\textnormal{ for any distribution $P$ on $\cX\times\cY$,}
\end{equation}
where the probability is calculated with respect to $(X_1,Y_1),\dots,(X_{n+1},Y_{n+1})\iidsim P$. Note that while we refer to $\cC$ as a `confidence interval', it may not necessarily return sets that are intervals---that is, $\cC(X_{n+1})\subseteq\cY$ might be a region comprised of multiple disconnected intervals. 

A key question is whether~\eqref{eqn:DF_confidence_regression} can be achieved while still producing a confidence interval with vanishing width.
We phrase this question informally as follows:
\begin{quote}
    For i.i.d.\ data in $\cX\times\cY$, is there any $\cC$ that satisfies distribution-free validity~\eqref{eqn:DF_confidence_regression}, and also satisfies $\E[\textnormal{Leb}(\cC(X_{n+1}))]\to0$ for `nice' distributions $P$?
\end{quote}
In classical settings such as parametric models, a confidence interval for $\mu_P(x)$ can have vanishing width as $n\to\infty$.
Our key question asks if it is possible to achieve the informativeness of these classical constructions while simultaneously achieving distribution-free validity.

\section{Necessity of a boundedness assumption}\label{sec:bahadur_savage}

This section shows that distribution-free inference for the regression function $\mu_P(x)$ is meaningful only in the setting where the response $Y$ is bounded: in an unbounded case (say, $\cY=\R$), a distribution-free method cannot even return finite-length confidence intervals for $\mu_P(x)$, or indeed, even for the marginal mean of $Y$.
We will justify this claim via the following classical result:

\begin{theorem}[The Bahadur--Savage theorem]\label{thm:bahadur_savage}
    Let $\cY\subseteq\R$ be unbounded both from above and below, i.e., $\sup \cY=+\infty$ and $\inf\cY = -\infty$. Let $\cC\subseteq\R$ be trained on data $Y_1,\dots,Y_n\in\cY$. Suppose $\cC$ is a distribution-free confidence interval for the mean, i.e., $\cC$ satisfies
    \[\P(\E_P[Y] \in \cC) \geq 1-\alpha \textnormal{ for any distribution $P$ on $\cY$ with finite mean},\]
    where the probability is calculated with respect to $Y_1,\dots,Y_n\iidsim P$.
    Then it must hold that 
    \[\P( y \in \cC) \geq 1-\alpha \textnormal{ for any distribution $P$ on $\cY$ with finite mean and for any $y\in\R$.}\]
    In particular, regardless of the distribution of the data, $\cC$ has infinite expected length (if $\alpha<1$):
    \[\E[\textnormal{Leb}(\cC)]=\infty.\]
\end{theorem} \index{Bahadur--Savage} \index{hardness result!mean estimation}

In other words, if $\cY$ is unbounded, then distribution-free inference is impossible even for the \emph{marginal} mean, $\E_P[Y]$.

This has implications for both positive and negative results. For positive results, constructing methods that satisfy distribution-free validity as in~\eqref{eqn:DF_confidence_regression} will inevitably require a bounded $\cY$. On the other hand, establishing hardness results (i.e., proving that distribution-free validity~\eqref{eqn:DF_confidence_regression} leads inevitably to wide confidence intervals) is interesting only for a bounded $\cY$---if $\cY$ is unbounded, the theorem above already establishes the impossibility of distribution-free inference even for the easier task of the marginal mean $\E_P[Y]$.
Consequently, for the remaining sections of this chapter, we will generally restrict our attention to the setting of a bounded $\cY\subseteq\R$.

\begin{proof}[Proof of Theorem~\ref{thm:bahadur_savage}]
    Let $P$ be any distribution with mean $\E_P[Y]$, and fix any $y\in\R$.  Without loss of generality assume $y\geq \E_P[Y]$; the case $y\leq \E_P[Y]$ is proved similarly.
    We now show that $\P(y\in\cC)\geq 1-\alpha$.
    
    Fix a small $\epsilon>0$. First, since $\cY$ is unbounded, we can find some value 
    \[y'\in\cY\textnormal{ such that }y' \geq \frac{y - (1-\epsilon)\E_P[Y]}{\epsilon}.\]
    Now define a mixture distribution $P'$ on $\cY$ as
    \[P' = (1-\epsilon')\cdot P + \epsilon'\cdot \delta_{y'},\]
    where
    \[\epsilon' = \frac{y - \E_P[Y]}{y' - \E_P[Y]}.\]
    Note that, by construction, we have $\dtv(P,P')\leq\epsilon'\leq\epsilon$, and $\E_{P'}[Y] = y$.
    The validity of $\cC$ as a confidence interval for the mean implies
    \[\P_{(P')^n}(y\in\cC) = \P_{(P')^n}(\E_{P'}[Y]\in\cC) \geq 1-\alpha.\]
    But we can also calculate
    \[\P_{P^n}(y\in\cC) \geq \P_{(P')^n}(y\in\cC) -\dtv(P^n,(P')^n)  \geq \P_{(P')^n}(y\in\cC) - n\epsilon' \geq 1-\alpha - n\epsilon.\]
    Since $\epsilon>0$ can be taken to be arbitrarily small, this completes the proof of the first claim. Finally, we calculate
    \[\E_{P^n}[\textnormal{Leb}(\cC)] = \E_{P^n}\left[\int_{y\in\R}\ind{y\in\cC}\;\mathsf{d}y\right]=\int_{y\in\R}\P_{P^n}(y\in\cC)\;\mathsf{d}y\geq \int_{y\in\R} (1-\alpha)\;\mathsf{d}y = \infty,\]
    where the second step holds by the Fubini--Tonelli theorem.
\end{proof}

\section{The discrete case}\label{sec:regression-discrete}
We are now ready to consider the question of inference on the regression function, $\mu_P(x) = \E_P[Y\mid X=x]$.
We first consider the discrete setting. Suppose that $\cX = \{x_1,\dots,x_K\}$ is a finite set. Then to perform inference on the regression function $\mu_P(x)$, we only need to perform inference on $\mu_P(x_k)$ for each $k$ with $P_X(\{x_k\})>0$---and each such question is straightforward, since 
\[\mu_P(x_k) = \E_P[Y \mid X = x_k],\]
where we are conditioning on an event $X=x_k$ of positive probability.

As established in Section~\ref{sec:bahadur_savage} above, the problem of distribution-free inference on $\mu_P$ is meaningful only if $\cY$ is bounded, so we will assume that this is the case for the following result.
\begin{theorem}[Inference for the regression function in the discrete setting]\label{thm:DF_regression_discrete}
Let $P$ be a distribution on $\cX\times\cY$, where $\cX = \{x_1,\dots,x_K\}$ and $\cY \subseteq [a,b]$, and let $\alpha\in(0,1)$.
    Let $(X_1,Y_1),\dots,(X_n,Y_n)\iidsim P$. For each $k\in[K]$, define $n_k =\sum_{i=1}^n\ind{X_i = x_k}$, and let
    \[\hat{\mu}(x_k) = \frac{1}{n_k}\sum_{i=1}^n Y_i \cdot\ind{X_i = x_k}\]
    for each $k$ with $n_k\geq 1$, i.e., the mean response observed among all data points with $X_i =x_k$. Define
    \[\cC(x_k) = \begin{cases} \hat{\mu}(x_k)\pm (b-a)\sqrt{\frac{\log(2/\alpha)}{2n_k}}, & \textnormal{ if $n_k\geq 1$},\\ [a,b], & \textnormal{ if $n_k=0$.}\end{cases}\]
    Then $\cC$ is a valid distribution-free confidence interval for regression, i.e., $\cC$ satisfies~\eqref{eqn:DF_confidence_regression}. Moreover,
    \[\E\left[\textnormal{Leb}(\cC(X_{n+1}))\right] \leq 2(b-a)\sqrt{\log(2/\alpha)} \cdot \sqrt{\frac{K}{n}}.\]
\end{theorem}
This bound on the (expected) length of $\cC$ gives us an affirmative answer to the question posed in Section~\ref{sec:regression_intro} for discrete $\cX$.
That is, if $|\cX|$ is finite, then it is possible to achieve distribution-free coverage without resorting to wide and uninformative intervals, since the expected length of this interval vanishes as $n\to\infty$.

\begin{proof}[Proof of Theorem~\ref{thm:DF_regression_discrete}]
    Let $p_k = \P_P(X=x_k)$. Write $\cD_n=\big((X_1,Y_1),\dots,(X_n,Y_n)\big)$ to denote the training dataset. Then, since $X_{n+1}$ is independent from $\cD_n$, we can calculate
    \[\P\big(\mu_P(X_{n+1}) \not\in \cC(X_{n+1})\mid \cD_n\big) = \sum_{k=1}^K p_k \cdot \ind{\mu_P(x_k) \not\in \cC(x_k)}.\]
    Therefore, marginalizing over $\cD_n$,
    \begin{align*}
        \P\big(\mu_P(X_{n+1}) &\not\in \cC(X_{n+1})\big)
        =\sum_{k=1}^K p_k\cdot \P\big(\mu_P(x_k) \not\in \cC(x_k)\big)\\
        &=\sum_{k=1}^K p_k\cdot \E\left[\P\big(\mu_P(x_k) \not\in \cC(x_k)\mid n_k\big)\right]\\
        &=\sum_{k=1}^K p_k\cdot \E\left[\ind{n_k\geq 1}\cdot \P\left(\left|\hat\mu(x_k) - \mu_P(x_k)\right| > (b-a)\sqrt{\frac{\log(2/\alpha)}{2n_k}} \,\middle|\, n_k\right)\right],
    \end{align*}
    by construction of $\cC$. But conditional on $n_k$ (and on the event $n_k\geq 1$), the random variable $\hat\mu(x_k)$ is the sample mean of $n_k$ many draws from the distribution of $Y\mid X=x_k$, which is supported on $[a,b]$. By Hoeffding's inequality, therefore,
    \[\P\left(\left|\hat\mu(x_k) - \mu_P(x_k)\right| > (b-a)\sqrt{\frac{\log(2/\alpha)}{2n_k}} \,\middle|\, n_k\right) \leq \alpha,\]
    and consequently, 
    \[\P\big(\mu_P(X_{n+1}) \not\in \cC(X_{n+1})\big) \leq \sum_{k=1}^K p_k \cdot \E[\ind{n_k\geq 1}\cdot \alpha] \leq \alpha\sum_{k=1}^K p_k = \alpha.\]
    This verifies that $\cC$ offers distribution-free coverage.

Next we turn to the question of length. 
By construction, we have \[\textnormal{Leb}(\cC(x_k)) = \begin{cases} b-a,&\textnormal{ if }n_k=0,\\
2(b-a)\sqrt{\frac{\log(2/\alpha)}{2n_k}}, & \textnormal{ if }n_k\geq 1,\end{cases}\]
which satisfies $\textnormal{Leb}(\cC(x_k)) \leq 2(b-a)\sqrt{\frac{\log(2/\alpha)}{n_k+1}}$ across both cases.
Now, for a new feature $X_{n+1}\sim P_X$, we have
\begin{align*}
    \E\left[\textnormal{Leb}(\cC(X_{n+1}))\right]
    &=\E\left[\sum_{k=1}^K \ind{X_{n+1}=x_k}\cdot \textnormal{Leb}(\cC(x_k))\right]\\
    &=\sum_{k=1}^K p_k \cdot \E\left[\textnormal{Leb}(\cC(x_k))\right]\\
    &\leq \sum_{k=1}^K p_k \cdot \E\left[2(b-a)\sqrt{\frac{\log(2/\alpha)}{n_k+1}}\right]\\
    &=2(b-a)\sqrt{\log(2/\alpha)} \cdot \sum_{k=1}^K p_k \cdot\E\left[\frac{1}{\sqrt{n_k+1}}\right].
\end{align*}
Next, we will need the following fact about the Binomial distribution: it holds that
\begin{equation}\label{eqn:binomial_reciprocal}
    \textnormal{If $B\sim \textnormal{Binomial}(n,p)$ then $\E\left[\frac{1}{B+1}\right] = \frac{1-(1-p)^{n+1}}{p(n+1)} \leq \frac{1}{p(n+1)}$}.
\end{equation}
Consequently, since $n_k\sim\textnormal{Binomial}(n,p_k)$, by Jensen's inequality it holds that
\[\E\left[\frac{1}{\sqrt{n_k+1}}\right] \leq \sqrt{\E\left[\frac{1}{n_k+1}\right]} \leq \sqrt{\frac{1}{p_k(n+1)}}\leq \frac{1}{\sqrt{p_kn}},\]
for each $k\in[K]$ with $p_k>0$.
Combined with the calculations above, this yields
\[\E\left[\textnormal{Leb}(\cC(X_{n+1}))\right] \leq 2(b-a)\sqrt{\log(2/\alpha)} \cdot \sum_{k=1}^K \sqrt{\frac{p_k}{n}} \leq 2(b-a)\sqrt{\log(2/\alpha)} \cdot \sqrt{K/n},\]
where the last step holds since $\sum_{k=1}^K \sqrt{p_k}\leq \sqrt{K}$ (because $(p_1,\dots,p_K)$ lies in the probability simplex).
\end{proof}

\section{The continuous case}\label{sec:regression_continuous}
Next, we turn to the continuous setting: instead of a discrete $X$, we let $X$ have a nonatomic distribution (recall Definition~\ref{def:nonatomic}).
In this setting, unlike the discrete case, there are fundamental limits on our ability to provide informative confidence intervals for the regression function $\mu_P$, even if $\cY$ is bounded. 
Our aim in this section is to make these limits precise.

\subsection{Comparing regression intervals and prediction intervals}\label{sec:regression_leads_to_prediction}
To move towards the above aim, we begin with a surprising connection: in the nonatomic setting, any distribution-free confidence interval for $\mu_P(X_{n+1})$ must be \emph{at least as wide} as a distribution-free prediction interval for $Y_{n+1}$. 
This may seem counterintuitive, since we would expect to have more uncertainty regarding $Y_{n+1}$, which is inherently noisy, unlike $\mu_P$.
Indeed, in nonparametric regression with standard smoothness assumptions, a confidence interval for $\mu_P(X_{n+1})$ is typically much narrower than a prediction interval for $Y_{n+1}$---but this is no longer the case in the distribution-free setting.

\begin{figure}[t]
    \centering
    \includegraphics[width=0.4\textwidth]{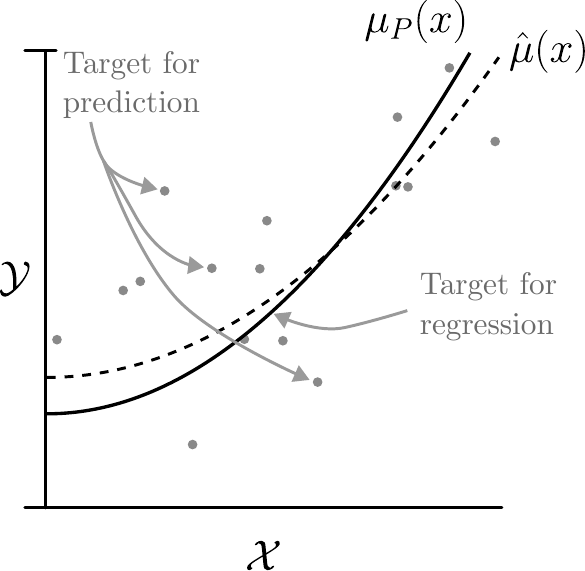}
    \caption{\textbf{Visualization of the different inference targets for regression versus for prediction.} The dots represent observed values of $(X,Y)$, the solid black line represents the true regression function $\mu_P(x)$, and the dotted black line represents an estimate, $\hat\mu(x)$. The inference targets for prediction are the data points themselves, which can be noisy. By contrast, the inference target for regression is the fixed curve $\mu_P(x)$.}
    \commentAlt{A scatterplot of $(X,Y)$ data, overlaid with a solid curve labeled $\mu_P(x)$ and a similar dashed curve labeled $\hat\mu(x)$. The solid curve is labeled as `Target for regression' and the data points are labeled as `Target for prediction'.}
    \label{fig:regression-vs-prediction}
\end{figure}

\begin{theorem}[Comparing regression and prediction (nonatomic setting)]\label{thm:regression_nonatomic}
    Let $\cY\subseteq\R$. Suppose $\cC$ is any procedure that satisfies distribution-free coverage of the regression function, i.e., for any distribution $P$ on $\cX\times\cY$, 
    \[\P\big(\mu_P(X_{n+1})\in\cC(X_{n+1})\big)\geq 1-\alpha,\]
    where the probability is taken with respect to $(X_1,Y_1),\dots,(X_{n+1},Y_{n+1})\iidsim P$, and where $\cC$ implicitly depends on $(X_1,Y_1),\dots,(X_n,Y_n)$.
    Then, for any distribution $P$ on $\cX\times\cY$ for which the marginal $P_X$ is nonatomic, 
    \[\P\big(Y_{n+1}\in\cC(X_{n+1})\big)\geq 1-\alpha.\]
\end{theorem}

In fact, the same result holds when our confidence interval $\cC$ is required to cover
the conditional median instead of the regression function (which is the conditional mean). For any distribution $P$ on $\cX\times\cY$, we will write $\textnormal{Med}_P(x)$ to denote the median of the conditional distribution of $Y\mid X=x$. \index{conditional median}
\begin{theorem}[Comparing median regression and prediction (nonatomic setting)]\label{thm:regression_nonatomic_median}
    Let $\cY\subseteq\R$. Suppose $\cC$ is any procedure that satisfies distribution-free coverage of the conditional median, i.e., for any distribution $P$ on $\cX\times\cY$, 
    \[\P\big(\textnormal{Med}_P(X_{n+1})\in\cC(X_{n+1})\big)\geq 1-\alpha,\]
    where the probability is taken with respect to $(X_1,Y_1),\dots,(X_{n+1},Y_{n+1})\iidsim P$, and where $\cC$ implicitly depends on $(X_1,Y_1),\dots,(X_n,Y_n)$.
    Then, for any distribution $P$ on $\cX\times\cY$ for which the marginal $P_X$ is nonatomic, 
    \[\P\big(Y_{n+1}\in\cC(X_{n+1})\big)\geq 1-\alpha.\]
\end{theorem}

\begin{proof}[Proof of Theorems~\ref{thm:regression_nonatomic} and~\ref{thm:regression_nonatomic_median}]
    The proof for these two hardness results will follow the sample--resample construction introduced in Lemma~\ref{lem:sample-resample}. 

    Writing $Z^{(i)} = (X^{(i)},Y^{(i)})$, as in Lemma~\ref{lem:sample-resample} we let $Z^{(1)},\dots,Z^{(M)}\iidsim P$, and define $\widehat{P}_M = \frac{1}{M}\sum_{i=1}^M \delta_{Z^{(i)}}$ as the corresponding empirical distribution. Now we calculate the conditional mean function $\mu_{\widehat{P}_M}$, and the conditional median function $\textnormal{Med}_{\widehat{P}_M}$, for this distribution. In particular, we only need to calculate $\mu_{\widehat{P}_M}(x)$ and $\textnormal{Med}_{\widehat{P}_M}(x)$ for values $x\in\{X^{(1)},\dots,X^{(M)}\}$, since this is the support of the marginal distribution $(\widehat{P}_M)_X$. We calculate
    \[\mu_{\widehat{P}_M}(X^{(i)}) = \frac{\sum_{j=1}^M Y^{(j)} \cdot\ind{X^{(j)}=X^{(i)}}}{\sum_{j=1}^M \ind{X^{(j)}=X^{(i)}}},\]
    and similarly,
    \[\textnormal{Med}_{\widehat{P}_M}(X^{(i)}) = \textnormal{Med}\left(\big(Y^{(j)}\big)_{j\in[M], X^{(j)}=X^{(i)}}\right).\]
    In particular, on the event that $X^{(1)},\dots,X^{(M)}$ are all distinct, we therefore have
    \[\mu_{\widehat{P}_M}(X^{(i)}) = \textnormal{Med}_{\widehat{P}_M}(X^{(i)}) = Y^{(i)},\]
    for all $i\in[M]$---or in other words, we have
    \begin{equation}\label{eqn:regression_function_equals_Y}\mu_{\widehat{P}_M}(X) = \textnormal{Med}_{\widehat{P}_M}(X) = Y\end{equation}
    almost surely with respect to a draw $(X,Y)\sim \widehat{P}_M$.

    Now let $(Z_i)_{i\in[n+1]} = \big((X_i,Y_i)\big)_{i\in[n+1]}\iidsim \widehat{P}_M$. If $\cC$ is a distribution-free confidence interval for either the regression function (as in Theorem~\ref{thm:regression_nonatomic}) or the conditional median (as in Theorem~\ref{thm:regression_nonatomic_median}), by~\eqref{eqn:regression_function_equals_Y} we must have
    \[\P_{\widehat{P}_M}\left(Y_{n+1}\in\cC(X_{n+1}) \,\middle|\, \widehat{P}_M\right) \geq 1-\alpha\]
    on the event that $X^{(1)},\dots,X^{(M)}$ are all distinct. Since this event holds almost surely (due to our assumption that $P_X$ is nonatomic), after marginalizing over $\widehat{P}_M$ we therefore have
    \[\P_Q\left(Y_{n+1}\in\cC(X_{n+1})\right)\geq 1-\alpha,\]
    where as in Lemma~\ref{lem:sample-resample}, $Q$ is the distribution on $(Z_i)_{i\in[n+1]}$ obtained by first sampling $Z^{(1)},\dots,Z^{(M)}\iidsim P$ and constructing $\widehat{P}_M$, and then sampling $Z_1,\dots,Z_{n+1}\iidsim \widehat{P}_M$. Applying Lemma~\ref{lem:sample-resample}, then,
     \[\P_{P^{n+1}}\left(Y_{n+1}\in\cC(X_{n+1})\right)\geq 1-\alpha - \dtv(P^{n+1},Q) \geq 1-\alpha - \frac{n(n+1)}{2M}.\]
     Since $M$ can be taken to be arbitrarily large, this completes the proof.
\end{proof}

\subsection{Impossibility of vanishing width}

The result of Theorem~\ref{thm:regression_nonatomic} shows that, in the setting where $P_X$ is nonatomic, any distribution-free confidence interval $\cC$ for $\mu_P$ must satisfy predictive coverage.
However, we recall that although we informally refer to $\cC(X_{n+1})$ as a `confidence interval', it may in fact be a disconnected set---and therefore, although we have proved that $\cC$ must offer predictive coverage, this does not necessarily imply that the set cannot have vanishing Lebesgue measure. 
For instance, in the setting of a binary response $Y\in\cY = \{0,1\}$, a set such as $\cC(X_{n+1}) = \{0\} \cup \{0.5\}\cup \{1\}$ has Lebesgue measure zero, but has a predictive coverage of 100\%.  Throughout, we will often informally refer to the Lebesgue measure of $\cC(X_{n+1})$ as the `width' of this set, even for sets that are not intervals.

This section shows that a valid confidence interval $\cC(X_{n+1})$ cannot have vanishing width.
To show this, we will need to carry out a more refined analysis of the problem, which will require developing a more complex version of the sample--resample construction. The resulting bound will show that the Lebesgue measure of our confidence interval, $\textnormal{Leb}(\cC(X_{n+1}))$, can be bounded away from zero even as $n\to\infty$, meaning that any $\cC$ with distribution-free validity cannot have vanishing width.
Figure~\ref{fig:regression-hard} visualizes the core difficulty: in a distribution-free setting, using only finitely many data points drawn from the unknown distribution $P$, it is impossible to certify that the regression function $\mu_P$ is smooth. Consequently, we are forced to build confidence intervals that are wide enough to accommodate highly nonsmooth possibilities for $\mu_P$.

\begin{figure}[t]
    \centering
    \includegraphics[width=0.7\textwidth]{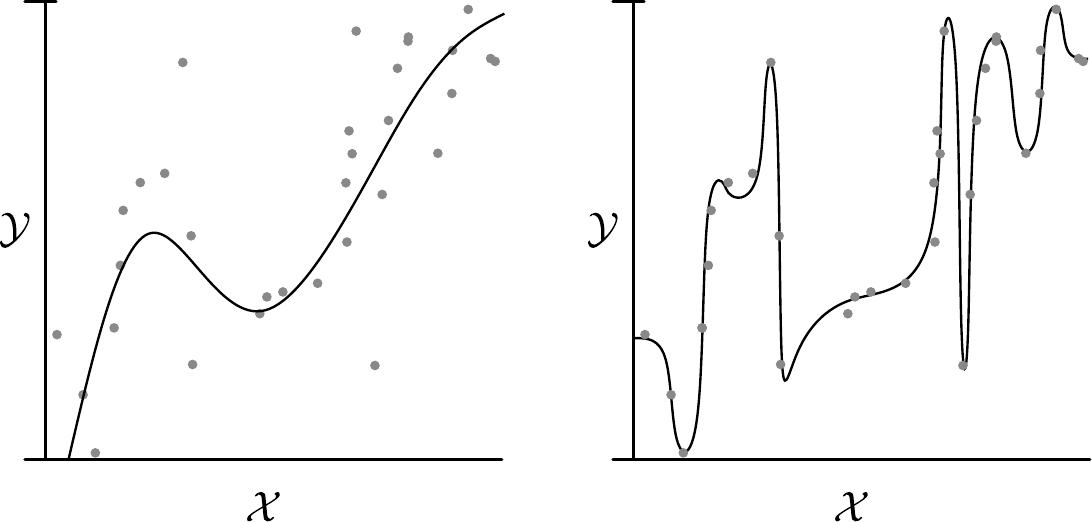}
    \caption{\textbf{Visualization of the difficulty in distribution-free regression.} In each plot, the dots represent observed values of $(X,Y)$ (which are identical across the two plots), while the solid black lines represent two possibilities for the true regression function $\mu_P(x)$. Without assumptions, it is impossible to test whether the observed data points $(X,Y)$ are drawn from a noisy distribution characterized by a smooth regression function $\mu_P$ (as in the panel on the left), or are drawn from a nearly noiseless distribution with a highly nonsmooth regression function (as in the panel on the right).}
    \commentAlt{Two scatterplots show the same $(X,Y)$ points. The left plot shows a smooth curve, with all data points lying a moderate distance from the curve. The right plot has a highly fluctuating curve, with all data points lying extremely close to the curve.}
    \label{fig:regression-hard}
\end{figure}
\begin{theorem}[Hardness of inference for regression in the nonatomic setting]\label{thm:regression_nonatomic_width}
    Let $\cY = [a,b]$, and suppose $\cC$ is any procedure that satisfies distribution-free coverage of the regression function~\eqref{eqn:DF_confidence_regression}, i.e., for any distribution $P$ on $\cX\times[a,b]$, 
    \[\P\big(\mu_P(X_{n+1})\in\cC(X_{n+1})\big)\geq 1-\alpha,\]
    where the probability is taken with respect to $(X_1,Y_1),\dots,(X_{n+1},Y_{n+1})\iidsim P$, and where $\cC$ implicitly depends on $(X_1,Y_1),\dots,(X_n,Y_n)$.
    Then, for any distribution $P$ on $\cX\times[a,b]$ for which the marginal $P_X$ is nonatomic and for which $\Var(Y\mid X)\geq \sigma^2_*$ almost surely, 
    \[\E\left[\textnormal{Leb}\big(\cC(X_{n+1})\big)\right]\geq \frac{\sigma^2_*}{b-a} \cdot 2(1-\alpha).\]
\end{theorem} \index{hardness result!regression}
Because the lower bound in Theorem~\ref{thm:regression_nonatomic_width} does not depend on $n$, no distribution-free confidence interval can have vanishing width as $n\to\infty$.
Thus, it is impossible to obtain vanishing-width distribution-free confidence intervals on the regression function $\mu_P$ in a nonatomic setting.

\begin{proof}[Proof of Theorem~\ref{thm:regression_nonatomic_width}]
    Without loss of generality, we can take $[a,b]=[0,1]$. Fix any distribution $P$ with nonatomic $P_X$. Below we will verify that, for any $t\in[0,1]$, it holds that
    \begin{equation}\label{eqn:mirror_image_for_thm:regression_nonatomic_width}\frac{1}{2}\P\big(\mu_P(X_{n+1}) - t\sigma^2_* \in \cC(X_{n+1})\big) + \frac{1}{2}\P\big(\mu_P(X_{n+1}) + t\sigma^2_* \in \cC(X_{n+1})\big) \geq 1-\alpha.\end{equation}
    Assuming that this holds, we then calculate
    \begin{align*}
        \E\left[\textnormal{Leb}\big(\cC(X_{n+1})\big)\right]
        &=\E\left[\int_\R\ind{y\in \cC(X_{n+1})}\;\mathsf{d}y\right]\\
        &\geq \E\left[\int_{y=\mu_P - \sigma^2_*}^{\mu_P+\sigma^2_*}\ind{y\in \cC(X_{n+1})}\;\mathsf{d}y\right]\\
        &=\sigma^2_* \E\left[\int_{t=-1}^1\ind{\mu_P + t\sigma^2_*\in \cC(X_{n+1})}\;\mathsf{d}t\right]\\
        &=\sigma^2_* \int_{t=-1}^1\P\left(\mu_P + t\sigma^2_*\in \cC(X_{n+1})\right)\;\mathsf{d}t\textnormal{ by the Fubini--Tonelli theorem}\\
        &=\sigma^2_* \int_{t=0}^1\Big[\P(\mu_P - t\sigma^2_*\in \cC(X_{n+1})) + \P(\mu_P + t\sigma^2_*\in \cC(X_{n+1}))\Big]\;\mathsf{d}t\\
        &\geq \sigma^2_* \cdot 2(1-\alpha),
    \end{align*}
    where the last step applies~\eqref{eqn:mirror_image_for_thm:regression_nonatomic_width}. This proves the desired bound. 
    
    For the remainder of the proof, then, our task is to verify the claim~\eqref{eqn:mirror_image_for_thm:regression_nonatomic_width}. The proof of this claim will rely on a variant of the sample--resample technique (Lemma~\ref{lem:sample-resample}). First, let $P_{Y\mid X}$ denote the conditional distribution of $Y\mid X$, and let $\textnormal{Med}_P(X)$ denote the median of this conditional distribution. Following Lemma~\ref{lem:split_distribution_into_two} below, let $P_{Y\mid X} = \frac{1}{2} P_{Y\mid X}^0 + \frac{1}{2}P_{Y\mid X}^1$ denote the decomposition of $P_{Y\mid X}$ into distributions $P_{Y\mid X}^0,P_{Y\mid X}^1$ supported on $[0,\textnormal{Med}_P(X)]$ and on $[\textnormal{Med}_P(X),1]$, respectively. Writing $\mu_P^0(X)$ and $\mu_P^1(X)$ as the means of these two distributions, by Lemma~\ref{lem:split_distribution_into_two} we have $\mu_P^1(X)-\mu_P^0(X) \geq 2\Var(Y\mid X) \geq 2\sigma^2_*$, almost surely.

    Now fix any $t\in[0,1]$. First we define a joint distribution $\tilde{P}$ on $\cX\times\{\pm 1\}\times [0,1]$ as follows: sample $(X,B)\sim P_X\times\textnormal{Unif}(\{\pm 1\})$, then sample $Y\mid (X,B) \sim \tilde{P}_{Y\mid (X,B)}$ where
    \[\tilde{P}_{Y\mid (X,B)} =\left(\frac{1}{2} - \frac{B\cdot t\sigma^2_*}{\mu_P^1(X) - \mu_P^0(X)}\right) \cdot P^0_{Y\mid X} + \left(\frac{1}{2} + \frac{B \cdot t\sigma^2_*}{\mu_P^1(X) - \mu_P^0(X)}\right)\cdot P^1_{Y\mid X}.\]
    For intuition, this construction is designed so that, if $B=+1$ (respectively, $-1$), then $Y$ is slightly more likely (respectively, slightly less likely) to lie above its conditional median $\textnormal{Med}_P(X)$.
    (Note that, since $\mu_P^1(X) - \mu_P^0(X)\geq 2\sigma^2_*$, the weights on $P^0_{Y\mid X}$ and on $P^1_{Y\mid X}$ are each nonnegative in the mixture distribution for $Y$, i.e., the mixture distribution is well-defined.) We can calculate that
    \begin{multline}\label{eqn:calculate_EY_XB}\E_{\tilde{P}}[Y\mid X,B] = \left(\frac{1}{2} - \frac{B\cdot t\sigma^2_*}{\mu_P^1(X) - \mu_P^0(X)}\right) \cdot \mu_P^0(X) + \left(\frac{1}{2} + \frac{B \cdot t\sigma^2_*}{\mu_P^1(X) - \mu_P^0(X)}\right)\cdot \mu_P^1(X)\\= \mu_P(X) + t\sigma^2_*\cdot B.\end{multline}
    using the fact that $P_{Y\mid X} = \frac{1}{2}P^0_{Y\mid X} + \frac{1}{2}P^1_{Y\mid X}$ and so $\mu_P(X) = \frac{1}{2}\mu_P^0(X) + \frac{1}{2}\mu_P^1(X)$.
    Moreover, by construction, the marginal distribution of $(X,Y)$ under $\tilde{P}$ is equal to the original joint distribution $P$.

    Next, fix any $M\geq 1$.
    Let $(X^{(1)},B^{(1)}),\dots,(X^{(M)},B^{(M)})\iidsim P_X\times\textnormal{Unif}(\{\pm 1\})$, and define $\widehat{P}_M = \frac{1}{M}\sum_{i=1}^M \delta_{(X^{(i)},B^{(i)})}$ as the corresponding empirical distribution. 
    Now we define a distribution $\tilde{P}_M$ on $(X,Y)\in\cX\times[0,1]$ as follows: sample $(X,B)\sim \widehat{P}_M$, then sample $Y\mid (X,B)\sim \tilde{P}_{Y\mid X,B}$, and return $(X,Y)$. Applying~\eqref{eqn:calculate_EY_XB}, we can then calculate
    \[        \mu_{\tilde{P}_M}(X^{(i)})
    =\frac{\sum_{j=1}^M \big(\mu_P(X^{(j)}) + t\sigma^2_* \cdot B^{(j)}\big) \cdot \ind{X^{(j)}=X^{(i)}}}{\sum_{j=1}^M\ind{X^{(j)}=X^{(i)}}},
    \]
    which simplifies to
    \[\mu_{\tilde{P}_M}(X^{(i)}) = \mu_P(X^{(i)}) + t\sigma^2_* \cdot B^{(i)}\]
    if $X^{(1)},\dots,X^{(M)}$ are all distinct.

Now let $(X_1,Y_1),\dots,(X_{n+1},Y_{n+1})\iidsim \tilde{P}_M$---or equivalently, we can generate this data by drawing $(X_1,B_1),\dots,(X_{n+1},B_{n+1})$ uniformly \emph{with} replacement from $\widehat{P}_M$, and then sampling the $Y_i$'s according to the conditional distribution $\tilde{P}_{Y\mid X,B}$.
    Since $\cC$ satisfies coverage of the regression function with respect to any data distribution, in particular coverage must hold with respect to $\tilde{P}_M$, and therefore
    \[\P\Big( \mu_{\tilde{P}_M}(X_{n+1}) \in\cC(X_{n+1}) \,\Big|\, \widehat{P}_M \Big) \geq 1-\alpha.\]
    In particular, from our calculation above we must have
    \[\P\Big( \mu_P(X_{n+1}) + t\sigma^2_* \cdot B_{n+1} \in\cC(X_{n+1}) \,\Big|\, \widehat{P}_M \Big) \geq 1-\alpha\]
    if $X^{(1)},\dots,X^{(M)}$ are all distinct---and since this holds almost surely (because $P_X$ is assumed to be nonatomic), after marginalizing over $\widehat{P}_M$ we have
    \[\P\Big( \mu_P(X_{n+1}) + t\sigma^2_* \cdot B_{n+1} \in\cC(X_{n+1})  \Big) \geq 1-\alpha.\]

    Next, consider an alternative distribution: assuming that $M\geq n+1$, suppose that $(X_1,B_1),\dots,(X_{n+1},B_{n+1})$ are sampled uniformly \emph{without} replacement from $\widehat{P}_M$, and then $Y_i$'s are again sampled according to the conditional distribution $\tilde{P}_{Y\mid X,B}$. By construction, after marginalizing over the random draw of $\widehat{P}_M$, we can see that this is equivalent to simply drawing $(X_1,B_1,Y_1),\dots,(X_{n+1},B_{n+1},Y_{n+1})\iidsim \tilde{P}$.
    Bounding the difference between sampling with versus without replacement as in Lemma~\ref{lem:sample-resample}, we then have
    \[\P_{\tilde{P}^{n+1}}\Big( \mu_P(X_{n+1}) + t\sigma^2_* \cdot B_{n+1} \in\cC(X_{n+1})\Big) \geq 1-\alpha - \frac{n(n+1)}{2M}.\]
    By taking $M$ to be arbitrarily large, we therefore have
    \[\P_{\tilde{P}^{n+1}}\Big( \mu_P(X_{n+1}) + t\sigma^2_* \cdot B_{n+1} \in\cC(X_{n+1})\Big) \geq 1-\alpha.\]
    Now observe that this event depends only on $(X_1,Y_1),\dots,(X_n,Y_n)$ and on $(X_{n+1},B_{n+1})$. Recalling that under the joint distribution $\tilde{P}$ on $(X,B,Y)$, we have marginal distributions $(X,Y)\sim P$ and $(X,B)\sim P_X\times \textnormal{Unif}(\{\pm1\})$, we can therefore equivalently write
    \[\P_{P^n\times P_X\times \textnormal{Unif}(\{\pm1\})}\Big( \mu_P(X_{n+1}) + t\sigma^2_* \cdot B_{n+1} \in\cC(X_{n+1})\Big) \geq 1-\alpha,\]
    where the probability is taken over $(X_1,Y_1),\dots,(X_n,Y_n)\iidsim P$ (implicitly used in the construction of $\cC$) and $(X_{n+1},B_{n+1})\sim P_X\times\textnormal{Unif}(\{\pm 1\})$. Since $B_{n+1}\sim \textnormal{Unif}(\{\pm 1\})$, we have therefore verified~\eqref{eqn:mirror_image_for_thm:regression_nonatomic_width}, as desired.
\end{proof}

\begin{lemma}\label{lem:split_distribution_into_two}
   Let $Y\in[0,1]$ be a random variable with distribution $P$, and let $\textnormal{Med}_P$ and $\sigma^2_P$ denote its median and its variance. 
    Consider the unique decomposition of $P$ into a mixture
   \begin{equation}\label{eqn:split_distribution}P = \frac{1}{2}P_0 + \frac{1}{2}P_1\end{equation}
   such that $P_0$ is supported on $[0,\textnormal{Med}_P]$ and $P_1$ is supported on $[\textnormal{Med}_P,1]$. Then
   \[\E_{P_1}[Y] - \E_{P_0}[Y] \geq 2\sigma^2_P.\]
\end{lemma}
See Figure~\ref{fig:split-distribution} for a visualization of the decomposition in~\eqref{eqn:split_distribution}.
\begin{proof}[Proof of Lemma~\ref{lem:split_distribution_into_two}]
    First we formally construct $P_0,P_1$. Let $F$ denote the CDF of the distribution $P$, and let $F^{-1}(t) = \quantile(P;t)= \inf\{y\in[0,1] : F(y)\geq t\}$ be its generalized inverse, so that $F^{-1}(U)\sim P$ when $U\sim\textnormal{Unif}[0,1]$. In particular note that $F^{-1}(0.5) = \textnormal{Med}_P$.
    Let $P_0$ denote the distribution of $F^{-1}(U)$ for $U\sim\textnormal{Unif}[0,0.5]$, and let $P_1$ denote the distribution of $F^{-1}(U)$ for $U\sim\textnormal{Unif}[0.5,1]$. Then $P = \frac{1}{2}P_0 + \frac{1}{2}P_1$, by construction. And, since $F^{-1}$ is monotone nondecreasing, $P_0$ is supported on $[0,F^{-1}(0.5)] = [0,\textnormal{Med}_P]$, and $P_1$ is supported on $[F^{-1}(0.5),1] = [\textnormal{Med}_P,1]$. Moreover, it is straightforward to verify by construction that this decomposition is unique.

    With this construction in place, we now verify the bound on the difference in means.    
    Define
    \[\mu_0 = \E_{P_0}[Y], \quad \mu_1 = \E_{P_1}[Y],\quad \mu_P = \E_P[Y] = \frac{\mu_0+\mu_1}{2}.\]
    We then calculate
    \begin{align*}
        \sigma^2_P &= \E_P\left[\left(Y - \mu_P\right)^2\right]
        = \frac{1}{2}\E_{P_0}\left[\left(Y - \mu_P\right)^2\right] + \frac{1}{2}\E_{P_1}\left[\left(Y - \mu_P\right)^2\right]\\
        &= \frac{1}{2}\Var_{P_0}(Y) + \frac{1}{2}\big(\mu_0 - \mu_P\big)^2 + \frac{1}{2}\Var_{P_1}(Y) + \frac{1}{2}\big(\mu_1 - \mu_P\big)^2\\
        &=\frac{1}{2}\Var_{P_0}(Y) + \frac{1}{2}\Var_{P_1}(Y) + \frac{1}{4}\big(\mu_1 - \mu_0\big)^2.
    \end{align*}
    Since $P_0$ is supported on $[0,\textnormal{Med}_P]$, with mean $\mu_0$, its variance is therefore bounded as
    \begin{multline*}\Var_{P_0}(Y)  \leq \mu_0\big(\textnormal{Med}_P - \mu_0\big) = \mu_0\big(\textnormal{Med}_P - \mu_P\big) + \mu_0(\mu_P-\mu_0) \\= \mu_0\big(\textnormal{Med}_P - \mu_P\big) + \frac{1}{2}\mu_0(\mu_1-\mu_0).\end{multline*}
    Similarly,
    \[\Var_{P_1}(Y)  \leq (1-\mu_1)\big(\mu_1 - \textnormal{Med}_P\big) = (1-\mu_1)\big(\mu_P - \textnormal{Med}_P\big) + \frac{1}{2}(1-\mu_1)(\mu_1-\mu_0).\]    Therefore,
\begin{align*}
    \sigma^2_P 
    &\leq \frac{1}{2}(\mu_0-(1-\mu_1))\big(\textnormal{Med}_P-\mu_P\big) + \frac{1}{4}(\mu_0 + (1-\mu_1)) (\mu_1-\mu_0) + \frac{1}{4}\big(\mu_1 - \mu_0\big)^2\\
    &=\frac{1}{2}(2\mu_P - 1)\big(\textnormal{Med}_P-\mu_P\big) + \frac{1}{4}(\mu_1-\mu_0)\\
    &\leq\frac{1}{2}\big|\textnormal{Med}_P-\mu_P\big| + \frac{1}{4}(\mu_1-\mu_0),
\end{align*}
where the last step holds since $\mu_P\in[0,1]$ and so $|2\mu_P-1|\leq 1$.
Moreover, since $\mu_0\leq \textnormal{Med}_P\leq \mu_1$,
\[\big|\textnormal{Med}_P - \mu_P\big| \leq \max\left\{|\mu_1 - \mu_P| ,|\mu_0- \mu_P|\right\} =  \frac{1}{2}(\mu_1-\mu_0).\]
Therefore
$\sigma^2_P \leq \frac{1}{2}(\mu_1 -\mu_0)$, which completes the proof.
\end{proof}

\begin{figure}[t]
    \centering
    \includegraphics[width=0.75\textwidth]{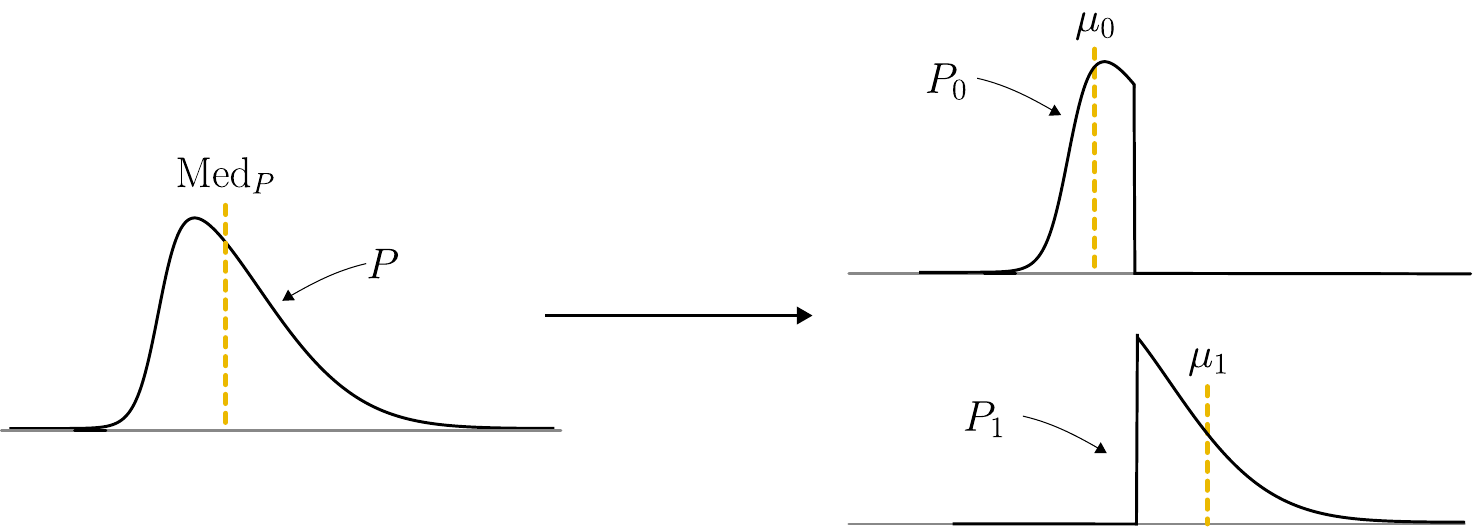}
    \caption{\textbf{Visualization of Lemma~\ref{lem:split_distribution_into_two}.} The panel on the left shows the density for a distribution $P$, with its median $\textnormal{Med}_P$ indicated by the dashed vertical line. In the panels on the right, we decompose $P$ into a mixture of two distributions $P_0$ and $P_1$, with means $\mu_0$ and $\mu_1$, respectively. As we can see in the figure, the difference in means, $\mu_1-\mu_0$, cannot be too small relative to the variance of the original distribution $P$.}
    \commentAlt{The left panel shows density $P$, with a dashed vertical line at $\textnormal{Med}_P$. See long description.}
    \commentLongAlt{The left panel shows density $P$, with a dashed vertical line at $\textnormal{Med}_P$. The right panel shows two densities. $P_0$ is the portion of $P$ lying to the left of the cutoff $\textnormal{Med}_P$, and has a dashed vertical line indicating its mean $\hat\mu_0$, which lies slightly to the left of the cutoff. Similarly, $P_1$ is the portion of $P$ lying to the right of the cutoff $\textnormal{Med}_P$, and has a dashed vertical line indicating its mean $\hat\mu_1$, which lies slightly to the right of the cutoff.}
    \label{fig:split-distribution}
\end{figure}

\section{Relaxing the target for the continuous case}\label{sec:regression_relax}
As we have now seen, in the setting where the feature $X$ has a nonatomic distribution, it is impossible to provide vanishing-width distribution-free inference on $\mu_P(X)$ in the original sense of~\eqref{eqn:DF_confidence_regression}. But can we modify the definition of validity to enable a more informative distribution-free inference procedure? In this section, we will consider several possible relaxations.

\subsection{Relaxation by binning}\label{sec:regression_relax_binning}
We first consider a binning-based approach that effectively converts our question from a continuous problem to a discrete problem. 
Consider a (prespecified) partition $\cX = \cX_1 \cup \dots \cup \cX_K$. For each $k\in[K]$, we define a new target,
\[\mu_P(\cX_k) = \E[Y\mid X\in\cX_k],\]
which measures the conditional expectation of $Y$ when $X$ lies in the bin $\cX_k$, under the unknown distribution $P$. Defining $k(x)\in[K]$ as the index of the bin containing $x$, we can then relax the original notion of valid coverage~\eqref{eqn:DF_confidence_regression} to the following: we aim to provide a confidence interval $\cC(\cX_k)$ for $\mu_P(\cX_k)$ satisfying
\begin{equation}\label{eqn:DF_confidence_regression_binned}\P\big(\mu_P(\cX_{k(X_{n+1})}) \in \cC(\cX_{k(X_{n+1})})\big) \geq 1-\alpha.\end{equation}
That is, we replace the original goal (coverage of $\mu_P(X)$, the mean of $Y$ given $X$) with its binning-based relaxation. We now construct a distribution-free confidence interval for this relaxed target.

\begin{theorem}[Inference for the regression function via binning]\label{thm:DF_regression_binned}
Let $P$ be a distribution on $\cX\times\cY$, where $\cY \subseteq [a,b]$, and let $\alpha\in(0,1)$. Let $\cX = \cX_1\cup\dots\cup\cX_K$ be a fixed partition.
    Let $(X_1,Y_1),\dots,(X_n,Y_n)\iidsim P$. For each $k\in[K]$, define $n_k =\sum_{i=1}^n\ind{X_i \in\cX_k}$, and let
    \[\hat{\mu}(\cX_k) = \frac{1}{n_k}\sum_{i=1}^n Y_i \cdot\ind{X_i \in\cX_k}\]
    for each $k$ with $n_k\geq 1$, i.e., the mean response observed among all data points with $X_i \in\cX_k$. Define
    \[\cC(\cX_k) = \begin{cases} \hat{\mu}(\cX_k)\pm (b-a)\sqrt{\frac{\log(2/\alpha)}{2n_k}}, & \textnormal{ if $n_k\geq 1$},\\ [a,b], & \textnormal{ if $n_k=0$.}\end{cases}\]
    Then $\cC$ satisfies~\eqref{eqn:DF_confidence_regression_binned}.
\end{theorem}
\begin{proof}[Proof of Theorem~\ref{thm:DF_regression_binned}]
    Let $\tilde{P}$ be the distribution of $(\tilde{X},Y)=(k(X),Y)$, when $(X,Y)\sim P$. Then $\mu_P(\cX_k) = \E_{\tilde{P}}[Y\mid \tilde{X}=k] = \mu_{\tilde{P}}(k)$, by construction, meaning that the relaxed inference target~\eqref{eqn:DF_confidence_regression_binned} for $P$ is exactly equivalent to the original aim~\eqref{eqn:DF_confidence_regression} for $\tilde{P}$. The result then follows by applying Theorem~\ref{thm:DF_regression_discrete} for the distribution $\tilde{P}$, which has a discrete covariate.
\end{proof}
We can observe a parallel between the ideas explored here for inference on $\mu_P$, and the results of Section~\ref{sec:test-conditional-binning}, where we studied a binning-based relaxation for the problem of test-conditional predictive inference: in both cases, we circumvent an impossibility result (for nonatomic $X$) by relaxing the problem via binning (effectively moving to the setting of a discrete $X$). 

\subsection{Relaxation by blurring the target}\label{sec:regression_relax_blurring}
The previous method, defined in Theorem~\ref{thm:DF_regression_binned}, replaces the target $\mu_P(x)$ with its binned version, $\mu_P(\cX_k)$, in order to make the inference problem easier. We next consider a more general version of this approach.

Fix a function $H:\cX\times\cX\to[0,\infty)$, with large values of $H(x,x')$ indicating that $x,x'\in\cX$ are similar, or lie nearby to each other. As in Chapter~\ref{chapter:weighted-conformal} (recall the localized conformal prediction method introduced in Section~\ref{sec:localized}), we refer to $H$ as a \emph{localization kernel}, and can think of a Gaussian kernel, $H(x,x') = \exp\{-\|x-x'\|^2_2/2h^2\}$, as a canonical example.
Next, define $\tilde\mu_P(x)$ as an approximation to $\mu_P(x)$, 
\begin{equation}\label{eq:define_blurred_regression}\tilde\mu_P(x) = \frac{\E_P[\mu_P(X) \cdot H(x,X)]}{\E_P[H(x,X)]}\end{equation}
(implicitly, we assume $\E_P[H(x,X)]>0$ for all $x\in\cX$). In particular, the binned mean $\mu_P(\cX_k)$ can be viewed as a special case: by defining $H(x,x') = \sum_{k=1}^K \ind{x,x'\in\cX_k}$ (that is, $H(x,x')=1$ whenever $x$ and $x'$ lie in the same bin), we obtain $\tilde\mu_P(x) = \mu_P(\cX_k)$ for any $x\in\cX_k$.
More generally, if $x\mapsto \mu_P(x)$ is reasonably smooth, then we should expect $\tilde\mu_P(x)\approx \mu_P(x)$, as long as the kernel $H$ is reasonably strongly localized---for instance, a Gaussian kernel with a small bandwidth $h>0$, or a binning-based kernel with small bin size. 

We will now see that we can perform  inference on $\tilde\mu_P(x)$ regardless of the properties of the underlying distribution, via a rejection sampling strategy.

\begin{theorem}[Inference for the regression function via blurring]\label{thm:DF_regression_blurred}
    Let $P$ be a distribution on $\cX\times\cY$, where $\cY\subseteq[a,b]$, and let $\alpha\in(0,1)$. Let $H:\cX\times\cX\to[0,B]$ be a function satisfying  $\E_P[H(x,X)]>0$ for all $x\in\cX$. Let $(X_1,Y_1),\dots,(X_n,Y_n)\iidsim P$.
    
    Let $U_1,\dots,U_n\iidsim \textnormal{Unif}[0,1]$ be drawn independently of the data. For each $x\in\cX$, define $n(x) = \sum_{i=1}^n \ind{U_i \leq \frac{H(x,X_i)}{B}}$, and let
    \[\hat\mu(x) = \frac{1}{n(x)}\sum_{i=1}^n Y_i \cdot \ind{U_i \leq \frac{H(x,X_i)}{B}}\]
    if $n(x)\geq 1$. 
    Define
    \[\cC(x) = \begin{cases} \hat{\mu}(x)\pm (b-a)\sqrt{\frac{\log(2/\alpha)}{2n(x)}}, & \textnormal{ if $n(x)\geq 1$},\\ [a,b], & \textnormal{ if $n(x)=0$.}\end{cases}\]
    Then $\cC$ satisfies
    \[\P(\tilde\mu_P(x)\in\cC(x))\geq 1-\alpha\]
    for every $x\in\cX$,
    where $\tilde\mu_P(x)$ is defined as in~\eqref{eq:define_blurred_regression}.
\end{theorem}

\begin{proof}[Proof of Theorem~\ref{thm:DF_regression_blurred}]
    Fix any $x\in\cX$. Let $P_X$ and $P_{Y\mid X}$ denote the marginal distribution of $X$ and the conditional distribution of $Y\mid X$, respectively, under $P$. Recalling the covariate shift setting studied in Section~\ref{sec:covariate-shift}, define a distribution
    \[Q = Q_X\times P_{Y\mid X},\]
    where $Q_X$ is the distribution on $\cX$ defined by the Radon--Nikodym derivative
    \[\frac{\mathsf{d}Q_X}{\mathsf{d}P_X}(x',y) \propto H(x,x').\]
    Moreover, writing $Q_Y$ to denote the marginal of $Y$ under the joint distribution $Q$, we can calculate its mean as
    \begin{align*}
        \mu_{Q_Y}
        &= \E_Q[Y]
        =\E_P\left[Y \cdot \frac{\mathsf{d}Q}{\mathsf{d}P}(X,Y)\right]\\
        &=\E_P\left[Y \cdot \frac{H(x,X)}{\E_P[H(x,X)]}\right]\\
        &=\E_P\left[ \mu_P(X) \cdot \frac{H(x,X)}{\E_P[H(x,X)]}\right]\\
        &=\tilde\mu_P(x).
    \end{align*}
    
    Next, let $\cI(x) = \{i\in[n]: U_i\leq \frac{H(x,X_i)}{B}\}$, and note that $n(x) = |\cI(x)|$. Then the data points $\big((X_i,Y_i)\big)_{i\in\cI(x)}$ can be viewed as the output of rejection sampling, when run on available data sampled from the distribution $P$, and when the target distribution is given by $Q$. More specifically, conditioning on $n(x)$, these data points are distributed as $n(x)$ many i.i.d.\ draws from $Q$. In particular, this means that the corresponding response values, $(Y_i)_{i\in \cI(x)}$, are distributed as $n(x)$ many i.i.d.\ draws from $Q_Y$.

    Therefore, conditional on $n(x)$ (and on the event $n(x)\geq 1$), the random variable $\hat\mu(x)$ is the sample mean of $n(x)$ many i.i.d.\ draws from $Q_Y$, which is a distribution supported on $[a,b]$. By Hoeffding's inequality, therefore, $\cC(x)$ provides a confidence interval for $\mu_{Q_Y} = \tilde\mu_P(x)$ at level $1-\alpha$---that is,
    \[\P(\tilde\mu_P(x)\in\cC(x)\mid n(x)) \geq 1-\alpha,\]
    on the event $n(x)\geq 1$.
    Since we also have $\P(\tilde\mu_P(x)\in\cC(x)\mid n(x)=0)=1$ by construction,
    marginalizing over $n(x)$ completes the proof.
\end{proof}

\section{Connections to test-conditional predictive coverage}\label{sec:regression_vs_conditional_prediction}
\index{coverage!test-conditional|(}

The results developed in this chapter bear strong resemblance to our findings for an earlier question: that of test-conditional coverage for predictive inference, in Chapter~\ref{chapter:conditional}. In particular, for both problems, we have seen that the discrete case offers a straightforward solution, while the continuous case (or, more precisely, the case where $X$ is nonatomic) faces a fundamental hardness result that prohibits meaningful solutions. However, at least on the surface, the two questions appear to be quite different: in Chapter~\ref{chapter:conditional} we were aiming for \emph{conditional} validity for inference on the \emph{prediction} problem, 
\[\P(Y_{n+1}\in\cC(X_{n+1}) \mid X_{n+1})\geq 1-\alpha,\] while here we are aiming for \emph{marginal} validity for inference on the \emph{regression} problem,
\[\P(\mu_P(X_{n+1})\in\cC(X_{n+1}))\geq 1-\alpha.\]

In this section, we will examine the way in which these two problems are actually essentially equivalent. 
At a high level, this is true for the following reason: given a prediction set $\cC$, if we define a new response variable $\tilde{Y} = \ind{Y\not\in \cC(X)}$, then 
\[\P(Y\in\cC(X)\mid X)\geq 1-\alpha \textnormal{ almost surely} \ \Longleftrightarrow \ \E[\tilde{Y}\mid X]\leq \alpha\textnormal{ almost surely}.\]
That is, $\cC$ satisfies test-conditional coverage at level $1-\alpha$ if and only if the regression function $\mu_{\tilde{P}}(X)$ (where $\tilde{P}$ denotes the joint distribution of $(X,\tilde{Y})$) is always $\leq \alpha$. Therefore, our ability to perform inference on $\mu_{\tilde{P}}(X)$ is closely linked to our ability to provide test-conditional coverage for $P$. 
(For clarity, we emphasize that this question is distinct from the results of Section~\ref{sec:regression_leads_to_prediction}, where we discussed connections to marginal, rather than conditional, predictive coverage.)

\subsection{From regression to test-conditional prediction}
We will now develop these ideas in more detail.
Concretely, we will develop an explicit construction to show that any solution $\cC_{\textnormal{regr}}$ to the problem of distribution-free confidence intervals for regression can be leveraged to construct a prediction interval $\cC_{\textnormal{pred}}$ that offers a relaxed notion of distribution-free test-conditional predictive coverage. 

Suppose we have data $(X_i,Y_i)\iidsim P$, for some unknown distribution $P$ on $(X,Y)\in\cX\times\cY$, and we also have a pretrained prediction set $\cC_{\textnormal{init}}$ that does not depend on this data---for example, we might have constructed $\cC_{\textnormal{init}}$ by running conformal prediction on an independent batch of data. We will now develop a strategy to convert $\cC_{\textnormal{init}}$ into a prediction interval that offers test-conditional predictive coverage for $P$.

Assume that we have access to a procedure $\cC_{\textnormal{regr}}$ that provides distribution-free inference for regression, as in~\eqref{eqn:DF_confidence_regression}, in the setting of a bounded response variable: it holds that
\begin{equation}\label{eqn:DF_confidence_regression_eps}
    \textnormal{For any distribution $\tilde{P}$ on $\cX\times[0,1]$, \
    $\P(\mu_{\tilde{P}}(X_{n+1})\in\cC_{\textnormal{regr}}(X_{n+1}))\geq 1-\epsilon$,}
\end{equation}
for some small constant $\epsilon>0$, when $\cC_{\textnormal{regr}}$ is trained on an i.i.d.\ sample of size $n$ drawn from $\tilde{P}$. (Note that this is the same notion of validity as defined earlier in~\eqref{eqn:DF_confidence_regression}, aside from notation---we replace $\cC,\alpha,P$ with $\cC_{\textnormal{regr}},\epsilon,\tilde{P}$ to distinguish the regression problem from the prediction problem.)

We will now see that $\cC_{\textnormal{regr}}$ can be used to help construct a prediction interval that satisfies the relaxed notion of test-conditional coverage defined in the condition~\eqref{eqn:test_conditional_relax_alpha_delta}, which we studied in Chapter~\ref{chapter:conditional}.
Write $\tilde{P}$ to denote the distribution of 
\[(X,\tilde{Y}) = \big(X,\ind{Y\not\in\cC_{\textnormal{init}}(X)}\big)\in\cX\times\{0,1\}.\] 
Similarly, define $\tilde{Y}_i = \ind{Y_i \not\in \cC_{\textnormal{init}}(X_i)}$, for $i\in[n]$, so that we have $(X_1,\tilde{Y}_1),\dots,(X_n,\tilde{Y}_n)\iidsim \tilde{P}$. 
We now train $\cC_{\textnormal{regr}}$ on the data $\big((X_i,\tilde{Y}_i)\big)_{i\in[n]}$, and finally we
 define our prediction set as
\[\cC_{\textnormal{pred}}(x) = \begin{cases} \cC_{\textnormal{init}}(x), &\textnormal{ if }\cC_{\textnormal{regr}}(x)\subseteq [0,\alpha -\epsilon/\delta],\\ \cY,&\textnormal{ otherwise},\end{cases}\]
for some parameters $\alpha,\delta>0$. Effectively, we are using $\cC_{\textnormal{regr}}$ to \emph{certify} that $\cC_{\textnormal{init}}$ has test-conditional coverage---or if not, to return the conservative answer $\cY$ instead.
\begin{proposition}[Relating inference for regression to test-conditional prediction]\label{prop:from_regression_to_test_conditional_prediction}
    Under the assumptions and definitions above, the prediction set $\cC_{\textnormal{pred}}$ satisfies
    \[\P(Y_{n+1}\in\cC_{\textnormal{pred}}(X_{n+1})\mid X_{n+1}\in\cX_0) \geq 1 - \alpha\textnormal{ for all $P$ and all $\cX_0\subseteq\cX$ with $\P_P(X\in\cX_0)\geq\delta$}.\]
\end{proposition}
This guarantee on $\cC_{\textnormal{pred}}$ is a relaxed form of test-conditional coverage for predictive inference---in fact, it is exactly the same as the condition~\eqref{eqn:test_conditional_relax_alpha_delta} studied in Chapter~\ref{chapter:conditional}.

In the setting of a nonatomic $X$, we have previously established hardness results for inference for regression (Theorem~\ref{thm:regression_nonatomic_width}) and for (relaxed) test-conditional predictive inference (Theorem~\ref{thm:hardness-test-conditional-coverage-relaxed}). Through Proposition~\ref{prop:from_regression_to_test_conditional_prediction}, we now see that these two inference problems are connected; this proposition explains why the two questions face a hardness result in the same regime, i.e., when $X$ is nonatomic. 

\begin{proof}[Proof of Proposition~\ref{prop:from_regression_to_test_conditional_prediction}]
    Fix any distribution $P$ on $\cX\times\cY$, and any subset $\cX_0\subseteq\cX$ with $\P_P(X\in\cX_0)\geq \delta$. 
    First, by definition,
    \[Y_{n+1}\not\in\cC_{\textnormal{pred}}(X_{n+1}) \ \Longleftrightarrow \ Y_{n+1}\not\in\cC_{\textnormal{init}}(X_{n+1})\textnormal{ and } \cC_{\textnormal{regr}}(X_{n+1})\subseteq [0,\alpha -\epsilon/\delta].\]
    Therefore,
    writing $\cD_n=((X_1,Y_1),\dots,(X_n,Y_n))$, we calculate
\begin{align*}
    &\P\left(Y_{n+1}\not\in\cC_{\textnormal{pred}}(X_{n+1})\mid \cD_n; X_{n+1}\in\cX_0\right)\\
    &=\P\left(Y_{n+1}\not\in\cC_{\textnormal{init}}(X_{n+1}),\, \cC_{\textnormal{regr}}(X_{n+1})\subseteq [0,\alpha -\epsilon/\delta]\mid \cD_n; X_{n+1}\in\cX_0\right)\\
    &=\E\Big[\P\big(Y_{n+1}\not\in\cC_{\textnormal{init}}(X_{n+1})\,\big|\, \cD_n,X_{n+1}\big) \cdot\ind{\cC_{\textnormal{regr}}(X_{n+1})\subseteq [0,\alpha -\epsilon/\delta]}\,\Big|\, \cD_n; X_{n+1}\in\cX_0\Big]\\
    &=\E\Big[\mu_{\tilde{P}}(X_{n+1}) \cdot\ind{\cC_{\textnormal{regr}}(X_{n+1})\subseteq [0,\alpha-\epsilon/\delta]}\,\Big|\, \cD_n; X_{n+1}\in\cX_0\Big]\\
    &\leq (\alpha -\epsilon/\delta)  + \P\Big(\mu_{\tilde{P}}(X_{n+1}) \not\in \cC_{\textnormal{regr}}(X_{n+1})\,\Big|\, \cD_n; X_{n+1}\in\cX_0\Big)\\
    &= (\alpha -\epsilon/\delta) + \frac{\P\Big(\mu_{\tilde{P}}(X_{n+1}) \not\in \cC_{\textnormal{regr}}(X_{n+1}), X_{n+1}\in\cX_0\,\Big|\, \cD_n\Big)}{\P(X_{n+1}\in\cX_0)}\\
    &\leq (\alpha -\epsilon/\delta) + \delta^{-1}  \cdot \P\Big(\mu_{\tilde{P}}(X_{n+1}) \not\in \cC_{\textnormal{regr}}(X_{n+1})\,\Big|\, \cD_n\Big).
\end{align*}
And, marginalizing over $\cD_n$, we have
\[\E\left[\P\Big(\mu_{\tilde{P}}(X_{n+1}) \not\in \cC_{\textnormal{regr}}(X_{n+1})\,\Big|\, \cD_n\Big)\right] =\P\Big(\mu_{\tilde{P}}(X_{n+1}) \not\in \cC_{\textnormal{regr}}(X_{n+1})\Big)\leq \epsilon,\]
by our assumption on the validity of $\cC_{\textnormal{regr}}$~\eqref{eqn:DF_confidence_regression_eps}.
Therefore,
\[\P\left(Y_{n+1}\not\in\cC_{\textnormal{pred}}(X_{n+1})\mid X_{n+1}\in\cX_0\right)\leq (\alpha -\epsilon/\delta) + \delta^{-1}\cdot\epsilon = \alpha.\]
\end{proof}

\index{coverage!test-conditional|)}

\section{Connections to estimation}\label{sec:regression__C_or_muhat} 

Throughout this chapter, we have focused on the question of inference on $\mu_P$---and more specifically, a particular version of this question, where our aim is to construct a confidence interval $\cC(x)$ for $\mu_P(x)$. In some settings it may be more natural to phrase this question in a different way: can we construct an estimate $\hat\mu(x)$ for $\mu_P(x)$, and quantify our uncertainty around this estimate, in a way that offers some notion of distribution-free validity?

In this section, we will examine this alternative formulation of the question of distribution-free inference, based on an estimator $\hat\mu$ for $\mu_P$. To make this concrete, consider the following definition of distribution-free validity: for any distribution $P$,
\begin{equation}\label{eqn:DF_validity_regression_L1}\P\left(\|\hat\mu - \mu_P\|_{L_1(P)} \leq \hat\epsilon\right)\geq 1-\delta,\end{equation}
where the probability is computed with respect to the distribution of the training data points, $(X_1,Y_1),\dots,(X_n,Y_n)\iidsim P$, and where $\hat\epsilon>0$ is a function of the training data intended to express our uncertainty around the estimated regression function $\hat\mu$. (For a function $f:\cX\to\R$, its $L_1(P)$ norm is defined as $\|f\|_{L_1(P)} = \E_P[|f(X)|]$.) 

In fact, the two different versions of the problem of inference for regression that we have considered---namely, providing inference around an estimate $\hat\mu$ as defined here, or, directly constructing a confidence interval $\cC$ for $\mu_P$ as has been our focus throughout the chapter---are qualitatively the same. At a high level, any solution for one of these inference problems can be leveraged to provide a solution to the other. Specifically, any $\hat\mu,\hat\epsilon$ satisfying~\eqref{eqn:DF_validity_regression_L1} can be used to construct a confidence interval $\cC$ for $\mu_P$, by centering our interval around $\hat\mu$ and choosing its width as a function of $\hat\epsilon$. Conversely, given a confidence interval $\cC$ for $\mu_P$ satisfying~\eqref{eqn:DF_confidence_regression}, we can construct $\hat\mu,\hat\epsilon$ based on the center and the width of $\cC$. 

We now state two formal results that reveal that these two versions of the inference question have the same statistical difficulty. In particular, regardless of which definition of distribution-free validity we choose, meaningful inference on $\mu_P$ is straightforward when $X$ is discrete, but faces fundamental hardness results when $X$ is nonatomic.

First, we consider the discrete setting. 
\begin{theorem}[Estimating the regression function, discrete case]\label{thm:DF_regression_discrete_eps}
Let $P$ be a distribution on $\cX\times\cY$, where $\cX = \{x_1,\dots,x_K\}$ and $\cY \subseteq [a,b]$, and let $\delta\in(0,1)$.
    Let $(X_1,Y_1),\dots,(X_n,Y_n)\iidsim P$. 
As in Theorem~\ref{thm:DF_regression_discrete}, define $n_k =\sum_{i=1}^n\ind{X_i = x_k}$, and let
    \[\hat{\mu}(x_k) = \frac{1}{n_k}\sum_{i=1}^n Y_i \cdot\ind{X_i = x_k}\]
    for each $k$ with $n_k\geq 1$, and $\hat\mu(x_k) = \frac{a+b}{2}$ for any $k$ with $n_k=0$. Define
    \[\hat\epsilon = \frac{b-a}{\sqrt{2\delta}}\cdot \sqrt{\frac{K}{n}}.\]
    Then $\hat\mu,\hat\epsilon$ satisfy the distribution-free validity condition~\eqref{eqn:DF_validity_regression_L1}. 
\end{theorem}
In particular we note that $\hat\epsilon$ is vanishing as $n\to\infty$. 

On the other hand, in the continuous case where $X$ is nonatomic, we are again faced with a hardness result.

\begin{theorem}[Hardness result for estimating the regression function]\label{thm:regression_nonatomic_eps}
Suppose $\hat\mu$ and $\hat\epsilon$ satisfy the condition~\eqref{eqn:DF_validity_regression_L1} for estimation of the regression function $\mu_P$.
    Then, for any distribution $P$ on $\cX\times[a,b]$ for which the marginal $P_X$ is nonatomic, and for any sample size $n\geq 1$,
    \[\P\left( \, \hat\epsilon \, \geq \, \frac{\E_P[\Var_P(Y\mid X)]}{b-a} \,\right) \geq 1-\delta,\]
    where $\Var_P(Y\mid X)$ is the conditional variance of $Y\mid X$ under the joint distribution $P$.
\end{theorem} \index{hardness result!regression}
In other words, the upper bound $\hat\epsilon$ on the error of our estimate, cannot be vanishing as $n\to\infty$---it is lower-bounded by the conditional variance of $Y\mid X$ under the distribution $P$. This result is similar in flavor to the hardness result in Theorem~\ref{thm:regression_nonatomic_width}, which proves that any distribution-free confidence interval for $\mu_P$ cannot have vanishing width in the nonatomic setting.

\begin{proof}[Proof of Theorem~\ref{thm:DF_regression_discrete_eps}]
        Let $p_k = \P_P(X=x_k)$. Then we can calculate
    \begin{align*}
        &\|\hat\mu - \mu_P\|_{L_1(P)}\\
        &=\sum_{k=1}^K p_k \cdot |\hat\mu(x_k) - \mu_P(x_k)|\\
        &\leq\sqrt{\sum_{k=1}^K p_k  (\hat\mu(x_k) - \mu_P(x_k))^2}\\
        &=\sqrt{\sum_{k=1}^K p_k  \left\{(\hat\mu(x_k) - \mu_P(x_k))^2\cdot \ind{n_k\geq 1} + \left(\frac{a+b}{2} - \mu_P(x_k)\right)^2\cdot \ind{n_k=0}\right\}}\\
        &\leq \sqrt{\sum_{k=1}^K p_k  \left\{(\hat\mu(x_k) - \mu_P(x_k))^2 \cdot \ind{n_k\geq 1} + \left(\frac{b-a}{2}\right)^2\cdot \ind{n_k=0}\right\}}.
    \end{align*}
    For each $k\in[K]$, conditional on $n_k$ (and on the event $n_k\geq 1$), as in the proof of Theorem~\ref{thm:DF_regression_discrete} 
    the random variable $\hat\mu(x_k)$ is the sample mean of $n_k$ many draws from the distribution of $Y\mid X=x_k$, which is supported on $[a,b]$ and has mean $\mu_P(x_k)$. Therefore,
    \[\E\left[(\hat\mu(x_k) - \mu_P(x_k))^2 \mid n_k\right]\leq \frac{(b-a)^2}{4n_k} \leq \frac{(b-a)^2}{2(n_k+1)},\]
    where for the last step we use $n_k\geq 1$. We can therefore relax the calculations above to
    \begin{align*}
        \E\left[\|\hat\mu - \mu_P\|_{L_1(P)}^2\right]
        &\leq \E\left[\sum_{k=1}^K p_k \cdot \frac{(b-a)^2}{2(n_k+1)}\right] \\
        &\leq \E\left[\sum_{k=1}^K p_k \cdot \frac{(b-a)^2}{2\cdot np_k}\right]\\
        &=\frac{(b-a)^2}{2} \cdot \frac{K}{n},
    \end{align*}
    where the second inequality holds by~\eqref{eqn:binomial_reciprocal}. Therefore, by Markov's inequality,
    \[\P\left(\|\hat\mu - \mu_P\|_{L_1(P)} > \frac{b-a}{\sqrt{2\delta}}\cdot \sqrt{\frac{K}{n}}\right) = \P\left(\|\hat\mu - \mu_P\|_{L_1(P)}^2 > \frac{(b-a)^2}{2\delta}\cdot \frac{K}{n}\right) \leq\delta.\]
\end{proof}

\begin{proof}[Proof of Theorem~\ref{thm:regression_nonatomic_eps}]
    Our proof will rely on the sample--resample technique of Lemma~\ref{lem:sample-resample}, and is similar to the proof of Theorem~\ref{thm:regression_nonatomic}. 
    Without loss of generality, we can take $[a,b]=[0,1]$. Let $(X^{(1)},Y^{(1)}),\dots,(X^{(M)},Y^{(M)})\iidsim P$, and let $\widehat{P}_M$ be the empirical distribution. Now condition on $\widehat{P}_M$,
    let $\big((X_i,Y_i)\big)_{i\in[n]}\iidsim \widehat{P}_M$, and let $\hat\mu,\hat\epsilon$ be trained on this data. Then by~\eqref{eqn:DF_validity_regression_L1}, we must have
    \begin{multline*}1-\delta \leq \P\left(\|\hat\mu - \mu_{\widehat{P}_M}\|_{L_1(\widehat{P}_M)} \leq \hat\epsilon \ \middle| \ \widehat{P}_M \right)\\\leq \P\left(\|\hat\mu - \mu_{\widehat{P}_M}\|_{L_1(\widehat{P}_M)} \leq c_M \ \middle| \ \widehat{P}_M\right) + \P(\hat\epsilon > c_M \mid \widehat{P}_M),\end{multline*}
    where we define the constant
    \[c_M = \E_P[\Var_P(Y\mid X)] - \frac{n}{4M} - \frac{1}{\sqrt[4]{M}}.\]
    And by definition of $\hat{P}_M$, we can calculate
    \[\|\hat\mu - \mu_{\widehat{P}_M}\|_{L_1(\widehat{P}_M)} 
    = \frac{1}{M}\sum_{i=1}^M \left|\hat\mu(X^{(i)}) - \mu_{\hat{P}_M}(X^{(i)})\right| \\
    =\frac{1}{M}\sum_{i=1}^M \left|\hat\mu(X^{(i)}) - Y^{(i)}\right| ,\]
    where the last step holds since we have $\mu_{\widehat{P}_M}(X^{(i)})=Y^{(i)}$, for all $i$, on the event that $X^{(1)},\dots,X^{(M)}$ are all distinct (which holds almost surely since $P_X$ is nonatomic).
    Therefore,
    \[   \P_{\widehat{P}_M}(\hat\epsilon > c_M\mid \widehat{P}_M )\geq 1-\delta - \P_{\widehat{P}_M}\left(\frac{1}{M}\sum_{i=1}^M \left|\hat\mu(X^{(i)}) - Y^{(i)}\right| \leq c_M \ \middle| \ \widehat{P}_M\right) ,\]
    where we have now added subscripts $\widehat{P}_M$ to emphasize that $\hat\mu$ and $\hat\epsilon$ are computed on data sampled i.i.d.\ from $\widehat{P}_M$.
    
    Next, as in Lemma~\ref{lem:sample-resample}, let $Q$ denote the joint distribution on $((X_i,Y_i))_{i\in[n]}$ obtained by first sampling $((X^{(i)},Y^{(i)}))_{i\in[M]}\iidsim P$ to construct $\widehat{P}_M$, and then sampling $((X_i,Y_i))_{i\in[n]}\iidsim\widehat{P}_M$.
By Lemma~\ref{lem:sample-resample}, then, we have $\dtv(P^n,Q)\leq \frac{n(n-1)}{2M}$, and so
\begin{align*}\P_P(\hat\epsilon > c_M) 
&\geq \E\left[ \P_{\widehat{P}_M}(\hat\epsilon > c_M\mid \widehat{P}_M )\right] - \frac{n(n-1)}{2M} \\
&\geq 1-\delta- \frac{n(n-1)}{2M} - \E\left[\P_{\widehat{P}_M}\left(\frac{1}{M}\sum_{i=1}^M \left|\hat\mu(X^{(i)}) - Y^{(i)}\right| \leq c_M \ \middle| \ \widehat{P}_M\right)\right]\\
&= 1-\delta- \frac{n(n-1)}{2M} - \P\left(\frac{1}{M}\sum_{i=1}^M \left|\hat\mu(X^{(i)}) - Y^{(i)}\right| \leq c_M\right),\end{align*}
where the last probability is calculated with respect to the following distribution: first we sample $(X^{(1)},Y^{(1)}),\dots,(X^{(M)},Y^{(M)})\iidsim P$ (to form $\widehat{P}_M$), and then we train the fitted regression function $\hat\mu$ on an i.i.d.\ sample of size $n$ drawn from $\widehat{P}_M$.

Our remaining step is to bound this last probability. Let $I_1,\dots,I_n\iidsim \textnormal{Unif}([M])$ be the indices corresponding to the draws of the data points (that is, $(X_i,Y_i) = (X^{(I_i)},Y^{(I_i)})$ for each $i=1,\dots,n$). 
Note that, conditional on $I_1,\dots,I_n$ and on $\hat\mu$, the remaining data points $\big((X^{(i)},Y^{(i)})\big)_{i\in[M]\backslash\{I_1,\dots,I_n\}}$ are i.i.d.\ draws from $P$ (and, in particular, are independent from $\hat\mu$). This implies that
\[\P\left(\frac{1}{M}\sum_{i=1}^M \left|\hat\mu(X^{(i)}) - Y^{(i)}\right| \leq c_M \,\middle|\,\hat\mu \right)
\leq \sup_{f:\cX\to[0,1]} \P\left(\frac{1}{M}\sum_{i=1}^{M-n} \left|f(X^{(i)}) - Y^{(i)}\right| \leq c_M\right),\]
where now we are taking a supremum over all \emph{fixed} functions $f$, since there are at least $M-n$ data points $(X^{(i)},Y^{(i)})$ that are i.i.d.\ draws from $P$ even after conditioning on $\hat\mu$.
Next, for any fixed function $f:\cX\to[0,1]$, we have
\[\E\left[ |f(X^{(i)})-Y^{(i)}| \right] \geq \E_P[|Y-\textnormal{Med}_P(X)|] \geq \E_P[(Y-\textnormal{Med}_P(X))^2] \geq \E_P[\Var_P(Y\mid X)],\]
where the first step holds since  $\E_P[|Y-t| \mid X]$ is minimized at $t=\textnormal{Med}_P(X)$, the second step holds since $\cY\subseteq[0,1]$, and the last step holds since $\E_P[(Y-t)^2\mid X]$ has minimum value $\Var_P(Y\mid X)$ (attained at $t=\mu_P(X)$). 
Therefore,
\begin{multline*}\E\left[\frac{1}{M}\sum_{i=1}^{M-n} \left|f(X^{(i)}) - Y^{(i)}\right|\right] \geq \frac{M-n}{M}\cdot \E_P[\Var_P(Y\mid X)] \\\geq \E_P[\Var_P(Y\mid X)] - \frac{n}{4M} = c_M + \frac{1}{\sqrt[4]{M}}. \end{multline*}
And, $\Var\left(\frac{1}{M}\sum_{i=1}^{M-n} \left|f(X^{(i)}) - Y^{(i)}\right|\right) \leq \frac{1}{4M}$. Therefore, by Chebyshev's inequality,
\[\P\left(\frac{1}{M}\sum_{i=1}^{M-n} \left|f(X^{(i)}) - Y^{(i)}\right| \leq  c_M\right) \leq \frac{1}{4\sqrt{M}},\]
for any fixed function $f:\cX\to[0,1]$.
Returning to our calculations above we have shown that 
\[\P_P(\hat\epsilon > c_M) \geq 1-\delta - \frac{n(n-1)}{2M} - \frac{1}{4\sqrt{M}}.\]

Since $\lim_{M\to\infty} c_M = \E_P[\Var_P(Y\mid X)]$, taking $M\to\infty$ completes the proof.
\end{proof}

\index{regression function!inference|)}

\section*{Bibliographic notes}
\addcontentsline{toc}{section}{\protect\numberline{}\textnormal{\hspace{-0.8cm}Bibliographic notes}}

Theorem~\ref{thm:bahadur_savage}, which establishes a hardness result for inference on the marginal mean in the unbounded setting, is a classical result by \citet{bahadur1956nonexistence}.
Turning to the problem of inference on the conditional mean, the aim of distribution-free marginal coverage on the conditional mean, as defined in~\eqref{eqn:DF_confidence_regression}, is proposed in \cite{vovk2005algorithmic} (originally in the context of a binary response, under the terminology of constructing a `weakly valid probability estimator'). For achieving this type of distribution-free coverage in the discrete case, the proof of Theorem~\ref{thm:DF_regression_discrete} relies on the calculation~\eqref{eqn:binomial_reciprocal} for the Binomial distribution, which is due to \citet{chao1972negative}. The confidence interval constructed in Theorem~\ref{thm:DF_regression_discrete} 
offers a solution that is simple, but generally wider than necessary because it only relies on Hoeffding's inequality \citep{hoeffding1963}; a tighter confidence interval for this type of problem is developed by \citet{gupta2020distribution}, using an empirical Bernstein inequality to adapt to the variance of $Y\mid X$. (See \citet{audibert2009exploration,maurer2009empirical,waudby2024estimating} for background and additional results on empirical Bernstein inequalities.) Similarly, the constructions of Theorems~\ref{thm:DF_regression_binned}, ~\ref{thm:DF_regression_blurred}, and~\ref{thm:DF_regression_discrete_eps} can also be tightened by using more refined concentration inequalities.

In the case of a nonatomic feature $X$, Theorem~\ref{thm:regression_nonatomic}, which shows that any distribution-free confidence interval for $\mu_P$ must necessarily also be a valid prediction interval, appears in various forms in \citet{nouretdinov2001pattern,vovk2005algorithmic,barber2020distributionfree}. Theorem~\ref{thm:regression_nonatomic_median}, which gives an analogous result for any distribution-free confidence interval for the conditional median, is established by~\citet{medarametla2021distribution}; the same work also shows a converse, that a distribution-free prediction interval at level $1-\alpha$ must also offer distribution-free coverage, at level $1-2\alpha$, for the conditional median.Theorem~\ref{thm:regression_nonatomic_width}, which proves a lower bound on the width any distribution-free confidence interval for $\mu_P$, is based on results from \citet{barber2020distributionfree,lee2021distribution}; the latter work also refines this bound to handle the case where $X$ is not nonatomic, but has large support, building on assumption-free tools for testing hypotheses about discrete distributions \citep{chan2014optimal}. On the other hand, the hardness result presented in Theorem~\ref{thm:regression_nonatomic_eps}, formulated in terms of producing an estimate $\hat\epsilon$ of the error of the fitted regression model, is more closely related to hardness results for calibration; see the bibliographic notes in Chapter~\ref{chapter:calibration} for references.

In Section~\ref{sec:regression_relax}, where we discuss relaxations of the problem for the nonatomic case, the procedure for distribution-free inference on the blurred mean (Theorem~\ref{thm:DF_regression_blurred}) is based on the work of \citet{jang2023tight}, which provides an analogous procedure but for inference on the conditional quantile, rather than conditional mean, of $Y\mid X$. For background on rejection sampling (also called accept--reject sampling, or acceptance--rejection sampling), see \citet{owen2013monte}.

The problem of estimating the regression function $\mu_P$ has a long history in statistics, both within parametric models and in the nonparametric setting. Concretely, we can compare the distribution-free results established in this chapter to related results from the nonparametric statistics literature. Taking $\cX=\cY=[0,1]$ for simplicity, suppose we assume that $\mu_P$ is $s$-H{\"o}lder smooth~\citep[see, e.g.,][Definition 1.2]{tsybakov2008introduction} for some $s>0$ (when $s$ is an integer, this is equivalent to assuming $\mu_P$ has bounded $s$th derivative). Then it is possible to construct a confidence band for $\mu_P$ of width $\bigo(n^{-s/(2s+1)})$, up to log factors, but this procedure relies on knowing the value of $s$; if $s$ is unknown, then adaptivity is impossible, in the sense that any procedure guaranteeing coverage for any $s\geq s_{\min}$ will inevitably lead to width $\gtrsim n^{-s_{\min}/(2s_{\min}+1)}$, i.e., it cannot adapt to a higher smoothness level $s>s_{\min}$.
Returning to the distribution-free setting (with nonatomic $X$), even if we believe that the regression function must exhibit some (unknown) level of smoothness $s$, if we wish to avoid untestable assumptions then we would be forced to take $s_{\min}\to 0$, leading to constant-width confidence intervals (as derived in this chapter).
For background on these types of results from the nonparametric statistics literature, see \citet{genovese2008adaptive,cai2014adaptive} and \citet[Chapter 8]{gine2021mathematical}, and the references therein; these works also propose relaxations of the notion of valid coverage to enable avoiding hardness results. A closely related problem is that of nonparametric density estimation, with earlier results on the impossibility of adaptivity, such as \citet{low1997nonparametric}.

\section*{Exercises}
\addcontentsline{toc}{section}{\protect\numberline{}\textnormal{\hspace{-0.8cm}Exercises}}
\begin{enumerate}[font=\bfseries, label={\thechapter.\arabic*}, labelsep=1em, itemsep=1em]
\item Let $(X_1,Y_1),\dots,(X_n,Y_n),(X_{n+1},Y_{n+1})\iidsim P$ for some distribution $P$ on $\cX\times\cY$, where $\cY=\R$. Consider the interval $\cC = [\min_{i\in[n]} Y_i, \max_{i\in[n]} Y_i]$. \begin{enumerate}
    \item Prove that, for any distribution $P$, the interval $\cC$ satisfies
        \[\P_P(Y_{n+1}\in\cC) \geq 1-\frac{2}{n+1},\] 
        which is a marginal predictive coverage guarantee.
    \item Prove that, for any distribution $P$, the interval $\cC$ satisfies
    \[\P_P(\textnormal{Med}_{P_Y} \in\cC) \geq 1 - 2^{-(n-1)},\]
    where $\textnormal{Med}_{P_Y}$ is the median of the marginal distribution $P_Y$. This verifies that $\cC$ is a valid distribution-free confidence interval for the (marginal) median of $Y$.
    \item Prove that, for any distribution $P$, the interval $\cC$ satisfies
    \[\P_P(\textnormal{Med}_P(X_{n+1}) \in\cC) \geq 1-\frac{2(2-2^{-n})}{n+1},\]
    where $\textnormal{Med}_P(x)$ is the conditional median of $Y$ given $X=x$ (under the joint distribution $P$). This verifies that $\cC$ is a valid distribution-free confidence interval for the conditional median of $Y$. (Note that this does not contradict the hardness result of Theorem~\ref{thm:regression_nonatomic_median}, since $\cC$ is also a valid prediction interval.)
\end{enumerate}
\item Let $(X_1,Y_1),\dots,(X_n,Y_n)\iidsim P$, where $P$ is a distribution on $\cX\times\cY$  and $\cY\subseteq [0,1]$. In this exercise, we will perform distribution-free inference for the parameter 
    \[\mu^+_P = \E_P[Y \mid X \geq \textnormal{Med}_{P_X}],\]
    which is the expected value of $Y$ conditional on the event that $X$ is at or above its marginal median, $\textnormal{Med}_{P_X}$. 
    
    Note that since $P$ is unknown, the quantity $\textnormal{Med}_{P_X}$ is unknown as well. We will show that it is nonetheless possible to estimate $\mu^+_P$ with a distribution-free guarantee of accuracy. For simplicity, we will assume throughout this problem that $P_X$ is nonatomic.
    \begin{enumerate}
    \item First consider an oracle method, which does assume knowledge of the marginal  median $\textnormal{Med}_{P_X}$. Let $N_+=\sum_{i=1}^n\ind{X_i\geq \textnormal{Med}_{P_X}}$, and define
        \[\tilde\mu^+_P = \begin{cases} \frac{1}{N_+}\sum_{i=1}^n Y_i \cdot\ind{X_i\geq \textnormal{Med}_{P_X}}, & \textnormal{ 
        if $N_+>0$,}\\ \frac{1}{2}, & \textnormal{ otherwise}.\end{cases}\]
        Prove that
        \[\E\left[\big(\tilde\mu^+_P - \mu^+_P\big)^2\right] \leq \frac{1}{n}.\]
        \emph{Hint: some of the steps are similar to the proof of Theorem~\ref{thm:DF_regression_discrete_eps}.}

    \item Next, define
    \[\widehat\mu^+_P = \frac{1}{\lceil n/2\rceil}\sum_{i=1}^n Y_i \cdot\ind{X_i\geq X_{\lfloor n/2\rfloor+1}}\]
    as the sample mean of the $Y$ values corresponding to the highest $\lceil n/2\rceil$ $X$ values. Prove that
    \[\E\left[ (\widehat\mu^+_P - \tilde\mu^+_P)^2\right] \leq \frac{1}{n}.\]
    (Combined with the previous part, this implies that $\widehat\mu^+_P$ is an accurate estimator of $\mu^+_P$, with a distribution-free bound on its error, $\E[(\widehat\mu^+_P-\mu^+_P)^2]\leq 4/n$.) 
\end{enumerate}
\item Let $P$ be any distribution on $\cX\times\cY$, where $\cY\subseteq[0,1]$, and where $\cX = \{x_1,\dots,x_K\}$ is a finite set. In Theorem~\ref{thm:DF_regression_discrete}, we constructed a confidence interval for $\mu_P(x)$ that offers distribution-free marginal coverage. For this exercise, construct a confidence interval for the conditional median, $\textnormal{Med}_P(x)$, again with marginal coverage: we require that
    \[\P_P(\textnormal{Med}_P(X_{n+1})\in\cC(X_{n+1}))\geq 1-\alpha\]
    holds for data $(X_1,Y_1),\dots,(X_{n+1},Y_{n+1})\iidsim P$, for any $P$. Note that, since we are in the discrete setting, the resulting interval should not be as wide as a prediction interval (as long as $K\ll n$).
\end{enumerate}

 \chapter{Calibration}
\label{chapter:calibration}

We next consider \emph{calibration}, a property that measures the reliability of a model $f(X)$ for predicting $Y$. We will restrict our attention to the setting where the response is binary, $Y\in\{0,1\}$, and so $f(X)$ is an estimate of the conditional probability of a positive label, $\P(Y=1\mid X)$.
Generally speaking, calibration requires that the estimate satisfies $f(X)\approx \P(Y=1\mid f(X))$. For example, if we consider all data points for which $f(X)=20\%$ (i.e., the model estimates a 20\% probability of a positive label), then approximately $20\%$ of those cases should have $Y=1$, and so on.
In this chapter, we discuss distribution-free calibration guarantees, showing that they are possible in certain regimes but impossible in others.

\section{Calibration: definition and methods}
\label{sec:calibration-intro}

Consider a random variable $(X,Y)$ drawn from a distribution $P$ on $\cX \times \{0,1\}$. Suppose we have a function $f : \cX \to [0,1]$ that is an estimate of the probability of a positive label, $Y=1$, based on the features $X$. One basic notion of validity for $f$ is calibration, defined as follows.
\begin{definition}[Perfect calibration] \label{def:perfect_calibration}
The function $f$ is \emph{perfectly calibrated} if 
\begin{equation}
\label{eq:perfect_calibration}
    \P(Y = 1 \mid f(X)) = f(X),
\end{equation}
almost surely.
\end{definition} \index{perfect calibration}

This is an appealing basic property. For instance, one example of a perfectly calibrated model is the oracle model that outputs the true conditional probabilities: $f(x) = \P(Y = 1 \mid X = x)$. Note that many other models also satisfy perfect calibration, for example, the constant function $f(x) = \E[Y]$. The latter example illustrates the fact that while calibration is a desirable property, it does not imply that $f$ is an accurate estimate of the true regression function for $Y\mid X$.

Notice that this definition is with respect to the distribution $P$, but in practical settings, we would only have access to a finite sample of data drawn from $P$. As such, we wish to have procedures that use the available data to address the following two goals:
\begin{itemize}
    \item[(a)] Given a pretrained model $f$, we wish to check whether perfect calibration holds, and to quantify violations of calibration.   
    \item[(b)] We wish to construct a function $f$, or modify a given pretrained function $f$, to satisfy or approximately satisfy perfect calibration.
\end{itemize}
These two goals form the focus of this chapter.

\subsection{Post-hoc calibration}
\label{subsec:post-hoc-calibration}
\index{post-hoc calibration|(}

With respect to goal (b) above, we will focus on \emph{post-hoc calibration}, a widely-used technique in applied machine learning. We start with a pretrained model $f : \cX \to [0,1]$, and a sample of $n$ data points, $(X_1, Y_1),\dots,(X_n,Y_n)\iidsim P$, that were not used for model training. As with split conformal prediction, we call this the calibration set. Then, we seek to post-process the outputs of $f$ via a function $h : [0,1] \to [0,1]$ so that the composition $h \circ f$ (approximately) satisfies calibration. Throughout this section, we will discuss algorithms that take as input a function $f$ and calibration data, and output such an $h$.

To begin, we introduce three commonly used post-hoc calibration algorithms. Each algorithm uses the observed data to select the function $h$ from some class of functions $\cH$, where the choice of this class varies across the three methods.

\begin{figure}[t]
    \centering
    \includegraphics[width=\textwidth]{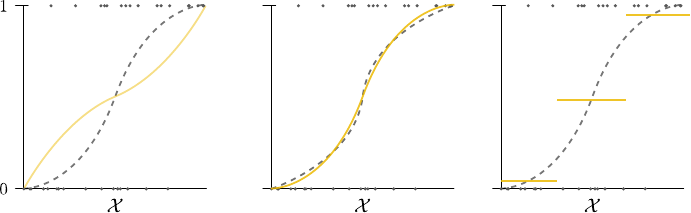}
    \caption{\textbf{Visualization of three regressions with varying degrees of calibration.} In each plot, the gray dashed line represents the true conditional probability $\P_P(Y=1\mid X=x)$, and the dots represent observed values of $Y$. The solid line shows a different fitted model $f(X)$ in each panel. The left plot shows a poorly calibrated $f(x)$. The middle plot shows a model $f(x)$ that is an accurate estimate of $\P_P(Y=1\mid X=x)$ and consequently has good calibration. Finally, the right plot shows a binned regression with high error (i.e., $f(x)$ is \emph{not} an accurate estimate of $\P_P(Y=1\mid X=x)$), but nonetheless good calibration: for each value $t$ that occurs in the output of $f(x)$, we can see that $\P_P(Y=1\mid f(X)=t)\approx t$.}
    \commentAlt{Three plots show the same S-shaped dashed curve and same $(X,Y)$ points. Each is overlaid with a solid curve: in the left panel this curve is very different from the dashed line, in the middle it is similar, and in the right it is piecewise constant.}
    \label{fig:calibration}
\end{figure}

\paragraph{Binning.} A natural strategy is to bin the values of $f(X_i)$ into $K$ bins, and then use $h$ to adjust the output value for each bin. Concretely, let $[0,1]=B_1\cup\dots\cup B_K$ be a partition. Let $n_k = \sum_{i=1}^n \ind{f(X_i) \in B_k}$ be the number of calibration points in bin $k$, and for any $z \in [0,1]$, let $k(z)$ be the index of the bin containing $z$.
Then, we define $h$ as
\begin{equation}
    \label{eq:binning_h}
    \hat h(z) = \frac{1}{n_{k(z)}} \sum_{i=1}^n Y_i\cdot  \ind{f(X_i) \in B_{k(z)}}
\end{equation}
whenever $n_{k(z)} \geq 1$, and $\hat h(z) = 1/2$ otherwise. That is, for points with $f(x) \in B_k$, the function $\hat h \circ f$ returns the frequency of $Y=1$ for points in that bin based on the calibration data. This is equivalent to minimizing the squared error of $h \circ f$ on the calibration data:
\begin{equation}
\label{eq:calibration_se_minimizer}
    \hat h = \argmin_{h \in \cH} \sum_{i=1}^n \left([h \circ f](X_i) - Y_i\right)^2,
\end{equation}
where we choose the function class as $\cH = \cH_{\textnormal{bin}}$, the set of functions that are constant within each bin $B_k$. This has the advantage that as the number of bins grows, the function class $\cH_{\textnormal{bin}}$ is larger and more flexible.
However, with a large number of bins, more data is required to fit the function precisely.

\paragraph{Isotonic regression.}
Next, we consider isotonic regression, an alternative approach that avoids the need to specify bins, and is more commonly used than binning. Here, we restrict the post-hoc calibration function to lie in $\cH=\cH_{\textnormal{iso}}$, the set of all  nondecreasing functions. Restricting $h$ to be nondecreasing is often natural, since the initial function $f$ is assumed to be a model where larger outputs correspond to higher confidence that $Y=1$. We then fit $\hat h$ by minimizing the squared error on the calibration data as in~\eqref{eq:calibration_se_minimizer}.
While this approach avoids the specification of bins, it can only produce nondecreasing $h$. As such, it is neither more nor less general than binning.

\paragraph{Temperature scaling.} We can impose even more structure by taking $\cH$ to be a parametric class. One such approach is to use logistic regression, by taking the class of functions $\cH = \cH_{\textnormal{logistic}} = \{h_{\beta_0,\beta_1} : \beta_0,\beta_1\in\R\}$, where
\begin{equation}
    h_{\beta_0,\beta_1}(z) = \frac{1}{1 + e^{-(\beta_0 + \beta_1 \textnormal{logit}(z))}},
\end{equation}
with $\smash{\textnormal{logit}(z) = \log(\frac{z}{1-z})}$. 
Notice that $\beta_1\approx 0$ encourages the output probabilities to be similar to each other, while larger (positive or negative) $\beta_1$ leads to more extreme probabilities that differ across the range of values of $f(X)$.
From a statistical perspective, this approach has only two parameters, so it requires less calibration data than the other two approaches. However, the function class $\cH_{\textnormal{logistic}}$ is less flexible, so the resulting fit may not be as well calibrated as other approaches when ample data is available.
\index{post-hoc calibration|(}

\subsection{Quantifying violations of perfect calibration}
\index{perfect calibration!violations}
While Definition~\ref{def:perfect_calibration} is the natural definition of perfect calibration, in practice we never expect this to hold exactly, and it is necessary to quantify the degree to which calibration is violated---goal (a) above.
\begin{figure}[t]
    \centering
    \includegraphics[width=0.7\linewidth]{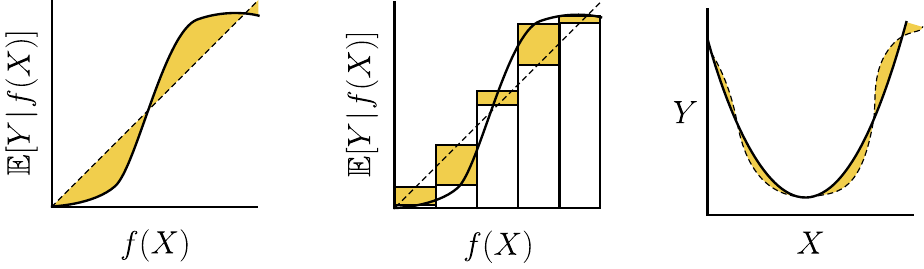}
    \caption{\textbf{An illustration of the ECE (left), binned ECE (middle), and dCE (right),} measured for the same function $f(x)$. The left panel plots the conditional mean as a function of the probabilistic prediction, $f(X)$, as a solid line. The calibration error $\textnormal{ECE}(f)$ is depicted by the shaded area between $f(X)$ and the dotted diagonal line (which represents perfect calibration). The middle panel depicts $\textnormal{binECE}(f)$ similarly. In the right panel, the solid line depicts $f(X)$, and the dotted line depicts the nearest perfectly calibrated function to $f(X)$, so that the shaded area represents $\textnormal{dCE}(f)$.}
    \commentAlt{The first two panels have horizontal axis labeled $f(X)$ and vertical axis $\E[Y\mid f(X)]$, with a dotted diagonal line and solid S-shaped curve. The right panel has horizontal axis labeled $X$ and vertical axis labeled $Y$. See long description.}
    \commentLongAlt{The first two panels have horizontal axis labeled $f(X)$ and vertical axis $\E[Y\mid f(X)]$, with a dotted diagonal line and solid S-shaped curve. The left panel has a shaded region between the two. The middle panel is split vertically into $5$ bars, with shaded rectangular regions indicating difference between the curve within each bar. The right panel has horizontal axis labeled $X$ and vertical axis labeled $Y$, and shows a solid U-shaped curve and a slightly different dotted curve, with the region between the two curves shaded.}
    \label{fig:calibration-relaxations}
\end{figure}

One widely-used measure of the extent to which a function $f$ is miscalibrated is given by the \emph{expected calibration error} (ECE): \index{expected calibration error (ECE)}
\begin{equation}
    \textnormal{ECE}(f) = \E\Big[\Big|\E[Y \mid f(X)] - f(X)\Big|\Big].
\end{equation}
Notice that this is defined with respect to a probability distribution $P$, although the notation suppresses this.
See the left panel of Figure~\ref{fig:calibration-relaxations} for an illustration.
A closely related measure is the \emph{binned ECE}:
\begin{equation}
\label{eq:bin_ece_hat}
    \textnormal{binECE}(f) = \sum_{k=1}^K\Big|\E[Y  \mid f(X) \in B_k] - \E[f(X) \mid f(X) \in B_k] \Big| \cdot \P(f(X) \in B_k),
\end{equation}
where $[0,1]=B_1\cup\dots\cup B_K$ is some partition of the unit interval. (In fact, $\textnormal{binECE}(f)$ has an interesting alternative interpretation: it is equivalent to $\textnormal{ECE}(\bar{f})$, where $\bar{f}:\cX\to[0,1]$ is the piecewise-constant function mapping $x\mapsto \E[f(X)\mid f(X)\in B_k]$, for each $x$ with $f(x)\in B_k$.)
This quantity has a natural plug-in estimator,
\begin{equation}\label{eqn:binECEhat}
    \widehat{\textnormal{binECE}}(f) = \sum_{k=1}^K\bigg|  \frac{1}{n_k} \cdot \sum_{\substack{i\in[n]\\ f(X_i)\in B_k}} \left(Y_i - f(X_i)\right)\bigg| \cdot \frac{n_k}{n},
\end{equation}
where $n_k = \sum_{i=1}^n\ind{f(X_i) \in B_k}$ is the number of  data points observed in bin $k$. This estimator is widely used in practice. However, we will see that binned ECE can be arbitrarily far from ECE, which means that $\textnormal{binECE}(f)$ (and its empirical estimate $\widehat{\textnormal{binECE}}(f)$) should not necessarily be viewed as a reliable proxy for $\textnormal{ECE}(f)$.
See the middle panel of Figure~\ref{fig:calibration-relaxations} for an illustration of the binned ECE.

An alternative relaxation of perfect calibration is the \emph{distance to calibration} (dCE), which is the minimum $L_1$ distance between $f$ and a perfectly calibrated function: \index{distance to calibration (dCE)}
\begin{equation}
\label{eq:dce_def}
    \textnormal{dCE}(f) = \inf_{\substack{g : \cX \to [0,1] \\ \E[Y\mid g(X)] = g(X)}}
 \E[|g(X) - f(X)|].
\end{equation}
See the right panel of Figure~\ref{fig:calibration-relaxations} for an illustration of the dCE.

A critical advantage of both binned ECE and dCE (compared to ECE) is that it is possible to give distribution-free confidence bounds for these measures of miscalibration, which we will state formally soon.

\section{Properties of ECE and binned ECE}
\index{expected calibration error (ECE)|(}

We now consider the ECE in more detail. In the distribution-free setting, we will see that ECE is straightforward to work with if $f(X)$ has a discrete distribution, but faces hardness results if $f(X)$ is continuously distributed.
This arises from a connection with the hardness of distribution-free regression that we will describe shortly.
These challenges make binned ECE an appealing alternative to ECE, but we will see that these two measures of miscalibration are fundamentally quite different.

\subsection{Discontinuity of ECE}
\label{subsec:ece_discont}

A basic difficulty with the ECE is that it can change by a large amount with only small changes in the function $f$. This is illustrated by the following example.
\begin{example}[Discontinuity of ECE]\label{example:discontinuity_ECE}
Consider the following distribution $P$ on $(X,Y)$: let $X\sim\textnormal{Unif}[0,1]$, and let
\[\P(Y=1\mid X) = \begin{cases} 0, & X\in [0,0.25] \cup [0.75,1],\\ 1, & X\in(0.25,0.75).\end{cases}\]
Then the constant function $f(x)= 1/2$ is perfectly calibrated. On the other hand, the function 
\[f_\epsilon(x) = \frac{1-\epsilon}{2} + \epsilon x,\]
which is arbitrarily close to $f$ when $\epsilon>0$ is small, is highly miscalibrated in terms of its ECE: since $f_\epsilon$ is injective, we have $\E[Y\mid f_\epsilon(X)] = \E[Y\mid X]$, and so
\[\textnormal{ECE}(f_\epsilon) = \E\left[ \left|\E[Y\mid f_\epsilon(X)] - f_\epsilon(X)\right|\right] = \E\left[\left| \ind{0.25 < X < 0.75} - \left(\frac{1-\epsilon}{2} + \epsilon X\right)\right|\right] = \frac{1}{2}.\]
\end{example}

This discontinuity is central to the challenges inherent to estimating ECE empirically. 

\subsection{Comparing ECE and binned ECE}
Next we consider binned ECE. As we will see shortly, this discretized measure of calibration error is easier to estimate than the ECE. First, to compare the two measures,
the following result tells us that binned ECE cannot be larger than ECE:
\begin{proposition}[Relating ECE and binned ECE]\label{prop:binECE_vs_ECE}
    For any distribution $P$ on $\cX\times\{0,1\}$, any function $f:\cX\to[0,1]$, and any partition $[0,1]=B_1\cup\dots\cup B_K$,
    \[\textnormal{binECE}(f) \leq \textnormal{ECE}(f).\]
\end{proposition}
In fact, however, this inequality can be extremely loose. For instance, returning to Example~\ref{example:discontinuity_ECE} above, if we choose $K=2$ bins to form the partition $[0,1] = [0,0.5]\cup (0.5,1]$, we can verify that $\textnormal{binECE}(f_\epsilon) = \frac{\epsilon}{4}$, whereas the ECE is $1/2$.
This illustrates an important point: estimating binned ECE, and observing a small value, does not necessarily imply that ECE is low.
\begin{proof}[Proof of Proposition~\ref{prop:binECE_vs_ECE}]
Let $k(f(x))$ denote the bin to which $f(x)$ belongs, i.e., for any $x$ with $f(x)\in B_k$, we have $k(f(x))=k$. The definition of binned ECE can then equivalently be expressed as
\begin{multline*}\textnormal{binECE}(f) = \sum_{k=1}^K\Big|\E[Y  -f(X) \mid f(X) \in B_k] \Big| \cdot \P(f(X) \in B_k) \\= \E\Big[ \Big| \E\big[Y - f(X) \mid k(f(X))\big]\Big|\Big].\end{multline*}
The result then follows by Jensen's inequality: 
\[\E\Big[ \Big| \E\big[Y - f(X) \mid k(f(X))\big]\Big|\Big] \leq \E\Big[ \Big| \E\big[Y - f(X) \mid f(X)\big]\Big|\Big] = \textnormal{ECE}(f).\]
\end{proof}

\subsection{Can ECE and binned ECE be estimated?}
We now turn to the topic of estimating ECE and binned ECE. More precisely, we are interested in providing upper bounds on these quantities---that is, we would like to use the available data to certify that ECE or binned ECE is low. 
We will see that this is possible in general with binned ECE but not ECE.

For binned ECE, it is straightforward to construct a meaningful upper bound, using the estimator given in~\eqref{eqn:binECEhat}.

\begin{theorem}[Distribution-free guarantee for estimating binned ECE]\label{thm:binECEhat}
    Fix any sample size $n$ and function $f:\cX\to[0,1]$, and let $[0,1] = B_1\cup\dots\cup B_K$ be a fixed partition. Define $\widehat{\textnormal{binECE}}(f)$ as in~\eqref{eqn:binECEhat}. Then  for any distribution $P$ on $\cX\times \{0,1\}$, for any $\delta\in(0,1)$,
    \[\P\left(\widehat{\textnormal{binECE}}(f) \geq \textnormal{binECE}(f) - \sqrt{\frac{2\log(1/\delta)}{n}}\right)\geq 1- \delta,\]
    and
    \[\P\left(\widehat{\textnormal{binECE}}(f) \leq \textnormal{binECE}(f) + \sqrt{\frac{K}{n}} + \sqrt{\frac{2\log(1/\delta)}{n}}\right)\geq 1- \delta.\]
\end{theorem}
In particular, the first statement shows that $\widehat{\textnormal{binECE}}(f) + \sqrt{\frac{2\log(1/\delta)}{n}}$ provides a distribution-free valid upper confidence bound on the binned ECE---and thus we can view this as a way to \emph{certify} that binned ECE is low, without any assumptions. On the other hand, the second statement tells us that this upper confidence bound is fairly tight (as long as $K\ll n$), so this is a meaningful way to check whether a function $f$ has low binned ECE.

Next, we show that the picture is substantially different for ECE. 
Any upper confidence bound for ECE that has distribution-free validity can only return uninformative answers for a nonatomic $f(X)$.
\begin{theorem}[Hardness result for estimating ECE]\label{thm:ece_hardness_of_testing}
Fix any sample size $n$ and function $f:\cX\to[0,1]$. Let $\widehat{\textnormal{ECE}}(f)\in[0,1]$ be a function of the observed data $(X_1,Y_1),\dots,(X_n,Y_n)\in\cX\times\{0,1\}$, satisfying the following distribution-free validity property:
\begin{equation}\label{eqn:ece_test_DF_valid}
    \P\left(\widehat{\textnormal{ECE}}(f) \geq \textnormal{ECE}(f)\right) \geq 1-\delta\textnormal{ for all $P$}.
\end{equation}
Then for any distribution $P$ on $\cX\times\{0,1\}$ for which $f(X)$ has a nonatomic distribution,
\[\P\left(\widehat{\textnormal{ECE}}(f) \geq \E[|Y-f(X)|] \right)\geq1-\delta.\]
\end{theorem} \index{hardness result!expected calibration error (ECE)}
In words, this means we cannot have a distribution-free upper confidence bound on ECE that shrinks to zero as the sample size grows, for \emph{any} function $f$ where $f(X)$ is nonatomic.
Notice this result is very similar to Theorem~\ref{thm:regression_nonatomic_eps}, which established the impossibility of providing a nontrivial distribution-free upper bound on the error in estimating $\E_P[Y\mid X]$ when $X$ is nonatomic. Here, we are instead interested in $\E_P[Y\mid f(X)]$, but if $f(X)$ is nonatomic then this is essentially the same question.

\begin{proof}[Proof of Theorem~\ref{thm:binECEhat}]
    Let $\mu_k = \E[Y \mid f(X) \in B_k]$, $\mu^f_k = \E[f(X)\mid f(X)\in B_k]$, $p_k = \P(f(X) \in B_k)$, and $n_k = \sum_{i=1}^n \ind{f(X_i)\in B_k}$ for $k=1,\dots,K$.
    Then 
    for each $k$, conditional on $n_k$ (and assuming $n_k\geq 1$), the quantity
    \[\sum_{\substack{i\in[n]\\f(X_i)\in B_k}} (Y_i - f(X_i))\]
    is a sum of $n_k$ many i.i.d.\ terms, each with mean $\mu_k - \mu^f_k$ and variance $\leq 1$. Consequently, a straightforward calculation verifies that
    \[n_k |\mu_k - \mu^f_k| \leq \E\left[\left|\sum_{\substack{i\in[n]\\f(X_i)\in B_k}} (Y_i - f(X_i))\right| \ \bigg| \ n_k\right] \leq n_k|\mu_k - \mu^f_k| + \sqrt{n_k}.\]
    The same upper and lower bounds hold trivially in the case $n_k=0$.
    Therefore, marginalizing over $n_k$,
    \begin{multline*}\E\left[ \widehat{\textnormal{binECE}}(f) \right] = \E\left[\frac{1}{n}\sum_{k=1}^K\Bigg|   \sum_{\substack{i\in[n]\\f(X_i)\in B_k}} \left(Y_i - f(X_i)\right)\Bigg|\right]\\ \geq  \E\left[\sum_{k=1}^K \frac{n_k}{n} |\mu_k - \mu^f_k| \right] = \sum_{k=1}^K p_k|\mu_k - \mu^f_k| =  \textnormal{binECE}(f) ,\end{multline*}
    since $\E[n_k] = np_k$, and similarly
    \begin{multline*}\E\left[ \widehat{\textnormal{binECE}}(f) \right]\leq \E\left[\sum_{k=1}^K \frac{n_k}{n} |\mu_k - \mu^f_k|  + \frac{\sqrt{n_k}}{n}\right] =  \textnormal{binECE}(f)  + \sum_{k=1}^K \frac{\E[\sqrt{n_k}]}{n} \\\leq \textnormal{binECE}(f) + \sum_{k=1}^K \frac{\sqrt{np_k}}{n}\leq \textnormal{binECE}(f) +\sqrt{\frac{K}{n}}. \end{multline*}
    
    Finally, we apply a bounded-differences argument: observe that resampling a single data point $(X_i,Y_i)$ can only change $\widehat{\textnormal{binECE}}(f)$ by at most $\pm \frac{2}{n}$.
    By McDiarmid's inequality, we therefore have
    \[\P\left( \widehat{\textnormal{binECE}}(f) \geq \E\left[\widehat{\textnormal{binECE}}(f) \right] - \sqrt{\frac{2\log(1/\delta)}{n}}\right)\geq 1-\delta\]
    and
    \[\P\left( \widehat{\textnormal{binECE}}(f) \leq \E\left[\widehat{\textnormal{binECE}}(f) \right] + \sqrt{\frac{2\log(1/\delta)}{n}}\right)\geq 1-\delta,\]
    which completes the proof.
\end{proof}

\begin{proof}[Proof of Theorem~\ref{thm:ece_hardness_of_testing}]
Our proof relies on the sample--resample technique of Lemma~\ref{lem:sample-resample}, and follows similar arguments to the proof of Theorem~\ref{thm:regression_nonatomic_eps}.

Fix any $M\geq 1$, let $(X^{(1)},Y^{(1)}),\dots,(X^{(M)},Y^{(M)})\iidsim P$, and let $\widehat{P}_M$ be the empirical distribution. Now condition on $\widehat{P}_M$.  On the event that $f(X^{(1)}),\dots,f(X^{(M)})$ are all distinct (which holds almost surely), we have 
\[E_{\widehat{P}_M}[Y \mid f(X) = f(X^{(i)})] = Y^{(i)}\]
for all $i\in[M]$, and consequently
\[\textnormal{ECE}_{\widehat{P}_M}(f)
= \frac{1}{M}\sum_{i=1}^M |Y^{(i)} - f(X^{(i)})|\]
where we use the subscript to denote that this is the ECE with respect to the distribution $\widehat{P}_M$.
Then by the distribution-free validity of our estimator of ECE~\eqref{eqn:ece_test_DF_valid}, we have
\[\P_{\widehat{P}_M}\left(\widehat{\textnormal{ECE}}(f) \geq \frac{1}{M}\sum_{i=1}^M |Y^{(i)} - f(X^{(i)})| \ \middle| \ \widehat{P}_M\right)\geq 1-\delta,\]
where $\widehat{\textnormal{ECE}}(f)$ is computed using a sample $(X_1,Y_1),\dots,(X_n,Y_n)\iidsim \widehat{P}_M$.
Defining
$c_M = \E_P[|Y-f(X)|] - \frac{1}{\sqrt[4]{M}}$,
we can then write
\[\P_{\widehat{P}_M}\big( \widehat{\textnormal{ECE}}(f) > c_M \mid \widehat{P}_M\big) \geq 1-\delta - \ind{\frac{1}{M}\sum_{i=1}^M |Y^{(i)} - f(X^{(i)})| \leq c_M}.\]

Next, by Lemma~\ref{lem:sample-resample}, we have $\dtv(P^n,Q)\leq \frac{n(n-1)}{2M}$, where $Q$ denotes the joint distribution on $((X_i,Y_i))_{i\in[n]}$ obtained by first sampling $((X^{(i)},Y^{(i)}))_{i\in[M]}\iidsim P$ to construct $\widehat{P}_M$, and then sampling $((X_i,Y_i))_{i\in[n]}\iidsim\widehat{P}_M$. Then, marginalizing over $\widehat{P}_M$ in the calculations above, we obtain
\[\P_P\big( \widehat{\textnormal{ECE}}(f) > c_M\big) \geq 1-\delta - \P\left(\frac{1}{M}\sum_{i=1}^M |Y^{(i)} - f(X^{(i)})| \leq c_M\right) - \frac{n(n-1)}{2M},\]
where now $\widehat{\textnormal{ECE}}(f)$ is computed using data $(X_1,Y_1),\dots,(X_n,Y_n)\iidsim P$.
Since the terms $|Y^{(i)}-f(X^{(i)})|\in[0,1]$ are i.i.d., the quantity $\frac{1}{M}\sum_{i=1}^M |Y^{(i)} - f(X^{(i)})|$ has mean $\E_P[|Y-f(X)|]$ and variance $\leq\frac{1}{4M}$, and therefore by Chebyshev's inequality,
\[\P\left(\frac{1}{M}\sum_{i=1}^M |Y^{(i)} - f(X^{(i)})| \leq c_M\right) \leq \frac{1}{4\sqrt{M}}.\]
Therefore,
\[\P\big( \widehat{\textnormal{ECE}}(f) > c_M\big) \geq 1-\delta - \frac{1}{4\sqrt{M}}- \frac{n(n-1)}{2M}.\]
Finally, since $M$ can be taken to be arbitrarily large, and $\lim_{M\to\infty} c_M = \E_P[|Y-f(X)|]$, this completes the proof.
\end{proof}

\subsection{Can ECE be controlled with post-hoc calibration?}
The result of Theorem~\ref{thm:ece_hardness_of_testing} tells us that, given a trained model $f$, we cannot accurately estimate its ECE with distribution-free validity---unless $f$ returns outputs in a discrete space (i.e., $f(X)$ is effectively performing binning). In particular this means that a continuous function $f$, which may return a continuously-distributed output $f(X)$, generally cannot be certified to have low ECE. But this does not exclude another possibility: instead of testing whether $\textnormal{ECE}(f)$ is low, can we instead perform post-hoc calibration so that the resulting output is now guaranteed to have low ECE?

We will now see that the same types of challenges arise for this alternative version of the question. If post-hoc calibration is performed via binning, as in~\eqref{eq:binning_h}, then we can ensure low ECE since we are now in a discrete case. On the other hand, an injective post-hoc calibration (such as temperature scaling as in Section~\ref{subsec:post-hoc-calibration}) cannot result in a guarantee of low ECE in the continuous setting.

We begin with a positive result showing distribution-free ECE is possible via the binning algorithm~\eqref{eq:binning_h}. For each bin $B_k$, we compute the empirical mean of $Y$ when $f(X)$ takes values in $B_k$, and we use this as the value of $\hat h \circ f$. The approach will result in approximate calibration, with the approximation quality improving with the number of points per bin (i.e., as the sample size grows).

\begin{theorem}[Distribution-free ECE control for binning]
\label{thm:ece_binning}
Fix a choice of bins $B_1,\dots,B_K$ that partition $[0,1]$. Let $\hat h$ be the binning post-hoc adjustment as in~\eqref{eq:binning_h}. Then, for any sample size $n$, any function $f:\cX\to[0,1]$, and any distribution $P$ on $\cX\times\{0,1\}$,
\begin{equation}
\E\left[\textnormal{ECE}(\hat h \circ f)\right] \le \sqrt{K / 2n},
\end{equation}
and moreover, for any $\delta\in(0,1)$,
\[\P\left(\textnormal{ECE}(\hat h \circ f)\leq \frac{1}{\sqrt{2\delta}}\cdot \sqrt{\frac{K}{n}}\right) \geq 1-\delta.\]
\end{theorem}

\begin{proof}[Proof of Theorem~\ref{thm:ece_binning}]
Let $\hat\mu_k$ denote the value of $\hat{h}(z)$ for $z\in B_k$ (i.e., by the construction of $\hat{h}$ as in~\eqref{eq:binning_h}, we have $\hat\mu_k = \frac{1}{n_k}\sum_{i=1}^n Y_i\cdot  \ind{f(X_i) \in B_k}$, if $n_k\geq 1$, and $\hat\mu_k=\frac{1}{2}$, if $n_k=0$.) By definition of ECE, since $\hat h\circ f$ is piecewise constant, we can verify that
\[\textnormal{ECE}(\hat h \circ f) \leq \sum_{k=1}^K \P_P(f(X)\in B_k) \cdot |\E[Y\mid f(X)\in B_k] - \hat\mu_k|.\]

The proof of the tail bound on $\textnormal{ECE}(\hat h \circ f)$ (the second claim) follows directly from Theorem~\ref{thm:DF_regression_discrete_eps}: specifically, we are applying Theorem~\ref{thm:DF_regression_discrete_eps} to the sample $(f(X_1),Y_1),\dots,(f(X_n),Y_n)$ (in place of $(X_1,Y_1),\dots,(X_n,Y_n)$ as in the statement of that theorem).
The bound in expectation follows from a similar calculation as in the proof of Theorem~\ref{thm:DF_regression_discrete_eps}, and we omit the details.
\end{proof}

By construction, the binning procedure coarsens the information coming from $f$, since $\hat{h} \circ f$ takes only finitely many values. Intuitively, this means that binning loses some of the information present in $f$.

To avoid this loss of information, we can instead ask whether we can use a post-hoc calibration procedure that returns an injective function $\hat{h}$.
However, we will now see that this goal is incompatible with distribution-free validity:
if $f(X)$ is nonatomic, then requiring a distribution-free guarantee on ECE will inevitably prohibit the post-hoc calibration function $\hat{h}$ from being injective. In particular, this implies that strategies such as temperature scaling are intrinsically unable to provide a distribution-free ECE guarantee.

\begin{theorem}[Hardness of achieving ECE via injective post-hoc calibration]\label{thm:ece_hardness_posthoc}
    Fix any sample size $n$ and function $f:\cX\to[0,1]$. Consider any post-hoc calibration procedure that satisfies a distribution-free guarantee on ECE,
    \[\P\left(\textnormal{ECE}\big(\hat{h}\circ f\big)\leq \epsilon\right) \geq 1- \delta \textnormal{ \ for all $P$}.\]
    Then, for any distribution $P$ on $\cX\times\{0,1\}$ for which $f(X)$ has a nonatomic distribution,
    \[\textnormal{If $\epsilon < \E_P[\Var_P(Y\mid X)]$ then }\P\left(\textnormal{$\hat{h}$ is an injective function}\right) \leq \delta.\]
\end{theorem} \index{hardness result!expected calibration error (ECE)}
Note that this bound on $\epsilon$ is not vanishing with $n$. In other words, regardless of the sample size $n$, we cannot ensure that ECE will be very low, with any injective post-hoc calibration procedure.

  As before, the key idea of this proof is to connect the present aim, distribution-free calibration, with the problem of distribution-free regression as studied in Chapter~\ref{chapter:regression}. In particular, a distribution-free guarantee on ECE would imply that we are able to perform distribution-free inference on the regression of $Y$ onto the post-hoc calibrated predictor $\hat{h}(f(X))$---but from the results of Chapter~\ref{chapter:regression}, we know that this is impossible in settings where $\hat{h}(f(X))$ is nonatomic. 
  \begin{proof}[Proof of Theorem~\ref{thm:ece_hardness_posthoc}]
    First, given a distribution $P$ on $\cX\times\{0,1\}$, define $\tilde{P}$ as the induced distribution of $(f(X),Y)$. Let $\mu_{\tilde{P}}$ be the regression function for this distribution, i.e.,
    \[\mu_{\tilde{P}}(f(X)) = \E_P[Y\mid f(X)]\]
    almost surely. Let $h:[0,1]\to[0,1]$ be any injective function. Then $h(f(X))$ contains the same information as $f(X)$, and so
    \[\mu_{\tilde{P}}(f(X)) = \E_P[Y\mid h(f(X))]\]
    almost surely. In particular,
      \begin{multline*}\textnormal{ECE}(h\circ f) = \E_P\left[\left| \E_P\left[Y\mid h(f(X))\right] - h(f(X))\right|\right] \\= \E_P\left[\left|\mu_{\tilde{P}}(f(X)) - h(f(X))\right|\right]= \|h - \mu_{\tilde{P}}\|_{L_1(\tilde{P})},\end{multline*}
      where the $L_1(\tilde{P})$ norm follows the notation of Section~\ref{sec:regression__C_or_muhat}.
On the other hand, if $h$ is not injective, then we nonetheless have $\|h - \mu_{\tilde{P}}\|_{L_1(\tilde{P})}\leq 1$, so combining both cases,
\[\|h - \mu_{\tilde{P}}\|_{L_1(\tilde{P})} \leq \textnormal{ECE}(h\circ f)\cdot\ind{\textnormal{$h$ is injective}} + \ind{\textnormal{$h$ is not injective}}.\]

Next, we return to the setting where $\hat{h}$ is trained on data. Define
\[\hat\epsilon = \epsilon\cdot \ind{\textnormal{$\hat{h}$ is injective}} + \ind{\textnormal{$\hat{h}$ is not injective}}.\]
Then by the calculations above,
\[\textnormal{ECE}(\hat{h}\circ f) \leq \epsilon \ \Longrightarrow \ \|\hat{h} - \mu_{\tilde{P}}\|_{L_1(\tilde{P})} \leq \hat\epsilon,\]
and therefore,
\[\P\big(\|\hat{h} - \mu_{\tilde{P}}\|_{L_1(\tilde{P})} \leq \hat\epsilon\big) \geq \P\big(\textnormal{ECE}(\hat{h}\circ f) \leq \epsilon\big) \geq 1-\delta,\]
by our assumption on the ECE calibration properties of $\hat{h}$.
In other words, $\hat{h}$ provides an estimate of the regression function $\mu_{\tilde{P}}$, and $\hat\epsilon$ is a data-dependent estimate of its error, as in~\eqref{eqn:DF_validity_regression_L1}. 
Since $f(X)$ has a nonatomic distribution by assumption, we can apply Theorem~\ref{thm:regression_nonatomic_eps}, which yields
\[\P\left(\hat\epsilon \geq \E_P[\Var_P(Y\mid X)]\right) \geq 1- \delta.\]
But on the event that $\hat{h}$ is injective, we have $\hat\epsilon = \epsilon < \E_P[\Var_P(Y\mid X)]$, so this completes the proof.
\end{proof}

\index{expected calibration error (ECE)|)}

\section{Properties of dCE}
\index{distance to calibration (dCE)|(}

We next turn to studying the properties of dCE. As we will see in this section, unlike ECE, it is possible to estimate dCE empirically with nontrivial distribution-free guarantees. Before considering the problem of estimation, we first compare dCE and ECE.

\subsection{Comparing dCE and ECE}
\index{expected calibration error (ECE)}
Our next result shows that dCE is always bounded by ECE---we can interpret this as saying that dCE is a strictly weaker notion of calibration than ECE.
\begin{proposition}[Relating ECE and dCE]
\label{prop:dce_less_ece}
For any distribution $P$ on $\cX\times\{0,1\}$ and any function $f: \cX \to [0,1]$, $\textnormal{dCE}(f) \le \textnormal{ECE}(f)$.
\end{proposition}
\begin{proof}[Proof of Proposition~\ref{prop:dce_less_ece}]
By definition of dCE, we have
\begin{equation}
    \textnormal{dCE}(f) = \inf_{\substack{g : \cX \to [0,1] \\ \E[Y\mid g(X)] = g(X)}}
 \E[|g(X) - f(X)|],
\end{equation}
where the infimum is taken over all perfectly calibrated functions $g$.
One such function $g$ is the true conditional mean given $f(X)$, i.e., $g^*(x) = \E[Y \mid f(X) = f(x)]$. With this choice, we have
\[\textnormal{dCE}(f) \leq \E[|g^*(X)-f(X)|] = \E\left[\left|\E[Y\mid f(X)] - f(X)\right|\right] = \textnormal{ECE}(f).\]
\end{proof}
The inequality in this bound can sometimes be extremely loose, as in the following example.
\begin{example}[Discontinuity of ECE, revisited]\label{example:discontinuity_ECE_revisited}
    Returning to the setting of Example~\ref{example:discontinuity_ECE}, recall that the constant function $f(x) = 1/2$ is a perfectly calibrated, while $f_\epsilon(x) = \frac{1-\epsilon}{2} + \epsilon x$ is a small perturbation with high ECE. On the other hand, for dCE, we have
    \[\textnormal{dCE}(f_\epsilon) \leq \E[|f_\epsilon(X) - f(X)|] = \frac{\epsilon}{4}. \]
\end{example}
In particular, this example highlights that, unlike ECE, the dCE measure of miscalibration \emph{is} continuous: by definition, for any functions $f$ and $g$, we have
\[\textnormal{dCE}(f) \leq \textnormal{dCE}(g) + \E[|f(X) - g(X)|].\]
That is, dCE is $1$-Lipschitz with respect to the $L_1(P)$ norm on functions. 

\subsection{Estimating dCE}
We now show that it is possible to provide meaningful distribution-free inference on dCE.
To define our estimator, we begin by fixing any $K\geq 1$. Let
\[\widehat{\textnormal{dCE}}(f) = \frac{1}{n}\sum_{k=1}^K \left|\sum_{\substack{i\in [n]\\ f(X_i)\in B_k}} \left(Y_i - \frac{k}{K}\right)\right|,\]
where the bins $B_1,\dots,B_K$ partition $[0,1]$ into equal-length intervals, 
\[[0,1] = \left[0,\frac{1}{K}\right]\cup \left(\frac{1}{K},\frac{2}{K}\right] \cup \dots \cup \left(\frac{K-1}{K},1\right] =: B_1 \cup \dots \cup B_K.\]

\begin{theorem}[Distribution-free guarantee for bounding dCE]\label{thm:dCEhat}
    Fix any sample size $n$ and function $f:\cX\to[0,1]$. Then for any distribution $P$ on $\cX\times\{0,1\}$, for any $\delta\in(0,1)$,
    \[\P\left(\widehat{\textnormal{dCE}}(f) + \frac{1}{K} +  \sqrt{\frac{2\log(1/\delta)}{n}} \geq \textnormal{dCE}(f) \right) \geq 1-\delta.\]
\end{theorem}
This means we have constructed an upper confidence bound on the dCE of $f$ with distribution-free validity. That is, if we observe a small estimate $\widehat{\textnormal{dCE}}(f)$ (and if $K$ and $n$ are both reasonably large), then we can be confident that $\textnormal{dCE}(f)$ is low.

Moreover, unlike for ECE, this estimator can return values close to zero even when $f$ is continuous, certifying that the dCE is below a desired level. For example, if $f$ is perfectly calibrated (as in Definition~\ref{def:perfect_calibration}), a straightforward calculation shows that $\E[\widehat{\textnormal{dCE}}(f)] \leq \frac{1}{K}+\sqrt{\frac{K}{n}}$, which can be made small by choosing $K$ and $n$ appropriately. This is in contrast to ECE calibration, since Theorem~\ref{thm:ece_hardness_of_testing} shows that distribution-free certification of low ECE error is not possible when $f(x)$ is nonatomic.

\begin{proof}[Proof of Theorem~\ref{thm:dCEhat}]
    First, define $\tilde{g}(x) = \sum_k \frac{k}{K}\cdot \ind{f(x)\in B_k}$. 
    Next, let $\mu_k = \E[Y \mid f(X)\in B_k]$ be the true conditional probability for the $k$th bin, and define $g(x) = \sum_k \mu_k \ind{f(x)\in B_k}$, which is a perfectly calibrated function, by construction. Following an identical argument as in the proof of Theorem~\ref{thm:binECEhat}, we can verify that
    \[\E\left[ \widehat{\textnormal{dCE}}(f)\right] \geq \E[|g(X) - \tilde{g}(X)|]. \]
    Moreover, by construction, we have
    \[|\tilde{g}(x) - f(x)| \leq \max_k \sup_{t\in B_k} \left| t -\frac{k}{K}\right|\leq \frac{1}{K},\]
    for all $x$, and therefore $\E[|\tilde{g}(X)-f(X)|]\leq \frac{1}{K}$. 
    Since $g$ is perfectly calibrated, we therefore have
    \[\textnormal{dCE}(f) \leq \E[|g(X)-f(X)|] \leq \E[|g(X)-\tilde{g}(X)|] + \E[|\tilde{g}(X)-f(X)|] \leq \E[\widehat{\textnormal{dCE}}(f)] + \frac{1}{K}.\]
    To complete the proof, we apply a bounded-differences argument, as in the proof of Theorem~\ref{thm:binECEhat}: since resampling a single data point $(X_i,Y_i)$ can only change $\widehat{\textnormal{dECE}}(f)$ by at most $\pm \frac{2}{n}$, we therefore have
    \[\P\left( \widehat{\textnormal{dECE}}(f) \geq \E\left[\widehat{\textnormal{dECE}}(f) \right] - \sqrt{\frac{2\log(1/\delta)}{n}}\right)\geq 1-\delta.\]
\end{proof}
\index{distance to calibration (dCE)|)}

\section{Venn--Abers Predictors}
\index{Venn--Abers predictors|(}

This section introduces the Venn--Abers predictor, a post-hoc calibration algorithm that leverages exchangeability in order to achieve its theoretical guarantee.
Unlike previous methods discussed in this chapter, the Venn--Abers predictor outputs an interval-valued prediction of $\P(Y=1 \mid f(X))$. This interval contains a prediction that is calibrated, in a sense that will be made precise below.

The Venn--Abers predictor adapts post-hoc isotonic regression, as described in Section~\ref{subsec:post-hoc-calibration}, by running \emph{two} isotonic regressions---one on $\cD_{n+1}^0$ and one on $\cD_{n+1}^1$, where we recall the notation $\cD_{n+1}^y = ((X_1,Y_1),\dots,(X_n, Y_n), (X_{n+1}, y))$ for any hypothesized test point value $y$. The method then outputs an interval whose endpoints are given by the predictions at $x=X_{n+1}$ from these two regressions. The Venn--Abers prediction algorithm is summarized below.

\begin{algbox}[Venn--Abers Prediction]
    \label{alg:Venn--Abers}
    \begin{enumerate}
        \item Input calibration data with binary labels $((X_1, Y_1), ..., (X_n, Y_n))$, a test point $X_{n+1}$, and a pre-trained model $f : \cX \to [0,1]$.
        \item For each $y\in\{0,1\}$, perform isotonic regression on the augmented dataset $\cD_{n+1}^y$,
        \[\hat h^y = \argmin_{h \in \cH_{\textnormal{iso}}} \sum_{i=1}^n \left([h \circ f](X_i) - Y_i\right)^2 + \left([h \circ f](X_{n+1}) - y\right)^2,\]
        and return the fitted values,
        \[\hat{p}^y = \big([\hat{h}^y\circ f](X_1),\dots,[\hat{h}^y\circ f](X_{n+1})\big). \]
        \item Output the interval $[\hat{p}^0_{n+1}, \hat{p}^1_{n+1}]$.
    \end{enumerate}
\end{algbox}
Note that for each $y \in \{0,1\}$, while the fitted isotonic regression function $\hat{h}^y$ is not uniquely defined, the vector of fitted values $\hat{p}^y$ returned by this function \emph{is} unique.

The Venn--Abers predictor enjoys the following distribution-free guarantee:
\begin{theorem}[Calibration of the Venn--Abers predictor]
    \label{thm:venn--abers}
    Let $(X_1,Y_1), \ldots, (X_{n+1},Y_{n+1})\in\cX\times\{0,1\}$ be exchangeable, and let $f : \cX \to [0,1]$ be a pre-trained model. Let $\hat{p}^0$ and $\hat{p}^1$ be defined as in Algorithm~\ref{alg:Venn--Abers}.
    Then,
    \begin{equation}
        \P\left( Y_{n+1}=1 \ \middle| \ \hat{p}^{Y_{n+1}}_{n+1}\right) = \hat{p}^{Y_{n+1}}_{n+1}
    \end{equation}
    almost surely.
\end{theorem}
Informally, Theorem~\ref{thm:venn--abers} is often interpreted as a guarantee on the interval-valued Venn--Abers prediction, $[\hat{p}^0_{n+1}, \hat{p}^1_{n+1}]$, in the following sense:
\begin{center}
    There exists a random variable $W\in\left[\hat{p}^0_{n+1}, \hat{p}^1_{n+1}\right]$ that is perfectly calibrated, i.e., $\E[Y_{n+1}\mid W] = W$.
\end{center}
Thus, we can interpret the output of the Venn--Abers predictor as an interval-valued probabilistic prediction. If the interval is wide, it signifies high uncertainty in our probabilistic prediction.

The procedure's validity hinges on the fact that $Y_{n+1}$ is either equal to $0$ or $1$, and so the true dataset $\cD_{n+1} = \big((X_i,Y_i)\big)_{i\in[n+1]}$ is equal to either $\cD_{n+1}^0$ or $\cD_{n+1}^1$.
When we guess $Y_{n+1}\in\{0,1\}$ correctly, the resulting isotonic regression trained on $\cD_{n+1}$ is perfectly calibrated, and therefore one of the endpoints of the resulting interval is perfectly calibrated.

\begin{proof}[Proof of Theorem~\ref{thm:venn--abers}]
Let $\hat{h} = \hat{h}^{Y_{n+1}}$. By definition of the isotonic regression problem~\eqref{eq:calibration_se_minimizer} (which depends symmetrically on the input data), the data points
\[\big(([\hat{h}\circ f](X_1),Y_1), \dots, ([\hat{h}\circ f](X_{n+1}),Y_{n+1})\big)\]
are exchangeable.

We will need to use a deterministic property of isotonic regression, which we state without proof: the fitted values of isotonic regression are piecewise constant, and the value within each bin is equal to the sample mean within the bin, i.e.,
\begin{equation}\label{eqn:Venn--Abers_iso_piecewise}
    \textnormal{For some partition $[n+1]=I_1\cup \dots \cup I_M$, \ $[\hat{h}\circ f](X_i) = \bar{Y}_{I_m}$ for all $m\in[M],i\in I_m$,}
\end{equation}
where $\bar{Y}_{I_m} = \frac{1}{|I_m|}\sum_{i\in I_m} Y_i$ denotes the mean of the $Y$ values indexed by $i\in I_m$.

Next, fix any function $g:[0,1]\to[0,1]$. We calculate
\begin{align*}
    \E\left[\left(Y_{n+1} -\hat{p}^{Y_{n+1}}_{n+1} \right) \cdot g(\hat{p}^{Y_{n+1}}_{n+1})\right]
    &=\E\left[\left(Y_{n+1} -[\hat{h}\circ f](X_{n+1})\right) \cdot g\big([\hat{h}\circ f](X_{n+1})\big)\right]\\
    &=\frac{1}{n+1}\E\left[\sum_{i=1}^{n+1} \left(Y_i 
    - [\hat{h}\circ f](X_i)\right)\cdot g\big([\hat{h}\circ f](X_i)\big)\right]\\
    &=\frac{1}{n+1}\E\left[\sum_{m=1}^M \sum_{i\in I_m}\left(Y_i 
    - [\hat{h}\circ f](X_i)\right)  \cdot g\big([\hat{h}\circ f](X_i)\big)\right]\\
    &=\frac{1}{n+1}\E\left[\sum_{m=1}^M \sum_{i\in I_m} \left(Y_i 
    - \bar{Y}_{I_m}\right) \cdot g(\bar{Y}_{I_m})\right]\\
    &=0,
\end{align*}
where the second step holds by exchangeability, the next-to-last step holds by~\eqref{eqn:Venn--Abers_iso_piecewise}, and the last step holds since $\sum_{i\in I_m} (Y_i-\bar{Y}_{I_m}) = 0$ by definition of $\bar{Y}_{I_m}$.
Since this holds for any function $g$, by definition of conditional expectation we have proved that $\E[Y_{n+1}\mid \hat{p}^{Y_{n+1}}_{n+1} ] = \hat{p}^{Y_{n+1}}_{n+1} $ almost surely, as desired.
\end{proof}
\index{Venn--Abers predictors|)}

\section*{Bibliographic notes}
\addcontentsline{toc}{section}{\protect\numberline{}\textnormal{\hspace{-0.8cm}Bibliographic notes}}

Calibration is a classical validity property for probability estimates~\citep[e.g.,][]{murphy1977reliability, dawid1982well, degroot1983comparison}. Post-hoc calibration for machine-learning models dates back to the work of \citet{platt1999probabilistic}, which introduced a form of temperature scaling. Binning for post-hoc calibration is another classical approach \citep[e.g.,][]{zadrozny2001obtaining}. Isotonic regression for post-hoc calibration is introduced in~\cite{zadrozny2002transforming}, leveraging the paired adjacent violators algorithm for isotonic regression~\citep{ayer1955empirical, barlow1972theory}. For more recent empirical studies of calibration in machine learning, see~\cite{guo2017calibration} and~\cite{minderer2021revisiting}. 

Plotting empirical frequencies versus probabilities, known as reliability diagrams, goes back to at least~\cite{murphy1977reliability}. The ECE measure of miscalibration can be viewed as a numeric summary of the reliability diagram; this measure is widely adopted in the statistics and machine learning communities, and dates back to the work of \citet{naeini2014binary,naeini2015obtaining}, which proposes ECE in a discrete setting---essentially, the binned ECE. Many works have studied the tradeoff between the informativeness of this miscalibration metric, and the bin size, including the work of \citet{arrieta2022metrics}, which also examines implications for the hardness of testing (non-binned) ECE. Many results on the accuracy of empirical estimates of the binned ECE, and of the ECE for a binned (i.e., piecewise-constant) function, can be found in the literature \citep[e.g.,][]{kumar2019verified,gupta2020distribution,gupta2021distribution}; our calculations in Theorems~\ref{thm:binECEhat} and~\ref{thm:ece_binning} are similar in flavor. The fact that binned ECE can drastically underestimate ECE (see Proposition~\ref{prop:binECE_vs_ECE}, and the discussion following) is established in \cite{kumar2019verified}.
\cite{gupta2020distribution,rossellini2025can} connect the hardness of guaranteeing ECE via post-hoc calibration to the hardness of distribution-free inference for regression, while \cite{lee2023t} establishes the hardness of testing the hypothesis of calibration; the hardness results presented in Theorems~\ref{thm:ece_hardness_of_testing} and~\ref{thm:ece_hardness_posthoc} build on these works. 

The distance to calibration (dCE) is introduced by~\citet{blasiok2023unifying}, which also discusses the discontinuity of ECE (as in Example~\ref{example:discontinuity_ECE}) and proves the result of Proposition~\ref{prop:dce_less_ece} (relating dCE to ECE). The same paper also establishes connections between dCE and binned ECE, as well as the fact that dCE can be estimated in a distribution-free setting (as in Theorem~\ref{thm:dCEhat}). Another important view on calibration, known as weighted calibration error~\citep{gopalan2022low}, considers comparing the prediction error to a class of test functions
\begin{equation}
    \sup_{g \in \cG}\E\left[g(f(X))\left(\E[Y \mid f(X)] - f(X)\right)\right],
\end{equation}
for some class of functions $\cG$.
ECE and dCE can both be viewed as special cases of this; see also \citet{okoroafor2025near,rossellini2025can} for an exploration of additional notions of calibration that lie in between ECE and dCE. Moreover, this view is closely related to the important topic of \emph{multicalibration}~\citep{hebert2018multicalibration}, a stronger notion of calibration that requires validity across subgroups of the data; see~\cite{roth2022uncertain} for an introduction.

Venn--Abers predictors, and their calibration guarantee (Theorem~\ref{thm:venn--abers}), were introduced in~\citet{vovk2014vennabers}, building on the earlier methodology of Venn predictors~\citep{vovk2003self}. The inductive version is introduced in~\citet{lambrou2015inductive} and further elaborated in~\citet{nouretdinov2018inductive}. See~\citet{vovk2005algorithmic} for further exposition, and see \citet{johansson2018venn,johansson2019efficient,van2024self,van2025generalized} for more recent work on Venn--Abers predictors and Venn predictors. See also conformal predictive distributions~\citep{vovk2017nonparametric, vovk2018cross} for a different approach to calibration-type guarantees leveraging exchangeability.

This chapter only considers calibration with respect to a fixed (though unknown) data distribution. A different framework, appearing earlier in the literature, is that of adversarial calibration in an online setting \citep{foster1998asymptotic}: at each time $t$ the analyst trains a `forecaster' $f_t$, and then the adversary reveals a new data point $(X_t,Y_t)$. The goal in this online setting is to provide a \emph{sequence} of $f_t$'s satisfying a notion of calibration. (This type of adversarial online framework appears also in the more recent online conformal prediction methods of the type discussed in Section~\ref{sec:ACI_and_quantile_tracking}).

Finally, we add some technical notes for several proofs appearing in this chapter. Several proofs use McDiarmid's inequality to ensure concentration of estimators (namely, the estimators of dCE and of binned ECE); this concentration inequality is due to \citet{mcdiarmid1989method}. In the proof of Theorem~\ref{thm:venn--abers} we rely on properties of isotonic regression, as in~\eqref{eqn:Venn--Abers_iso_piecewise}; for background on isotonic regression, see \citet{barlow1972theory}.

\section*{Exercises}
\addcontentsline{toc}{section}{\protect\numberline{}\textnormal{\hspace{-0.8cm}Exercises}}
\begin{enumerate}[font=\bfseries, label={\thechapter.\arabic*}, labelsep=1em, itemsep=1em]
\item  Prove that calibration is strictly weaker than consistency: specifically, writing $\mu_P$ to denote the regression function, $\mu_P(x) = \E_P[Y\mid X=x]$, prove that
\[\textnormal{ECE}(f) \leq \|f-\mu_P\|_{L_1}\]
for any function $f:\cX\to[0,1]$, and that this inequality may be strict. (Here $\|f-\mu_P\|_{L_1}=\E_P[|f(X)-\mu_P(X)|]$, as in Section~\ref{sec:regression__C_or_muhat}.)
\item Prove the expected value bound in Theorem~\ref{thm:ece_binning}.
\item This exercise will compare the dCE and binned ECE measures of calibration. Our aim is to show that these two measures are not strictly comparable; either one may be larger or smaller than the other.
\begin{enumerate}
    \item Construct an example of a distribution $P$ on $\cX\times\{0,1\}$, a function $f$, and a finite partition $[0,1]=B_1\cup \dots \cup B_K$ (where each $B_k$ is an interval), for which 
        $\textnormal{dCE}(f)<\textnormal{binECE}(f)$.
    \item Construct an example of a distribution $P$ on $\cX\times\{0,1\}$, a function $f$, and a finite partition $[0,1]=B_1\cup\dots\cup B_K$ (where each $B_k$ is an interval), for which 
                $\textnormal{dCE}(f)>\textnormal{binECE}(f)$.

\end{enumerate}
\end{enumerate}

 \chapter{Conditional Independence Testing}
\label{chapter:conditional-independence-testing}
\index{conditional independence testing}

In this chapter, we turn to a different type of inference problem: the question of \emph{conditional independence testing}, where we would like to determine whether a feature $X$ and response $Y$ are independent after conditioning on an additional random variable $W$ (the `confounder', which is often high-dimensional in many practical applications).
Formally, given data sampled from some joint distribution $P$ on $(X,Y,W)$, we would like to test the hypothesis 
\begin{equation}\label{eqn:H0_conditional_indep}H_0 : \ X\indep Y \mid W,\end{equation}
or equivalently,
\[H_0: \ P \in \cP_{X\indep Y\mid W},\]
where $\cP_{X\indep Y\mid W}$ is the set of all joint distributions on $(X,Y,W)\in\cX\times\cY\times\cW$ for which it holds that $X\indep Y\mid W$.

Note that if we instead test $X\indep Y$, we are asking about marginal, rather than conditional, independence---i.e., testing whether the feature $X$ is associated with the response $Y$, but without accounting for the confounding effects of $W$. The questions of testing independence or conditional independence are both intimately related to the same sorts of permutation-type arguments that arose in Chapter~\ref{chapter:exchangeability}.
As we will see shortly, in the distribution-free setting, testing marginal independence is typically straightforward, but testing conditional independence faces fundamental hardness results unless the confounder $W$ is discrete.

\section{Testing marginal independence}
\label{sec:marginal-independence-testing}
To begin, we tackle the easier problem of testing \emph{marginal} independence.
We are given a random sample $(X_1,Y_1),\dots,(X_n,Y_n)$ drawn i.i.d.\ from some distribution $P\in\cP_{(X,Y)}$, where $\cP_{(X,Y)}$ is the space of all distributions on $(X,Y)\in\cX\times\cY$. The goal is to test the null hypothesis of independence:
\begin{equation}
    \label{eq:null-marginal-independence}
    H_0 : \; X\indep Y,
\end{equation}
or equivalently,
\[H_0: \, P \in \cP_{X \indep Y}, \]
where $\cP_{X \indep Y}\subseteq \cP_{(X,Y)}$ is the set of all joint distributions $P$ on $(X,Y)$ for which it holds that $X\indep Y$.

We will consider any test statistic
\[T:(\cX\times\cY)^n\to\R\]
that measures evidence against the null. For example, 
for the case of real-valued data ($\cX = \cY = \R$), we might choose
\[T\big((X_1,Y_1),\dots,(X_n,Y_n)\big) = \left|\textnormal{Corr}\big((X_1,\dots,X_n),(Y_1,\dots,Y_n)\big)\right|,\]
the absolute value of the sample correlation. How can we use the value of $T$ to decide whether we have sufficient evidence to reject the null?

Using a permutation test, it is straightforward to design tests of marginal independence with bounded Type I error under the null.
If $H_0:X\indep Y$ is in fact true, then permuting the $X_i$'s while leaving the $Y_i$'s fixed would not change the distribution of the data.
This yields a natural hypothesis test against the null of independence.

Before stating the test and the result, we first define some notation. Write
\[\bfX = (X_1,\dots,X_n)\in\cX^n, \ \bfY = (Y_1,\dots,Y_n)\in\cY^n\]
to denote the vectors of observations of $X$ and of $Y$, respectively, and also
\[\bfX_\sigma = (X_{\sigma(1)},\dots,X_{\sigma(n)})\]
to denote the permuted vector of $X$ observations, for any $\sigma\in\cS_n$. Moreover, for any test statistic
$T:(\cX\times\cY)^n\to\R$,
abusing notation we write $T(\bfX,\bfY)$ to denote $T((X_1,Y_1),\dots,(X_n,Y_n))$. 

\begin{theorem}[Validity of the permutation test for marginal independence]
    \label{thm:marginal-independence}
    Fix any function $T : (\cX\times\cY)^n \to \R$. Let $(X_1,Y_1),\dots,(X_n,Y_n)\iidsim P$, and let
    \begin{equation}
        \label{eq:test-marginal-independence}
        \psi(\bfX,\bfY) = \ind{\frac{1}{n!}\sum_{\sigma\in\cS_n}\ind{T(\bfX_\sigma,\bfY)\geq T(\bfX,\bfY)} \, \leq \,  \alpha}.
    \end{equation}
    Then $\psi$ is a valid test of the null hypothesis $H_0:X\indep Y$, i.e.,
    \begin{equation}
        \P_P\left( \psi(\bfX,\bfY) = 1 \right) \leq \alpha \textnormal{ for all } P \in \cP_{X \indep Y}.
    \end{equation}
\end{theorem} \index{permutation test}
Intuitively, the validity of this test arises from the fact that, if $X\indep Y$, we have
\begin{equation}
    (\bfX_\sigma,\bfY) \eqd (\bfX,\bfY),
\end{equation}
for all permutations $\sigma \in \cS_n$---and consequently, the observed test statistic value $T(\bfX,\bfY)$ can be compared against permuted values $T(\bfX_\sigma,\bfY)$ to test the null.

\begin{proof}[Proof of Theorem~\ref{thm:marginal-independence}]
Our proof will rely on the validity of the permutation test, as established in Chapter~\ref{chapter:exchangeability} (see Theorem~\ref{thm:perm-test}). In order to apply those results, we will need to condition on $\bfY$. Under the null hypothesis of marginal independence~\eqref{eq:null-marginal-independence}, the joint distribution $P$ can be written as a product, $P=P_X\times P_Y$. Therefore, even after conditioning on $\bfY$, it still holds that the feature values $X_1,\dots,X_n$ are i.i.d.\ draws from $P_X$. 

In particular, this means that $X_1,\dots,X_n$ are exchangeable, after conditioning on $\bfY$. Defining \index{permutation test!p-value}
\begin{equation}
\label{eqn:pvalue_marginal-independence}
p = \frac{\sum_{\sigma\in\cS_n}\ind{T(\bfX_\sigma,\bfY)\geq T(\bfX,\bfY)}}{n!},
\end{equation}
we therefore have $\P_P(p\leq \alpha \mid \bfY)\leq \alpha$, by Theorem~\ref{thm:perm-test} (specifically, since we condition on $\bfY$, we are applying this theorem with the test statistic $T$ replaced by the map $\bfx \mapsto T(\bfx,\bfY)$). Finally, by definition, $\psi(\bfX,\bfY)=1$ if and only if $p\leq \alpha$.
\end{proof}

\section{Conditional independence with a discrete confounder}\label{sec:cond_indep_discrete_case}

\begin{figure}
    \centering
    \includegraphics[width=0.6\textwidth]{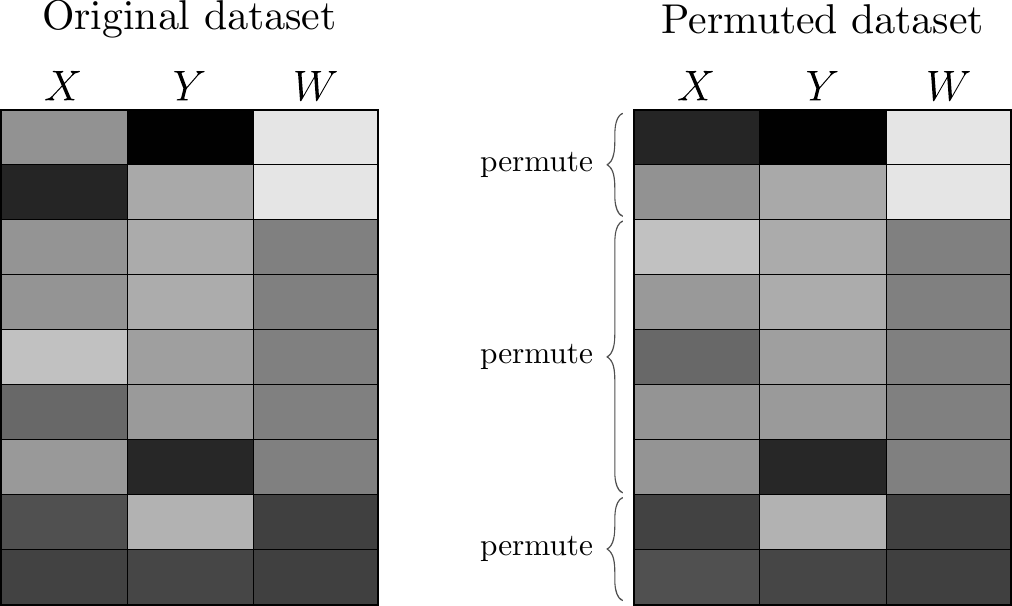}
    \caption{\textbf{Visualization of a dataset with a discrete confounder along with a local permutation of the covariate.} Grayscale levels indicate numerical value. On the left, we show the dataset in its original ordering; note that $W$ only takes on three values, defining three subgroups. On the right, we show a locally permuted version of the dataset, where the values of $X$ are randomly permuted within each subgroup, as in the local permutation test from Section~\ref{sec:define_local_perm_test_discrete}.}
    \commentAlt{Two arrays of grayscale cells, with columns $X$, $Y$, $W$. The arrays are labeled `Original dataset' and `Permuted dataset'. In the $X$ column of the second array, the grayscale values are shuffled within groups determined by constant values of $W$.}
    \label{fig:local-perm-test-discrete}
\end{figure}

Next, we study the problem of conditional independence testing, i.e., testing the null \eqref{eqn:H0_conditional_indep}, in a special setting: we will consider the case that the confounder $W$ is a discrete random variable.
In this context, the problem of conditional independence testing is solvable via an extension of the permutation test, which groups together data points according to their value of the confounder $W$, and then simply permutes the $(X,Y)$ data \emph{within groups}.
This is sometimes called a \emph{local} permutation test. (Later on in this chapter, we will consider versions of this test where permutations are allowed between data points where the $W$ values are only approximately the same.)

Throughout the section, we encourage the reader to note the parallel nature of the arguments used for the results on test-point-conditional coverage with discrete features (Section~\ref{sec:test_conditional_discrete}), the results on regression for discrete features (Section~\ref{sec:regression-discrete}), and the results on calibration via binning (Theorem~\ref{thm:ece_binning}).
The technical arguments used in these sections are roughly the same: when we have access to repeated measurements (in this case, repeated values of $W$), it becomes possible to perform distribution-free inference.

\subsection{Formulating a local permutation test}\label{sec:define_local_perm_test_discrete}
\index{permutation test!local|(}

For testing marginal (i.e., unconditional) independence, $X\indep Y$, we can simply permute the $X$ values to simulate a draw from the null distribution as in Section~\ref{sec:marginal-independence-testing} above. When testing conditional independence, $X\indep Y \mid W$, for a \emph{discrete} confounder $W$, we now see that we can use exactly the same strategy by partitioning the data into groups. 
If we think of the distribution $P$ on $(X,Y,W)$ as describing the general population from which data points are sampled, for any specific value $w$ for the confounder we can consider the `subpopulation' of data points for which $W=w$. Within this subpopulation, $X$ and $Y$ are independent---and so it is again valid to permute $X$ values in order to resample the data and test the null hypothesis. 
See Figure~\ref{fig:local-perm-test-discrete} for an illustration.

Formally, we can define the local permutation test as follows.
We are given an i.i.d.\ sample of data $((X_i, Y_i, W_i) )_{i\in[n]}$ drawn from $P\in\cP_{(X,Y,W)}$, where $\cP_{(X,Y,W)}$ is the space of all distributions on $\cX\times\cY\times\cW$. Let $\cS_n(\bfW) \subseteq \cS_n$ be the set of permutations of $\{1, \ldots, n\}$ that preserve the value of the vector of confounders $\bfW = (W_1, \ldots, W_n)$---i.e.,
\begin{equation}\label{eqn:permutations_discrete_confounder}
    \cS_n(\bfW) = \left\{ \sigma \in \cS_n : W_{\sigma(i)} = W_i \textnormal{ for all } i \in \{1, \ldots, n\} \right\}.
\end{equation}
Then, given some choice of test statistic $T:(\cX\times\cY\times\cW)^n\rightarrow\R$, the local permutation test is defined as
\begin{equation}
    \label{eq:local-permtest}
        \psi(\bfX,\bfY,\bfW) = \\ 
\ind{\frac{1}{|\cS_n(\bfW)|}\sum\limits_{\sigma\in\cS_n(\bfW)}\ind{T(\bfX_\sigma,\bfY,\bfW) \geq T(\bfX,\bfY,\bfW)} \leq \alpha}.
\end{equation}
(Analogously to the notation defined for the marginal independence testing setting, we will write $T(\bfX,\bfY,\bfW)$ to denote $T((X_1,Y_1,W_1),\dots,(X_n,Y_n,W_n))$.)

The local permutation test is valid for testing the null hypothesis of conditional exchangeability.
\begin{theorem}[Distribution-free validity of the local permutation test]
    \label{thm:localperm-validity}
     Fix any function $T : (\cX\times\cY\times\cW)^n \to \R$. Let $(X_1,Y_1,W_1),\dots,(X_n,Y_n,W_n)\iidsim P$, and let $\psi(\bfX,\bfY,\bfW)$ be defined as in~\eqref{eq:local-permtest}. 
Then $\psi$ is a valid test of the null hypothesis $H_0:X\indep Y\mid W$, i.e.,
    \begin{equation}
        \P_P\left( \psi(\bfX,\bfY,\bfW) = 1 \right) \leq \alpha \textnormal{ for all } P \in \cP_{X \indep Y\mid W}.
    \end{equation}     
\end{theorem}
As before, the intuition is that under the null,
the permuted data $(\bfX_\sigma,\bfY,\bfW)$ has the same distribution as the original data $(\bfX,\bfY,\bfW)$,
for all permutations $\sigma \in \cS_n(\bfW)$---but this statement is not well defined since, of course, the set $\cS_n(\bfW)$ is itself random. To be more precise, what we mean is that conditional on $\bfW$, for each $\sigma\in\cS_n(\bfW)$ it holds that $(\bfX_\sigma,\bfY)$ and $(\bfX,\bfY)$ are equal in (conditional) distribution.
We will formalize this intuition in our proof of the validity of the local permutation test, below.

\paragraph{Will the local permutation test be powerful?}
The power of the local  permutation test will depend heavily on several factors.
First, we need to consider the number of `clashes'---i.e., data points $i\neq j\in[n]$ for which $W_i=W_j$, and therefore, feature values $X_i$ and $X_j$ may be permuted in our permutation test. A large number of clashes will lead to a rich set of candidate permutations $\cS_n(\bfW)$ and, potentially, a powerful test. At the other extreme, if each value $W_i$ is unique, then $\cS_n(\bfW)$ will be a singleton set (containing only the identity permutation), leading to $\psi(\bfX,\bfY,\bfW)=0$.

Second, we need to consider the choice of the test statistic $T$ as well. As is the case for any test, if the test statistic cannot discriminate between the null and alternative, then the test will not be powerful. We omit further discussion of this standard point.

Lastly, we note that this test can actually be applied even if $W$ is not discrete, and the theorem above continues to hold. However, this test will not have power when the distribution of $W$ is continuous, since all $W_i$ will be distinct. It may have nontrivial power when the distribution of $W$ is a mixture of discrete and continuous.

\subsection{Proof of validity of the local permutation test}
We now prove that the local permutation test is valid as a test of conditional independence, with no assumptions on the distribution or on the test statistic.
\index{exchangeability!conditional}

\begin{proof}[Proof of Theorem~\ref{thm:localperm-validity}]
First, for any $(\bfx,\bfy,\bfw)\in \cX^n\times\cY^n\times\cW^n$, define
\[\hat{q}(\bfx,\bfy,\bfw) = \quantile\left( \big(T(\bfx_\sigma,\bfy,\bfw)\big)_{\sigma\in\cS_n(\bfw)} ; 1-\alpha\right).\] Then we have
\begin{multline*}
    \psi(\bfx,\bfy,\bfw)
    = \ind{\frac{1}{|\cS_n(\bfw)|}\sum\limits_{\sigma\in\cS_n(\bfw)}\ind{T(\bfx_\sigma,\bfy,\bfw) \geq T(\bfx,\bfy,\bfw)} \leq \alpha}\\
    = \ind{T(\bfx,\bfy,\bfw) > \hat{q}(\bfx,\bfy,\bfw)},
\end{multline*}
where the second step holds by definition of the quantile of a finite list (recall Definition~\ref{def:quantile}; this is also similar to the argument in the proof of Proposition~\ref{prop:conformal-via-pvalues}).
Moreover, for any $\sigma\in\cS_n(\bfw)$, we have an equality of sets,
\[\cS_n(\bfw) = \{\sigma\circ\sigma' : \sigma'\in\cS_n(\bfw)\},\]
where $\circ$ denotes a composition of permutations; this holds specifically because $\cS_n(\bfw)$ is a \emph{subgroup}, rather than an arbitrary subset, of $\cS_n$. Therefore, for any $\sigma\in\cS_n(\bfw)$,
\begin{multline*}\hat{q}(\bfx_\sigma,\bfy,\bfw) = \quantile\left( \big(T((\bfx_\sigma)_{\sigma'},\bfy,\bfw)\big)_{\sigma'\in\cS_n(\bfw)} ; 1-\alpha\right)\\ = \quantile\left( \big(T(\bfx_{\sigma'},\bfy,\bfw)\big)_{\sigma'\in\cS_n(\bfw)} ; 1-\alpha\right) = \hat{q}(\bfx,\bfy,\bfw),\end{multline*}
and consequently,
\[\psi(\bfx_\sigma,\bfy,\bfw)  =\ind{T(\bfx_\sigma,\bfy,\bfw)>\hat{q}(\bfx_\sigma,\bfy,\bfw)} = \ind{T(\bfx_\sigma,\bfy,\bfw)>\hat{q}(\bfx,\bfy,\bfw)}.\]
Therefore,
\begin{multline}\label{eqn:psi_XYW_deterministic_step}
    \frac{1}{|\cS_n(\bfw)|}\sum_{\sigma\in\cS_n(\bfw)}
    \ind{\psi(\bfx_\sigma,\bfy,\bfw)=1}
    =\frac{1}{|\cS_n(\bfw)|}\sum_{\sigma\in\cS_n(\bfw)}
    \ind{T(\bfx_\sigma,\bfy,\bfw)>\hat{q}(\bfx,\bfy,\bfw)}\\
    =\frac{1}{|\cS_n(\bfw)|}\sum_{\sigma\in\cS_n(\bfw)}
    \ind{T(\bfx_\sigma,\bfy,\bfw)>\quantile\left( \big(T(\bfx_{\sigma'},\bfy,\bfw)\big)_{\sigma'\in\cS_n(\bfw)} ; 1-\alpha\right)}
    \leq \alpha,
\end{multline}
by standard properties of the quantile (see Fact~\ref{fact:conversion-quantiles-cdfs}\ref{fact:conversion-quantiles-cdfs_part3}).

Next, fix any distribution $P\in \cP_{X \indep Y \mid W}$. We now claim that
\begin{equation}\label{eqn:eqd_XY_W_localperm}(\bfX_\sigma,\bfY) \mid \bfW \ \eqd \ 
 (\bfX,\bfY) \mid \bfW,\end{equation}
 for any $\sigma\in\cS_n(\bfW)$.
To see why, first note that conditional on $\bfW$, we have $(X_i,Y_i)\sim P_{X\mid W}(\cdot\mid W_i)\times P_{Y\mid W}(\cdot\mid W_i)$, since $X\indep Y\mid W$. Moreover, the data points $i=1,\dots,n$ are independent, and so 
\[(\bfX,\bfY)\mid \bfW\sim \big(P_{X\mid W}(\cdot\mid W_1)\times \dots \times  P_{X\mid W}(\cdot\mid W_n)\big) \times \big(P_{Y\mid W}(\cdot\mid W_1)\times \dots \times  P_{Y\mid W}(\cdot\mid W_n)\big).\]
Similarly, for any $\sigma\in\cS_n(\bfW)$ (treating $\sigma$ as fixed once we have conditioned on $\bfW$), we have
\begin{multline*}(\bfX_\sigma,\bfY)\mid \bfW\sim \big(P_{X\mid W}(\cdot\mid W_{\sigma(1)})\times \dots \times  P_{X\mid W}(\cdot\mid W_{\sigma(n)})\big) \\\times \big(P_{Y\mid W}(\cdot\mid W_1)\times \dots \times  P_{Y\mid W}(\cdot\mid W_n)\big).\end{multline*}
But since $\sigma\in\cS_n(\bfW)$, we have $W_{\sigma(i)} = W_i$ for all $i$---and therefore these two conditional distributions are the same. In particular, the equality~\eqref{eqn:eqd_XY_W_localperm} implies that
\[\P_P(\psi(\bfX,\bfY,\bfW) = 1 \mid \bfW) = \P_P(\psi(\bfX_\sigma,\bfY,\bfW) = 1 \mid \bfW)\]
holds for each $\sigma\in\cS_n(\bfW)$. Taking an average over $\sigma\in\cS_n(\bfW)$, we therefore have
\begin{multline*}\P_P(\psi(\bfX,\bfY,\bfW) = 1 \mid \bfW) = \frac{1}{|\cS_n(\bfW)|}\sum_{\sigma\in\cS_n(\bfW)}\P_P(\psi(\bfX_\sigma,\bfY,\bfW) = 1 \mid \bfW)\\
=\E_P\left[ \frac{1}{|\cS_n(\bfW)|}\sum_{\sigma\in\cS_n(\bfW)} \ind{\psi(\bfX_\sigma,\bfY,\bfW) = 1} \ \middle| \ \bfW\right]
\leq \alpha,
\end{multline*}
where the last step holds by~\eqref{eqn:psi_XYW_deterministic_step}. After marginalizing over $\bfW$, this completes the proof.
\end{proof}
\index{permutation test!local|(}

\section{The hardness of conditional independence testing}\label{sec:cond_indep_hardness}
\index{conditional independence testing|(}

The results of Section~\ref{sec:cond_indep_discrete_case} above show that testing conditional independence is possible in the setting where the confounder $W$ is discrete (or, perhaps, is a mixture of a discrete and a continuous distribution). In this section, we turn to the more challenging setting where the distribution of $W$ is instead continuous, or more generally, nonatomic---recall Definition~\ref{def:nonatomic}. 

We will now see that in this setting, it is impossible to construct a distribution-free test with nontrivial power. We will use the notation
\[\psi: (\cX\times\cY\times\cW)^n \rightarrow \{0,1\}\]
to represent a test function mapping a dataset of $n$ triples $((X_i,Y_i,W_i))_{i\in[n]}$ to the indicator of rejection.
That is, $\psi(\bfX,\bfY,\bfW)=0$ indicates a failure to reject the null hypothesis of conditional independence~\eqref{eqn:H0_conditional_indep}, while $\psi(\bfX,\bfY,\bfW)=1$ indicates rejection.
\begin{theorem}[Hardness of testing conditional independence]\label{thm:hardness-conditional-independence-nonatomic}
    If the test $\psi$ has Type I error bounded by $\alpha$ for any distribution, i.e.,
    \[\P_P\big(\psi(\bfX,\bfY,\bfW)=1 \big)\leq \alpha\textnormal{ for all $P\in\cP_{X\indep Y \mid W}$,}\]
    then the power of the test is not better than random for \emph{any} distribution $P$ such that $P_W$ is nonatomic,
    \[\P_P\big(\psi(\bfX,\bfY,\bfW)=1 \big) \leq \alpha\textnormal{ for all $P\in\cP_{(X,Y,W)}$ with $P_W$ nonatomic.}\]
\end{theorem} \index{hardness result!conditional independence testing}

The proof of this hardness result will rely on the sample--resample technique introduced earlier, in Lemma~\ref{lem:sample-resample}. (Recall that this proof technique was used in Chapter~\ref{chapter:conditional} for establishing the hardness result for test-conditional predictive inference, and in Chapters~\ref{chapter:regression} and~\ref{chapter:calibration} for establishing hardness results for regression and calibration.)

\begin{proof}[Proof of Theorem~\ref{thm:hardness-conditional-independence-nonatomic}]
First, fix a large integer $M$, and let $((X^{(i)},Y^{(i)},W^{(i)}))_{i\in[M]}$ be an arbitrary sequence of $M$ data points. Let $\widehat{P}_M = \frac{1}{M}\sum_{i=1}^M\delta_{(X^{(i)},Y^{(i)},W^{(i)})}$ be the empirical distribution of the sequence.
If $W^{(1)},\dots,W^{(M)}$ are distinct, then under the distribution $\widehat{P}_M$ on $(X,Y,W)$ it holds that $X$ and $Y$ are deterministic functions of $W$---and in particular it therefore trivially holds that $X\indep Y\mid W$ under the distribution $\widehat{P}_M$. By distribution-free validity of the test $\psi$, we must therefore have
\[\P_{\widehat{P}_M}\big(\psi(\bfX,\bfY,\bfW)=1 \big) \leq \alpha\]
as long as $W^{(1)},\dots,W^{(M)}$ are distinct, where the probability is taken with respect to samples $((X_i,Y_i,W_i))_{i\in[n]}$ sampled i.i.d.\ from $\widehat{P}_M$. 

Next, we will construct $\widehat{P}_M$ by sampling the $M$ data points i.i.d.\ from $P$: specifically,
\begin{enumerate}
    \item Sample $(X^{(1)},Y^{(1)},W^{(1)}),\dots,(X^{(M)},Y^{(M)},W^{(M)})\iidsim P$, and define the empirical distribution $\widehat{P}_M = \frac{1}{M}\sum_{i=1}^M \delta_{(X^{(i)},Y^{(i)},W^{(i)})}$;
    \item Conditional on $\widehat{P}_M$, sample $(X_1,Y_1,W_1),\dots,(X_n,Y_n,W_n)\iidsim \widehat{P}_M$.
\end{enumerate}
We have assumed that $P_W$ is nonatomic, and therefore $W^{(1)},\dots,W^{(M)}$ are  distinct almost surely. By the calculation above, then, 
\[\P_{\widehat{P}_M}\big(\psi(\bfX,\bfY,\bfW)=1\mid \widehat{P}_M\big) \leq \alpha\]
holds almost surely.

Now let $Q$ be the marginal distribution of $(X_1,Y_1,W_1),\dots,(X_n,Y_n,W_n)$ under the construction above, after marginalizing over the random empirical distribution $\widehat{P}_M$. We therefore have
\[\P_Q\big(\psi(\bfX,\bfY,\bfW)=1\big) = \E\left[\P_{\widehat{P}_M}\big(\psi(\bfX,\bfY,\bfW)=1\mid \widehat{P}_M\big)\right] \leq \alpha.\]
Finally, by Lemma~\ref{lem:sample-resample}, it holds that 
\[\dtv\big(P^n, Q\big) \leq \frac{n(n-1)}{2M}.\]
Therefore,
\[\P_{P}\big(\psi(\bfX,\bfY,\bfW)=1\big) \leq \alpha + \frac{n(n-1)}{2M}.\]
Since $M$ can be chosen to be arbitrarily large, this proves the desired result.
\end{proof}

\section{Testing conditional independence under smoothness assumptions}
\label{sec:conditional_indep_smoothness}

The hardness result in Theorem~\ref{thm:hardness-conditional-independence-nonatomic} establishes that a test $\psi$ cannot have nontrivial power for nonatomic $W$ if we require it to be valid against \emph{all} distributions on $(X,Y,W)$ for which conditional independence holds. 
We can conclude from this that, if we would like to design a test $\psi$ that has nontrivial power to detect conditional dependence
given a nonatomic confounder $W$, we need to relax our notion of validity. In other words, we need to require Type I error 
to be controlled against a more restricted set of possibilities under the null. 

In this section, we examine one such possibility: placing a smoothness assumption on the conditional dependence between the variables. This assumption leads to an approximate form of exchangeability, namely, data points with close values of $W$ will then have $X$ values that are nearly conditionally exchangeable.
For instance, if the space $\cW$ for the confounder $W$ is equipped with a norm $\|\cdot\|$, we can add the assumptions that the map
\[w \mapsto P_{X\mid W}(\cdot \mid w)\]
is $L$-Lipschitz with respect to the Hellinger distance---that is, for any $w,w'\in\cW$, \index{Hellinger distance}
\begin{equation}\label{eqn:assume_XYW_Hellinger}\textnormal{d}_{\textnormal{H}}\big(P_{X\mid W}(\cdot \mid w),P_{X\mid W}(\cdot \mid w')\big) \leq L\|w-w'\| .\end{equation}
Here $\textnormal{d}_{\textnormal{H}}$ denotes the Hellinger distance: for two distributions $P_1,P_2$ that have densities $f_1,f_2$ with respect to a common base measure $\mu$, 
\[\textnormal{d}_{\textnormal{H}}(P_1,P_2) = \left(\frac{1}{2}\int_{\cX} \left(\sqrt{f_1(x)}-\sqrt{f_2(x)}\right)^2\;\mathsf{d}\mu(x)\right)^{1/2}.\]
We now define a more restricted set of nulls,
\[\cP^L_{X\indep Y\mid W} = \{P\in\cP_{X\indep Y\mid W} : \textnormal{ condition~\eqref{eqn:assume_XYW_Hellinger} holds}\},\]
i.e., all joint distributions $P$ on $(X,Y,W)$ for which $X\indep Y\mid W$ and for which the smoothness assumption~\eqref{eqn:assume_XYW_Hellinger} holds for the conditional distribution $P_{X\mid W}$, with Lipschitz constant $L$.

At a high level, this assumption allows for the continuous case to be approximately as easy as the discrete case. In the case of discrete $W$, we are able to permute $X$ values within any subset of data points that all have the same value of $W$, as in~\eqref{eqn:permutations_discrete_confounder}. In the case of continuous $W$, this is not a meaningful approach because no value of $W$ appears twice---but with a smoothness assumption in place, it again becomes possible to perform a test of conditional independence via a permutation test, because we can now permute $X$ values across data points that have \emph{similar} values of $W$. We can think of this as a binned version of the local permutation test studied in Section~\ref{sec:cond_indep_discrete_case}. See Figure~\ref{fig:binned-local-perm-test} for an illustration.

\begin{figure}[t]
    \centering
    \includegraphics[width=0.7\textwidth]{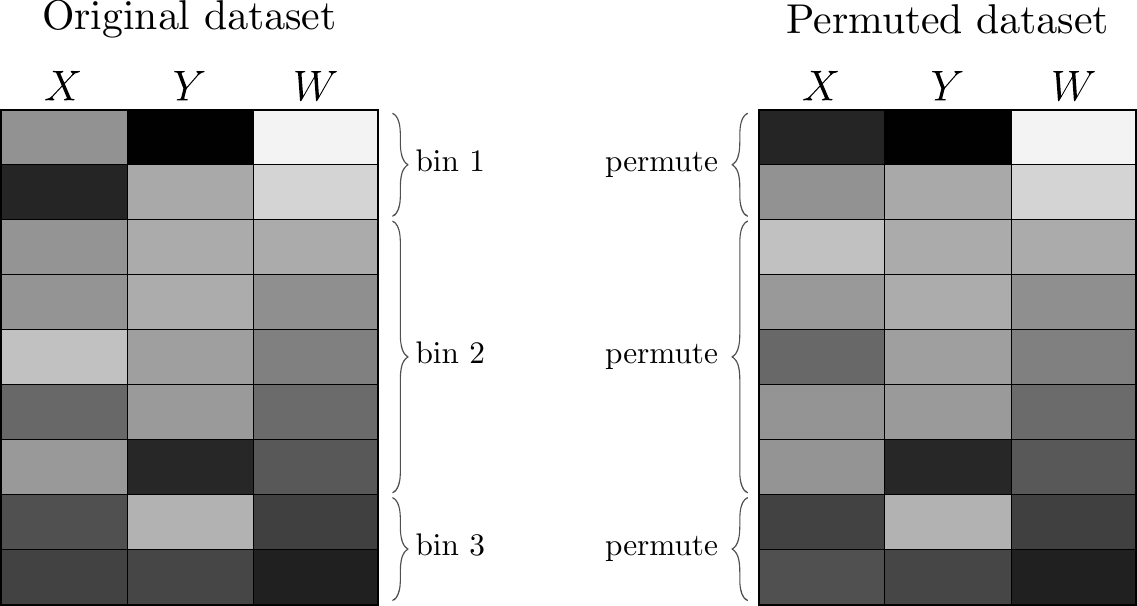}
    \caption{\textbf{Visualization of a dataset with a continuous confounder along with a local permutation of the covariate under a smoothness assumption.} The figure is similar to Figure~\ref{fig:local-perm-test-discrete}, but $W$ is continuous.
    To address this, we bin $W$ and permute within these bins. This is the type of permutation used for the test described Section~\ref{sec:conditional_indep_smoothness}.}
    \commentAlt{Two arrays of grayscale cells, with columns $X$, $Y$, $W$. The arrays are labeled `Original dataset' and `Permuted dataset'. In the $X$ column of the second array, the grayscale values are shuffled within groups determined by similar values of $W$.}
    \label{fig:binned-local-perm-test}
\end{figure}

To formalize this, first construct a partition of $\cW$, given by
\[\cW = \cup_k \cW_k,\]
where each bin $\cW_k$ contains only values of $W$ that are close together in the norm $\|\cdot\|$: we assume
\[\max_k \sup_{w,w'\in\cW_k} \|w - w'\|\leq h,\]
for some finite $h$. (The number of bins may be finite or countably infinite.)
Now define a permutation test, analogous to the one defined in Section~\ref{sec:cond_indep_discrete_case}:
we define the set of allowed permutations as
\begin{equation}\label{eqn:permutations_smooth_confounder}
    \cS_n^{\textnormal{bin}}(\bfW) = \left\{ \sigma \in \cS_n : W_{\sigma(i)}, W_i\textnormal{ are in the same bin} \textnormal{ for all } i \in \{1, \ldots, n\} \right\}.
\end{equation}
Finally, define
\begin{equation}
    \label{eq:local-permtest-bin}
        \psi(\bfX,\bfY,\bfW) = \\ 
\ind{\frac{1}{|\cS_n^{\textnormal{bin}}(\bfW)|}\sum\limits_{\sigma\in\cS_n^{\textnormal{bin}}(\bfW)}\ind{T(\bfX_\sigma,\bfY,\bfW) \geq T(\bfX,\bfY,\bfW)} \leq \alpha},
\end{equation}
which is analogous to the test $\psi(\bfX,\bfY,\bfW)$ defined in~\eqref{eq:local-permtest} for the discrete case.

\begin{theorem}[Validity of the binned local permutation test]
    \label{thm:localperm-validity-smooth}
    Fix any function $T : (\cX\times\cY\times\cW)^n \to \R$. Let $(X_1,Y_1,W_1),\dots,(X_n,Y_n,W_n)\iidsim P$, let $\cW = \cup_k \cW_k$ be any partition with $\max_k \sup_{w,w'\in\cW_k} \|w - w'\|\leq h$, and let $\psi(\bfX,\bfY,\bfW)$ be defined as in~\eqref{eq:local-permtest-bin}. 
Then $\psi$ satisfies
    \begin{equation}
        \P_P\left( \psi(\bfX,\bfY,\bfW) = 1 \right) \leq \alpha + Lh\sqrt{2n}\textnormal{ \ for all } P \in \cP_{X \indep Y\mid W}^L.
    \end{equation}     
\end{theorem} \index{permutation test!local}
In other words, by taking bins of diameter $h=\littleo(n^{-1/2})$, we can ensure that the test $\psi$ is (nearly) a valid test of the null that $X \indep Y \mid W$, as long as we also assume the smoothness condition.

\begin{proof}[Proof of Theorem~\ref{thm:localperm-validity-smooth}]
First, for each $k\geq 1$, let $\pi_k = \P_P(W\in\cW_k)$, and (if $\pi_k>0$) let $P^{(k)}$ denote the distribution of the triple $(X,Y,W)$ conditional on the event $W\in\cW_k$. Then we can express $P$ as a mixture,
\[P = \sum_{k\geq 1} \pi_k \cdot  P^{(k)}.\]

The idea of the proof is to find a distribution $\tilde{P}$ that is similar to $P$, for which the test $\psi$ is exactly valid. We define the distribution $\tilde{P}$ on $(X,Y,W)$ as another mixture,
\[\tilde{P} = \sum_{k\geq 1} \pi_k \cdot \big(P^{(k)}_X \times P^{(k)}_{(Y,W)}\big),\]
where $P^{(k)}_X$ and $P^{(k)}_{(Y,W)}$ are the marginal distributions of $X$ and of $(Y,W)$, respectively, under the joint distribution $P^{(k)}$ on $(X,Y,W)$. Note that $P^{(k)}_X$ can be viewed as a convex combination of conditional distributions $P_{X\mid W}(\cdot \mid w)$ indexed by $w\in\cW_k$---this is because $P^{(k)}_X$ is the marginal distribution of $X$, under the joint distribution of $(X,W)$ conditional on $W\in\cW_k$. In particular, since $\textnormal{d}_{\textnormal{H}}(P_{X\mid W}(\cdot\mid w),P_{X\mid W}(\cdot\mid w'))\leq Lh$ for all $w,w'\in\cW_k$ due to the smoothness assumption~\eqref{eqn:assume_XYW_Hellinger}, this means that
\begin{equation}\label{eqn:dH_within_bin}\textnormal{d}_{\textnormal{H}}(P_{X\mid W}(\cdot\mid w),P^{(k)}_X)\leq Lh\textnormal{ for all $k$ and all $w\in\cW_k$,}\end{equation}
since $\textnormal{d}_{\textnormal{H}}(\cdot,\cdot)$ is convex in each argument.

Next, we calculate a bound on the difference between sampling from $P$ and sampling from $\tilde{P}$. By construction, the marginal distribution of $(Y,W)$ is the same under $P$ as under $\tilde{P}$, since in both cases it is equal to $\sum_k \pi_k P^{(k)}_{(Y,W)}$. Therefore, the only difference between $P$ and $\tilde{P}$ arises from the differences in the conditional distribution of $X$: under $P$ this conditional distribution is given by
\[X\mid (Y,W) \sim P_{X\mid W}(\cdot \mid W),\]
while under $\tilde{P}$, we have
\[X\mid (Y,W) \sim P^{(k(W))}_X,\] 
where $k(W)$ identifies the bin to which the confounder $W$ belongs---that is, for any $k$ and any $w\in\cW_k$, $k(w) = k$. Therefore,
\[\textnormal{d}_{\textnormal{H}}^2(P,\tilde{P})
= \E_P\left[\textnormal{d}_{\textnormal{H}}^2\big(P_{X\mid W}(\cdot \mid W),P_X^{(k(W))}\big)\right] \leq (Lh)^2,\]
where the first step holds by properties of the Hellinger distance, and the last step applies~\eqref{eqn:dH_within_bin}.

Since squared Hellinger distance is subadditive, we therefore have
\[\textnormal{d}^2_{\textnormal{H}}(P^n,\tilde{P}^n) \leq n(Lh)^2,\]
and since total variation distance is bounded by the Hellinger distance up to a factor of $\sqrt{2}$, this implies
\[\dtv(P^n,\tilde{P}^n)\leq \sqrt{2} \textnormal{d}_{\textnormal{H}}(P^n,\tilde{P}^n) \leq Lh \sqrt{2n}.\]
Therefore,
\[\P_P(\psi(\bfX,\bfY,\bfW)=1)
\leq \P_{\tilde{P}}(\psi(\bfX,\bfY,\bfW)=1) + Lh\sqrt{2n}, \]
i.e., the probability of $\psi$ rejecting $H_0$ when applied to data $((X_i,Y_i,W_i))_{i\in[n]}$ sampled from $P$, can be bounded by the probability of rejection for data drawn from $\tilde{P}$.

To complete the proof, we need to bound $\P_{\tilde{P}}(\psi(\bfX,\bfY,\bfW)=1)$. Indeed, we will see that it is simply bounded by $\alpha$, because of the construction of $\tilde{P}$. 
Define random variable $K=k(W)$ and consider a new triple,
\[(X, (Y,W), K).\]
Note that, under the distribution $\tilde{P}$, by construction it holds that
\begin{equation}\label{eqn:XYWK}X \indep (Y,W) \mid K.\end{equation}
Next we construct a test statistic $\tilde{T}$ on these new triples, 
\[\tilde{T}: (\cX \times (\cY\times \cW) \times [K])^n \rightarrow \R.\]
defined by
\[\tilde{T}\Big(\big(X_1,(Y_1,W_1),K_1\big),\dots,\big(X_n,(Y_n,W_n),K_n\big)\Big) = T\big((X_1,Y_1,W_1),\dots,(X_n,Y_n,W_n)\big).\]
We can observe that running the local permutation test~\eqref{eq:local-permtest} defined in Section~\ref{sec:cond_indep_discrete_case}, with data $\big((X_i,(Y_i,W_i),K_i)\big)_{i\in[n]}$ and test statistic $\tilde{T}$, is exactly equivalent to our test $\psi(\bfX,\bfY,\bfW)$ as defined in~\eqref{eq:local-permtest-bin}. But since the data $\big((X_i,(Y_i,W_i),K_i)\big)_{i\in[n]}$ satisfies the null hypothesis of conditional independence under $\tilde{P}$ by~\eqref{eqn:XYWK}, we can apply Theorem~\ref{thm:localperm-validity} (with $(X,(Y,W),K)$ in place of $(X,Y,W)$, and with $\tilde{T}$ in place of $T$), to see that $\P_{\tilde{P}}(\psi(\bfX,\bfY,\bfW)=1)\leq \alpha$, as desired.
\end{proof}

\index{conditional independence testing|)}

\section*{Bibliographic notes}
\addcontentsline{toc}{section}{\protect\numberline{}\textnormal{\hspace{-0.8cm}Bibliographic notes}}

Conditional independence testing is a foundational topic in statistics~\citep[e.g.,][]{dawid1979conditional}. 
The problem of testing conditional independence is closely linked to the problem of \emph{variable selection} in high-dimensional regression: given a high-dimensional feature vector $X=(X_1,\dots,X_d)$ and response $Y$, testing whether a particular feature $X_j$ should be included into the fitted model is often framed as a question of conditional independence, i.e., testing
\[H_{0,j}: \, X_j \indep Y \mid X_{-j},\]
where $X_{-j}=(X_1,\dots,X_{j-1},X_{j+1},\dots,X_d)$ denotes the remaining covariates once $X_j$ has been removed. 

Using permutation tests to test (marginal) independence is a classical topic; see the bibliographic notes in Chapter~\ref{chapter:exchangeability}. Ideas related to the use of permutations or randomization within strata to test for conditional independence with a discrete confounder are widespread~\citep[e.g.,][]{birch1965detection, rosenbaum1984conditional}, and the local permutation test presented in Section~\ref{sec:cond_indep_discrete_case} (for the case of a discrete confounder $W$) is studied in detail in~\citet{canonne2018testing}. For a continuous confounder, one can extend this idea to consider permutations that approximately preserve the values of the confounder~\citep[e.g.,][]{margaritis2005distribution, fukumizu2007kernel, doran2014permutation, sen2017model, kim2022local}.
The hardness of conditional independence testing in the continuous case was first established by \citet[Theorem 2]{shah2020hardness}; the version of this hardness result presented here, in Theorem~\ref{thm:hardness-conditional-independence-nonatomic}, is due to~\citet[Theorem 1]{kim2022local}.

In Section~\ref{sec:conditional_indep_smoothness}, we assumed a smoothness condition to enable testing conditional independence beyond the discrete case. 
Assuming a smoothness condition is one way to avoid the hardness result of Section~\ref{sec:cond_indep_hardness}, but there are many alternative assumption-lean approaches in the literature as well. For example, \citet{azadkia2021simple,shi2024azadkia}
consider a conditional form of Chatterjee's correlation coefficient \citep{chatterjee2021new} for asymptotically valid estimates of conditional dependence (and an asymptotically valid test of conditional independence) without such assumptions.
As another example, in the \emph{model-X} framework \citep{candes2018panning}, which assumes that we have knowledge (or approximate knowledge) of the conditional distribution of $X\mid W$.
Within this framework, testing conditional independence becomes possible, since we can use the (approximate) knowledge of the conditional distribution $P_{X\mid W}$ to resample $\bfX$ under the null hypothesis of conditional independence---see \citet{candes2018panning,berrett2020conditional}.
For variable selection in this setting, the knockoff filter method leverages a different type of exchangeability, \emph{pairwise exchangeability}, for false discovery rate control \citep{barber2015controlling,candes2018panning}.

The binned version of the local permutation test and its accompanying theoretical guarantee, presented in Section~\ref{sec:conditional_indep_smoothness}, are adapted from the work of \citet{kim2022local}.
For background on the Hellinger distance and its properties (which play a key role in the proof of Theorem~\ref{thm:localperm-validity-smooth}), see \citet[Chapter 4]{le2012asymptotic}.

\section*{Exercises}
\addcontentsline{toc}{section}{\protect\numberline{}\textnormal{\hspace{-0.8cm}Exercises}}
\begin{enumerate}[font=\bfseries, label={\thechapter.\arabic*}, labelsep=1em, itemsep=1em]
\item  The p-value $p$ defined in~\eqref{eqn:pvalue_marginal-independence}, for the marginal test of independence studied in Theorem~\ref{thm:marginal-independence}, may be slightly conservative due to discreteness and due to the potential presence of ties. Construct a randomized version of this p-value (analogous to Lemma~\ref{lem:randomize-perm}) and prove that it follows a uniform distribution under the null hypothesis $P\in\cP_{X\indep Y}$.

\item Recall that Theorem~\ref{thm:hardness-conditional-independence-nonatomic} established the hardness of testing conditional independence, $X\indep Y\mid W$, for any joint distribution $P$ on $(X,Y,W)$ for which the marginal $P_W$ is nonatomic. In this exercise, we will consider a variant of this result: suppose instead that $P_W$ is discrete, but has no large point masses, so that we are unlikely to observe repeated values of $W$.

    Concretely, let $\cW = \{w_1,\dots,w_K\}$ be a finite set. Define
    \[1-\epsilon_n(P) = \P_P\left(\textnormal{$W_1$, \dots, $W_n$ are all distinct}\right),\]
    where the probability is calculated for $(X_1,Y_1,W_1),\dots,(X_n,Y_n,W_n)\iidsim P$.
    Prove that, for any test $\psi$ with Type I error bounded by $\alpha$,
    \[\P_P(\psi(\bfX,\bfY,\bfW)) \leq \alpha + \epsilon_n(P).\]
    \emph{Hint: to prove this, we can use a variant of the sample--resample technique (Lemma~\ref{lem:sample-resample}). Specifically, for each $k=1,\dots,K$, sample $(x_k,y_k)$ from the conditional distribution $P_{(X,Y)\mid W=w_k}$, and construct a distribution with this support.}
    
\item  In Theorem~\ref{thm:marginal-independence}, we have seen that it is possible to construct a powerful test of marginal independence, which has distribution-free validity. Specifically, with such a test, we may confidently \emph{reject} the null hypothesis that $X\indep Y$, and conclude that $X$ and $Y$ are \emph{not} independent. We might also be interested in asking the converse question: is it possible to instead confidently conclude that $X$ and $Y$ \emph{are} independent? 
    
    In this exercise, we will weaken this goal further and consider the null hypothesis that $Y$ is deterministically equal to a function of $X$. Let $\psi:(\cX\times\cY)^n\to\{0,1\}$ be a test satisfying distribution-free validity for testing this null hypothesis,
    \[\P_P(\psi(\bfX,\bfY)=1)\leq \alpha \textnormal{ for any $P$ for which $\P_P(Y=f(X))=1$ for some $f:\cX\to\cY$}.\]
    Prove that, for any distribution $P$ with a nonatomic marginal $P_X$ (but not necessarily satisfying the null), this test is powerless: $\P_P(\psi(\bfX,\bfY)=1)\leq \alpha$. \emph{Hint: consider the proof of Theorems~\ref{thm:regression_nonatomic} and~\ref{thm:regression_nonatomic_median}.}

\end{enumerate}

\backmatter
\appendix

\bibliographystyle{plainnat}
\bibliography{bibliography}

\addcontentsline{toc}{chapter}{Index}
\index{inductive conformal prediction|see{split conformal prediction}}
\index{overcoverage|see{coverage}}
\index{conformal score|see{score function}}
\index{conformalized quantile regression|see{score function}}
\index{cumulative-probability score|see{score function}}
\index{high-density score|see{score function}}
\index{residual score|see{score function}}
\index{scaled residual score|see{score function}}
\index{multiset|see{bag}}
\index{set size|see{optimality}}
\index{jackknife+|see{cross-validation}}
\index{martingale|see{supermartingale}}
\index{loss function|see{risk control}}
\index{aggregation|see{model aggregation}}
\index{ensemble|see{model aggregation}}
\index{binned expected calibration error|see{expected calibration error (ECE)}}
\index{calibration|see{perfect calibration}}
\index{calibration|see{post-hoc calibration}}
\index{conformal risk control|see{risk control}}
\index{exchangeability!violations|see{nonexchangeability}}
\index{quantile regression|see{score function}}
\index{CV+|see{cross-validation}}
\index{jackknife|see{cross-validation}}
\index{anomaly detection|see{outlier detection}}
\printindex

\end{document}